\documentclass[10pt,reqno]{article}

\usepackage{etex}
\usepackage{enumitem}
\usepackage{setspace}
\usepackage{amsmath}

\usepackage{tabu}
\usepackage{tikz}
\usepackage{tikz-cd}
\usepackage[margin=1in]{geometry}
\usepackage[matrix,arrow,curve,frame]{xy}
\usepackage{hyperref}
\usepackage{amsthm}
\usepackage[capitalise,noabbrev]{cleveref}
\usepackage{textcomp}
\usepackage{amsfonts}
\usepackage{amssymb}
\usepackage{wasysym}
\usepackage{mathrsfs}
\usepackage{mathtools}

\definecolor{my-linkcolor}{rgb}{0.75,0,0}
\definecolor{my-citecolor}{rgb}{0.1,0.57,0}
\definecolor{my-urlcolor}{rgb}{0,0,0.75}
\hypersetup{
	colorlinks, 
	linkcolor={my-linkcolor},
	citecolor={my-citecolor}, 
	urlcolor={my-urlcolor}
}

\title{Tensor categories for vertex operator algebra extensions}
\author{Thomas Creutzig
  \thanks{T.~C. is supported by the Natural
    Sciences and Engineering Research Council of Canada (RES0020460).}
  \and Shashank Kanade
  \thanks{S.~K. gratefully acknowledges the support provided by
    post-doctoral fellowship awarded by the Pacific Institute for the
    Mathematical Sciences and Endeavour Research Fellowship (6127\_2017)
    awarded by the Department of Education and Training, Australian
    Government. S.~K. is presently supported by the Collaboration Grant 
    for Mathematicians \#636937, awarded by the Simons Foundation.
}
  \and Robert McRae } \date{}

\numberwithin{equation}{section}

\theoremstyle{definition}\newtheorem{rema}{Remark}[section]
\theoremstyle{plain}\newtheorem{propo}[rema]{Proposition}
\newtheorem{theo}[rema]{Theorem}
\theoremstyle{definition}\newtheorem{defi}[rema]{Definition}
\theoremstyle{plain}\newtheorem{lemma}[rema]{Lemma}
\newtheorem{corol}[rema]{Corollary}
\theoremstyle{definition}\newtheorem{exam}[rema]{Example}
\theoremstyle{definition}
\theoremstyle{definition}\newtheorem{nota}[rema]{Notation}
    
\theoremstyle{definition}\newtheorem{assum}[rema]{Assumption}

\newcommand{\cY}{\mathcal{Y}}
\newcommand{\cA}{\mathcal{A}}
\newcommand{\cR}{\mathcal{R}}
\newcommand{\cM}{\mathcal{M}}
\newcommand{\cS}{\mathcal{S}}
\newcommand{\cF}{\mathcal{F}}
\newcommand{\cI}{\mathcal{I}}
\newcommand{\cJ}{\mathcal{J}}

\newcommand{\cC}{\mathcal{C}}
\newcommand{\cG}{\mathcal{G}}
\newcommand{\cW}{\mathcal{W}}

\newcommand{\sC}{{\mathcal{SC}}}
\newcommand{\sCeven}{{\underline{\sC}}}
\newcommand{\sR}{{\underline{\cR}}}
\newcommand{\sA}{{\underline{\cA}}}

\newcommand{\CC}{\mathbb{C}}
\newcommand{\ZZ}{\mathbb{Z}}
\newcommand{\NN}{\mathbb{N}}

\newcommand{\RR}{\mathbb{R}}

\newcommand{\sunit}{{\underline{\unit}}}
\newcommand{\sleft}{{\underline{l}}}
\newcommand{\sright}{{\underline{r}}}
\newcommand{\ztwo}{\ZZ/2\ZZ}
\newcommand{\sic}{{J}}
\newcommand{\unit}{\mathbf{1}}
\newcommand{\hopflink}{{\,\text{\textmarried}}}
\newcommand{\fus}[1]{\mathbin{\boxtimes_{#1}}}

\newcommand{\zero}{{\bar{0}}}
\newcommand{\one}{{\bar{1}}}

\newcommand{\VOA}{{vertex operator algebra}}
\newcommand{\VOSA}{{vertex operator superalgebra}}
\newcommand{\VOAs}{{vertex operator algebras}}
\newcommand{\VOSAs}{{vertex operator superalgebras}}

\DeclareMathOperator{\rank}{rank}

\DeclareMathOperator{\tw}{tw}
\DeclareMathOperator{\fix}{fix}
\DeclareMathOperator{\ord}{ordtw}
\DeclareMathOperator{\qdim}{qdim}
\DeclareMathOperator{\adim}{adim}

\DeclareMathOperator{\id}{1}
\DeclareMathOperator{\im}{Im}

\DeclareMathOperator{\tr}{Tr}
\DeclareMathOperator{\rep}{Rep}

\let\ker\relax
\let\hom\relax
\DeclareMathOperator{\ker}{Ker}
\DeclareMathOperator{\hom}{Hom}
\DeclareMathOperator{\Endo}{End}
\DeclareMathOperator{\ch}{ch}

\DeclareMathOperator{\vir}{Vir}
\DeclareMathOperator{\com}{Com}
\newcommand{\even}{{\bar{0}}}
\newcommand{\odd}{{\bar{1}}}

\newcommand{\reploc}{\rep^0 V}
\newcommand{\reptw}{{\rep^{\tw}} V}
\newcommand{\repA}{\rep A}
\newcommand{\repzA}{\rep^0 A}

\newcommand{\cSloc}{[\rep^0 V]}
\newcommand{\cStw}{[{\rep^{\tw}} V]}
\newcommand{\cSfix}{[{\rep^{\fix}} V]}
\newcommand{\cSord}{[{\rep^{\ord}} V]}

\usepackage{xparse}

\usetikzlibrary{decorations.pathreplacing}
\usetikzlibrary{decorations.markings}
\usetikzlibrary{calc}

\allowdisplaybreaks

\begin{document}

\title{Tensor categories for vertex operator superalgebra extensions}	
\date{}
\maketitle

\abstract{ 
  Let $V$ be a \VOA{} with a category $\cC$ of (generalized)
  modules that has vertex tensor category structure, and thus
  braided tensor category structure, and let $A$ be a vertex operator
  (super)algebra extension of $V$. We employ tensor categories to
  study untwisted (also called local)
  $A$-modules in $\cC$, using results of
  Huang-Kirillov-Lepowsky that show that $A$ is a (super)algebra object in $\cC$ and that generalized $A$-modules in $\cC$ correspond exactly to local modules for the corresponding (super)algebra object. Both categories, of local modules for a $\cC$-algebra and (under suitable conditions) of generalized $A$-modules, have natural braided monoidal category structure, given in the first case by Pareigis and Kirillov-Ostrik and in the second case by Huang-Lepowsky-Zhang.

Our main result is that the Huang-Kirillov-Lepowsky isomorphism of
  categories between local (super)algebra modules and extended vertex
  operator (super)algebra modules is also an isomorphism of braided
  monoidal (super)categories. Using this result, we show that
  induction from a suitable subcategory of $V$-modules to $A$-modules
  is a vertex tensor functor. Two applications are given:
    
  First, we derive Verlinde formulae for regular vertex operator superalgebras and 
  regular $\frac{1}{2}\ZZ$-graded vertex operator algebras by realizing them as (super)algebra objects 
  in the vertex tensor categories of their even and $\ZZ$-graded components, respectively.   
  
  Second, we analyze parafermionic cosets $C=\mathrm{Com}(V_L,V)$
  where $L$ is a positive definite even lattice and $V$ is regular.
  If the vertex tensor category of either $V$-modules or $C$-modules 
  is understood, then our results classify all inequivalent simple modules 
  for the other algebra and determine their fusion rules and modular character 
  transformations. We illustrate both directions with 
  several examples.
  
  \bigskip
  
  \noindent\textbf{MSC 2020:} Primary 17B69, 18M15, 18M20, 81R10, 81T40

}

\setcounter{tocdepth}{2}
\setcounter{secnumdepth}{4}

\newpage

\tableofcontents

\section{Introduction}

Vertex operator algebras can be viewed as an attempt to mathematically
formalize the symmetry algebra of a two-dimensional conformal quantum
field theory. In addition to physics, \VOAs{} have by now
intimate connections to many areas of mathematics, including affine Lie
algebras, modular tensor categories, geometry, and modular forms. In
this paper, we use tensor categories to deepen our understanding of the representation theory of vertex operator  algebras, and more generally, \VOSAs.

The tensor category approach to conformal quantum field theory began
in the late 1980s with Verlinde's observation that fusion rules of
rational conformal field theory are connected to modular properties of
characters \cite{V}. This was soon followed by the work of Moore and
Seiberg \cite{MS}, who realized that the representation theory of
conformal field theory has certain features that Turaev \cite{T} was
then able to formalize through introducing the concept of modular
tensor category. Two important questions were then when a category of
\VOA{} modules is a modular tensor category, and whether the modular
$S$-matrix defined by the transformation properties of torus one-point
functions for the vertex operator algebra agrees (up to normalization)
with the categorical $S$-matrix of a modular tensor category given by
Hopf links. Both questions were resolved by Huang \cite{H-verlinde,
  H-rigidity}: in particular, the module category of a so-called
regular \VOA{} is a modular tensor category, and the modular and
categorical $S$-matrices indeed agree up to normalization. Huang's
work depends not merely on the braided tensor category structure of
vertex operator algebra module categories, but also on the deeper
\textit{vertex tensor category} structure formulated and developed in
\cite{HL-VTC, HL-tensor1, HL-tensor2, HL-tensor3, H-tensor4,
  HL-tensorAffine}. Recently, Huang, Lepowsky, and Zhang have
developed a vertex tensor category theory for more general vertex
operator algebras: the tensor categories
constructed in \cite{HLZ1}-\cite{HLZ8} are no longer necessarily semisimple or
finite, but they are braided and have a twist. It is not known
generally under what conditions they are rigid.

Our setting is thus a \VOA{} $V$ with a vertex tensor category $\cC$ of
modules that is braided and has a twist. Let $A\supseteq V$ be a
$V$-module that itself carries the structure of a vertex operator
(super)algebra. The aim is then to use the tensor category structure on $\cC$ to understand its relation with the representation theory of $A$ (see \cite{Hohn} for an early realization that tensor categories can be used to study vertex operator algebra extensions). Our starting point is the
work of Pareigis \cite{Pa} and Kirillov-Ostrik \cite{KO} on
commutative algebra objects inside abstract modular tensor categories,
in particular their construction of a braided tensor category $\repzA$
of ``local modules'' (called ``dyslectic modules'' in \cite{Pa}) 
for an algebra $A$ in $\cC$. A
fundamental theorem of Huang, Kirillov, and Lepowsky \cite{HKL} states that
\VOA{} extensions of $V$ in $\cC$ are precisely commutative algebras
$A$ in $\cC$ with injective unit and trivial twist (the generalization to superalgebra
extensions was obtained in \cite{CKL}). It is also shown in \cite{HKL}
that the module category for the extension algebra $A$
agrees with $\rep^0 A$ as categories (see \cite{CKL} for the
superalgebra case).  Our main theoretical contribution is to show that
in fact $\repzA$ agrees with the category of $A$-modules in $\cC$ as
braided tensor categories, where the category of $A$-modules is given
the braided tensor category structure of \cite{HLZ1}-\cite{HLZ8}. This
means that all abstract tensor-categorical
constructions in \cite{KO} apply naturally in the vertex
algebraic setting.

The main tensor-categorical tool from \cite{KO} that we use in
applications to specific vertex operator (super)algebra extensions is the
induction functor. For any object $W$ in the category $\cC$
of $V$-modules, the tensor product $A\boxtimes W$ in $\cC$
naturally has the structure of a possibly non-local $A$-module. By restricting to the subcategory $\cC^0$ consisting of objects $W$ such that $A\boxtimes W$ is local, we get
an induction tensor functor from $\cC^0$ to
$\repzA$. Induced modules have been constructed previously in the context of vertex operator algebra extensions $V^G\subseteq V$ where $V^G$ is the fixed-point subalgebra of a finite group $G$ of automorphisms of $V$ \cite{DLin, Ya2, CarM}, but so far, there has not been systematic use of induction for general vertex operator (super)algebra extensions. In this work, we shall show that in a suitable sense, induction is a
vertex tensor functor, which allows one to relate tensor-categorical
data for $V$-modules to tensor-categorical data for $A$-modules. 

In
the following subsections, we discuss in more detail the vertex tensor
category theory needed for applying \cite{KO} to vertex operator
algebras, the types of vertex operator superalgebra extensions we would like to
consider, and the results we obtain in this paper.

\subsection{Vertex tensor categories}
\label{sec:intro-vtc}

Module categories for vertex operator (super)algebras that are vertex
tensor categories have a rich complex analytic structure that can be
specialized to braided tensor category structure. However, the full
analytic structure is crucial for our investigation, as we explain
here and later. For excellent background on vertex tensor categories,
see the survey article \cite{HL-rev}, and for more details, see
\cite{HL-VTC}.

We recall the main features in the definition of a \VOA{}. A vertex
operator algebra is a (typically infinite-dimensional) graded vector
space $V$ with two distinguished vectors: the vacuum $\mathbf{1}$ and
the conformal vector $\omega$, and equipped with a ``multiplication''
map, the vertex operator
$Y(\cdot, x )\cdot: V\otimes V \rightarrow V((x))$, where $x$ is a
formal variable and $V((x))$ is the space of Laurent series with
coefficients in $V$. 
The grading on $V$ is typically given by the semisimple action of the operator
$L(0) = \mathrm{Coeff}_{x^{-2}}Y(\omega,x)$.
One may specialize the variable $x$ in $Y(\cdot,x)\cdot$ to a non-zero complex number $z$ 
thereby yielding an element of the
algebraic completion $\overline{V}$ of $V$ ($\overline{V}$ is the direct product, as opposed to direct sum, of the graded
homogeneous subspaces of $V$, or equivalently is the full dual of the
graded dual of $V$).  The deepest axiom that $Y(\cdot,x)\cdot$
satisfies is the so-called Jacobi identity, which can be equivalently
formulated (in the presence of the other axioms) in terms of complex
variables as follows: the three series
\begin{align}
          \langle v',Y(v_1,z_1)Y(v_2,z_2)v_3 \rangle
          &\quad\quad |z_1|>|z_2|>0,
	\label{eqn:VOAv1v2}\\
          \langle v',Y(Y(v_1,z_1-z_2)v_2,z_2)v_3 \rangle
          &\quad\quad |z_2|>|z_1-z_2|>0, 
          \label{eqn:VOAv1v2assoc}\\
          \langle v',Y(v_2,z_2)Y(v_1,z_1)v_3 \rangle
          &\quad\quad |z_2|>|z_1|>0,
	\label{eqn:VOAv2v1}
\end{align}
where $v_i\in V$ and $v'\in V'$, where $V'$ is the graded dual of $V$,
converge absolutely in the
indicated domains to a common rational function
in $z_1, z_2$. The equality of the first and second
functions above is a complex-analytic version of associativity for
$Y$, while the equality of the first and third corresponds to
commutativity of left multiplication operators. The vertex operator
$Y$ also satisfies a more direct analogue of commutativity, called
skew-symmetry:
\begin{align}
e^{xL(-1)}Y(v_2,-x)v_1=Y(v_1,x)v_2
\label{eqn:voa-skewsym}
\end{align}
for $v_1, v_2\in V$. The domains of \eqref{eqn:VOAv1v2},
\eqref{eqn:VOAv1v2assoc}, and \eqref{eqn:VOAv2v1} and the
corresponding equalities of rational functions are governed by certain
sewing constraints for Riemann spheres with positively and negatively
oriented punctures, thereby making a \VOA{} an algebra over the
(partial) operad given by the moduli space of such spheres
\cite{H-book}. In this sense, vertex operator algebras are a ``complexification'' of the notion of commutative
associative algebra.

Given modules $W_i$ for a \VOA{} $V$, an intertwining operator
$\cY: W_1\otimes W_2\rightarrow W_3\{x\}$ (here braces denote
arbitrary complex powers of $x$) is again defined using a Jacobi
identity \cite{FHL}.  For \VOAs{} in logarithmic conformal field theory, it
becomes necessary to consider logarithmic intertwining operators
\cite{Mil1, Mil2}, whose target space is $W_3[\log x]\{x\}$.
One can choose a branch of logarithm and specialize complex powers of
$x$ and integer powers of $\log x$ in a logarithmic intertwining
operator to corresponding powers of a specific complex number
$z\in\CC^\times$ and its logarithm to obtain $P(z)$-intertwining maps
\cite{HLZ3}.  This, however, is not a definition of intertwining maps,
which are instead defined via a Jacobi identity involving
$z$. Intertwining operators and maps are central to the definition of
tensor product of vertex operator algebra modules.

The equalities of the analytic functions in \eqref{eqn:VOAv1v2},
\eqref{eqn:VOAv1v2assoc}, and \eqref{eqn:VOAv2v1} when the vertex
operators $Y$ in each series are replaced with general (and possibly
different) intertwining operators between appropriate modules is now
very subtle. In this generality, the series no longer converge to
rational functions; in fact, they should converge to multivalued
analytic functions due to the presence of non-integral powers of $z$
(and perhaps logarithms as well). The associativity relation for
general intertwining operators is called the operator product
expansion in physics and is not known to exist for general categories
of modules for arbitrary \VOA{}s.  Importantly, the operator product
expansion is required in many scenarios, the earliest instance being
the proof of Frenkel, Lepowsky, and Meurman that the moonshine module
$V^\natural$ is a \VOA{} \cite{FLM}. In vertex tensor category theory,
the existence of the operator product expansion is equivalent to the
existence of associativity isomorphisms between triple tensor products
of modules (see \cite[Theorem 14.11]{H-tensor4} and \cite[Theorems 10.3 and 10.6]{HLZ6}).

As with the vertex operator of a vertex operator algebra, the required
commutativity and associativity properties for multivalued analytic
functions defined by compositions of intertwining operators are
governed by sewing constraints for Riemann spheres with punctures. The
notion of vertex tensor category is then designed precisely to reflect
these constraints.  Given modules $W_1$ and $W_2$ for a vertex
operator algebra, one should have tensor products $W_1\fus{P(z)}W_2$
parametrized by elements $P(z)$ of the moduli space of Riemann spheres
with punctures at $\infty$, $z\neq 0$, and $0$, and with specified
local coordinates at the punctures.  The $P(z)$-tensor product,
\emph{if it exists}, is a module $W_1\fus{P(z)} W_2$ equipped with a
$P(z)$-intertwining map
$\fus{P(z)}: W_1\otimes W_2\rightarrow \overline{W_1\boxtimes_{P(z)}
  W_2}$ that satisfies the following universal property: Given any
test module $W$ and $P(z)$-intertwining map
$I:W_1\otimes W_2\rightarrow \overline{W}$, there exists a unique
morphism $\eta:W_1\fus{P(z)}W_2\rightarrow W$ such that the following
diagram commutes:
\begin{align*}
\xymatrixcolsep{4pc}
\xymatrix{
W_1\otimes W_2 \ar[r]^{I} \ar[d]^{\boxtimes_{P(z)}} & \overline{W} \\
\overline{W_1\fus{P(z)} W_2} \ar[ur]_{\overline{\eta}} & \\
}
\end{align*}
where $\overline{\eta}$ is the natural extension of $\eta$ to
algebraic completions. The vertex operator algebra is a unit with
respect to each $P(z)$-tensor product. Associativity of tensor
products follows from existence of operator product expansions for
compositions of $P(z)$-tensor product intertwining maps, and braiding
comes from skew-symmetry for intertwining operators.  The module
homomorphism $e^{2\pi i L(0)}$ yields a twist on a vertex tensor
category.  For a quick yet detailed review of these concepts, see
\cref{subsec:VTC}.  When the varying $z\in\CC^\times$ are carefully
specialized to a single fixed complex number (typically, $z=1$) one
obtains a braided monoidal category with a twist. However, this
procedure necessarily loses information, analogous to forgetting the
complex structure on a Riemann surface and remembering only the
topological structure.  The specialization procedure is necessarily
subtle: since pairs $(z_1, z_2)$ with $z_1=z_2$ do not belong to any
of the domains in \eqref{eqn:VOAv1v2}, \eqref{eqn:VOAv1v2assoc},
\eqref{eqn:VOAv2v1}, the multivalued analytic functions arising from
compositions of tensor product intertwining maps typically have
singularities at these points.

As mentioned above, the notion of vertex tensor category incorporating
the considerations above was first formulated by Huang and Lepowsky in
\cite{HL-VTC}. In \cite{HL-tensor1,HL-tensor2,HL-tensor3, H-tensor4}
they showed how to endow the module categories of ``rational'' vertex
operator algebras with such structure. Their results embrace many
well-known examples, for example \VOAs{} based on affine Lie algebras
at positive integral levels \cite{HL-tensorAffine}, Virasoro minimal
models \cite{H-Virasoro}, lattice vertex operator algebras, and the
moonshine module $V^\natural$. Using vertex tensor category structure,
Huang solved the two major important problems we have mentioned above:
determination of the modularity of these categories and the Verlinde
conjecture.

Going beyond rationality, vertex tensor categories for
\VOAs{} in logarithmic conformal field theory have been constructed in
\cite{HLZ1}-\cite{HLZ8}, subject to certain conditions, in particular
existence of the logarithmic operator product expansion.  In this
case, the action of $L(0)$ on modules need not be semisimple, and
correspondingly, intertwining operators may have logarithmic
terms. Huang has shown in \cite{H-cofin} that the full category of
(grading-restricted, generalized) modules for a $C_2$-cofinite \VOA{}
satisfies the conditions of \cite{HLZ1}-\cite{HLZ8} for the existence
of vertex tensor category structure (actually, the conditions on the
vertex operator algebra needed in \cite{H-cofin} are somewhat more
relaxed than $C_2$-cofiniteness, but this is sufficient).  However,
there are many interesting \VOAs{} for which it is unknown
whether natural representation categories can be given vertex tensor
structure. Examples include the $C_1$-cofinite module categories for
singlet vertex operator algebras (see Section 6 of \cite{CMR}) and module
categories for affine Lie algebras at admissible levels. (Since the first version of this paper was written, there has been progress on these two examples. For singlet algebras, rigid vertex tensor structure is now known on the subcategory of so-called ``atypical'' modules \cite{CMY}, though not yet on the full $C_1$-cofinite module category. For affine Lie algebras at admissible levels, vertex tensor structure is now known on the $C_1$-cofinite module category \cite{CHY}, and this category is rigid at least for simply-laced types \cite{Cr2}. There are also results for some affine vertex tensor categories at non-admissible levels \cite{CY}. However, it is still an open problem to construct vertex tensor categories that include non-$C_1$-cofinite modules for affine Lie algebras, such as relaxed highest-weight modules and their images under spectral flow.)

The concept of vertex tensor category applies equally
well to vertex operator superalgebras. Vertex tensor
categories for superalgebras include minimal models for the
$N=1$ superconformal algebra and unitary minimal models for the $N=2$
superconformal algebra \cite{HM1, HM2}, as well as a non-semisimple example coming from the affine vertex operator 
superalgebra of $\mathfrak{gl}(1|1)$ \cite{CMY2}.

\subsection{Motivation}

The first motivation for our study of vertex operator algebra
extensions concerns the invariant theory of \VOAs{}, in which two
types of extensions appear. First, if $V$ is a \VOA{} and $G$ is a
group of automorphisms of $V$, then we have an extension
$V^G\subseteq V$, where $V^G$ is the vertex operator subalgebra of
$G$-fixed points. Secondly, if $B$ is a vertex operator
subalgebra of $V$ (with generally different conformal vector), then the vertex operator subalgebra of $V$ that
commutes with $B$, $C=\text{Com}\left(B, V\right)$, is called the
$B$-coset \VOA{} of $V$, and $V$ is an extension of $B\otimes C$.

In many applications, we understand $V$ and $B$ fairly well, but we
would like to understand the representation theory
of either the orbifold $V^G$ or the coset $C$. These are called
the orbifold and coset problems, respectively. Conversely, the inverse orbifold
(respectively, inverse coset) problem occurs when we understand $V^G$
(respectively, $B\otimes C$) reasonably well, but we would like to
understand the representation theory of $V$.  The theory of our work
exactly applies to the inverse orbifold and coset problems because our
starting assumption is an extension $V\subseteq A$ where the smaller
algebra $V$ already has a braided tensor category of
modules. Interestingly, we can also handle the coset problem
efficiently, and we illustrate this in Section \ref{sec:coset} for
parafermionic cosets, where the vertex operator subalgebra
$B$ is a lattice \VOA. In this setting, when $V$ is in particular an affine vertex operator algebra at positive integral level, the coset problem
has been analyzed previously in \cite{ADJR, ALY, DLY, DLWY,
  DR, DW1, DW2, DW3}.  A better understanding of the orbifold problem
requires the study of $G$-crossed tensor categories \cite{T-Gcross,  Ga, Mu1,
  Mu2, K, K2} in the context of vertex operator algebras; we plan to
contribute to this problem in the near future (since the first version of this paper was written, one of us has completed an extensive study of the orbifold problem in \cite{McR}, using the theory developed in this work).

The examples of inverse coset problems that we have in mind and which
can be understood better with the aid of our results are certain
affine Lie superalgebra vertex operator superalgebras at integer and
also non-integer admissible levels. Very little is known about the
representation theory of affine vertex operator superalgebras,
although there are character formulae for certain atypical modules in
terms of mock Jacobi forms \cite{KW2, KW3, KW4} and there is a fairly
good understanding of $V_k\left(\mathfrak{gl}(1|1)\right)$ \cite{CR1}.
Let us announce here that various affine vertex operator superalgebras
are extensions of known \VOAs. Unlike for parafermionic cosets, these
superalgebra extensions are not of simple current type. Nonetheless,
in examples like the affine vertex operator superalgebra of
$\mathfrak{osp}(1|2)$, the representation category can be completely
described in terms of the underlying vertex operator subalgebra
\cite{CFK}. That work is a nice interplay between the use of tensor
categories (including our results) and Jacobi form computations. The key steps are: first, a
decomposition of characters that determines the decomposition of
$L_k\left(\mathfrak{osp}(1|2) \right)$ into
$L_k(\mathfrak{sl}_2)\otimes \text{Vir}$-modules, where Vir is
the simple Virasoro \VOA{} of appropriate central charge. This
determines the superalgebra object in the modular tensor category of
$L_k(\mathfrak{sl}_2)\otimes \text{Vir}$-modules corresponding to
$L_k\left(\mathfrak{osp}(1|2) \right)$.  The second step is to use the
induction functor to find as many simple local modules as possible, and
the last step is to use the modular $S$-matrix together with the
Frobenius-Perron theorem of \cite{DMNO} to conclude that one has found
all inequivalent simple modules.  The induction tensor functor then
also determines the fusion rules of both local and twisted
$L_k\left(\mathfrak{osp}(1|2) \right)$-modules.  A few additional
subtleties arise because $L_k\left(\mathfrak{osp}(1|2) \right)$ is a
vertex operator superalgebra; to our knowledge, this is the first
description of a series of $\ZZ$-graded rational vertex operator
superalgebras.  We are optimistic that some affine vertex operator
superalgebras of type $\mathfrak{sl}(n|1)$ and $D(2, 1; \alpha)$ can
also be studied using our methods, and this is under
investigation. The latter \VOSA{} plays a crucial role in the quantum
geometric Langlands correspondence of $\mathfrak{sl}_2$ as well as in
four-dimensional supersymmetric quantum field theories \cite{Gai}.

The second motivation for our work here is the Verlinde formula
\cite{V}, which is central to \VOA{} theory since it relates
analytic data in the form of modular properties of torus one-point
functions to categorical data in the form of representations of the
fusion ring. This formula has been proven for module categories of
regular \VOAs{} by Huang \cite{H-verlinde} and it is important to
extend these results beyond rationality, for example to $C_2$-cofinite
but non-rational \VOAs{} (see \cite{CG} for the state of the art) or
even beyond finite tensor categories; see \cite{AC, CM, CR2, CR3, CR4}
for examples. A nice application of our main results are Verlinde
formulae for regular vertex operator superalgebras and
$\frac{1}{2}\ZZ$-graded regular \VOAs; see Sections \ref{sec:ver} and
\ref{sec:Verlinde}.  Understanding the Verlinde formula amounts to
relating the complex analytic properties of intertwining operators and
their correlation functions to categorical and topological field
theory data. Correlation functions are solutions to sewing
constraints, and such solutions are uniquely specified by a simple
special symmetric Frobenius algebra in the representation category of
the chiral algebra of the conformal field theory. A related motivation or future goal is thus to connect
the \VOA{} approach to conformal field theory and its correlation
functions (see \cite{HK1,HK2,HK3,Kon} for important developments) with the topological field theory approach of Fuchs,
Fjelstad, Runkel and Schweigert (\cite{FRS1, FRS2, FFRS1, FFRS2}; see
\cite{FS} for a review).

Conformal nets are another approach to conformal field theory using
tensor categories. There has been progress in relating this area to
the \VOA{} point of view \cite{Kaw, CKLW, BLKR}. Each approach has its
advantages and \VOA{} theory together with tensor categories might
succeed in proving important statements that are known from the
conformal net standpoint. A nice example is the fairly recent solution to
the mirror extension problem for coset \VOAs{} \cite{OS, Lin, CKM}.  Another
useful tool in the conformal net approach is $\alpha$-induction
\cite{LR, X, BE1, BE2, BEK, Ost}, a braided functor from a modular
tensor category $\cC$ to the category of module endofunctors of a
certain module category over $\cC$. This functor is used to construct
modular invariants for conformal field theory and is in general
important in understanding the categorical part of bulk and boundary
conformal field theory (see for example \cite[Section
3]{FRSalpha}). We now have a good induction functor for \VOAs{}, and
it is surely a valuable continuation of our work to connect \VOAs{} to
$\alpha$-induction.
  
  Beyond vertex operator algebras and conformal nets, it is well known that quantum groups are rich sources of braided tensor categories. Since the results we prove in Section \ref{sec:TensCats} below hold for (super)algebras in essentially any braided tensor category, they also apply to quantum groups. For example, after the first version of this paper appeared, and inspired by our work, the following problem was solved. It was conjectured in \cite{FGST1} that the representation categories of triplet vertex operator algebras and of restricted quantum groups of $\mathfrak{sl}_2$ at corresponding roots of unity are braided tensor equivalent. However, it turns out that the latter categories are not braidable \cite{KS}, so the next question is whether the quantum group category can be modified to a braided tensor category. 
Indeed, in \cite{CGR}, a quasi Hopf algebra was constructed that is isomorphic to the quantum group as an algebra but has modified comultiplication and associator; its representation category is braidable and now conjecturally equivalent to that of the triplet vertex operator algebra. This quasi Hopf algebra was constructed through careful study of the category of local modules for a certain algebra object in the category of weight modules for the corresponding unrolled restricted quantum group of $\mathfrak{sl}_2$. This procedure, using (super)algebras in braided tensor categories, has now become a powerful method for constructing non-degenerate finite braided tensor categories and supercategories \cite{GLO, N, CRu}.

%

\subsection{Results on vertex operator (super)algebra extensions and their modules}

With the above motivation in mind we recall the known results
on vertex operator algebra extensions and present our main theorem.
We always work in the following setting:
\begin{assum}
  $V$ is a \VOA{} and $\cC$ is a category of $V$-modules with a
  natural vertex tensor category structure.
\end{assum} 
If $V$ is regular, its representation category $\cC$ is modular.
There are however
interesting \VOAs{}, such as the triplet algebras $\cW(p)$ \cite{TW1},
whose representations form braided tensor categories, but
not all modules are completely reducible. The foundation of our work
is the following key result of \cite[Theorems 3.2 and 3.4]{HKL} (see
also \cite[Theorems 3.13 and 3.14]{CKL} for the natural superalgebra
generalization):
\begin{theo}
  A \VOA{} extension $V\subseteq A$ in $\cC$ is equivalent to a
  commutative associative algebra in the braided tensor category $\cC$
  with trivial twist and injective unit.
  Moreover, the category of
  local $\cC$-algebra modules $\repzA$ is isomorphic to the category
  of modules in $\cC$ for the extended \VOA{} $A$.
\end{theo}
This powerful result allows one to study \VOA{} extensions
using abstract tensor category theory. For example, Lin has used it to
solve the mirror extension problem in coset \VOA s \cite{Lin} (see also \cite{CKM}). 
In order to study representations of the extended \VOA{}, one needs the monoidal structure on the category of local $\cC$-algebra
modules $\repzA$. Kirillov and Ostrik have given $\repzA$ a braided
tensor category structure \cite{KO}, but several problems need
resolution before this can be applied in the vertex operator
algebra context. First, we need to work in a more general setting than
\cite{KO}: we do not assume $\cC$ is rigid, semisimple, or strict (in
fact, vertex tensor categories are highly non-strict due to the
subtlety of the associativity isomorphisms). Secondly, we would
like to consider superalgebra extensions of vertex
operator algebras, and therefore we need superalgebra analogues of all
relevant constructions and results in \cite{KO}. 

To handle superalgebra extensions of $V$, we first introduce an auxiliary supercategory $\sC$ which amounts to the category $\cC$ of $V$-modules but with $V$ viewed as a vertex operator superalgebra having zero odd component. In other words, $V$-modules in $\sC$ have parity decompositions which may or may not be trivial. As categories, $\cC$ and $\sC$ are equivalent, but $\sC$ has additional supercategory structure, that is, morphisms can be even or odd according as they preserve or reverse parity decompositions. We define a superalgebra $A$ in $\cC$ to be a commutative associative algebra with even structure morphisms in the braided tensor supercategory $\sC$, and we define the categories $\repA$ and $\repzA$ of $A$-modules and local $A$-modules in $\cC$ to be those $A$-modules in $\sC$ having even structure morphisms. Assuming only that
$\cC$ is abelian and that the tensor product on $\cC$ is right exact,
we then fix a superalgebra $A$ in $\cC$ 
and prove:
\begin{itemize}
 \item $\repzA$ is an additive braided monoidal supercategory (Theorem
  \ref{thm:rep0}, corresponding to \cite[Theorem 1.10]{KO}).
 
 \item The induction functor from $\cC$ (or $\sC$) to the category $\repA$ of not
   necessarily local $A$-modules is a tensor functor (Theorem
   \ref{thm:inductionfucntor}, corresponding to \cite[Theorem
   1.6]{KO}).
\end{itemize}

More importantly, before applying the constructions and results of
\cite{KO} to vertex operator (super)algebras, we need to make sure
that the braided monoidal (super)category structure on $\repzA$ is the
correct structure in the vertex-algebraic sense, that is, we need to make sure
that it agrees with the vertex tensor category constructions of
\cite{HLZ1}-\cite{HLZ8}. We accomplish this in Section \ref{sec:vtc},
culminating in our main result, Theorem \ref{thm:repzABTCvrtxSalg}:
\begin{theo} 
  Suppose $A$ is a (super)algebra extension of $V$ in $\cC$ such that
  $V$ is contained in the even part of $A$. Then the isomorphism of
  categories given in \cite[Theorem 3.4]{HKL} between $\repzA$ and the category of vertex-algebraic
  $A$-modules in $\cC$ is an isomorphism of braided
  monoidal (super)categories.
\end{theo}
This theorem is a considerable strengthening of \cite[Theorem
3.4]{HKL} and has numerous useful consequences. Perhaps the first is
that if $A$ is a vertex operator (super)algebra extension of $\cC$,
then the category of $A$-modules in $\cC$ actually has vertex tensor
category structure given by the constructions of
\cite{HLZ1}-\cite{HLZ8} (with the corresponding braided monoidal
(super)category structure as in \cite{KO}). Secondly, combining this
theorem with Theorems \ref{thm:inductionfucntor} and
\ref{thm:Fisbraidedtensor}, we prove in \cref{thm:Fisvertextensor}:
\begin{theo}
Induction is a vertex tensor functor on the (not necessarily abelian)
full subcategory of $V$-modules inducing to $A$-modules, 
with respect to the vertex tensor category structures on these categories constructed in \cite{HLZ8}.
\end{theo}

This theorem means that induction respects all vertex
tensor category structure, in particular $P(z)$-tensor products for
each $z\in\CC^\times$, unit, associativity, and braiding
isomorphisms. Thus induction can be applied in a vertex-algebraic
setting: some applications are given in Section
\ref{sec:applications}. In particular, we find Verlinde formulae for
\VOSAs{} and $\frac{1}{2}\ZZ$-graded \VOAs{}, and we study cosets of
rational vertex operator algebras by lattice \VOAs{} in
generality. The proof of \cref{thm:repzABTCvrtxSalg} occupies much of this paper, and we now briefly sketch the main steps:
\begin{enumerate}
\item First, after discussing the notion of a commutative associative
  algebra $A$ in $\cC$ and its representation category $\repA$, as
  defined in \cite{KO, Pa}, we formulate superalgebras in $\cC$ as algebras
  (with even structure morphisms) in the auxiliary supercategory $\sC$
  (see for example \cite{BE} for the formulation of supercategory and
  tensor supercategory that we use). The objects of $\sC$ are
  essentially direct sums of two objects (even and odd) in $\cC$, and
  the morphisms of $\sC$ correspondingly have superspace
  decompositions into parity-preserving even morphisms and
  parity-reversing odd morphisms. Sign factors necessary for the
  supercommutativity of $A$ are built into the braiding in $\sC$.

\item Next, we prove, under minimal assumptions on $\cC$ (namely,
  $\cC$ is abelian with right exact tensor functor) that $\repA$ with
  its tensor product $\boxtimes_A$ is an additive monoidal
  (super)category; in the superalgebra setting, $\repA$ might not be
  abelian since non-homogeneous morphisms might not have kernels and
  cokernels. We provide details of this proof in Section
  \ref{sec:TensCats}.  Actually, the assumption that $\cC$ is abelian is not strictly necessary: we only need certain morphisms to
  have cokernels.

\item In proving the associativity of $\boxtimes_A$, we
  introduce \emph{categorical $\repA$-intertwining operators} among a
  triple of objects in $\repA$ that intertwine the action of $A$ on
  each of these three objects.  We then show that $\boxtimes_A$
  satisfies a universal property in terms of categorical
  intertwining operators in Section \ref{subsec:catintwop}, in analogy
  with $P(z)$-tensor products in vertex tensor categories. We further
  prove an associativity result for categorical intertwining
  operators in Section \ref{subsec:associnrepA}, which easily leads to
  the associativity of $\boxtimes_A$.
  
\item Once the proof that $\repA$ is a monoidal (super)category is
  complete, we prove that the subcategory $\repzA$ of local
  $A$-modules is braided.

\item In Section \ref{sec:vtc}, we begin to ``complexify'' these
  results when $\cC$ is a module category for a vertex operator algebra
  $V$ with vertex tensor category structure. From the definition of $P(1)$-tensor products, every categorical $\repA$-intertwining operator
  $\eta:W_1\boxtimes W_2\rightarrow W_3$, where the $W_i$ are objects of
  $\repA$, corresponds to an
  intertwining operator of $V$-modules, say $\cY_{\eta}$, of type
  $\binom{W_3}{W_1,\, W_2}$ such that
  $\overline{\eta}(w_1\fus{P(1)} w_2) = \cY_{\eta}(w_1,e^{\log
    1})w_2$.\footnote{The perhaps strange notation
    $x\mapsto e^{l_p(z)}$ indicates the substitutions
    $x^h \mapsto e^{h\, l_p(z)}$ and $(\log x)^n \mapsto (l_p(z))^n$
    for $h\in\CC$ and $n\in\NN$, where $l_p$ for $p\in\ZZ$ denotes the
    $p$th branch of logarithm, that is,
    $2\pi i p \leq \im\,l_p(z)<2\pi i (p+1)$. The principal branch
    corresponding to $p=0$ we simply denote by $\log z$.} In particular, the multiplication maps
  $\mu_W:A\boxtimes W\rightarrow W$ for $W$ in $\repA$ correspond to intertwining operators denoted $Y_W$.
%
%

\item In Section \ref{subsec:complexIntw} we show that $\eta$ is a
  categorical $\repA$-intertwining operator if and only if $\cY_{\eta}$ satisfies
  a certain complex analytic associativity condition (analogue of the equality of
  \eqref{eqn:VOAv1v2} and \eqref{eqn:VOAv1v2assoc}) and a
  skew-associativity condition (analogue of the equality of \eqref{eqn:VOAv1v2assoc}
  and \eqref{eqn:VOAv2v1}) with the module maps $Y_{W_i}$. This is one
  of the technical sections of the paper, where some careful use of complex analysis is required (some results we collect in Section
  \ref{subsec:technical}), similar to \cite{HKL} and
  \cite{HLZ8}.

\item In Theorems \ref{thm:complextoformal} and
  \ref{repzAintwopcorrect} we show that for modules in $\repzA$, the
  associativity and skew-associativity of a $\repA$-intertwining
  operator is equivalent to the Jacobi identity for vertex-algebraic
  intertwining operators among $A$-modules.  Then the universal
  properties of $\boxtimes_A$ and the vertex-algebraic tensor product
  $\boxtimes_{P(1)}$ for $A$-modules quickly imply that these two
  tensor products are equivalent.

\item Finally, in Section \ref{subsec:VTConRepARep0A} we show that,
  with the identification of $\boxtimes_A$ with $\fus{P(1)}$, the
  braided monoidal (super)category structure on $\repzA$
  (in particular, unit, associativity, and braiding isomorphisms)
  agrees with the (vertex-algebraic) braided tensor category constructions for
  $A$-modules in \cite{HLZ8}. This completes the proof of Theorem
  \ref{thm:repzABTCvrtxSalg}. We in fact show that $\repzA$ has a vertex tensor category structure. 
  The category $\repA$ has some (not all!) features of a vertex tensor category.
\end{enumerate}	

In Section \ref{sec:TensCats} we also prove properties
of the induction functor $\cF: \sC\rightarrow\repA$ that will be
needed for applications. In Section \ref{subsec:inductionfunctor}, we
prove that $\cF$ is a tensor functor and left adjoint to the restriction functor $\cG: \repA\rightarrow \sC$. Then in
Section \ref{subsec:miscresults}, we show that induction respects
additional structures on $\sC$, $\repA$, and $\repzA$, including duals
and, under suitable conditions, twists. We also give a necessary and sufficient condition for an
object of $\sC$ to induce to $\repzA$, and show that on the full
subcategory $\sC^0$ of objects that induce to $\repzA$, $\cF$ is a braided
tensor functor. Section \ref{subsec:miscresults} culminates in the
proof that, if $\cC$ (or $\sC$) is a ribbon category and $A$ has
trivial twist, then $\cF$ respects categorical $S$-matrices (traces of
monodromy isomorphisms) of objects that induce to $\repzA$.  We use
this result to deduce $S$-matrices for parafermions in Section
\ref{sec:applications}. We show that induction respects all vertex tensor category structure on $\sC^0$ and $\repzA$, that is, induction is a vertex tensor functor, in Section \ref{subsec:vrtxtensfunctor}.

\subsection{Tensor categories of regular \VOSAs}

It was shown in \cite{H-rigidity} that if $V$ is a ($\ZZ$-graded)
simple, CFT type, self-contragredient, rational, $C_2$-cofinite vertex
operator algebra, then the vertex tensor category of $V$-modules is a
modular tensor category. CFT type means that $V$ has no negative
conformal weight spaces and $V_0=\CC\unit$, and rational means that
$\NN$-gradable weak $V$-modules (also called admissible $V$-modules)
are completely reducible. By \cite{ABD}, in the presence of the other
assumptions, rationality and $C_2$-cofiniteness can be replaced by the
assumption that $V$ is regular in the sense of
\cite{DLM-regular}. Thus for simplicity we shall use the term regular
to refer to any vertex operator algebra satisfying the assumptions of
\cite{H-rigidity}.

There are three natural ways to generalize the notion of \VOA{}, and hence also of regular \VOA:
\begin{enumerate}
\item $V=\bigoplus\limits_{n\in \ZZ} V_n$ is a $\ZZ$-graded \VOSA,
\item $V=\bigoplus\limits_{n\in \frac{1}{2}\ZZ} V_n$ is a $\frac{1}{2}\ZZ$-graded \VOA, and
\item $V=\bigoplus\limits_{n\in \frac{1}{2}\ZZ} V_n$ is a $\frac{1}{2}\ZZ$-graded \VOSA{}.
\end{enumerate}
The axioms defining $\frac{1}{2}\ZZ$-graded vertex operator algebras are the same as those briefly discussed in Section \ref{sec:intro-vtc}, except that the eigenvalues of the semisimple action of $L(0)$ are allowed to be half integers. The notions of $\ZZ$- and $\frac{1}{2}\ZZ$-graded vertex operator superalgebras are obtained by introducing a parity decomposition and an appropriate sign factor in \eqref{eqn:VOAv2v1} (and also \eqref{eqn:voa-skewsym}). Concrete examples of all three types of (super)algebra are discussed in Section \ref{sec:applications}: $\ZZ$-graded vertex operator superalgebras include affine vertex superalgebras, while superconformal algebras are $\frac{1}{2}\ZZ$-graded superalgebras. The $\beta\gamma$-vertex algebra and the Bershadsky-Polyakov algebra \cite{Ber, Pol} are perhaps the best-known $\frac{1}{2}\ZZ$-graded vertex operator algebras. All three generalizations of \VOA{}  arise as $\cW$-superalgebras, that is, quantum Hamiltonian reductions of affine vertex superalgebras, which are important in connections to physics and geometry (see for example \cite{GR, CL3}).

The boson-fermion statistics of particle physics says that a boson has
integral spin and a fermion has half-integral spin. We thus call the
first two generalizations above a \VOA{}, respectively a \VOSA{}, of incorrect or
wrong statistics, while the third we call a \VOSA{} of correct
statistics if its even part is $V^\even=\bigoplus_{n\in\ZZ} V_n$ and
its odd part is $V^\odd=\bigoplus_{n\in\frac{1}{2}\ZZ} V_n$. If
$V$ falls into one of the three cases above, then $\mathrm{Aut}(V)$
contains a copy of $\ZZ/2\ZZ$, generated by the parity involution
$P_V=1_{V^\even}\oplus(-1_{V^\odd})$ in the superalgebra cases and
generated by the twist $\theta_V=e^{2\pi i L(0)}$ in the
$\frac{1}{2}\ZZ$-graded cases (these two automorphisms coincide for a
vertex operator superalgebra of correct statistics). Either way, we
have a decomposition into $\pm 1$-eigenspaces for the action of
$\ZZ/2\ZZ$ on $V$:
\[
V=V_+\oplus V_-.
\]
If the orbifold ($\ZZ$-graded) vertex operator subalgebra $V_+$ is
regular, then a small generalization of \cite[Theorem 4.2]{CarM} (see
\cite[Appendix A]{CKLR}) shows that $V_-$ is a self-dual simple
current for the orbifold \VOA{} $V_+$. In particular, the fusion
product of $V_-$ with itself is
\[
V_- \boxtimes_{V_+} V_-\cong V_+.
\]
The statistics of $V$ is reflected in the categorical dimension of
$V_-$. One gets a vertex operator (super)algebra of correct
boson-fermion statistics if $\qdim(V_-)=1$ and of incorrect type if
$\qdim(V_-)=-1$.

In each of the three cases the representation category of the vertex
operator (super)algebra $V$ is the category of local modules of a
corresponding (super)algebra object in the category $\cC$ of
$V_+$-modules \cite{CKL}. This has a few nice direct
consequences. First, it tells us something about the monoidal category
of local $V$-modules: since the twist in a vertex tensor category is
$e^{2\pi i L(0)}$ with $L(0)$ the Virasoro zero-mode giving the
conformal weight grading, the twist is not a $V$-module homomorphism
when $V$ is a $\frac{1}{2}\ZZ$-graded vertex operator
(super)algebra. On the other hand, $e^{2\pi i L(0)}$ still defines a
twist on the category of local $V$-modules in case 1. In any case,
braiding induces nicely and so local modules in case 2 form a braided
tensor category, while for cases 1 and 3 one needs a superbraiding, a
braiding with parity. This superbraiding is defined using the braiding
on objects in $\cC$, and it satisfies the hexagon axioms and an
appropriate ``supernaturality'' so that it defines the structure of a
braided monoidal supercategory on local $V$-modules. See Section
\ref{subsec:superalgebra} for more details on braided monoidal
supercategories.

In superconformal field theory both Neveu-Schwarz and Ramond modules
are of interest. The first correspond to local objects in the category
of superalgebra modules and the second to certain non-local modules
which we will call twisted (in fact, they are twisted modules for the
parity automorphism of $V$). We will see in the discussion of
Verlinde's formula that it is indeed necessary to also consider twisted
superalgebra modules.

As mentioned above, we will
give a few examples in Section \ref{sec:applications}, in particular a
fairly detailed discussion of rational $N=2$ superconformal algebras in
Section \ref{sec:N=2} and then short discussions of a
Bershadsky-Polyakov algebra in Section \ref{sec:BP} and a
rational $W$-superalgebra associated to
$\widehat{\mathfrak{sl}}(3|1)$ in Section \ref{sec:wrongsuper}.
For the detailed example of regular affine \VOSAs{} associated to
$\widehat{\mathfrak{osp}}(1|2)$, see \cite{CFK}.

\subsection{Verlinde formulae for regular \texorpdfstring{$\frac{1}{2}\ZZ$}{(1/2)Z}-graded \VOSA s}\label{sec:ver}

One nice application of our general results is a Verlinde formula for
the three cases of (super)algebra extensions discussed in the previous
subsection. The Verlinde formula, conjectured in \cite{V} and
proven in the regular \VOA{} setting by Huang \cite{H-verlinde}, is a true highlight of vertex operator algebra theory:
it says that fusion rules are given by a concrete expression in terms
of the entries of the modular $S$-matrix of characters, thus, for
example, providing an efficient way of computing fusion rules.  Since
the representation category for each case of regular vertex operator
(super)algebra discussed above is the category of local modules for a
corresponding (super)algebra object in the modular tensor cateogry of
$V_+$, we can use induction to deduce interesting consequences for
$V$-modules (as for example Verlinde's formula) from corresponding
properties for $V_+$-modules. We do so in Section \ref{sec:Verlinde},
and since this is quite technical, we here redescribe the results
purely in terms of modular and asymptotic properties of characters and
supercharacters of $V$.

\subsubsection{$\ZZ$-graded \VOSA s}
\allowdisplaybreaks

Suppose $V$ is a regular $\ZZ$-graded \VOSA. Then we view $V$ as an
order-$2$ simple current extension of its even vertex operator
subalgebra, which means the representation category of $V$ is the
representation category of the superalgebra extension
$V^\even\subseteq V$. Every simple $V^\even$-module induces to either
a local or twisted module in $\rep V$. By \cite[Theorem 3.4]{HKL} and
our Theorem \ref{thm:repzABTCvrtxSalg}, $\rep^0 V$ is the braided
monoidal supercategory of untwisted $V$-modules.  The subcategories of local
and twisted modules we denote by $\reploc$ and $\reptw$,
respectively. We call an object in $\rep V$ homogeneous if it is
either local or twisted. There are only finitely many inequivalent
simple objects in $\rep V$, and we denote by $[\rep V]$ the set of
equivalence classes for the relation
\[
  X \sim Y \qquad \text{if and only if} \ X \cong Y \ \text{as modules
    for} \ V^\even.
\]
That is, we consider objects that only differ by parity to be
equivalent.  Let $\CC[\rep V]$ be the complex vector space with basis
labeled by the elements of $[\rep V]$ .  Let $L(0)$ be the Virasoro
zero-mode of $V$ and $c$ the central charge; then the character and
supercharacter of objects are defined respectively as the graded trace
and supertrace:
\[
  \ch^+[X](\tau, v) := \text{tr}_X\left(q^{L(0)-\frac{c}{24}} o(v)\right) \qquad
  \text{and}\qquad \ch^-[X](\tau,v) := \text{str}_X\left(q^{L(0)-\frac{c}{24}}
    o(v)\right), \qquad q=e^{2\pi i\tau}.
\]
Here $v\in V^\even$ is even and $o(v)$ the zero-mode of the
corresponding vertex operator. This is needed to ensure linear
independence of characters of inequivalent modules. Characters of
equivalent modules coincide and supercharacters coincide up to a
possible sign.  The $\CC$-linear span of characters as well as
supercharacters is thus isomorphic to the vector space $\CC[\rep V]$
whose basis elements are labeled by the elements of $[\rep V]$.  We
fix one representative of each equivalence class and thus fix a choice
of parity for each class.  By abuse of notation, we denote this set of
representatives by $[\rep V]$ as well and its untwisted and twisted
subsets by $\cSloc$ and $\cStw$, respectively.  Then M\"{o}bius
transformations on the linear span of these characters and
supercharacters define an action of SL$(2,\ZZ)$. If $v\in V^\even$ has
conformal weight $\mathrm{wt}[v]$ with respect to Zhu's modified
vertex operator algebra structure on $V^\even$ \cite{Zhu}, the modular
$S$-transformation has the form
\begin{equation}
\begin{split}
  \ch^+[X]\left(-\frac{1}{\tau}, v\right)
  &= \tau^{\text{wt}[v]} \sum_{Y\in \cStw} S_{X, Y}  \ch^-[Y](\tau, v) \\
  \ch^-[X]\left(-\frac{1}{\tau}, v\right) &= \tau^{\text{wt}[v]}
  \sum_{Y\in \cSloc} S_{X, Y} \ch^-[Y](\tau, v)
\end{split}
\end{equation}
if $X$ is local and 
\begin{equation}
\begin{split}
  \ch^+[X]\left(-\frac{1}{\tau}, v\right)
  &= \tau^{\text{wt}[v]} \sum_{Y\in \cStw} S_{X, Y}  \ch^+[Y](\tau, v) \\
  \ch^-[X]\left(-\frac{1}{\tau}, v\right) &= \tau^{\text{wt}[v]}
  \sum_{Y\in \cSloc} S_{X, Y} \ch^+[Y](\tau, v)
\end{split}
\end{equation}
if $X$ is twisted.  The modular $T$-transformation acts diagonally on
characters and supercharacters of simple local modules and maps
characters of twisted modules to the corresponding supercharacters
and vice versa.  Especially, we see that the linear span of local
supercharacters already forms a vector-valued modular form for the
modular group.

Let $Z$ be a simple module; we call an intertwining operator of type
$\binom{Z}{X\, Y}$ for $X, Y$ in $[\rep V]$ even if $Z$ has the same
parity as its representative in $[\rep V]$ and it is called odd if it
has reversed parity. Obviously this choice of parity depends on our
choice of representatives $[\rep V]$. We define the fusion rules as
dimensions and superdimensions of spaces of intertwining
operators. Verlinde's formula for $\ZZ$-graded \VOSAs{} is then
\begin{equation}
  \begin{split}
    {N^+}_{X,Y}^{W}&=  \delta_{t(X)+t(Y), t(W)}\sum\limits_{Z \in \cStw}
    \dfrac{{S}_{X,Z}\cdot {S}_{Y,Z}\cdot ({S}^{-1})_{Z,W}}{{S}_{V, Z}}\\
    {N^-}_{X,Y}^{W}&=  \delta_{t(X)+t(Y), t(W)} \sum\limits_{Z \in \cSloc}
    \dfrac{{S}_{X,Z}\cdot {S}_{Y,Z}\cdot ({S}^{-1})_{Z,W}}{{S}_{V, Z}}
\end{split}
\end{equation}
with 
\[
  t :[\rep V]\rightarrow \ZZ/2\ZZ, \qquad X \mapsto \begin{cases} \even &
    \qquad \text{if} \ X \in \cSloc \\ \odd & \qquad \text{if} \ X \in
    \cStw\end{cases}.
\]
The fusion product is then
\[
  X \boxtimes_V Y \cong \bigoplus_{W\in[\rep V]} \left({N^\even}_{X,Y}^{W}W
    \oplus {N^\odd}_{X,Y}^{W} \Pi(W)\right), \qquad\qquad
  {N^\pm}_{X,Y}^{W} = {N^\even}_{X,Y}^{W} \pm {N^\odd}_{X,Y}^{W},
\]
where $\Pi$ is the parity reversal functor. 

Under certain conditions one can get information on the fusion rules
from asymptotic dimensions.  Suppose there is a unique simple
untwisted module $Z_{(+)}$ with lowest conformal dimension among all
untwisted simple modules, and suppose there is also a unique
simple twisted module $Z_{(-)}$ of lowest conformal dimension among
all simple twisted modules.  Assume that $\ch^-[Z_{(+)}]$ does not
vanish and that the lowest conformal weight spaces of $Z_{(\pm)}$ are
purely even.
For a simple $V$-module $X$, we define the asymptotic (super)dimension by
\begin{align}\label{eq:adimdef}
  \adim^\pm[X] = \lim_{iy\rightarrow i\infty}
  \dfrac{\ch^\pm[X](-1/iy)}{\ch^\pm[V](-1/iy)} = 
\dfrac{{S}_{X,Z_{(\mp)}}}{{S}_{V,Z_{(\mp)}}}
\end{align}
We find that for $X, Y\in[\rep V]$, the product of asymptotic
dimensions respects the fusion product in the sense that
\begin{align}\label{eq:Introadimprod1}
  \adim^\pm[X]\adim^{\pm}[Y]
  &= \sum_{W\in [\rep V]} {N^\pm}_{X, Y}^{W}\adim^\pm[W].
\end{align}
Some recent works such as \cite{DJX, DW3} have used the asymptotic
dimensions of certain ordinary rational \VOAs{} to derive results on
the representation categories of these \VOAs{}. 
Note that what we call asymptotic dimension, these papers call quantum dimension; we prefer to avoid this
terminology as it might lead to confusion with categorical dimension of
objects.

An example illustrating the
statements of this subsection is given in Section
\ref{sec:wrongsuper}.
Such Verlinde formulae were also conjectured in a non-rational setting, for the affine \VOSA{} of
$\mathfrak{gl}(1|1)$ \cite{CR1, CQS}; these conjectures have recently been confirmed by the fusion rule computations in \cite{CMY2}.

\subsubsection{$\frac{1}{2}\ZZ$-graded \VOAs}\label{sec:BPlikevoas}

Let $V$ be a regular $\frac{1}{2}\ZZ$-graded \VOA; examples include Bershadsky-Polyakov algebras at specific levels \cite{Ber, Pol, Ar1}. Then $V$ is an
extension of its $\ZZ$-graded subalgebra $V_+$, that
is,
\[
  V=\bigoplus_{n\in\frac{1}{2}\ZZ}V_n = V_+ \oplus V_-, \qquad\qquad
  V_+ = \bigoplus_{n\in\ZZ}V_n \qquad\text{and}\qquad V_- =
  \bigoplus_{n\in\ZZ+\frac{1}{2}}V_n.
\]
The vertex operator algebra automorphism $\pm 1$ on $V_\pm$
generates an action of $\ztwo$ on $V$ with $V_+$ the corresponding
orbifold vertex operator subalgebra. By \cite{CarM} (see also the appendix of \cite{CKLR}),
$V_-$ is then a self-dual simple current of $V_+$.  The representation
category of $V$ is as before the category of local modules for
$V$ viewed as an algebra in the category of $V_+$-modules, and every simple $V_+$-module induces to either a simple local
$V$-module or a simple twisted $V$-module. Both are naturally
$\ZZ/2\ZZ$-graded,
\[
X = X_+\oplus X_-.
\]
We denote by $\cSloc$
and $\cStw$ sets of representatives of inequivalent simple local
and twisted modules as explained in the last subsection.  We
define characters as before as well:
\[
  \ch^\pm[X](\tau, v) := \text{tr}_{X_+}\left(q^{L(0)-\frac{c}{24}} o(v)\right) \pm
  \text{tr}_{X_-}\left(q^{L(0)-\frac{c}{24}} o(v)\right), \qquad q=e^{2\pi i\tau},
\]
where $o(v)$ is the zero-mode of the vertex operator for $v\in
V_+$. It turns out that the form of the modular $S$-transformation and
Verlinde's formula are essentially decided by the categorical
dimension of $V_-$ as a $V_+$-module, which is $-1$ as in the case of
$\ZZ$-graded \VOSAs. Thus these formulae are the same as before, that
is, the modular $S$-transformation has the form
\begin{equation}
\begin{split}
  \ch^+[X]\left(-\frac{1}{\tau}, v\right)
  &= \tau^{\text{wt}[v]} \sum_{Y\in \cStw} S_{X, Y}  \ch^-[Y](\tau, v) \\
  \ch^-[X]\left(-\frac{1}{\tau}, v\right) &= \tau^{\text{wt}[v]}
  \sum_{Y\in \cSloc} S_{X, Y} \ch^-[Y](\tau, v)
\end{split}
\end{equation}
if $X$ is local and 
\begin{equation}
\begin{split}
  \ch^+[X]\left(-\frac{1}{\tau}, v\right)
  &= \tau^{\text{wt}[v]} \sum_{Y\in \cStw} S_{X, Y}  \ch^+[Y](\tau, v) \\
  \ch^-[X]\left(-\frac{1}{\tau}, v\right) &= \tau^{\text{wt}[v]}
  \sum_{Y\in \cSloc} S_{X, Y} \ch^+[Y](\tau, v)
\end{split}
\end{equation}
if $X$ is twisted.  The modular $T$-transformation acts diagonally on
characters and supercharacters of simple twisted modules and maps
characters of local modules to the corresponding supercharacters and
vice versa.  We observe that the linear span of twisted characters
already forms a vector-valued modular form for the modular group.
Verlinde's formula for $\frac{1}{2}\ZZ$-graded \VOAs{} is then
\begin{equation}
\begin{split}
  {N^+}_{X,Y}^{W} &= \delta_{t(X)+t(Y), t(W)}\sum\limits_{Z \in \cStw}
  \dfrac{{S}_{X,Z}\cdot {S}_{Y,Z}\cdot ({S}^{-1})_{Z,W}}{{S}_{V, Z}}\\
  {N^-}_{X,Y}^{W} &= \delta_{t(X)+t(Y), t(W)} \sum\limits_{Z \in
    \cSloc}\dfrac{{S}_{X,Z}\cdot {S}_{Y,Z}\cdot
    ({S}^{-1})_{Z,W}}{{S}_{V, Z}}
\end{split}
\end{equation}
with 
\[
  t : [\rep V] \rightarrow \ZZ/2\ZZ, \qquad X \mapsto \begin{cases} \even &
    \qquad \text{if} \ X \in \cSloc \\ \odd & \qquad \text{if} \ X \in
    \cStw\end{cases}.
\]
The fusion product is then
\[
  X \boxtimes_V Y \cong \bigoplus_{W\in[\rep V]} \left({N^\even}_{X,Y}^{W}W
    \oplus {N^\odd}_{X,Y}^{W} W'\right), \qquad\qquad {N^\pm}_{X,Y}^{W} =
  {N^\even}_{X,Y}^{W} \pm {N^\odd}_{X,Y}^{W},
\]
where $W'$ is in the same class as $W$ but has reversed
$\ZZ/2\ZZ$-grading. The
properties \eqref{eq:adimdef} and \eqref{eq:Introadimprod1} of asymptotic dimensions also hold, 
provided there are unique modules $Z_{(\pm)}$ of local and twisted type with
minimal conformal weight, ch$^-[Z_{(-)}]\neq 0$, and with 
lowest conformal weight spaces purely even.

\subsubsection{$\frac{1}{2}\ZZ$-graded \VOSAs}

The case of $\frac{1}{2}\ZZ$-graded \VOSAs{} $V$ is more subtle than the
two previous cases due to the possible existence of fixed points. As
before, $V$ is an order-$2$ simple current
extension of its even vertex operator subalgebra,
$V=V^\even\oplus V^\odd$, and we choose a set of representatives of
equivalence classes of local modules $\cSloc$ as before. However, for
the twisted modules there are two possibilities. Let $X$ be a simple
twisted module; as a $V^\even$-module,
\[
X\cong X^\even\oplus X^\odd
\]
where $X^\even$ and $X^\odd$ are two simple $V^\even$-modules. Then
either $X^\even\cong X^\odd$ or not. In the first case we say $X$ is
a fixed point. Representatives of equivalence classes of simple
fixed-point modules are denoted by $\cSfix$, while the set of
representatives of ordinary non-fixed-point twisted modules we
call $\cSord$, so that $\cStw$ is the disjoint union of these two sets
of representatives of equivalence classes of simple twisted
objects. A simple local module is never a fixed point.  We
define characters and supercharacters as before:
\[
  \ch^+[X](\tau, v) := \text{tr}_X\left(q^{L(0)-\frac{c}{24}} o(v)\right) \qquad
  \text{and}\qquad \ch^-[X](\tau, v) := \text{str}_X\left(q^{L(0)-\frac{c}{24}}
    o(v)\right), \qquad q=e^{2\pi i\tau}.
\]
Note that
\[
 \ch^-[X](\tau, v) = 0 \qquad\qquad\text{for} \ X\in\cSfix. 
\]
The modular $S$-transformation has the form
\begin{equation}
\begin{split}
  \ch^+[X]\left(-\frac{1}{\tau}, v\right)
  &= \tau^{\text{wt}[v]} \sum_{Y\in \cSloc} S_{X, Y}  \ch^+[Y](\tau, v) \\
  \ch^-[X]\left(-\frac{1}{\tau}, v\right) &= \tau^{\text{wt}[v]}
  \sum_{Y\in \cStw} S_{X, Y} \ch^+[Y](\tau, v)
\end{split}
\end{equation}
if $X$ is local and 
\begin{equation}
\begin{split}
  \ch^+[X]\left(-\frac{1}{\tau}, v\right)
  &= \tau^{\text{wt}[v]} \sum_{Y\in \cSloc} S_{X, Y}  \ch^-[Y](\tau, v) \\
  \ch^-[X]\left(-\frac{1}{\tau}, v\right) &= \tau^{\text{wt}[v]}
  \sum_{Y\in \cSord} S_{X, Y} \ch^-[Y](\tau, v)
\end{split}
\end{equation}
if $X$ is twisted (and in addition ordinary in the case of the
$S$-transformation for $\ch^-[X]$). We set $S_{X, Y}=0$ if both $X$
and $Y$ are twisted and in addition at least one of the two is in
$\cSfix$.  The modular $T$-transformation acts diagonally on
characters and supercharacters of simple twisted modules and maps
characters of simple local modules to the corresponding
supercharacters and vice versa.  In this case, we see that the linear
span of twisted ordinary supercharacters already forms a vector-valued
modular form for the modular group. For known results on modularity of
\VOSAs{} see \cite{DZ}.  This modularity observation has a nice
application in the context of Mathieu moonshine: it is needed in order
to identify self-dual regular \VOSAs{} together with their unique
twisted module with potential bulk superconformal field theories; see
\cite{CDR}.
 
Going back to Verlinde's formula for $\frac{1}{2}\ZZ$-graded
\VOSAs{}, we have:
\begin{equation}
\begin{split}
  {N^+}_{X,Y}^{W}&= n_W \delta_{t(X)+t(Y), t(W)}\sum\limits_{Z \in \cSloc}
  \dfrac{{S}_{X,Z}\cdot {S}_{Y,Z}\cdot ({S}^{-1})_{Z,W}}{{S}_{V, Z}}\\
  {N^-}_{X,Y}^{W}&=  \delta_{t(X)+t(Y), t(W)} \sum\limits_{Z \in \cStw}
  \dfrac{{S}_{X,Z}\cdot {S}_{Y,Z}\cdot ({S}^{-1})_{Z,W}}{{S}_{V, Z}}
\end{split}
\end{equation}
with 
\[
  t : [\rep V] \rightarrow \ZZ/2\ZZ, \hspace{1em} X \mapsto \begin{cases} \even &
    \qquad \text{if} \ X \in \cSloc \\ \odd & \qquad \text{if} \ X \in \
    \cStw\end{cases}\qquad\text{and} \qquad n_W = \begin{cases} 2
    &\qquad \text{if} \ W \ \in\cSfix \\ 1
    &\qquad\text{else} \end{cases}.
\]
Thus the Verlinde formula is slightly more complicated for
the fusion rules of a local with a twisted module.  The fusion product
is 
\[
  X \boxtimes_V Y \cong \bigoplus_{W\in\cS} \left({N^\even}_{X,Y}^{W}W
    \oplus {N^\odd}_{X,Y}^{W} \Pi(W)\right), \qquad\qquad {N^\pm}_{X,Y}^{W} =
  {N^\even}_{X,Y}^{W} \pm {N^\odd}_{X,Y}^{W},
\]
where $\Pi$ is the parity-reversing endofunctor.

We assume the existence of unique simple local and twisted modules
$Z_{(\pm)}$ as before, we assume that $\ch^-[Z_{(+)}]$ does not
vanish, and that the lowest conformal weight spaces of $Z_{(\pm)}$ are
purely even.  For a simple $V$-module $X$, the asymptotic
(super)dimension is defined as before by
\begin{align}
  \adim^\pm[X] = \lim_{iy\rightarrow i\infty}
  \dfrac{\ch^\pm[X](-1/iy)}{\ch^\pm[V](-1/iy)} = 
\dfrac{{S}_{X,Z_{(\mp)}}}{{S}_{V,Z_{(\mp)}}}.
\end{align}
We find that for $X, Y\in[\rep V]$, the product of asymptotic
superdimensions respects the fusion product in the sense that
\begin{align}
\adim^-[X]\adim^{-}[Y]&=\sum_{W\in [\rep V]} {N^-}_{X, Y}^{W}\adim^-[W].
\end{align}
There is also a formula for the product of asymptotic dimensions, but
it requires knowledge of fixed points. However, if $X$ and $Y$ are
both local or both twisted, then we also get
\begin{align}\label{eq:Introadimprod2}
\adim^+[X]\adim^{+}[Y]&= \sum_{W\in [\rep V]} {N^+}_{X, Y}^{W}\adim^+[W].
\end{align}

\subsection{Outline}

Here we summarize the structure of this paper. In Section
\ref{sec:TensCats}, we work in an abstract tensor-categorical setting
to study representations of (super)algebra objects in braided tensor
categories. We prove that, for a (super)algebra $A$ in a braided
tensor category $\cC$, the category $\repA$ of $A$-modules with its
tensor product $\boxtimes_A$ forms an additive monoidal
(super)category, under minimal assumptions on $\cC$. To prove the
associativity in $\repA$, we introduce and study what we call
categorical $\repA$-intertwining operators. Moreover, we prove that
the category $\repzA$ of local $A$-modules is braided monoidal.  We
provide detailed proofs, often in terms of graphical calculus. Though many results proved here are
known, at least in more restrictive settings (see for example
\cite{Pa, KO, EGNO}), readers new to the theory of
tensor categories may find our exposition helpful. In Section
\ref{subsec:inductionfunctor} we introduce the induction functor
$\cF:\cC\rightarrow\repA$ and prove that it is a tensor
functor. In Section \ref{subsec:miscresults}, we collect several
useful properties of the induction functor, culminating in the
relationship between $S$-matrices in $\cC$ and $\repzA$, provided
both are ribbon categories.

In Section \ref{sec:vtc}, we begin with a quick yet fairly detailed
exposition of vertex tensor category constructions.  Readers new to
vertex operator algebras or vertex tensor categories may
find this introductory material useful. Then, we proceed to prove our
main theorem, deriving complex-analytic properties of categorical
$\repA$-intertwining operators.  In Sections \ref{subsec:complexIntw}
and \ref{subsec:VTConRepARep0A}, we use these complex-analytic
properties to match the braided monoidal category structure on
$\repzA$ from Section \ref{sec:TensCats} with vertex tensor
categorical structure on the category of $A$-modules in $\cC$. Some
analytic lemmas needed in these proofs as well as in Section
\ref{sec:applications} are collected in Section
\ref{subsec:technical}.

In Section \ref{sec:applications} we apply our results.  We first give
several examples of vertex operator (super)algebra extensions. After
recalling several results on simple current extensions, we show how to
use screening charges to construct extensions of several copies of
(non-rational but $C_2$-cofinite) triplet algebras $\cW(p)$. The next
examples are extensions of products of rational Virasoro
\VOAs. Interestingly, Corollary \ref{cor:Etype} gives type $E$
extensions.  Our next application is Verlinde formulae for regular
\VOSAs{} as well as $\frac{1}{2}\ZZ$-graded \VOAs{} (see Section
\ref{sec:ver} above for a summary of results). Finally, we take $V$ to
be a rational vertex operator (super)algebra containing as vertex
operator subalgebra a lattice \VOA{} $V_L$ of a positive-definite integral
lattice $L$. The coset $C=\com\left(V_L, V\right)$ is often
called the parafermion coset of $V$. Building on \cite{CKLR}, we
describe these cosets in generality: we classify simple
inequivalent $C$-modules in terms of those for $V$, we express the
fusion rules for coset modules in terms of those for 
$V$, and we express characters of $C$-modules and their modular
transformations in terms of those for $V$-modules. These results are applied to several examples, in which $V$ is a rational affine \VOA{}, a
rational $N=2$ superconformal algebra, and a rational
Bershadsky-Polyakov algebra.

\section{Tensor categories and supercategories}\label{sec:TensCats}

This section largely follows \cite[Section 1]{KO}, where a tensor category $\repA$ of representations of a commutative associative algebra $A$ in a braided tensor category $\cC$ was constructed, as well as an induction tensor functor from $\cC$ to $\repA$. See also \cite{Pa} for the first construction of the braided tensor subcategory $\repzA$ of $\repA$. However, we work in a more general setting than \cite{KO} in two important respects: first, we do not assume that $\cC$ is rigid or semisimple, and second, we work with \textit{super}algebras in $\cC$. We need these generalizations to study braided tensor categories associated to potentially non-regular \VOAs{} (whose module categories are typically non-semisimple and are typically not known to be rigid) and \VOSAs. Also, our approach to the associativity isomorphisms in $\repA$, making use of what we call ``categorical $\repA$-intertwining operators,'' is inspired by vertex operator algebra theory and is necessary for relating the tensor-categorical structure of $\repA$ to the vertex-tensor-categorical constructions of \cite{HLZ1}-\cite{HLZ8}, which we accomplish in Section \ref{sec:vtc}. For these reasons, and for the reader's convenience, we include proofs of the main results in this section, although most are straightforward generalizations of results in \cite{Pa, KO}. For brevity and clarity, many tensor-categorical calculations are summarized only, using graphical calculus; for fuller details, the reader may consult the first arXiv version of this paper: \texttt{https://arxiv.org/abs/1705.05017v1}.

\subsection{Commutative associative algebras in braided tensor categories}

Our motivation for this subsection is a vertex operator algebra extension $V\subseteq A$ where $V$ has a vertex tensor category $\cC$ of modules, so $\cC$ is also an (abelian) braided tensor category, and where $A$ is an object of $\cC$. See \cite{H-cofin} for conditions on $V$ under which the full category of grading-restricted generalized $V$-modules is a finite abelian category and has vertex tensor category structure; for example, these conditions hold if $V$ is $C_2$-cofinite but not necessarily rational.

For background on (braided) tensor categories, see for example \cite{EGNO, Ka, BK}. As the term ``tensor category'' has several variant meanings in the literature, we first clarify our terminology. For us, a monoidal category is a category $\cC$ equipped with tensor product bifunctor $\boxtimes$, unit object $\unit$, natural left and right unit isomorphisms $l$ and $r$, and a natural associativity isomorphism $\cA$, which satisfy the triangle and pentagon axioms. A tensor category is a monoidal category with additionally a compatible abelian category structure in the sense that the tensor product of morphisms is bilinear. A monoidal or tensor category is braided if it has a natural braiding isomorphism $\cR$ satisfying the hexagon axioms.

We fix a braided tensor category $(\cC, \boxtimes, \unit, l, r, \cA, \cR)$. We also assume that morphism sets in $\cC$ are vector spaces over a field $\mathbb{F}$ such that composition and tensor product of morphisms are both $\mathbb{F}$-bilinear, that is, $\cC$ is an $\mathbb{F}$-linear braided tensor category. For braided tensor categories coming from vertex operator algebras, $\mathbb{F}=\CC$. The only additional assumption on $\mathcal{C}$ that we will need for now, which is justified for braided tensor categories arising from vertex operator algebras thanks to Proposition 4.26 in \cite{HLZ3}, is the following:
\begin{assum}
For any object $W$ in $\mathcal{C}$, the functors $W\boxtimes\cdot$ and $\cdot\boxtimes W$ are right exact.
\end{assum}

\begin{defi}
	A \textit{commutative associative algebra} in $\mathcal{C}$ is an object $A$ of $\mathcal{C}$ equipped with morphisms $\mu: A\boxtimes A\rightarrow A$ and $\iota_A: \mathbf{1}\rightarrow A$ satisfying the following axioms:
	\begin{enumerate}
		\item Associativity: $\mu\circ(\mu\boxtimes 1_A)\circ\mathcal{A}_{A,A,A} = \mu\circ(1_A\boxtimes\mu)$ as morphisms $A\boxtimes (A\boxtimes A)\rightarrow A$.
		\item Commutativity: $\mu\circ \cR_{A,A}=\mu$ as morphisms $A\boxtimes A\rightarrow A$.
		\item Unit:  $\mu\circ(\iota_A\boxtimes 1_A)\circ l_A^{-1}=1_A$ as morphisms $A\rightarrow A$.
	\end{enumerate}
\end{defi}
\begin{rema}\label{rightunit}
	The commutativity and unit axioms in the above definition together imply the right unit property $\mu\circ(1_A\boxtimes\iota_A)\circ r_A^{-1}=1_A$.
\end{rema}

\begin{rema}
	If $\mathcal{C}$ is a vertex tensor category of modules for a vertex operator algebra $V$, then a commutative associative algebra $A$ in $\mathcal{C}$ with trivial twist $\theta_A=e^{2\pi i L(0)}$ and with $\iota_A$ injective is a vertex operator algebra extension of $V$, by Theorem 3.2 and Remark 3.3 in \cite{HKL}. (The assumption that $\iota_A$ is injective is implicit in the proof of \cite[Theorem 3.2]{HKL}.)
\end{rema}

Given a commutative associative algebra $A$ in $\mathcal{C}$, we have the following category of ``$A$-modules'':
\begin{defi}\label{repAdef}
	The category $\mathrm{Rep}\,A$ has objects $(W, \mu_W)$ where $W$ is an object of $\mathcal{C}$ and $\mu_W: A\boxtimes W\rightarrow W$ satisfies:
	\begin{enumerate}
		\item Associativity:  $ \mu_W\circ(\mu\boxtimes 1_W)\circ\mathcal{A}_{A,A,W}=\mu_W\circ (1_A\boxtimes\mu_W)$ as morphisms $A\boxtimes (A\boxtimes W)\rightarrow W$.
		\item Unit:  $\mu_W\circ(\iota_A\boxtimes 1_W)\circ l_W^{-1}=1_W$ as morphisms $W\rightarrow W$.
	\end{enumerate}
	The morphisms between objects $(W_1,\mu_{W_1})$ and $(W_2, \mu_{W_2})$ of $\mathrm{Rep}\,A$ are all $\mathcal{C}$-morphisms $f: W_1\rightarrow W_2$ such that $f\circ\mu_{W_1}=\mu_{W_2}\circ(1_A\boxtimes f): A\boxtimes W_1\rightarrow W_2$. 
\end{defi}

\begin{rema}
 Note that $(A,\mu)$ is an object of $\repA$.
\end{rema}

\begin{rema} Here, we preview the relation of $\repA$-objects to vertex-algebraic intertwining operators.
	When $\mathcal{C}$ is a braided tensor category of modules for a vertex operator algebra $V$ and $V\subseteq A$ is a vertex operator algebra extension, we will show that an object of $\mathrm{Rep}\,A$ is a $V$-module $W$ equipped with a $V$-intertwining operator $Y_W$ of type $\binom{W}{A\,W}$ such that $Y_W(\mathbf{1}, x)=1_W$ and as multivalued functions,
	\begin{equation*}
	\langle w', Y_W(a_1,z_1)Y_W(a_2,z_2)w\rangle=\langle w', Y_W(Y(a_1,z_1-z_2)a_2, z_2)w\rangle
	\end{equation*}
	for $a_1, a_2\in A$, $w\in W$, $w'\in W'$, and $z_1,z_2\in\CC^\times$ such that
	\begin{equation*}
	\vert z_1\vert>\vert z_2\vert>\vert z_1-z_2\vert >0.
	\end{equation*}
	Here $Y$ is the vertex operator for $A$. If $W_1$ and $W_2$ are two objects of $\mathrm{Rep}\,A$, then a morphism $f: W_1\rightarrow W_2$ in $\mathrm{Rep}\,A$ will be a $V$-module homomorphism which satisfies
	\begin{equation*}	f(Y_{W_1}(a,x)w_1)=Y_{W_2}(a,x)f(w_1)
	\end{equation*}
	for $a\in A$ and $w_1\in W_1$.
\end{rema}
\begin{rema}
	In the setting of the previous remark, objects of $\mathrm{Rep}\,A$ are more general than ordinary $A$-modules, because we do not require the $V$-intertwining operator $Y_W(\cdot, x)\cdot$ for a module $W$ in $\repA$ to involve only integral powers of the formal variable $x$. For example, if $V$ is the fixed-point subalgebra of a group of automorphisms $G$ of $A$, then $g$-twisted $A$-modules for $g\in G$ are objects of $\mathrm{Rep}\,A$. By \cite[Theorem 3.4]{HKL}, ordinary $A$-modules are objects of the subcategory $\mathrm{Rep}^0\,A$ recalled in Section \ref{subsec:rep0A} below.
\end{rema}

We will show that $\repA$ has a natural tensor category structure. In particular:
\begin{enumerate}
\item There is a tensor product bifunctor $\boxtimes_A: \repA\times\repA\rightarrow\repA$.

 \item The object $(A,\mu)$ is a unit for $\boxtimes_A$, equipped with natural left and right unit isomorphisms $l^A$ and $r^A$ from $A\boxtimes_A\cdot$ and $\cdot\boxtimes_A A$, respectively, to the identity functor on $\repA$.
 
 \item $\repA$ has a natural associativity isomorphism $\cA^A:\boxtimes_A\circ(1_{\repA}\times\boxtimes_A)\rightarrow\boxtimes_A\circ(\boxtimes_A\times 1_{\repA})$.
 
 \item The unit and associativity isomorphisms in $\repA$ satisfy the pentagon and triangle axioms. (This will complete the proof that $\repA$ is a monoidal category.)
 
 \item $\repA$ is an abelian category for which the tensor product of morphisms is bilinear.
\end{enumerate}
We will show that $\repA$ is abelian in this subsection. After introducing superalgebras in $\cC$ in the next subsection, we will construct $\boxtimes_A$ and the unit isomorphisms. We will defer the construction of the associativity isomorphisms and the proof that $\repA$ is a monoidal category to Section \ref{subsec:associnrepA} after introducing categorical $\repA$-intertwining operators.

\begin{theo}[{See also\ \cite[Lemma 1.4]{KO}}]\label{thm:repAabelian} 
$\repA$ is an $\mathbb{F}$-linear abelian category.
\end{theo}
\begin{proof}

To show that $\repA$ is an abelian category, we need to show the following:
\begin{enumerate}
 \item $\repA$ has a zero object.
 
 \item Morphism sets in $\repA$ are $\mathbb{F}$-vector spaces, with composition of morphisms bilinear.
 
 \item Biproducts of finite sets of objects in $\repA$ exist. Recall that a biproduct of a finite set of objects $\lbrace W_i\rbrace$ is an object $\bigoplus W_i$ equipped with morphisms $p_i: \bigoplus W_i\rightarrow W_i$ and $q_i: W_i\rightarrow\bigoplus W_i$ such that $p_i\circ q_i=1_{W_i}$ for all $i$, $p_i\circ q_j=0$ for $i\neq j$, and $\sum q_i\circ p_i=1_{\bigoplus W_i}$. A biproduct of $\lbrace W_i\rbrace$ is both a product and a coproduct of $\lbrace W_i\rbrace$.
 
 \item Every morphism in $\repA$ has a kernel and a cokernel.
 
 \item Every monomorphism in $\repA$ is a kernel and every epimorphism is a cokernel.
\end{enumerate}
The zero object of $\repA$ is the zero object $0$ of $\cC$ equipped with $\mu_0: A\boxtimes 0\rightarrow 0$ the unique morphism into $0$. Then since composition of morphisms in $\cC$ is $\mathbb{F}$-bilinear, $\mathrm{Hom}_{\repA} (W_1, W_2)$ is an $\mathbb{F}$-vector subspace of $\hom_{\cC}(W_1,W_2)$ for any objects $W_1$, $W_2$ of $\repA$.

Now for a finite set of objects $\lbrace (W_i,\mu_{W_i})\rbrace$ in $\repA$ with biproduct $(\bigoplus W_i, \lbrace p_i\rbrace, \lbrace q_i\rbrace)$ in $\cC$, the object $\bigoplus W_i$ has the left $A$-action
%
%
\begin{equation*}
 \mu_{\bigoplus W_i} =\sum q_i\circ\mu_{W_i}\circ(1_A\boxtimes p_i): A\boxtimes\bigoplus W_i\rightarrow\bigoplus W_i.
\end{equation*}
Using the unit and associativity properties for each $\mu_{W_i}$, together with properties of the $p_i$ and $q_i$, it is straightforward to show that $\mu_{\bigoplus W_i}$ also satisfies the unit and associativity properties. Moreover, it is easy to see from the definition of $\mu_{\bigoplus W_i}$ that each $p_i$ and $q_i$ is a morphism in $\repA$. Thus $(\bigoplus W_i,\lbrace p_i\rbrace,\lbrace q_i\rbrace)$ is also a biproduct of $\lbrace(W_i,\mu_{W_i})\rbrace$ in $\repA$.

Now suppose that $f: W_1\rightarrow W_2$ is a morphism in $\repA$; we need to show that $f$ has a kernel and a cokernel in $\repA$. Since $\cC$ is an abelian category, there is a kernel $k: K\rightarrow W_1$ of $f$ in $\cC$.
	We need to show that $K$ has an $A$-module structure such that $k$ is a morphism in $\repA$, and that $(K,k)$ satisfy the universal property of a kernel in $\repA$. Consider 
	the morphism
	\begin{equation*}
	 \mu_{W_1}\circ(1\boxtimes k): A\boxtimes K\rightarrow W_1.
	\end{equation*}
Since $f$ is a morphism in $\repA$, 
\begin{align*}
 f\circ\mu_{W_1}\circ(1\boxtimes k)=\mu_{W_2}\circ(1\boxtimes f)\circ(1\boxtimes k)=0,
\end{align*}
so the universal property of the kernel $(K,k)$ induces a unique $\cC$-morphism  $\mu_K: A\boxtimes K\rightarrow K$
such that 
\begin{equation}\label{muKdef}
 k\circ\mu_K=\mu_{W_1}\circ(1\boxtimes k).
\end{equation}
 We prove that $(K,\mu_K)$ is an object of $\repA$. For the unit property for $\mu_K$, observe that
 \begin{equation*}
  K\xrightarrow{l_K^{-1}}\mathbf{1}\boxtimes K\xrightarrow{\iota_A\boxtimes 1_K} A\boxtimes K\xrightarrow{\mu_K} K\xrightarrow{k} W_1
 \end{equation*}
equals the composition
\begin{equation*}
 K\xrightarrow{l_K^{-1}}\mathbf{1}\boxtimes K\xrightarrow{\iota_A\boxtimes 1_K} A\boxtimes K\xrightarrow{1\boxtimes k} A\boxtimes W_1\xrightarrow{\mu_{W_1}} W_1
\end{equation*}
by \eqref{muKdef}. By the naturality of the unit isomorphisms and the unit property for $\mu_{W_1}$, this composition equals $k$. Since $k$ is a kernel in $\cC$ and thus a monomorphism, it follows that $\mu_K\circ(\iota_A\boxtimes 1_K)\circ l_{K}^{-1}= 1_K$. To prove the associativity of $\mu_K$, we use \eqref{muKdef}, the associativity of $\mu_{W_1}$, and naturality of the associativity isomorphisms:
\begin{align*}
 k\circ\mu_K\circ(1_A & \boxtimes\mu_K)  =\mu_{W_1}\circ(1_A\boxtimes k)\circ(1_A\boxtimes\mu_K) =\mu_{W_1}\circ(1_A\boxtimes\mu_{W_1})\circ(1_A\boxtimes(1_A\boxtimes k))\nonumber\\
 & =\mu_{W_1}\circ(\mu\boxtimes 1_{W_1})\circ\cA_{A,A,W_1}\circ(1_A\boxtimes(1_A\boxtimes k)) =\mu_{W_1}\circ(1_A\boxtimes k)\circ(\mu\boxtimes 1_K)\circ\cA_{A,A,K}\nonumber\\
 & =k\circ\mu_K\circ(\mu\boxtimes 1_K)\circ\cA_{A,A,K}.
\end{align*}
Since $k$ is a monomorphism, associativity for $\mu_K$ follows.

Now that $(K,\mu_K)$ is an object of $\repA$, \eqref{muKdef} says that $k$ is a morphism in $\repA$. To show that
	$(K,\mu_K)$ and $k$ form a kernel of $f$ in $\repA$, let $(K',\mu_{K'})$ be an object of $\repA$ and let $k': K'\rightarrow W_1$ be a $\repA$-morphism such that
	$f\circ k'=0$. As $(K, k)$ is a kernel of $f$ in $\cC$, there is a unique $\cC$-morphism $g: K'\rightarrow K$ such that $k\circ g=k'$. We just need to show that $g$ is
	a morphism in $\repA$. For this, we use the fact that $k$ and $k'$ are $\repA$-morphisms to calculate:
	\begin{align*}
	k\circ\mu_K\circ(1_A\boxtimes g)&=\mu_{W_1}\circ(1_A\boxtimes k)\circ(1_A\boxtimes g)=\mu_{W_1}\circ (1_A\boxtimes k')
	= k'\circ\mu_{K'}=k\circ g\circ \mu_{K'}.
	\end{align*}
	Since $k$ is monic, it follows that $\mu_K\circ(1_A\boxtimes g)=g\circ\mu_{K'}$, as desired.

To prove that $f: W_1\rightarrow W_2$ also has a cokernel in $\repA$, take a cokernel $(C, c)$ of $f$ in $\cC$; we need to show that $C$ is an object of $\repA$. Since we assume that $A\boxtimes\cdot$ is right exact, $(A\boxtimes C, 1_A\boxtimes c)$ is a cokernel of $1_A\boxtimes f$ in $\cC$. Then since $f$ is a morphism in $\repA$,
\begin{equation*}
 c\circ\mu_{W_2}\circ(1_A\boxtimes f)=c\circ f\circ\mu_{W_1}=0,
\end{equation*}
so that the universal property of the cokernel $(A\boxtimes C, 1_A\boxtimes c)$ implies there is a unique $\cC$-morphism $\mu_C: A\boxtimes C\rightarrow C$ such that
\begin{equation}\label{muCdef}
 \mu_C\circ(1_A\boxtimes c)=c\circ\mu_{W_2}.
\end{equation}
The proof that $(C,\mu_C)$ is an object of $\repA$, that $c: W_2\rightarrow C$ is a morphism in $\repA$, and that $C$ is a cokernel of $f$ in $\repA$ is then similar to the proof for kernels.

Next, we show that if $f: W_1\rightarrow W_2$ is a monomorphism in $\repA$, then $f$ is a kernel. First, we show that $f$ is a monomorphism in $\cC$. Thus suppose $f\circ g=0$ where $g: W_1'\rightarrow W_1$ is any morphism in $\cC$; we need to show that $g=0$. By the universal property of the kernel $(K, k)$ of $f$, there is a unique $\cC$-morphism $g': W_1'\rightarrow K$ such that $g=k\circ g'$. But we have already shown that $k$ is a morphism in $\repA$; thus because $f$ is a monomorphism in $\repA$ and $f\circ k=0$ by definition of the kernel of $f$, it follows that $k=0$, and thus $g=0$ as well.

Now that $f$ is a monomorphism in $\cC$, it is the kernel of a $\cC$-morphism $g: W_2\rightarrow W_2'$. As $f$ is also a morphism in $\repA$, it has a cokernel $(C,c)$ in both $\cC$ and $\repA$. We will show that $(W_1, f)$ is also a kernel of $c$. Now, since $g\circ f=0$, the universal property of the cokernel in $\cC$ implies there is a unique $\cC$-morphism $g': C\rightarrow W_2'$ such that $g=g'\circ c$. Then if $h: W\rightarrow W_2$ is any morphism in $\repA$ such that $c\circ h=0$, we have
\begin{equation*}
 g\circ h=g'\circ c\circ h=0
\end{equation*}
as well, so by the universal property of the kernel $(W_1, f)$ of $g$, there is a unique $\cC$-morphism $h': W\rightarrow W_1$ such that $f\circ h'=h$. To show that $h'$ is actually a morphism in $\repA$, we observe that
\begin{equation*}
 f\circ h'\circ\mu_W=h\circ\mu_W=\mu_{W_2}\circ(1_A\boxtimes h)=\mu_{W_2}\circ(1_A\boxtimes f)\circ(1_A\boxtimes h')=f\circ\mu_{W_1}\circ(1_A\boxtimes h')
\end{equation*}
since $f$ and $h$ are morphisms in $\repA$. Thus since $f$ is a monomorphism, $h'\circ\mu_W=\mu_{W_1}\circ(1_A\boxtimes h')$.

Finally, we can show similarly that if $f: W_1\rightarrow W_2$ is an epimorphism in $\repA$, then it is the cokernel of its kernel in $\repA$.
\end{proof}

\subsection{Superalgebras in braided tensor categories}
\label{subsec:superalgebra}

In this subsection we formulate the notion of supercommutative associative superalgebra in our braided tensor category $\cC$. The idea is to realize a superalgebra as an (ordinary) commutative associative algebra in a braided tensor \textit{super}category associated to $\cC$. Our motivation is to study vertex operator superalgebra extensions of vertex operator algebras using tensor category theory. Note, for example, that any vertex operator superalgebra $V$ with parity decomposition $V^{\bar{0}}\oplus V^{\bar{1}}$ is an extension of its even vertex operator subalgebra $V^{\bar{0}}$. If $V^{\bar{0}}$ has a vertex tensor category $\cC$ of modules that includes $V^{\bar{1}}$, then $V$ will be a superalgebra object in $\cC$ in the sense defined below.

First we clarify what we mean by supercategory: should objects or morphism spaces have natural $\ZZ/2\ZZ$-gradings? We will follow the approach of \cite{BE} and give morphism spaces $\ZZ/2\ZZ$-gradings; also see \cite{BE} for supercategory analogues of notions from tensor category theory. We denote the parity of a parity-homogeneous morphism $f$ in a supercategory by $\vert f\vert\in\ZZ/2\ZZ$. Here, we highlight some differences that appear in the supercategory setting; the general philosophy is that sign factors should appear whenever the order of two morphisms with parity is reversed. First, composition of morphisms in a supercategory should be even as a linear map between superspaces, and superfunctors between supercategories should induce even linear maps on morphisms. If $\sC$ is a supercategory, then $\sC\times\sC$ is naturally a supercategory with composition of morphisms given by
\begin{equation*}
 (f_1, f_2)\circ(g_1, g_2)=(-1)^{\vert f_2\vert\vert g_1\vert}(f_1\circ g_1, f_2\circ g_2),
\end{equation*}
when $f_2$ and $g_1$ are parity-homogeneous morphisms in $\sC$ (this is called the super interchange law in \cite{BE}). Moreover, $\sC\times\sC$ has a supersymmetry superfunctor $\sigma$ given by $\sigma(W_1,W_2)=(W_2,W_1)$ for objects and $\sigma(f_1,f_2)=(-1)^{\vert f_1\vert\vert f_2\vert}(f_2,f_1)$ for parity-homogeneous morphisms.

A monoidal supercategory $\sC$ is equipped with a tensor product superfunctor $\boxtimes: \sC\times\sC\rightarrow\sC$. In particular, $\boxtimes$ induces an even linear map on morphisms: $\vert f_1\boxtimes f_2\vert=\vert f_1\vert+\vert f_2\vert$ for parity-homogeneous morphisms $f_1$ and $f_2$, and the super interchange law holds:
\begin{equation}\label{exchange}
 (f_1\boxtimes f_2)\circ(g_1\boxtimes g_2)=(-1)^{\vert f_2\vert\vert g_1\vert} (f_1\circ g_1)\boxtimes(f_2\circ g_2)
\end{equation}
for appropriately composable morphisms $f_1, f_2, g_1, g_2$ with $f_2$ and $g_1$ parity-homogeneous. The natural unit and associativity isomorphisms are even, and in a braided monoidal supercategory, the natural braiding isomorphism $\cR: \boxtimes\rightarrow\boxtimes\circ\sigma$ is also even. Naturality of $\cR$ means that for parity-homogeneous morphisms $f_1: W_1\rightarrow\widetilde{W}_1$ and $f_2: W_2\rightarrow\widetilde{W}_2$ in $\sC$,
\begin{equation}\label{braidingsupernatural}
 \cR_{\widetilde{W}_1,\widetilde{W}_2}\circ(f_1\boxtimes f_2)=(-1)^{\vert f_1\vert\vert f_2\vert}(f_2\boxtimes f_1)\circ\cR_{W_1,W_2}.
\end{equation}
A (braided) tensor supercategory is a (braided) monoidal supercategory with a compatible abelian category structure, that is, the tensor product of morphisms is bilinear.

We continue to fix an $\mathbb{F}$-linear braided tensor category $\cC$ in which tensoring is right exact. Since we will be working with parity automorphisms on $\ztwo$-graded objects in this and following subsections, we will now for convenience assume that the characteristic of $\mathbb{F}$ is not $2$. As $\boxtimes$ is biadditive, tensor products in $\cC$ distribute over biproducts:
\begin{propo}\label{distribute}
 For any finite sets of objects $\lbrace W_i\rbrace$ and $\lbrace X_j\rbrace$ in $\cC$, there is a unique isomorphism
 \begin{equation*}
  F: \left(\bigoplus W_i\right)\boxtimes\left(\bigoplus X_j\right)\rightarrow\bigoplus(W_i\boxtimes X_j)
 \end{equation*}
such that $F\circ(q_{W_i}\boxtimes q_{X_j})=q_{W_i\boxtimes X_j}$ and $p_{W_i\boxtimes X_j}\circ F=p_{W_i}\boxtimes p_{X_j}$ for all $i$, $j$.
\end{propo}

We now introduce a supercategory $\sC$ associated to $\cC$ that essentially consists of objects of $\cC$ having a parity decomposition. Alternatively, one can view $\sC$ as the Deligne product of $\cC$ with the supercategory of finite-dimensional vector superspaces:
\begin{defi}
	\label{def:superC} 
	Let ${\sC}$ be the category
	whose objects are ordered pairs $W=(W^\even ,W^\odd )$ where $W^\even$ and $W^\odd$ are objects of $\cC$
	and 
	\begin{equation*}
	\hom_{\sC}(W_1, W_2) = \hom_{\cC}(W_1^\even\oplus W_1^\odd, W_2^\even\oplus W_2^\odd).
	\end{equation*}
	For an object $W$ in $\sC$, the \textit{parity involution} of $W$ is 
	\begin{equation*}P_W=1_{
	 W^\even}\oplus (-1_{W^\odd})=q_{W^\even}\circ p_{W^\even}-q_{W^\odd}\circ p_{W^\odd}.
	\end{equation*}
	We let $\Pi$ denote the parity reversing endofunctor on $\sC$.
\end{defi}

We say that a morphism $f\in\hom_{\sC}(W_1, W_2)$ has parity $\vert f\vert\in\ZZ/2\ZZ$ if $f\circ P_{W_1}=(-1)^{\vert f\vert} P_{W_2}\circ f$, and we say that $f$ is \textit{even} or \textit{odd} according as $\vert f\vert=\even$ or $\vert f\vert=\odd$. It is clear that even and odd morphisms form subspaces of $\hom_{\sC}(W_1,W_2)$ for any objects $W_1$, $W_2$ in $\sC$. In fact, $\hom_{\sC}(W_1,W_2)$ is a superspace with parity decomposition given by the even and odd morphism subspaces, that is, every morphism in $\sC$ is uniquely the sum of an even morphism and an odd morphism. To see this, first note that for $f\in\hom_{\sC}(W_1,W_2)$, we can write $f = f^\even+ f^\odd$ where
\begin{equation*}
 f^\even=\dfrac{1+P_{W_2}}{2}\circ f\circ\dfrac{1+P_{W_1}}{2}+\dfrac{1-P_{W_2}}{2}\circ f\circ\dfrac{1-P_{W_1}}{2}
\end{equation*}
is even and
\begin{equation*}
 f^\odd=\dfrac{1+P_{W_2}}{2}\circ f\circ\dfrac{1-P_{W_1}}{2}+\dfrac{1-P_{W_2}}{2}\circ f\circ\dfrac{1+P_{W_1}}{2}
\end{equation*}
is odd. To understand this decomposition, note that for $W$ in $\sC$, $\frac{1+P_W}{2}=q_{W^\even}\circ p_{W^\even}$ is projection of $W$ onto $W^\even$, while $\frac{1-P_W}{2}=q_{W^\odd}\circ p_{W^\odd}$ is projection of $W$ onto $W^\odd$. Thus $f^\even$ consists of the components of $f$ mapping $W_1^\even$ onto $W_2^\even$ and $W_1^\odd$ onto $W_2^\odd$, while $f^\odd$ consists of the components of $f$ mapping $W_1^\even$ into $W_2^\odd$ and $W_1^\odd$ into $W_2^\even$. So every morphism in $\sC$ is the sum of even and odd components.

We also need the decomposition of $f$ into even and odd components to be unique. For this, we show that $0$ is the only morphism that is both even and odd. In fact, suppose $f\circ P_{W_1}=P_{W_2}\circ f=-f\circ P_{W_1}$. Then $2(f\circ P_{W_1})=0$, or $f\circ P_{W_1}=f\circ(q_{W^\even}\circ p_{W^\even}-q_{W^\odd}\circ p_{W^\odd})=0$. Thus $f\circ q_{W^\even}\circ p_{W^\even}=f\circ q_{W^\odd}\circ p_{W^\odd}$. By composing both sides of this relation with $q_{W^i}$, $i=\even, \odd$, on the right, we get $f\circ q_{W^i}=0$ for $i=\even,\odd$. But then
\begin{equation*}
 f=f\circ(q_{W^\even}\circ p_{W^\even} +q_{W^\odd}\circ p_{W^\odd}) =0.
\end{equation*}
Thus morphism spaces in $\sC$ are superspaces with parity decomposition given by even and odd morphisms. Moreover, it is clear that if $f_1: W_2\rightarrow W_3$ and $f_2: W_1\rightarrow W_2$ are parity-homogeneous morphisms in $\sC$, then $\vert f_1\circ f_2\vert=\vert f_1\vert+\vert f_2\vert$. Thus composition is an even bilinear map, and so $\sC$ is a supercategory.

\begin{rema}
 Since our convention is that $\sC$ contains both even and odd morphisms, morphisms do not generally preserve parity decompositions of objects. For instance, $W\cong\Pi(W)$ for any object $W$ in $\sC$ since the identity on $W^\even\oplus W^\odd=W^\odd\oplus W^\even$ is an (odd) isomorphism from $W$ to $\Pi(W)$. If one instead wants to use the convention that morphisms preserve parity decompositions of objects, one should use the \textit{underlying category} $\sCeven$ whose objects are the objects of $\sC$ and whose morphisms are the \textit{even} morphisms of $\sC$. In general we can expect more isomorphism classes of objects in $\sCeven$ since a given object in $\cC$ could have multiple distinct parity decompositions. Note also that the odd morphisms in $\hom_{\sC}(W_1,W_2)$ appear in $\sCeven$ as morphisms from $W_1$ to $\Pi(W_2)$. Although we shall generally work with $\sC$ we shall frequently remark on how our results apply to $\sCeven$.  
\end{rema}

\begin{defi}\label{def:SCtoC}
There is a natural inclusion functor $\cI: \cC\rightarrow {\sC}$ 
such that $\cI(W) = (W, 0)$ and $\cI(f) = f\oplus 0$. We also have a natural forgetful functor $\cJ:\sC \rightarrow \cC$ with
$\cJ: (W^\even, W^\odd) \rightarrow W^\even\oplus W^\odd$, $\cJ(f)=f$ (so $\cJ$ forgets the superspace structure on morphisms). With these functors in mind, we shall sometimes abuse notation and denote an object $W=(W^\even,W^\odd)$ of $\sC$
simply by $W = W^\even\oplus W^\odd$.
\end{defi}

The following proposition is easy to prove:
\begin{propo}\label{SCabelian}
 The supercategory $\sC$ is an $\mathbb{F}$-linear abelian category:
 \begin{enumerate}
  \item The zero object of $\sC$ is $\underline{0}=(0,0)$.
  \item The biproduct of a finite collection of objects $\lbrace W_i\rbrace$ in $\sC$ is given by $\bigoplus W_i=\left(\bigoplus W_i^\even, \bigoplus W_i^\odd\right)$ with $p_i=p_{W_i^\even}\oplus p_{W_i^\odd}$ and $q_i=q_{W_i^\even}\oplus q_{W_i^\odd}$.
  \item If $(K,k)$ is a kernel in $\cC$ of a morphism $f: W_1^\even\oplus W_1^\odd\rightarrow W_2^\even\oplus W_2^\odd$ in $\sC$, then $(K,0)$ is a kernel of $f$ in $\sC$, with kernel morphism
  \begin{equation*}
   K\oplus 0 \xrightarrow{p_K} K \xrightarrow{k} W_1^\even\oplus W_1^\odd.
  \end{equation*}

%
%
  \item If $(C,c)$ is a cokernel in $\cC$ of a morphism $f: W_1^\even\oplus W_1^\odd\rightarrow W_2^\even\oplus W_2^\odd$ in $\sC$, then $(C,0)$ is a cokernel of $f$ in $\sC$, with cokernel morphism
  \begin{equation*}
    W_2^\even\oplus W_2^\odd\xrightarrow{c} C\xrightarrow{q_C} C\oplus 0.
  \end{equation*}
 \end{enumerate}
\end{propo}

Note that the structure morphisms $p_i$ and $q_i$ for a biproduct in $\sC$ can be taken to be even, but the same is not necessarily the case for kernel and cokernel morphisms. However, we do have:
\begin{propo}\label{evenkercoker}
If $f: W_1\rightarrow W_2$ is a parity-homogeneous morphism in $\sC$, then its kernel and cokernel morphisms can be taken to be even. Moreover, every parity-homogeneous monomorphism in $\sC$ is the kernel of an even morphism, and every parity-homogeneous epimorphism in $\sC$ is the cokernel of an even morphism.
\end{propo}
\begin{proof}
Let $(K, k)$ be a kernel of $f$ in $\cC$ and consider the $\cC$-morphism $P_{W_1}\circ k: K\rightarrow W_1$. Since $f$ is parity-homogeneous, we have
\begin{equation*}
 f\circ P_{W_1}\circ k=(-1)^{\vert f\vert} P_{W_2}\circ f\circ k=0,
\end{equation*}
so the universal property of the kernel implies there is a unique morphism $P_K: K\rightarrow K$ such that
\begin{equation*}
 k\circ P_K=P_{W_1}\circ k.
\end{equation*}
Then set $\widetilde{K} =(K^\even, K^\odd)$ in $\sC$ where $(K^\even, k^\even)$ is a kernel of $\frac{1-P_K}{2}$ in $\cC$ and $(K^\odd, k^\odd)$ is a kernel of $\frac{1+P_K}{2}$ in $\cC$; also define $\widetilde{k}: \widetilde{K}\rightarrow W_1$ by
\begin{equation*}
 \widetilde{k}=k\circ(k^\even\circ p_{K^\even}+k^\odd\circ p_{K^\odd}).
\end{equation*}
We will show that $(\widetilde{K},\widetilde{k})$ is a kernel of $f$ in $\sC$ and that $\widetilde{k}$ is even.

First let us show that $\widetilde{k}$ is even. We have
\begin{align*}
 P_{W_1}\circ\widetilde{k} & = P_{W_1}\circ k\circ(k^\even\circ p_{K^\even}+k^\odd\circ p_{K^\odd}) = k\circ P_K\circ(k^\even\circ p_{K^\even}+k^\odd\circ p_{K^\odd})\nonumber\\
 & =k\circ(k^\even\circ p_{K^\even}- k^\odd\circ p_{K^\odd}) = k\circ(k^\even\circ p_{K^\even}+k^\odd\circ p_{K^\odd})\circ(q_{K^\even}\circ p_{K^\even}-q_{K^\odd}\circ p_{K^\odd}) =\widetilde{k}\circ P_{\widetilde{K}};
\end{align*}
in the third equality we have used the definition of $k^\even$, $k^\odd$ as kernel morphisms. Thus $\widetilde{k}$ is even.

To show that $(\widetilde{K}, \widetilde{k})$ is a kernel of $f$, it is enough to show that $k^\even\circ p_{K^\even}+k^\odd\circ p_{K^\odd}: K^\even\oplus K^\odd\rightarrow K$ is an isomorphism (in $\cC$). To construct the inverse, note that since $\frac{1+P_K}{2}\circ\frac{1-P_K}{2}=\frac{1-P_K}{2}\circ\frac{1+P_K}{2}=0$, the universal properties of the kernels $k^i$, $i=\even,\odd$, imply that there are morphisms $f^i: K\rightarrow K^i$, $i=\even,\odd$, such that
\begin{equation*}
 k^\even\circ f^\even=\dfrac{1+P_K}{2}\hspace{2em}\mathrm{and}\hspace{2em} k^\odd\circ f^\odd=\dfrac{1-P_K}{2}.
\end{equation*}
We show that $q_{K^\even}\circ f^\even+q_{K^\odd}\circ f^\odd$ is the inverse of $k^\even\circ p_{K^\even}+k^\odd\circ p_{K^\odd}$. In fact,
\begin{equation*}
 (k^\even\circ p_{K^\even}+k^\odd\circ p_{K^\odd})\circ(q_{K^\even}\circ f^\even+q_{K^\odd}\circ f^\odd) =k^\even\circ f^\even+k^\odd\circ f^\odd=\dfrac{1+P_K}{2}+\dfrac{1-P_K}{2}=1_K.
\end{equation*}
For the reverse composition, note that
\begin{align*}
 k^\even\circ f^\even\circ k^\even=\dfrac{1+P_K}{2}\circ k^\even =\left(1_K-\dfrac{1-P_K}{2}\right)\circ k^\even=k^\even;
\end{align*}
since $k^\even$ is a monomorphism, it follows that $f^\even\circ k^\even =1_{K^\even}$. Also,
\begin{equation*}
 k^\even\circ f^\even\circ k^\odd=\dfrac{1+P_K}{2}\circ k^\odd =0,
\end{equation*}
so that $f^\even\circ k^\odd=0$. Similarly, $f^\odd\circ k^\even=0$ and $f^\odd\circ k^\odd=1_{K^\odd}$. Thus
\begin{align*}
 (q_{K^\even}\circ f^\even +q_{K^\odd}\circ  f^\odd)\circ(k^\even\circ p_{K^\even}+k^\odd\circ p_{K^\odd})
 = q_{K^\even}\circ p_{K^\even}+q_{K^\odd}\circ p_{K^\odd}=1_{K^\even\oplus K^\odd}.
\end{align*}
This completes the proof that $(\widetilde{K}, \widetilde{k})$ is a kernel of $f$ in $\sC$ with $\widetilde{k}$ even.

The proof that $f$ has a cokernel $(\widetilde{C},\widetilde{c})$ in $\sC$ with $\widetilde{c}$ even is entirely analogous. Now suppose that $f$ is a parity-homogeneous monomorphism in $\sC$. The argument in the proof of Theorem \ref{thm:repAabelian} shows that $f$ is the kernel of its cokernel, which can be taken to be even. Similarly, if $f$ is a parity-homogeneous epimorphism in $\sC$, then $f$ is the cokernel of its kernel, which can be taken to be even.
\end{proof}	

\begin{rema}
 Propositions \ref{SCabelian} and \ref{evenkercoker} combined imply that the underlying category $\sCeven$ is abelian.
\end{rema}

We now give ${\sC}$ the structure of a braided tensor supercategory. On objects, the tensor product bifunctor $\boxtimes$ is given by $W_1\boxtimes W_2=((W_1\boxtimes W_2)^\even, (W_1\boxtimes W_2)^\odd)$, where \begin{equation*}
 (W_1\boxtimes W_2)^i=\bigoplus_{i_1+i_2=i} W_1^{i_1}\boxtimes W_2^{i_2}                                                                                                                                                                                                                                           \end{equation*}
for $i=\even, \odd$. The tensor product of morphisms $f_1: W_1\rightarrow\widetilde{W}_1$, $f_2: W_2\rightarrow\widetilde{W}_2$ needs to be a $\cC$-morphism:
\begin{equation*}
 f_1\boxtimes f_2: \bigoplus_{i_1, i_2\in\ztwo} W_1^{i_1}\boxtimes W_2^{i_2}\rightarrow\bigoplus_{i_1,i_2\in\ztwo} \widetilde{W}_1^{i_1}\boxtimes\widetilde{W}_2^{i_2}.
\end{equation*}
Since tensor products in $\cC$ distribute over biproducts (Proposition \ref{distribute}), we can identify $f_1\boxtimes f_2$ as a morphism $(W_1^\even\oplus W_1^\odd)\boxtimes(W_2^\even\oplus W_2^\odd)\rightarrow(\widetilde{W}_1^\even\oplus \widetilde{W}_1^\odd)\boxtimes(\widetilde{W}_2^\even\oplus \widetilde{W}_2^\odd)$. If $f_2$ is parity-homogeneous, we define this morphism to be
 $(f_1\circ P_{W_1}^{\vert f_2\vert})\boxtimes f_2$, where $\boxtimes$ here is the tensor product of morphisms in $\cC$. That is,
 \begin{equation*}
  f_1\boxtimes_{\sC} f_2=F\circ\left((f_1\circ P_{W_1}^{\vert f_2\vert})\boxtimes_\cC f_2\right)\circ F^{-1}
 \end{equation*}
with $F$ as in Proposition \ref{distribute}. If $f_2$ is not necessarily homogeneous, we define
\begin{equation*}
 f_1\boxtimes f_2=f_1\boxtimes f_2^\even+f_1\boxtimes f_2^\odd.
\end{equation*}
Then it is clear that the tensor product of morphisms in $\sC$ is bilinear.

Note the factor of $P_{W_1}^{\vert f_2\vert}$ in the definition of tensor product of morphisms in $\sC$. This is needed to prove  the super interchange law \eqref{exchange} in the following proposition: 
\begin{propo}\label{SCtenseven}
The tensor product $\boxtimes: \sC\times\sC\rightarrow\sC$ is a superfunctor.
\end{propo}
\begin{proof}
We need to show that $\boxtimes$ induces an even linear map on morphisms and that tensor products of morphisms satisfy the  super interchange law \eqref{exchange}.

To show that
\begin{equation*}
 \boxtimes: \hom_{\sC}(W_1, \widetilde{W}_1)\otimes\hom_{\sC}(W_2,\widetilde{W}_2)\rightarrow\hom_{\sC}(W_1\boxtimes W_2,\widetilde{W}_1\boxtimes\widetilde{W}_2)
\end{equation*}
is an even linear map between superspaces, we need to show that if $f_1: W_1\rightarrow\widetilde{W}_1$ and $f_2: W_2\rightarrow\widetilde{W}_2$ are parity-homogeneous morphisms in $\sC$, then $\vert f_1\boxtimes f_2\vert =\vert f_1\vert+\vert f_2\vert$. Using the definitions and Proposition \ref{distribute},
\begin{align*}
 P_{\widetilde{W}_1\boxtimes\widetilde{W_2}} &\circ(f_1\boxtimes_{\sC} f_2) =\left(\sum_{i_1,i_2\in\ztwo} (-1)^{i_1+i_2} q_{\widetilde{W}_1^{i_1}\boxtimes\widetilde{W}_2^{i_2}}\circ p_{\widetilde{W}_1^{i_1}\boxtimes\widetilde{W}_2^{i_2}}\right)\circ F\circ\left((f_1\circ P_{W_1}^{\vert f_2\vert})\boxtimes_{\cC} f_2\right)\circ F^{-1}\nonumber\\
 & =\sum_{i_1,i_2\in\ztwo} (-1)^{i_1+i_2} q_{\widetilde{W}_1^{i_1}\boxtimes\widetilde{W}_2^{i_2}}\circ\left((p_{\widetilde{W}_1^{i_1}}\circ f_1\circ P_{W_1}^{\vert f_2\vert})\boxtimes_{\cC} (p_{\widetilde{W}_2^{i_2}}\circ f_2)\right)\circ F^{-1}\nonumber\\
& =F\circ\left(\sum_{i_1\in\ztwo} (-1)^{i_1}(q_{\widetilde{W}_1^{i_1}}\circ p_{\widetilde{W}_1^{i_1}}\circ f_1\circ P_{W_1}^{\vert f_2\vert})\boxtimes_{\cC}\sum_{i_2\in\ztwo} (-1)^{i_2}(q_{\widetilde{W}_2^{i_2}}\circ p_{\widetilde{W}_2^{i_2}}\circ f_2)\right)\circ F^{-1}\nonumber\\
& =F\circ\left( (P_{\widetilde{W}_1}\circ f_1\circ P_{W_1}^{\vert f_2\vert})\boxtimes_{\cC} (P_{\widetilde{W}_2}\circ f_2)\right)\circ F^{-1}\nonumber\\
& =(-1)^{\vert f_1\vert+\vert f_2\vert} F\circ\left((f_1\circ P_{W_1}^{\vert f_2\vert})\boxtimes_{\cC} f_2\right)\circ(P_{W_1}\boxtimes_{\cC} P_{W_2})\circ F^{-1}\nonumber\\
& =(-1)^{\vert f_1\vert+\vert f_2\vert} F\circ\left((f_1\circ P_{W_1}^{\vert f_2\vert})\boxtimes_{\cC} f_2\right)\circ F^{-1}\circ P_{W_1\boxtimes W_2}\nonumber\\
& =(-1)^{\vert f_1\vert+\vert f_2\vert}(f_1\boxtimes_{\sC} f_2)\circ P_{W_1\boxtimes W_2},
\end{align*}
as desired.

Now suppose $f_1: W_2\rightarrow W_3$, $f_2: W_1\rightarrow W_2$, $g_1: X_2\rightarrow X_3$, and $g_2: X_1\rightarrow X_2$ are morphisms in $\sC$ with $g_1$ and $f_2$ homogeneous. To verify the  super interchange law, we may assume that $g_2$ is also homogeneous, and then
\begin{align*}
 (f_1\boxtimes_{\sC} g_1)\circ(f_2\boxtimes_{\sC} g_2) & =F\circ((f_1\circ P_{W_2}^{\vert g_1\vert})\boxtimes_\cC g_1)\circ((f_2\circ P_{W_1}^{\vert g_2\vert})\boxtimes_{\cC} g_2)\circ F^{-1}\nonumber\\
 & =F\circ((f_1\circ P_{W_2}^{\vert g_1\vert}\circ f_2\circ P_{W_1}^{\vert g_2\vert})\boxtimes_{\cC}(g_1\circ g_2))\circ F^{-1}\nonumber\\
 & =((-1)^{\vert f_2\vert})^{\vert g_1\vert} F\circ((f_1\circ f_2\circ P_{W_1}^{\vert g_1\vert+\vert g_2\vert})\boxtimes_{\cC} (g_1\circ g_2))\circ F^{-1}\nonumber\\
 & =(-1)^{\vert g_1\vert \vert f_2\vert} (f_1\circ f_2)\boxtimes_{\sC}(g_1\circ g_2)
\end{align*}
as desired.
\end{proof}
\begin{rema}
 Proposition \ref{SCtenseven} shows that the tensor product of even morphisms in $\sC$ is even, so that $\boxtimes$ defines a functor on $\sCeven$ as well.
\end{rema}

Now we can define the rest of the braided tensor supercategory structure on $\sC$:
\begin{enumerate}
 \item The unit object in $\sC$ is $\sunit=(\unit, 0)$, and for any object $W$ in $\sC$, the left unit isomorphism is $\sleft=l_{W^\even}\oplus l_{W^\odd}$ and the right unit isomorphism is $\sright=r_{W^\even}\oplus r_{W^\odd}$.
 
 \item For objects $W_1$, $W_2$, and $W_3$ in $\sC$, the associativity isomorphism is
 \begin{equation*}
  \sA_{W_1,W_2,W_3}=\left(\bigoplus_{i_1+i_2+i_3=\even} \cA_{W_1^{i_1}, W_2^{i_2}, W_3^{i_3}}\right)\oplus\left(\bigoplus_{i_1+i_2+i_3=\odd} \cA_{W_1^{i_1}, W_2^{i_2}, W_3^{i_3}}\right).
 \end{equation*}
 
 \item For objects $W_1$ and $W_2$ in $\sC$, the braiding isomorphism is
 \begin{equation*}
  \sR_{W_1,W_2} =\left(\bigoplus_{i_1+i_2=\even}(-1)^{i_1 i_2}\cR_{W_1^{i_1},W_2^{i_2}}\right)
\oplus 	\left(\bigoplus_{i_1+i_2=\odd}(-1)^{i_1 i_2}\cR_{W_1^{i_1},W_2^{i_2}}\right)
 \end{equation*}

\end{enumerate}

\begin{rema}\label{rema:SCdata}
 It is straightforward to verify, using Proposition \ref{distribute} and its notation, that, if we take $\mathbf{1}\oplus 0$ in $\cC$ to be $\mathbf{1}$ with $q_\mathbf{1}=p_\mathbf{1}=1_\mathbf{1}$ and $q_0=p_0=0$, then
 \begin{equation*}
  \sleft_W = l_{W^\even\oplus W^\odd}\circ F^{-1}
 \end{equation*}
and
\begin{equation*}
\sright_W=r_{W^\even\oplus W^\odd}\circ F^{-1}
\end{equation*}
for any object $W$ in $\sC$. Moreover, for any objects $W_1$, $W_2$, $W_3$ in $\sC$,
\begin{equation*}
 \sA_{W_1,W_2,W_3}=F\circ(F\boxtimes 1_{W_3})\circ\cA_{W_1^\even\oplus W_1^\odd, W_2^\even\oplus W_2^\odd, W_3^\even\oplus W_3^\odd}\circ(1_{W_1}\boxtimes F^{-1})\circ F^{-1},
\end{equation*}
and
\begin{equation*}
 \sR_{W_1,W_2}=F\circ Q_{W_2,W_1}\circ\cR_{W_1^\even\oplus W_1^\odd,W_2^\even\oplus W_2^\odd}\circ F^{-1},
\end{equation*}
where
\begin{equation*}
 Q_{W_2,W_1}=\sum_{i_1,i_2\in\ztwo} (-1)^{i_1 i_2} (q_{W_2^{i_2}}\circ p_{W_2^{i_2}})\boxtimes (q_{W_1^{i_1}}\circ p_{W_1^{i_1}}).
\end{equation*}
Note that $Q_{W_2, W_1}$ would be the identity if it were not for the $(-1)^{i_1 i_2}$ factor.
\end{rema}

\begin{theo}\label{propo:superCisbraided}
The supercategory $\sC$ together with $\boxtimes$, $\sunit$, $\sleft$, $\sright$, $\sA$, and $\sR$ is an $\mathbb{F}$-linear braided tensor supercategory. 
\end{theo}
\begin{proof}
It is straightforward to use the definitions, Remark \ref{rema:SCdata}, and the fact that $\cC$ is a tensor category to verify that the unit and associativity isomorphisms in $\sC$ are natural and satisfy the triangle and pentagon axioms. However, the braiding isomorphisms are complicated by the extra $(-1)^{i_1 i_2}$ factor.

We need to show that the braiding isomorphisms are natural in the sense of \eqref{braidingsupernatural}. In fact,
\begin{align}\label{bsncalc}
 \sR_{\widetilde{W}_1, \widetilde{W}_2} & \circ(f_1\boxtimes_{\sC} f_2)  =F\circ Q_{\widetilde{W}_2, \widetilde{W}_1}\circ\cR_{\widetilde{W}_1^\even\oplus\widetilde{W}_1^\odd, \widetilde{W}_2^\even\oplus\widetilde{W}_2^\odd}\circ F^{-1}\circ F\circ\left((f_1\circ P_{W_1}^{\vert f_2\vert})\boxtimes_\cC f_2\right)\circ F^{-1}\nonumber\\
 & =F\circ Q_{\widetilde{W}_2,\widetilde{W}_1}\circ\left(f_2\boxtimes_\cC (f_1\circ P_{W_1}^{\vert f_2\vert})\right)\circ\cR_{W_1^\even\oplus W_1^\odd, W_2^\even\oplus W_2^\odd}\circ F^{-1}
\end{align}
by the naturality of the braiding isomorphisms in $\cC$. Now, we can write 
\begin{equation*}
 f_2=\sum_{j_2+j_2'=\vert f_2\vert} q_{\widetilde{W}_2^{j_2'}}\circ f^{(2)}_{j_2',j_2}\circ p_{W_2^{j_2}},
\end{equation*}
where $f^{(2)}_{j_2', j_2}= p_{\widetilde{W}_2^{j_2'}}\circ f\circ q_{W_2^{j_2}}$ is a morphism in $\cC$ from $W_2^{j_2}$ to $\widetilde{W}_2^{j_2'}$. Similarly, 
\begin{align*}
 f_1\circ P_{W_1}^{\vert f_2\vert} &=\left(\sum_{j_1+j_1'=\vert f_1\vert} q_{\widetilde{W}_1^{j_1'}}\circ f^{(1)}_{j_1',j_1}\circ p_{W_1^{j_1}}\right)\circ\left(\sum_{k\in\ztwo} (-1)^{k\vert f_2\vert} q_{W_1^k}\circ p_{W_1^k}\right)\nonumber\\
 &=\sum_{j_1+j_1'=\vert f_1\vert} (-1)^{j_1\vert f_2\vert} q_{\widetilde{W}_1^{j_1'}}\circ f^{(1)}_{j_1',j_1}\circ p_{W_1^{j_1}}.
\end{align*}
Thus
\begin{align*}
 Q_{\widetilde{W}_2,\widetilde{W}_1} & \circ\left(f_2\boxtimes (f_1\circ P_{W_1}^{\vert f_2\vert})\right)\nonumber\\
 &= \sum_{\stackrel{  i_1,i_2\in\ztwo}{j_1+j_1'=\vert f_1\vert,\,\,j_2+j_2'=\vert f_2\vert}} (-1)^{i_1 i_2+j_1\vert f_2\vert}\cdot\nonumber\\
 &\hspace{5em}\cdot\left(q_{\widetilde{W}_2^{i_2}}\circ p_{\widetilde{W}_2^{i_2}}\circ q_{\widetilde{W}_2^{j_2'}}\circ f^{(2)}_{j_2',j_2}\circ p_{W_2^{j_2}}\right)\boxtimes\left(q_{\widetilde{W}_1^{i_1}}\circ p_{\widetilde{W}_1^{i_1}}\circ q_{\widetilde{W}_1^{j_1'}}\circ f^{(1)}_{j_1',j_1}\circ p_{W_2^{j_1}}\right)\nonumber\\
 & =\sum_{\stackrel{i_1+j_1=\vert f_1\vert}{i_2+j_2=\vert f_2\vert}} (-1)^{i_1 i_2+j_1\vert f_2\vert}\left(q_{\widetilde{W}_2^{i_2}}\circ f^{(2)}_{i_2,j_2}\circ p_{W_2^{j_2}}\right)\boxtimes\left(q_{\widetilde{W}_1^{i_1}}\circ  f^{(1)}_{i_1,j_1}\circ p_{W_2^{j_1}}\right)\nonumber\\
 & =\sum_{\stackrel{i_1+j_1=\vert f_1\vert}{i_2+j_2=\vert f_2\vert}} (-1)^{j_1 j_2+i_2\vert f_1\vert}\left(q_{\widetilde{W}_2^{i_2}}\circ f^{(2)}_{i_2,j_2}\circ p_{W_2^{j_2}}\right)\boxtimes\left(q_{\widetilde{W}_1^{i_1}}\circ  f^{(1)}_{i_1,j_1}\circ p_{W_2^{j_1}}\right)\nonumber\\
 & =\sum_{\stackrel{i_1+i_1'=\vert f_1\vert}{i_2+i_2'=\vert f_2\vert}} (-1)^{i_2\vert f_1\vert}\left(q_{\widetilde{W}_2^{i_2}}\circ f^{(2)}_{i_2,i_2'}\circ p_{W_2^{i_2'}}\right)\boxtimes\left(q_{\widetilde{W}_1^{i_1}}\circ  f^{(1)}_{i_1,i_1'}\circ p_{W_2^{i_1'}}\right)\circ\nonumber\\
 & \hspace{5em}\circ\sum_{j_1,j_2\in\ztwo} (-1)^{j_1 j_2} (q_{W_2^{j_2}}\circ p_{W_2^{j_2}})\boxtimes(q_{W_1^{j_1}}\circ p_{W_1^{j_1}})\nonumber\\
 & =\sum_{i_2+i_2'=\vert f_2\vert} (-1)^{(i_2'+\vert f_2\vert)\vert f_1\vert}\left(\left(q_{\widetilde{W}_2^{i_2}}\circ f^{(2)}_{i_2,i_2'}\circ p_{W_2^{i_2'}}\right)\boxtimes f_1\right)\circ Q_{W_2, W_1}\nonumber\\
 & =(-1)^{\vert f_1\vert \vert f_2\vert} \left((f_2\circ P_{W_2}^{\vert f_1\vert})\boxtimes f_1\right)\circ Q_{W_2,W_1}.
\end{align*}
Putting this into \eqref{bsncalc} and using the definitions quickly yields \eqref{braidingsupernatural}.

Now to check the hexagon diagram for $\sR$, take objects $W_1$, $W_2$, and $W_3$ in $\sC$. For $i_1, i_2, i_3\in \mathbb{Z}/2\mathbb{Z}$, we have
	the following commuting diagram thanks to the hexagon axiom for $\cR$:
	\begin{align}
	\xymatrixcolsep{6pc}
	\xymatrix{
		W_1^{i_1}\boxtimes(W_2^{i_2}\boxtimes W_3^{i_3})\ar[r]^{(-1)^{i_1(i_2+i_3)}\cR} \ar[d]^{\cA}
		& (W_2^{i_2}\boxtimes W_3^{i_3})\boxtimes W_1^{i_1} \ar[r]^{\cA^{-1}} & W_2^{i_2}\boxtimes(W_3^{i_3}\boxtimes W_1^{i_1})\\
		(W_1^{i_1}\boxtimes W_2^{i_2})\boxtimes W_3^{i_3}\ar[r]^{(-1)^{i_1 i_2}\cR\boxtimes 1}  & (W_2^{i_2}\boxtimes W_1^{i_1})\boxtimes W_3^{i_3}\ar[r]^{\cA^{-1}} &
		W_2^{i_2}\boxtimes(W_1^{i_1}\boxtimes W_3^{i_3})	\ar[u]^{1\boxtimes (-1)^{i_1 i_3}\cR}	},
	\end{align}
	and a similar diagram involving $\cR^{-1}$. These immediately yield the hexagon axioms for $\sR$.
\end{proof}

\begin{rema}
 Note the essential role of the parity factor in the definition of the tensor product of morphisms in $\sC$ for proving the naturality of $\sR$. The naturality of $\sR$ says that the tensor product of morphisms in $\sC$ supercommutes up to conjugation by braiding isomorphisms.
\end{rema}

\begin{rema}
 Since all the structure morphisms in $\sC$ are even, the braided tensor supercategory structure on $\sC$ reduces to a braided tensor category structure on $\sCeven$. In particular, the supercategory naturality of $\sR$ reduces to ordinary naturality when both morphisms $f_1$ and $f_2$ are even.
\end{rema}

We will need tensoring in $\sC$ to be right exact, just as it is in $\cC$:
\begin{propo}
For any object $W$ in $\sC$, the functors $W\boxtimes\cdot$ and $\cdot\boxtimes W$ are right exact.
\end{propo}
\begin{proof}
First consider $\cdot\boxtimes W$. Suppose $W_1\xrightarrow{f_1} W_2\xrightarrow{f_2} W_3\rightarrow 0$ is a right exact sequence in $\sC$, that is, $(W_3, f_2)$ is a cokernel of $f_1$. We need to verify that $(W_3\boxtimes_\sC W, f_2\boxtimes_\sC 1_W)$ is a cokernel of $f_1\boxtimes_\sC 1_W$. Since $f_i\boxtimes_\sC 1_W=F\circ(f_i\boxtimes_\cC 1_{W^\even\oplus W^\odd})\circ F^{-1}$ for $i=1,2$, using the notation of Proposition \ref{distribute} (since $1_W$ is even), it is enough to show that $((W_2^\even\oplus W_2^\odd)\boxtimes(W^\even\oplus W^\odd), f_2\boxtimes_\cC 1_{W^\even\oplus W^\odd})$ is a cokernel of $f_1\boxtimes_\cC 1_{W^\even\oplus W^\odd}$ in $\cC$. But this is immediate because the tensor functor on $\cC$ is right exact by assumption.

For $W\boxtimes\cdot$, right exactness follows from the right exactness of $\cdot\boxtimes W$ and the naturality of the braiding isomorphisms (which reduces to ordinary naturality in this case since $1_W$ is even).
\end{proof}

\begin{rema}
 It is not hard to show that the tensor product functor in $\sCeven$ is also right exact.
\end{rema}

Now that we have a braided tensor supercategory associated to $\cC$ with right exact tensor functor, we can introduce the notion of supercommutative associative superalgebra in $\cC$:
\begin{defi}\label{defi:SA}
A \textit{supercommutative associative superalgebra} 
(or \emph{superalgebra} for short) in $\mathcal{C}$ is a commutative associative algebra in $\sC$ whose structure morphisms are even. Explicitly, a superalgebra in $\cC$ is an object
	$A=(A^\even,A^\odd)$ of $\sC$ equipped with even morphisms $\mu: A\boxtimes A\rightarrow A$ and $\iota_A: \sunit \rightarrow A$ in $\sC$ satisfying the following axioms:
	\begin{enumerate}
		\item Associativity: $\mu\circ(\mu\boxtimes 1_A)\circ\sA_{A,A,A}= \mu\circ(1_A\boxtimes\mu)$ as morphisms $A\boxtimes (A\boxtimes A)\rightarrow A$.
		\item Supercommutativity: $\mu\circ \sR_{A,A}=\mu$ as morphisms $A\boxtimes A\rightarrow A$.
		\item Unit: $\mu\circ(\iota_A\boxtimes 1_A)\circ \sleft_A^{-1}=1_A$ as morphisms $A\rightarrow A$.
	\end{enumerate}
\end{defi}

\begin{rema}\label{rightunitSA}
	The supercommutativity and unit axioms of the above definition (together with the naturality of $\sR$ with respect to even morphisms) imply the right unit property: $\mu\circ(1_A\boxtimes\iota_A)\circ \sright_A^{-1}=1_A$.
\end{rema}

\begin{rema}
 Since the structure morphisms of a superalgebra in $\cC$ are even, superalgebras in $\cC$ are precisely commutative associative algebras in the braided tensor category $\sCeven$.
\end{rema}

\begin{rema}
 If $\cC$ is a braided tensor category of modules for a vertex operator algebra $V$ and $A$ is a superalgebra in $\cC$, recall the twist $\cC$-automorphism $\theta_A=e^{2\pi i L(0)}$ of $A$. It is shown in \cite{CKL}, slightly generalizing the results of \cite{HKL, H-two}, that superalgebras $A$ in $\cC$ which additionally satisfy:
 \begin{enumerate}
  \item[4.] $\iota_A$ is injective;
  \item[5.] $A$ is $\frac{1}{2}\ZZ$-graded by conformal weights: $\theta_A^2=1_A$; and
  \item[6.] $\mu$ has no monodromy: $\mu\circ(\theta_A\boxtimes\theta_A)=\theta_{A}\circ\mu$
 \end{enumerate}
are precisely vertex operator superalgebra extensions of $V$ in $\cC$ such that $V\subseteq A^\even$.
\end{rema}

Since a superalgebra $A$ in $\cC$ is an algebra in $\sC$ with even structure morphisms, we can define the category $\repA$ of ``$A$-modules in $\cC$'' to be the category of ``$A$-modules in $\sC$'' with even structure morphisms: 
\begin{defi}\label{repSAdef}
	Given a superalgebra $A$ in $\cC$, $\mathrm{Rep}\,A$ is the category with
	objects $(W, \mu_W)$ where $W$ is an object of  $\sC$, 
	and $\mu_W: A\boxtimes W\rightarrow W$ is an even morphism satisfying:
	\begin{enumerate}
\item Associativity: $\mu_W\circ(\mu\boxtimes 1_W)\circ\sA_{A,A,W}=\mu_W\circ (1_A\boxtimes\mu_W)$ as morphisms  $A\boxtimes (A\boxtimes W)\rightarrow W$.
\item Unit: $\mu_W\circ(\iota_A\boxtimes 1_W)\circ \sleft_W^{-1}=1_W$ as morphisms $W\rightarrow W$.
	\end{enumerate}
	The morphisms between objects $(W_1,\mu_{W_1})$ and $(W_2, \mu_{W_2})$ of $\mathrm{Rep}\,A$ consist of all $\sC$-morphisms $f: W_1\rightarrow W_2$ such that $f\circ\mu_{W_1} =\mu_{W_2}\circ(1_A\boxtimes f): A\boxtimes W_1\rightarrow W_2$. 
\end{defi}

\begin{rema}
If $A$ is a superalgebra in $\cC$, let $\underline{\repA}$ denote the category of $A$-modules where $A$ is considered as an algebra in $\sCeven$, as defined in the previous subsection. Then $\underline{\repA}$ contains exactly the same objects as $\repA$, but $\repA$ contains more morphisms since $\underline{\repA}$ has only even morphisms.
\end{rema}

\begin{rema}\label{rem:RepAnotab}
Since a superalgebra $A$ in $\cC$ is an algebra in the braided tensor category $\sCeven$, Theorem \ref{thm:repAabelian} shows that the category $\underline{\repA}$ is abelian. However, we cannot guarantee that the larger category $\repA$ is abelian since, although Theorem \ref{thm:repAabelian} shows that kernels and cokernels of not necessarily parity-homogeneous morphisms in $\repA$ are equipped with $A$-actions that satisfy associativity and the unit property, we cannot guarantee that these $A$-actions are even.
\end{rema}

However, we do have:
\begin{propo}\label{SuperRepAadditive}
 If $A$ is a superalgebra in $\cC$, then $\repA$ is an $\mathbb{F}$-linear additive supercategory. Moreover, parity-homogeneous morphisms in $\repA$ have kernels and cokernels, every parity-homogeneous monomorphism in $\repA$ is a kernel in $\repA$, and every parity-homogeneous epimorphism is a cokernel in $\repA$.
\end{propo}
\begin{proof}
 The $\mathbb{F}$-vector space structure on morphisms in $\repA$ is inherited from $\sC$. To show that morphisms in $\repA$ have superspace structure, we need to show that if $f: W_1\rightarrow W_2$ is a morphism in $\repA$, then its parity-homogeneous components $f^\even$ and $f^\odd$ are also $\repA$-morphisms. Since $f$ is a morphism in $\repA$, we have $f\circ\mu_{W_1}=\mu_{W_2}\circ(1_A\boxtimes f)$. Then
 \begin{equation*}
  f^\even\circ\mu_{W_1}+f^\odd\circ\mu_{W_2}=\mu_{W_2}\circ(1_A\boxtimes f^\even)+\mu_{W_2}\circ(1_A\boxtimes f^\odd).
 \end{equation*}
Since $\mu_{W_1}$ and $\mu_{W_2}$ are both even, matching even and odd components of the above equation shows that $f^\even$ and $f^\odd$ are $\repA$-morphisms. The evenness and bilinearity of composition of morphisms in $\repA$ is then inherited from $\sC$, so $\repA$ is an $\mathbb{F}$-linear supercategory.

 The existence of a zero object in $\repA$ is clear, and the existence of biproducts in $\repA$ is the same as in the proof of Theorem \ref{thm:repAabelian}, except that we need to additionally show that the $A$-action $\mu_{\bigoplus W_i}$ on a biproduct $\bigoplus W_i$ is even. Since $\mu_{\bigoplus W_i}=\sum q_i\circ \mu_{W_i}\circ(1_A\boxtimes p_i)$ is the composition of even morphisms, this is clear. Thus $\repA$ is an $\mathbb{F}$-linear additive supercategory.
 
 Now we need to show that if $f: W_1\rightarrow W_2$ is a parity-homogeneous morphism in $\repA$, then $f$ has a cokernel and kernel in $\repA$. We know from Theorem \ref{thm:repAabelian} that $f$ has a kernel $(K,k)$ and a cokernel $(C,c)$ in $\sC$ equipped with $A$-actions $\mu_K$ and $\mu_C$ satisfying the associativity and unit properties. It will follow that $(K,\mu_K, k)$ and $(C,\mu_C, c)$ are the kernel and cokernel of $f$ in $\repA$ if $\mu_K$ and $\mu_C$ are even. Since $f$ is parity-homogeneous, Proposition \ref{evenkercoker} implies we may assume $k$ and $c$ are even. To show that $\mu_C$ is even, recall from Theorem \ref{thm:repAabelian} that $\mu_C\circ(1_A\boxtimes c)=c\circ\mu_{W_2}$, so that
 \begin{align*}
  P_C\circ\mu_C\circ(1_A\boxtimes c)=P_C\circ c\circ\mu_{W_2}=c\circ\mu_{W_2}\circ P_{A\boxtimes W_2}=\mu_C\circ(1_A\boxtimes c)\circ P_{A\boxtimes W_2}=\mu_C\circ P_{A\boxtimes C}\circ(1_A\boxtimes c).
 \end{align*}
Since $c$ is an epimorphism and $A\boxtimes\cdot$ in $\sC$ is right exact, $1_A\boxtimes c$ is an epimorphism and it follows that $P_C\circ\mu_C=\mu_C\circ P_{A\boxtimes C}$. Thus $\mu_C$ is even. Similarly, $\mu_K$ is even.

Finally, we need to verify that every parity-homogeneous monomorphism in $\repA$ is a kernel and every parity-homogeneous epimorphism in $\repA$ is a cokernel. In fact, the proof of Theorem \ref{thm:repAabelian} shows that a parity-homogeneous monomorphism in $\repA$ is the kernel of its cokernel, and a parity-homogeneous epimorphism in $\repA$ is the cokernel of its kernel.
\end{proof}

\begin{rema}
 In the following subsections where we construct a tensor product on $\repA$, we shall only need the existence of cokernels of certain even morphisms in $\repA$.
\end{rema}

\subsection{Tensor products and unit isomorphisms in \texorpdfstring{$\repA$}{Rep A}}

Here we fix a superalgebra $A$ in our $\mathbb{F}$-linear braided tensor category $\cC$ and construct a tensor product $\boxtimes_A: \repA\times\repA\rightarrow\repA$ as well as left and right unit isomorphisms. The construction of $\boxtimes_A$ is based on \cite{KO}; later, we shall show that this tensor product is characterized by a universal property analogous to the defining property of tensor products of vertex operator algebra modules. Although we work in the generality of superalgebras, all results and constructions apply as well to the category $\repA$ of modules for an ordinary commutative associative algebra in $\cC$; the only difference is that for an ordinary algebra we do not have to worry about evenness of certain morphisms or certain sign factors (as in the super interchange law). In particular, the following results and constructions  apply to algebras in $\sCeven$ since $\sCeven$ is a braided tensor category with right exact tensor functor.

\begin{nota}
  In this and the following subsections, we will frequently give summaries of lengthy proofs in braided
  tensor (super)categories using graphical calculus. We shall denote the unit
  $\bf 1$ (or $\sunit$) by dashed lines, the (super)algebra $A$ by thin strands, objects of
  $\cC$ or $\sC$ by thick strands, and morphisms and multiplications
  $\mu_\bullet$ by coupons. We read diagrams from bottom to top: for instance, 
  in the composition $f\circ g$, the morphism $g$ appears at the bottom.
\end{nota}

\begin{defi}\label{def:mu1mu2}
For objects $(W_1,\mu_{W_1})$ and $(W_2,\mu_{W_2})$ of
$\mathrm{Rep}\,A$, we define $\mu^{(1)}$ to be the composition
\begin{equation*}
A\boxtimes (W_1\boxtimes W_2)\xrightarrow{\sA_{A,W_1,W_2}}(A\boxtimes W_1)\boxtimes W_2\xrightarrow{\mu_{W_1}\boxtimes 1_{W_2}} W_1\boxtimes W_2
\end{equation*}
and we define $\mu^{(2)}$ to be the composition
\begin{align*}
A\boxtimes(W_1\boxtimes W_2)\xrightarrow{\sA_{A, W_1, W_2}} (A\boxtimes W_1)\boxtimes W_2 & \xrightarrow{\sR_{A,W_1}\boxtimes 1_{W_2}} (W_1\boxtimes A)\boxtimes W_2\nonumber\\
& \xrightarrow{\sA_{W_1,A, W_2}^{-1}} W_1\boxtimes(A\boxtimes W_2)\xrightarrow{1_{W_1}\boxtimes\mu_{W_2}} W_1\boxtimes W_2.
\end{align*}
\end{defi}

\begin{rema} \label{rem:KOmu2}
	By the hexagon axiom for a braided tensor (super)category and naturality of the braiding for even morphisms, $\mu^{(2)}$ is also given by the composition
	\begin{align}
	A\boxtimes (W_1\boxtimes W_2)\xrightarrow{1_A\boxtimes \sR_{W_2,W_1}^{-1}} A\boxtimes (W_2\boxtimes W_1)& \xrightarrow{\sA_{A,W_2,W_1}} (A\boxtimes W_2)\boxtimes W_1\nonumber\\
&\xrightarrow{\mu_{W_2}\boxtimes 1_{W_1}} W_2\boxtimes W_1\xrightarrow{\sR_{W_2,W_1}} W_1\boxtimes W_2.\label{eqn:mu2alt}
	\end{align} 
We shall frequently use this alternate formulation of $\mu^{(2)}$.
\end{rema}

\begin{lemma}\label{muiassoc}
	For $i=1,2$, we have
	\begin{equation*}
	\mu^{(i)}\circ(1_A\boxtimes\mu^{(i)})=\mu^{(i)}\circ(\mu\boxtimes 1_{W_1\boxtimes W_2})\circ\sA_{A,A,W_1\boxtimes W_2}
	\end{equation*}
	as morphisms $A\boxtimes (A\boxtimes (W_1\boxtimes W_2))\rightarrow W_1\boxtimes W_2$.
\end{lemma}
\begin{proof}
	We indicate the proof for $\mu^{(2)}$ using graphical calculus; the proof for $\mu^{(1)}$ is similar but simpler.
	 In the first diagram, we have used Remark \ref{rem:KOmu2} to rewrite $\mu^{(2)}$, and
	in the last step we exchange the order of $\mu$ and the braiding without a sign factor because they are even:
	\begin{align*}
	\begin{matrix}
		\begin{tikzpicture}[scale=0.8, out=up, in=down, line width=0.5pt]
			\node (mub) at (0.5,1.1) [draw] {$\scriptstyle{\mu_{W_2}}$};
			\node (mut) at (-0.5,4) [draw,minimum width=30pt,minimum height=10pt,thick, fill=white] {$\scriptstyle{\mu_{W_2}}$};
			\node (a1) at (-1,-0.3) {$\scriptstyle{A}$};
			\node (a2) at (0,-0.3) {$\scriptstyle{A}$};
			\node (w1) at (1,-0.3) {$\scriptstyle{W_1}$};
			\node (w2) at (2,-0.3) {$\scriptstyle{W_2}$};
			\node (w1f) at (0,5.8) {$\scriptstyle{W_1}$};
			\node (w2f) at (2,5.8) {$\scriptstyle{W_2}$};
			\draw (a2.90) to (mub.210);
			\draw (a1.90) to (mut.210);
			\draw [line width=1pt] (w1.90) to (2,1) to (0.5,2.5) to (2,4) to (0,5.5);
			\draw [white, double=black, line width=3pt, double distance=1pt] (mub.90) to (0.5,1.5) to (2,2.5) to (mut.330);
			\draw [white, double=black, line width=3pt, double distance=1pt] (w2.90) to (mub.330);
			\draw [white, double=black, line width=3pt, double distance=1pt] (mut.90) to (w2f.270);
		\end{tikzpicture}
	\end{matrix}
	=
	\begin{matrix}
		\begin{tikzpicture}[scale=0.8, out=up, in=down, line width=0.5pt]
			\node (mub) at (0.5,1.5) [draw] {$\scriptstyle{\mu_{W_2}}$};
			\node (mut) at (-0.5,4) [draw,minimum width=30pt,minimum height=10pt,thick, fill=white] {$\scriptstyle{\mu_{W_2}}$};
			\node (a1) at (-1,-0.3) {$\scriptstyle{A}$};
			\node (a2) at (0,-0.3) {$\scriptstyle{A}$};
			\node (w1) at (1,-0.3) {$\scriptstyle{W_1}$};
			\node (w2) at (2,-0.3) {$\scriptstyle{W_2}$};
			\node (w1f) at (0,5.8) {$\scriptstyle{W_1}$};
			\node (w2f) at (2,5.8) {$\scriptstyle{W_2}$};
			\draw (a2.90) to (mub.210);
			\draw (a1.90) to (mut.210);
			\draw [line width=1pt] (w1.90) to (1.8,1.5) to (1.8,3.5) to (0,5.5);
			\draw [white, double=black, line width=3pt, double distance=1pt] (mub.90) to (mut.330);
			\draw [white, double=black, line width=3pt, double distance=1pt] (w2.90) to (mub.330);
			\draw [white, double=black, line width=3pt, double distance=1pt] (mut.90) to (w2f.270);
		\end{tikzpicture}
	\end{matrix}
	=
	\begin{matrix}
		\begin{tikzpicture}[scale=0.8, out=up, in=down, line width=0.5pt]
			\node (mub) at (-0.5,2) [draw] {$\scriptstyle{\mu_{W}}$};
			\node (mut) at (0,3.8) [draw,minimum width=30pt,minimum height=10pt,thick, fill=white] {$\scriptstyle{\mu_{W_2}}$};
			\node (a1) at (-1,-0.3) {$\scriptstyle{A}$};
			\node (a2) at (0,-0.3) {$\scriptstyle{A}$};
			\node (w1) at (1,-0.3) {$\scriptstyle{W_1}$};
			\node (w2) at (2,-0.3) {$\scriptstyle{W_2}$};
			\node (w1f) at (0,5.8) {$\scriptstyle{W_1}$};
			\node (w2f) at (2,5.8) {$\scriptstyle{W_2}$};
			\draw (a2.90) to (mub.330);
			\draw (a1.90) to (mub.210);
			\draw [line width=1pt] (w1.90) to (1.8,1.5) to (1.8,3.5) to (0,5.5);
			\draw  (mub.90) to (mut.210);
			\draw [white, double=black, line width=3pt, double distance=1pt] (w2.90) to (0.7,1.5) to (mut.330);
			\draw [white, double=black, line width=3pt, double distance=1pt] (mut.90) to (w2f.270);
		\end{tikzpicture}
	\end{matrix}
	=
	\begin{matrix}
		\begin{tikzpicture}[scale=0.8, out=up, in=down, line width=0.5pt]
			\node (mub) at (-0.5,1) [draw] {$\scriptstyle{\mu_{W}}$};
			\node (mut) at (0,3.8) [draw,minimum width=30pt,minimum height=10pt,thick, fill=white] {$\scriptstyle{\mu_{W_2}}$};
			\node (a1) at (-1,-0.3) {$\scriptstyle{A}$};
			\node (a2) at (0,-0.3) {$\scriptstyle{A}$};
			\node (w1) at (1,-0.3) {$\scriptstyle{W_1}$};
			\node (w2) at (2,-0.3) {$\scriptstyle{W_2}$};
			\node (w1f) at (0,5.8) {$\scriptstyle{W_1}$};
			\node (w2f) at (2,5.8) {$\scriptstyle{W_2}$};
			\draw (a2.90) to (mub.330);
			\draw (a1.90) to (mub.210);
			\draw [line width=1pt] (w1.90) to (1,1.5) to (2,3.5) to (0,5.5);
			\draw  (mub.90) to (mut.210);
			\draw [white, double=black, line width=3pt, double distance=1pt] (w2.90) to (2,1.5) to (mut.330);
			\draw [white, double=black, line width=3pt, double distance=1pt] (mut.90) to (w2f.270);
		\end{tikzpicture}
	\end{matrix}
	\end{align*}
\end{proof}

Since cokernels exist in $\sC$, we can define the tensor product of $W_1$ and $W_2$ in $\mathrm{Rep}\,A$ to be
\begin{equation*}
 W_1\boxtimes_A W_2=\mathrm{Coker}(\mu^{(1)}-\mu^{(2)})
\end{equation*}
with cokernel $\cC$-morphism
\begin{equation*}
 \eta_{W_1,W_2}: W_1\boxtimes W_2\rightarrow W_1\boxtimes_A W_2.
\end{equation*}
Both $\mu^{(1)}$ and $\mu^{(2)}$ are compositions of even morphisms and are therefore even; thus by Proposition \ref{evenkercoker} we may and do take $\eta_{W_1,W_2}$ even.
Our definition of $W_1\boxtimes_A W_2$ is equivalent to that in \cite{KO},
\begin{equation*}
W_1\boxtimes_A W_2=W_1\boxtimes W_2/\mathrm{Im}(\mu^{(1)}-\mu^{(2)}),
\end{equation*}
with $\eta_{W_1,W_2}$ now the quotient morphism.

The left action of $A$ on $W_1\boxtimes_A W_2$ is given by the following proposition:
\begin{propo}\label{propo:welldefmult}
There is a unique $\sC$-morphism
\begin{equation*}
 \mu_{W_1\boxtimes_A W_2}: A\boxtimes(W_1\boxtimes_A W_2)\rightarrow W_1\boxtimes_A W_2
\end{equation*}
such that for both $i=1,2$, the following diagram commutes:
\begin{equation}\label{repAtensdef}
\begin{matrix}
\xymatrixcolsep{4pc}
\xymatrix{
	A\boxtimes(W_1\boxtimes W_2) 
	\ar[d]_{1_A\boxtimes\eta_{W_1,W_2}}
	\ar[r]^(0.53){\mu^{(i)} }
	& W_1\boxtimes W_2 \ar[d]^{\eta_{W_1,W_2}} \\
	A\boxtimes(W_1\boxtimes_A W_2)
	\ar[r]^(0.54){\mu_{W_1\boxtimes_A W_2}}
	& W_1\boxtimes_A W_2 \\}
\end{matrix}.
\end{equation}
Moreover, $\mu_{W_1\boxtimes_A W_2}$ is even and $(W_1\boxtimes_A W_2, \mu_{W_1\boxtimes_A W_2})$ is an object of $\repA$.
\end{propo}
\begin{proof}
Since $A\boxtimes\cdot$ is right exact, $(A\boxtimes(W_1\boxtimes_A W_2), 1_A\boxtimes \eta_{W_1,W_2})$
is a cokernel of the morphism
\begin{equation*}
 1_A\boxtimes(\mu^{(1)}-\mu^{(2)}): A\boxtimes(A\boxtimes(W_1\boxtimes W_2))\rightarrow A\boxtimes(W_1\boxtimes W_2).
\end{equation*}
Thus if we can show that 
\begin{equation}\label{repAtenswelldef}
 \eta_{W_1,W_2}\circ\mu^{(i)}\circ(1_A\boxtimes(\mu^{(1)}-\mu^{(2)}))=0
\end{equation}
for $i=1,2$, the universal property of the cokernel will imply that there are unique morphisms
\begin{equation*}
 \mu^{(i)}_{W_1,W_2}: A\boxtimes(W_1\boxtimes_A W_2)\rightarrow W_1\boxtimes_A W_2
\end{equation*}
such that the diagram
\begin{equation}
\label{reAtensdef2}
\xymatrixcolsep{4pc}
\xymatrix{
	A\boxtimes(W_1\boxtimes W_2) 
	\ar[d]_{1_A\boxtimes\eta_{W_1,W_2}}
	\ar[r]^(0.53){\mu^{(i)} }
	& W_1\boxtimes W_2 \ar[d]^{\eta_{W_1,W_2}} \\
	A\boxtimes(W_1\boxtimes_A W_2)
	\ar[r]^(0.54){\mu^{(i)}_{W_1,W_2}}
	& W_1\boxtimes_A W_2 \\}
\end{equation}
commutes for $i=1,2$. But then since $\eta_{W_1,W_2}\circ\mu^{(1)}=\eta_{W_1,W_2}\circ\mu^{(2)}$, the uniqueness will imply $\mu^{(1)}_{W_1,W_2}=\mu^{(2)}_{W_1,W_2}$, and we can call this common morphism $\mu_{W_1\boxtimes_A W_2}$.

To show \eqref{repAtenswelldef}, we see from Lemma \ref{muiassoc} that
	\begin{align*}
	\mu^{(1)}\circ & (1_A\boxtimes (\mu^{(1)}-\mu^{(2)}))=\mu^{(1)}\circ (\mu\boxtimes 1_{W_1\boxtimes W_2})\circ\sA_{A,A,W_1\boxtimes W_2}-\mu^{(1)}\circ(1_A\boxtimes\mu^{(2)})\nonumber\\
	& =(\mu^{(1)}-\mu^{(2)})\circ (\mu\boxtimes 1_{W_1\boxtimes W_2})\circ\sA_{A,A,W_1\boxtimes W_2}-(\mu^{(1)}-\mu^{(2)})\circ(1_A\boxtimes\mu^{(2)}).  
	\end{align*}
	Similarly,
	\begin{equation*}
	\mu^{(2)}\circ (1_A\boxtimes (\mu^{(1)}-\mu^{(2)}))=(\mu^{(1)}-\mu^{(2)})\circ ((\mu\boxtimes 1_{W_1\boxtimes W_2})\circ\sA_{A,A,W_1\boxtimes W_2}-1_A\boxtimes \mu^{(1)}).
	\end{equation*}
	Thus \eqref{repAtenswelldef} holds for $i=1,2$ because $\eta_{W_1,W_2}\circ(\mu^{(1)}-\mu^{(2)})=0$. This proves the existence and uniqueness of $\mu_{W_1\boxtimes_A W_2}$. The evenness of $\mu_{W_1\boxtimes_A W_2}$ follows easily from the evenness of $\mu^{(i)}$ and $\eta_{W_1,W_2}$ and the surjectivity of the cokernel $1_A\boxtimes\eta_{W_1,W_2}$.

Now we show that $(W_1\boxtimes_A W_2, \mu_{W_1\boxtimes_A W_2})$ is an object of $\repA$. For associativity of $\mu_{W_1\boxtimes_A W_2}$, we first calculate
	\begin{align*}
		\begin{matrix}
			\begin{tikzpicture}
				[scale = 1, baseline = {(current bounding box.center)}, line width=0.75pt]
				\node at (0,-0.3) {$\scriptstyle{A}$};
				\node at (0.5,-0.3) {$\scriptstyle{A}$};
				\node at (1.1,-0.3) {$\scriptstyle{W_1}$};
				\node at (1.9,-0.3) {$\scriptstyle{W_2}$};
				\node at (1,3.3) {$\scriptstyle{W_1\boxtimes_A W_2}$};
				\node (mu) at (0.25, 1.4) [draw,minimum width=21pt,minimum 	height=10pt, fill=white] {$\scriptstyle{\mu}$};
				\node (eta) at (1.5, 0.6) [draw,minimum width=21pt,minimum 	height=10pt, fill=white] {$\scriptstyle{\eta_{W_1,W_2}}$};
				\node (mut) at (1, 2.5) [draw,minimum width=21pt,minimum 	height=10pt, fill=white] {$\scriptstyle{\mu_{W_1\boxtimes_A W_2}}$};
				\draw [line width=0.5pt] (0,0) to[out=up,in=down] (mu.220);
				\draw [line width=0.5pt] (0.5,0) to[out=up,in=down] (mu.320);
				\draw [line width=1pt] (1.1,0) to[out=up,in=down] (eta.210);
				\draw [line width=1pt] (1.9,0) to[out=up,in=down] (eta.330);
				\draw [line width=0.5pt] (mu.90) to[out=up,in=down] (mut.210);
				\draw [line width=1pt] (eta.90) to[out=up,in=down] (mut.330);
				\draw [line width=1pt] (mut.90) to[out=up,in=down] (1,3);
			\end{tikzpicture}
		\end{matrix}				
		=
		\begin{matrix}
			\begin{tikzpicture}
				[scale = 1, baseline = {(current bounding box.center)}, line width=0.75pt]
				\node at (0,-0.3) {$\scriptstyle{A}$};
				\node at (0.5,-0.3) {$\scriptstyle{A}$};
				\node at (1,-0.3) {$\scriptstyle{W_1}$};
				\node at (2,-0.3) {$\scriptstyle{W_2}$};
				\node at (1,3.3) {$\scriptstyle{W_1\boxtimes_A W_2}$};
				\node (mu) at (0.25, 0.6) [draw,minimum width=21pt,minimum 	height=10pt, fill=white] {$\scriptstyle{\mu}$};
				\node (eta) at (1.5, 1.4) [draw,minimum width=21pt,minimum 	height=10pt, fill=white] {$\scriptstyle{\eta_{W_1,W_2}}$};
				\node (mut) at (1, 2.5) [draw,minimum width=21pt,minimum 	height=10pt, fill=white] {$\scriptstyle{\mu_{W_1\boxtimes_A W_2}}$};
				\draw [line width=0.5pt] (0,0) to[out=up,in=down] (mu.220);
				\draw [line width=0.5pt] (0.5,0) to[out=up,in=down] (mu.320);
				\draw [line width=1pt] (1,0) to[out=up,in=down] (eta.210);
				\draw [line width=1pt] (2,0) to[out=up,in=down] (eta.330);
				\draw [line width=0.5pt] (mu.90) to[out=up,in=down] (mut.210);
				\draw [line width=1pt] (eta.90) to[out=up,in=down] (mut.330);
				\draw [line width=1pt] (mut.90) to[out=up,in=down] (1,3);
			\end{tikzpicture}
		\end{matrix}				
		=
		\begin{matrix}
			\begin{tikzpicture}
				[scale = 1, baseline = {(current bounding box.center)}, line width=0.75pt]
				\node at (0,-0.3) {$\scriptstyle{A}$};
				\node at (0.5,-0.3) {$\scriptstyle{A}$};
				\node at (1,-0.3) {$\scriptstyle{W_1}$};
				\node at (1.5,-0.3) {$\scriptstyle{W_2}$};
				\node at (1,3.3) {$\scriptstyle{W_1\boxtimes_A W_2}$};
				\node (mu) at (0.25, 0.6) [draw,minimum width=21pt,minimum 	height=10pt, fill=white] {$\scriptstyle{\mu}$};
				\node (mui) at (1, 1.6) [draw,minimum width=30pt,minimum 	height=10pt, fill=white] {$\scriptstyle{\mu^{(i)}}$};
				\node (eta) at (1, 2.5) [draw,minimum width=21pt,minimum 	height=10pt, fill=white] {$\scriptstyle{\eta_{W_1, W_2}}$};
				\draw [line width=0.5pt] (0,0) to[out=up,in=down] (mu.220);
				\draw [line width=0.5pt] (0.5,0) to[out=up,in=down] (mu.320);
				\draw [line width=1pt] (1,0) to[out=up,in=down] (mui.270);
				\draw [line width=1pt] (1.5,0) to[out=up,in=down] (mui.330);
				\draw [line width=0.5pt] (mu.90) to[out=up,in=down] (mui.210);
				\draw [line width=1pt] (mui.30) to[out=up,in=down] (eta.330);
				\draw [line width=1pt] (mui.150) to[out=up,in=down] (eta.210);
				\draw [line width=1pt] (eta.90) to[out=up,in=down] (1,3);
			\end{tikzpicture}
		\end{matrix}				
		=	
		\begin{matrix}
			\begin{tikzpicture}
				[scale = 1, baseline = {(current bounding box.center)}, line width=0.75pt]
				\node at (0,-0.3) {$\scriptstyle{A}$};
				\node at (0.5,-0.3) {$\scriptstyle{A}$};
				\node at (1,-0.3) {$\scriptstyle{W_1}$};
				\node at (1.5,-0.3) {$\scriptstyle{W_2}$};
				\node at (0.8,3.3) {$\scriptstyle{W_1\boxtimes_A W_2}$};
				\node (mui1) at (1,0.6) [draw,minimum width=30pt,minimum height=10pt,fill=white] {$\scriptstyle{\mu^{(i)}}$};
				\node (mui2) at (0.8,1.6) [draw,minimum width=30pt,minimum height=10pt,fill=white] {$\scriptstyle{\mu^{(i)}}$};
				\node (eta) at (0.8, 2.5) [draw,minimum width=21pt,minimum 	height=10pt, fill=white] {$\scriptstyle{\eta_{W_1, W_2}}$};
				\draw [line width=0.5pt] (0,0) to[out=up,in=down] (mui2.210);
				\draw [line width=0.5pt] (0.5,0) to[out=up,in=down] (mui1.210);
				\draw [line width=1pt] (1,0) to[out=up,in=down] (mui1.270);
				\draw [line width=1pt] (1.5,0) to[out=up,in=down] (mui1.330);
				\draw [line width=1pt] (mui1.140) to[out=up,in=down] (mui2.270);
				\draw [line width=1pt] (mui1.40) to[out=up,in=down] (mui2.330);
				\draw [line width=1pt] (mui2.145) to[out=up,in=down] (eta.210);
				\draw [line width=1pt] (mui2.35) to[out=up,in=down] (eta.330);
				\draw [line width=1pt] (eta.90) to[out=up,in=down] (0.8,3);
			\end{tikzpicture}
		\end{matrix}						
		=
		\begin{matrix}
			\begin{tikzpicture}
				[scale = 1, baseline = {(current bounding box.center)}, line width=0.75pt]
				\node at (0,-0.3) {$\scriptstyle{A}$};
				\node at (0.5,-0.3) {$\scriptstyle{A}$};
				\node at (1.1,-0.3) {$\scriptstyle{W_1}$};
				\node at (1.9,-0.3) {$\scriptstyle{W_2}$};
				\node at (0.8,3.3) {$\scriptstyle{W_1\boxtimes_A W_2}$};
				\node (eta) at (1.5,0.6) [draw,minimum width=21pt,minimum height=10pt, fill=white] {$\scriptstyle{\eta_{W_1,W_2}}$};
				\node (mu1) at (1.2,1.5) [draw,minimum width=21pt,minimum height=10pt, fill=white] {$\scriptstyle{\mu_{W_1\boxtimes_A W_2}}$};
				\node (mu2) at (0.8,2.5) [draw,minimum width=21pt,minimum height=10pt, fill=white] {$\scriptstyle{\mu_{W_1\boxtimes_A W_2}}$};
				\draw [line width=0.5pt] (0,0) to[out=up, in=down] (mu2.200);
				\draw [line width=0.5pt] (0.5,0) to[out=up, in=down] (mu1.200);
				\draw [line width=1pt] (1.1,0) to[out=up, in=down] (eta.210);
				\draw [line width=1pt] (1.9,0) to[out=up, in=down] (eta.330);
				\draw [line width=1pt] (eta.90) to[out=up, in=down] (mu1.320);				
				\draw [line width=1pt] (mu1.90) to[out=up, in=down] (mu2.330);
				\draw [line width=1pt] (mu2.90) to[out=up, in=down] (0.8,3);
			\end{tikzpicture}
		\end{matrix}						
	\end{align*}
The first equality uses properties of associativity, the second follows from \eqref{repAtensdef}, the third is due to Lemma \ref{muiassoc}, and the last is a repeated application of \eqref{reAtensdef2}. Now associativity of $\mu_{W_1\boxtimes_A W_2}$ follows since $\eta_{W_1,W_2}$ is surjective and $A\boxtimes \cdot$ is right exact.

For the unit property of $\mu_{W_1\boxtimes_A W_2}$, we need to prove
\begin{align}
	\mu_{W_1\boxtimes_A W_2}\circ (\iota_A\boxtimes 1_{W_1\boxtimes_A W_2}) \circ \sleft_{W_1\boxtimes_A W_2}^{-1} = 1_{W_1\boxtimes_A W_2}.
	\label{eqn:murepAunitproperty}
\end{align}
Using first naturality of the left unit isomorphism and the triangle axiom, then \eqref{repAtensdef}, and finally the unit property of $\mu_{W_1}$, we get 
	\begin{align*}
		\begin{matrix}
			\begin{tikzpicture}
				[scale = 1, baseline = {(current bounding box.center)}, line width=0.5pt]
				\node at (0,-0.3) {$\scriptstyle{W_1}$};
				\node at (1,-0.3) {$\scriptstyle{W_2}$};
				\node at (0.1,3.8) {$\scriptstyle{W_1\boxtimes_AW_2}$};
				\node (eta) at (0.5,1)  [draw,minimum width=21pt,minimum height=10pt, fill=white] {$\scriptstyle{\eta_{W_1,W_2}}$};
				\node (mu) at (0.1,2.8) [draw,minimum width=21pt,minimum height=10pt, fill=white] {$\scriptstyle{\mu_{W_1\boxtimes_AW_2}}$};
				\node (unit) at (-0.2,2) {$\bullet$};
				\draw[white, double=black, line width = 1pt, double distance=1.25 pt] (0,0) -- (eta.210);
				\draw[white, double=black, line width = 1pt, double distance=1.25 pt] (1,0) -- (eta.330);
				\draw[white, double=black, line width = 1pt, double distance=1.25 pt] (eta.90) to[out=up,in=down] (mu.330);
				\draw[white, double=black, line width = 1pt, double distance=1.25 pt] (mu.90) to[out=up,in=down] (0.1,3.5);
				\draw[dashed] (0.5,1.5) to[out=left, in=down] (-0.2,2);
				\draw[line width=0.5 pt] (-0.2,2) to[out=up, in=down] (mu.220);
			\end{tikzpicture}
		\end{matrix}				
		=
		\begin{matrix}
			\begin{tikzpicture}
				[scale = 1, baseline = {(current bounding box.center)}, line width=0.5pt]
				\node at (0.5,-0.3) {$\scriptstyle{W_1}$};
				\node at (1.5,-0.3) {$\scriptstyle{W_2}$};
				\node at (0.1,3.8) {$\scriptstyle{W_1\boxtimes_AW_2}$};
				\node (eta) at (1,1.8)  [draw,minimum width=21pt,minimum 	height=10pt, fill=white] {$\scriptstyle{\eta_{W_1,W_2}}$};
				\node (mu) at (0.1,2.8) [draw,minimum width=21pt,minimum height=10pt, fill=white] {$\scriptstyle{\mu_{W_1\boxtimes_AW_2}}$};
				\node (unit) at (-0.2,1) {$\bullet$};
				\draw[white, double=black, line width = 1pt, double distance=1.25 pt] (0.5,0) -- (eta.210);
				\draw[white, double=black, line width = 1pt, double distance=1.25 pt] (1.5,0) -- (eta.330);
				\draw[white, double=black, line width = 1pt, double distance=1.25 pt] (eta.90) to[out=up,in=down] (mu.330);
				\draw[white, double=black, line width = 1pt, double distance=1.25 pt] (mu.90) to[out=up,in=down] (0.1,3.5);
				\draw[dashed] (0.5,0.5) to[out=left, in=down] (-0.2,1);
				\draw[line width=0.5 pt] (-0.2,1) to[out=up, in=down] (mu.220);
			\end{tikzpicture}
		\end{matrix}				
		=
		\begin{matrix}
			\begin{tikzpicture}
				[scale = 1, baseline = {(current bounding box.center)}, line width=0.5pt]
				\node at (0.5,-0.3) {$\scriptstyle{W_1}$};
				\node at (1.5,-0.3) {$\scriptstyle{W_2}$};
				\node at (1,3.8) {$\scriptstyle{W_1\boxtimes_AW_2}$};
				\node (eta) at (1,2.8)  [draw,minimum width=21pt,minimum 	height=10pt, fill=white] {$\scriptstyle{\eta_{W_1,W_2}}$};
				\node (mu) at (0.2,1.8) [draw,minimum width=21pt,minimum height=10pt, fill=white] {$\scriptstyle{\mu_{W_1}}$};
				\node (unit) at (-0.2,0.8) {$\bullet$};
				\draw[white, double=black, line width = 1pt, double distance=1.25 pt] (0.5,0) -- (mu.320);
				\draw[white, double=black, line width = 1pt, double distance=1.25 pt] (1.5,0) -- (eta.330);
				\draw[white, double=black, line width = 1pt, double distance=1.25 pt] (mu.90) to[out=up,in=down] (eta.210);
				\draw[white, double=black, line width = 1pt, double distance=1.25 pt] (eta.90) to[out=up,in=down] (1,3.5);
				\draw[dashed] (0.5,0.3) to[out=left, in=down] (-0.2,0.8);
				\draw[line width=0.5 pt] (-0.2,0.8) to[out=up, in=down] (mu.220);
			\end{tikzpicture}
		\end{matrix}
		=				
		\begin{matrix}
			\begin{tikzpicture}
				[scale = 1, baseline = {(current bounding box.center)}, line width=0.5pt]
				\node at (0.5,-0.3) {$\scriptstyle{W_1}$};
				\node at (1.5,-0.3) {$\scriptstyle{W_2}$};
				\node at (1,3.8) {$\scriptstyle{W_1\boxtimes_AW_2}$};
				\node (eta) at (1,2)  [draw,minimum width=21pt,minimum 	height=10pt, fill=white] {$\scriptstyle{\eta_{W_1,W_2}}$};
				\draw[white, double=black, line width = 1pt, double distance=1.25 pt] (0.5,0) -- (eta.210);
				\draw[white, double=black, line width = 1pt, double distance=1.25 pt] (1.5,0) -- (eta.330);
				\draw[white, double=black, line width = 1pt, double distance=1.25 pt] (eta.90) to[out=up,in=down] (1,3.5);
			\end{tikzpicture}
		\end{matrix}
	\end{align*}
This calculation indeed proves	\eqref{eqn:murepAunitproperty} since the cokernel morphism $\eta_{W_1,W_2}$ is surjective.
\end{proof}

For $\boxtimes_A$ to define a functor $\repA\times\repA\rightarrow\repA$, we also need tensor products of morphisms:
\begin{propo}\label{propo:repAtensnat}
	If $f_1: W_1\rightarrow \widetilde{W}_1$ and $f_2: W_2\rightarrow \widetilde{W}_2$ are morphisms in $\repA$, then there is a unique $\repA$-morphism
	\begin{equation*}
	 f_1\boxtimes_A f_2: W_1\boxtimes_A W_2\rightarrow \widetilde{W}_1\boxtimes_A \widetilde{W}_2
	\end{equation*}
such that the following diagram commutes.
\begin{equation}\label{repAtensprodmorphdef}
\begin{matrix}
\xymatrixcolsep{3pc}
\xymatrix{
	W_1\boxtimes W_2 
	\ar[d]_{\eta_{W_1,W_2}}
	\ar[r]^(0.5){ f_1\boxtimes f_2}
	& \widetilde{W}_1\boxtimes \widetilde{W}_2 \ar[d]^{\eta_{\widetilde{W}_1,\widetilde{W}_2}} \\
	W_1\boxtimes_A W_2
	\ar[r]^(0.5){f_1\boxtimes_A f_2}
	& \widetilde{W}_1\boxtimes_A \widetilde{W}_2 \\}
\end{matrix}
\end{equation}
\end{propo}
\begin{proof}
	For $i=1,2$, it is easy to show that
	\begin{equation}\label{repAtensprodmorph}
	\mu^{(i)}\circ(1_A\boxtimes (f_1\boxtimes f_2))=(f_1\boxtimes f_2)\circ\mu^{(i)}: A\boxtimes(W_1\boxtimes W_2)\rightarrow \widetilde{W}_1\boxtimes \widetilde{W}_2.
	\end{equation}
	Indeed, for $i=1$, this follows from the naturality of associativity, evenness  of $\mu_{W_1}$, and the fact that $f_1$ is a morphism in $\repA$, while
	for $i=2$, we use the same properties of $\mu_{W_2}$ and $f_2$ plus naturality of the braiding. 
	Now, \eqref{repAtensprodmorph} implies that the composition
	\begin{equation*}
	 \eta_{\widetilde{W}_1,\widetilde{W}_2}\circ(f_1\boxtimes f_2)\circ(\mu^{(1)}-\mu^{(2)}): A\boxtimes(W_1\boxtimes W_2)\rightarrow \widetilde{W}_1\boxtimes_A \widetilde{W}_2
	\end{equation*}
is $0$. Thus the universal property of the cokernel $(W_1\boxtimes_A W_2, \eta_{W_1,W_2})$ induces a unique $\sC$-morphism 
$$f_1\boxtimes_A f_2: W_1\boxtimes_A W_2\rightarrow \widetilde{W}_1\boxtimes_A \widetilde{W}_2$$
making the diagram \eqref{repAtensprodmorphdef} commute. 
To show that $f_1\boxtimes_A f_2$ is a $\repA$-morphism:
\begin{align*}
	\begin{matrix}
		\begin{tikzpicture}[scale = 0.8, baseline = {(current bounding box.center)}, line width=1pt,out=up, in=down]
			\node (a) at (0,-0.3) {$\scriptstyle{A}$};
			\node (w1) at (0.75,-0.3) {$\scriptstyle{W_1}$};
			\node (w2) at (1.75,-0.3) {$\scriptstyle{W_2}$};
			\node (wt) at (0.8,5) {$\scriptstyle{\widetilde{W}_1\boxtimes_A\widetilde{W}_2 }$};
			\node (eta) at (1.25, 1) [draw, fill=white,line width=0.5pt] {$\scriptstyle{\eta_{W_1,W_2}}$};
			\node (mu) at (0.8,2.5) [draw, fill=white,line width=0.5pt] {$\scriptstyle{\mu_{W_1\boxtimes_A W_2}}$};
			\node (f) at (0.8,3.75) [draw,fill=white,line width=0.5pt] {$\scriptstyle{f_1\boxtimes_A f_2}$};
			\draw [line width=0.5pt] (a.90) to (mu.210);
			\draw (w1.90) to (eta.210);
			\draw (w2.90) to (eta.330);
			\draw (eta.90) to (mu.330);
			\draw (mu.90) to (f.270);
			\draw (f.90) to (wt.270);
		\end{tikzpicture}	
	\end{matrix}
	=
	\begin{matrix}
		\begin{tikzpicture}[scale = 0.8, baseline = {(current bounding box.center)}, line width=1pt,out=up, in=down]
			\node (a) at (0,-0.3) {$\scriptstyle{A}$};
			\node (w1) at (0.5,-0.3) {$\scriptstyle{W_1}$};
			\node (w2) at (1.1,-0.3) {$\scriptstyle{W_2}$};
			\node (wt) at (0.5,5) {$\scriptstyle{\widetilde{W}_1\boxtimes_A\widetilde{W}_2 }$};
			\node (mui) at (0.5,1) [draw, fill=white, minimum width=30pt, line width=0.5pt] {$\scriptstyle{\mu^{(i)}}$};
			\node (eta) at (0.5, 2.5) [draw, fill=white, line width=0.5pt] {$\scriptstyle{\eta_{W_1,W_2}}$};
			\node (f) at (0.5,3.75) [draw,fill=white, line width=0.5pt] {$\scriptstyle{f_1\boxtimes_A f_2}$};
			\draw [line width=0.5pt] (a.90) to (mui.215);
			\draw (w1.90) to (mui.270);
			\draw (w2.90) to (mui.330);
			\draw (mui.140) to (eta.220);
			\draw (mui.40) to (eta.320);
			\draw (eta.90) to (f.270);
			\draw (f.90) to (wt.270);
		\end{tikzpicture}	
	\end{matrix}	
	=
	\begin{matrix}
		\begin{tikzpicture}[scale = 0.8, baseline = {(current bounding box.center)}, line width=1pt,out=up, in=down]
			\node (a) at (0,-0.3) {$\scriptstyle{A}$};
			\node (w1) at (0.5,-0.3) {$\scriptstyle{W_1}$};
			\node (w2) at (1.1,-0.3) {$\scriptstyle{W_2}$};
			\node (wt) at (0.5,5) {$\scriptstyle{\widetilde{W}_1\boxtimes_A\widetilde{W}_2 }$};
			\node (eta) at (0.5, 3.75) [draw, fill=white, line width=0.5pt] {$\scriptstyle{\eta_{\widetilde{W}_1,\widetilde{W}_2}}$};
			\node (f1) at (0.1,2.5) [draw,fill=white, line width=0.5pt] {$\scriptstyle{f_1}$};
			\node (f2) at (0.9,2.5) [draw,fill=white, line width=0.5pt] {$\scriptstyle{f_2}$};
			\node (mui) at (0.5,1) [draw, fill=white, minimum width=30pt, line width=0.5pt] {$\scriptstyle{\mu^{(i)}}$};
			\draw [line width=0.5pt] (a.90) to (mui.215);
			\draw (w1.90) to (mui.270);
			\draw (w2.90) to (mui.330);
			\draw (mui.140) to (f1.270);
			\draw (mui.40) to (f2.270);
			\draw (f1.90) to (eta.220);
			\draw (f2.90) to (eta.320);
			\draw (eta.90) to (wt.270);
		\end{tikzpicture}	
	\end{matrix}	
	=
	\begin{matrix}
		\begin{tikzpicture}[scale = 0.8, baseline = {(current bounding box.center)}, line width=1pt,out=up, in=down]
			\node (a) at (0,-0.3) {$\scriptstyle{A}$};
			\node (w1) at (0.5,-0.3) {$\scriptstyle{W_1}$};
			\node (w2) at (1.3,-0.3) {$\scriptstyle{W_2}$};
			\node (wt) at (0.5,5) {$\scriptstyle{\widetilde{W}_1\boxtimes_A\widetilde{W}_2 }$};
			\node (eta) at (0.5, 3.75) [draw, fill=white, line width=0.5pt] {$\scriptstyle{\eta_{\widetilde{W}_1,\widetilde{W}_2}}$};
			\node (mui) at (0.5,2.5) [draw, fill=white, minimum width=30pt, line width=0.5pt] {$\scriptstyle{\mu^{(i)}}$};
			\node (f1) at (0.5,1) [draw,fill=white, line width=0.5pt] {$\scriptstyle{f_1}$};
			\node (f2) at (1.3,1) [draw,fill=white, line width=0.5pt] {$\scriptstyle{f_2}$};
			\draw [line width=0.5pt] (a.90) to (mui.215);
			\draw (w1.90) to (f1.270);
			\draw (w2.90) to (f2.270);
			\draw (f1.90) to (mui.270);
			\draw (f2.90) to (mui.320);
			\draw (mui.140) to (eta.220);
			\draw (mui.40) to (eta.320);
			\draw (eta.90) to (wt.270);
		\end{tikzpicture}	
	\end{matrix}	
	=
	\begin{matrix}
		\begin{tikzpicture}[scale = 0.8, baseline = {(current bounding box.center)}, line width=1pt,out=up, in=down]
			\node (a) at (0,-0.3) {$\scriptstyle{A}$};
			\node (w1) at (0.6,-0.3) {$\scriptstyle{W_1}$};
			\node (w2) at (1.4,-0.3) {$\scriptstyle{W_2}$};
			\node (wt) at (0.5,5) {$\scriptstyle{\widetilde{W}_1\boxtimes_A\widetilde{W}_2 }$};
			\node (mu) at (0.5,3.75) [draw, fill=white, minimum width=30pt, line width=0.5pt] {$\scriptstyle{\mu_{\widetilde{W}_1\boxtimes_A\widetilde{W}_2}}$};
			\node (eta) at (1, 2.5) [draw, fill=white, line width=0.5pt] {$\scriptstyle{\eta_{\widetilde{W}_1,\widetilde{W}_2}}$};
			\node (f1) at (0.6,1) [draw,fill=white, line width=0.5pt] {$\scriptstyle{f_1}$};
			\node (f2) at (1.4,1) [draw,fill=white, line width=0.5pt] {$\scriptstyle{f_2}$};
			\draw [line width=0.5pt] (a.90) to (mu.215);
			\draw (w1.90) to (f1.270);
			\draw (w2.90) to (f2.270);
			\draw (f1.90) to (eta.220);
			\draw (f2.90) to (eta.320);
			\draw (eta.90) to (mu.330);
			\draw (mu.90) to (wt.270);
		\end{tikzpicture}	
	\end{matrix}	
	=
	\begin{matrix}
		\begin{tikzpicture}[scale = 0.8, baseline = {(current bounding box.center)}, line width=1pt,out=up, in=down]
			\node (a) at (0,-0.3) {$\scriptstyle{A}$};
			\node (w1) at (0.6,-0.3) {$\scriptstyle{W_1}$};
			\node (w2) at (1.4,-0.3) {$\scriptstyle{W_2}$};
			\node (wt) at (0.5,5) {$\scriptstyle{\widetilde{W}_1\boxtimes_A\widetilde{W}_2 }$};
			\node (mu) at (0.5,3.75) [draw, fill=white, minimum width=30pt, line width=0.5pt] {$\scriptstyle{\mu_{\widetilde{W}_1\boxtimes_A\widetilde{W}_2}}$};
			\node (eta) at (1, 1) [draw, fill=white, line width=0.5pt] {$\scriptstyle{\eta_{W_1,W_2}}$};
			\node (f) at (1,2.5) [draw,fill=white, line width=0.5pt] {$\scriptstyle{f_1\boxtimes_A f_2}$};
			\draw [line width=0.5pt] (a.90) to (mu.215);
			\draw (w1.90) to (eta.220);
			\draw (w2.90) to (eta.320);
			\draw (eta.90) to (f.270);
			\draw (f.90) to (mu.320);
			\draw (mu.90) to (wt.270);
		\end{tikzpicture}	
	\end{matrix}	
\end{align*}	
We use \eqref{repAtensdef} for the first and fourth equalities, \eqref{repAtensprodmorphdef} for the second and last, and \eqref{repAtensprodmorph} for the third.
Recalling that $1_A\boxtimes\eta_{W_1,W_2}$ is surjective because $A\boxtimes\cdot$ is right exact, $f_1\boxtimes_A f_2$ is a $\repA$-morphism.
\end{proof}

The characterization of $f_1\boxtimes_A f_2$ as the unique morphism making \eqref{repAtensprodmorphdef} commute implies that $\boxtimes_A$ is a functor from $\repA\times\repA$ to $\repA$. That is, the super interchange law
\begin{equation*}
 (f_1\boxtimes_A f_2)\circ(g_1\boxtimes_A g_2)=(-1)^{\vert f_2\vert \vert g_1\vert}(f_1\circ g_1)\boxtimes_A (f_2\circ g_2)
\end{equation*}
holds for appropriately composable parity-homogeneous morphisms $f_1$, $f_2$, $g_1$, and $g_2$ in $\repA$, and $1_{W_1}\boxtimes_A 1_{W_2}=1_{W_1\boxtimes_A W_2}$ for objects $W_1$ and $W_2$ of $\repA$. The characterization \eqref{repAtensprodmorphdef} also implies that the tensor product of morphisms in $\repA$ is bilinear, since the tensor product of morphisms in the tensor supercategory $\sC$ is bilinear.

Now we construct the left and right unit isomorphisms in $\repA$.
\begin{propo}\label{repAleftunit}
	If $(W,\mu_W)$ is an object of $\mathrm{Rep}\,A$, there is a unique isomorphism $l^A_W: A\boxtimes_A W\rightarrow W$ in $\repA$ such that
	\begin{equation}\label{leftisodef}
	 l^A_W\circ\eta_{A, W}=\mu_W: A\boxtimes W\rightarrow W,
	\end{equation}
	and $l^A_W$ is even.
\end{propo}
\begin{proof}
The existence and uniqueness of the $\cC$-morphism $l^A_W$ follows from the universal property of the cokernel $(A\boxtimes_A W, \eta_{A,W})$, provided $\mu_W\circ\mu^{(1)}=\mu_W\circ\mu^{(2)}$. 
In fact, this is immediate from
\begin{align*}
	\begin{matrix}
		\begin{tikzpicture}[line width=0.5pt,out=up, in=down]
			\node (a1) at (0,-0.3) {$\scriptstyle{A}$};
			\node (a2) at (0.5,-0.3) {$\scriptstyle{A}$};
			\node (w) at (1,-0.3) {$\scriptstyle{W}$};
			\node (wf) at (0.6,3)  {$\scriptstyle{W}$};
			\node (mu) at (0.25,1) [draw] {$\scriptstyle{\mu}$};
			\node (muw) at (0.6, 2)[draw] {$\scriptstyle{\mu_W}$};
			\draw (a1.90) to (mu.230);
			\draw (a2.90) to (mu.310);
			\draw (mu.90) to (muw.220);
			\draw [line width=1 pt] (w.90) to (muw.320);
			\draw [line width=1 pt] (muw.90) to (wf.270);
		\end{tikzpicture}		
	\end{matrix}
	=
	\begin{matrix}
		\begin{tikzpicture}[line width=0.5pt,out=up, in=down]
			\node (a1) at (0,-0.3) {$\scriptstyle{A}$};
			\node (a2) at (0.5,-0.3) {$\scriptstyle{A}$};
			\node (w) at (1,-0.3) {$\scriptstyle{W}$};
			\node (wf) at (0.6,3)  {$\scriptstyle{W}$};
			\node (mu) at (0.25,1) [draw] {$\scriptstyle{\mu}$};
			\node (muw) at (0.6, 2)[draw] {$\scriptstyle{\mu_W}$};
			\draw [white, line width=1 pt, double=black, double distance= 0.5 pt] (a2.90) to (mu.230);
			\draw [white, line width=3 pt, double=black, double distance= 0.5 pt] (a1.90) to (mu.310);
			\draw (mu.90) to (muw.220);
			\draw [line width=1 pt] (w.90) to (muw.320);
			\draw [line width=1 pt] (muw.90) to (wf.270);
		\end{tikzpicture}		
	\end{matrix}
	=
	\begin{matrix}
		\begin{tikzpicture}[ line width=0.5pt,out=up, in=down]
			\node (a1) at (0,-0.3) {$\scriptstyle{A}$};
			\node (a2) at (0.5,-0.3) {$\scriptstyle{A}$};
			\node (w) at (1,-0.3) {$\scriptstyle{W}$};
			\node (wf) at (0.5,3)  {$\scriptstyle{W}$};
			\node (mu) at (0.75,1) [draw] {$\scriptstyle{\mu_W}$};
			\node (muw) at (0.5, 2)[draw] {$\scriptstyle{\mu_W}$};
			\draw (a2.90) to (0.1,1) to (muw.220);
			\draw [white, line width=3 pt, double=black, double distance= 0.5 pt] (a1.90) to (mu.230);
			\draw [line width=1 pt] (mu.90) to (muw.320);
			\draw [line width=1 pt] (w.90) to (mu.310);
			\draw [line width=1 pt] (muw.90) to (wf.270);
		\end{tikzpicture}		
	\end{matrix}
\end{align*}
The evenness of $l^A_W$ is clear from the the evenness of $\mu_W$ and $\eta_{A,W}$ and the surjectivity of $\eta_{A,W}$.

Now we need to show that $l^A_W$ is a morphism in $\mathrm{Rep}\,A$, that is, that 
$$l^A_W\circ\mu_{A\boxtimes_A W}=\mu_W\circ(1_A\boxtimes l^A_W): A\boxtimes(A\boxtimes_A W)\rightarrow W.$$
By \eqref{repAtensdef}, \eqref{leftisodef}, and the associativity of $\mu_W$,
\begin{align*}
l^A_W\circ\mu_{A\boxtimes_A W}\circ(1_A\boxtimes\eta_{A,W}) & = l^A_W\circ\eta_{A,W}\circ\mu^{(1)} = \mu_W\circ(\mu\boxtimes 1_A)\circ\sA_{A,A,W}\nonumber\\
& = \mu_W\circ(1_A\boxtimes\mu_W) = \mu_W\circ(1_A\boxtimes l^A_W)\circ(1_A\boxtimes\eta_{A,W})
\end{align*}
as morphisms from $A\boxtimes(A\boxtimes W)\rightarrow W$. Because $1_A\boxtimes\eta_{A,W}$ is an epimorphism, $l^A_W$ is a $\repA$-morphism.
	
To show that $l^A_W$ is an isomorphism in $\mathrm{Rep}\,A$, we construct its inverse. It is enough to show that $l^A_W$ has an inverse in $\mathcal{SC}$, since an inverse in $\mathcal{SC}$ will automatically commute with $A$-actions because $l^A_W$ does. We will show that the composition 
\begin{equation*}
\widetilde{l}^A_W: W\xrightarrow{\sleft_W^{-1}}\sunit\boxtimes W\xrightarrow{\iota_A\boxtimes 1_W} A\boxtimes W\xrightarrow{\eta_{A,W}} A\boxtimes_A W
\end{equation*}
is the inverse to $l^A_W$. It is clear from \eqref{leftisodef} and the unit property for $\mu_W$ that $l^A_W\circ\widetilde{l}^A_W=1_W$. On the other hand, let us use $\lambda$ to denote the composition
\begin{equation*}
\lambda: A\boxtimes W\xrightarrow{\mu_W} W\xrightarrow{\sleft_W^{-1}}\sunit\boxtimes W\xrightarrow{\iota_A\boxtimes 1_W} A\boxtimes W.
\end{equation*}
Since $\eta_{A,W}\circ\lambda=\widetilde{l}^A_W\circ l^A_W\circ\eta_{A,W}$, the surjectivity of $\eta_{A,W}$ will imply $\widetilde{l}^A_W\circ l^A_W=1_{A\boxtimes_A W}$ provided that $\eta_{A,W}\circ\lambda=\eta_{A,W}$. We will show
\begin{align}\label{id-1}
1_{A\boxtimes W}\circ (1_A\boxtimes \sleft_W)&=\mu^{(1)}\circ (1_A\boxtimes (\iota_A\boxtimes 1_W)),\\
\label{L-2}
\lambda\circ (1_A\boxtimes \sleft_W)&=\mu^{(2)}\circ (1_A\boxtimes (\iota_A\boxtimes 1_W))
\end{align}
as morphisms from $A\boxtimes(\sunit\boxtimes W)\rightarrow A\boxtimes W$, and then
\begin{align*}
\eta_{A,W}\circ(1_{A\boxtimes W}-\lambda) & = \eta_{A,W}\circ(1_{A\boxtimes W}-\lambda)\circ(1_A\boxtimes \sleft_W)\circ(1_A\boxtimes \sleft_W^{-1})\nonumber\\
& = \eta_{A,W}\circ(\mu^{(1)}-\mu^{(2)})\circ(1_A\boxtimes(\iota_A\boxtimes 1_W))\circ(1_A\boxtimes \sleft_W^{-1})=0,
\end{align*}
as desired.
Equation \eqref{id-1} follows from the right unit property for $A$ (Remark \ref{rightunit}) and the triangle axiom:
\begin{align*}
	\begin{matrix}
		\begin{tikzpicture}[scale = 1, baseline = {(current bounding box.center)}, line width=0.5pt,out=up, in=down]
			\node (a) at (0,-0.3) {$\scriptstyle{A}$};
			\node (o) at (0.5,-0.3) {$\scriptstyle{\sunit}$};
			\node (w) at (1,-0.3) {$\scriptstyle{W}$};
			\node (af) at (0.25,2.3) {$\scriptstyle{A}$};
			\node (wf) at (1,2.3) {$\scriptstyle{W}$};
			\node (i) at (0.5, 0.5) {$\scriptstyle{\bullet}$};
			\node (mu) at (0.25, 1.5) [draw, minimum width=20 pt] {$\scriptstyle{\mu}$};
			\draw (a.90) to (mu.220);
			\draw [dashed] (o.90) -- (0.5,0.5);
			\draw (0.5,0.5) -- (mu.320);
			\draw [line width=1pt](w.90) to (wf.270);
			\draw (mu.90) to (af.270);
		\end{tikzpicture}
	\end{matrix}
	=
	\begin{matrix}
		\begin{tikzpicture}[scale = 1, baseline = {(current bounding box.center)}, line width=0.5pt,out=up, in=down]
			\node (a) at (0,-0.3) {$\scriptstyle{A}$};
			\node (o) at (0.5,-0.3) {$\scriptstyle{\sunit}$};
			\node (w) at (1,-0.3) {$\scriptstyle{W}$};
			\node (af) at (0,2.3) {$\scriptstyle{A}$};
			\node (wf) at (1,2.3) {$\scriptstyle{W}$};
			\node (i) at (0, 1.5) {$\scriptstyle{\bullet}$};
			\draw (a.90) to (af.270);
			\draw [dashed] (o.90) to (0.5,0.5) to[in=east] (i);
			\draw [line width=1pt](w.90) to (wf.270);
		\end{tikzpicture}
	\end{matrix}
	=
	\begin{matrix}
		\begin{tikzpicture}[scale = 1, baseline = {(current bounding box.center)}, line width=0.5pt,out=up, in=down]
			\node (a) at (0,-0.3) {$\scriptstyle{A}$};
			\node (o) at (0.5,-0.3) {$\scriptstyle{\sunit}$};
			\node (w) at (1,-0.3) {$\scriptstyle{W}$};
			\node (af) at (0,2.3) {$\scriptstyle{A}$};
			\node (wf) at (1,2.3) {$\scriptstyle{W}$};
			\node (i) at (1, 1.5) {$\scriptstyle{\bullet}$};
			\draw (a.90) to (af.270);
			\draw [dashed] (o.90) to (0.5,0.5) to[in=west] (i);
			\draw [line width=1pt](w.90) to (wf.270);
		\end{tikzpicture}
	\end{matrix}
\end{align*}
The proof of \eqref{L-2} is guided by the diagrams:
\begin{align*}
\begin{matrix}
	\begin{tikzpicture}[scale=0.8,out=up,in=down,line width=0.5pt]
		\node (mW) at (1,3) [draw,minimum width=20pt] {$\scriptstyle{\mu_{W}}$};
		\node (b) at (0.7,1) {$\bullet$};
		\node (a) at (0,-0.3) {$\scriptstyle{A}$};
		\node (u) at (0.7,-0.3) {$\scriptstyle{\sunit}$};
		\node (w) at (1.5,-0.3) {$\scriptstyle{W}$};
		\node (af) at (0,3.8) {$\scriptstyle{A}$};
		\node (wf) at (1,3.8) {$\scriptstyle{W}$};
		\draw (b.90) to (0.7, 1.5) to (0,2.7) to (af.270);
		\draw [dashed] (u.90) to (0.7,1);
		\draw [white, double=black, double distance=0.5pt, line width=3pt] (a.90) to (0,1.5) to (mW.210);
		\draw [line width=1pt] (mW.north) to (wf.south);
		\draw [line width=1pt] (w.90) to (mW.330);
	\end{tikzpicture}
\end{matrix}
=
\begin{matrix}
	\begin{tikzpicture}[scale=0.8,out=up,in=down,line width=0.5pt]
		\node (mW) at (1,3) [draw,minimum width=20pt] {$\scriptstyle{\mu_{W}}$};
		\node (b) at (0.7,1.4) {$\bullet$};
		\node (a) at (0,-0.3) {$\scriptstyle{A}$};
		\node (u) at (0.7,-0.3) {$\scriptstyle{\sunit}$};
		\node (w) at (1.5,-0.3) {$\scriptstyle{W}$};
		\node (af) at (0,3.8) {$\scriptstyle{A}$};
		\node (wf) at (1,3.8) {$\scriptstyle{W}$};
		\draw (b.90) to (0.7, 1.5) to (0,2.7) to (af.270);
		\draw [dashed] (u.90) to[in=west] (1.5,0.5);
		\draw [dashed] (1.5,0.9) to[out=west] (0.7,1.2);
		\draw [white, double=black, double distance=0.5pt, line width=3pt] (a.90) to (0,1.5) to (mW.210);
		\draw [line width=1pt] (mW.north) to (wf.south);
		\draw [line width=1pt] (w.90) to (mW.330);
	\end{tikzpicture}
\end{matrix}		
=
\begin{matrix}
	\begin{tikzpicture}[scale=0.8,out=up,in=down,line width=0.5pt]
		\node (mW) at (1,3) [draw,minimum width=20pt] {$\scriptstyle{\mu_{W}}$};
		\node (b) at (0,2.7) {$\bullet$};
		\node (a) at (0,-0.3) {$\scriptstyle{A}$};
		\node (u) at (0.7,-0.3) {$\scriptstyle{\sunit}$};
		\node (w) at (1.5,-0.3) {$\scriptstyle{W}$};
		\node (af) at (0,3.8) {$\scriptstyle{A}$};
		\node (wf) at (1,3.8) {$\scriptstyle{W}$};
		\draw (b.90) to (af.270);
		\draw [dashed] (u.90) to[in=west] (1.5,0.5);
		\draw [dashed] (1.5,0.9) to[out=west] (0.7,1.2) to (0,2.7);
		\draw [white, double=black, double distance=0.5pt, line width=3pt] (a.90) to (0,1.25) to (mW.210);
		\draw [line width=1pt] (mW.north) to (wf.south);
		\draw [line width=1pt] (w.90) to (mW.330);
	\end{tikzpicture}
\end{matrix}		
=
\begin{matrix}
	\begin{tikzpicture}[scale=0.8,out=up,in=down,line width=0.5pt]
		\node (mW) at (1,3) [draw,minimum width=20pt] {$\scriptstyle{\mu_{W}}$};
		\node (b) at (0,2.7) {$\bullet$};
		\node (a) at (0,-0.3) {$\scriptstyle{A}$};
		\node (u) at (0.7,-0.3) {$\scriptstyle{\sunit}$};
		\node (w) at (1.5,-0.3) {$\scriptstyle{W}$};
		\node (af) at (0,3.8) {$\scriptstyle{A}$};
		\node (wf) at (1,3.8) {$\scriptstyle{W}$};
		\draw (b.90) to (af.270);
		\draw [dashed] (u.90) to[in=west] (1.5,0.5);
		\draw [dashed] (0,0.9) to[out=east] (0.7,1.2) to (0,2.7);
		\draw [white, double=black, double distance=0.5pt, line width=3pt] (a.90) to (0,1.25) to (mW.210);
		\draw [line width=1pt] (mW.north) to (wf.south);
		\draw [line width=1pt] (w.90) to (mW.330);
	\end{tikzpicture}
\end{matrix}		
\,=\,
\begin{matrix}
	\begin{tikzpicture}[scale=0.8, out=up, in=down]
		\node (mW) at (1,3) [draw,minimum width=20pt] {$\scriptstyle{\mu_{W}}$};
		\node (b) at (0,2.7) {$\bullet$};
		\node at (0.5,-0.3) {$\scriptstyle{A}$};
		\node at (1,-0.3) {$\scriptstyle{\sunit}$};
		\node at (1.5,-0.3) {$\scriptstyle{W}$};
		\node at (0,3.8) {$\scriptstyle{A}$};
		\node at (1,3.8) {$\scriptstyle{W}$};
		\draw (0.5,0) to(mW.210);
		\draw[line width=1pt] (mW.north) to (1,3.5);
		\draw[line width=1pt](1.5,0) to (mW.330);
		\draw[dashed] (1,0)	to [in=left] (1.5,0.5);
		\draw[dashed] (0.5,1) to [out=left] (0,1.5) to (0,2.7);
		\draw (0,2.7) to (0,3.5);
	\end{tikzpicture}
\end{matrix}
=
\begin{matrix}
	\begin{tikzpicture}[scale=0.8,out=up, in=down,line width=0.5pt]
		\node (mW) at (1,1.7) [draw,minimum width=20pt] {$\scriptstyle{\mu_{W}}$};
		\node (d) at (0.5,3) {$\bullet$};
		\node at (0.5,-0.3) {$\scriptstyle{A}$};
		\node at (1,-0.3) {$\scriptstyle{\sunit}$};
		\node at (1.5,-0.3) {$\scriptstyle{W}$};
		\node at (0.5,3.8) {$\scriptstyle{A}$};
		\node at (1.1,3.8) {$\scriptstyle{W}$};
		\draw (0.5,0) to (mW.210);	
		\draw[line width=1pt] (mW.north) to (1,3.5);
		\draw[line width=1pt] (1.5,0) to (mW.330);
		\draw[dashed] (1,0) to[in=west] (1.5,0.5);
		\draw[dashed] (1,2.5) to[out=west] (0.5,3);
		\draw (0.5,3) to (0.5,3.5);
	\end{tikzpicture}
\end{matrix}
\end{align*}		
The detailed argument heavily uses properties of the unit isomorphisms, especially 
those in Proposition XIII.1.2 and Lemma XI.2.2 of \cite{Ka}.
This completes the proof that $\widetilde{l}^A_W\circ l^A_W=1_{A\boxtimes_A W}$.
\end{proof}

\begin{propo}
	The isomorphisms $l^A_W$ for $W$ in $\mathrm{Rep}\,A$ define a natural isomorphism from $A\boxtimes_A\cdot$ to the identity functor on $\mathrm{Rep}\,A$.
\end{propo}
\begin{proof}
	We need to show that for any morphism $f: W_1\rightarrow W_2$ in $\mathrm{Rep}\,A$, we have
	\begin{equation}\label{leftunitnat}
	l^A_{W_2}\circ(1_A\boxtimes_A f)=f\circ l^A_{W_1}: A\boxtimes_A W_1\rightarrow W_2. 
	\end{equation}
	Noting that $\eta_{A,W_1}$ is surjective, this follows from the calculation
	\begin{align*}
	\begin{matrix}
		\begin{tikzpicture}[scale=1,line width=0.5pt, out=up, in=down]
			\node (a) at (0,0) {$\scriptstyle{A}$};
			\node (w1) at (0.5,0) {$\scriptstyle{W_1}$};
			\node (eta) at (0.25, 1) [draw] {$\scriptstyle{\eta_{A,W_1}}$};
			\node (f) at (0.25,2) [draw] {$\scriptstyle{1_A\boxtimes_A\,f}$};
			\node (l) at (0.25,3) [draw] {$\scriptstyle{l^A_{W_2}}$};
			\node (w2) at (0.25, 4) {$\scriptstyle{W_2}$};
			\draw (a.90) to (eta.220);
			\draw [line width=1pt] (w1.90) to (eta.320);
			\draw [line width=1pt] (eta.90) to (f.270);
			\draw [line width=1pt] (f.90) to (l.270);
			\draw [line width=1pt] (l.90) to (w2.270);
		\end{tikzpicture}
	\end{matrix}
	=
	\begin{matrix}
		\begin{tikzpicture}[scale=1,line width=0.5pt, out=up, in=down]
			\node (a) at (0,0) {$\scriptstyle{A}$};
			\node (w1) at (0.5,0) {$\scriptstyle{W_1}$};
			\node (f) at (0.5,1) [draw] {$\scriptstyle{f}$};
			\node (eta) at (0.25, 2) [draw] {$\scriptstyle{\eta_{A,W_2}}$};
			\node (l) at (0.25,3) [draw] {$\scriptstyle{l^A_{W_2}}$};
			\node (w2) at (0.25, 4) {$\scriptstyle{W_2}$};
			\draw (a.90) to (eta.220);
			\draw [line width=1pt] (w1.90) to (f.270);
			\draw [line width=1pt] (f.90) to (eta.320);
			\draw [line width=1pt] (eta.90) to (l.270);
			\draw [line width=1pt] (l.90) to (w2.270);
		\end{tikzpicture}
	\end{matrix}
	=
	\begin{matrix}
		\begin{tikzpicture}[scale=1,line width=0.5pt, out=up, in=down]
			\node (a) at (0,0) {$\scriptstyle{A}$};
			\node (w1) at (0.5,0) {$\scriptstyle{W_1}$};
			\node (mu) at (0.25, 2.5) [draw] {$\scriptstyle{\mu_{W_2}}$};
			\node (f) at (0.5,1) [draw] {$\scriptstyle{f}$};
			\node (w2) at (0.25, 4) {$\scriptstyle{W_2}$};
			\draw (a.90) to (mu.220);
			\draw [line width=1pt] (w1.90) to (f.270);
			\draw [line width=1pt] (f.90) to (mu.320);
			\draw [line width=1pt] (mu.90) to (w2.270);
		\end{tikzpicture}
	\end{matrix}
	=
	\begin{matrix}
		\begin{tikzpicture}[scale=1,line width=0.5pt, out=up, in=down]
			\node (a) at (0,0) {$\scriptstyle{A}$};
			\node (w1) at (0.5,0) {$\scriptstyle{W_1}$};
			\node (mu) at (0.25, 1) [draw] {$\scriptstyle{\mu_{W_1}}$};
			\node (f) at (0.25,2.5) [draw] {$\scriptstyle{f}$};
			\node (w2) at (0.25, 4) {$\scriptstyle{W_2}$};
			\draw (a.90) to (mu.220);
			\draw [line width=1pt] (w1.90) to (mu.320);
			\draw [line width=1pt] (f.90) to (w2.270);
			\draw [line width=1pt] (mu.90) to (f.270);
		\end{tikzpicture}
	\end{matrix}
	=
	\begin{matrix}
		\begin{tikzpicture}[scale=1,line width=0.5pt, out=up, in=down]
			\node (a) at (0,0) {$\scriptstyle{A}$};
			\node (w1) at (0.5,0) {$\scriptstyle{W_1}$};
			\node (f) at (0.25,3) [draw] {$\scriptstyle{f}$};
			\node (l) at (0.25,2) [draw] {$\scriptstyle{l^A_{W_1}}$};
			\node (eta) at (0.25, 1) [draw] {$\scriptstyle{\eta_{A,W_1}}$};
			\node (w2) at (0.25, 4) {$\scriptstyle{W_2}$};
			\draw (a.90) to (eta.220);
			\draw [line width=1pt] (w1.90) to (eta.320);
			\draw [line width=1pt] (eta.90) to (l.270);
			\draw [line width=1pt] (l.90) to (f.270);
			\draw [line width=1pt] (f.90) to (w2.270);
		\end{tikzpicture}
	\end{matrix}
	\end{align*}
	The first equality is due to \eqref{repAtensprodmorphdef}, the second and last are due to \eqref{leftisodef}, and the third follows because $f$ is a morphism in $\repA$.
%
\end{proof}

\begin{propo}\label{prop:rightisodef}
If $(W, \mu_W)$ is an object of $\repA$, then there is a unique $\repA$-isomorphism $r^A_W: W\boxtimes_A A\rightarrow W$ such that
\begin{equation}\label{rightisodef}
 r^A_{W}\circ\eta_{W, A}=\mu_W\circ\sR_{A,W}^{-1}: W\boxtimes A\rightarrow W,
\end{equation}
and $r^A_W$ is even.
\end{propo}
\begin{proof}
The existence and uniqueness of the $\sC$-morphism $r^A_W$ satisfying \eqref{rightisodef} follows from the universal property of the cokernel $(W\boxtimes_A A, \eta_{W,A})$, provided that $\mu_W\circ \sR_{A,W}^{-1}\circ\mu^{(2)}=\mu_W\circ \sR_{A,W}^{-1}\circ\mu^{(1)}$. 
This is proved in the following diagrams:	
\begin{align*}
	\begin{matrix}
		\begin{tikzpicture}[scale=0.9,out=up, in=down,line width=0.5pt]
			\node at (-0.2, -0.3) {$\scriptstyle{A}$};
			\node at (0.7, -0.3) {$\scriptstyle{W}$};
			\node at (1.5, -0.3) {$\scriptstyle{A}$};
			\node at (0.7, 4.3) {$\scriptstyle{W}$};
			\node (m) at (0.2,1.3) [draw,minimum width=20pt] {$\scriptstyle{\mu}$};
			\node (mW) at (0.7,3.5) [draw,minimum width=30pt] {$\scriptstyle{\mu_{W}}$};
			\draw [line width=1pt] (0.7,0) to (1.5,1) to (1.5,1.35) to (0,2.3) to (mW.330);
			\draw (-0.2,0) to (m.210);
			\draw [double=black, white, line width=3pt, double distance=0.5pt] (1.5,0) to (m.330);
			\draw [double=black, white, line width=3pt, double distance=0.5pt] (m.north) to (1.5,2.4) to (mW.210);
			\draw[line width=1pt](mW.north) to (0.7,4);
		\end{tikzpicture}
	\end{matrix}
	\,=\,	
	\begin{matrix}
		\begin{tikzpicture}[scale=0.9, out=up, in=down, line width=0.5pt]
			\node at (-0.3, -0.3) {$\scriptstyle{A}$};
			\node at (0.7, -0.3) {$\scriptstyle{W}$};
			\node at (1.5, -0.3) {$\scriptstyle{A}$};
			\node at (0.7, 4.3) {$\scriptstyle{W}$};
			\node (m) at (0.2,1.3) [draw,minimum width=20pt] {$\scriptstyle{\mu}$};
			\node (mW) at (0.7,3.5) [draw,minimum width=30pt] {$\scriptstyle{\mu_{W}}$};
			\draw[](-0.3,0) to (m.210);
			\draw[](m.north) to (mW.210);
			\draw[line width=1pt] (0.7,0) to (1.2,1) to (mW.330);
			\draw[line width=1pt](mW.north) to (0.7,4);
			\draw[line width=3pt, white, double=black, double distance=0.5pt](1.5,0) to (m.330);
		\end{tikzpicture}
	\end{matrix}
	\,=\,	
	\begin{matrix}
		\begin{tikzpicture}[scale=0.9, out=up, in=down,line width=0.5pt]
			\node at (-0.3, -0.3) {$\scriptstyle{A}$};
			\node at (0.7, -0.3) {$\scriptstyle{W}$};
			\node at (1.5, -0.3) {$\scriptstyle{A}$};
			\node at (0.7, 4.3) {$\scriptstyle{W}$};
			\node (m) at (0.2,2.3) [draw,minimum width=20pt] {$\scriptstyle{\mu}$};
			\node (mW) at (0.7,3.5) [draw,minimum width=30pt] {$\scriptstyle{\mu_{W}}$};
			\draw[](m.north) to (mW.210);
			\draw[](-0.3,0) to (m.330);
			\draw[line width=1pt] (0.7,0) to (1.2,2.3) to (mW.330);
			\draw[line width=1pt](mW.north)	to (0.7,4);
			\draw[white, line width=3pt, double=black, double distance=0.5pt](1.5,0)	to (m.210);
		\end{tikzpicture}
	\end{matrix}
	\,=\,
	\begin{matrix}
		\begin{tikzpicture}[scale=0.9, out=up, in=down,line width=0.5pt]
			\node at (0, -0.3) {$\scriptstyle{A}$};
			\node at (0.7, -0.3) {$\scriptstyle{W}$};
			\node at (1.5, -0.3) {$\scriptstyle{A}$};
			\node at (0.7, 4.3) {$\scriptstyle{W}$};
			\node (m) at (1.2,2.3) [draw,minimum width=30pt] {$\scriptstyle{\mu_W}$};
			\node (mW) at (0.7,3.5) [draw,minimum width=30pt] {$\scriptstyle{\mu_{W}}$};
			\draw[line width=1pt](m.north) to (mW.330);
			\draw[](0,0) to (m.210);
			\draw[line width=1pt] (0.7,0) to (m.330);
			\draw[line width=1pt](mW.north) to (0.7,4);
			\draw[white, line width=3pt, double=black, double distance=0.5pt](1.5,0) to (0.2,2.0)	to (mW.210);
		\end{tikzpicture}
	\end{matrix}
	\,=\,	
	\begin{matrix}
		\begin{tikzpicture}[scale=0.9, out=up, in=down,line width=0.5pt]
			\node at (0, -0.3) {$\scriptstyle{A}$};
			\node at (1, -0.3) {$\scriptstyle{W}$};
			\node at (1.5, -0.3) {$\scriptstyle{A}$};
			\node at (1, 4.3) {$\scriptstyle{W}$};
			\node (m) at (0.5,1) [draw,minimum width=30pt] {$\scriptstyle{\mu_W}$};
			\node (mW) at (1,3.5) [draw,minimum width=30pt] {$\scriptstyle{\mu_{W}}$};
			\draw[line width=1pt](m.north)  to (0.5,2) to (mW.330);
			\draw[](0,0) to (m.210);
			\draw[line width=1pt] (1,0) to (m.330);
			\draw[line width=1pt](mW.north) to (1,4);
			\draw[white, line width=3pt, double=black, double distance=0.5pt](1.5,0) to (1.5,2.0) to (mW.210);
		\end{tikzpicture}
	\end{matrix}
\end{align*}
In the first figure, we have used Remark \ref{rem:KOmu2} to rewrite $\mu^{(2)}$. The second equality uses commutativity of $\mu$, the third uses associativity of $\mu_W$, and the last uses the hexagon identity and naturality of the braiding applied to the even morphism $\mu_W$. This shows $r^A_W$ exists, and $r_W^A$ is even because $\eta_{W,A}$, $\mu_W$, $\sR_{A,W}^{-1}$ are even, and because $\eta_{W,A}$ is surjective.
	
	Next, we need to show that $r^A_W$ is a morphism in $\mathrm{Rep}\,A$, that is, $$r^A_W\circ\mu_{W\boxtimes_A A}=\mu_W\circ (1_A\boxtimes r^A_W): A\boxtimes (W\boxtimes_A A)\rightarrow W.$$
	For this, we use the following diagrams together with the fact that 
$1_A\boxtimes\eta_{W,A}$ is an epimorphism:
	\begin{align*}
		\begin{matrix}
			\begin{tikzpicture}[scale=1, line width=0.5pt, out=up, in=down]
				\node (a1) at (0,-0.3) {$\scriptstyle{A}$};
				\node (w) at (0.5,-0.3) {$\scriptstyle{W}$};
				\node (a2) at (1,-0.3) {$\scriptstyle{A}$};
				\node (wf) at (0.5, 3.7) {$\scriptstyle{W}$};
				\node (eta) at (0.75,0.7) [draw] {$\scriptstyle{\eta_{W,A}}$};
				\node (mu) at (0.5, 1.8) [draw] {$\scriptstyle{\mu_{W\boxtimes_A A}}$};
				\node (r) at (0.5, 2.8) [draw] {$\scriptstyle{r^A_W}$};
				\draw (a1.90) to (mu.210);
				\draw [line width=1pt] (w.90) to (eta.220);
				\draw (a2.90) to (eta.320);
				\draw [line width=1pt] (eta.90) to (mu.320);
				\draw [line width=1pt] (mu.90) to (r.270);
				\draw [line width=1pt] (r.90) to (wf.270);
			\end{tikzpicture}	
		\end{matrix}
		=
		\begin{matrix}
			\begin{tikzpicture}[scale=1, line width=0.5pt, out=up, in=down]
				\node (a1) at (0,-0.3) {$\scriptstyle{A}$};
				\node (w) at (0.5,-0.3) {$\scriptstyle{W}$};
				\node (a2) at (1,-0.3) {$\scriptstyle{A}$};
				\node (wf) at (0.5, 3.7) {$\scriptstyle{W}$};
				\node (mu) at (0.8, 0.7) [draw, minimum width=20 pt] {$\scriptstyle{\mu}$};
				\node (eta) at (0.5,1.8) [draw] {$\scriptstyle{\eta_{W,A}}$};
				\node (r) at (0.5, 2.8) [draw] {$\scriptstyle{r^A_W}$};
				\draw [line width=1pt] (w.90) to (0,0.7) to (eta.220);
				\draw [white, line width=3 pt, double=black] (a1.90) to (mu.220);
				\draw (a2.90) to (mu.320);
				\draw [line width=1pt] (eta.90) to (r.270);
				\draw [line width=1pt] (mu.90) to (eta.320);
				\draw [line width=1pt] (r.90) to (wf.270);
			\end{tikzpicture}	
		\end{matrix}
		=
		\begin{matrix}
			\begin{tikzpicture}[scale=1, line width=0.5pt, out=up, in=down]
				\node (a1) at (0,-0.3) {$\scriptstyle{A}$};
				\node (w) at (0.5,-0.3) {$\scriptstyle{W}$};
				\node (a2) at (1,-0.3) {$\scriptstyle{A}$};
				\node (wf) at (0.5, 3.7) {$\scriptstyle{W}$};
				\node (mu) at (0.8, 1.2) [draw, minimum width=20 pt] {$\scriptstyle{\mu}$};
				\node (mu2) at (0.5, 2.8) [draw] {$\scriptstyle{\mu_W}$};
				\draw [line width=1pt] (w.90) to (0,1.2) to (mu2.320);
				\draw [white, line width=3 pt, double=black] (a1.90) to (mu.220);
				\draw [white, line width=3 pt, double=black] (mu.90) to (mu2.220);
				\draw (a2.90) to (mu.320);
				\draw [line width=1pt] (mu2.90) to (wf.270);
			\end{tikzpicture}	
		\end{matrix}
		=
		\begin{matrix}
			\begin{tikzpicture}[scale=1, line width=0.5pt, out=up, in=down]
				\node (a1) at (0,-0.3) {$\scriptstyle{A}$};
				\node (w) at (0.5,-0.3) {$\scriptstyle{W}$};
				\node (a2) at (1,-0.3) {$\scriptstyle{A}$};
				\node (wf) at (0.5, 3.7) {$\scriptstyle{W}$};
				\node (mu) at (0.2, 1.2) [draw, minimum width=20 pt] {$\scriptstyle{\mu}$};
				\node (mu2) at (0.5, 2.8) [draw] {$\scriptstyle{\mu_W}$};
				\draw [line width=1pt] (w.90) to (1,1) to (mu2.320);
				\draw [white, line width=3 pt, double=black] (a1.90) to (mu.220);
				\draw [white, line width=3 pt, double=black] (mu.90) to (mu2.220);
				\draw [white, line width=3 pt, double=black] (a2.90) to (mu.320);
				\draw [line width=1pt] (mu2.90) to (wf.270);
			\end{tikzpicture}	
		\end{matrix}	
		=
		\begin{matrix}
			\begin{tikzpicture}[scale=1, line width=0.5pt, out=up, in=down]
				\node (a1) at (0,-0.3) {$\scriptstyle{A}$};
				\node (w) at (0.5,-0.3) {$\scriptstyle{W}$};
				\node (a2) at (1,-0.3) {$\scriptstyle{A}$};
				\node (wf) at (0.5, 3.7) {$\scriptstyle{W}$};
				\node (mu) at (0.8, 1.2) [draw, minimum width=20 pt] {$\scriptstyle{\mu_W}$};
				\node (mu2) at (0.5, 2.8) [draw] {$\scriptstyle{\mu_W}$};
				\draw [line width=1pt] (w.90) to  (mu.320);
				\draw [white, line width=3 pt, double=black] (a1.90) to (mu2.220);
				\draw [line width=1pt] (mu.90) to (mu2.320);
				\draw [white, line width=3 pt, double=black] (a2.90) to (mu.220);
				\draw [line width=1pt] (mu2.90) to (wf.270);
			\end{tikzpicture}	
		\end{matrix}	
		=
		\begin{matrix}
			\begin{tikzpicture}[scale=1, line width=0.5pt, out=up, in=down]
				\node (a1) at (0,-0.3) {$\scriptstyle{A}$};
				\node (w) at (0.6,-0.3) {$\scriptstyle{W}$};
				\node (a2) at (1.2,-0.3) {$\scriptstyle{A}$};
				\node (wf) at (0.5, 3.7) {$\scriptstyle{W}$};
				\node (mu2) at (0.5, 2.8) [draw] {$\scriptstyle{\mu_W}$};
				\node (r) at ( 0.9,1.8) [draw] {$\scriptstyle{r^A_W}$};
				\node (eta) at ( 0.9,0.8) [draw] {$\scriptstyle{\eta_{W,A}}$};
				\draw [white, line width=3 pt, double=black] (a1.90) to (mu2.220);
				\draw [line width=1pt] (mu2.90) to (wf.270);
				\draw [line width=1pt] (w.90) to (eta.220);
				\draw (a2.90) to (eta.320);
				\draw [line width=1pt] (eta.90) to (r.270);
				\draw [line width=1pt] (r.90) to (mu2.325);
			\end{tikzpicture}	
		\end{matrix}	
	\end{align*}
 The first equality uses \eqref{repAtensdef} with $i=2$,	while the second and last use \eqref{rightisodef}. The third and fourth equalities use naturality of the braiding applied to $\mu$, the hexagon axiom, and the associativity of $\mu_W$.
	
	Now we show that $r^A_W$ has an inverse in $\mathcal{SC}$ (which will also be an inverse in $\repA$). We define $\widetilde{r}^A_W$ to be the composition
	\begin{equation*}
	W\xrightarrow{\sright_W^{-1}} W\boxtimes\sunit\xrightarrow{1_W\boxtimes\iota_A} W\boxtimes A\xrightarrow{\eta_{W,A}} W\boxtimes_A A.
	\end{equation*}
	To show that $\widetilde{r}^A_W$ is the inverse of $r^A_W$, note that
	by \eqref{rightisodef}, $r^A_W\circ\widetilde{r}^A_W$ equals the composition
	\begin{align*}
	W\xrightarrow{\sright_W^{-1}} W\boxtimes\sunit\xrightarrow{1_W\boxtimes\iota_A} W\boxtimes A\xrightarrow{\sR_{A,W}^{-1}} A\boxtimes W\xrightarrow{\mu_W} W.
	\end{align*}
	Using naturality of the braiding with respect to even morphisms, this in turn equals
	\begin{equation*}
	W\xrightarrow{\sright_W^{-1}} W\boxtimes\sunit\xrightarrow{\sR_{1,W}^{-1}}\sunit\boxtimes W\xrightarrow{\iota_A\boxtimes 1_W} A\boxtimes W\xrightarrow{\mu_W} W.
	\end{equation*}
	Now, $\sR_{\sunit,W}^{-1}\circ \sright_W^{-1}=\sleft_W^{-1}$ (see for instance Proposition XIII.1.2 in \cite{Ka}), so the above composition equals $1_W$ by the unit property of $\mu_W$. Consequently $r^A_W\circ\widetilde{r}^A_W=1_W$.
	
	On the other hand, $\widetilde{r}^A_W\circ r^A_W\circ\eta_{W,A}=\eta_{W,A}\circ\rho$, where $\rho$ is the composition
	\begin{align*}
	W\boxtimes A\xrightarrow{\sR_{A,W}^{-1}} A\boxtimes W\xrightarrow{\mu_W} W\xrightarrow{\sright_W^{-1}} W\boxtimes\sunit\xrightarrow{1_W\boxtimes\,\iota_A} W\boxtimes A.
	\end{align*}
	It is now enough to show that $\eta_{W,A}\circ(\rho-1_{W\boxtimes A})=0$, and for this it is enough to show that
	\begin{equation*}
	(\rho-1_{W\boxtimes A})\circ \sR_{A,W}\circ (1_A\boxtimes \sright_W)=(\mu^{(1)}-\mu^{(2)})\circ(1_A\boxtimes (1_W\boxtimes\iota_A)): A\boxtimes (W\boxtimes\sunit)\rightarrow W\boxtimes A.
	\end{equation*}
	First, we prove 
		$\mu^{(2)}\circ(1_A\boxtimes(1_W\boxtimes\iota_A))=1_{W\boxtimes A}\circ \sR_{A,W}\circ(1_A\boxtimes \sright_W)$
	using the following diagrams:
		\begin{align*}
		\begin{matrix}
			\begin{tikzpicture}[scale=1, out=up, in=down, line width=0.5pt]
				\node at (0,-0.3) {$\scriptstyle{A}$};
				\node at (0.7,-0.3) {$\scriptstyle{W}$};
				\node at (1.5,-0.3) {$\scriptstyle{\sunit}$};
				\node at (0.3,3.8) {$\scriptstyle{W}$};
				\node at (1.3,3.8) {$\scriptstyle{A}$};
				\node (m) at (0.5,2) [draw,minimum width=30pt] {$\scriptstyle{\mu}$};
				\node (mW) at (1.5,0.5) {$\bullet$};
				\draw[line width=1pt] (0.7,0) to  (0.7,0.5) to  (1.7,2) to  (0.3,3.5);
				\draw[](0,0) to  (m.210);
				\draw[dashed](1.5,0) to  (1.5,0.5);
				\draw[white, line width=3pt, double=black, double distance=0.5pt](mW.90) to  (m.330);
				\draw[white, line width=3pt, double=black, double distance=0.5pt](m.north) to  (1.3,3.5);
			\end{tikzpicture}
		\end{matrix}
		\,=\,	
		\begin{matrix}
			\begin{tikzpicture}[scale=1, out=up, in=down, line width=0.5pt]
				\node at (0,-0.3) {$\scriptstyle{A}$};
				\node at (0.7,-0.3) {$\scriptstyle{W}$};
				\node at (1.5,-0.3) {$\scriptstyle{\sunit}$};
				\node at (0.3,3.8) {$\scriptstyle{W}$};
				\node at (1.3,3.8) {$\scriptstyle{A}$};
				\node (m) at (0.5,2) [draw,minimum width=30pt] {$\scriptstyle{\mu}$};
				\node (mW) at (1,1.2) {$\bullet$};
				\draw[line width=1pt] (0.7,0) to (1.7,1) to (1.7,2) 		to (0.3,3.5);
				\draw[](0,0)to (m.210);
				\draw[white, line width=6pt](1.5,0) to  (mW.270);
				\draw[dashed](1.5,0) to (mW.270);
				\draw[](mW.90) to (m.330);
				\draw[white, line width=3pt, double=black, double distance=0.5pt](m.north) to (1.3,3.5);
			\end{tikzpicture}
		\end{matrix}
		\,=\,
		\begin{matrix}
			\begin{tikzpicture}[scale=1, out=up, in=down, line width=0.5pt]
				\node at (0,-0.3) {$\scriptstyle{A}$};
				\node at (0.7,-0.3) {$\scriptstyle{W}$};
				\node at (1.5,-0.3) {$\scriptstyle{\sunit}$};
				\node at (0.3,3.8) {$\scriptstyle{W}$};
				\node at (1.3,3.8) {$\scriptstyle{A}$};
				\draw[line width=1pt] (0.7,0) to (1.5,2) to (0.3,3.5);
				\draw[white, line width=4 pt,double=black, double distance=0.5pt](0,0) to (0,2) to (1.3,3.5); 					
				\draw[white, line width=6 pt](1.5,0) to [out=up, in=right] (0,1.8);
				\draw[dashed](1.5,0) to [out=up, in=right] (0,1.8);
			\end{tikzpicture}
		\end{matrix}
		\,=\,		
		\begin{matrix}
			\begin{tikzpicture}[scale=1, out=up, in=down, line width=0.5pt]
				\node at (0,-0.3) {$\scriptstyle{A}$};
				\node at (0.7,-0.3) {$\scriptstyle{W}$};
				\node at (1.5,-0.3) {$\scriptstyle{\sunit}$};
				\node at (0.3,3.8) {$\scriptstyle{W}$};
				\node at (1.3,3.8) {$\scriptstyle{A}$};
				\draw[line width=1pt] (0.7,0) to (1.5,1) to (1.5,2) to (0.3,3.5);
				\draw[white, line width=4pt, double=black, double distance=0.5pt](0,0) to (0,2) to (1.3,3.5);
				\draw[white, line width=6pt](1.5,0) to  (0.7,1) to [in=left] (1.5,1.8);
				\draw[dashed](1.5,0) to  (0.7,1) to [in=left] (1.5,1.8);
			\end{tikzpicture}
		\end{matrix}
		\,=\,
		\begin{matrix}
			\begin{tikzpicture}[scale=1, out=up, in=down, line width=0.5pt]
				\node at (0,-0.3) {$\scriptstyle{A}$};
				\node at (1,-0.3) {$\scriptstyle{W}$};
				\node at (1.5,-0.3) {$\scriptstyle{\sunit}$};
				\node at (0,3.8) {$\scriptstyle{W}$};
				\node at (1,3.8) {$\scriptstyle{A}$};
				\draw[line width=1pt] (1,0) to (1,2) to (0,3.5);
				\draw[white, line width=3pt, double=black, double distance=0.5pt](0,0) to (0,2) to (1,3.5);
				\draw[dashed](1.5,0) to (1.5,1) to [in=right] (1,1.8);
			\end{tikzpicture}
		\end{matrix}			
	\end{align*}
	We complete the proof by establishing $\rho\circ \sR_{A,W}\circ(1_A\boxtimes \sright_W)=\mu^{(1)}\circ(1_A\boxtimes(1_W\boxtimes\iota_A))$:
	\begin{align*}
	\begin{matrix}
		\begin{tikzpicture}[scale=1, out=up, in=down, line width=0.5pt]
			\node at (0,-0.3) {$\scriptstyle{A}$};
			\node at (1,-0.3) {$\scriptstyle{W}$};
			\node at (1.5,-0.3) {$\scriptstyle{\sunit}$};
			\node at (0.5, 4.3) {$\scriptstyle{W}$};
			\node at (1.5, 4.3) {$\scriptstyle{A}$};
			\node (d) at (1.5,3.5) {$\bullet$};
			\node (mW) at (0.5,2.5) [draw,minimum width=30pt] {$\scriptstyle{\mu_W}$};
			\draw[dashed](1.5,0) to [out=up, in=right] (1,0.5);
			\draw[line width=1pt] (1,0) to (1,0.6) to (0,1.5) to (mW.330);
			\draw[line width=1pt](mW.north) to (0.5,4);
			\draw[dashed](0.5,3) to [out=right, in=down] (1.5,3.5);
			\draw [] (1.5,3.5) to [out=up, in=down] (1.5,4);
			\draw[white, line width=3 pt, double=black, double distance=0.5pt](0,0) to (0,0.6) to (1,1.5) to (mW.210);
		\end{tikzpicture}
	\end{matrix}
	\,=\,	
	\begin{matrix}
		\begin{tikzpicture}[scale=1, out=up, in=down, line width=0.5pt]
			\node at (0,-0.3) {$\scriptstyle{A}$};
			\node at (1,-0.3) {$\scriptstyle{W}$};
			\node at (1.5,-0.3) {$\scriptstyle{\sunit}$};
			\node at (0.5, 4.3) {$\scriptstyle{W}$};
			\node at (1.5, 4.3) {$\scriptstyle{A}$};
			\node (d) at (1.5,3.5) {$\bullet$};
			\node (mW) at (0.5,2) [draw,minimum width=30pt,minimum height=8pt,fill=white] {$\scriptstyle{\mu_W}$};
			\draw[](0,0) to (0,0.6) to (mW.210);
			\draw[dashed](1.5,0) to [out=up, in=right] (1,0.5);
			\draw[line width=1pt] (1,0) to (1,0.6) to (mW.330);
			\draw[line width=1pt](mW.north) to (0.5,4);
			\draw[dashed](0.5,3) to [out=right, in=down] (1.5,3.5);
			\draw [] (1.5,3.5) to (1.5,4);
		\end{tikzpicture}
	\end{matrix}
	\,=\,	
	\begin{matrix}
		\begin{tikzpicture}[scale=1, out=up, in=down, line width=0.5pt]
			\node at (0,-0.3) {$\scriptstyle{A}$};
			\node at (1,-0.3) {$\scriptstyle{W}$};
			\node at (1.5,-0.3) {$\scriptstyle{\sunit}$};
			\node at (0.5, 4.3) {$\scriptstyle{W}$};
			\node at (1.5, 4.3) {$\scriptstyle{A}$};
			\node (d) at (1.5,3.5) {$\bullet$};
			\node (mW) at (0.5,1) [draw,minimum width=30pt] {$\scriptstyle{\mu_W}$};
			\draw[](0,0) to (mW.210);
			\draw[dashed](1.5,0) to [out=up, in=right] (0.5,2.5);
			\draw[line width=1pt] (1,0) to (mW.330);
			\draw[line width=1pt](mW.north) to (0.5,4);
			\draw[dashed](0.5,3) to [out=right, in=down] (1.5,3.5);
			\draw [] (1.5,3.5) to (1.5,4);
		\end{tikzpicture}
	\end{matrix}
	\,=\,	
	\begin{matrix}
		\begin{tikzpicture}[scale=1, out=up, in=down, line width=0.5pt]	
			\node at (0,-0.3) {$\scriptstyle{A}$};
			\node at (1,-0.3) {$\scriptstyle{W}$};
			\node at (1.5,-0.3) {$\scriptstyle{\sunit}$};
			\node at (0.5, 4.3) {$\scriptstyle{W}$};
			\node at (1.5, 4.3) {$\scriptstyle{A}$};
			\node (c) at (1.5,1) {$\bullet$};
			\node (mW) at (0.5,2.5) [draw,minimum width=30pt] {$\scriptstyle{\mu_W}$};
			\draw[](0,0) to (mW.210);
			\draw[dashed](1.5,0) to  (1.5,1);
			\draw[line width=1pt] (1,0) to (mW.330);
			\draw[line width=1pt](mW.north) to (0.5,4);
			\draw [] (1.5,1) to (1.5,4);
		\end{tikzpicture}
	\end{matrix}.
	\end{align*}	
The second step uses $1_A\boxtimes \sright_W=\sright_{A\boxtimes W}\circ\sA_{A,W,\sunit}$ (see Lemma XI.2.2 in \cite{Ka}) and naturality of $\sright$.
\end{proof}

\begin{propo}
	The $\repA$-isomorphisms $r^A_W$ define a natural isomorphism from $\cdot\boxtimes_A A$ to the identity functor on $\repA$.
\end{propo}
\begin{proof}
	We just need to verify that for any morphism $f: W_1\rightarrow W_2$ in $\repA$, we have
	\begin{equation}\label{rightunitnat}
	r^A_{W_2}\circ(f\boxtimes_A 1_A)=f\circ r^A_{W_1}: W_1\boxtimes_A A\rightarrow W_2.
	\end{equation}
	Since $\eta_{W_1,A}$ is an epimorphism, \eqref{rightunitnat} follows  from:
	\begin{align*}
		\begin{matrix}
			\begin{tikzpicture}[scale=1, line width=0.5pt, out=up, in=down]
				\node (w) at (0,0) {$\scriptstyle{W_1}$};
				\node (a) at (0.5,0) {$\scriptstyle{A}$};
				\node (eta) at (0.25,1) [draw] {$\scriptstyle{\eta_{W_1,A}}$};
				\node (f) at (0.25,2) [draw] {$\scriptstyle{f\,\boxtimes_A 1_A}$};
				\node (r) at (0.25,3) [draw] {$\scriptstyle{r^A_{W_2}}$};
				\node (wf) at (0.25, 4) {$\scriptstyle{W_2}$};
				\draw [line width=1pt] (w.90) to (eta.220);
				\draw (a.90) to (eta.320);
				\draw [line width=1pt] (eta.90) to (f.270);
				\draw [line width=1pt] (f.90) to (r.270);
				\draw [line width=1pt] (r.90) to (wf.270);
			\end{tikzpicture}					
		\end{matrix}
		=
		\begin{matrix}
			\begin{tikzpicture}[scale=1, line width=0.5pt, out=up, in=down]
				\node (w) at (0,0) {$\scriptstyle{W_1}$};
				\node (a) at (0.5,0) {$\scriptstyle{A}$};
				\node (f) at (0,1) [draw] {$\scriptstyle{f}$};
				\node (eta) at (0.25,2) [draw] {$\scriptstyle{\eta_{W_2,A}}$};
				\node (r) at (0.25,3) [draw] {$\scriptstyle{r^A_{W_2}}$};
				\node (wf) at (0.25, 4) {$\scriptstyle{W_2}$};
				\draw [line width=1pt] (w.90) to (f.270);
				\draw (a.90) to (eta.320);
				\draw [line width=1pt] (eta.90) to (r.270);
				\draw [line width=1pt] (f.90) to (eta.220);
				\draw [line width=1pt] (r.90) to (wf.270);
			\end{tikzpicture}					
		\end{matrix}
		=
		\begin{matrix}
			\begin{tikzpicture}[scale=1, line width=0.5pt, out=up, in=down]
				\node (w) at (0,0) {$\scriptstyle{W_1}$};
				\node (a) at (0.5,0) {$\scriptstyle{A}$};
				\node (f) at (0,1) [draw] {$\scriptstyle{f}$};
				\node (mu) at (0.25,3) [draw] {$\scriptstyle{\mu_{W_2}}$};
				\node (wf) at (0.25, 4) {$\scriptstyle{W_2}$};
				\draw [line width=1pt] (w.90) to (f.270);
				\draw [line width=1pt] (f.90) to (0,1.5) to (mu.320);
				\draw [white, line width=3pt, double=black] (a.90) to (0.5,1.5) to (mu.220);
				\draw [line width=1pt] (mu.90) to (wf.270);
			\end{tikzpicture}					
		\end{matrix}
		=
		\begin{matrix}
			\begin{tikzpicture}[scale=1, line width=0.5pt, out=up, in=down]
				\node (w) at (0,0) {$\scriptstyle{W_1}$};
				\node (a) at (0.5,0) {$\scriptstyle{A}$};
				\node (f) at (0.5,2) [draw] {$\scriptstyle{f}$};
				\node (mu) at (0.25,3) [draw] {$\scriptstyle{\mu_{W_2}}$};
				\node (wf) at (0.25, 4) {$\scriptstyle{W_2}$};
				\draw [line width=1pt] (w.90) to (f.270);
				\draw [line width=1pt] (f.90) to (mu.320);
				\draw [white, line width=3pt, double=black] (a.90) to (0,2) to (mu.220);
				\draw [line width=1pt] (mu.90) to (wf.270);
			\end{tikzpicture}					
		\end{matrix}
		=	
		\begin{matrix}
			\begin{tikzpicture}[scale=1, line width=0.5pt, out=up, in=down]
				\node (w) at (0,0) {$\scriptstyle{W_1}$};
				\node (a) at (0.6,0) {$\scriptstyle{A}$};
				\node (wf) at (0.25, 4) {$\scriptstyle{W_2}$};
				\node (f) at (0.25, 3) [draw, minimum width=20 pt] {$\scriptstyle{f}$};
				\node (mu) at (0.25,2) [draw] {$\scriptstyle{\mu_{W_1}}$};
				\draw [line width=1pt] (f.90) to (wf.270);
				\draw [line width=1pt] (mu.90) to (f.270);
				\draw [line width=1pt] (w.90) to (mu.320);
				\draw [white, line width=3pt, double=black] (a.90) to (mu.220);
			\end{tikzpicture}					
		\end{matrix}
		=
		\begin{matrix}
			\begin{tikzpicture}[scale=1, line width=0.5pt, out=up, in=down]
				\node (w) at (0,0) {$\scriptstyle{W_1}$};
				\node (a) at (0.52,0) {$\scriptstyle{A}$};
				\node (wf) at (0.25, 4) {$\scriptstyle{W_2}$};
				\node (f) at (0.25, 3) [draw, minimum width=20 pt] {$\scriptstyle{f}$};
				\node (r) at (0.25,2) [draw] {$\scriptstyle{r^A_{W_1}}$};
				\node (eta) at (0.25,1) [draw] {$\scriptstyle{\eta_{W_1,A}}$};
				\draw [line width=1pt] (f.90) to (wf.270);
				\draw [line width=1pt] (r.90) to (f.270);
				\draw [line width=1pt] (eta.90) to (r.270);
				\draw [line width=1pt] (w.90) to (eta.220);
				\draw [white, line width=3pt, double=black] (a.90) to (eta.320);
			\end{tikzpicture}					
		\end{matrix}
	\end{align*}
The first equality uses \eqref{repAtensprodmorphdef}, the second and last use \eqref{rightisodef}, and the third is due to naturality of braiding and evenness of $1_A$. The fourth follows because $f$ is a morphism in $\repA$.
\end{proof}

\subsection{Categorical intertwining operators and a universal property of tensor products in \texorpdfstring{$\repA$}{Rep A}}
\label{subsec:catintwop}

The goal of this subsection is to characterize the tensor product in $\repA$ by a universal property analogous to the defining property of tensor products of vertex operator (super)algebra modules in terms of intertwining operators. For this, we introduce ``categorical $\repA$-intertwining operators'' that exhibit, in the abstract tensor-categorical setting, the commutativity and associativity properties of vertex-algebraic intertwining operators.

\begin{defi}\label{intwop}
Given objects $(W_1,\mu_{W_1})$, $(W_2, \mu_{W_2})$, and $(W_3, \mu_{W_3})$ of $\repA$,	a \textit{categorical $\repA$-intertwining operator of type $\binom{W_3}{W_1\,W_2}$} is an $\mathcal{SC}$-morphism $\eta: W_1\boxtimes W_2\rightarrow W_3$ such that the three compositions
	\begin{align*}
	A\boxtimes (W_1\boxtimes W_2)\xrightarrow{\sA_{A,W_1,W_2}} (A\boxtimes W_1)\boxtimes W_2\xrightarrow{\mu_{W_1}\boxtimes 1_{W_2}} W_1\boxtimes W_2\xrightarrow{\eta} W_3,
	\end{align*}
	\begin{align*}
	A\boxtimes(W_1\boxtimes W_2)\xrightarrow{\sA_{A,W_1,W_2}} (A\boxtimes W_1) & \boxtimes W_2  \xrightarrow{\sR_{A,W_1}\boxtimes 1_{W_2}} (W_1\boxtimes A)\boxtimes W_2\nonumber\\
	&\xrightarrow{\sA_{W_1,A,W_2}^{-1}} W_1\boxtimes(A\boxtimes W_2)\xrightarrow{1_{W_1}\boxtimes\mu_{W_2}} W_1\boxtimes W_2\xrightarrow{\eta} W_3,
	\end{align*}
	and
	\begin{align*}
	A\boxtimes(W_1\boxtimes W_2)\xrightarrow{1_A\boxtimes \eta} A\boxtimes W_3\xrightarrow{\mu_{W_3}} W_3
	\end{align*}
	are equal. Below, we shall often drop the qualifier ``categorical'' when it is clear.
\end{defi}

\begin{rema}
	The first two compositions in Definition \ref{intwop} are $\eta\circ\mu^{(1)}$ and $\eta\circ\mu^{(2)}$, respectively, so a $\repA$-intertwining operator $\eta$ of type $\binom{W_3}{W_1\,W_2}$ is defined by the equalities
	\begin{equation*}
	\eta\circ\mu^{(1)}=\eta\circ\mu^{(2)}=\mu_{W_3}\circ(1_A\boxtimes \eta).
	\end{equation*}
	In other words, a $\repA$-intertwining operator intertwines the action of $A$ on $W_3$ with the left and right actions of $A$ on $W_1\boxtimes W_2$.
\end{rema}

From the preceding remark and the definition of the action of $A$ on $W_1\boxtimes_A W_2$, we have:
\begin{propo}\label{prop:univIntw}
	If $W_1$ and $W_2$ are objects of $\repA$, the cokernel morphism $\eta_{W_1,W_2}: W_1\boxtimes W_2\rightarrow W_1\boxtimes_A W_2$ is a $\repA$-intertwining operator of type $\binom{W_1\boxtimes_A W_2}{W_1\,W_2}$.
\end{propo}

The pair $(W_1\boxtimes_A W_2, \eta_{W_1,W_2})$ satisfies the following universal property:
\begin{propo}\label{prop:repAuniv}
	For any $\repA$-intertwining operator $\eta$ of type $\binom{W_3}{W_1\,W_2}$, there is a unique $\repA$-morphism
	\begin{equation*}
	f_\eta: W_1\boxtimes_A W_2\rightarrow W_3
	\end{equation*}
	such that $f_\eta\circ \eta_{W_1,W_2}=\eta$.
\end{propo}
\begin{proof}
	Since $\eta\circ\mu^{(1)} =\eta\circ\mu^{(2)}$, the universal property of the cokernel $(W_1\boxtimes_A W_2, \eta_{W_1,W_2})$ induces a unique $\sC$-morphism $f_\eta: W_1\boxtimes_A W_2\rightarrow W_3$ such that $f_\eta\circ\eta_{W_1,W_2}=\eta$. To see that $f_\eta$ is a morphism in $\repA$, the definitions imply
	\begin{align*}
	\mu_{W_3}\circ (1_A & \boxtimes f_\eta)\circ(1_A\boxtimes \eta_{W_1,W_2})=\mu_{W_3}\circ(1_A\boxtimes \eta)=\eta\circ\mu^{(i)}\nonumber\\
	&=f_\eta\circ \eta_{W_1,W_2}\circ\mu^{(i)}=f_\eta\circ\mu_{W_1\boxtimes_A  W_2}\circ(1_A\boxtimes \eta_{W_1,W_2}),
	\end{align*}
	where $i=1$ or $2$. Since $1_A\boxtimes \eta_{W_1,W_2}$ is surjective, $\mu_{W_3}\circ(1_A\boxtimes f_\eta)=f_\eta\circ\mu_{W_1\boxtimes_A W_2}$ as required.
\end{proof}

As another example of $\repA$-intertwining operators, we have:
\begin{propo}
	For any object $(W, \mu_W)$ of $\repA$, $\mu_W$ is a $\repA$-intertwining operator of type $\binom{W}{A\,W}$.
\end{propo}
\begin{proof}
The associativity of $\mu_W$ implies
\begin{equation*}
 \mu_W\circ\mu^{(1)}=\mu_W\circ(\mu\boxtimes 1_W)\circ\sA_{A,A,W}=\mu_W\circ(1_A\boxtimes\mu_W),
\end{equation*}
and we already showed at the beginning of the proof of Proposition \ref{repAleftunit} that $\mu_W\circ\mu^{(1)}=\mu_W\circ\mu^{(2)}$. 
\end{proof}

Composing an intertwining operator with a morphism in $\repA$ yields a new intertwining operator:
\begin{propo}\label{compintwopmorph}
 If $\eta$ is a $\repA$-intertwining operator of type $\binom{W_3}{W_1\,W_2}$ and $f: W_3\rightarrow\widetilde{W}_3$ is a morphism in $\repA$, then $f\circ\eta$ is a $\repA$-intertwining operator of type $\binom{\widetilde{W}_3}{W_1\,W_2}$.
\end{propo}
\begin{proof}
 We have
 \begin{equation*}
  (f\circ\eta)\circ\mu^{(1)}=(f\circ\eta)\circ\mu^{(2)}=f\circ\mu_{W_3}\circ(1_A\boxtimes\eta)=\mu_{\widetilde{W}_3}\circ(1_A\boxtimes f)\circ(1_A\boxtimes\eta)=\mu_{\widetilde{W}_3}\circ(1_A\boxtimes(f\circ\eta))
 \end{equation*}
since $\eta$ is a $\repA$-intertwining operator and $f$ is a $\repA$-morphism.
\end{proof}

 From the bilinearity of composition and tensor products in $\sC$, categorical $\repA$-intertwining operators of type $\binom{W_3}{W_1\,W_2}$ form a subspace of $\hom_{\sC}(W_1\boxtimes W_2, W_3)$. In fact, we can use the universal property of $\repA$-tensor products to show that they form a sub-superspace:
 \begin{propo}
  The parity-homogeneous components of a $\repA$-intertwining operator are $\repA$-intertwining operators.
 \end{propo}
\begin{proof}
 Suppose $\eta$ is a $\repA$-intertwining operator of type $\binom{W_3}{W_1\,W_2}$. Then
 \begin{equation*}
  \eta=f_\eta\circ\eta_{W_1,W_2}=f_\eta^\even\circ\eta_{W_1,W_2}+f_\eta^\odd\circ\eta_{W_1,W_2}.
 \end{equation*}
Since $\eta_{W_1,W_2}$ is even, $f_\eta^\even\circ\eta_{W_1,W_2}$ and $f_\eta^\odd\circ\eta_{W_1,W_2}$ are the even and odd components, respectively, of $\eta$. Because $\repA$ is a supercategory by Proposition \ref{SuperRepAadditive}, the even and odd components of $f_\eta$ are $\repA$-morphisms, so the even and odd components of $\eta$ are $\repA$-intertwining operators by Proposition \ref{compintwopmorph}.
\end{proof}

\subsection{Associativity of categorical \texorpdfstring{$\repA$}{Rep A}-intertwining operators and the tensor product in \texorpdfstring{$\repA$}{Rep A}} \label{subsec:associnrepA}

In this subsection we prove an associativity theorem for categorical $\repA$-intertwining operators that is motivated by the associativity of intertwining operators among modules for a vertex operator algebra proved in \cite{H-tensor4, HLZ6}. As in the vertex operator algebra setting, associativity of $\repA$-intertwining operators yields associativity isomorphisms in $\repA$. Using these associativity isomorphisms, we prove that $\repA$ is a monoidal supercategory.

\begin{theo}\label{repAintwopassoc}
	Suppose $(W_i,\mu_{W_i})$ for $i=1,2,3,4$ and $(M_i,\mu_{M_i})$ for $i=1,2$ are objects of $\repA$.
	\begin{enumerate}
		\item If $\eta_1$ and $\eta_2$ are $\repA$-intertwining operators of types $\binom{W_4}{W_1\,M_1}$ and $\binom{M_1}{W_2\,W_3}$, respectively, then there is a unique $\repA$-intertwining operator $\eta$ of type $\binom{W_4}{W_1\boxtimes_A W_2\,W_3}$ such that
		\begin{equation}\label{eqn:repAintwopassoc_1_def}
		\eta_1\circ(1_{W_1}\boxtimes\eta_2)=\eta\circ(\eta_{W_1,W_2}\boxtimes 1_{W_3})\circ\sA_{W_1,W_2,W_3}: W_1\boxtimes (W_2\boxtimes W_3)\rightarrow W_4.
		\end{equation}
		
		\item If $\eta^1$ and $\eta^2$ are $\repA$-intertwining operators of types $\binom{W_4}{M_2\,W_3}$ and $\binom{M_2}{W_1\,W_2}$, respectively, then there is a unique $\repA$-intertwining operator $\eta$ of type $\binom{W_4}{W_1\,W_2\boxtimes_A W_3}$ such that 
		\begin{equation}\label{eqn:repAintwopassoc_2_def}
		\eta^1\circ(\eta^2\boxtimes 1_{W_3})=\eta\circ(1_{W_1}\boxtimes\eta_{W_2,W_3})\circ\sA^{-1}_{W_1,W_2,W_3}: (W_1\boxtimes W_2)\boxtimes W_3\rightarrow W_4.
		\end{equation}
	\end{enumerate}
\end{theo}
\begin{proof}
We present the proof of the first assertion only, since the second is similar. Since $\cdot\boxtimes W_3$ is right exact, the sequence
	\begin{equation*}
	(A\boxtimes(W_1\boxtimes W_2))\boxtimes W_3\xrightarrow{(\mu^{(1)}-\mu^{(2)})\boxtimes 1_{W_3}}(W_1\boxtimes W_2)\boxtimes W_3\xrightarrow{\eta_{W_1,W_2}\boxtimes 1_{W_3}} (W_1\boxtimes_A W_2)\boxtimes W_3\rightarrow 0
	\end{equation*}
	is exact, that is, $((W_1\boxtimes_A W_2)\boxtimes W_3, \eta_{W_1,W_2}\boxtimes 1_{W_3})$ is a cokernel of $(\mu^{(1)}-\mu^{(2)})\boxtimes 1_{W_3}$. Thus the universal property of the cokernel will imply the existence of a unique $\mathcal{SC}$-morphism
	\begin{equation*}
	\eta: (W_1\boxtimes_A W_2)\boxtimes W_3\rightarrow W_4
	\end{equation*}
	such that
	\begin{equation*}
	\eta_1\circ(1_{W_1}\boxtimes\eta_2)=\eta\circ(\eta_{W_1,W_2}\boxtimes 1_{W_3})\circ\sA_{W_1,W_2,W_3}
	\end{equation*}
	provided that
	\begin{equation*}
	\eta_1\circ(1_{W_1}\boxtimes\eta_2)\circ\sA_{W_1,W_2,W_3}^{-1}\circ(\mu^{(2)}\boxtimes 1_{W_3})=
	\eta_1\circ(1_{W_1}\boxtimes\eta_2)\circ\sA_{W_1,W_2,W_3}^{-1}\circ(\mu^{(1)}\boxtimes 1_{W_3}).
	\end{equation*}
	This is proved as indicated in the diagrams:
	\begin{align*}
	\begin{matrix}
		\begin{tikzpicture}[scale=0.85, out=up, in=down, line width=0.5pt]
			\node at (0,-0.3) {$\scriptstyle{A}$};
			\node at (1,-0.3) {$\scriptstyle{W_1}$};
			\node at (2,-0.3) {$\scriptstyle{W_2}$};
			\node at (3,-0.3) {$\scriptstyle{W_3}$};
			\node at (0.5, 4.8) {$\scriptstyle{W_4}$};
			\node (mW2) at (1.5,1.2) [draw,minimum width=30pt,minimum height=8pt,fill=white] {$\scriptstyle{\mu_{W_2}}$};
			\node (n1) at (0.5,3.5) [draw,minimum width=30pt,minimum height=8pt,fill=white] {$\scriptstyle{\eta_1}$};
			\node (n2) at (2,2.3) [draw,minimum width=30pt,minimum height=8pt,fill=white] {$\scriptstyle{\eta_2}$};
			\draw[line width=1pt] (1,0) to (0,1) to (0,2) to (n1.210);
			\draw[line width=1pt] (2,0) to (mW2.330);
			\draw[line width=1pt] (3,0) to (n2.330);
			\draw[line width=1pt] (mW2.north) to (n2.210);
			\draw[line width=1pt] (n2.north) to (n1.330);
			\draw[line width=1pt] (n1.north) to (0.5,4.5);
			\draw[white, line width=3pt, double=black, double distance=0.5pt](0,0) to (mW2.210);
		\end{tikzpicture}
	\end{matrix}
	=
	\begin{matrix}
		\begin{tikzpicture}[scale=0.85, out=up, in=down, line width=0.5pt]
			\node at (0,-0.3) {$\scriptstyle{A}$};
			\node at (0.9,-0.3) {$\scriptstyle{W_1}$};
			\node at (1.7,-0.3) {$\scriptstyle{W_2}$};
			\node at (2.5,-0.3) {$\scriptstyle{W_3}$};
			\node at (0.5, 4.8) {$\scriptstyle{W_4}$};
			\node (mW3) at (2,1.3) [draw,minimum width=30pt,minimum height=8pt,fill=white] {$\scriptstyle{\eta_2}$};
			\node (n1) at (0.5,3.5) [draw,minimum width=30pt,minimum height=8pt,fill=white] {$\scriptstyle{\eta_1}$};
			\node (n2) at (1.5,2.35) [draw,minimum width=30pt,minimum height=8pt,fill=white] {$\scriptstyle{\mu_{M_1}}$};
			\draw[line width=1pt] (1,0) to (0,1) to (0,2) to (n1.210);
			\draw[line width=1pt] (1.5,0) to (mW3.210); 
			\draw[line width=1pt] (2.5,0) to (mW3.330);
			\draw[line width=1pt] (mW3.north) to (n2.330);
			\draw[line width=1pt] (n2.north) to (n1.330);
			\draw[line width=1pt] (n1.north) to (0.5,4.5);
			\draw[white, line width=3pt, double=black, double distance=0.5pt](0,0) to (1,1) to (n2.210);
		\end{tikzpicture}
	\end{matrix}
	=
	\begin{matrix}
		\begin{tikzpicture}[scale=0.85, out=up, in=down, line width=0.5pt]
			\node at (0,-0.3) {$\scriptstyle{A}$};
			\node at (0.9,-0.3) {$\scriptstyle{W_1}$};
			\node at (1.7,-0.3) {$\scriptstyle{W_2}$};
			\node at (2.5,-0.3) {$\scriptstyle{W_3}$};
			\node at (0.5, 4.8) {$\scriptstyle{W_4}$};
			\node (mW3) at (2,0.7) [draw,minimum width=30pt,minimum height=8pt,fill=white] {$\scriptstyle{\eta_2}$};
			\node (n1) at (0.5,3.5) [draw,minimum width=30pt,minimum height=8pt,fill=white] {$\scriptstyle{\eta_1}$};
			\node (n2) at (1.5,2.3) [draw,minimum width=30pt,minimum height=8pt,fill=white] {$\scriptstyle{\mu_{M_1}}$};
			\draw[line width=1pt] (1,0) to (1,1) to (0,2.3) to (n1.210);
			\draw[line width=1pt] (1.5,0) to (mW3.210);
			\draw[line width=1pt] (2.5,0) to (mW3.330);
			\draw[line width=1pt] (mW3.north) to (n2.330);
			\draw[line width=1pt] (n2.north) to (n1.330);
			\draw[line width=1pt] (n1.north) to (0.5,4.5);
			\draw[white, line width=3pt, double=black, double distance=0.5pt](0,0) to (0,1) to (n2.210);
		\end{tikzpicture}
	\end{matrix}
	=
	\begin{matrix}
		\begin{tikzpicture}[scale=0.85, out=up, in=down, line width=0.5pt]
			\node at (0,-0.3) {$\scriptstyle{A}$};
			\node at (0.9,-0.3) {$\scriptstyle{W_1}$};
			\node at (1.7,-0.3) {$\scriptstyle{W_2}$};
			\node at (2.5,-0.3) {$\scriptstyle{W_3}$};
			\node at (1.5, 4.8) {$\scriptstyle{W_4}$};
			\node (mW3) at (2,1) [draw,minimum width=30pt,minimum height=8pt,fill=white] {$\scriptstyle{\eta_2}$};
			\node (n1) at (1.5,3.5) [draw,minimum width=30pt,minimum height=8pt,fill=white] {$\scriptstyle{\eta_1}$};
			\node (n2) at (0.5,2.3) [draw,minimum width=30pt,minimum height=8pt,fill=white] {$\scriptstyle{\mu_{W_1}}$};
			\draw[](0,0) to (n2.210);
			\draw[line width=1pt] (1,0) to (n2.330);
			\draw[line width=1pt] (1.5,0) to (mW3.210);
			\draw[line width=1pt] (2.5,0) to (mW3.330);
			\draw[line width=1pt] (mW3.north) to (n1.330);
			\draw[line width=1pt] (n2.north) to (n1.210);
			\draw[line width=1pt] (n1.north) to (1.5,4.5);
		\end{tikzpicture}
	\end{matrix}			
	=
	\begin{matrix}
		\begin{tikzpicture}[scale=0.85, out=up, in=down, line width=0.5pt]
			\node at (0,-0.3) {$\scriptstyle{A}$};
			\node at (0.9,-0.3) {$\scriptstyle{W_1}$};
			\node at (1.7,-0.3) {$\scriptstyle{W_2}$};
			\node at (2.5,-0.3) {$\scriptstyle{W_3}$};
			\node at (1.5, 4.8) {$\scriptstyle{W_4}$};
			\node (mW3) at (2,2.3) [draw,minimum width=30pt,minimum height=8pt,fill=white] {$\scriptstyle{\eta_2}$};
			\node (n1) at (1.5,3.5) [draw,minimum width=30pt,minimum height=8pt,fill=white] {$\scriptstyle{\eta_1}$};
			\node (n2) at (0.5,1) [draw,minimum width=30pt,minimum height=8pt,fill=white] {$\scriptstyle{\mu_{W_1}}$};
			\draw[](0,0) to (n2.210);
			\draw[line width=1pt] (1,0) to (n2.330);
			\draw[line width=1pt] (1.5,0) to (mW3.210);
			\draw[line width=1pt] (2.5,0) to (mW3.330);
			\draw[line width=1pt] (mW3.north) to (n1.330);
			\draw[line width=1pt] (n2.north) to (n1.210);
			\draw[line width=1pt] (n1.north) to (1.5,4.5);
		\end{tikzpicture}
	\end{matrix}			
	\end{align*}
The first step uses $\eta_2\circ\mu^{(1)}=\mu_{M_1}\circ(1_A\boxtimes\eta_2)$, while the third uses $\eta_1\circ\mu^{(2)}= \eta_1\circ\mu^{(1)}$.
The second and last equalities use evenness of the braiding and $\mu_{W_1}$, respectively, to exchange their orders with $\eta_2$.

	Now we need to show that $\eta$ is a $\repA$-intertwining operator, that is, we need to show
\begin{equation}\label{eqn:eta_repAintwop}	\mu_{W_4}\circ(1_A\boxtimes\eta)=\eta\circ\mu^{(1)}=\eta\circ\mu^{(2)}.
\end{equation}
 Since $\eta_{W_1,W_2}$ is surjective and $\cdot \boxtimes W_3$ and
$A\boxtimes \cdot$ are right exact, $1_A\boxtimes(\eta_{W_1,W_2}\boxtimes 1_{W_3})$ is a surjective morphism.
Thus, to prove \eqref{eqn:eta_repAintwop}, it is enough to show
\begin{align}	\mu_{W_4}\circ(1_A\boxtimes\eta)&\circ (1_A\boxtimes(\eta_{W_1,W_2}\boxtimes 1_{W_3}))
	=\eta\circ\mu^{(1)}\circ (1_A\boxtimes(\eta_{W_1,W_2}\boxtimes 1_{W_3}))\nonumber\\
	&	=\eta\circ\mu^{(2)}\circ (1_A\boxtimes(\eta_{W_1,W_2}\boxtimes 1_{W_3})).\label{eqn:eta_repAintwop_alt}
\end{align}
To prove the first equality in \eqref{eqn:eta_repAintwop_alt}, we proceed as follows, starting with $\eta\circ\mu^{(1)}\circ (1_A\boxtimes(\eta_{W_1,W_2}\boxtimes 1_{W_3}))$:
	\begin{align*}
		\begin{matrix}
			\begin{tikzpicture}
				[line width=0.5pt, out=up, in=down]
				\node at (0,-0.3) {$\scriptstyle{A}$};
				\node at (0.5,-0.3) {$\scriptstyle{W_1}$};
				\node at (1.5,-0.3) {$\scriptstyle{W_2}$};
				\node at (2,-0.3) {$\scriptstyle{W_3}$};
				\node at (1.2,3.8) {$\scriptstyle{W_4}$};
				\node (etab) at (1,1)  [draw,minimum width=21pt,minimum height=10pt] {$\scriptstyle{\eta_{W_1,W_2}}$};
				\node (mu) at (0.5,2) [draw,minimum width=21pt,minimum height=10pt] {$\scriptstyle{\mu_{W_1\boxtimes_AW_2}}$};
				\node (eta) at (1.2,3) [draw,minimum width=21pt,minimum height=10pt] {$\scriptstyle{\eta}$};
				\draw[white, double=black, line width = 1pt] (0,0) to (mu.220);
				\draw[white, double=black, line width = 1pt, double distance=1.25 pt] (0.5,0) to (etab.220);
				\draw[white, double=black, line width = 1pt, double distance=1.25 pt] (1.5,0) to (etab.320);
				\draw[white, double=black, line width = 1pt, double distance=1.25 pt] (2,0) to (eta.320);
				\draw[white, double=black, line width = 1pt, double distance=1.25 pt] (etab.90) to (mu.320);
				\draw[white, double=black, line width = 1pt, double distance=1.25 pt] (mu.90) to (eta.220);
				\draw[white, double=black, line width = 1pt, double distance=1.25 pt]  (eta.90) to (1.2,3.5) ;
			\end{tikzpicture}						
		\end{matrix}
		=
		\begin{matrix}
			\begin{tikzpicture}
				[line width=0.5pt,out=up,in=down]
				\node at (0,-0.3) {$\scriptstyle{A}$};
				\node at (1,-0.3) {$\scriptstyle{W_1}$};
				\node at (1.5,-0.3) {$\scriptstyle{W_2}$};
				\node at (2,-0.3) {$\scriptstyle{W_3}$};
				\node at (1.5,3.8) {$\scriptstyle{W_4}$};
				\node (mu) at (0.5,1)  [draw,minimum width=21pt,minimum height=10pt] {$\scriptstyle{\mu_{W_1}}$};
				\node (etab) at (1,2)  [draw,minimum width=21pt,minimum height=10pt] {$\scriptstyle{\eta_{W_1,W_2}}$};
				\node (eta) at (1.5,3)  [draw,minimum width=21pt,minimum height=10pt] {$\scriptstyle{\eta}$};
				\draw[white, double=black, line width = 1pt] (0,0) to (mu.220);
				\draw[white, double=black, line width = 1pt, double distance=1.25 pt] (1,0) to (mu.320);
				\draw[white, double=black, line width = 1pt, double distance=1.25 pt] (1.5,0) to (etab.320);
				\draw[white, double=black, line width = 1pt, double distance=1.25 pt] (2,0) to (eta.320);
				\draw[white, double=black, line width = 1pt, double distance=1.25 pt] (mu.90) to (etab.220);
				\draw[white, double=black, line width = 1pt, double distance=1.25 pt] (etab.90) to (eta.220);
				\draw[white, double=black, line width = 1pt, double distance=1.25 pt] (eta.90) to (1.5,3.5);
			\end{tikzpicture}						
		\end{matrix}
		=
		\begin{matrix}
			\begin{tikzpicture}
				[line width=0.5pt,out=up,in=down]
				\node at (0,-0.3) {$\scriptstyle{A}$};
				\node at (1,-0.3) {$\scriptstyle{W_1}$};
				\node at (1.5,-0.3) {$\scriptstyle{W_2}$};
				\node at (2.3,-0.3) {$\scriptstyle{W_3}$};
				\node at (1,3.8) {$\scriptstyle{W_4}$};
				\node (mu) at (0.5,1)  [draw,minimum width=21pt,minimum height=10pt] {$\scriptstyle{\mu_{W_1}}$};
				\node (etab) at (1.7,2)  [draw,minimum width=21pt,minimum height=10pt] {$\scriptstyle{\eta_2}$};
				\node (eta) at (1,3)  [draw,minimum width=21pt,minimum height=10pt] {$\scriptstyle{\eta_1}$};
				\draw[white, double=black, line width = 1pt] (0,0) to (mu.220);
				\draw[white, double=black, line width = 1pt, double distance=1.25 pt] (1,0) to (mu.320);
				\draw[white, double=black, line width = 1pt, double distance=1.25 pt] (1.5,0) to (etab.220);
				\draw[white, double=black, line width = 1pt, double distance=1.25 pt] (2.3,0) to (etab.320);
				\draw[white, double=black, line width = 1pt, double distance=1.25 pt] (mu.90) to (eta.220);
				\draw[white, double=black, line width = 1pt, double distance=1.25 pt] (etab.90) to (eta.320);
				\draw[white, double=black, line width = 1pt, double distance=1.25 pt] (eta.90) to (1,3.5);
			\end{tikzpicture}						
		\end{matrix}
		=
		\begin{matrix}
			\begin{tikzpicture}
				[line width=0.5pt, out=up,in=down]
				\node at (0,-0.3) {$\scriptstyle{A}$};
				\node at (0.5,-0.3) {$\scriptstyle{W_1}$};
				\node at (1,-0.3) {$\scriptstyle{W_2}$};
				\node at (2,-0.3) {$\scriptstyle{W_3}$};
				\node at (0.5,3.8) {$\scriptstyle{W_4}$};
				\node (mu) at (0.5,3)  [draw,minimum width=21pt,minimum height=10pt] {$\scriptstyle{\mu_{W_4}}$};
				\node (etab) at (1.5,1)  [draw,minimum width=21pt,minimum height=10pt] {$\scriptstyle{\eta_2}$};
				\node (eta) at (1,2)  [draw,minimum width=21pt,minimum height=10pt] {$\scriptstyle{\eta_1}$};
				\draw[white, double=black, line width = 1pt] (0,0) to (mu.220);
				\draw[white, double=black, line width = 1pt, double distance=1.25 pt] (0.5,0) to (eta.220);
				\draw[white, double=black, line width = 1pt, double distance=1.25 pt] (1,0) to (etab.220);
				\draw[white, double=black, line width = 1pt, double distance=1.25 pt] (2,0) to (etab.320);
				\draw[white, double=black, line width = 1pt, double distance=1.25 pt] (eta.90) to (mu.320);
				\draw[white, double=black, line width = 1pt, double distance=1.25 pt] (etab.90) to (eta.320);
				\draw[white, double=black, line width = 1pt, double distance=1.25 pt] (mu.90) to (0.5,3.5);
			\end{tikzpicture}						
		\end{matrix}
		=
		\begin{matrix}
			\begin{tikzpicture}
				[line width=0.5pt, out=up,in=down]
				\node at (0,-0.3) {$\scriptstyle{A}$};
				\node at (0.5,-0.3) {$\scriptstyle{W_1}$};
				\node at (1.5,-0.3) {$\scriptstyle{W_2}$};
				\node at (2,-0.3) {$\scriptstyle{W_3}$};
				\node at (0.5,3.8) {$\scriptstyle{W_4}$};
				\node (mu) at (0.5,3)  [draw,minimum width=21pt,minimum height=10pt] {$\scriptstyle{\mu_{W_4}}$};
				\node (etab) at (1,1)  [draw,minimum width=21pt,minimum height=10pt] {$\scriptstyle{\eta_{W_1,W_2}}$};
				\node (eta) at (1.5,2)  [draw,minimum width=21pt,minimum height=10pt] {$\scriptstyle{\eta}$};
				\draw[white, double=black, line width = 1pt] (0,0) to (mu.220);
				\draw[white, double=black, line width = 1pt, double distance=1.25 pt] (0.5,0) to (etab.220);
				\draw[white, double=black, line width = 1pt, double distance=1.25 pt] (1.5,0) to (etab.320);
				\draw[white, double=black, line width = 1pt, double distance=1.25 pt] (2,0) to (eta.320);
				\draw[white, double=black, line width = 1pt, double distance=1.25 pt] (eta.90) to (mu.320);
				\draw[white, double=black, line width = 1pt, double distance=1.25 pt] (etab.90) to (eta.220);
				\draw[white, double=black, line width = 1pt, double distance=1.25 pt] (mu.90) to (0.5,3.5);
			\end{tikzpicture}						
		\end{matrix}
\end{align*}
The first step uses the definition of $\mu_{W_1\boxtimes_A W_2}$ from \eqref{repAtensdef}, while the second and last steps use \eqref{eqn:repAintwopassoc_1_def}. For the third step, we first use evenness of $\mu_{W_1}$ to exchange the order of $\mu_{W_1}$ and $\eta_1$, and then we use the fact that $\eta_2$ is an intertwining operator.

For the other equality in \eqref{eqn:eta_repAintwop_alt}, we start with $\eta\circ\mu^{(2)}\circ (1_A\boxtimes(\eta_{W_1,W_2}\boxtimes 1_{W_3}))$:
	\begin{align*}
		\begin{matrix}
			\begin{tikzpicture}
				[scale = 0.87, baseline = {(current bounding box.center)}, line width=0.5pt, out=up, in=down]
				\node at (0,-0.3) {$\scriptstyle{A}$};
				\node at (0.7,-0.3) {$\scriptstyle{W_1}$};
				\node at (1.3,-0.3) {$\scriptstyle{W_2}$};
				\node at (2,-0.3) {$\scriptstyle{W_3}$};
				\node at (0.7,3.8) {$\scriptstyle{W_4}$};
				\node (etab) at (1,0.5)  [draw,minimum width=21pt,minimum height=10pt] {$\scriptstyle{\eta_{W_1,W_2}}$};
				\node (mu) at (1.3,2) [draw,minimum width=20pt,minimum height=10pt] {$\scriptstyle{\mu_{W_3}}$};
				\node (eta) at (0.7,3) [draw,minimum width=21pt,minimum height=10pt] {$\scriptstyle{\eta}$};
				\draw[white, double=black, line width = 1pt, double distance=1.25 pt] (etab.90) to [out=up, in=down] (0,1.5) to [out=up, in=down] (eta.220) ;
				\draw[white, double=black, line width = 3pt] (0,0) to [out=up, in=down] (0,0.8) to [out=up, in=down] (mu.220);
				\draw[white, double=black, line width = 1pt, double distance=1.25 pt] (0.7,0) to (etab.220);
				\draw[white, double=black, line width = 1pt, double distance=1.25 pt] (1.3,0) to (etab.320);
				\draw[white, double=black, line width = 1pt, double distance=1.25 pt] (2,0) to [out=up,in=down] (mu.320);
				\draw[white, double=black, line width = 1pt, double distance=1.25 pt] (mu.90) to [out=up,in=down] (eta.320);
				\draw[white, double=black, line width = 1pt, double distance=1.25 pt] (eta.90) to (0.7,3.5);
			\end{tikzpicture}						
		\end{matrix}
		=
		\begin{matrix}
			\begin{tikzpicture}
				[scale = 0.87, baseline = {(current bounding box.center)}, line width=0.5pt, out=up, in=down]
				\node at (0,-0.3) {$\scriptstyle{A}$};
				\node at (0.5,-0.3) {$\scriptstyle{W_1}$};
				\node at (1,-0.3) {$\scriptstyle{W_2}$};
				\node at (2,-0.3) {$\scriptstyle{W_3}$};
				\node at (1.3,3.8) {$\scriptstyle{W_4}$};
				\node (etab) at (0.7,2)  [draw,minimum width=21pt,minimum height=10pt] {$\scriptstyle{\eta_{W_1,W_2}}$};
				\node (mu) at (1.6,1) [draw,minimum width=20pt,minimum height=10pt] {$\scriptstyle{\mu_{W_3}}$};
				\node (eta) at (1.3,3) [draw,minimum width=21pt,minimum height=10pt] {$\scriptstyle{\eta}$};
				\draw[white, double=black, line width = 1pt, double distance=1.25 pt] (etab.90) to  (eta.220) ;
				\draw[white, double=black, line width = 1pt, double distance=1.25 pt] (0.5,0) to [out=up,in=down] (0.25,1) to [out=up,in=down] (etab.220);
				\draw[white, double=black, line width = 1pt, double distance=1.25 pt] (1,0)  to [out=up,in=down] (0.55,1) to [out=up,in=down] (etab.320);
				\draw[white, double=black, line width = 1pt, double distance=1.25 pt] (2,0) to [out=up,in=down] (mu.320);
				\draw[white, double=black, line width = 3pt] (0,0) to [out=up,in=down] (mu.220);
				\draw[white, double=black, line width = 1pt, double distance=1.25 pt] (mu.90) to [out=up,in=down] (eta.320);
				\draw[white, double=black, line width = 1pt, double distance=1.25 pt] (eta.90) to (1.3,3.5);
			\end{tikzpicture}						
		\end{matrix}
		=
		\begin{matrix}
			\begin{tikzpicture}
				[scale = 0.87, baseline = {(current bounding box.center)}, line width=0.5pt, out=up, in=down]
				\node at (0,-0.3) {$\scriptstyle{A}$};
				\node at (0.5,-0.3) {$\scriptstyle{W_1}$};
				\node at (1,-0.3) {$\scriptstyle{W_2}$};
				\node at (1.8,-0.3) {$\scriptstyle{W_3}$};
				\node at (1,3.8) {$\scriptstyle{W_4}$};
				\node (etab) at (1.4,2.2)  [draw,minimum width=21pt,minimum height=10pt] {$\scriptstyle{\eta_2}$};
				\node (mu) at (1.6,1.2) [draw,minimum width=20pt,minimum height=10pt] {$\scriptstyle{\mu_{W_3}}$};
				\node (eta) at (1,3) [draw,minimum width=21pt,minimum height=10pt] {$\scriptstyle{\eta_1}$};
				\draw[white, double=black, line width = 1pt, double distance=1.25 pt] (etab.90) to  (eta.320) ;
				\draw[white, double=black, line width = 1pt, double distance=1.25 pt] (0.5,0) to [out=up,in=down] (0.25,1) to [out=up,in=down] (eta.220);
				\draw[white, double=black, line width = 1pt, double distance=1.25 pt] (1,0)  to [out=up,in=down] (0.55,1) to [out=up,in=down] (etab.220);
				\draw[white, double=black, line width = 1pt, double distance=1.25 pt] (1.8,0) to [out=up,in=down] (mu.320);
				\draw[white, double=black, line width = 3pt] (0,0) to [out=up,in=down] (mu.220);
				\draw[white, double=black, line width = 1pt, double distance=1.25 pt] (mu.90) to [out=up,in=down] (etab.320);
				\draw[white, double=black, line width = 1pt, double distance=1.25 pt] (eta.90) to (1,3.5);
			\end{tikzpicture}						
		\end{matrix} 
		=
		\begin{matrix}
			\begin{tikzpicture}
				[scale = 0.87, baseline = {(current bounding box.center)}, line width=0.5pt,out=up, in=down]
				\node at (0,-0.3) {$\scriptstyle{A}$};
				\node at (0.5,-0.3) {$\scriptstyle{W_1}$};
				\node at (1.1,-0.3) {$\scriptstyle{W_2}$};
				\node at (1.7,-0.3) {$\scriptstyle{W_3}$};
				\node at (0.5,3.8) {$\scriptstyle{W_4}$};
				\node (etab) at (1.4,1.2)  [draw,minimum width=21pt,minimum height=10pt] {$\scriptstyle{\eta_2}$};
				\node (mu) at (1,2) [draw,minimum width=20pt,minimum height=10pt] {$\scriptstyle{\mu_{M_1}}$};
				\node (eta) at (0.5,3) [draw,minimum width=21pt,minimum height=10pt] {$\scriptstyle{\eta_1}$};
				\draw[white, double=black, line width = 1pt, double distance=1.25 pt] (etab.90) to [out=up,in=down] (mu.320) ;
				\draw[white, double=black, line width = 1pt, double distance=1.25 pt] (0.5,0) to (0,1) to [out=up,in=down] (eta.220);
				\draw[white, double=black, line width = 1pt, double distance=1.25 pt] (1.1,0) to  (etab.220);
				\draw[white, double=black, line width = 1pt, double distance=1.25 pt] (1.7,0) to [out=up,in=down] (etab.320);
				\draw[white, double=black, line width = 3pt] (0,0) to [out=up,in=down] (0.7,1) to [out=up,in=down] (mu.220);
				\draw[white, double=black, line width = 1pt, double distance=1.25 pt] (mu.90) to [out=up,in=down] (eta.320);
				\draw[white, double=black, line width = 1pt, double distance=1.25 pt] (eta.90) to (0.5,3.5);
			\end{tikzpicture}						
		\end{matrix} 
		=
		\begin{matrix}
			\begin{tikzpicture}
				[scale = 0.87, baseline = {(current bounding box.center)}, line width=0.5pt,out=up, in=down]
				\node at (0,-0.3) {$\scriptstyle{A}$};
				\node at (0.5,-0.3) {$\scriptstyle{W_1}$};
				\node at (1,-0.3) {$\scriptstyle{W_2}$};
				\node at (1.5,-0.3) {$\scriptstyle{W_3}$};
				\node at (0.4,3.8) {$\scriptstyle{W_4}$};
				\node (mu) at (0.4,3) [draw,minimum width=20pt,minimum height=10pt] {$\scriptstyle{\mu_{W_4}}$};
				\node (eta2) at (0.8,2) [draw,minimum width=20pt,minimum height=10pt] {$\scriptstyle{\eta_1}$};
				\node (eta1) at (1.2,1)  [draw,minimum width=20pt,minimum height=10pt] {$\scriptstyle{\eta_2}$};
				\draw[white, double=black, line width = 3pt] (0,0) to [out=up,in=down] (mu.220);
				\draw[white, double=black, line width = 1pt, double distance=1.25 pt] (0.5,0) to [out=up,in=down] (eta2.220);
				\draw[white, double=black, line width = 1pt, double distance=1.25 pt] (1,0) to [out=up,in=down] (eta1.220);
				\draw[white, double=black, line width = 1pt, double distance=1.25 pt] (1.5,0) to [out=up,in=down] (eta1.320);
				\draw[white, double=black, line width = 1pt, double distance=1.25 pt] (eta1.90) to [out=up,in=down] (eta2.320);
				\draw[white, double=black, line width = 1pt, double distance=1.25 pt] (eta2.90) to [out=up,in=down] (mu.320);
				\draw[white, double=black, line width = 1pt, double distance=1.25 pt] (mu.90) to [out=up,in=down] (0.4,3.5);
			\end{tikzpicture}						
		\end{matrix} 
		=
		\begin{matrix}
			\begin{tikzpicture}
				[scale = 0.85, baseline = {(current bounding box.center)}, line width=0.5pt, out=up, in=down]
				\node at (0,-0.3) {$\scriptstyle{A}$};
				\node at (0.5,-0.3) {$\scriptstyle{W_1}$};
				\node at (1.3,-0.3) {$\scriptstyle{W_2}$};
				\node at (2,-0.3) {$\scriptstyle{W_3}$};
				\node at (0.4,3.8) {$\scriptstyle{W_4}$};
				\node (mu) at (0.4,3) [draw,minimum width=20pt,minimum height=10pt] {$\scriptstyle{\mu_{W_4}}$};
				\node (eta) at (1.5,2) [draw,minimum width=20pt,minimum height=10pt] {$\scriptstyle{\eta}$};
				\node (etab) at (0.9,1)  [draw,minimum width=20pt,minimum height=10pt] {$\scriptstyle{\eta_{W_1,W_2}}$};
				\draw[white, double=black, line width = 3pt] (0,0) to [out=up,in=down] (mu.220);
				\draw[white, double=black, line width = 1pt, double distance=1.25 pt] (0.5,0) to [out=up,in=down] (etab.220);
				\draw[white, double=black, line width = 1pt, double distance=1.25 pt] (1.3,0) to [out=up,in=down] (etab.320);
				\draw[white, double=black, line width = 1pt, double distance=1.25 pt] (2,0) to [out=up,in=down] (eta.320);
				\draw[white, double=black, line width = 1pt, double distance=1.25 pt] (etab.90) to [out=up,in=down] (eta.220);
				\draw[white, double=black, line width = 1pt, double distance=1.25 pt] (eta.90) to [out=up,in=down] (mu.320);
				\draw[white, double=black, line width = 1pt, double distance=1.25 pt] (mu.90) to [out=up,in=down] (0.4,3.5);
			\end{tikzpicture}						
		\end{matrix} 
\end{align*}
The first equality uses naturality of braiding applied to the even morphism $\eta_{W_1,W_2}$, and then the evenness
of either $\eta_{W_1,W_2}$ or $\mu_{W_3}$ to exchange their order. The
second and the last equalities are due to \eqref{eqn:repAintwopassoc_1_def}.
The third equality uses the hexagon axiom and then $\eta_2\circ\mu^{(2)}=\mu_{M_1}\circ(1_A\boxtimes \eta_2)$. In the fourth equality, we change the order of the braiding and $\eta_1$ using evenness of the braiding and then use $\eta_1\circ\mu^{(2)}=\mu_{W_4}\circ(1_A\boxtimes\eta_1)$.
\end{proof}

The associativity of $\repA$-intertwining operators yields associativity isomorphisms in $\repA$:
\begin{theo}
	There is a unique natural isomorphism $\cA^A: \boxtimes_A\circ(1_{\repA}\times\boxtimes_A)\rightarrow\boxtimes_A\circ(\boxtimes_A\times 1_{\repA})$ such that for any objects $(W_1,\mu_{W_1})$, $(W_2,\mu_{W_2})$, and $(W_3,\mu_{W_3})$ in $\repA$, the diagram
	\begin{equation}\label{repAassocdef}
	\xymatrixcolsep{6pc}
	\xymatrix{W_1\boxtimes(W_2\boxtimes W_3) \ar[d]^{1_{W_1}\boxtimes\eta_{W_2,W_3}} \ar[r]^{\sA_{W_1,W_2,W_3}} & (W_1\boxtimes W_2)\boxtimes W_3 \ar[d]^{\eta_{W_1,W_2}\boxtimes 1_{W_3}}\\
		W_1\boxtimes (W_2\boxtimes_A W_3) \ar[d]^{\eta_{W_1,W_2\boxtimes_A W_3}} & (W_1\boxtimes_A W_2)\boxtimes W_3 \ar[d]^{\eta_{W_1\boxtimes_A W_2, W_3}}\\
		W_1\boxtimes_A (W_2\boxtimes_A W_3) \ar[r]^{\cA^A_{W_1,W_2,W_3}} & (W_1\boxtimes_A W_2)\boxtimes_A W_3\\}
	\end{equation}
	commutes. Moreover, $\cA^A_{W_1, W_2,W_3}$ is even.
\end{theo}
\begin{proof}
	Suppose $(W_i,\mu_{W_i})$ for $i=1,2,3$ are objects of $\repA$. By the second assertion of Theorem \ref{repAintwopassoc}, there is a unique $\repA$-intertwining operator $\eta$ of type $\binom{(W_1\boxtimes_A W_2)\boxtimes_A W_3}{W_1\,W_2\boxtimes_A W_3}$ such that
	\begin{equation*}
	\eta\circ(1_{W_1}\boxtimes\eta_{W_2,W_3})=\eta_{W_1\boxtimes_A W_2, W_3}\circ(\eta_{W_1,W_2}\boxtimes 1_{W_3})\circ\sA_{W_1,W_2,W_3}.
	\end{equation*}
	Then the universal property of the tensor product $W_1\boxtimes_A (W_2\boxtimes_A W_3)$ induces a unique $\repA$-morphism
	\begin{equation*}
	\cA^A_{W_1,W_2,W_3}: W_1\boxtimes_A (W_2\boxtimes_A W_3)\rightarrow (W_1\boxtimes_A W_2)\boxtimes_A W_3
	\end{equation*}
	such that $\eta=\cA^A_{W_1,W_2,W_3}\circ\eta_{W_1, W_2\boxtimes_A W_3}$. Thus $\cA^A_{W_1,W_2,W_3}$ is the unique $\repA$-morphism such that \eqref{repAassocdef} commutes. Similarly, by the first assertion of Theorem \ref{repAintwopassoc}, there is a unique $\repA$-morphism
	\begin{equation*}
	\widetilde{\cA}^A_{W_1,W_2,W_3}: (W_1\boxtimes_A W_2)\boxtimes_A W_3\rightarrow W_1\boxtimes_A (W_2\boxtimes_A W_3)
	\end{equation*}
	such that the diagram
	\begin{equation*}
	\xymatrixcolsep{6pc}
	\xymatrix{(W_1\boxtimes W_2)\boxtimes W_3 \ar[d]^{\eta_{W_1,W_2}\boxtimes 1_{W_3}} \ar[r]^{\sA_{W_1,W_2,W_3}^{-1}} & W_1\boxtimes (W_2\boxtimes W_3) \ar[d]^{1_{W_1}\boxtimes\eta_{W_2,W_3}}\\
		(W_1\boxtimes_A W_2)\boxtimes W_3 \ar[d]^{\eta_{W_1\boxtimes_A W_2,W_3}} & W_1\boxtimes (W_2\boxtimes_{A} W_3) \ar[d]^{\eta_{W_1, W_2\boxtimes_A W_3}}\\
		(W_1\boxtimes_A W_2)\boxtimes_A W_3 \ar[r]^{\widetilde{\cA}^A_{W_1,W_2,W_3}} & W_1\boxtimes_A (W_2\boxtimes_A W_3)\\}
	\end{equation*}
	commutes.
	
	To show that $\widetilde{\cA}^A_{W_1,W_2,W_3}$ is inverse to $\cA^A_{W_1,W_2,W_3}$, so that $\cA^A_{W_1,W_2,W_3}$ is an isomorphism, note that
	\begin{equation*}
	\xymatrixcolsep{10pc}
	\xymatrix{W_1\boxtimes(W_2\boxtimes W_3) \ar[d]^{\eta_{W_1,W_2\boxtimes_A W_3}\circ(1_{W_1}\boxtimes\eta_{W_2,W_3})} \ar[r]^{1_{W_1\boxtimes(W_2\boxtimes W_3)}} & W_1\boxtimes (W_2\boxtimes W_3) \ar[d]^{\eta_{W_1,W_2\boxtimes_A W_3}\circ(1_{W_1}\boxtimes\eta_{W_2,W_3})}\\
		W_1\boxtimes_A (W_2\boxtimes_A W_3) \ar[r]_{\widetilde{\cA}^A_{W_1,W_2,W_3}\circ\cA^A_{W_1,W_2,W_3}} & W_1\boxtimes_A (W_2\boxtimes_A W_3)\\}
	\end{equation*}
	commutes. Since $\eta_{W_1,W_2\boxtimes_A W_3}$ and $1_{W_1}\boxtimes\eta_{W_2,W_3}$ are both cokernels (by definition and since $W_1\boxtimes\cdot$ is right exact), they are both epimorphisms, and we conclude that $\widetilde{\cA}^A_{W_1,W_2,W_3}\circ\cA^A_{W_1,W_2,W_3}=1_{W_1\boxtimes_A (W_2\boxtimes_A W_3)}$. Similarly, $\cA^A_{W_1,W_2,W_3}\circ\widetilde{\cA}^A_{W_1,W_2,W_3}=1_{(W_1\boxtimes_A W_2)\boxtimes_A W_3}$.
	
	The evenness of $\cA^A_{W_1,W_2,W_3}$ follows from the evenness of $\sA_{W_1,W_2,W_3}$, $\eta_{W_1,W_2}\boxtimes 1_{W_3}$, and $\eta_{W_1\boxtimes_A W_2, W_3}$, and the surjectivity and evenness of $1_{W_1}\boxtimes\eta_{W_2,W_3}$ and $\eta_{W_1,W_2\boxtimes_A W_3}$.

	Now we have to show that the $\repA$-isomorphisms $\cA^A_{W_1,W_2,W_3}$ define a natural transformation, that is, we need to show that if $f_i: W_i\rightarrow\widetilde{W}_i$ for $i=1,2,3$ are $\repA$-morphisms, then the diagram
	\begin{equation*}
	\xymatrixcolsep{6pc}
	\xymatrix{
		W_1\boxtimes_A (W_2\boxtimes_A W_3) \ar[r]^{\cA^A_{W_1,W_2,W_3}} \ar[d]^{f_1\boxtimes_A (f_2\boxtimes_A f_3)} & (W_1\boxtimes_A W_2)\boxtimes_A W_3 \ar[d]^{(f_1\boxtimes_A f_2)\boxtimes_A f_3} \\
		\widetilde{W}_1\boxtimes_A(\widetilde{W}_2\boxtimes_A\widetilde{W}_3) \ar[r]^{\cA^A_{\widetilde{W}_1,\widetilde{W}_2,\widetilde{W}_3}} &  (\widetilde{W}_1\boxtimes_A\widetilde{W}_2)\boxtimes_A\widetilde{W}_3\\}
	\end{equation*}
	commutes. First, note that the diagrams
	\begin{equation*}
	\xymatrixcolsep{6pc}
	\xymatrix{
		W_1\boxtimes(W_2\boxtimes W_3) \ar[r]^{\sA_{W_1,W_2,W_3}} \ar[d]^{1_{W_1}\boxtimes \eta_{W_2,W_3}} & (W_1\boxtimes W_2)\boxtimes W_3 \ar[r]^{(f_1\boxtimes f_2)\boxtimes f_3} \ar[d]^{\eta_{W_1,W_2}\boxtimes 1_{W_3}} & (\widetilde{W}_1\boxtimes\widetilde{W}_2)\boxtimes\widetilde{W}_3) \ar[d]^{\eta_{\widetilde{W}_1,\widetilde{W}_2}\boxtimes 1_{\widetilde{W}_3}} \\
		W_1\boxtimes(W_2\boxtimes_A W_3) \ar[d]^{\eta_{W_1,W_2\boxtimes_A W_3}} & (W_1\boxtimes_A W_2)\boxtimes W_3 \ar[d]^{\eta_{W_1\boxtimes_A W_2,W_3}} \ar[r]^{(f_1\boxtimes_A f_2)\boxtimes f_3} & (\widetilde{W}_1\boxtimes_A\widetilde{W}_2)\boxtimes\widetilde{W}_3 \ar[d]^{\eta_{\widetilde{W}_1\boxtimes_A\widetilde{W}_2,\widetilde{W}_3}} \\
		W_1\boxtimes_A (W_2\boxtimes_A W_3) \ar[r]^{\cA^{A}_{W_1,W_2,W_3}} & (W_1\boxtimes_A W_2)\boxtimes_A W_3 \ar[r]^{(f_1\boxtimes_A f_2)\boxtimes_A f_3} & (\widetilde{W}_1\boxtimes_A\widetilde{W}_2)\boxtimes_A\widetilde{W}_3 \\} 
	\end{equation*}
	and
	\begin{equation*}
	\xymatrixcolsep{6pc}
	\xymatrix{
		W_1\boxtimes(W_2\boxtimes W_3) \ar[r]^{f_1\boxtimes (f_2\boxtimes f_3)} \ar[d]^{1_{W_1}\boxtimes\eta_{W_2,W_3}} & \widetilde{W}_1\boxtimes(\widetilde{W}_2\boxtimes\widetilde{W}_3) \ar[r]^{\sA_{\widetilde{W}_1,\widetilde{W}_2,\widetilde{W}_3}} \ar[d]^{1_{\widetilde{W}_1}\boxtimes\eta_{\widetilde{W}_2,\widetilde{W}_3}} & (\widetilde{W}_1\boxtimes\widetilde{W}_2)\boxtimes\widetilde{W}_3 \ar[d]^{\eta_{\widetilde{W}_1,\widetilde{W}_2}\boxtimes 1_{\widetilde{W}_3}} \\
		W_1\boxtimes(W_2\boxtimes_A W_3) \ar[r]^{f_1\boxtimes(f_2\boxtimes_A f_3)} \ar[d]^{\eta_{W_1,W_2\boxtimes_A W_3}} & \widetilde{W}_1\boxtimes(\widetilde{W}_2\boxtimes_A \widetilde{W}_3) \ar[d]^{\eta_{\widetilde{W}_1,\widetilde{W}_2\boxtimes_A\widetilde{W}_3}} & (\widetilde{W}_1\boxtimes_A\widetilde{W}_2)\boxtimes\widetilde{W}_3 \ar[d]^{\eta_{\widetilde{W}_1\boxtimes_A\widetilde{W}_2,\widetilde{W}_3}} \\ 
		W_1\boxtimes_A (W_2\boxtimes_A W_3) \ar[r]^{f_1\boxtimes_A(f_2\boxtimes_A f_3)} & \widetilde{W}_1\boxtimes_A (\widetilde{W}_2\boxtimes_A \widetilde{W}_3) \ar[r]^{\cA^A_{\widetilde{W}_1,\widetilde{W}_2,\widetilde{W}_3}} & (\widetilde{W}_1\boxtimes_A \widetilde{W}_2)\boxtimes_A \widetilde{W}_3 \\   }
	\end{equation*}
	commute by the definition of $\cA^A$, the definition of tensor products of morphisms in $\repA$, and the evenness of $\eta_{W_1,W_2}$ and $\eta_{\widetilde{W}_2,\widetilde{W}_3}$. Since the associativity isomorphisms in $\sC$ define a natural transformation, we therefore have
	\begin{align*}
	((f_1\boxtimes_A f_2)\boxtimes_A f_3) & \circ  \cA^A_{W_1,W_2,W_3}\circ\eta_{W_1,W_2\boxtimes_A W_3}\circ(1_{W_1}\boxtimes\eta_{W_2,W_3})\nonumber\\
	&=\cA^A_{\widetilde{W}_1,\widetilde{W}_2,\widetilde{W}_3}\circ(f_1\boxtimes_A(f_2\boxtimes_A f_3))\circ\eta_{W_1,W_2\boxtimes_A W_3}\circ(1_{W_1}\boxtimes\eta_{W_2,W_3}).
	\end{align*}
	Since $\eta_{W_1,W_2\boxtimes_A W_3}$ and $1_{W_1}\boxtimes\eta_{W_2,W_3}$ are both cokernels (by definition and since $W_1\boxtimes\cdot$ is right exact), they are both epimorphisms, and so we conclude
	\begin{equation*}
	((f_1\boxtimes_A f_2)\boxtimes_A f_3)\circ \cA^A_{W_1,W_2,W_3}=\cA^A_{\widetilde{W}_1,\widetilde{W}_2,\widetilde{W}_3}\circ(f_1\boxtimes_A(f_2\boxtimes_A f_3)),
	\end{equation*}
	as desired.
\end{proof}

We have now shown that $\repA$ is an $\mathbb{F}$-linear additive supercategory, and we have constructed bilinear tensor products, natural even unit isomorphisms, and now even natural associativity isomorphisms in $\repA$. To show that $\repA$ is a monoidal supercategory, we still need to check the pentagon and triangle axioms.
We now prove these properties and conclude the following theorem:

\begin{theo}\label{thm:repAmonoidal}
	The category $\repA$ with tensor product bifunctor $\boxtimes_A$, unit object $(A,\mu)$, left unit isomorphisms $l^A$, right unit isomorphisms $r^A$, and associativity isomorphisms $\cA^A$, is an $\mathbb{F}$-linear additive monoidal supercategory.
\end{theo}
\begin{proof}

	To prove the pentagon axiom, we need to show that for any objects $W_1$, $W_2$, $W_3$, and $W_4$ of $\repA$, the compositions
	\begin{align}\label{pent1}
	W_1\boxtimes_A & (W_2\boxtimes_A(W_3\boxtimes_A W_4))\xrightarrow{1_{W_1}\boxtimes_A\cA^A_{W_2,W_3,W_4}} W_1\boxtimes_A((W_2\boxtimes_A W_3)\boxtimes_A W_4)\nonumber\\
	&\xrightarrow{\cA^A_{W_1,W_2\boxtimes_A W_3,W_4}} (W_1\boxtimes_A (W_2\boxtimes_A W_3))\boxtimes_A W_4\xrightarrow{\cA^A_{W_1,W_2,W_3}\boxtimes_A 1_{W_4}} ((W_1\boxtimes_A W_2)\boxtimes_A W_3)\boxtimes_A W_4
	\end{align}
	and
	\begin{align}\label{pent2}
	W_1\boxtimes_A  (W_2\boxtimes_A(W_3\boxtimes_A W_4) ) & \xrightarrow{\cA^A_{W_1,W_2,W_3\boxtimes_A W_4}} (W_1\boxtimes_A W_2)\boxtimes_A (W_3\boxtimes_A W_4)\nonumber\\
	&\xrightarrow{\cA^A_{W_1\boxtimes_A W_2,W_3,W_4}} ((W_1\boxtimes_A W_2)\boxtimes_A W_3)\boxtimes_A W_4
	\end{align}
	are equal. Using the naturality of the associativity isomorphisms in $\sC$ and the definition of the associativity isomorphisms in $\repA$, we can prove that the diagram
	\begin{equation*}
	\xymatrixcolsep{2pc}
	\xymatrix{
		W_1\boxtimes(W_2\boxtimes(W_3\boxtimes W_4)) \ar[d]^{   } \ar[r] &  ((W_1\boxtimes W_2)\boxtimes W_3)\boxtimes W_4 \ar[d] \\
		W_1\boxtimes_A(W_2\boxtimes_A(W_3\boxtimes_A W_4)) \ar[r] & ((W_1\boxtimes_A W_2)\boxtimes_A W_3)\boxtimes_A W_4 \\} 
	\end{equation*}
	commutes, where the lower horizontal arrow is either \eqref{pent1} or \eqref{pent2}, the top horizontal arrow is the corresponding composition of associativity isomorphisms in $\sC$, and the left and right vertical arrows are 
	\begin{equation}\label{pent3}
	\eta_{W_1,W_2\boxtimes_A(W_2\boxtimes_A W_3)}\circ(1_{W_1}\boxtimes\eta_{W_2,W_3\boxtimes_A W_4})\circ(1_{W_1}\boxtimes(1_{W_2}\boxtimes\eta_{W_3,W_4}))              \end{equation}
	and $$\eta_{(W_1\boxtimes_A W_2)\boxtimes_A W_3, W_4}\circ(\eta_{W_1\boxtimes_A W_2,W_3}\boxtimes1_{W_4})\circ((\eta_{W_1,W_2}\boxtimes 1_{W_3})\boxtimes 1_{W_4}),$$ respectively. Since the pentagon axiom holds in $\sC$, it follows that
	\begin{equation*}
	(\mathcal{A}^A_{W_1,W_2,W_3}\boxtimes_A 1_{W_4})\circ\cA^A_{W_1,W_2\boxtimes_A W_3,W_4}\circ(1_{W_1}\boxtimes_A\cA^A_{W_2,W_3,W_4})\circ\eta=\cA^A_{W_1\boxtimes_A W_2,W_3,W_4}\circ\mathcal{A}^A_{W_1,W_2,W_3\boxtimes_A W_4}\circ\eta,
	\end{equation*}
	where $\eta$ is the composition \eqref{pent3}. Since $\eta$ is a composition of cokernel morphisms (since $W_1\boxtimes\cdot$ and $W_2\boxtimes\cdot$ are right exact), $\eta$ is an epimorphism, and we conclude that \eqref{pent1} and \eqref{pent2} are the same.
	
	To prove the triangle axiom, we need to show that the composition
	\begin{equation*}
	W_1\boxtimes_A W_2\xrightarrow{1_{W_1}\boxtimes_A (l^A_{W_2})^{-1}} W_1\boxtimes_A(A\boxtimes_A W_2)\xrightarrow{\cA^A_{W_1,A,W_2}} (W_1\boxtimes_A A)\boxtimes_A W_2\xrightarrow{r^A_{W_1}\boxtimes_A 1_{W_2}} W_1\boxtimes_A W_2
	\end{equation*}
	equals $1_{W_1\boxtimes_A W_2}$. By the definitions of the unit and associativity isomorphisms in $\repA$, the definition of tensor products of morphisms in $\repA$, and the evenness of all morphisms involved, the diagram
	\begin{equation*}
	 \xymatrixcolsep{4.5pc}
	 \xymatrix{
	 W_1\boxtimes W_2 \ar[rr]^{\eta_{W_1,W_2}} \ar[rd]^(.55){1_{W_1}\boxtimes(l^A_{W_2})^{-1}} \ar[d]_{1_{W_1}\boxtimes((\iota_A\boxtimes 1_{W_2})\circ\sleft_{W_2}^{-1})} & & W_1\boxtimes_A W_2 \ar[d]^{1_{W_1}\boxtimes_A (l^A_{W_2})^{-1}}  \\
	 W_1\boxtimes(A\boxtimes W_2) \ar[r]^{1_{W_1}\boxtimes\eta_{A,W_2}} \ar[d]_{\sA_{W_1,A,W_2}} & W_1\boxtimes(A\boxtimes_A W_2) \ar[r]^{\eta_{W_1,A\boxtimes_A W_2}} & W_1\boxtimes_A(A\boxtimes_A W_2) \ar[d]^{\cA^A_{W_1,A,W_2}} \\
	 (W_1\boxtimes A)\boxtimes W_2 \ar[r]^{\eta_{W_1,A}\boxtimes 1_{W_2}} \ar[d]_{\sR_{A,W_1}^{-1}\boxtimes 1_{W_2}} & (W_1\boxtimes_A A)\boxtimes W_2 \ar[r]^{\eta_{W_1\boxtimes_A A, W_2}} \ar[d]^{r^A_{W_1}\boxtimes 1_{W_2}} & (W_1\boxtimes_A A)\boxtimes_A W_2 \ar[d]^{r^A_{W_1}\boxtimes_A 1_{W_2}} \\
	 (A\boxtimes W_1)\boxtimes W_2 \ar[r]^{\mu_{W_1}\boxtimes 1_{W_2}} & W_1\boxtimes W_2 \ar[r]^{\eta_{W_1, W_2}} & W_1\boxtimes_A W_2 \\
	 }
	\end{equation*}
commutes. Since $\eta_{W_1,W_2}$ is an epimorphism, 
it is enough to show that the top and right composition equals $\eta_{W_1,W_2}$. By commutativity of the diagram, it is then enough to show that the left and bottom composition equals $\eta_{W_1,W_2}$.
	For this, we use:
	\begin{align*}
		\begin{matrix}
			\begin{tikzpicture}[line width=0.5pt, out=up, in=down]
				\node (w1) at (0,0) {$\scriptstyle{W_1}$};
				\node (w2) at (1,0) {$\scriptstyle{W_2}$};
				\node (i) at (0.5,1.25) {$\scriptstyle{\bullet}$};
				\node (mu) at (0.2, 2.75) [draw] {$\scriptstyle{\mu_{W_1}}$};
				\node (eta) at (0.6, 3.75) [draw] {$\scriptstyle{\eta_{W_1,W_2}}$};
				\node (wf) at (0.6, 4.5) {$\scriptstyle{W_1\boxtimes_A W_2}$};
				\draw [line width=1pt] (w1.90) to (0,1.25) to (mu.320);
				\draw [dashed] (1,0.5) to[out=west] (i.270);
				\draw [white, double=black, line width=3pt, double distance=0.5pt] (i.90) to (mu.220);
				\draw [line width=1pt] (w2.90) to (eta.330);
				\draw [line width=1pt] (mu.90) to (eta.210);
				\draw [ line width=1pt] (eta.90) to (wf.270);
			\end{tikzpicture}
		\end{matrix}
		=
		\begin{matrix}
			\begin{tikzpicture}[line width=0.5pt, out=up, in=down]
				\node (w1) at (0,0) {$\scriptstyle{W_1}$};
				\node (w2) at (1.2,0) {$\scriptstyle{W_2}$};
				\node (i) at (0.5,1) {$\scriptstyle{\bullet}$};
				\node (mu) at (1, 3) [draw] {$\scriptstyle{\mu_{W_2}}$};
				\node (eta) at (0.6, 3.75) [draw] {$\scriptstyle{\eta_{W_1,W_2}}$};
				\node (wf) at (0.6, 4.5) {$\scriptstyle{W_1\boxtimes_A W_2}$};
				\draw [line width=1pt] (w1.90) to (0,1.25) to (0.5,2) to (0.2,3) to (eta.210);
				\draw [dashed] (1.2,0.5) to[out=west] (i.270);
				\draw [white, double=black, line width=3pt, double distance=0.5pt] (i.90) to (0,2) to (mu.220);
				\draw [line width=1pt] (w2.90) to (mu.320);
				\draw [line width=1pt] (mu.90) to (eta.330);
				\draw [ line width=1pt] (eta.90) to (wf.270);
			\end{tikzpicture}
		\end{matrix}
		=
		\begin{matrix}
			\begin{tikzpicture}[line width=0.5pt, out=up, in=down]
				\node (w1) at (0,0) {$\scriptstyle{W_1}$};
				\node (w2) at (1.2,0) {$\scriptstyle{W_2}$};
				\node (i) at (0.6,1.5) {$\scriptstyle{\bullet}$};
				\node (mu) at (0.9, 2.5) [draw] {$\scriptstyle{\mu_{W_2}}$};
				\node (eta) at (0.5, 3.75) [draw] {$\scriptstyle{\eta_{W_1,W_2}}$};
				\node (wf) at (0.5, 4.5) {$\scriptstyle{W_1\boxtimes_A W_2}$};
				\draw [line width=1pt] (w1.90) to (eta.210);
				\draw [dashed] (1.2,0.5) to[out=west] (i.270);
				\draw [white, double=black, line width=3pt, double distance=0.5pt] (i.90) to  (mu.220);
				\draw [line width=1pt] (w2.90) to (mu.320);
				\draw [line width=1pt] (mu.90) to (eta.330);
				\draw [ line width=1pt] (eta.90) to (wf.270);
			\end{tikzpicture}
		\end{matrix}
		=
		\begin{matrix}
			\begin{tikzpicture}[line width=0.5pt, out=up, in=down]
				\node (w1) at (0,0) {$\scriptstyle{W_1}$};
				\node (w2) at (1,0) {$\scriptstyle{W_2}$};
				\node (eta) at (0.5, 3) [draw] {$\scriptstyle{\eta_{W_1,W_2}}$};
				\node (wf) at (0.5, 4.5) {$\scriptstyle{W_1\boxtimes_A W_2}$};
				\draw [line width=1pt] (w1.90) to (eta.210);
				\draw [line width=1pt] (w2.90) to (eta.330);
				\draw [line width=1pt] (eta.90) to (wf.270);
			\end{tikzpicture}
		\end{matrix}
	\end{align*}
	The first step uses $\eta_{W_1,W_2}\circ\mu^{(1)}=\eta_{W_1,W_2}\circ\mu^{(2)}$
	and the third is the unit property of $\mu_{W_2}$.
\end{proof}

\subsection{The braided tensor category \texorpdfstring{$\repzA$}{Rep0 A}}\label{subsec:rep0A}

We have shown that $\repA$ is a monoidal category, but it is not braided. However, it was shown in \cite{Pa} that $\repA$ has a braided tensor subcategory $\rep^0 A$ consisting of ``local modules.'' In this subsection, we define $\repzA$ and construct its braiding isomorphisms.

For any objects $W_1$, $W_2$ of $\sC$, we define the \textit{monodromy isomorphism} to be
\begin{equation*}
 \cM_{W_1,W_2}=\sR_{W_2,W_1}\circ\sR_{W_1,W_2}: W_1\boxtimes W_2\rightarrow W_1\boxtimes W_2.
\end{equation*}
These isomorphisms thus define a natural isomorphism $\cM$ from $\boxtimes$ to itself.
\begin{defi}
 The category $\mathrm{Rep}^0\,A$ is the full subcategory of $\repA$ consisting of objects $(W, \mu_W)$ which satisfy
 \begin{equation*}
  \mu_W\circ\cM_{A,W}=\mu_W: A\boxtimes W\rightarrow W.
 \end{equation*}
\end{defi}
\begin{theo}\label{thm:rep0}
The category $\mathrm{Rep}^0\,A$ is an $\mathbb{F}$-linear additive braided monoidal supercategory with tensor product $\boxtimes_A$, unit object $(A,\mu)$, left unit isomorphisms $l^A$, right unit isomorphisms $r^A$, associativity isomorphisms $\mathcal{A}^A$, and (even) braiding isomorphisms $\cR^A$ characterized by the commutativity of
 \begin{equation}\label{rep0Abraidingdef}
  \xymatrixcolsep{4pc}
  \xymatrix{
  W_1\boxtimes W_2 \ar[d]_{\eta_{W_1,W_2}} \ar[r]^{\sR_{W_1,W_2}} & W_2\boxtimes W_1 \ar[d]^{\eta_{W_2,W_1}} \\
  W_1\boxtimes_A W_2 \ar[r]^{\cR^A_{W_1,W_2}} & W_2\boxtimes_A W_1\\
  }
 \end{equation}
for objects $W_1$ and $W_2$ of $\mathrm{Rep}^0\,A$.
\end{theo}
\begin{proof}
To show that $\repzA$ is $\mathbb{F}$-linear additive, note first that the zero object of $\repA$ is in $\repzA$ since $\mu_0\circ\cM_{A,0}$ and $\mu_0$ are both the unique morphism $A\boxtimes 0\rightarrow 0$. Second, since $\repzA$ is a full subcategory of the $\mathbb{F}$-linear supercategory $\repA$, morphism sets have $\mathbb{F}$-superspace structure for which composition (and tensor product) of morphisms is bilinear. Third, if $\lbrace W_i\rbrace$ is a finite set of objects in $\repzA$, then the $\repA$-biproduct $(\bigoplus W_i, \lbrace p_i\rbrace, \lbrace q_i\rbrace)$ is also in $\repzA$ because
\begin{align*}
 \mu_{\bigoplus W_i}\circ\cM_{A,\bigoplus W_i}&= \sum q_i\circ\mu_{W_i}\circ(1_A\boxtimes p_i)\circ\cM_{A,\bigoplus W_i} =\sum q_i\circ\mu_{W_i}\circ\cM_{A,W_i}\circ(1_A\boxtimes p_i)\nonumber\\
 &=\sum q_i\circ\mu_{W_i}\circ(1_A\boxtimes p_i)=\mu_{\bigoplus W_i}.
\end{align*}

To show that $\repzA$ is braided monoidal, note first that $(A,\mu)$ is an object of $\mathrm{Rep}^0\,A$ by the commutativity of $\mu$. Then we need to show that $\boxtimes_A$ is closed on $\mathrm{Rep}^0\,A$ and that \eqref{rep0Abraidingdef} defines an even $\mathrm{Rep}^0\,A$-isomorphism which satisfies the hexagon axiom and is natural in the sense of \eqref{braidingsupernatural}.
 
 First we show that if $(W_1,\mu_{W_1})$ and $(W_2, \mu_{W_2})$ are objects of $\repzA$, then so is $(W_1\boxtimes_A W_2,\mu_{W_1\boxtimes_A W_2})$. Consider the commutative diagram
 \begin{equation*}
  \xymatrixcolsep{5pc}
  \xymatrix{
  A\boxtimes(W_1\boxtimes W_2) \ar[r]^{\cM_{A, W_1\boxtimes W_2}} \ar[d]_{1_A\boxtimes\eta_{W_1,W_2}} & A\boxtimes(W_1\boxtimes W_2) \ar[r]^(.54){\mu^{(1)}\,or\,\mu^{(2)}} \ar[d]^{1_A\boxtimes\eta_{W_1,W_2}} & W_1\boxtimes W_2 \ar[d]^{\eta_{W_1,W_2}} \\
  A\boxtimes(W_1\boxtimes_A W_2) \ar[r]^{\cM_{A, W_1\boxtimes_A W_2}} & A\boxtimes(W_1\boxtimes_A W_2) \ar[r]^(.54){\mu_{W_1\boxtimes_A W_2}} & W_1\boxtimes_A W_2\\
  }
 \end{equation*}
Since $1_A\boxtimes\eta_{W_1,W_2}$ is an epimorphism and since $\mu_{W_1\boxtimes_AW_2}\circ (1_A\boxtimes\eta_{W_1,W_2})=\eta_{W_1,W_2}\circ\mu^{(2)}$, to show that
\begin{equation*}
 \mu_{W_1\boxtimes_A W_2}\circ\cM_{A,W_1\boxtimes_A W_2}=\mu_{W_1\boxtimes_A W_2},
\end{equation*}
it is enough to show that 
\begin{equation*}\label{Rep0Atensprod}
 \eta_{W_1,W_2}\circ\mu^{(1)}\circ\cM_{A,W_1\boxtimes W_2}=\eta_{W_1,W_2}\circ\mu^{(2)}.
\end{equation*}
The proof is guided by the diagrams:
	\begin{align*}
	\begin{matrix}
		\begin{tikzpicture}[scale=0.9, out=up, in=down, line width=0.5pt]
			\node at (0,-0.3) {$\scriptstyle{A}$};
			\node at (1.4,-0.3) {$\scriptstyle{W_1}$};
			\node at (2.1,-0.3) {$\scriptstyle{W_2}$};
			\node at (1.5,5.5) {$\scriptstyle{W_1\boxtimes_A W_2}$};
			\node (m) at (0.5,3.2) [draw,minimum width=25pt,minimum height=10pt,thick, fill=white] {$\scriptstyle{\mu_{W_1}}$};
			\node (e) at (1.5,4.6) [draw,minimum width=25pt,minimum height=10pt,fill=white] {$\scriptstyle{\eta_{W_1,W_2}}$};
			\draw [line width=1pt] (2,0) to (1,1.5);
			\draw [line width=1pt] (1.5,0) to (0,1.5);
			\draw [white, line width=3pt, double=black,double distance=0.5pt] (0,0) to (2,1.2);
			\draw (2,1.2) to (m.210);
			\draw [white, line width=3pt, double=black, double distance=1pt] (0,1.5) to (m.330);
			\draw [white, line width=3pt, double=black, double distance=1pt] (1,1.5) to (2,3) to (e.330);				
			\draw [line width=1pt] (m.90) to (e.210);
			\draw [line width=1pt] (e.90) to (1.5,5.2);
		\end{tikzpicture}
	\end{matrix}
	=
	\begin{matrix}
		\begin{tikzpicture}[scale=0.9, out=up, in=down, line width=0.5pt]
			\node at (0,-0.3) {$\scriptstyle{A}$};
			\node at (1.4,-0.3) {$\scriptstyle{W_1}$};
			\node at (2.1,-0.3) {$\scriptstyle{W_2}$};
			\node at (1.5,5.5) {$\scriptstyle{W_1\boxtimes_A W_2}$};
			\node (m) at (0.5,3.2) [draw,minimum width=25pt,minimum height=10pt,thick, fill=white] {$\scriptstyle{\mu_{W_1}}$};
			\node (e) at (1.5,4.6) [draw,minimum width=25pt,minimum height=10pt] {$\scriptstyle{\eta_{W_1,W_2}}$};
			\draw [line width=1pt] (2,0) to (1,1.5);
			\draw [line width=1pt] (1.5,0) to (0,1.5);
			\draw [white, line width=3pt, double=black,double distance=0.5pt] (0,0) to (2,1.2);
			\draw [white, line width=3pt, double=black, double distance=1pt] (0,1.5) to (m.330);
			\draw [white, line width=3pt, double=black, double distance=0.5pt] (2,1.2) to (m.210);
			\draw [white, line width=3pt, double=black, double distance=1pt] (1,1.5) to (2,3) to (e.330);				
			\draw [line width=1pt] (m.90) to (e.210);
			\draw [line width=1pt] (e.90) to (1.5,5.2);
		\end{tikzpicture}
	\end{matrix}
	=
	\begin{matrix}
		\begin{tikzpicture}[scale=0.9, out=up, in=down, line width=0.5pt]
			\node at (0,-0.3) {$\scriptstyle{A}$};
			\node at (1.4,-0.3) {$\scriptstyle{W_1}$};
			\node at (2.1,-0.3) {$\scriptstyle{W_2}$};
			\node at (0.5,5.5) {$\scriptstyle{W_1\boxtimes_A W_2}$};
			\node (m) at (1.5,3.3) [draw,minimum width=25pt,minimum height=10pt,thick, fill=white] {$\scriptstyle{\mu_{W_2}}$};
			\node (e) at (0.5,4.6) [draw,minimum width=25pt,minimum height=10pt] {$\scriptstyle{\eta_{W_1,W_2}}$};
			\draw [line width=1pt] (1.5,0) 	to (0,1.2) to (0.7,2.3) to (0,3.5) to (e.210);
			\draw [white, line width=3pt, double=black, double distance=0.5pt] (2,1.2) to (0,2.3) to (m.210);
			\draw [line width=1pt] (2,0) to (1,1.2);
			\draw[line width=1pt] (m.north) 	to (e.330);
			\draw[line width=1pt] (e.north)	to (0.5,5.2);
			\draw[white, line width=3pt, double=black, double distance=0.5pt] (0,0) to (2,1.2);
			\draw[white, line width=3pt, double=black, double distance=1pt]	 (1,1.2) to (2,2.3) to (m.330);
		\end{tikzpicture}
	\end{matrix}
	=
	\begin{matrix}
		\begin{tikzpicture}[scale=0.9, out=up, in=down, line width=0.5pt]
			\node at (0,-0.3) {$\scriptstyle{A}$};
			\node at (1.4,-0.3) {$\scriptstyle{W_1}$};
			\node at (2.1,-0.3) {$\scriptstyle{W_2}$};
			\node at (0.5,5.5) {$\scriptstyle{W_1\boxtimes_A W_2}$};
			\node (m) at (1.5,3.3) [draw,minimum width=25pt,minimum height=10pt] {$\scriptstyle{\mu_{W_2}}$};
			\node (e) at (0.5,4.6) [draw,minimum width=25pt,minimum height=10pt,line width=1pt, fill=white] {$\scriptstyle{\eta_{W_1,W_2}}$};
			\draw (2,1.2) to (1,2.3) to (m.210);
			\draw [line width=1pt] (2,0) to (1,1.2);
			\draw [line width=1pt] (1.5,0) 	to (0,1.2) to (e.210);
			\draw[line width=1pt] (m.north) 	to (e.330);
			\draw[line width=1pt] (e.north)	to (0.5,5.2);
			\draw[white, line width=3pt, double=black, double distance=1pt]	 (1,1.2) to (2,2.3) to (m.330);
			\draw[white, line width=3pt, double=black, double distance=0.5pt] (0,0) to (2,1.2);
		\end{tikzpicture}
	\end{matrix}
	=
	\begin{matrix}
		\begin{tikzpicture}[scale=0.9, out=up, in=down, line width=0.5pt]
			\node at (0,-0.3) {$\scriptstyle{A}$};
			\node at (0.9,-0.3) {$\scriptstyle{W_1}$};
			\node at (1.6,-0.3) {$\scriptstyle{W_2}$};
			\node at (0.5,5.5) {$\scriptstyle{W_1\boxtimes_A W_2}$};
			\node (m) at (1.2,3.3) [draw,minimum width=25pt,minimum height=10pt] {$\scriptstyle{\mu_{W_2}}$};
			\node (e) at (0.5,4.6) [draw,minimum width=25pt,minimum height=10pt,line width=1pt, fill=white] {$\scriptstyle{\eta_{W_1,W_2}}$};
			\draw [line width=1pt] (1,0) to (0,2.5) to (e.210);
			\draw [line width=1pt] (1.5,0) to (m.330);
			\draw [line width=1pt] (m.north) to (e.330);
			\draw [line width=1pt] (e.north) to (0.5,5.2);
			\draw [white, line width=3pt, double=black, double distance=0.5pt] (0,0) to (m.210);
		\end{tikzpicture}
	\end{matrix}
\end{align*}
The first and last equalities use the hexagon axiom together with the assumption that $W_1$ and $W_2$ are objects of $\repzA$, while the second uses $\eta_{W_1,W_2}\circ\mu^{(1)}=\eta_{W_1,W_2}\circ\mu^{(2)}$.
This proves that $(W_1\boxtimes_A W_2, \mu_{W_1\boxtimes_A W_2})$ is an object of $\repzA$, so $\boxtimes_A$ is a bifunctor on $\repzA$.

To get braiding isomorphisms in $\repzA$, the universal property of the cokernel $(W_1\boxtimes_A W_2, \eta_{W_1,W_2})$ induces a unique $\sC$-morphism $\cR^A_{W_1,W_2}: W_1\boxtimes_A W_2\rightarrow W_2\boxtimes_A W_1$ such that \eqref{rep0Abraidingdef} commutes, provided
\begin{equation*}
 \eta_{W_2,W_1}\circ\sR_{W_1,W_2}\circ\mu^{(2)}=\eta_{W_2,W_1}\circ\sR_{W_1,W_2}\circ\mu^{(1)}.
\end{equation*}
We use the following diagrams to establish this:
	\begin{align*}
	\begin{matrix}
		\begin{tikzpicture}[scale=0.8, out=up, in=down, line width=0.5pt]
			\node at (0,-0.3) {$\scriptstyle{A}$};
			\node at (1,-0.3) {$\scriptstyle{W_1}$};
			\node at (2,-0.3) {$\scriptstyle{W_2}$};
			\node at (1,6.5) {$\scriptstyle{W_1\boxtimes_A W_2}$};
			\node (m) at (1.5,2.5) [draw,minimum width=25pt,minimum height=10pt] {$\scriptstyle{\mu_{W_2}}$};
			\node (e) at (1,5.5) [draw,minimum width=25pt,minimum height=10pt] {$\scriptstyle{\eta_{W_2,W_1}}$};
			\draw[line width=1pt] (2,0) to (m.330);
			\draw[line width=1pt] (m.90) to (e.210);
			\draw[line width=1pt] (e.90) to (1,6.2); 
			\draw[line width=1pt] (1,0) to (0.2,2);
			\draw[white, line width=3pt, double=black, double distance=0.5pt] (0,0) to (m.210);
			\draw[white, line width=3pt, double=black, double distance=1pt] (0.2,2) to (e.330);
		\end{tikzpicture}
	\end{matrix}
	=
	\begin{matrix}
		\begin{tikzpicture}[scale=0.8, out=up, in=down, line width=0.5pt]
			\node at (0,-0.3) {$\scriptstyle{A}$};
			\node at (1,-0.3) {$\scriptstyle{W_1}$};
			\node at (1.7,-0.3) {$\scriptstyle{W_2}$};
			\node at (1,6.5) {$\scriptstyle{W_1\boxtimes_A W_2}$};
			\node (m) at (0.5,4) [draw,minimum width=25pt,minimum height=10pt] {$\scriptstyle{\mu_{W_2}}$};
			\node (e) at (1,5.5) [draw,minimum width=25pt,minimum height=10pt] {$\scriptstyle{\eta_{W_2,W_1}}$};
			\draw[line width=1pt] (1.7,0) to (1.7,2) to (m.330);
			\draw[line width=1pt] (1,0) to  (0,1.7); 
			\draw[white, line width=3pt, double=black, double distance=0.5pt] (0,0) to (1,2) to (m.210);
			\draw[white, line width=3pt, double=black, double  distance=1pt] (0,1.7) to (1.7,3.6) to (e.330);
			\draw[line width=1pt] (m.north)to (e.210);
			\draw[line width=1pt] (e.north) to (1,6.2);
		\end{tikzpicture}
	\end{matrix}
	=
	\begin{matrix}
		\begin{tikzpicture}[scale=0.8, out=up, in=down, line width=0.5pt]
			\node at (0,-0.3) {$\scriptstyle{A}$};
			\node at (1,-0.3) {$\scriptstyle{W_1}$};
			\node at (2,-0.3) {$\scriptstyle{W_2}$};
			\node at (1,6.5) {$\scriptstyle{W_1\boxtimes_A W_2}$};
			\node (m) at (1.5,4) [draw,minimum width=25pt,minimum height=10pt] {$\scriptstyle{\mu_{W_1}}$};
			\node (e) at (1,5.5) [draw,minimum width=25pt,minimum height=10pt] {$\scriptstyle{\eta_{W_2,W_1}}$};
			\draw [line width=1pt] (2,0) to (2,1.5) to (0.5,3.5) to (e.210);
			\draw [line width=1pt] (1,0) to (0,1.5);
			\draw [white,line width=3pt,double=black,double  distance=0.5pt] (0,0) to (1,1.5) to (0,2.7) to (m.210);
			\draw [white,line width=3pt,double=black,double distance=1pt] (0,1.5) to (m.330);
			\draw[line width=3pt,white,double=black,double distance=1pt]  (m.north) to (e.330);
			\draw[line width=1pt] (e.north) to (1,6.2);
		\end{tikzpicture}
	\end{matrix}
	=
	\begin{matrix}
		\begin{tikzpicture}[scale=0.8, out=up, in=down, line width=0.5pt]
			\node at (0,-0.3) {$\scriptstyle{A}$};
			\node at (1,-0.3) {$\scriptstyle{W_1}$};
			\node at (2,-0.3) {$\scriptstyle{W_2}$};
			\node at (1,6.5) {$\scriptstyle{W_1\boxtimes_A W_2}$};
			\node (m) at (0.5,3) [draw,minimum width=25pt,minimum height=10pt] {$\scriptstyle{\mu_{W_1}}$};
			\node (e) at (1,5.5) [draw,minimum width=25pt,minimum height=10pt] {$\scriptstyle{\eta_{W_2,W_1}}$};
			\draw [line width=1pt] (1,0) to (0,1.5);
			\draw [white,line width=3pt,double=black,double distance=0.5pt] (0,0) to (1,1.5) to (m.210);
			\draw [white,line width=3pt,double=black,double distance=1pt] (0,1.5) to (m.330);
			\draw [line width=1pt] (2,0) to (2,2.5) to (e.210);
			\draw[line width=3pt,white,double=black,double distance=1pt] (m.north) to (e.330);
			\draw[line width=1pt] (e.north) to (1,6.2);
		\end{tikzpicture}
	\end{matrix}
	=
	\begin{matrix}
		\begin{tikzpicture}[scale=0.8, out=up, in=down, line width=0.5pt]
			\node at (0,-0.3) {$\scriptstyle{A}$};
			\node at (1,-0.3) {$\scriptstyle{W_1}$};
			\node at (2,-0.3) {$\scriptstyle{W_2}$};
			\node at (1,6.5) {$\scriptstyle{W_1\boxtimes_A W_2}$};
			\node (m) at (0.5,3) [draw,minimum width=25pt,minimum height=10pt] {$\scriptstyle{\mu_{W_1}}$};
			\node (e) at (1,5.5) [draw,minimum width=25pt,minimum height=10pt] {$\scriptstyle{\eta_{W_2,W_1}}$};
			\draw [line width=1pt] (2,0) to (2,2.5) to (e.210);
			\draw[] (0,0) to (m.210);
			\draw[line width=1pt] (1,0) to (m.330);
			\draw[line width=3pt,white,double=black,double distance=1pt] (m.north) to (e.330);
			\draw[line width=1pt] (e.north) to (1,6.2);
		\end{tikzpicture}
	\end{matrix}
\end{align*}
These steps use naturality of braiding applied to the even morphisms $\mu_{W_2}$ and $\mu_{W_1}$, the hexagon axiom twice, the relation $\eta_{W_2,W_1}\circ\mu^{(1)}=\eta_{W_2,W_1}\circ\mu^{(2)}$, and finally the assumption that $W_1$ is an object of $\repzA$. This shows  $\cR^A_{W_1,W_2}$ exists, and $\cR^A_{W_1,W_2}$ is even because  $\sR_{W_1,W_2}$, $\eta_{W_2,W_1}$ are even and because $\eta_{W_1,W_2}$ is even and surjective.

To show that $\cR^A_{W_1,W_2}$ is an isomorphism in $\sC$, we will use the universal property of the cokernel $(W_2\boxtimes_A W_1, \eta_{W_2,W_1})$ to show that there is a unique $\sC$-morphism $(\cR^A_{W_1,W_2})^{-1}: W_2\boxtimes_A W_1\rightarrow W_1\boxtimes_A W_2$ such that
\begin{equation}\label{rep0Abraidinv}
 \xymatrixcolsep{4pc}
 \xymatrix{
 W_2\boxtimes W_1 \ar[r]^{\sR_{W_1,W_2}^{-1}} \ar[d]_{\eta_{W_2,W_1}} & W_1\boxtimes W_2 \ar[d]^{\eta_{W_1,W_2}} \\
 W_2\boxtimes_A W_1 \ar[r]^{(\cR^A_{W_1,W_2})^{-1}} & W_1\boxtimes_A W_2 \\
 }
\end{equation}
commutes. The commutativity of \eqref{rep0Abraidingdef} and \eqref{rep0Abraidinv} together with the surjectivity of $\eta_{W_1,W_2}$ and $\eta_{W_2,W_1}$ will then imply that $\cR^A_{W_1,W_2}$ and $(\cR^A_{W_1,W_2})^{-1}$ are indeed inverses of each other. To show that $(\cR^A_{W_1,W_2})^{-1}$ exists, it suffices to show
\begin{equation}
 \eta_{W_1,W_2}\circ\sR_{W_1,W_2}^{-1}\circ\mu^{(2)}=\eta_{W_1,W_2}\circ\sR^{-1}_{W_1,W_2}\circ\mu^{(1)}.
\end{equation}
We use the following sequence of diagrams to prove this, first using \eqref{eqn:mu2alt} to rewrite $\mu^{(2)}$:
	\begin{align*}
	\begin{matrix}
		\begin{tikzpicture}[scale=0.8, out=up, in=down, line width=0.5pt]
			\node at (0,-0.3) {$\scriptstyle{A}$};
			\node at (1,-0.3) {$\scriptstyle{W_2}$};
			\node at (2,-0.3) {$\scriptstyle{W_1}$};				
			\node at (1, 5.3) {$\scriptstyle{W_1\boxtimes_A W_2}$};
			\node (m) at (0.5,1.5) [draw,minimum width=25pt,minimum height=10pt] {$\scriptstyle{\mu_{W_1}}$};
			\node (e) at (1,4) [draw,minimum width=25pt,minimum height=10pt] {$\scriptstyle{\eta_{W_1,W_2}}$};
			\draw (0,0) to (m.210);
			\draw [line width=1pt] (e.90) to (1,5);
			\draw [line width=1pt] (1,0) to (2,1.4) to (0.5,2.8) to (e.330);
			\draw [white, line width=3pt, double=black, double  distance=1pt] (m.90) to (1.5,2.8) to (e.210);
			\draw [white, line width=3pt, double=black, double  distance=1pt] (2,0) to (m.330);
		\end{tikzpicture}
	\end{matrix}
	=
	\begin{matrix}
		\begin{tikzpicture}[scale=0.8, out=up, in=down, line width=0.5pt]	
			\node at (0,-0.3) {$\scriptstyle{A}$};
			\node at (1,-0.3) {$\scriptstyle{W_2}$};
			\node at (2,-0.3) {$\scriptstyle{W_1}$};				
			\node at (1, 5.3) {$\scriptstyle{W_1\boxtimes_A W_2}$};
			\node (m) at (0.5,1.5) [draw,minimum width=25pt,minimum height=10pt] {$\scriptstyle{\mu_{W_1}}$};
			\node (e) at (1,4) [draw,minimum width=25pt,minimum height=10pt] {$\scriptstyle{\eta_{W_1,W_2}}$};
			\draw (0,0) to (m.210);
			\draw[line width=1pt] (m.90) to (e.210);
			\draw [line width=1pt] (e.90) to (1,5);
			\draw [line width=1pt] (1,0) to (1.7,1.3) to (e.330);
			\draw [white, line width=3pt, double=black, double distance=1pt] (2,0) to (m.330);
		\end{tikzpicture}
	\end{matrix}
	=
	\begin{matrix}
		\begin{tikzpicture}[scale=0.8, out=up, in=down, line width=0.5pt]
			\node at (0,-0.3) {$\scriptstyle{A}$};
			\node at (1,-0.3) {$\scriptstyle{W_2}$};
			\node at (2,-0.3) {$\scriptstyle{W_1}$};				
			\node at (1, 5.3) {$\scriptstyle{W_1\boxtimes_A W_2}$};
			\node (m) at (0.5,2.7) [draw,minimum width=25pt,minimum height=10pt] {$\scriptstyle{\mu_{W_1}}$};
			\node (e) at (1,4) [draw,minimum width=25pt,minimum height=10pt] {$\scriptstyle{\eta_{W_1,W_2}}$};
			\draw (0,0) to (1,1.3);
			\draw[line width=1pt] (1,0) to (1.7,1.7) to (e.330);
			\draw[white, line width=3pt, double=black, double distance=1pt] (2,0) to (0,1.3) to (m.330);
			\draw[white, line width=3pt, double=black, double distance=0.5pt] (1,1.3) to (m.210);				
			\draw [line width=1pt] (m.90) to (e.210);
			\draw [line width=1pt] (e.90) to (1,5);
		\end{tikzpicture}
	\end{matrix}
	=
	\begin{matrix}
		\begin{tikzpicture}[scale=0.8, out=up, in=down, line width=0.5pt]
			\node at (0,-0.3) {$\scriptstyle{A}$};
			\node at (1,-0.3) {$\scriptstyle{W_2}$};
			\node at (2,-0.3) {$\scriptstyle{W_1}$};				
			\node at (0.5, 5.3) {$\scriptstyle{W_1\boxtimes_A W_2}$};
			\node (m) at (1.5,2.8) [draw,minimum width=25pt,minimum height=10pt] {$\scriptstyle{\mu_{W_2}}$};
			\node (e) at (0.5,4) [draw,minimum width=25pt,minimum height=10pt] {$\scriptstyle{\eta_{W_1,W_2}}$};
			\draw (0,0) to (1,1);
			\draw [line width=1pt] (1,0) to (2,1.5) to (m.330);
			\draw [white, line width=3pt, double=black, double distance=1pt] (0,1) to (1,1.7) to (0,2.8) to (e.210);
			\draw [white, line width=3pt, double=black, double distance=0.5pt] (1,1) to (0,1.7) to (m.210);
			\draw [white, line width=3pt, double=black, double distance=1pt] (2,0) to (0,1);
			\draw [line width=1pt] (m.90) to (e.330);
			\draw [line width=1pt] (e.90) to (0.5,5);
		\end{tikzpicture}
	\end{matrix}
	=
	\begin{matrix}
		\begin{tikzpicture}[scale=0.8, out=up, in=down, line width=0.5pt]
			\node at (0,-0.3) {$\scriptstyle{A}$};
			\node at (1,-0.3) {$\scriptstyle{W_2}$};
			\node at (2,-0.3) {$\scriptstyle{W_1}$};				
			\node at (0.5, 5.3) {$\scriptstyle{W_1\boxtimes_A W_2}$};
			\node (m) at (1.5,2.8) [draw,minimum width=25pt,minimum height=10pt] {$\scriptstyle{\mu_{W_2}}$};
			\node (e) at (0.5,4) [draw,minimum width=25pt,minimum height=10pt] {$\scriptstyle{\eta_{W_1,W_2}}$};
			\draw (0,0) to (m.210);
			\draw [line width=1pt] (1,0) to (m.330);
			\draw [line width=1pt] (m.90) to (e.330);
			\draw[white, line width=3pt, double=black, double  distance=1pt] (2,0) to (0,2) to (e.210);
			\draw [line width=1pt] (e.90) to (0.5,5);
		\end{tikzpicture}
	\end{matrix}
	=
	\begin{matrix}
		\begin{tikzpicture}[scale=0.8, out=up, in=down, line width=0.5pt]
			\node at (0,-0.3) {$\scriptstyle{A}$};
			\node at (1,-0.3) {$\scriptstyle{W_2}$};
			\node at (2,-0.3) {$\scriptstyle{W_1}$};				
			\node at (1, 5.3) {$\scriptstyle{W_1\boxtimes_A W_2}$};
			\node (m) at (0.5,1.5) [draw,minimum width=25pt,minimum height=10pt] {$\scriptstyle{\mu_{W_2}}$};
			\node (e) at (1,4) [draw,minimum width=25pt,minimum height=10pt] {$\scriptstyle{\eta_{W_1,W_2}}$};
			\draw (0,0) to (m.210);
			\draw [line width=1pt] (1,0) to (m.330);
			\draw [line width=1pt] (m.90) to (e.330);
			\draw[white, line width=3pt, double=black, double distance=1pt] (2,0) to (2,1) to (e.210);
			\draw [line width=1pt] (e.90) to (1,5);
		\end{tikzpicture}
	\end{matrix}
\end{align*}
The second equality uses the assumption that $W_1$ is in $\repzA$ along with the hexagon axiom, the third uses $\eta_{W_1,W_2}\circ\mu^{(1)}=\eta_{W_1,W_2}\circ\mu^{(2)}$, and finally we apply naturality of braiding to the even morphism $\mu_{W_2}$.

Now to show that $\cR^A_{W_1,W_2}$ is an isomorphism in $\repA$, and thus also in $\repzA$, it remains to show that $\cR^A_{W_1,W_2}$ is a morphism in $\repA$, that is,
\begin{equation}
 \mu_{W_2\boxtimes_A W_1}\circ(1_A\boxtimes\cR^A_{W_1,W_2})=\cR^A_{W_1,W_2}\circ\mu_{W_1\boxtimes_A W_2}.
\end{equation}
This is a consequence of  the surjectivity of $1_A\boxtimes\eta_{W_1,W_2}$ and the following diagrams:
\begin{align*}
	\begin{matrix}
		\begin{tikzpicture}[scale=1, line width=0.5pt, out=up, in=down]
			\node (a) at (0,0) {$\scriptstyle{A}$};
			\node (w1) at (0.6,0) {$\scriptstyle{W_1}$};
			\node (w2) at (1.2, 0) {$\scriptstyle{W_2}$};
			\node (eta) at (0.9,1.2) [draw] {$\scriptstyle{\eta_{W_1,W_2}}$};
			\node (r) at (0.9,2.4) [draw] {$\scriptstyle{\cR^A_{W_1,W_2}}$};
			\node (mu) at (0.5,3.6) [draw] {$\scriptstyle{\mu_{W_2\boxtimes_A W_1}}$};
			\node (wf) at (0.5,4.5) {$\scriptstyle{W_2\boxtimes_A W_1}$};
			\draw (a.90) to (mu.210);
			\draw [line width=1pt] (w1.90) to (eta.220);
			\draw [line width=1pt] (w2.90) to (eta.320);
			\draw [line width=1pt] (eta.90) to (r.270);
			\draw [line width=1pt] (r.90) to (mu.330);
			\draw [line width=1pt] (mu.90) to (wf.270);
		\end{tikzpicture}
	\end{matrix}
	=
	\begin{matrix}
		\begin{tikzpicture}[scale=1, line width=0.5pt, out=up, in=down]
			\node (a) at (0,0) {$\scriptstyle{A}$};
			\node (w1) at (0.6,0) {$\scriptstyle{W_1}$};
			\node (w2) at (1.2, 0) {$\scriptstyle{W_2}$};
			\node (eta) at (0.9,2.4) [draw] {$\scriptstyle{\eta_{W_2,W_1}}$};
			\node (mu) at (0.5,3.6) [draw] {$\scriptstyle{\mu_{W_2\boxtimes_A W_1}}$};
			\node (wf) at (0.5,4.5) {$\scriptstyle{W_2\boxtimes_A W_1}$};
			\draw (a.90) to (mu.210);
			\draw [line width=1pt] (w2.90) to (eta.220);
			\draw [white, line width=3pt, double=black, double distance=1pt] (w1.90) to (eta.320);
			\draw [line width=1pt] (eta.90) to (mu.330);
			\draw [line width=1pt] (mu.90) to (wf.270);
		\end{tikzpicture}
	\end{matrix}
	=
	\begin{matrix}
		\begin{tikzpicture}[scale=1, line width=0.5pt, out=up, in=down]
			\node (a) at (0,0) {$\scriptstyle{A}$};
			\node (w1) at (0.6,0) {$\scriptstyle{W_1}$};
			\node (w2) at (1.2, 0) {$\scriptstyle{W_2}$};
			\node (mu) at (1,2.3) [draw] {$\scriptstyle{\mu_{W_1}}$};
			\node (eta) at (0.6,3.6) [draw] {$\scriptstyle{\eta_{W_2, W_1}}$};
			\node (wf) at (0.6,4.5) {$\scriptstyle{W_2\boxtimes_A W_1}$};
			\draw [line width=1pt] (w2.90) to (0.2,2) to (eta.210);
			\draw [white, line width=3pt, double=black, double distance=0.5pt] (a.90) to (mu.220);
			\draw [white, line width=3pt, double=black, double distance=1pt] (w1.90) to (mu.320);
			\draw [line width=1pt] (mu.90) to (eta.330);
			\draw [line width=1pt] (eta.90) to (wf.270);
		\end{tikzpicture}
	\end{matrix}
	=
	\begin{matrix}
		\begin{tikzpicture}[scale=1, line width=0.5pt, out=up, in=down]
			\node (a) at (0,0) {$\scriptstyle{A}$};
			\node (w1) at (0.6,0) {$\scriptstyle{W_1}$};
			\node (w2) at (1.2, 0) {$\scriptstyle{W_2}$};
			\node (mu) at (0.3,1.2) [draw] {$\scriptstyle{\mu_{W_1}}$};
			\node (eta) at (0.6,3.6) [draw] {$\scriptstyle{\eta_{W_2, W_1}}$};
			\node (wf) at (0.6,4.5) {$\scriptstyle{W_2\boxtimes_A W_1}$};
			\draw [line width=1pt] (w2.90) to (1.2,1.2) to (eta.210);
			\draw [white, line width=3pt, double=black, double distance=0.5pt] (a.90) to (mu.220);
			\draw [white, line width=3pt, double=black, double distance=1pt] (w1.90) to (mu.320);
			\draw [white, line width=3pt, double=black, double distance=1pt] (mu.90) to (eta.330);
			\draw [line width=1pt] (eta.90) to (wf.270);
		\end{tikzpicture}
	\end{matrix}
	=
	\begin{matrix}
		\begin{tikzpicture}[scale=1, line width=0.5pt, out=up, in=down]
			\node (a) at (0,0) {$\scriptstyle{A}$};
			\node (w1) at (0.6,0) {$\scriptstyle{W_1}$};
			\node (w2) at (1.2, 0) {$\scriptstyle{W_2}$};
			\node (mu) at (0.3,1.2) [draw] {$\scriptstyle{\mu_{W_1}}$};
			\node (eta) at (0.6,2.4) [draw] {$\scriptstyle{\eta_{W_1, W_2}}$};
			\node (r) at (0.6,3.6) [draw] {$\scriptstyle{\cR^A_{W_1, W_2}}$};
			\node (wf) at (0.6,4.5) {$\scriptstyle{W_2\boxtimes_A W_1}$};
			\draw (a.90) to (mu.220);
			\draw [line width=1pt] (w1.90) to (mu.320);
			\draw [line width=1pt] (mu.90) to (eta.220);
			\draw [line width=1pt] (w2.90) to (eta.330);
			\draw [line width=1pt] (eta.90) to (r.270);
			\draw [line width=1pt] (r.90) to (wf.270);
		\end{tikzpicture}
	\end{matrix}
	=
	\begin{matrix}
		\begin{tikzpicture}[scale=1, line width=0.5pt, out=up, in=down]
			\node (a) at (0,0) {$\scriptstyle{A}$};
			\node (w1) at (0.7,0) {$\scriptstyle{W_1}$};
			\node (w2) at (1.4, 0) {$\scriptstyle{W_2}$};
			\node (mu) at (0.6,2.4) [draw] {$\scriptstyle{\mu_{W_1\boxtimes_A W_2}}$};
			\node (eta) at (1,1.2) [draw] {$\scriptstyle{\eta_{W_1, W_2}}$};
			\node (r) at (0.6,3.6) [draw] {$\scriptstyle{\cR^A_{W_1, W_2}}$};
			\node (wf) at (0.6,4.5) {$\scriptstyle{W_2\boxtimes_A W_1}$};
			\draw (a.90) to (mu.210);
			\draw [line width=1pt] (w1.90) to (eta.220);
			\draw [line width=1pt] (eta.90) to (mu.330);
			\draw [line width=1pt] (w2.90) to (eta.330);
			\draw [line width=1pt] (mu.90) to (r.270);
			\draw [line width=1pt] (r.90) to (wf.270);
		\end{tikzpicture}
	\end{matrix}
\end{align*}
The first and fourth equalities  use the definition of $\cR^A_{W_1,W_2}$ from \eqref{rep0Abraidingdef}. The second and last equalities use \eqref{repAtensdef} with $i=2$ and $i=1$, respectively, and the third uses naturality of braiding applied to the even morphism $\mu_{W_1}$.

Now we show that the braiding isomorphism $\cR^A$ is natural, that is,
\begin{equation*}
 \cR^A_{\widetilde{W}_1,\widetilde{W}_2}\circ(f_1\boxtimes_A f_2)=(-1)^{\vert f_1\vert \vert f_2\vert}(f_2\boxtimes_A f_1)\circ\cR^A_{W_1,W_2}
\end{equation*}
for parity-homogeneous morphisms $f_1: W_1\rightarrow\widetilde{W}_1$ and $f_2: W_2\rightarrow\widetilde{W}_2$ in $\repzA$. This follows from the commutative diagrams
\begin{equation*}
\xymatrixcolsep{4pc}
\xymatrix{
W_1\boxtimes W_2 \ar[r]^{f_1\boxtimes f_2} \ar[d]_{\eta_{W_1,W_2}} & \widetilde{W}_1\boxtimes\widetilde{W}_2 \ar[r]^{\sR_{\widetilde{W}_1,\widetilde{W}_2}} \ar[d]^{\eta_{\widetilde{W}_1,\widetilde{W}_2}} & \widetilde{W}_2\boxtimes\widetilde{W}_1 \ar[d]^{\eta_{\widetilde{W}_2,\widetilde{W}_1}} \\
W_1\boxtimes_A W_2 \ar[r]^{f_1\boxtimes_A f_2} & \widetilde{W}_1\boxtimes_A\widetilde{W}_2 \ar[r]^{\cR^A_{\widetilde{W}_1,\widetilde{W}_2}} & \widetilde{W}_2\boxtimes_A\widetilde{W}_1 \\
}
\end{equation*}
and
\begin{equation*}
\xymatrixcolsep{4pc}
\xymatrix{
W_1\boxtimes W_2 \ar[r]^{\sR_{W_1,W_2}} \ar[d]_{\eta_{W_1,W_2}} & W_2\boxtimes W_1 \ar[r]^{f_2\boxtimes f_1} \ar[d]_{\eta_{W_2,W_1}} & \widetilde{W}_2\boxtimes\widetilde{W}_1 \ar[d]^{\eta_{\widetilde{W}_2,\widetilde{W}_1}} \\
W_1\boxtimes_A W_2 \ar[r]^{\cR^A_{W_1,W_2}} & W_2\boxtimes_A W_1 \ar[r]^{f_2\boxtimes_A f_1} & \widetilde{W}_2\boxtimes_A\widetilde{W}_1 \\
}
\end{equation*}
together with the surjectivity of $\eta_{W_1,W_2}$ and the naturality \eqref{braidingsupernatural} of the braiding isomorphisms in $\sC$.
%
%

Finally, we prove the hexagon axiom for $\cR^A$, that is, we show that
\begin{equation}\label{repzAhex}
\xymatrixcolsep{5pc}
\xymatrix{
 W_1\boxtimes_A (W_2\boxtimes_A W_3) \ar[r]^{1_{W_1}\boxtimes_A\cR^A_{W_2,W_3}} \ar[d]_{\cA^A_{W_1,W_2,W_3}} & W_1\boxtimes_A (W_3\boxtimes_A W_2) \ar[r]^{\cA^A_{W_1,W_3,W_2}} & (W_1\boxtimes_A W_3)\boxtimes_A W_2 \ar[d]^{\cR^A_{W_1,W_3}\boxtimes_A 1_{W_2}} \\
 (W_1\boxtimes_A W_2)\boxtimes_A W_3 \ar[r]^{\cR^A_{W_1\boxtimes_A W_2, W_3}} & W_3\boxtimes_A (W_1\boxtimes_A W_2) \ar[r]^{\cA^A_{W_3,W_1,W_2}} & (W_3\boxtimes_A W_1)\boxtimes_A W_2 \\
 }
\end{equation}
commutes, as does the same diagram with $\cR^A$ replaced by $(\cR^A)^{-1}$. The commutativity of \eqref{repzAhex} follows from the commutative diagrams
\begin{equation*}
 \xymatrixcolsep{4pc}
 \xymatrix{
 W_1\boxtimes(W_2\boxtimes W_3) \ar[r]^{1_{W_1}\boxtimes\eta_{W_2,W_3}} \ar[d]_{1_{W_1}\boxtimes\sR_{W_2,W_3}} & W_1\boxtimes(W_2\boxtimes_A W_3) \ar[r]^{\eta_{W_1, W_2\boxtimes_A W_3}} \ar[d]^{1_{W_1}\boxtimes\cR^A_{W_2,W_3}} & W_1\boxtimes_A (W_2\boxtimes_A W_3) \ar[d]^{1_{W_1}\boxtimes_A\cR^A_{W_2,W_3}} \\
 W_1\boxtimes(W_3\boxtimes W_2) \ar[r]^{1_{W_1}\boxtimes\eta_{W_3,W_2}} \ar[d]_{\sA_{W_1,W_3,W_2}} & W_1\boxtimes (W_3\boxtimes_A W_2) \ar[r]^{\eta_{W_1,W_3\boxtimes_A W_2}} & W_1\boxtimes_A (W_3\boxtimes_A W_2) \ar[d]^{\cA^A_{W_1,W_3,W_2}} \\
 (W_1\boxtimes W_3)\boxtimes W_2 \ar[r]^{\eta_{W_1,W_3}\boxtimes 1_{W_2}} \ar[d]_{\sR_{W_1,W_3}\boxtimes 1_{W_2}} & (W_1\boxtimes_A W_3)\boxtimes W_2 \ar[r]^{\eta_{W_1\boxtimes_A W_3,W_2}} \ar[d]^{\cR^A_{W_1,W_3}\boxtimes 1_{W_2}} & (W_1\boxtimes_A W_3)\boxtimes_A W_2 \ar[d]^{\cR^A_{W_1,W_3}\boxtimes_A 1_{W_2}} \\
 (W_3\boxtimes W_1)\boxtimes W_2 \ar[r]^{\eta_{W_3,W_1}\boxtimes 1_{W_2}} & (W_3\boxtimes_A W_1)\boxtimes W_2 \ar[r]^{\eta_{W_3\boxtimes_A W_1, W_2}} & (W_3\boxtimes_A W_1)\boxtimes_A W_2
 }
\end{equation*}
and
\begin{equation*}
 \xymatrixcolsep{4pc}
 \xymatrix{
 W_1\boxtimes(W_2\boxtimes W_3) \ar[r]^{1_{W_1}\boxtimes\eta_{W_2,W_3}} \ar[d]_{\sA_{W_1,W_2,W_3}} & W_1\boxtimes(W_2\boxtimes_A W_3) \ar[r]^{\eta_{W_1,W_2\boxtimes_A W_3}} & W_1\boxtimes_A (W_2\boxtimes_A W_3) \ar[d]^{\cA^A_{W_1,W_2,W_3}} \\
 (W_1\boxtimes W_2)\boxtimes W_3 \ar[r]^{\eta_{W_1,W_2}\boxtimes 1_{W_3}} \ar[d]_{\sR_{W_1\boxtimes W_2, W_3}} & (W_1\boxtimes_A W_2)\boxtimes W_3 \ar[r]^{\eta_{W_1\boxtimes_A W_2, W_3}} \ar[d]^{\sR_{W_1\boxtimes_A W_2, W_3}} & (W_1\boxtimes_A W_2)\boxtimes_A W_3 \ar[d]^{\cR^A_{W_1\boxtimes_A W_2, W_3}} \\
 W_3\boxtimes(W_1\boxtimes W_2) \ar[r]^{1_{W_3}\boxtimes\eta_{W_1,W_2}} \ar[d]_{\sA_{W_3,W_1,W_2}} & W_3\boxtimes(W_1\boxtimes_A W_2) \ar[r]^{\eta_{W_3,W_1\boxtimes_A W_2}} & W_3\boxtimes_A (W_1\boxtimes_A W_2) \ar[d]^{\sA^A_{W_3,W_1,W_2}} \\
 (W_3\boxtimes W_1)\boxtimes W_2 \ar[r]^{\eta_{W_3,W_1}\boxtimes 1_{W_2}} & (W_3\boxtimes_A W_1)\boxtimes W_2 \ar[r]^{\eta_{W_3\boxtimes_A W_1, W_2}} & (W_3\boxtimes_A W_1)\boxtimes_A W_2 \\
 },
\end{equation*}
the surjectivity of $\eta_{W_1,W_2\boxtimes_A W_3}\circ(1_{W_1}\boxtimes\eta_{W_2,W_3})$, and the hexagon axiom in $\sC$. The commutativity of \eqref{repzAhex} with $\cR^A$ replaced by $(\cR^A)^{-1}$ follows similarly.
\end{proof}

We cannot say that $\repzA$ is a braided tensor supercategory because $\repzA$ may not be abelian. However:
\begin{propo}
 Every parity-homogeneous morphism in $\repzA$ has a kernel and cokernel. Moreover, every parity-homogeneous monomorphism in $\repzA$ is a kernel, and every parity-homogeneous epimorphism is a cokernel.
\end{propo}
\begin{proof}
We showed in Proposition \ref{SuperRepAadditive} that every parity-homogeneous morphism $f: W_1\rightarrow W_2$ in $\repzA$ has a kernel $(K,\mu_K, k)$ and cokernel $(C,\mu_C, c)$ in $\repA$. So we just need to show that $K$ and $C$ are objects of $\repzA$. By the definition of $\mu_K$, we have
\begin{equation*}
 k\circ\mu_K\circ\cM_{A,K}=\mu_{W_1}\circ(1_A\boxtimes k)\circ\cM_{A,K}=\mu_{W_1}\circ\cM_{A,W_1}\circ(1_A\boxtimes k)=\mu_{W_1}\circ(1_A\boxtimes k)=k\circ\mu_K;
\end{equation*}
since $k$ is a monomorphism, $K$ is an object of $\repzA$. We get $\mu_C\circ\cM_{A,C}=\mu_C$ similarly, so $C$ is an object of $\repzA$ as well.

To show that every parity-homogeneous monomorphism in $\repzA$ is a kernel and every parity-homogeneous epimorphism in $\repzA$ is a cokernel, the same arguments as in the proof of Theorem \ref{thm:repAabelian}, with $\cC$ replaced by $\repA$ and with $\repA$ replaced by $\repzA$, show that a monomorphism in $\repzA$ is also monic in $\repA$ and an epimorphism in $\repzA$ is also epic in $\repA$. Then similar to Theorem \ref{thm:repAabelian}, any monomorphism in $\repzA$ is the kernel of its cokernel (which is an object of $\repzA$), and any epimorphism in $\repzA$ is the cokernel of its kernel (which is an object of $\repzA$).
\end{proof}

\begin{rema}
 The preceding theorem and proposition show that the subcategory $\underline{\repzA}$ of $\repzA$ which has the same objects but only even morphisms is an $\mathbb{F}$-linear braided tensor category. In general, if $A$ is a commutative associative algebra in a tensor category $\cC$, $\repzA$ is a braided tensor category.
\end{rema}

\subsection{Induction and restriction functors}
\label{subsec:inductionfunctor}

The bulk of this subsection contains a construction (following \cite{KO}, but generalized to the superalgebra setting) of an induction functor $\cF: \mathcal{SC}\rightarrow\repA$. Since induction is critical for the applications later in this paper, we give a full proof (in diagram form) that it is a tensor functor. We also briefly discuss the corresponding right adjoint restriction functor from $\repA$ to $\sC$.

Suppose $W$ is an object of $\sC$; we define $\cF(W)=A\boxtimes W$ and
\begin{equation*}
\mu_{\cF(W)}: A\boxtimes (A\boxtimes W)\xrightarrow{\sA_{A,A,W}} (A\boxtimes A)\boxtimes W\xrightarrow{\mu\boxtimes 1_W} A\boxtimes W.
\end{equation*}
\begin{propo}
	The pair $(\cF(W),\mu_{\cF(W)})$ is an object of $\repA$ and $$\cF: W\mapsto (\cF(W),\mu_{\cF(W)});\,\,f\mapsto 1_A\boxtimes f$$ defines a functor from $\sC$ to $\repA$.
\end{propo}
\begin{proof}
	First, $\mu_{\cF({W})}$ is even because $\sA_{A,A,W}$ and $\mu\boxtimes 1_W$ are even. The proof that $\mu_{\cF(W)}$ is associative is shown in the following diagrams; note that the pentagon axiom and naturality of associativity isomorphisms are needed in the detailed proof:
	\begin{align*}
		\begin{matrix}
			\begin{tikzpicture}[scale=1, out=up, in=down, line width=0.5pt]
				\node (a1) at (0,0) {$\scriptstyle{A}$};
				\node (a2) at (0.5,0) {$\scriptstyle{A}$};
				\node (w) at (1.1,0) {$\scriptstyle{\cF(W)}$};
				\node (mu1) at (0.75, 1) [draw] {$\scriptstyle{\mu_{\cF(W)}}$};
				\node (mu2) at (0.25, 2) [draw] {$\scriptstyle{\mu_{\cF(W)}}$};		
				\node (wf) at (0.25, 3) {$\scriptstyle{\cF(W)}$};
				\draw (a1.90) to (mu2.220);
				\draw (a2.90) to (mu1.220);
				\draw [line width=1pt] (w.90) to (mu1.320);
				\draw [line width=1pt] (mu1.90) to (mu2.320);
				\draw [line width=1pt] (mu2.90) to (wf.270);
			\end{tikzpicture}
		\end{matrix}
		=
		\begin{matrix}
			\begin{tikzpicture}[out=up,in=down,line width=0.5pt]
				\node (a1) at (0,0) {$\scriptstyle{A}$};
				\node (a2) at (0.5,0) {$\scriptstyle{A}$};
				\node (a3) at (1,0) {$\scriptstyle{A}$};
				\node (w) at (1.5,0) {$\scriptstyle{W}$};
				\node (mu1) at (0.75, 1) [draw, minimum width=20pt] {$\scriptstyle{\mu}$};
				\node (mu2) at (0.25, 2) [draw, minimum width=20pt] {$\scriptstyle{\mu}$};		
				\node (af) at (0.25, 3) {$\scriptstyle{A}$};
				\node (wf) at (1.5, 3) {$\scriptstyle{W}$};
				\draw (a1.90) to (mu2.220);
				\draw (a2.90) to (mu1.220);
				\draw (a3.90) to (mu1.320);
				\draw [line width=1pt] (w.90) to (wf.270);
				\draw  (mu1.90) to (mu2.320);
				\draw (mu2.90) to (af.270);
			\end{tikzpicture}
		\end{matrix}
		=
		\begin{matrix}
			\begin{tikzpicture}[out=up,in=down,line width=0.5pt]
				\node (a1) at (0,0) {$\scriptstyle{A}$};
				\node (a2) at (0.5,0) {$\scriptstyle{A}$};
				\node (a3) at (1,0) {$\scriptstyle{A}$};
				\node (w) at (1.5,0) {$\scriptstyle{W}$};
				\node (mu1) at (0.25, 1) [draw, minimum width=20pt] {$\scriptstyle{\mu}$};
				\node (mu2) at (0.5, 2) [draw, minimum width=20pt] {$\scriptstyle{\mu}$};		
				\node (af) at (0.5, 3) {$\scriptstyle{A}$};
				\node (wf) at (1.5, 3) {$\scriptstyle{W}$};
				\draw (a1.90) to (mu1.220);
				\draw (a2.90) to (mu1.320);
				\draw (a3.90) to (mu2.320);
				\draw [line width=1pt] (w.90) to (wf.270);
				\draw  (mu1.90) to (mu2.220);
				\draw (mu2.90) to (af.270);
			\end{tikzpicture}
		\end{matrix}
		=
		\begin{matrix}
			\begin{tikzpicture}[scale=1, out=up, in=down, line width=0.5pt]
				\node (a1) at (0,0) {$\scriptstyle{A}$};
				\node (a2) at (0.5,0) {$\scriptstyle{A}$};
				\node (w) at (1.1,0) {$\scriptstyle{\cF(W)}$};
				\node (mu1) at (0.25, 1) [draw,minimum width=20pt] {$\scriptstyle{\mu}$};
				\node (mu2) at (0.6, 2) [draw] {$\scriptstyle{\mu_{\cF(W)}}$};		
				\node (wf) at (0.6, 3) {$\scriptstyle{\cF(W)}$};
				\draw (a1.90) to (mu1.220);
				\draw (a2.90) to (mu1.320);
				\draw [line width=1pt] (w.90) to (mu2.320);
				\draw  (mu1.90) to (mu2.220);
				\draw [line width=1pt] (mu2.90) to (wf.270);
			\end{tikzpicture}
		\end{matrix}
	\end{align*}
The unit property of $\mu_{\cF(W)}$ follows from properties of the associativity and unit isomorphisms in $\sC$ along with the unit property of $\mu$:
\begin{align*}
	\begin{matrix}
		\begin{tikzpicture}[line width=0.5pt, out=up, in=down]
			\node (w) at (0.5,0) {$\scriptstyle{\cF(W)}$};
			\node (i) at (0,1.2) {$\scriptstyle{\bullet}$};
			\node (l) at (0.5,0.6) {};
			\node (mu) at (0.25,2) [draw] {$\scriptstyle{\mu_{\cF(W)}}$};
			\node (wf) at (0.25,3) {$\scriptstyle{{\cF(W)}}$};
			\draw [line width=1pt] (w.90) to (mu.320);
			\draw [dashed] (l.180) to[out=west] (i.270);
			\draw (i.90) to (mu.220);
			\draw [line width=1pt] (mu.90) to (wf.270);
		\end{tikzpicture}
	\end{matrix}
	=
	\begin{matrix}
		\begin{tikzpicture}[line width=0.5pt, out=up, in=down]
			\node (a) at (0.5,0) {$\scriptstyle{A}$};
			\node (i) at (0,1.2) {$\scriptstyle{\bullet}$};
			\node (l) at (0.5,0.6) {};
			\node (w) at (1,0) {$\scriptstyle{W}$};
			\node (mu) at (0.25,2) [draw,minimum width=20pt] {$\scriptstyle{\mu}$};
			\node (af) at (0.25,3) {$\scriptstyle{A}$};
			\node (wf) at (1,3) {$\scriptstyle{W}$};
			\draw (a.90) to (mu.320);
			\draw [dashed] (l.180) to[out=west] (i.270);
			\draw (i.90) to (mu.220);
			\draw  (mu.90) to (af.270);
			\draw [line width=1pt] (w.90) to (wf.270);
		\end{tikzpicture}
	\end{matrix}
	=
	\begin{matrix}
		\begin{tikzpicture}[line width=0.5pt, out=up, in=down]
			\node (a) at (0,0) {$\scriptstyle{A}$};
			\node (w) at (0.5,0) {$\scriptstyle{W}$};
			\node (af) at (0,3) {$\scriptstyle{A}$};
			\node (wf) at (0.5,3) {$\scriptstyle{W}$};
			\draw (a.90) to (af.270);
			\draw [line width=1pt] (w.90) to (wf.270);
		\end{tikzpicture}
	\end{matrix}
\end{align*}

	To show that $\cF$ defines a functor from $\sC$ to $\repA$, we need to show that if $f: W_1\rightarrow W_2$ is a morphism in $\sC$, then
	\begin{equation*}
	\cF(f)=1_A\boxtimes f: A\boxtimes W_1\rightarrow A\boxtimes W_2
	\end{equation*}
	is a morphism in $\repA$.  This follows from
	\begin{align*}
	\begin{matrix}
		\begin{tikzpicture}[out=up,in=down,line width=0.5pt]
			\node (a) at (0,0) {$\scriptstyle{A}$};
			\node (w) at (1,0) {$\scriptstyle{\cF(W_1)}$};
			\node (f) at (1, 1) [draw] {$\scriptstyle{\cF(f)}$};
			\node (mu) at (0.5, 2) [draw] {$\scriptstyle{\mu_{\cF(W_2)}}$};
			\node (wf) at (0.5, 3) {$\scriptstyle{{\cF(W_2)}}$};
			\draw (a.90) to (mu.210);
			\draw [line width=1pt] (w.90) to (f.270);
			\draw [line width=1pt] (f.90) to (mu.330);
			\draw [line width=1pt] (mu.90) to (wf.270);
		\end{tikzpicture}
	\end{matrix}
	=
	\begin{matrix}
		\begin{tikzpicture}[out=up,in=down,line width=0.5pt]
			\node (a) at (0,0) {$\scriptstyle{A}$};
			\node (a2) at (0.5,0) {$\scriptstyle{A}$};
			\node (w) at (1,0) {$\scriptstyle{W_1}$};
			\node (f) at (1, 1) [draw] {$\scriptstyle{f}$};
			\node (mu) at (0.25, 2) [draw,minimum width=20pt] {$\scriptstyle{\mu}$};
			\node (af) at (0.25, 3)  {$\scriptstyle{A}$};
			\node (wf) at (1, 3) {$\scriptstyle{W_2}$};
			\draw (a.90) to (mu.220);
			\draw (a2.90) to (mu.320);
			\draw [line width=1pt] (w.90) to (f.270);
			\draw [line width=1pt] (f.90) to (wf.270);
			\draw  (mu.90) to (af.270);
		\end{tikzpicture}
	\end{matrix}
	=	
	\begin{matrix}
		\begin{tikzpicture}[out=up,in=down,line width=0.5pt]
			\node (a) at (0,0) {$\scriptstyle{A}$};
			\node (a2) at (0.5,0) {$\scriptstyle{A}$};
			\node (w) at (1,0) {$\scriptstyle{W_1}$};
			\node (f) at (1, 2) [draw] {$\scriptstyle{f}$};
			\node (mu) at (0.25, 1) [draw,minimum width=20pt] {$\scriptstyle{\mu}$};
			\node (af) at (0.25, 3)  {$\scriptstyle{A}$};
			\node (wf) at (1, 3) {$\scriptstyle{W_2}$};
			\draw (a.90) to (mu.220);
			\draw (a2.90) to (mu.320);
			\draw [line width=1pt] (w.90) to (f.270);
			\draw [line width=1pt] (f.90) to (wf.270);
			\draw  (mu.90) to (af.270);
		\end{tikzpicture}
	\end{matrix}
	=
	\begin{matrix}
		\begin{tikzpicture}[out=up,in=down,line width=0.5pt]
			\node (a) at (0,0) {$\scriptstyle{A}$};
			\node (w) at (1,0) {$\scriptstyle{\cF(W_1)}$};
			\node (mu) at (0.5,1) [draw] {$\scriptstyle{\mu_{\cF(W_1)}}$};
			\node (f) at (0.5,2) [draw] {$\scriptstyle{\cF(f)}$};
			\node (wf) at (0.5,3)  {$\scriptstyle{\cF(W_2)}$};
			\draw (a.90) to (mu.210);
			\draw [line width=1pt] (w.90) to (mu.330);
			\draw [line width=1pt] (mu.90) to (f.270);
			\draw [line width=1pt] (f.90) to (wf.270);
		\end{tikzpicture}
	\end{matrix}
	\end{align*} 
	It is clear that with this definition of $\cF(f)$, $\cF$ satisfies the axioms of a functor.
\end{proof}

For applications later in this paper, it is crucial that the induction functor $\cF$ be compatible with the tensor structures on its source and target categories.
The following theorem is essentially Theorem 1.6 in \cite{KO}, but we provide a detailed proof since we work in the generality of superalgebras and make no assumptions on $\sC$ besides right exactness
of the tensor functor.
\begin{theo}[{\cite[Theorem\ 1.6]{KO}}]\label{thm:inductionfucntor}
	The induction functor $\cF$ is a tensor functor.
\end{theo}
\begin{proof}
To show that $\cF$ is a tensor functor, we need to show the following:
\begin{enumerate}
 \item There is an even isomorphism $\varphi: \cF(\sunit)\rightarrow A$.
 
 \item There is an even natural isomorphism $f: \cF\circ\boxtimes\rightarrow\boxtimes_A\circ(\cF\times\cF)$.
 
 \item The natural isomorphism $f$ and the isomorphism $\varphi$ are compatible with the unit isomorphisms in $\sC$ and $\repA$ in the sense that the diagrams
 \begin{equation}\label{Funitcompat}
  \xymatrixcolsep{4pc}
  \xymatrix{
  \cF(\sunit\boxtimes W) \ar[r]^{f_{\sunit, W}} \ar[d]_{\cF(\sleft_W)} & \cF(\sunit)\boxtimes_A \cF(W) \ar[d]^{\varphi\boxtimes_A 1_{\cF(W)}} \\
  \cF(W) \ar[r]^{(l^A_{\cF(W)})^{-1}} & A\boxtimes_A \cF(W)\\
  }\hspace{1em}\mathrm{and}\hspace{1em}
  \xymatrixcolsep{4pc}
  \xymatrix{
  \cF(W\boxtimes\sunit) \ar[r]^{f_{W, \sunit}} \ar[d]_{\cF(\sright_W)} & \cF(W)\boxtimes_A \cF(\sunit) \ar[d]^{1_{\cF(W)}\boxtimes_A \varphi} \\
  \cF(W) \ar[r]^{(r^A_{\cF(W)})^{-1}} & \cF(W)\boxtimes_A A\\
  }
 \end{equation}
commute for any object $W$ in $\sC$.

\item The natural isomorphism $f$ is compatible with the associativity isomorphisms in $\sC$ and $\repA$ in the sense that the diagram
\begin{equation}\label{Fassoc}
	\xymatrixcolsep{7pc}
	\xymatrix{
	\cF(W_1\boxtimes(W_2\boxtimes W_3)) 
	\ar[r]^{\cF(\sA_{W_1,W_2,W_3})} \ar[d]_{ f_{W_1,W_2\boxtimes W_3}} & \cF((W_1\boxtimes W_2)\boxtimes W_3) \ar[d]^{f_{W_1\boxtimes W_2,W_3}}\\
	\cF(W_1)\boxtimes_A \cF(W_2\boxtimes W_3) \ar[d]_{1_{\cF(W_1)}\boxtimes_A f_{W_2,W_3}} & \cF(W_1\boxtimes W_2)\boxtimes_A \cF(W_3) \ar[d]^{f_{W_1,W_2}\boxtimes_A 1_{\cF(W_3)}}\\
	\cF(W_1)\boxtimes_A (\cF(W_2)\boxtimes_A \cF(W_3)) \ar[r]^{\mathcal{A}^A_{\cF(W_1),\cF(W_2),\cF(W_3)}} & (\cF(W_1)\boxtimes_A \cF(W_2))\boxtimes_A \cF(W_3)\\
	}
	\end{equation}
	commutes for all objects $W_1$, $W_2$, and $W_3$ in $\sC$.
\end{enumerate}

Since $\cF(\sunit)=A\boxtimes\sunit$, we can take the isomorphism $\varphi: \cF(\sunit)\rightarrow A$ to be $\sright_A$, provided this is a morphism in $\repA$. In fact, $\sright_A\circ\mu_{\cF(\sunit)}$ is the composition
\begin{equation*}
 A\boxtimes(A\boxtimes\sunit)\xrightarrow{\sA_{A,A,\sunit}} (A\boxtimes A)\boxtimes\sunit\xrightarrow{\mu\boxtimes 1_\sunit} A\boxtimes\sunit\xrightarrow{\sright_A} A.
\end{equation*}
By the naturality of the right unit, this composition equals
\begin{equation*}
  A\boxtimes(A\boxtimes\sunit)\xrightarrow{\sA_{A,A,\sunit}} (A\boxtimes A)\boxtimes\sunit\xrightarrow{\sright_{A\boxtimes A}} A\boxtimes A\xrightarrow{\mu} A.
\end{equation*}
Then since $\sright_{A\boxtimes A}\circ\sA_{A,A,\sunit}=1_A\boxtimes\sright_{A}$, this composition equals $\mu\circ(1_A\boxtimes\sright_A)$, as required.

	Next, for objects $W_1$ and $W_2$ in $\mathcal{SC}$, we need a natural isomorphism
	\begin{equation*}
	f_{W_1,W_2}: \cF(W_1\boxtimes W_2)\rightarrow \cF(W_1)\boxtimes_A \cF(W_2).
	\end{equation*}
	We can define $f_{W_1,W_2}$ by the composition
	\begin{align}
	A\boxtimes (W_1\boxtimes W_2)\xrightarrow{\sA_{A,W_1,W_2}} & (A\boxtimes W_1)\boxtimes W_2  \xrightarrow{1_{A\boxtimes W_1}\boxtimes \sleft_{W_2}^{-1}} (A\boxtimes W_1)\boxtimes (\sunit\boxtimes W_2)\label{eqn:fdef}\\
	& \xrightarrow{1_{A\boxtimes W_1}\boxtimes(\iota_A\boxtimes 1_{W_2})} (A\boxtimes W_1)\boxtimes (A\boxtimes W_2)\xrightarrow{\eta_{A\boxtimes W_1, A\boxtimes W_2}} (A\boxtimes W_1)\boxtimes_A (A\boxtimes W_2).\nonumber
	\end{align}
	We need to verify that $f_{W_1,W_2}$ so defined is a morphism in $\repA$, that is,
	\begin{align*}
	\mu_{(A\boxtimes W_1)\boxtimes_A (A\boxtimes W_2)}\circ(1_A\boxtimes f_{W_1,W_2})=f_{W_1,W_2}\circ\mu_{A\boxtimes(W_1\boxtimes W_2)}.
	\end{align*}
	This can be proved as follows:
	\begin{align*}
		\begin{matrix}
			\begin{tikzpicture}[out=up,in=down,line width=0.5pt]
				\node (a1) at (0,0) {$\scriptstyle{A}$};
				\node (a2) at (1,0) {$\scriptstyle{A}$};
				\node (w1) at (1.5,0) {$\scriptstyle{W_1}$};
				\node (w2) at (2.5,0) {$\scriptstyle{W_2}$};
				\node (l) at (2.5,0.5) {};
				\node (i) at (2,1) {$\scriptstyle{\bullet}$};
				\node (eta) at (1.75,2) [draw] {$\scriptstyle{\eta_{A\boxtimes W_1,A\boxtimes W_2}}$};
				\node (mu) at (1,3) [draw] {$\scriptstyle{\mu_{(A\boxtimes W_1)\boxtimes_A (A\boxtimes W_2)}}$};
				\node (wf) at (1,4) {$\scriptstyle{{(A\boxtimes W_1)\boxtimes_A (A\boxtimes W_2)}}$};
				\draw (a1.90) to (mu.210);
				\draw (a2.90) to (eta.200);
				\draw [line width=1pt] (w1.90) to (eta.210);
				\draw [line width=1pt] (w2.90) to (eta.340);
				\draw [dashed] (l.180) to[out=west] (i.270);
				\draw (i.90) to (eta.330);
				\draw [line width=1pt] (eta.90) to (mu.330);
				\draw [line width=1pt] (mu.90) to (wf.270);
			\end{tikzpicture}	
		\end{matrix}
		=
		\begin{matrix}
			\begin{tikzpicture}[out=up,in=down,line width=0.5pt]
				\node (a1) at (0,0) {$\scriptstyle{A}$};
				\node (a2) at (0.5,0) {$\scriptstyle{A}$};
				\node (w1) at (1,0) {$\scriptstyle{W_1}$};
				\node (w2) at (2,0) {$\scriptstyle{W_2}$};
				\node (l) at (2,0.5) {};
				\node (i) at (1.5,1) {$\scriptstyle{\bullet}$};
				\node (eta) at (1.3,3) [draw] {$\scriptstyle{\eta_{A\boxtimes W_1,A\boxtimes W_2}}$};
				\node (mu) at (0.25,2) [draw, minimum width=20pt] {$\scriptstyle{\mu}$};
				\node (wf) at (1.3,4) {$\scriptstyle{{(A\boxtimes W_1)\boxtimes_A (A\boxtimes W_2)}}$};
				\draw (a1.90) to (mu.220);
				\draw (a2.90) to (mu.320);
				\draw (mu.90) to (eta.200);
				\draw [line width=1pt] (w1.90) to (eta.210);
				\draw [line width=1pt] (w2.90) to (eta.340);
				\draw [dashed] (l.180) to[out=west] (i.270);
				\draw (i.90) to (eta.330);
				\draw [line width=1pt] (eta.90) to (wf.270);
			\end{tikzpicture}	
		\end{matrix}
		=
		\begin{matrix}
			\begin{tikzpicture}[out=up,in=down,line width=0.5pt]
				\node (a1) at (0,0) {$\scriptstyle{A}$};
				\node (a2) at (0.5,0) {$\scriptstyle{A}$};
				\node (w1) at (1,0) {$\scriptstyle{W_1}$};
				\node (w2) at (2,0) {$\scriptstyle{W_2}$};
				\node (l) at (2,1.5) {};
				\node (i) at (1.5,2) {$\scriptstyle{\bullet}$};
				\node (eta) at (1.3,3) [draw] {$\scriptstyle{\eta_{A\boxtimes W_1,A\boxtimes W_2}}$};
				\node (mu) at (0.25,1) [draw, minimum width=20pt] {$\scriptstyle{\mu}$};
				\node (wf) at (1.3,4) {$\scriptstyle{{(A\boxtimes W_1)\boxtimes_A (A\boxtimes W_2)}}$};
				\draw (a1.90) to (mu.220);
				\draw (a2.90) to (mu.320);
				\draw (mu.90) to (eta.200);
				\draw [line width=1pt] (w1.90) to (eta.210);
				\draw [line width=1pt] (w2.90) to (eta.340);
				\draw [dashed] (l.180) to[out=west] (i.270);
				\draw (i.90) to (eta.330);
				\draw [line width=1pt] (eta.90) to (wf.270);
			\end{tikzpicture}	
		\end{matrix}
	\end{align*}
	The detailed argument frequently uses properties of associativity isomorphisms, and in the first equality, we have used $\mu^{(1)}$ to define $\mu_{(A\boxtimes W_1)\boxtimes_A (A\boxtimes W_2)}$.
	
	Now we prove that the $\repA$-morphisms $f_{W_1,W_2}$ determine a natural transformation from $\cF\circ\boxtimes$ to $\boxtimes_A\circ(\cF\times \cF)$, that is, we show that if $f_1: W_1\rightarrow \widetilde{W}_1$ and $f_2: W_2\rightarrow \widetilde{W}_2$ are morphisms in $\sC$, then
	\begin{equation*}
	(\cF(f_1)\boxtimes_A \cF(f_2))\circ f_{W_1,W_2}=f_{\widetilde{W}_1,\widetilde{W}_2}\circ \cF(f_1\boxtimes f_2).
	\end{equation*}
	This follows using these diagrams: 
	\begin{align*}
		\begin{matrix}
			\begin{tikzpicture}[out=up,in=down,line width=0.5pt]
				\node (a) at (0,0) {$\scriptstyle{A}$};
				\node (w1) at (0.75,0) {$\scriptstyle{W_1}$};
				\node (w2) at (2,0) {$\scriptstyle{W_2}$};
				\node (l) at (2,0.5) {};
				\node (i) at (1.25,1) {$\scriptstyle{\bullet}$};
				\node (eta) at (1, 2) [draw] {$\scriptstyle{\eta_{A\boxtimes W_1,A\boxtimes W_2}}$};
				\node (f) at (1, 3) [draw] {$\scriptstyle{\cF(f_1)\boxtimes_A\cF(f_2) }$};
				\node (wf) at (1,4) {$\scriptstyle{\widetilde{W}_1\boxtimes_A\widetilde{W}_2}$};
				\draw (a.90) to (eta.200);
				\draw [line width=1pt] (w1.90) to (eta.210);
				\draw (i.90) to (eta.330);
				\draw [line width=1pt] (w2.90) to (eta.340);
				\draw [dashed] (l.180) to[out=west] (i.270);
				\draw [line width=1pt] (eta.90) to (f.270);
				\draw [line width=1pt] (f.90) to (wf.270);
			\end{tikzpicture}
		\end{matrix}		
		=
		\begin{matrix}
			\begin{tikzpicture}[out=up,in=down,line width=0.5pt]
				\node (a) at (0,0) {$\scriptstyle{A}$};
				\node (w1) at (0.75,0) {$\scriptstyle{W_1}$};
				\node (w2) at (2,0) {$\scriptstyle{W_2}$};
				\node (l) at (2,0.5) {};
				\node (i) at (1.35,1) {$\scriptstyle{\bullet}$};
				\node (eta) at (1, 3) [draw] {$\scriptstyle{\eta_{A\boxtimes \widetilde{W}_1,A\boxtimes \widetilde{W}_2}}$};
				\node (f1) at (0.75, 1.75) [draw] {$\scriptstyle{f_1}$};
				\node (f2) at (2, 1.75) [draw] {$\scriptstyle{f_2}$};
				\node (wf) at (1,4) {$\scriptstyle{\widetilde{W}_1\boxtimes_A\widetilde{W}_2}$};
				\draw (a.90) to (eta.200);
				\draw [line width=1pt] (w1.90) to (f1.270);
				\draw [line width=1pt] (f1.90) to (eta.210);
				\draw (i.90) to (eta.330);
				\draw [line width=1pt] (w2.90) to (f2.270);
				\draw [line width=1pt] (f2.90) to (eta.340);
				\draw [dashed] (l.180) to[out=west] (i.270);
				\draw [line width=1pt] (f.90) to (wf.270);
			\end{tikzpicture}
		\end{matrix}		
		=
		\begin{matrix}
			\begin{tikzpicture}[out=up,in=down,line width=0.5pt]
				\node (a) at (0,0) {$\scriptstyle{A}$};
				\node (w1) at (0.75,0) {$\scriptstyle{W_1}$};
				\node (w2) at (2,0) {$\scriptstyle{W_2}$};
				\node (l) at (2,1.5) {};
				\node (i) at (1.35,2) {$\scriptstyle{\bullet}$};
				\node (eta) at (1, 3) [draw] {$\scriptstyle{\eta_{A\boxtimes \widetilde{W}_1,A\boxtimes \widetilde{W}_2}}$};
				\node (f1) at (0.75, 0.8) [draw] {$\scriptstyle{f_1}$};
				\node (f2) at (2, 0.8) [draw] {$\scriptstyle{f_2}$};
				\node (wf) at (1,4) {$\scriptstyle{\widetilde{W}_1\boxtimes_A\widetilde{W}_2}$};
				\draw (a.90) to (eta.200);
				\draw [line width=1pt] (w1.90) to (f1.270);
				\draw [line width=1pt] (f1.90) to (eta.210);
				\draw (i.90) to (eta.330);
				\draw [line width=1pt] (w2.90) to (f2.270);
				\draw [line width=1pt] (f2.90) to (eta.340);
				\draw [dashed] (l.180) to[out=west] (i.270);
				\draw [line width=1pt] (f.90) to (wf.270);
			\end{tikzpicture}
		\end{matrix}		
	\end{align*}
In the second equality, we have used naturality of the left unit isomorphism and the evenness of $\iota_A$  and $l_{W_2}$ to exchange their order with $\cF(f_1)\boxtimes \cF(f_2)$.
	
	Next, we show that each $f_{W_1,W_2}$ is an isomorphism by constructing an inverse. We define $\widetilde{g}_{W_1,W_2}$ to be the composition
	\begin{align*}
	(A\boxtimes W_1)\boxtimes (A\boxtimes W_2) & \xrightarrow{assoc.} (A\boxtimes (W_1\boxtimes A))\boxtimes W_2\xrightarrow{(1_A\boxtimes \sR_{A,W_1}^{-1})\boxtimes 1_{W_2}} (A\boxtimes (A\boxtimes W_1))\boxtimes W_2\nonumber\\
	&\xrightarrow{assoc.} (A\boxtimes A)\boxtimes (W_1\boxtimes W_2)\xrightarrow{\mu\boxtimes 1_{W_1\boxtimes W_2}} A\boxtimes (W_1\boxtimes W_2),
	\end{align*}
	where the arrows marked $assoc.$ refer to any composition of associativity isomorphisms between the source and target of the arrows. If we can show that
	\begin{equation}\label{finversewelldef}	\widetilde{g}_{W_1,W_2}\circ\mu^{(1)}=\widetilde{g}_{W_1,W_2}\circ\mu^{(2)}: A\boxtimes((A\boxtimes W_1)\boxtimes (A\boxtimes W_2))\rightarrow A\boxtimes (W_1\boxtimes W_2),
	\end{equation}
	then the universal property of the cokernel $((A\boxtimes W_1)\boxtimes_A (A\boxtimes W_2), \eta_{A\boxtimes W_1, A\boxtimes W_2})$ implies that there is a unique morphism 
	\begin{equation*}
	g_{W_1,W_2}: (A\boxtimes W_1)\boxtimes_A (A\boxtimes W_2)\rightarrow A\boxtimes (W_1\boxtimes W_2).
	\end{equation*}
	such that $g_{W_1,W_2}\circ\eta_{A\boxtimes W_1, A\boxtimes W_2}=\widetilde{g}_{W_1,W_2}$. To show \eqref{finversewelldef}, we use the following diagrams, starting with $\widetilde{g}_{W_1,W_2}\circ\mu^{(2)}$ on the left and ending with $\widetilde{g}_{W_1,W_2}\circ\mu^{(1)}$ on the right:
	\begin{align*}
	\begin{matrix}
		\begin{tikzpicture}[scale=0.7, out=up, in=down, line width=0.5pt]
			\node at (0,-0.3) {$\scriptstyle{A}$};
			\node at (1,-0.3) {$\scriptstyle{A}$};
			\node at (1.6,-0.3) {$\scriptstyle{W_1}$};
			\node at (2.4, -0.3) {$\scriptstyle{A}$};
			\node at (3.1, -0.3) {$\scriptstyle{W_2}$};
			\node at (0.5,6.3) {$\scriptstyle{A}$};
			\node at (2.4,6.3) {$\scriptstyle{W_1}$};
			\node at (3.1,6.3) {$\scriptstyle{W_2}$};
			\node (U) at (0.5,5.3) [draw,minimum width=30pt,minimum height=8pt] {$\scriptstyle{\mu}$};
			\node (D) at (2,3) [draw,minimum width=30pt,minimum height=8pt] {$\scriptstyle{\mu}$};
			\draw[] (1,0) to (0,2) to (0,3.5) to (U.210);
			\draw[line width=1pt] (1.5,0) to (0.5,2) to (0.5,3) to (2.5,5) to (2.5,6);
			\draw[white, line width=3pt, double=black, double distance=0.5pt] (0,0) to (D.210);
			\draw[] (2.5,0)
			to (D.330);
			\draw[line width=1pt] (3,0) to (3,6);
			\draw[white, line width=3pt, double=black, double distance=0.5pt] (D.north) to (U.330);
			\draw[] (U.north) to (0.5,6);
		\end{tikzpicture}
	\end{matrix}
	=
	\begin{matrix}
		\begin{tikzpicture}[scale=0.7, out=up, in=down, line width=0.5pt]
			\node at (0,-0.3) {$\scriptstyle{A}$};
			\node at (1,-0.3) {$\scriptstyle{A}$};
			\node at (1.9,-0.3) {$\scriptstyle{W_1}$};
			\node at (2.5, -0.3) {$\scriptstyle{A}$};
			\node at (3.1, -0.3) {$\scriptstyle{W_2}$};
			\node at (.5,6.3) {$\scriptstyle{A}$};
			\node at (2.4,6.3) {$\scriptstyle{W_1}$};
			\node at (3.1,6.3) {$\scriptstyle{W_2}$};
			\node (U) at (0.5,5.3) [draw,minimum width=30pt,minimum height=8pt] {$\scriptstyle{\mu}$};
			\node (D) at (1.5,4) [draw,minimum width=30pt,minimum height=8pt] {$\scriptstyle{\mu}$};
			\draw[] (1,0)
			to (0,2)
			to (0,3.5)
			to (U.210);
			\draw[line width=1pt] (2,0)
			to (0.5,2)
			to (2.5,3.3)
			to (2.5,5)
			to (2.5,6);
			\draw[white, line width=3pt, double=black, double distance=0.5pt] (0,0)
			to (1,2.8)
			to (D.210);
			\draw[white, line width=3pt, double=black, double distance=0.5pt] (2.5,0)
			to (D.330);
			\draw[line width=1pt] (3,0)
			to (3,6);
			\draw[] (D.north)
			to (U.330);
			\draw[] (U.north)
			to (0.5,6);
			\draw[fill=white] (3.2,0)
			to (3.2,0);
		\end{tikzpicture}
	\end{matrix}
	=
	\begin{matrix}
		\begin{tikzpicture}[scale=0.7, out=up, in=down, line width=0.5pt]			
			\node at (0,-0.3) {$\scriptstyle{A}$};
			\node at (1,-0.3) {$\scriptstyle{A}$};
			\node at (1.9,-0.3) {$\scriptstyle{W_1}$};
			\node at (2.5, -0.3) {$\scriptstyle{A}$};
			\node at (3.1, -0.3) {$\scriptstyle{W_2}$};
			\node at (1,6.3) {$\scriptstyle{A}$};
			\node at (2.4,6.3) {$\scriptstyle{W_1}$};
			\node at (3.1,6.3) {$\scriptstyle{W_2}$};
			\node (U) at (1,5.3) [draw,minimum width=30pt,minimum height=8pt] {$\scriptstyle{\mu}$};
			\node (D) at (0.5,4) [draw,minimum width=30pt,minimum height=8pt] {$\scriptstyle{\mu}$};
			\draw[] (1,0)
			to (0,2)
			to (0,3.5)
			to (D.210);
			\draw[line width=1pt] (2,0)
			to (0.5,2)
			to (2.5,3.3)
			to (2.5,5)
			to (2.5,6);
			\draw[white, line width=3pt, double=black, double distance=0.5pt] (0,0)
			to (1,2.8)
			to (D.330);
			\draw[white, line width=3pt, double=black, double distance=0.5pt] (2.5,0)
			to (2.5,1)
			to (U.330);
			\draw[line width=1pt] (3,0)
			to (3,6);
			\draw[] (D.north)
			to (U.210);
			\draw[] (U.north)
			to (1,6);
		\end{tikzpicture}
	\end{matrix}
	=
	\begin{matrix}
		\begin{tikzpicture}[scale=0.7, out=up, in=down, line width=0.5pt]	
			\node at (0,-0.3) {$\scriptstyle{A}$};
			\node at (0.9,-0.3) {$\scriptstyle{A}$};
			\node at (1.6,-0.3) {$\scriptstyle{W_1}$};
			\node at (2.4, -0.3) {$\scriptstyle{A}$};
			\node at (3.1, -0.3) {$\scriptstyle{W_2}$};
			\node at (1,6.3) {$\scriptstyle{A}$};
			\node at (2.4,6.3) {$\scriptstyle{W_1}$};
			\node at (3.1,6.3) {$\scriptstyle{W_2}$};
			\node (U) at (1,5.3) [draw,minimum width=30pt,minimum height=8pt, fill=white] {$\scriptstyle{\mu}$};
			\node (D) at (0.5,4.2) [draw,minimum width=30pt,minimum height=8pt, fill=white] {$\scriptstyle{\mu}$};
			\draw[] (1,0)
			to (1,1.5)
			to (0,3)
			to (D.210);
			\draw[line width=1pt] (1.5,0)
			to (2.5,4)
			to (2.5,6);
			\draw[white, line width=3pt, double=black, double distance=0.5pt] (0,0)
			to (0,1.5)
			to (1,3)
			to (D.330);
			\draw[white, line width=3pt, double=black, double distance=0.5pt](2.5,0)
			to (U.330);
			\draw[line width=1pt](3,0)
			to (3,6);
			\draw[] (D.north)
			to (U.210);
			\draw[] (U.north)
			to (1,6);
		\end{tikzpicture}
	\end{matrix}
	=
	\begin{matrix}
		\begin{tikzpicture}[scale=0.7, out=up, in=down, line width=0.5pt]		
			\node at (0,-0.3) {$\scriptstyle{A}$};
			\node at (0.9,-0.3) {$\scriptstyle{A}$};
			\node at (1.6,-0.3) {$\scriptstyle{W_1}$};
			\node at (2.4, -0.3) {$\scriptstyle{A}$};
			\node at (3.1, -0.3) {$\scriptstyle{W_2}$};
			\node at (1,6.3) {$\scriptstyle{A}$};
			\node at (2.4,6.3) {$\scriptstyle{W_1}$};
			\node at (3.1,6.3) {$\scriptstyle{W_2}$};
			\node (U) at (1,5.3) [draw,minimum width=30pt,minimum height=8pt, fill=white] {$\scriptstyle{\mu}$};
			\node (D) at (0.5,1.5) [draw,minimum width=30pt,minimum height=8pt, fill=white] {$\scriptstyle{\mu}$};
			\draw[] (0,0)
			to (D.210);
			\draw[] (1,0)
			to (D.330);
			\draw[line width=1pt] (1.5,0)
			to (1.5,3.5)
			to (2.5,5.3)
			to (2.5,6);
			\draw[white, line width=3pt, double=black, double distance=0.5pt](2.5,0)
			to (2.5,3.5)
			to (U.330);
			\draw[line width=1pt](3,0)
			to (3,6);
			\draw[] (D.north)
			to (U.210);
			\draw[] (U.north)
			to (1,6);
		\end{tikzpicture}
	\end{matrix}
\end{align*}

  Next, we show that $g_{W_1,W_2}$ is the inverse of $f_{W_1,W_2}$.
  To show $g_{W_1,W_2}\circ f_{W_1,W_2}=1_{A\boxtimes(W_1\boxtimes W_2)}$:
	\begin{align*}
	\begin{matrix}
		\begin{tikzpicture}[scale=0.7, out=up, in=down, line width=0.5pt=0.75]
			\node at (0,-0.3) {$\scriptstyle{A}$};
			\node at (1,-0.3) {$\scriptstyle{W_1}$};
			\node at (2.5,-0.3) {$\scriptstyle{W_2}$};
			\node at (0.5, 3.8) {$\scriptstyle{A}$};
			\node at (1.9, 3.8) {$\scriptstyle{W_1}$};
			\node at (2.6, 3.8) {$\scriptstyle{W_2}$};
			\node (U) at (0.5,3) [draw,minimum width=30pt,minimum height=8pt, fill=white] {$\scriptstyle{\mu}$};
			\node (alg) at (2,1.5)  {$\bullet$};
			\draw[](0,0) to (U.210);
			\draw[line width=1pt](1,0) to (1,1.5) to (2,3.5);
			\draw[line width=1pt](2.5,0) to (2.5,3.5);
			\draw[dashed] (2.5,0.5) to [out=left, in=down] (2,1.5);
			\draw[white, line width=3pt, double=black, double distance=0.5pt](2,1.5) to (U.330);
			\draw[](U.north) to (0.5,3.5);
			\node at (2,1.5)  {$\bullet$};
		\end{tikzpicture}
	\end{matrix}
	\,=\,
	\begin{matrix}
		\begin{tikzpicture}[scale=0.7, out=up, in=down, line width=0.5pt]
			\node at (0,-0.3) {$\scriptstyle{A}$};
			\node at (1,-0.3) {$\scriptstyle{W_1}$};
			\node at (2.5,-0.3) {$\scriptstyle{W_2}$};
			\node at (0.5, 3.8) {$\scriptstyle{A}$};
			\node at (1.9, 3.8) {$\scriptstyle{W_1}$};
			\node at (2.6, 3.8) {$\scriptstyle{W_2}$};
			\node (U) at (0.5,3) [draw,minimum width=30pt,minimum height=8pt, fill=white] {$\scriptstyle{\mu}$};
			\node (alg) at (2,1.5) {$\bullet$};
			\draw[](0,0) to (U.210);
			\draw[line width=1pt](1,0) to (1,1.5) to (2,3.5);
			\draw[line width=1pt](2.5,0) to (2.5,3.5);
			\draw[dashed] (1,0.5) to [out=right, in=down] (2,1.5);
			\draw[white, line width=3pt, double=black, double distance=0.5pt](2,1.5) to (U.330);
			\draw[](U.north) to  (0.5,3.5);
			\node  at (2,1.5) {$\bullet$};
		\end{tikzpicture}
	\end{matrix}
	=
	\begin{matrix}
		\begin{tikzpicture}[scale=0.7, out=up, in=down, line width=0.5pt]
			\node at (0,-0.3) {$\scriptstyle{A}$};
			\node at (1,-0.3) {$\scriptstyle{W_1}$};
			\node at (2.5,-0.3) {$\scriptstyle{W_2}$};
			\node at (0.5, 3.8) {$\scriptstyle{A}$};
			\node at (1.9, 3.8) {$\scriptstyle{W_1}$};
			\node at (2.6, 3.8) {$\scriptstyle{W_2}$};
			\node (U) at (0.5,3) [draw,minimum width=30pt,minimum height=8pt, fill=white] {$\scriptstyle{\mu}$};
			\node (alg) at (1,2.2) {$\bullet$};
			\draw[](0,0) to (U.210);
			\draw[line width=1pt](1,0) to (1,1) to (2,2) to (2,3.5);
			\draw[line width=1pt](2.5,0) to (2.5,3.5);
			\draw[white, line width=6pt] (1,0.5) to[out=right] (2,1.1) to (1,2.2);
			\draw[dashed] (1,0.5) to[out=right] (2,1.1) to (1,2.2);
			\draw[](1,2.2) to (U.330);
			\draw[](U.north) to (0.5,3.5);
			\draw[fill=white](2.6,0) to (2.6,0);
			\node (alg) at (1,2.2) {$\bullet$};
		\end{tikzpicture}
	\end{matrix}
	=
	\begin{matrix}
		\begin{tikzpicture}[scale=0.7, out=up, in=down, line width=0.5pt]
			\node at (0,-0.3) {$\scriptstyle{A}$};
			\node at (0.6,-0.3) {$\scriptstyle{W_1}$};
			\node at (2.1,-0.3) {$\scriptstyle{W_2}$};
			\node at (0, 3.8) {$\scriptstyle{A}$};
			\node at (1.4, 3.8) {$\scriptstyle{W_1}$};
			\node at (2.1, 3.8) {$\scriptstyle{W_2}$};
			\node (alg) at (0,2.2) {$\bullet$};
			\draw[](0,0) to (0,3.5);
			\draw[line width=1pt](0.5,0) to (0.5,1) to (1.5,2) to (1.5,3.5);
			\draw[line width=1pt](2,0) to (2,3.5);
			\draw[white, line width=6pt] (0.5,0.5) to [out=right,in=down] (1.5,1.1) to [out=up, in=right] (0,2.2);
			\draw[dashed] (0.5,0.5) to [out=right,in=down] (1.5,1.1) to [out=up, in=right] (0,2.2);
			\node  at (0,2.2) {$\bullet$};
		\end{tikzpicture}
	\end{matrix}
	=
	\begin{matrix}
		\begin{tikzpicture}[scale=0.7, out=up, in=down, line width=0.5pt]
			\node at (0,-0.3) {$\scriptstyle{A}$};
			\node at (1,-0.3) {$\scriptstyle{W_1}$};
			\node at (2,-0.3) {$\scriptstyle{W_2}$};
			\node at (0, 3.8) {$\scriptstyle{A}$};
			\node at (1, 3.8) {$\scriptstyle{W_1}$};
			\node at (2, 3.8) {$\scriptstyle{W_2}$};
			\draw[](0,0) to (0,3.5);
			\draw[line width=1pt](1,0) to (1,3.5);
			\draw[line width=1pt](2,0) to (2,3.5);
			\draw[dashed] (1,0.5) to (0,2.2);
		\end{tikzpicture}
	\end{matrix}
	=
	\begin{matrix}
		\begin{tikzpicture}[scale=0.7, out=up, in=down, line width=0.5pt]
			\node at (0,-0.3) {$\scriptstyle{A}$};
			\node at (1,-0.3) {$\scriptstyle{W_1}$};
			\node at (2,-0.3) {$\scriptstyle{W_2}$};
			\node at (0, 3.8) {$\scriptstyle{A}$};
			\node at (1, 3.8) {$\scriptstyle{W_1}$};
			\node at (2, 3.8) {$\scriptstyle{W_2}$};
			\draw[](0,0) to (0,3.5);
			\draw[line width=1pt](1,0) to (1,3.5);
			\draw[line width=1pt](2,0) to (2,3.5);
			\draw[fill=white](2.2,0) to (2.2,0);
		\end{tikzpicture}
	\end{matrix}
	\end{align*}
	Now we consider $f_{W_1,W_2}\circ g_{W_1,W_2}$; note that
	\begin{equation*}
	 f_{W_1,W_2}\circ g_{W_1,W_2}\circ\eta_{A\boxtimes W_1, A\boxtimes W_2}=\eta_{A\boxtimes W_1, A\boxtimes W_2}\circ I,
	\end{equation*}
where $I$ is the endomorphism of $(A\boxtimes W_1)\boxtimes (A\boxtimes W_2)$ given by the composition
	\begin{align*}
	(A\boxtimes W_1) & \boxtimes(A\boxtimes W_2)\xrightarrow{assoc.}(A\boxtimes(W_1\boxtimes A))\boxtimes W_2\xrightarrow{(1_A\boxtimes \sR_{A,W_1}^{-1})\boxtimes 1_{W_2}}(A\boxtimes(A\boxtimes W_1))\boxtimes W_2\nonumber\\
	& \xrightarrow{assoc.}(A\boxtimes A)\boxtimes(W_1\boxtimes W_2)\xrightarrow{\mu\boxtimes 1_{W_1\boxtimes W_2}} A\boxtimes (W_1\boxtimes W_2)\xrightarrow{\sA_{A,W_1,W_2}} (A\boxtimes W_1)\boxtimes W_2\nonumber\\
	&\xrightarrow{1_{A\boxtimes W_1}\boxtimes((\iota_A\boxtimes 1_{W_2})\circ \sleft_{W_2}^{-1})}(A\boxtimes W_1)\boxtimes(A\boxtimes W_2).
	\end{align*}
Thus by the surjectivity of $\eta_{A\boxtimes W_1, A\boxtimes W_2}$, it is enough to prove that $\eta_{A\boxtimes W_1, A\boxtimes W_2}\circ I=\eta_{A\boxtimes W_1, A\boxtimes W_2}$, or
\begin{equation*}
 \eta_{A\boxtimes W_1, A\boxtimes W_2}\circ(I-1_{(A\boxtimes W_1)\boxtimes (A\boxtimes W_2)})=0.
\end{equation*}
For this, it is in turn enough to prove 
\begin{align}
	I\circ J &= \mu^{(1)}\circ(1_A\boxtimes(1_{A\boxtimes W_1}\boxtimes(\iota_A\boxtimes 1_{W_2})))\label{Icomp},\\
	J&=\mu^{(2)}\circ(1_A\boxtimes(1_{A\boxtimes W_1}\boxtimes(\iota_A\boxtimes 1_{W_2})))\label{Idcomp},
\end{align}
where $J:A\boxtimes ((A\boxtimes W_1)\boxtimes(\sunit\boxtimes W_2)) \rightarrow  (A\boxtimes W_1)\boxtimes(A\boxtimes W_2)$ is the isomorphism
\begin{align*}
	J:	A\boxtimes & ((A\boxtimes W_1)\boxtimes(\sunit\boxtimes W_2))\xrightarrow{\sA_{A,A\boxtimes W_1,\sunit\boxtimes W_2}}(A\boxtimes(A\boxtimes W_1))\boxtimes(\sunit\boxtimes W_2)\nonumber\\
	&\xrightarrow{\sR_{A,A\boxtimes W_1}\boxtimes \sleft_{W_2}}((A\boxtimes W_1)\boxtimes A)\boxtimes W_2\xrightarrow{\sA_{A\boxtimes W_1,A,W_2}^{-1}} (A\boxtimes W_1)\boxtimes(A\boxtimes W_2).
\end{align*}
	For \eqref{Icomp} we proceed as follows:	
	\begin{align*}
	\begin{matrix}
		\begin{tikzpicture}[scale=0.9, out=up, in=down, line width=0.5pt]
			\node at (0,-0.3) {$\scriptstyle{A}$};
			\node at (0.9,-0.3) {$\scriptstyle{A}$};
			\node at (1.6,-0.3) {$\scriptstyle{W_1}$};
			\node at (2.2,-0.3) {$\scriptstyle{\sunit}$};
			\node at (3,-0.3) {$\scriptstyle{W_2}$};
			\node at (0.5, 4) {$\scriptstyle{A}$};
			\node at (1.5, 4) {$\scriptstyle{W_1}$};
			\node at (2.2, 4) {$\scriptstyle{A}$};
			\node at (3, 4) {$\scriptstyle{W_2}$};
			\node (U) at (0.5,3) [draw,minimum width=30pt,minimum height=8pt] {$\scriptstyle{\mu}$};
			\node (alg) at (2.2,3)  {$\bullet$};
			\draw[] (1,0) to (0,1.5) to (U.210);
			\draw[line width=1pt] (1.5,0) to (0.8,1.5) to (1.5,3) to (1.5,3.8);
			\draw[line width=1pt] (3,0) to (3,3.8);
			\draw[dashed](2.2,0) to [out=up, in=left] (3,1);
			\draw[dashed](3,2) to [out=left, in=down] (2.2,3);
			\draw[](2.2,3) to (2.2,3.8);
			\draw[](U.north) to (0.5,3.8);
			\draw[white, line width=3pt, double=black,double  distance=0.5pt] (0,0) to (1.7,1.5) to (U.330);
		\end{tikzpicture}
	\end{matrix}
	\,\,=\,\,
	\begin{matrix}
		\begin{tikzpicture}[scale=0.9, out=up, in=down, line width=0.5pt]
			\node at (0,-0.3) {$\scriptstyle{A}$};
			\node at (0.9,-0.3) {$\scriptstyle{A}$};
			\node at (1.6,-0.3) {$\scriptstyle{W_1}$};
			\node at (2.2,-0.3) {$\scriptstyle{\sunit}$};
			\node at (3,-0.3) {$\scriptstyle{W_2}$};
			\node at (0.5, 4) {$\scriptstyle{A}$};
			\node at (1.5, 4) {$\scriptstyle{W_1}$};
			\node at (2.2, 4) {$\scriptstyle{A}$};
			\node at (3, 4) {$\scriptstyle{W_2}$};
			\node (U) at (0.5,3) [draw,minimum width=30pt,minimum height=8pt] {$\scriptstyle{\mu}$};
			\node (alg) at (2.2,3)  {$\bullet$};
			\draw[] (1,0) to (0,1.5) to (U.210);
			\draw[line width=1pt] (1.5,0) to (1.5,3.8);
			\draw[line width=1pt] (3,0) to (3,3.8);
			\draw[dashed](2.2,0) to (2.2,3);
			\draw[](2.2,3) to (2.2,3.8);
			\draw[](U.north) to (0.5,3.8);
			\draw[white, line width=3pt, double=black, double distance=0.5pt] (0,0) to (1,1.5) to (U.330);
		\end{tikzpicture}
	\end{matrix}
	\,\,=\,\,
	\begin{matrix}
		\begin{tikzpicture}[scale=0.9, out=up, in=down, line width=0.5pt]
			\node at (0,-0.3) {$\scriptstyle{A}$};
			\node at (0.9,-0.3) {$\scriptstyle{A}$};
			\node at (1.6,-0.3) {$\scriptstyle{W_1}$};
			\node at (2.2,-0.3) {$\scriptstyle{\sunit}$};
			\node at (3,-0.3) {$\scriptstyle{W_2}$};
			\node at (0.5, 4) {$\scriptstyle{A}$};
			\node at (1.5, 4) {$\scriptstyle{W_1}$};
			\node at (2.2, 4) {$\scriptstyle{A}$};
			\node at (3, 4) {$\scriptstyle{W_2}$};
			\node (U) at (0.5,3) [draw,minimum width=30pt,minimum height=8pt] {$\scriptstyle{\mu}$};
			\node (alg) at (2.2,1)  {$\bullet$};
			\draw[] (0,0) to (U.210);
			\draw[] (1,0) to (U.330);
			\draw[line width=1pt] (1.5,0) to (1.5,3.8);
			\draw[line width=1pt] (3,0) to (3,3.8);
			\draw[dashed](2.2,0) to (2.2,1);
			\draw[](2.2,1) to (2.2,3.8);
			\draw[](U.north) to (0.5,3.8);
		\end{tikzpicture}
	\end{matrix}
\end{align*}
On the other hand, $\mu^{(2)}\circ(1_A\boxtimes(1_{A\boxtimes W_1}\boxtimes(\iota_A\boxtimes 1_{W_2})))$
can be handled with the diagrams:
	\begin{align*}
	\begin{matrix}
		\begin{tikzpicture}[scale=0.85, out=up, in=down, line width=0.5pt]
			\node at (0,-0.3) {$\scriptstyle{A}$};
			\node at (0.9,-0.3) {$\scriptstyle{A}$};
			\node at (1.5,-0.3) {$\scriptstyle{W_1}$};
			\node at (2,-0.3) {$\scriptstyle{\sunit}$};
			\node at (2.6,-0.3) {$\scriptstyle{W_2}$};
			\node at (0,3.8) {$\scriptstyle{A}$};
			\node at (0.5,3.8) {$\scriptstyle{W_1}$};
			\node at (1.6,3.8) {$\scriptstyle{A}$};
			\node at (2.5,3.8) {$\scriptstyle{W_2}$};
			\node (U) at (1.6,2.3) [draw,minimum width=25pt,minimum height=8pt, fill=white] {$\scriptstyle{\mu}$};
			\node (alg) at (2,0.5)  {$\bullet$};
			\draw[](2,0.5) to (2,1) to (U.330);
			\draw[line width=1pt](2.5,0) to (2.5,3.5);
			\draw[dashed](2,0) to (2,0.5);
			\draw[](U.north) to (1.6,3.5);
			\draw[] (1,0) to (0,2) to (0,3.5);
			\draw[line width=1pt] (1.5,0) to (0.5,2) to (0.5,3.5);
			\draw[white, line width=3pt, double=black, double distance=0.5pt] (0,0) to (U.210);
		\end{tikzpicture}
	\end{matrix}
	\,=\,
	\begin{matrix}
		\begin{tikzpicture}[scale=0.85, out=up, in=down, line width=0.5pt]
			\node at (0.4,-0.3) {$\scriptstyle{A}$};
			\node at (0.9,-0.3) {$\scriptstyle{A}$};
			\node at (1.5,-0.3) {$\scriptstyle{W_1}$};
			\node at (2.3,-0.3) {$\scriptstyle{\sunit}$};
			\node at (3,-0.3) {$\scriptstyle{W_2}$};
			\node at (0.4,3.8) {$\scriptstyle{A}$};
			\node at (1,3.8) {$\scriptstyle{W_1}$};
			\node at (1.6,3.8) {$\scriptstyle{A}$};
			\node at (3,3.8) {$\scriptstyle{W_2}$};
			\draw[line width=1pt](2.8,0) to (2.8,3.5);
			\draw[dashed](2.3,0) to (2.3,0.5)
			to [out=up,in=right] (1.6,3);
			\draw[] (1,0) to (1,1) to (0.5,2.8) to (0.5,3.5);
			\draw[line width=1pt] (1.5,0) to (1.5,1) to (1,2.8) to (1,3.5);
			\draw[white, line width=3pt, double=black, double distance=0.5pt] (0.5,0) to (0.5,1) to (1.6,2.8) to (1.6,3.5);
		\end{tikzpicture}
	\end{matrix}
	\,=\,
	\begin{matrix}
		\begin{tikzpicture}[scale=0.85, out=up, in=down, line width=0.5pt]
			\node at (0.4,-0.3) {$\scriptstyle{A}$};
			\node at (0.9,-0.3) {$\scriptstyle{A}$};
			\node at (1.5,-0.3) {$\scriptstyle{W_1}$};
			\node at (2,-0.3) {$\scriptstyle{\sunit}$};
			\node at (2.6,-0.3) {$\scriptstyle{W_2}$};
			\node at (0.4,3.8) {$\scriptstyle{A}$};
			\node at (1,3.8) {$\scriptstyle{W_1}$};
			\node at (1.7,3.8) {$\scriptstyle{A}$};
			\node at (2.6,3.8) {$\scriptstyle{W_2}$};		
			\draw[line width=1pt](2.5,0) to (2.5,3.5);
			\draw[dashed](2,0) to (2,0.5) to [out=up,in=left] (2.5,1);
			\draw[] (1,0) to (1,1.5) to (0.5,3.5);
			\draw[line width=1pt] (1.5,0) to (1.5,1.5) to (1,3.5);
			\draw[white, line width=3pt, double=black, double distance=0.5pt] (0.5,0) to (0.5,1.5) to (1.8,3.5);
		\end{tikzpicture}
	\end{matrix}
\end{align*}
This completes the proof that $f_{W_1,W_2}\circ g_{W_1,W_2}=1_{(A\boxtimes W_1)\boxtimes_A (A\boxtimes W_2)}$, and thus that the $\repA$-morphisms $f_{W_1,W_2}$ define an (even) natural isomorphism $f :\cF\circ\boxtimes\rightarrow\boxtimes_A\circ(\cF\times \cF)$.
	
	Now we show that $f$ and $\varphi$ are compatible with the unit isomorphisms.
	To prove \eqref{Funitcompat}, we establish that for any object $W$ of $\mathcal{SC}$,
	\begin{align*}
	(l^A_{A\boxtimes W})\circ (\varphi\boxtimes_A 1_{A\boxtimes W})\circ f_{\sunit,W}&=\cF(\sleft_W): A\boxtimes(\sunit\boxtimes W)\rightarrow A\boxtimes W,\\
	(r^A_{A\boxtimes W})\circ(1_{A\boxtimes W}\boxtimes_A\varphi)\circ f_{W,\sunit}&=\cF(\sright_W): A\boxtimes(W\boxtimes\sunit)\rightarrow A\boxtimes W.
	\end{align*}
	The first follows from:
	\begin{align*}
		\begin{matrix}
			\begin{tikzpicture}[out=up,in=down,line width=0.5pt]
				\node (a) at (0,0) {$\scriptstyle{A}$};
				\node (u) at (0.5,0) {$\scriptstyle{\sunit}$};
				\node (w) at (1.5,0) {$\scriptstyle{W}$};
				\node (l) at (1.5,0.5) {};
				\node (i) at (1,1) {$\scriptstyle{\bullet}$};
				\node (eta) at (0.75,2) [draw] {$\scriptstyle{\eta_{A\boxtimes\sunit,A\boxtimes W}}$};
				\node (phi) at (0.75,3) [draw] {$\scriptstyle{\varphi\boxtimes_A\,1_{A\boxtimes W}}$};
				\node (le) at (0.75,4) [draw] {$\scriptstyle{l^A_{A\boxtimes W}}$};
				\node (wf) at (0.75, 5) {$\scriptstyle{A\boxtimes W}$};
				\draw (a.90) to (eta.200);
				\draw [dashed] (u.90) to (eta.210);
				\draw [dashed] (l.west) to[out=west] (i.south);
				\draw (i.90) to (eta.330);
				\draw [line width=1 pt] (w.90) to (eta.340);
				\draw [line width=1 pt] (eta.90) to (phi.270);
				\draw [line width=1 pt] (phi.90) to (le.270);
				\draw [line width=1 pt] (le.90) to (wf.270);
			\end{tikzpicture}
		\end{matrix}
		=
		\begin{matrix}
			\begin{tikzpicture}[out=up,in=down,line width=0.5pt]
				\node (a) at (0,0) {$\scriptstyle{A}$};
				\node (u) at (0.5,0) {$\scriptstyle{\sunit}$};
				\node (w) at (1.5,0) {$\scriptstyle{W}$};
				\node (l) at (1.5,0.5) {};
				\node (i) at (1,1) {$\scriptstyle{\bullet}$};
				\node (eta) at (0.75,3) [draw] {$\scriptstyle{\eta_{A,A\boxtimes W}}$};
				\node (le) at (0.75,4) [draw] {$\scriptstyle{l^A_{A\boxtimes W}}$};
				\node (wf) at (0.75, 5) {$\scriptstyle{A\boxtimes W}$};
				\draw (a.90) to (0.1,2.1);
				\draw (0.1,2) to (eta.200);
				\draw [dashed] (u.90) to[in=east] (0.1,2.1);
				\draw [dashed] (l.west) to[out=west] (i.south);
				\draw (i.90) to (eta.330);
				\draw [line width=1 pt] (w.90) to (eta.340);
				\draw [line width=1 pt] (eta.90) to (le.270);
				\draw [line width=1 pt] (le.90) to (wf.270);
			\end{tikzpicture}
		\end{matrix}
		=
		\begin{matrix}
			\begin{tikzpicture}[out=up,in=down,line width=0.5pt]
				\node (a) at (0,0) {$\scriptstyle{A}$};
				\node (u) at (0.5,0) {$\scriptstyle{\sunit}$};
				\node (w) at (1.5,0) {$\scriptstyle{W}$};
				\node (l) at (1.5,0.5) {};
				\node (i) at (1,1) {$\scriptstyle{\bullet}$};
				\node (wf) at (1.5, 5) {$\scriptstyle{W}$};
				\node (af) at (0.5, 5) {$\scriptstyle{A}$};
				\node (mu) at (0.5,4) [draw, minimum width=20pt] {$\scriptstyle{\mu}$};
				\draw (a.90) to (0.1,2.1);
				\draw (0.1,2) to (mu.240);
				\draw [dashed] (u.90) to[in=east] (0.1,2.1);
				\draw [dashed] (l.west) to[out=west] (i.south);
				\draw (i.90) to (mu.320);
				\draw [line width=1pt] (w.90) to (wf.270);
				\draw (mu.90) to (af.270);
			\end{tikzpicture}
		\end{matrix}
		=
		\begin{matrix}
			\begin{tikzpicture}[out=up,in=down,line width=0.5pt]
				\node (a) at (0,0) {$\scriptstyle{A}$};
				\node (u) at (0.5,0) {$\scriptstyle{\sunit}$};
				\node (w) at (1,0) {$\scriptstyle{W}$};
				\node (l) at (1,2.3) {};
				\node (i) at (0.6,3) {$\scriptstyle{\bullet}$};
				\node (wf) at (1, 5) {$\scriptstyle{W}$};
				\node (af) at (0.25, 5) {$\scriptstyle{A}$};
				\node (mu) at (0.25,4) [draw, minimum width=20pt] {$\scriptstyle{\mu}$};
				\draw (a.90) to (0.1,2.1);
				\draw (0.1,2) to (mu.240);
				\draw [dashed] (u.90) to[in=east] (0.1,1.8);
				\draw [dashed] (l.west) to[out=west] (i.south);
				\draw (i.90) to (mu.320);
				\draw [line width=1pt] (w.90) to (wf.270);
				\draw (mu.90) to (af.270);
			\end{tikzpicture}
		\end{matrix}
		=
		\begin{matrix}
			\begin{tikzpicture}[out=up,in=down,line width=0.5pt]
				\node (a) at (0,0) {$\scriptstyle{A}$};			
				\node (u) at (0.5,0) {$\scriptstyle{\sunit}$};
				\node (w) at (1,0) {$\scriptstyle{W}$};
				\node (af) at (0,5) {$\scriptstyle{A}$};
				\node (wf) at (1,5) {$\scriptstyle{W}$};
				\draw (a.90) to (af.270);
				\draw [line width=1pt] (w.90) to (wf.270);
				\draw [dashed] (u.90) to[in=east] (0,1.8);
				\draw [dashed] (1,3) to[out=west, in=east] (0,4);
			\end{tikzpicture}
		\end{matrix}
		=
		\begin{matrix}
			\begin{tikzpicture}[out=up,in=down,line width=0.5pt]
				\node (a) at (0,0) {$\scriptstyle{A}$};			
				\node (u) at (0.5,0) {$\scriptstyle{\sunit}$};
				\node (w) at (1,0) {$\scriptstyle{W}$};
				\node (af) at (0,5) {$\scriptstyle{A}$};
				\node (wf) at (1,5) {$\scriptstyle{W}$};
				\draw (a.90) to (af.270);
				\draw [line width=1pt] (w.90) to (wf.270);
				\draw [dashed] (u.90) to[in=east] (0,1.8);
			\end{tikzpicture}
		\end{matrix}
		=
		\begin{matrix}
			\begin{tikzpicture}[out=up,in=down,line width=0.5pt]
				\node (a) at (0,0) {$\scriptstyle{A}$};			
				\node (u) at (0.5,0) {$\scriptstyle{\sunit}$};
				\node (w) at (1,0) {$\scriptstyle{W}$};
				\node (af) at (0,5) {$\scriptstyle{A}$};
				\node (wf) at (1,5) {$\scriptstyle{W}$};
				\draw (a.90) to (af.270);
				\draw [line width=1pt] (w.90) to (wf.270);
				\draw [dashed] (u.90) to[in=west] (1,1.8);
			\end{tikzpicture}
		\end{matrix}
	\end{align*}
	The second follows from:
	\begin{align*}
		\begin{matrix}
			\begin{tikzpicture}[out=up,in=down,line width=0.5pt]
				\node (a) at (0,0) {$\scriptstyle{A}$};
				\node (w) at (0.5,0) {$\scriptstyle{W}$};
				\node (i) at (1,1) {$\scriptstyle{\bullet}$};
				\node (l) at (1.5,0.5) {};
				\node (u) at (1.5,0) {$\scriptstyle{\sunit}$};
				\node (eta) at (0.75,2) [draw] {$\scriptstyle{\eta_{A\boxtimes W, A\boxtimes \sunit}}$};
				\node (phi) at (0.75,3) [draw] {$\scriptstyle{1_{A\boxtimes W}\boxtimes_A\,\varphi}$};
				\node (r) at (0.75,4) [draw] {$\scriptstyle{r^A_{A\boxtimes W}}$};
				\node (wf) at (0.75, 5) {$\scriptstyle{A\boxtimes W}$};
				\draw (a.90) to (eta.200);
				\draw [dashed] (u.90) to (eta.340);
				\draw [dashed] (l.west) to[out=west] (i.south);
				\draw (i.90) to (eta.330);
				\draw [line width=1 pt] (w.90) to (eta.210);
				\draw [line width=1 pt] (eta.90) to (phi.270);
				\draw [line width=1 pt] (phi.90) to (r.270);
				\draw [line width=1 pt] (r.90) to (wf.270);
			\end{tikzpicture}
		\end{matrix}
		=
		\begin{matrix}
			\begin{tikzpicture}[out=up,in=down,line width=0.5pt]
				\node (a) at (0,0) {$\scriptstyle{A}$};
				\node (w) at (0.5,0) {$\scriptstyle{W}$};
				\node (i) at (1,1) {$\scriptstyle{\bullet}$};
				\node (l) at (1.5,0.5) {};
				\node (u) at (1.5,0) {$\scriptstyle{\sunit}$};
				\node (eta) at (0.75,3) [draw] {$\scriptstyle{\eta_{A\boxtimes W, A}}$};
				\node (r) at (0.75,4) [draw] {$\scriptstyle{r^A_{A\boxtimes W}}$};
				\node (wf) at (0.75, 5) {$\scriptstyle{A\boxtimes W}$};
				\draw (a.90) to (eta.200);
				\draw [dashed] (u.90) to (1.5,1) to[in=east] (1,2);
				\draw (i.90) to (1,2) to (eta.330);
				\draw [dashed] (l.west) to[out=west] (i.south);
				\draw [line width=1 pt] (w.90) to (eta.210);
				\draw [line width=1 pt] (eta.90) to (r.270);
				\draw [line width=1 pt] (r.90) to (wf.270);
			\end{tikzpicture}
		\end{matrix}
		=
		\begin{matrix}
			\begin{tikzpicture}[out=up,in=down,line width=0.5pt]
				\node (a) at (0,0) {$\scriptstyle{A}$};
				\node (w) at (0.5,0) {$\scriptstyle{W}$};
				\node (i) at (1,1) {$\scriptstyle{\bullet}$};
				\node (u) at (1,0) {$\scriptstyle{\sunit}$};
				\node (eta) at (0.5,3) [draw] {$\scriptstyle{\eta_{A\boxtimes W, A}}$};
				\node (r) at (0.5,4) [draw] {$\scriptstyle{r^A_{A\boxtimes W}}$};
				\node (wf) at (0.5, 5) {$\scriptstyle{A\boxtimes W}$};
				\draw (a.90) to (eta.210);
				\draw [line width=1 pt] (w.90) to (eta.220);
				\draw [dashed] (u.90) to (i.270);
				\draw  (i.90) to (eta.330);
				\draw [line width=1 pt] (eta.90) to (r.270);
				\draw [line width=1 pt] (r.90) to (wf.270);
			\end{tikzpicture}
		\end{matrix}
		=
		\begin{matrix}
			\begin{tikzpicture}[out=up,in=down,line width=0.5pt]
				\node (a) at (0,0) {$\scriptstyle{A}$};
				\node (w) at (0.5,0) {$\scriptstyle{W}$};
				\node (i) at (1,1) {$\scriptstyle{\bullet}$};
				\node (u) at (1,0) {$\scriptstyle{\sunit}$};
				\node (mu) at (0.3, 4) [draw, minimum width=20pt] {$\scriptstyle{\mu}$};
				\node (af) at (0.3, 5) {$\scriptstyle{A}$};
				\node (wf) at (1, 5) {$\scriptstyle{W}$};
				\draw (a.90) to (mu.320);
				\draw [line width=1pt] (w.90) to (wf.270);
				\draw [dashed] (u.90) to (i.270);
				\draw [white, line width=3pt, double=black, double distance=0.5pt] (i.90) to (mu.220);
				\draw (mu.90) to (af.270);
			\end{tikzpicture}
		\end{matrix}		
		=
		\begin{matrix}
			\begin{tikzpicture}[out=up,in=down,line width=0.5pt]
				\node (a) at (0,0) {$\scriptstyle{A}$};
				\node (w) at (0.5,0) {$\scriptstyle{W}$};
				\node (i) at (0.1,3) {$\scriptstyle{\bullet}$};
				\node (u) at (1,0) {$\scriptstyle{\sunit}$};
				\node (mu) at (0.3, 4) [draw, minimum width=20pt] {$\scriptstyle{\mu}$};
				\node (af) at (0.3, 5) {$\scriptstyle{A}$};
				\node (wf) at (1, 5) {$\scriptstyle{W}$};
				\draw (a.90) to (mu.320);
				\draw [line width=1pt] (w.90) to (wf.270);
				\draw [white, line width=6pt]  (u.90) to (i.270);
				\draw [dashed]  (u.90) to (i.270);
				\draw [white, line width=3pt, double=black, double distance=0.5pt] (i.90) to (mu.220);
				\draw (mu.90) to (af.270);
			\end{tikzpicture}
		\end{matrix}		
		=
		\begin{matrix}
			\begin{tikzpicture}[out=up,in=down,line width=0.5pt]
				\node (a) at (0, 0) {$\scriptstyle{A}$};
				\node (w) at (0.4, 0) {$\scriptstyle{W}$};
				\node (u) at (0.8, 0) {$\scriptstyle{\sunit}$};
				\node (l) at (0.5, 4) {};
				\node (af) at (0.4, 5) {$\scriptstyle{A}$};
				\node (wf) at (0.8, 5) {$\scriptstyle{W}$};
				\draw (a.90) to (af.270);
				\draw [line width=1pt] (w.90) to (wf.270);
				\draw [white, line width=6pt] (u.90) to (0,3.5);
				\draw [dashed] (u.90) to (0,3.5) to[in=west] (l.180);
			\end{tikzpicture}
		\end{matrix}
		=
		\begin{matrix}
			\begin{tikzpicture}[out=up,in=down,line width=0.5pt]
				\node (a) at (0, 0) {$\scriptstyle{A}$};
				\node (w) at (0.4, 0) {$\scriptstyle{W}$};
				\node (u) at (1, 0) {$\scriptstyle{\sunit}$};
				\node (af) at (0, 5) {$\scriptstyle{A}$};
				\node (wf) at (0.4, 5) {$\scriptstyle{W}$};
				\draw (a.90) to (af.270);
				\draw [line width=1pt] (w.90) to (wf.270);
				\draw [dashed] (u.90) to[in=east] (0.4,1.5);
			\end{tikzpicture}
		\end{matrix}
	\end{align*}
	
	Finally, to prove that $f$ is compatible with the associativity isomorphisms in $\sC$ and $\repA$, we have to prove the commutativity of \eqref{Fassoc}
	for objects $W_1$, $W_2$, and $W_3$ in $\mathcal{SC}$. To simplify notation, we use $\iota_W$ for the composition
	\begin{equation*}
	W\xrightarrow{\sleft_{W}^{-1}} \sunit\boxtimes W\xrightarrow{\iota_A\boxtimes 1_{W}} A\boxtimes W,
	\end{equation*}
	where $W$ is an object of $\sC$. Then the lower left composition of morphisms in \eqref{Fassoc} is given by
	\begin{align*}
	&A\boxtimes  (W_1\boxtimes(W_2\boxtimes W_3))\xrightarrow{\sA_{A,W_1,W_2\boxtimes W_3}} (A\boxtimes W_1)\boxtimes(W_2\boxtimes W_3)\nonumber\\
	&\xrightarrow{1_{A\boxtimes W_1}\boxtimes\iota_{W_2\boxtimes W_3}} (A\boxtimes W_1)\boxtimes(A\boxtimes(W_2\boxtimes W_3))\xrightarrow{1_{A\boxtimes W_1}\boxtimes\sA_{A,W_2,W_3}} (A\boxtimes W_1)\boxtimes((A\boxtimes W_2)\boxtimes W_3)\nonumber\\
	&\xrightarrow{1_{A\boxtimes W_1}\boxtimes(1_{A\boxtimes W_2}\boxtimes\iota_{W_3})}(A\boxtimes W_1)\boxtimes((A\boxtimes W_2)\boxtimes(A\boxtimes W_3))\nonumber\\
	&\xrightarrow{\sA_{A\boxtimes W_1,A\boxtimes W_2,A\boxtimes W_3}}((A\boxtimes W_1)\boxtimes(A\boxtimes W_2))\boxtimes(A\boxtimes W_3)\nonumber\\
	&\rightarrow ((A\boxtimes W_1)\boxtimes_A(A\boxtimes W_2))\boxtimes_A(A\boxtimes W_3),
	\end{align*}
	where the last arrow is the natural composition of cokernel morphisms. Now, since $\sA_{\sunit,W_2,W_3}\circ \sleft_{W_2\boxtimes W_3}^{-1}=\sleft_{W_2}^{-1}\boxtimes 1_{W_3}$, we have
	\begin{equation*}
	\sA_{A,W_2,W_3}\circ\iota_{W_2\boxtimes W_3}=\iota_{W_2}\boxtimes 1_{W_3}.
	\end{equation*}
	This together with the naturality of the associativity isomorphisms applied to $\iota_{W_2}$ and $\iota_{W_3}$ yields
	\begin{align*}
	A\boxtimes & (W_1\boxtimes(W_2\boxtimes W_3))\xrightarrow{assoc.} ((A\boxtimes W_1)\boxtimes W_2)\boxtimes W_3\nonumber\\
	&\xrightarrow{(1_{A\boxtimes W_1}\boxtimes\iota_{W_2})\boxtimes\iota_{W_3}} ((A\boxtimes W_1)\boxtimes(A\boxtimes W_2))\boxtimes(A\boxtimes W_3)\nonumber\\
	&\rightarrow ((A\boxtimes W_1)\boxtimes_A(A\boxtimes W_2))\boxtimes_A(A\boxtimes W_3).
	\end{align*}
	By the pentagon axiom, we can take the composition of associativity isomorphisms represented by the first arrow to be
	\begin{equation*}
	(\sA_{A,W_1,W_2}\boxtimes 1_{W_3})\circ\sA_{A,W_1\boxtimes W_2,W_3}\circ(1_A\boxtimes\sA_{W_1,W_2,W_3}).
	\end{equation*}
	Then we get
	\begin{align*}
	A\boxtimes & (W_1\boxtimes(W_2\boxtimes W_3))\xrightarrow{1_A\boxtimes\sA_{W_1,W_2,W_3}} A\boxtimes((W_1\boxtimes W_2)\boxtimes W_3)\nonumber\\
	&\xrightarrow{\sA_{A,W_1\boxtimes W_2,W_3}} (A\boxtimes(W_1\boxtimes W_2))\boxtimes W_3\xrightarrow{1_{A\boxtimes(W_1\boxtimes W_2)}\boxtimes\iota_{W_3}} (A\boxtimes(W_1\boxtimes W_2))\boxtimes(A\boxtimes W_3)\nonumber\\
	&\xrightarrow{\sA_{A,W_1,W_2}\boxtimes 1_{A\boxtimes W_3}} ((A\boxtimes W_1)\boxtimes W_2)\boxtimes(A\boxtimes W_3)\nonumber\\
	&\xrightarrow{(1_{A\boxtimes W_1}\boxtimes\iota_{W_2})\boxtimes 1_{A\boxtimes W_3}}((A\boxtimes W_1)\boxtimes(A\boxtimes W_2))\boxtimes(A\boxtimes W_3)\nonumber\\
	&\rightarrow ((A\boxtimes W_1)\boxtimes_A(A\boxtimes W_2))\boxtimes_A(A\boxtimes W_3),
	\end{align*}
	and this is the upper right composition in \eqref{Fassoc}, completing the proof of the theorem.
\end{proof}

In addition to the induction functor, there is a restriction functor from $\repA$ to $\sC$:
\begin{defi}\label{def:restriction}
The restriction functor $\cG: \repA\rightarrow\sC$ is given by
\begin{equation*}
 \cG: (W,\mu_W)\mapsto W;\,\,f\mapsto f.
\end{equation*}
\end{defi}
The following adjointness of $\cF$ and $\cG$ is well known (see for instance \cite[Lemma 7.8.12]{EGNO}):
\begin{lemma}\label{lem:FGadjoint}
	For any object $X$ of $\sC$ and any object $W$ of $\repA$, there is a natural isomorphism between 
	$\hom_{\repA}(\cF(X),W)$ and $\hom_{\sC}(X,\cG(W))$.
\end{lemma}
	For future reference, we record here the maps used in the proof:
	\begin{align*}
		\Phi: \hom_{\repA}(\cF(X),W)&\longrightarrow \hom_{\sC}(X,\cG(W))\\
		f&\longmapsto  f \circ (\iota_A \boxtimes 1_X)\circ \sleft_X^{-1}\\
		\Psi:  \hom_{\sC}(X,\cG(W))&\longrightarrow \hom_{\repA}(\cF(X),W)\\
		g&\longmapsto  \mu_W\circ(1_A\boxtimes g).
	\end{align*}
	It is easy to verify that $\Psi(g)$ is indeed a morphism in $\repA$, that $\Phi$ and $\Psi$ are mutual inverses, and that they are natural isomorphisms.

\subsection{More properties of \texorpdfstring{$\repA$}{Rep A} and the induction functor}\label{subsec:miscresults}

In this subsection we study the relation between the induction functor (and restriction functor) constructed in the previous subsection and additional structures that can be imposed on $\sC$, such as rigidity and ribbon structure, which we recall below. Many of the results here are contained in \cite{KO} and \cite{EGNO}, but in some cases a little additional care is needed in the superalgebra generality.

Our first goal is to show that tensoring in $\repA$ preserves right exact sequences of parity-homogeneous morphisms; for this we need the following lemma:
\begin{lemma}[{\cite[Theorem 1.6(1)]{KO}}]\label{lem:Gexact}
	The restriction functor $\cG$ is injective on morphisms and preserves exact sequences of parity-homogeneous morphisms.
\end{lemma}
\begin{proof}
	The first assertion is obvious. For the second, let $W_1\xrightarrow{f_1} W_2\xrightarrow{f_2} W_3\rightarrow 0$ be right exact in $\repA$ (that is, $(W_3, f_2)$ is a $\repA$-cokernel of $f_1$) with $f_1$ and $f_2$  parity-homogeneous. We need to show that $(W_3, f_2)$ is also an $\sC$-cokernel of $f_1$. Proposition \ref{SuperRepAadditive} and its proof show that $f_1$ has a $\repA$-cokernel $(C,c)$ which is also a cokernel of $f_1$ in $\sC$. Thus there is a $\repA$-isomorphism $\eta: W_3\rightarrow C$ such that $\eta\circ f_2=c$; using this isomorphism, it is clear that $(W_3, f_2)$ is a cokernel of $f_1$ in $\sC$.
	
	One can show similarly that if $0\rightarrow W_1\xrightarrow{f_1} W_2\xrightarrow{f_2} W_3$ is a left exact sequence of parity-homogeneous morphisms in $\repA$, then $(W_1, f_1)$ is also a kernel of $f_2$ in $\sC$.
\end{proof}

Now we can prove that tensoring in $\repA$ preserves right exact sequences of homogeneous morphisms:
\begin{propo}\label{lem:tensrightexact} 
For any object $W$ in $\repA$ and right exact sequence $W_1\xrightarrow{f_1} W_2\xrightarrow{f_2} W_3\rightarrow 0$ in $\repA$ with $f_1$ parity-homogeneous, the sequences
\begin{equation*}
 W\boxtimes_A W_1\xrightarrow{1_W\boxtimes_A f_1} W\boxtimes_A W_2\xrightarrow{1_W\boxtimes_A f_2} W\boxtimes_A W_3\rightarrow 0 
\end{equation*}
and
\begin{equation*}
  W_1\boxtimes_A W\xrightarrow{f_1\boxtimes_A 1_W} W_2\boxtimes_A W\xrightarrow{f_2\boxtimes_A 1_W} W_3\boxtimes_A W\rightarrow 0 
\end{equation*}
are also right exact.
\end{propo}
\begin{proof}
For the first sequence, we need to show that $(W\boxtimes_A W_3, 1_W\boxtimes_A f_2)$ is a cokernel of $1_W\boxtimes_A f_1$. By Lemma \ref{lem:Gexact}, $(W_3, f_2)$ is a cokernel of $f_1$ in $\sC$, so because $W\boxtimes\cdot$ is right exact in $\sC$, the top row of the following commutative diagram is right exact:
\begin{equation*}
 \xymatrixcolsep{3pc}
 \xymatrix{
 W\boxtimes W_1 \ar[r]^{1_W\boxtimes f_1} \ar[d]^{\eta_{W,W_1}} & W\boxtimes W_2 \ar[r]^{1_W\boxtimes f_2} \ar[d]^{\eta_{W,W_2}} & W\boxtimes W_3 \ar[d]^{\eta_{W,W_3}} \ar[r]^{} & 0\\
 W\boxtimes_A W_1 \ar[r]^{1_W\boxtimes_A f_1} & W\boxtimes_A W_2 \ar[r]^{1_W\boxtimes_A f_2} & W\boxtimes_A W_3 \ar[r]^{} & 0\\
 }
\end{equation*}
Now, if $g: W\boxtimes_A W_2\rightarrow\widetilde{W}$ is a morphism in $\repA$ such that $g\circ(1_W\boxtimes_A f_1)=0$, then $g\circ\eta_{W,W_2}\circ(1_W\boxtimes f_1)=0$ as well. Thus the universal property of the cokernel $(W\boxtimes W_3, 1_W\boxtimes f_2)$ implies that there is a unique $\sC$-morphism $\eta: W\boxtimes W_3\rightarrow\widetilde{W}$ such that 
\begin{equation*}
 \eta\circ(1_W\boxtimes f_2)=g\circ\eta_{W,W_2}.
\end{equation*}
To complete the proof that $W\boxtimes_A\cdot$ is right exact, we need to show that $\eta$ is a $\repA$-intertwining operator, since then the universal property of the $\repA$-tensor product will induce a unique $\repA$-morphism $\widetilde{g}: W\boxtimes_A W_3\rightarrow\widetilde{W}$ such that $\widetilde{g}\circ\eta_{W,W_3}=\eta$, from which we get
\begin{equation*}
 \widetilde{g}\circ(1_W\boxtimes_A f_2)=g
\end{equation*}
using the surjectivity of $\eta_{W,W_2}$.

To see that $\eta$ is a $\repA$-intertwining operator, we use the fact that $f_2$ is a $\repA$-morphism, \eqref{repAtensprodmorph}, the definition of $\eta$, and the fact that $g\circ\eta_{W,W_2}$ is a $\repA$-intertwining operator (recall Proposition \ref{compintwopmorph}):
\begin{align*}
 \eta\circ\mu^{(i)}\circ(1_A\boxtimes(1_W\boxtimes f_2)) & = \eta\circ(1_W\boxtimes f_2)\circ\mu^{(i)} = g\circ\eta_{W,W_2}\circ\mu^{(i)}\nonumber\\
 &= \mu_{\widetilde{W}}\circ(1_A\boxtimes(g\circ\eta_{W,W_2})) =\mu_{\widetilde{W}}\circ(1_A\boxtimes\eta)\circ(1_A\boxtimes(1_W\boxtimes f_2))
\end{align*}
for $i=1,2$. Since $1_A\boxtimes(1_W\boxtimes f_2)$ is surjective, $\eta\circ\mu^{(i)}=\mu_{\widetilde{W}}\circ(1_A\boxtimes \eta)$ for $i=1,2$, as required.
	
The proof that $(W_3\boxtimes_A W, f_2\boxtimes_A 1_W)$ is a cokernel of $f_1\boxtimes_A 1_W$ is the same. See also \cite[Exercise 7.8.23]{EGNO}, but note that they work in a rigid category, which we do not assume here.
\end{proof}

\begin{rema}
 The preceding lemma and proposition show that $\cG$ is an exact functor from $\underline{\repA}$ to $\sCeven$, and that for any object in $W$ in $\underline{\repA}$, $W\boxtimes_A\cdot$ and $\cdot\boxtimes_A W$ are right exact functors on $\underline{\repA}$. An easy consequence of Lemmas \ref{lem:FGadjoint} and \ref{lem:Gexact} is that if $W$ is a projective object in $\sCeven$, then $\cF(W)$ is projective in $\underline{\repA}$. All these results hold when $\sCeven$ is replaced by $\cC$ and $A$ is an ordinary algebra in $\cC$.
\end{rema}

We now study the interplay between induction, the braided monoidal supercategory $\repzA$ of local $A$-modules, and additional structures on $\sC$. First, we show when an object of $\sC$ to induces to an object in $\repzA$:
\begin{propo}\label{propo:rep0inductioncriteria}
 For an object $W$ of $\sC$,  $\cF(W)$ is an object of $\repzA$ if and only if $\cM_{A,W}=1_{A\boxtimes W}$.
\end{propo}
\begin{proof}
 First suppose $\cF(W)=A\boxtimes W$ is an object of $\repzA$, that is, $\mu_{A\boxtimes W}\circ\cM_{A,A\boxtimes W}=\mu_{A\boxtimes W}$. Then the following calculation shows that $\cM_{A,W}=1_{A\boxtimes W}$:
 \begin{align*}
 	\begin{matrix}
 		\begin{tikzpicture}[out=up,in=down,line width=0.5pt]
 			\node (a) at (0,0) {$\scriptstyle{A}$};
 			\node (w) at (1,0) {$\scriptstyle{W}$};
 			\node (af) at (0, 3.6) {$\scriptstyle{A}$};
 			\node (wf) at (1, 3.6) {$\scriptstyle{W}$};
 			\draw [line width=1pt] (w.90) to (0,1);
 			\draw [white, double=black, line width=3pt, double distance=0.5pt] (a.90) to (1,1) to (0,1.8) to (af.270);
 			\draw [white, double=black, line width=3pt, double distance=1pt] (0,1) to (1,1.8) to (wf.270);			
 		\end{tikzpicture}			
 	\end{matrix}
 	=			
 	\begin{matrix}
 		\begin{tikzpicture}[out=up,in=down,line width=0.5pt]
 			\node (a) at (0,0) {$\scriptstyle{A}$};
 			\node (w) at (1,0) {$\scriptstyle{W}$};
 			\node (mu) at (0.3,3) [draw, minimum width=20pt] {$\scriptstyle{\mu}$};
 			\node (af) at (0.3, 3.6) {$\scriptstyle{A}$};
 			\node (wf) at (1, 3.6) {$\scriptstyle{W}$};
 			\node (i) at (0.6,2.3) {$\scriptstyle{\bullet}$};
 			\node (r) at (0,1.8) {};
 			\draw [line width=1pt] (w.90) to (0,1) to (1,1.8);
 			\draw [white, line width=3pt, double=black, double distance=0.5pt] (a.90) to (1,1) to (0,1.8) to (mu.220);
 			\draw [white, line width=3pt, double=black, double distance=1pt] (0,1) to (1,1.8) to (wf.270);
 			\draw [dashed] (r.0) to[out=east] (i.270);
 			\draw (i.90) to (mu.320);
 			\draw (mu.90) to (af.270);
 		\end{tikzpicture}
 	\end{matrix}
 	=
 	\begin{matrix}
 		\begin{tikzpicture}[out=up,in=down,line width=0.5pt]
 			\node (a) at (0,0) {$\scriptstyle{A}$};
 			\node (w) at (1,0) {$\scriptstyle{W}$};
 			\node (mu) at (0.3,3) [draw, minimum width=20pt] {$\scriptstyle{\mu}$};
 			\node (af) at (0.3, 3.6) {$\scriptstyle{A}$};
 			\node (wf) at (1, 3.6) {$\scriptstyle{W}$};
 			\node (i) at (0.6,2.3) {$\scriptstyle{\bullet}$};
 			\node (r) at (1,1.8) {};
 			\draw [line width=1pt] (w.90) to (0,1) to (1,1.8);
 			\draw [white, line width=3pt, double=black, double distance=0.5pt] (a.90) to (1,1) to (0,1.8) to (mu.220);
 			\draw [white, line width=3pt, double=black, double distance=1pt] (0,1) to (1,1.8) to (wf.270);
 			\draw [dashed] (r.180) to[out=west] (i.270);
 			\draw (i.90) to (mu.320);
 			\draw (mu.90) to (af.270);
 		\end{tikzpicture}
 	\end{matrix}
 	=
 	\begin{matrix}
 		\begin{tikzpicture}[out=up,in=down,line width=0.5pt]
 			\node (a) at (0,0) {$\scriptstyle{A}$};
 			\node (w) at (1,0) {$\scriptstyle{W}$};
 			\node (mu) at (0.3,3) [draw, minimum width=20pt] {$\scriptstyle{\mu}$};
 			\node (af) at (0.3, 3.6) {$\scriptstyle{A}$};
 			\node (wf) at (1, 3.6) {$\scriptstyle{W}$};
 			\node (i) at (0.7,0.9) {$\scriptstyle{\bullet}$};
 			\node (r) at (1,0.4) {};
 			\draw [line width=1pt] (w.90) to (1,1.25) to	(0.5,2); 
 			\draw [dashed] (r.180) to[out=west] (i.270);
 			\draw (i.90) to (0,2);
 			\draw [white, line width=3pt, double=black, double distance=0.5pt] (a.90) to (0,1) to (1,2);
 			\draw (1,2) to (mu.220);
 			\draw [white, line width=3pt, double=black, double distance=1pt] (0.5, 2) to (1,2.75) to (wf.270);
 			\draw [white, line width=3pt, double=black,double distance=0.5pt] (0,2) to (mu.320);
 			\draw (mu.90) to (af.270);
 		\end{tikzpicture}
 	\end{matrix}
 	=
 	\begin{matrix}
 		\begin{tikzpicture}[out=up,in=down,line width=0.5pt]
 			\node (a) at (0,0) {$\scriptstyle{A}$};
 			\node (w) at (1,0) {$\scriptstyle{W}$};
 			\node (mu) at (0.3,3) [draw, minimum width=20pt] {$\scriptstyle{\mu}$};
 			\node (af) at (0.3, 3.6) {$\scriptstyle{A}$};
 			\node (wf) at (1, 3.6) {$\scriptstyle{W}$};
 			\node (i) at (0.6,2) {$\scriptstyle{\bullet}$};
 			\node (r) at (1,1) {};
 			\draw (a.90) to (mu.220);
 			\draw (i.90) to (mu.320);
 			\draw [dashed] (r.180) to[out=west] (i.270);
 			\draw [line width=1pt] (w.90) to (wf.270);
 			\draw (mu.90) to (af.270);
 		\end{tikzpicture}
 	\end{matrix}
 	=
 	\begin{matrix}
 		\begin{tikzpicture}[out=up,in=down,line width=0.5pt]
 			\node (a) at (0,0) {$\scriptstyle{A}$};
 			\node (w) at (1,0) {$\scriptstyle{W}$};
 			\node (mu) at (0.3,3) [draw, minimum width=20pt] {$\scriptstyle{\mu}$};
 			\node (af) at (0.3, 3.6) {$\scriptstyle{A}$};
 			\node (wf) at (1, 3.6) {$\scriptstyle{W}$};
 			\node (i) at (0.6,2) {$\scriptstyle{\bullet}$};
 			\node (r) at (0,1) {};
 			\draw (a.90) to (mu.220);
 			\draw (i.90) to (mu.320);
 			\draw [dashed] (r.0) to[out=east] (i.270);
 			\draw [line width=1pt] (w.90) to (wf.270);
 			\draw (mu.90) to (af.270);
 		\end{tikzpicture}
 	\end{matrix}
 	=
 	\begin{matrix}
 		\begin{tikzpicture}[out=up,in=down,line width=0.5pt]
 			\node (a) at (0,0) {$\scriptstyle{A}$};
 			\node (w) at (0.5,0) {$\scriptstyle{W}$};
 			\node (af) at (0, 3.6) {$\scriptstyle{A}$};
 			\node (wf) at (0.5, 3.6) {$\scriptstyle{W}$};
 			\draw (a.90) to (af.270);
 			\draw [line width=1pt] (w.90) to (wf.270);
 		\end{tikzpicture}
 	\end{matrix}
 \end{align*}
Conversely, if $\cM_{A, W}=1_{A\boxtimes W}$, then $\mu_{A\boxtimes W}\circ\cM_{A,A\boxtimes W}=\mu_{A\boxtimes W}$ follows from:
\begin{align*}
	\begin{matrix}
		\begin{tikzpicture}[out=up,in=down,line width=0.5pt]
			\node (a) at (0,0) {$\scriptstyle{A}$};
			\node (a2) at (0.75,0) {$\scriptstyle{A}$};
			\node (w) at (1.25,0) {$\scriptstyle{W}$};
			\node (af) at (0.25, 3.1) {$\scriptstyle{A}$};
			\node (wf) at (1.25, 3.1) {$\scriptstyle{W}$};
			\node (mu) at (0.25, 2.5) [draw, minimum width=20pt] {$\scriptstyle{\mu}$};
			\draw [line width=1pt] (w.90) to (0.5,1);
			\draw  (a2.90) to (0,1);
			\draw [white, double=black, line width=3pt, double distance=0.5pt] (a.90) to (1.25,1) to (0,2) to (mu.220);
			\draw [white, double=black, line width=3pt, double distance=1pt] (0.5,1) to (1.25,2) to (wf.270);			
			\draw [white, double=black, line width=3pt, double distance=0.5pt] (0,1) to (0.5,2) to (mu.320);			
			\draw (mu.90) to (af.270);
		\end{tikzpicture}			
	\end{matrix}
	=
	\begin{matrix}
		\begin{tikzpicture}[out=up,in=down,line width=0.5pt]
			\node (a) at (0,0) {$\scriptstyle{A}$};
			\node (a2) at (0.75,0) {$\scriptstyle{A}$};
			\node (w) at (1.25,0) {$\scriptstyle{W}$};
			\node (af) at (0.25, 3.1) {$\scriptstyle{A}$};
			\node (wf) at (1.25, 3.1) {$\scriptstyle{W}$};
			\node (mu) at (0.25, 2.5) [draw, minimum width=20pt] {$\scriptstyle{\mu}$};
			\draw [line width=1pt] (w.90) to (wf.270);
			\draw  (a2.90) to (0,1);
			\draw [white, double=black, line width=3pt, double distance=0.5pt] (a.90) to (0.5,1) to (0,2) to (mu.220);
			\draw [white, double=black, line width=3pt, double distance=0.5pt] (0,1) to (0.5,2) to (mu.320);			
			\draw (mu.90) to (af.270);
		\end{tikzpicture}	
	\end{matrix}
	=
	\begin{matrix}
		\begin{tikzpicture}[out=up,in=down,line width=0.5pt]
			\node (a) at (0,0) {$\scriptstyle{A}$};
			\node (a2) at (0.5,0) {$\scriptstyle{A}$};
			\node (w) at (1,0) {$\scriptstyle{W}$};
			\node (af) at (0.25, 3.1) {$\scriptstyle{A}$};
			\node (wf) at (1, 3.1) {$\scriptstyle{W}$};
			\node (mu) at (0.25, 2.5) [draw, minimum width=20pt] {$\scriptstyle{\mu}$};
			\draw [line width=1pt] (w.90) to (wf.270);
			\draw (a.90) to (mu.220);
			\draw  (a2.90) to (mu.320);
			\draw (mu.90) to (af.270);
		\end{tikzpicture}	
	\end{matrix}
\end{align*}
\end{proof}

\begin{defi}
Let $\sC^0$ denote the full subcategory of objects in $\sC$ that induce to $\repzA$. 
\end{defi}
In the next theorem, we show that induction $\cF: \sC^0\rightarrow\repzA$ is a braided tensor functor. Recall the natural isomorphism $f: \cF\circ\boxtimes\rightarrow\boxtimes_A\circ(\cF\times\cF)$ and its inverse $g$ from the proof of Theorem \ref{thm:inductionfucntor}.
\begin{theo}\label{thm:Fisbraidedtensor}
The category $\sC^0$ is an $\mathbb{F}$-linear additive braided monoidal supercategory, with structures induced from $\sC$, and $\cF: \sC^0\rightarrow\repzA$ is a braided tensor functor.
\end{theo}	
\begin{proof}
Since $\sC^0$ is a full subcategory of $\sC$, morphisms in $\sC^0$ form $\mathbb{F}$-superspaces, and the zero object of $\sC$ is in $\sC^0$. For biproducts,
let $\lbrace W_i\rbrace$ be a finite set of objects in $\sC^0$. Then Proposition \ref{propo:rep0inductioncriteria}, 
distributivity of tensor products over direct sums, and naturality of monodromy
shows that $\bigoplus W_i$ induces to $\repzA$. Thus $\sC^0$ is an $\mathbb{F}$-linear additive supercategory. To show that $\sC^0$ is braided monoidal, it is enough to show that it is closed under $\boxtimes$. By \cref{thm:rep0}, $\cF(W_1)\boxtimes_A\cF(W_2)$ is an object $\repzA$ for $W_1$, $W_2$ in $\sC^0$, and moreover $\cF(W_1\boxtimes W_2)\cong \cF(W_1)\boxtimes_A\cF(W_2)$ since $\cF$ is a tensor functor. Hence $W_1\boxtimes W_2$ is an object of $\sC^0$.

To show that $\cF$ is a braided tensor functor on $\sC^0$, we need to show that for objects $W_1$, $W_2$ in $\sC^0$, 
	\begin{align}
	\cF(\sR_{W_1,W_2})=g_{W_2,W_1}\circ\cR^A_{\cF(W_1),\cF(W_2)}\circ f_{W_1,W_2}.\label{eqn:inducedbraiding}
	\end{align}
By definition, $g_{W_2,W_1}\circ\cR^A_{\cF(W_1),\cF(W_2)}\circ f_{W_1,W_2}$ is the composition
\begin{align*}
 & A\boxtimes(W_1\boxtimes W_2)\xrightarrow{\sA_{A,W_1,W_2}} (A\boxtimes W_1)\boxtimes W_2\xrightarrow{1_{A\boxtimes W_1}\boxtimes\iota_{W_2}} (A\boxtimes W_1)\boxtimes(A\boxtimes W_2)\nonumber\\
 &\xrightarrow{\eta_{A\boxtimes W_1,A\boxtimes W_2}} (A\boxtimes W_1)\boxtimes_A (A\boxtimes W_2)\xrightarrow{\cR^A_{A\boxtimes W_1,A\boxtimes W_2}} (A\boxtimes W_2)\boxtimes_A(A\boxtimes W_1)\xrightarrow{g_{W_2,W_1}} A\boxtimes(W_2\boxtimes W_1),
\end{align*}
where $\iota_{W_2}=(\iota_A\boxtimes 1_{W_2})\circ\sleft_{W_2}^{-1}$. Then using the definition of $\cR^A_{A\boxtimes W_1,A\boxtimes W_2}$ and borrowing notation from the proof of Theorem \ref{thm:inductionfucntor}, 
\begin{align*}
 g_{W_2,W_1}\circ\cR^A_{A\boxtimes W_1,A\boxtimes W_2}\circ\eta_{A\boxtimes W_1,A\boxtimes W_2} & = g_{W_2,W_1}\circ\eta_{A\boxtimes W_2,A\boxtimes W_1}\circ\sR_{A\boxtimes W_1,A\boxtimes W_2}= \widetilde{g}_{W_2,W_1}\circ\sR_{A\boxtimes W_1, A\boxtimes W_2}.
\end{align*}
So we get the composition
\begin{align}\label{Fbraided1}
  A & \boxtimes(W_1\boxtimes W_2) \xrightarrow{\sA_{A,W_1,W_2}} (A\boxtimes W_1)\boxtimes W_2\xrightarrow{1_{A\boxtimes W_1}\boxtimes\iota_{W_2}} (A\boxtimes W_1)\boxtimes(A\boxtimes W_2)\nonumber\\
 &\xrightarrow{\sR_{A\boxtimes W_1, A\boxtimes W_2}} (A\boxtimes W_2)\boxtimes(A\boxtimes W_1)\xrightarrow{assoc.} (A\boxtimes(W_2\boxtimes A))\boxtimes W_1\nonumber\\
 & \xrightarrow{(1_A\boxtimes\sR_{A,W_2}^{-1})\boxtimes 1_{W_1}} (A\boxtimes(A\boxtimes W_2))\boxtimes W_1\xrightarrow{assoc.}(A\boxtimes A)\boxtimes(W_2\boxtimes W_1)\xrightarrow{\mu\boxtimes 1_{W_2\boxtimes W_1}} A\boxtimes(W_2\boxtimes W_1).
\end{align}
	Now we proceed as guided by the following diagrams:
	\begin{align*}
	\begin{matrix}
		\begin{tikzpicture}[scale=0.8, out=up, in=down, line width=0.5pt]
			\node at (0,-0.3) {$\scriptstyle{A}$};
			\node at (0.7,-0.3) {$\scriptstyle{W_1}$};
			\node at (2.3,-0.3) {$\scriptstyle{W_2}$};
			\node at (0.5,5.5) {$\scriptstyle{A}$};
			\node at (1.6,5.5) {$\scriptstyle{W_2}$};
			\node at (2.3,5.5) {$\scriptstyle{W_1}$};
			\node (unit) at (1.5,1) {$\bullet$};
			\node (A) at (0.5,4.5) [draw,minimum width=25pt,minimum height=10pt] {$\scriptstyle{\mu}$};
			\draw[](A.north) to (0.5,5.3);
			\draw[dashed](2.2,0.3) to [out=left, in=down] (1.5,1);
			\draw[](unit.north) to (0,3) to (A.210);
			\draw[line width=1pt](2.2,0) to(2.2,1.5) to(0.7,3) to(1.7,4) to(1.7,5.3); 
			\draw[white, line width=3pt, double=black, double distance=0.5pt](0,0) to (0,1.3) to (1.5,3) to (A.330);
			\draw[white, line width=3pt, double=black, double distance=1pt](0.7,0) to (0.7,1) to (2.2,3) to (2.2,5.3);
		\end{tikzpicture}
	\end{matrix}
	=
	\begin{matrix}
		\begin{tikzpicture}[scale=0.8, out=up, in=down, line width=0.5pt]
			\node at (0,-0.3) {$\scriptstyle{A}$};
			\node at (0.7,-0.3) {$\scriptstyle{W_1}$};
			\node at (2.3,-0.3) {$\scriptstyle{W_2}$};
			\node at (0.5,5.5) {$\scriptstyle{A}$};
			\node at (1.6,5.5) {$\scriptstyle{W_2}$};
			\node at (2.3,5.5) {$\scriptstyle{W_1}$};
			\node (unit) at (1.5,1) {$\bullet$};
			\node (A) at (0.5,4.5) [draw,minimum width=25pt,minimum height=10pt] {$\scriptstyle{\mu}$};
			\draw[](A.north) to (0.5,5.3);
			\draw[dashed](2.2,0.3) to [out=left, in=down] (1.5,1);
			\draw[](unit.north) to (0,3) to (A.210);
			\draw[line width=1pt](2.2,0) to(2.2,1.5) to(1.7,3) to(1.7,4) to(1.7,5.3);
			\draw[white, line width=3pt, double=black, double distance=0.5pt](0,0) to (0,1.3) to (1,3) to (A.330);
			\draw[white, line width=3pt, double=black, double distance=1pt](0.7,0) to (0.7,1) to (2.2,3) to (2.2,5.3);
		\end{tikzpicture}
	\end{matrix}
	=
	\begin{matrix}
		\begin{tikzpicture}[scale=0.8, out=up, in=down, line width=0.5pt]
			\node at (0,-0.3) {$\scriptstyle{A}$};
			\node at (0.7,-0.3) {$\scriptstyle{W_1}$};
			\node at (2.3,-0.3) {$\scriptstyle{W_2}$};			
			\node at (0.5,5.5) {$\scriptstyle{A}$};
			\node at (1.6,5.5) {$\scriptstyle{W_2}$};
			\node at (2.3,5.5) {$\scriptstyle{W_1}$};
			\node (unit) at (1.5,1) {$\bullet$};
			\node (A) at (0.5,4.5) [draw,minimum width=25pt,minimum height=10pt] {$\scriptstyle{\mu}$};
			\draw[](A.north) to (0.5,5.3);
			\draw[](0,0) to (A.210);
			\draw[dashed](2.2,0.3) to [out=left, in=down] (1.5,1);
			\draw[](unit.north) to (A.330);
			\draw[line width=1pt](2.2,0) to(2.2,1.5) to(1.7,3) to(1.7,4) to(1.7,5.3);
			\draw[white, line width=3pt, double=black, double distance=1pt](0.7,0) to (0.7,1) to (2.2,3) to (2.2,5.3);
		\end{tikzpicture}
	\end{matrix}
	=
	\begin{matrix}
		\begin{tikzpicture}[scale=0.8, out=up, in=down, line width=0.5pt]
			\node at (0,-0.3) {$\scriptstyle{A}$};
			\node at (1.1,-0.3) {$\scriptstyle{W_1}$};
			\node at (2.2,-0.3) {$\scriptstyle{W_2}$};
			\node at (0,5.5) {$\scriptstyle{A}$};
			\node at (1.1,5.5) {$\scriptstyle{W_2}$};
			\node at (2.2,5.5) {$\scriptstyle{W_1}$};
			\draw[](0,0) to (0,5.3);
			\draw[line width=1pt](2.2,0) to(1.1,5.3);
			\draw[dashed] (2.1,0.5) to [out=left,in=right] (0,4.5);
			\draw[white, line width=3pt, double=black, double distance=1pt](1.1,0) to (2.2,5.3);
		\end{tikzpicture}
	\end{matrix}
	=
	\begin{matrix}
		\begin{tikzpicture}[scale=0.8, out=up, in=down, line width=0.5pt]
			\node at (0.3,-0.3) {$\scriptstyle{A}$};
			\node at (1.1,-0.3) {$\scriptstyle{W_1}$};
			\node at (2.2,-0.3) {$\scriptstyle{W_2}$};
			\node at (0.3,5.5) {$\scriptstyle{A}$};
			\node at (1.1,5.5) {$\scriptstyle{W_2}$};
			\node at (2.2,5.5) {$\scriptstyle{W_1}$};
			\draw[](0.3,0) to (0.3,5.3);
			\draw[line width=1pt](2.2,0) to(1.1,5.3);
			\draw[white, line width=3pt, double=black, double distance=1pt](1.1,0) to (2.2,5.3);
		\end{tikzpicture}
	\end{matrix}.
\end{align*}
The last diagram simply represents $1_A\boxtimes\sR_{W_1,W_2}=\cF(\sR_{W_1,W_2})$.
\end{proof}

\begin{corol}\label{Fmonodromy}
	If $W_1$ and $W_2$ are objects of $\sC^0$,  then 
	\[f_{W_1,W_2}\circ\cF(\cM_{W_1,W_2})=\cM^A_{\cF(W_1),\cF(W_1)}\circ f_{W_1,W_2}.\]
\end{corol}

\begin{rema}
 The category $\sC^0$ may not be abelian: since $\boxtimes$ on $\sC$ is not necessarily left exact, $\sC^0$ might not be closed under kernels.
\end{rema}

Next, we look at the relation between induction and duals in $\sC$; first we recall the notion of duals in monoidal (super)categories:
\begin{defi} ({\cite[Definition\ 2.10.1]{EGNO}})
	A \textit{left dual} $(W^*, e_W, i_W)$ of an object $W$ in a monoidal (super)category is an object $W^*$ equipped with (even) morphisms $e_W: W^*\boxtimes W\rightarrow\sunit$ and $i_W: \sunit\rightarrow W\boxtimes W^*$ satisfying the \textit{rigidity axioms}:
	\begin{align}
		W\xrightarrow{\sleft_W^{-1}}\sunit\boxtimes W\xrightarrow{i_W\boxtimes 1_W} (W\boxtimes W^*)\boxtimes W \xrightarrow{\sA_{W,W^*,W}^{-1}}W\boxtimes(W^*\boxtimes W) \xrightarrow{1_W\boxtimes e_W}
		W\boxtimes\sunit\xrightarrow{\sright_W}W\label{eqn:rigid1}
	\end{align} 
	is $1_W$ and
	\begin{align}
		W^*\xrightarrow{\sright_{W^*}^{-1}}W^*\boxtimes\sunit\xrightarrow{1_{W^*}\boxtimes i_{W}} W^*\boxtimes ( W\boxtimes W^*) \xrightarrow{\sA_{W^*,W,W^*}} (W^*\boxtimes W)\boxtimes W^* \xrightarrow{e_W\boxtimes 1_{W^*}}
		\sunit\boxtimes W^*\xrightarrow{l_{W^*}}W^*
		\label{eqn:rigid2}
	\end{align}
	is $1_{W^*}$.
	
	A \textit{right dual} $({}^* W, e_W', i_W')$ of $W$ where now $e_W': W\boxtimes {}^*W\rightarrow\sunit$ and $i_W': \sunit\rightarrow {}^*W\boxtimes W$ can be defined using analogous rigidity axioms.

	A monoidal (super)category is \textit{rigid} if every object has a left and a right dual.
\end{defi}

\begin{rema}
 Left and right duals are unique up to unique isomorphism, if they exist (see for instance \cite[Proposition 2.10.5]{EGNO}).
\end{rema}

If the braided tensor category $\cC$ is rigid, then so is the associated supercategory $\sC$:
\begin{lemma}\label{lem:CrigidthenSCrigid}
Let $W=(W^\zero, W^\one)$ be an object of $\sC$. If both $W^\zero$
and $W^\one$ possess left, respectively right, duals in $\cC$, then the same holds for $W$ in $\sC$.
In particular, if $\cC$ is rigid then so is $\sC$.
\end{lemma}	
\begin{proof}
It is easy to check that if $W=W^\zero\oplus W^\one$ is an object of $\sC$,
then $W^*$ can be taken to be $(W^\zero)^*\oplus (W^\one)^*$ with
\begin{equation*}
 e_W: \left(((W^\even)^*\boxtimes W^\even)\oplus((W^\odd)^*\boxtimes W^\odd)\right)\oplus\left(((W^\even)^*\boxtimes W^\odd)\oplus((W^\odd)^*\boxtimes W^\even)\right)\rightarrow\unit\oplus 0
\end{equation*}
given by $e_W=(e_{W^\even}\circ p_{(W^\even)^*\boxtimes W^\even}+e_{W^\odd}\circ p_{(W^\odd)^*\boxtimes W^\odd})\oplus 0$ and 
\begin{equation*}
 i_W: \unit\oplus 0\rightarrow \left((W^\even\boxtimes (W^\even)^*)\oplus(W^\odd\boxtimes (W^\odd)^*)\right)\oplus\left((W^\even\boxtimes (W^\odd)^*)\oplus(W^\odd\boxtimes (W^\even)^*)\right)
\end{equation*}
given by $i_W=(q_{W^\even\boxtimes (W^\even)^*}\circ i_{W^\even} +q_{W^\odd\boxtimes (W^\odd)^*}\circ i_{W^\odd})\oplus 0$. Right duals can be arranged analogously.
\end{proof}

\begin{defi}[{See \cite[Equation 2.47]{EGNO}}]\label{def:dualmorph}
If $f: W_1\rightarrow W_2$ is a morphism and $W_1$, $W_2$ possess left duals, the left dual morphism $f^*$ is defined to be the composition
	\begin{align}\label{eqn:dualmorph}	W_2^*&\xrightarrow{\sright_{W_2^*}^{-1}}W_2^*\boxtimes\sunit
	\xrightarrow{1_{W_2^*}\boxtimes i_{W_1}}	W_2^*\boxtimes(W_1\boxtimes W_1^*)
	\xrightarrow{\sA_{W_2^*,W_1,W_1^*}}(W_2^*\boxtimes W_1)\boxtimes W_1^*\\
	&\xrightarrow{(1_{W_2^*}\boxtimes f)\boxtimes 1_{W_1^*}}(W_2^*\boxtimes W_2)\boxtimes W_1^*
	\xrightarrow{e_{W_2}\boxtimes 1_{W_1}}\sunit\boxtimes W_1^*\xrightarrow{\sleft_{W_1^*}}W_1^*.\nonumber
	\end{align}
The right dual morphism ${}^*f$ is defined analogously.
\end{defi}

\begin{rema}
 Left duality preserves the parity of morphisms because all structure morphisms including $i_{W_1}$ and $e_{W_2}$ are even. If all objects have left duals, the duality functor $W\mapsto W^*$, $f\mapsto f^*$ is contravariant in the sense that
 \begin{equation}\label{dualofcomp}
  (f_1\circ f_2)^*=(-1)^{\vert f_1\vert \vert f_2\vert} f_2^*\circ f_1^*
 \end{equation}
for parity-homogeneous morphisms $f_1$ and $f_2$. Similar assertions hold for right duals.
\end{rema}

We will need the following important properties of (left) dual morphisms:
\begin{propo}\label{dualmorphchar}
 Suppose $W_1$, $W_2$ are objects in a monoidal (super)category with left duals and $f: W_1\rightarrow W_2$ is a morphism. Then the diagrams
 \begin{equation*}
  \xymatrixcolsep{4pc}
  \xymatrix{
  \sunit \ar[r]^{i_{W_2}} \ar[d]^{i_{W_1}} & W_2\boxtimes W_2^* \ar[d]^{1_{W_2}\boxtimes f^*} \\
  W_1\boxtimes W_1^* \ar[r]^{f\boxtimes 1_{W_1^*}} & W_2\boxtimes W_1^* \\
  }\,\,\,\mathrm{and}\,\,\,
  \xymatrixcolsep{4pc}
  \xymatrix{
  W_2^*\boxtimes W_1 \ar[r]^{1_{W_2^*}\boxtimes f} \ar[d]^{f^*\boxtimes 1_{W_1}} & W_2^*\boxtimes W_2 \ar[d]^{e_{W_2}} \\
  W_1^*\boxtimes W_1 \ar[r]^{e_{W_1}} & \sunit\\
  }
 \end{equation*}
commute.
\end{propo}
\begin{proof}
The proof of $(1_{W_2}\boxtimes f^*)\circ i_{W_2}=(f\boxtimes 1_{W_1^*})\circ i_{W_1}$ is as follows:
\begin{align*}
	\begin{matrix}
		\begin{tikzpicture}[line width=0.5 pt]
			\node (u) at (0.25,0) {$\scriptstyle{\sunit}$};
			\node (w2f) at (0,4) {$\scriptstyle{W_2}$};
			\node (w1f) at (1.8,4) {$\scriptstyle{W_1^*}$};
			\node (f) at (1.3, 2.5) [draw] {$\scriptstyle{f}$};
			\draw [line width=1pt] (w2f.270) to[out=down,in=up] (0,1.25) to[out=down,in=west] (0.25,0.75) to[out=east,in=down] (0.5, 1.25) to[out=up,in=down] (0.5,2.5) to[out=up,in=west] (0.9,3) to[out=east,in=up] (f.90);
			\draw [line width=1pt] (f.270) to[out=down,in=west] (1.55,1.9) to[out=east,in=down] (1.8,2.5) to[out=up,in=down] (w1f.270);
			\draw [dashed] (u.90) to (0.25,0.75);
			\draw [dashed] (0.9,3) to[out=up,in=west] (1.8,3.5);
			\draw [dashed] (0.5,1.4) to[out=east,in=down] (1.55,1.9);
			\node at (-0.25,0.9) {$\scriptstyle{W_2}$};
			\node at (0.75,0.9) {$\scriptstyle{W_2^*}$};
			\node at (1.2,1.9) {$\scriptstyle{W_1}$};
			\node at (2,1.9) {$\scriptstyle{W_1^*}$};
		\end{tikzpicture}
	\end{matrix}
	=
	\begin{matrix}
		\begin{tikzpicture}[line width=0.5 pt]
			\node (u) at (0.25,0) {$\scriptstyle{\sunit}$};
			\node (w2f) at (0,4) {$\scriptstyle{W_2}$};
			\node (w1f) at (1.8,4) {$\scriptstyle{W_1^*}$};
			\node (f) at (1.3, 2.5) [draw] {$\scriptstyle{f}$};
			\draw [line width=1pt] (w2f.270) to[out=down,in=up] (0,1.5) to[out=down,in=west] (0.25,1) to[out=east,in=down] (0.5, 1.5) to[out=up,in=down] (0.5,2.5) to[out=up,in=west] (0.9,3) to[out=east,in=up] (f.90);
			\draw [line width=1pt] (f.270) to[out=down,in=west] (1.55,1.9) to[out=east,in=down] (1.8,2.5) to[out=up,in=down] (w1f.270);
			\draw [dashed] (u.90) to (0.25,1);
			\draw [dashed] (0.9,3) to[out=up,in=east] (0,3.5);
			\draw [dashed] (0.25,0.5) to[out=east,in=down] (1.55,1.9);
			\node at (-0.25,1.2) {$\scriptstyle{W_2}$};
			\node at (0.75,1.2) {$\scriptstyle{W_2^*}$};
			\node at (1.2,1.9) {$\scriptstyle{W_1}$};
			\node at (2,1.9) {$\scriptstyle{W_1^*}$};
		\end{tikzpicture}
	\end{matrix}
	=
	\begin{matrix}
		\begin{tikzpicture}[line width=0.5 pt]
			\node (u) at (0.25,0) {$\scriptstyle{\sunit}$};
			\node (w2f) at (0,4) {$\scriptstyle{W_2}$};
			\node (w1f) at (1.8,4) {$\scriptstyle{W_1^*}$};
			\node (f) at (1.3, 1.75) [draw] {$\scriptstyle{f}$};
			\draw [line width=1pt] (w2f.270) to[out=down,in=up] (0,2.8) to[out=down,in=west] (0.25,2.2)   to[out=east,in=west] (0.9,3) to[out=east,in=up] (f.90);
			\draw [line width=1pt] (f.270) to[out=down,in=west] (1.55,1.25) to[out=east,in=down] (1.8,2.5) to[out=up,in=down] (w1f.270);
			\draw [dashed] (u.90) to (0.25,2.2);
			\draw [dashed] (0.9,3) to[out=up,in=east] (0,3.5);
			\draw [dashed] (0.25,0.5) to[out=east,in=down] (1.55,1.25);
			\node at (-0.1,2.2) {$\scriptstyle{W_2}$};
			\node at (0.75,2.2) {$\scriptstyle{W_2^*}$};
			\node at (1.2,1.2) {$\scriptstyle{W_1}$};
			\node at (2,1.2) {$\scriptstyle{W_1^*}$};
		\end{tikzpicture}
	\end{matrix}
	=
	\begin{matrix}
		\begin{tikzpicture}[line width=0.5 pt]
			\node (u) at (1.55,0) {$\scriptstyle{\sunit}$};
			\node (w2f) at (0,4) {$\scriptstyle{W_2}$};
			\node (w1f) at (1.8,4) {$\scriptstyle{W_1^*}$};
			\node (f) at (1.3, 1.25) [draw] {$\scriptstyle{f}$};
			\draw [line width=1pt] (w2f.270) to[out=down,in=up] (0,2.8) to[out=down,in=west] (0.25,2.2)   to[out=east,in=west] (0.9,3) to[out=east,in=up] (f.90);
			\draw [line width=1pt] (f.270) to[out=down,in=west] (1.55,0.75) to[out=east,in=down] (1.8,2.5) to[out=up,in=down] (w1f.270);
			\draw [dashed] (u.90) to (1.55,0.75);
			\draw [dashed] (0.9,3) to[out=up,in=east] (0,3.5);
			\draw [dashed] (1.3,1.8) to[out=west,in=down] (0.25,2.2);
			\node at (-0.1,2.2) {$\scriptstyle{W_2}$};
			\node at (0.75,2.2) {$\scriptstyle{W_2^*}$};
			\node at (1.2,0.7) {$\scriptstyle{W_1}$};
			\node at (2.1,0.7) {$\scriptstyle{W_1^*}$};
		\end{tikzpicture}
	\end{matrix}
	=
	\begin{matrix}
		\begin{tikzpicture}[line width=0.5 pt]
			\node (u) at (0.25,0) {$\scriptstyle{\sunit}$};
			\node (w2f) at (0,4) {$\scriptstyle{W_2}$};
			\node (w1f) at (0.5,4) {$\scriptstyle{W_1^*}$};
			\node (f) at (0, 2) [draw] {$\scriptstyle{f}$};
			\draw [line width=1pt] (w2f.270) to[out=down,in=up] (f.90);
			\draw [line width=1pt] (f.270) to[out=down,in=west] (0.25,1) to[out=east,in=down] (0.5,2) to[out=up,in=down] (w1f.270);
			\draw [dashed] (u.90) to (0.25,1);
			\node at (-0.1,1) {$\scriptstyle{W_1}$};
			\node at (0.7,1) {$\scriptstyle{W_1^*}$};
		\end{tikzpicture}
	\end{matrix}
\end{align*}
The first and third equalities follow from properties of unit and associativity isomorphisms. In the second equality, we use evenness of $i_{W_2}$ to avoid a sign factor when we exchange its order with $i_{W_1}$ and $f$. The last equality follows from the rigidity of $W_2$.

The proof of $e_{W_1}\circ(f^*\boxtimes 1_{W_1})=e_{W_2}\circ(1_{W_2^*}\boxtimes f)$ is similar and uses the rigidity of $W_1$.
\end{proof}
\begin{rema}
 One can show that the left dual morphism $f^*$ is characterized by the commutativity of either diagram in the preceding proposition. This property allows a relatively easy proof of \eqref{dualofcomp}.
\end{rema}

Now we show that the induction functor $\cF: \sC\rightarrow\repA$ preserves duals; recall the natural isomorphism $f$, its inverse $g$, and the isomorphism $\varphi: \cF(\sunit)\rightarrow A$ from the proof of Theorem \ref{thm:inductionfucntor}:
\begin{propo}[{\cite[Exercise 2.10.6]{EGNO}, \cite[Lemma 1.16]{KO}}]
	\label{lem:induceduals}
	If $W$ is an object of $\sC$ with left dual $(W^*, e_W, i_W)$, then $(\cF(W^*), \widetilde{e}_W, \widetilde{i}_W)$ is a left dual of $\cF(W)$ in $\repA$, where
	\begin{equation*}
	 \widetilde{e}_W=\varphi\circ\cF(e_W)\circ g_{W^*,W}
	\end{equation*}
	and
	\begin{equation*}
	 \widetilde{i}_W=f_{W,W^*}\circ\cF(i_W)\circ\varphi^{-1}.
	\end{equation*}
An analogous result holds for right duals.
\end{propo}
\begin{proof}
The following diagram whose left side is obtained by applying $\cF$ to \eqref{eqn:rigid1} commutes because $\cF$ is a tensor functor:
\begin{align*}
 \xymatrixcolsep{5pc}
 \xymatrix{
  & \cF(W) \ar[ld]_{\cF(\sleft_{W}^{-1})} \ar[rd]^{(l^A_{\cF(W)})^{-1}} & \\
  \cF(\sunit\boxtimes W) \ar[r]^{f_{\sunit, W}} \ar[d]^{\cF(i_W\boxtimes 1_W)} & \cF(\sunit)\boxtimes_A \cF(W) \ar[r]^{\varphi\boxtimes_A 1_{\cF(W)}} \ar[d]^{\cF(i_W)\boxtimes_A 1_{\cF(W)}} & A\boxtimes_A \cF(W) \ar[d]^{\widetilde{i}_W\boxtimes_A 1_{\cF(W)}} \\
  \cF((W\boxtimes W^*)\boxtimes W) \ar[r]^{f_{W\boxtimes W^*, W}} \ar[d]^{\cF(\sA_{W,W^*,W}^{-1})} & \cF(W\boxtimes W^*)\boxtimes_A \cF(W) \ar[r]^(.47){f_{W,W^*}\boxtimes_A 1_{\cF(W)}} & (\cF(W)\boxtimes_A\cF(W^*))\boxtimes_A\cF(W) \ar[d]_{(\cA^A_{\cF(W),\cF(W^*),\cF(W)})^{-1}} \\
  \cF(W\boxtimes(W^*\boxtimes W)) \ar[r]^{f_{W, W^*\boxtimes W}} \ar[d]^{\cF(1_W\boxtimes e_W)} & \cF(W)\boxtimes_A\cF(W^*\boxtimes W) \ar[r]^(.47){1_{\cF(W)}\boxtimes_A f_{W^*,W}} \ar[d]^{1_{\cF(W)}\boxtimes_A \cF(e_W)} & \cF(W)\boxtimes_A(\cF(W^*)\boxtimes_A\cF(W)) \ar[d]^{1_{\cF(W)}\boxtimes_A \widetilde{e}_W} \\
  \cF(W\boxtimes\sunit) \ar[rd]_{\cF(\sright_W)} \ar[r]^{f_{W,\sunit}} & \cF(W)\boxtimes_A\cF(\sunit) \ar[r]^{1_{\cF(W)}\boxtimes_A\varphi} & \cF(W)\boxtimes_A A \ar[ld]^{r^A_{\cF(W)}} \\
  & \cF(W) & \\
 }
\end{align*}
Since the left side equals $\cF(1_W)=1_A\boxtimes 1_W = 1_{A\boxtimes W} = 1_{\cF(W)}$, 
	the same holds for the right side. The second duality axiom \eqref{eqn:rigid2} follows from a similar commutative diagram.

The analogous assertion for right duals follows similarly.
\end{proof}

We also show that the induced dual of an induced object in $\repzA$ is an object in $\repzA$:
\begin{lemma}\label{lem:dualinRep0}
Let $W$ be an object of $\sC$ such that $\cF(W)$ is an object of $\repzA$. If $(W^*,e_W,i_W)$ is a left dual of $W$ in $\sC$,
then $\cF(W^*)$ is an object of $\rep^0 A$; a similar result holds for right duals.
\end{lemma}
\begin{proof}
	By Proposition \ref{propo:rep0inductioncriteria}, we need to prove $\cM_{A,W^*}=1_{A\boxtimes W^*}$.
	Proposition \ref{propo:rep0inductioncriteria} implies that $\cM_{A,W}=1_{A\boxtimes W}$, which is equivalent to $\sR_{A,W} = \sR^{-1}_{W,A}$.
	Using \cite[Lemma 8.9.1]{EGNO}, which holds in any braided monoidal supercategory,
	we have (with $\sA$ denoting suitable compositions of associativity isomorphisms):
	\begin{align*}
	\sR_{W^*,A} &= \sleft_{A\boxtimes W^*}\circ(e_W\boxtimes 1_{A\boxtimes W^*})\circ\sA\circ(1_{W^*}\boxtimes \sR_{W,A}^{-1}\boxtimes 1_{W^*})\circ \sA\circ(1_{W^*\boxtimes A}\boxtimes i_W)
	\circ \sright_{W^*\boxtimes A}^{-1}\\
	\sR^{-1}_{A,W^*} &=\sleft_{A\boxtimes W^*}\circ (e_W\boxtimes 1_{A\boxtimes W^*})\circ\sA\circ(1_{W^*}\boxtimes \sR_{A,W}\boxtimes 1_{W^*})\circ \sA\circ(1_{W^*\boxtimes A}\boxtimes i_W)
	\circ \sright_{W^*\boxtimes A}^{-1}.
	\end{align*}
	Since $\sR_{A,W}=\sR_{W,A}^{-1}$, we get $\sR_{W^*,A}=\sR^{-1}_{A,W^*}$ and thus $\cM_{A,W^*}=1_{A\boxtimes W^*}$.
The proof for right duals is similar.
\end{proof}

Now we examine the relations between $\repzA$, the induction functor, and a twist on $\sC$ (we recall the notion of a twist in the following definition). If $\sC$ is a braided tensor category of modules for a vertex operator (super)algebra, then $\sC$ has a twist given by $\theta_W=e^{2\pi i L_W(0)}$ for any module $W$ in $\sC$.
\begin{defi}[{see \cite[Definition 8.10.1]{EGNO}}]
	A \textit{twist}, or \textit{balancing transformation}, is an even natural automorphism $\theta$ of the identity functor on $\sC$ which satisfies $\theta_{\sunit}=1_{\sunit}$ and the \textit{balancing equation}
	\[ \theta_{W_1\boxtimes W_2}=\cM_{W_1,W_2}\circ(\theta_{W_1}\boxtimes \theta_{W_2})\]
	for all objects $W_1$, $W_2$.
	
	A twist $\theta$ is a \textit{ribbon structure} if $\sC$ is rigid 
	and $\theta$ satisfies $(\theta_W)^* = \theta_{W^*}$ for all objects $W$.
\end{defi}
\begin{rema}
 In a ribbon category, left and right duals can be taken to be the same, so from now on, we shall omit references to right duals.
\end{rema}

In the case of a vertex operator (super)algebra extension $V\subseteq A$, we have $L_V(0)=L_A(0)\vert_{V}$, so that the braided monoidal (super)category $\repzA$ has a natural twist, provided each $\theta_W$ for $W$ in $\repzA$ is a morphism in $\repzA$. This is shown in the following lemma under a mild condition on $A$: 
\begin{lemma}[{see also \cite[Theorem 1.17(1)]{KO}}]\label{twistonrep0A}
Suppose $\sC$ has a twist $\theta$ such that $\theta_A=1_A$. Then an object $(W,\mu_W)$ of $\repA$ is an object of $\repzA$ if and only if $\theta_W\in\mathrm{End}_{\repA}\,W$. Moreover,
 $\theta$ satisfies
	$$\theta_{W_1\boxtimes_A W_2} = \cM^{A}_{W_1,W_2}\circ(\theta_{W_1}\boxtimes_A \theta_{W_2})$$
for objects $W_1$, $W_2$ in $\repzA$. In particular, $\theta$ is a twist on $\repzA$.
\end{lemma}
\begin{proof}
	By definition, $(W,\mu_W)$ is an object of $\mathrm{Rep}^0 A$ if and only if $\mu_W\circ \sR_{W,A}\circ\sR_{A,W} = \mu_W$.
	We have 
	$$\mu_W \circ \sR_{W,A}\circ\sR_{A,W} = \mu_W\circ\theta_{A\boxtimes W}\circ (\theta_A^{-1}\boxtimes \theta_W^{-1})=
	\theta_W\circ \mu_W\circ (1_A\boxtimes \theta_W^{-1}),
	$$
	where we have used the naturality of $\theta$ and $\theta_A=1_A$ in the second equality. Hence, $\mu_W=\mu_W \circ \sR_{W,A}\circ\sR_{A,W}$ if and only if 
	$\mu_W\circ(1_A\boxtimes \theta_W) = \theta_W\circ \mu_W$.
	
	Now the second assertion follows using the diagrams: 
	\begin{align*}
	\xymatrixcolsep{4pc}
	\xymatrix{
		W_1\boxtimes W_2 \ar[r]^{\theta_{W_1}\boxtimes \theta_{W_2}} \ar[d]^{\eta_{W_1,W_2}}
		 & W_1\boxtimes W_2 \ar[r]^{\cM_{W_1,W_2}}\ar[d]^{\eta_{W_1,W_2}} & W_1\boxtimes W_2 \ar[d]^{\eta_{W_1,W_2}} \\
		W_1\boxtimes_A W_2 \ar[r]^{\theta_{W_1}\boxtimes_A \theta_{W_2}} & W_1\boxtimes_A W_2 \ar[r]^{\cM_{W_1,W_2}^{A}} & W_1\boxtimes_A W_2 
		} , 
	\xymatrix{
		W_1\boxtimes W_2 \ar[r]^{\theta_{W_1\boxtimes W_2}} \ar[d]^{\eta_{W_1,W_2}} & W_1\boxtimes W_2 \ar[d]^{\eta_{W_1,W_2}} \\
		W_1\boxtimes_A W_2 \ar[r]^{\theta_{W_1\boxtimes_A W_2}} & W_1\boxtimes_A W_2 
		}		
	\end{align*}
	where the first square commutes by the definition of tensor product of morphisms in $\repA$, since $\theta_{W_1}$ and $\theta_{W_2}$ are morphisms in $\rep^0A$; the second square commutes by \eqref{rep0Abraidingdef}; and the last square commutes by the naturality of $\theta$.
	Since the top rows of the two diagrams are equal, and since $\eta_{W_1,W_2}$ is an epimorphism, the bottom rows
	are also equal.
\end{proof}

\begin{corol}\label{cor:inducedtheta}
Suppose $\sC$ has a twist $\theta$ such that $\theta_A=1_A$. If $W$ is an object of $\sC$ such that $\cF(W)$ is an object of $\repzA$, then $\cF(\theta_W)=\theta_{\cF(W)}$ as morphisms in $\rep^0 A$.
\end{corol}
\begin{proof}
By the previous lemma, $\cF(\theta_W)$ and $\theta_{\cF(W)}$ are morphisms in $\repzA$. Then by Proposition \ref{propo:rep0inductioncriteria},
\begin{equation*}
 \theta_{\cF(W)}=\theta_{A\boxtimes W}=\cM_{A,W}\circ(\theta_A\boxtimes \theta_W)=1_A\boxtimes\theta_W=\cF(\theta_W)
\end{equation*}
as required.
\end{proof}

In ribbon categories, there is a natural notion of trace of an endomorphism:
\begin{defi}\label{def:tr}
	The \textit{trace} of a morphism $f: W\rightarrow W$ in a ribbon (super)category with ribbon structure $\theta$ is the morphism $\tr(f)\in\Endo(\sunit)$ given by the composition
	\[
	\sunit\xrightarrow{i_W} W\boxtimes W^*\xrightarrow{(\theta_W\circ f)\boxtimes 1_{W^*}}	W\boxtimes W^*
	\xrightarrow{\sR_{W,W^*}} W^*\boxtimes W\xrightarrow{e_W}\sunit.
	\]
\end{defi}	

The trace of a morphism in a ribbon (super)category satisfies the following expected super-symmetry:
\begin{propo}
 If $f: W_2\rightarrow W_1$ and $g: W_1\rightarrow W_2$ are parity-homogeneous morphisms in a ribbon (super)category, then
 \begin{equation*}
  \tr(f\circ g)=(-1)^{\vert f\vert \vert g\vert} \tr(g\circ f).
 \end{equation*}
 \end{propo}
\begin{proof}
We use the following diagrams in the proof:
\begin{align*}
	\begin{matrix}
		\begin{tikzpicture}[line width=0.5pt]
			\node (u) at (0.4,-0.3) {$\scriptstyle{\sunit}$};
			\node (uf) at (0.4,5.3) {$\scriptstyle{\sunit}$};
			\node (t) at (0,3) [draw] {$\scriptstyle{\theta_{W_1}}$};
			\draw [line width=1pt]
			(t.270) 
			to[out=down, in=up] (0,1)
			to[out=down, in=west] (0.4,0.5)
			to[out=east, in=down] (0.8,1)
			to[out=up, in=down] (0.8,3.2)
			to[out=up, in=down] (0,4)
			to[out=up, in=west] (0.4,4.5);
			\draw [line width=3pt,white, double=black, double distance=1pt]
			(0.4,4.5)
			to[out=east, in=up] (0.8,4)
			to[out=down, in=up] (t.90);
			\node at (-0.1, 0.5) {$\scriptstyle{W_1}$};
			\node at (1, 0.5) {$\scriptstyle{W_1^*}$};
			\node at (-0.3, 4.2) {$\scriptstyle{W_1^*}$};
			\node at (1, 4.2) {$\scriptstyle{W_1}$};
			\draw [dashed] (u.90) to (0.4,0.5);
			\draw [dashed] (0.4,4.5) to (uf.270);
			\node (g) at (0,1.5) [draw,fill=white] {$\scriptstyle{g}$};
			\node (f) at (0,2) [draw,fill=white] {$\scriptstyle{f}$};
		\end{tikzpicture}
	\end{matrix}
	=
	\begin{matrix}
		\begin{tikzpicture}[line width=0.5pt]
			\node (u) at (0.4,-0.3) {$\scriptstyle{\sunit}$};
			\node (uf) at (0.4,5.3) {$\scriptstyle{\sunit}$};
			\node (t) at (0,2.5) [draw] {$\scriptstyle{\theta_{W_2}}$};
			\draw [line width=1pt]
			(t.270) 
			to[out=down, in=up] (0,1)
			to[out=down, in=west] (0.4,0.5)
			to[out=east, in=down] (0.8,1)
			to[out=up, in=down] (0.8,2.6)
			to[out=up, in=down] (0,3.7)
			to[out=up, in=west] (0.4,4.5);
			\draw [line width=3pt,white, double=black, double distance=1pt]
			(0.4,4.5)
			to[out=east, in=up] (0.8,3.6)
			to[out=down, in=up] (t.90);
			\node at (-0.1, 0.5) {$\scriptstyle{W_1}$};
			\node at (1, 0.5) {$\scriptstyle{W_1^*}$};
			\node at (-0.3, 4.2) {$\scriptstyle{W_1^*}$};
			\node at (1, 4.2) {$\scriptstyle{W_1}$};
			\draw [dashed] (u.90) to (0.4,0.5);
			\draw [dashed] (0.4,4.5) to (uf.270);
			\node (g) at (0,1.5) [draw,fill=white] {$\scriptstyle{g}$};
			\node (f) at (0.8,3.6) [draw,fill=white] {$\scriptstyle{f}$};
		\end{tikzpicture}
	\end{matrix}
	=
	\begin{matrix}
		\begin{tikzpicture}[line width=0.5pt]
			\node (u) at (0.4,-0.3) {$\scriptstyle{\sunit}$};
			\node (uf) at (0.4,5.3) {$\scriptstyle{\sunit}$};
			\node (t) at (0,2.5) [draw] {$\scriptstyle{\theta_{W_2}}$};
			\draw [line width=1pt]
			(t.270) 
			to[out=down, in=up] (0,1)
			to[out=down, in=west] (0.4,0.5)
			to[out=east, in=down] (0.8,1)
			to[out=up, in=down] (0.8,2.6)
			to[out=up, in=down] (0,3.7)
			to[out=up, in=west] (0.4,4.5);
			\draw [line width=3pt,white, double=black, double distance=1pt]
			(0.4,4.5)
			to[out=east, in=up] (0.8,3.6)
			to[out=down, in=up] (t.90);
			\node at (-0.1, 0.5) {$\scriptstyle{W_1}$};
			\node at (1, 0.5) {$\scriptstyle{W_1^*}$};
			\node at (-0.3, 4.3) {$\scriptstyle{W_2^*}$};
			\node at (1, 4.3) {$\scriptstyle{W_2}$};
			\draw [dashed] (u.90) to (0.4,0.5);
			\draw [dashed] (0.4,4.5) to (uf.270);
			\node (g) at (0,1.5) [draw,fill=white] {$\scriptstyle{g}$};
			\node (f) at (0,3.8) [draw,fill=white] {$\scriptstyle{f^*}$};
		\end{tikzpicture}
	\end{matrix}
	=
	\scriptstyle{(-1)^{|f||g|}}
	\begin{matrix}
		\begin{tikzpicture}[line width=0.5pt]
			\node (u) at (0.4,-0.3) {$\scriptstyle{\sunit}$};
			\node (uf) at (0.4,5.3) {$\scriptstyle{\sunit}$};
			\node (t) at (0,2.5) [draw] {$\scriptstyle{\theta_{W_2}}$};
			\draw [line width=1pt]
			(t.270) 
			to[out=down, in=up] (0,1)
			to[out=down, in=west] (0.4,0.5)
			to[out=east, in=down] (0.8,1)
			to[out=up, in=down] (0.8,2.6)
			to[out=up, in=down] (0,3.7)
			to[out=up, in=west] (0.4,4.5);
			\draw [line width=3pt,white, double=black, double distance=1pt]
			(0.4,4.5)
			to[out=east, in=up] (0.8,3.6)
			to[out=down, in=up] (t.90);
			\node at (-0.1, 0.5) {$\scriptstyle{W_1}$};
			\node at (1, 0.5) {$\scriptstyle{W_1^*}$};
			\node at (-0.3, 4.2) {$\scriptstyle{W_2^*}$};
			\node at (1, 4.2) {$\scriptstyle{W_2}$};
			\draw [dashed] (u.90) to (0.4,0.5);
			\draw [dashed] (0.4,4.5) to (uf.270);
			\node (g) at (0,2) [draw,fill=white] {$\scriptstyle{g}$};
			\node (f) at (0.8,1.5) [draw,fill=white] {$\scriptstyle{f^*}$};
		\end{tikzpicture}
	\end{matrix}
	=
	\scriptstyle{(-1)^{|f||g|}}
	\begin{matrix}
		\begin{tikzpicture}[line width=0.5pt]
			\node (u) at (0.4,-0.3) {$\scriptstyle{\sunit}$};
			\node (uf) at (0.4,5.3) {$\scriptstyle{\sunit}$};
			\node (t) at (0,3) [draw] {$\scriptstyle{\theta_{W_2}}$};
			\draw [line width=1pt]
			(t.270) 
			to[out=down, in=up] (0,1)
			to[out=down, in=west] (0.4,0.5)
			to[out=east, in=down] (0.8,1)
			to[out=up, in=down] (0.8,3.2)
			to[out=up, in=down] (0,4)
			to[out=up, in=west] (0.4,4.5);
			\draw [line width=3pt,white, double=black, double distance=1pt]
			(0.4,4.5)
			to[out=east, in=up] (0.8,4)
			to[out=down, in=up] (t.90);
			\node at (-0.1, 0.5) {$\scriptstyle{W_2}$};
			\node at (1, 0.5) {$\scriptstyle{W_2^*}$};
			\node at (-0.3, 4.2) {$\scriptstyle{W_2^*}$};
			\node at (1, 4.2) {$\scriptstyle{W_2}$};
			\draw [dashed] (u.90) to (0.4,0.5);
			\draw [dashed] (0.4,4.5) to (uf.270);
			\node (g) at (0,2) [draw,fill=white] {$\scriptstyle{g}$};
			\node (f) at (0,1.5) [draw,fill=white] {$\scriptstyle{f}$};
		\end{tikzpicture}
	\end{matrix}
\end{align*}
The first equality uses naturality of $\theta$ and the braiding, while the second and last use Proposition \ref{dualmorphchar}. 
The third equality uses the super interchange law \eqref{exchange}.
\end{proof}

\begin{corol}\label{conjtr}
 If $f: W_1\rightarrow W_1$ is a parity-homogenous endomorphism in a ribbon (super)category and $g: W_2\rightarrow W_1$ is a parity-homogeneous isomorphism, then
 \begin{equation*}
  \tr(g^{-1}\circ f\circ g)=(-1)^{\vert g\vert(\vert g\vert +\vert f\vert)} \tr(f).
 \end{equation*}
\end{corol}

Recall from Lemma \ref{twistonrep0A} that if $\sC$ has a twist $\theta$ such that $\theta_A=1_A$, then $\theta$ also defines a twist on $\repzA$. Moreover, if $\sC$ is rigid and $\theta$ is a ribbon structure on $\sC$, Proposition \ref{lem:induceduals} and Lemma \ref{lem:dualinRep0} imply that the subcategory of induced objects in $\repzA$ is rigid as well. In fact the subcategory of induced modules in $\repzA$ is ribbon:
\begin{propo}\label{prop:inducedrepzAribbon}
 If $\sC$ is a ribbon category with twist $\theta$ such that $\theta_A=1_A$, then the induced objects in $\repzA$ form a ribbon category.
\end{propo}
\begin{proof}
 We just need to show that if $\cF(W)$ is an induced object in $\repzA$, with dual $(\cF(W^*), \widetilde{e}_W, \widetilde{i}_W)$ as in Proposition \ref{lem:induceduals}, then $\theta_{\cF(W)}^*=\theta_{\cF(W^*)}$. To prove this, one first uses the definition of dual morphisms and a diagram similar to the one in the proof of Proposition \ref{lem:induceduals} (more precisely, the corresponding commutative diagram for $\cF(W^*)$), to show that $\theta_{\cF(W)}^*=\cF(\theta_W^*)$. Then because $\sC$ is a ribbon category and using Corollary \ref{cor:inducedtheta},
 \begin{equation*}
  \cF(\theta_W^*)=\cF(\theta_{W^*})=\theta_{\cF(W^*)}
 \end{equation*}
as required.
%
%
%
%
\end{proof}

Now we can relate traces in $\sC$ to traces in the ribbon category of induced objects in $\repzA$: 
\begin{lemma}\label{lem:indtr}
	Suppose $\sC$ is a ribbon category with twist $\theta$ such that $\theta_A=1_A$. If $W$ is an object of $\sC$ such that $\cF(W)$ is an object of $\repzA$ and if $f\in\Endo(W)$, then
	$$\tr(\cF(f))=\varphi\circ \cF(\tr(f))\circ \varphi^{-1},$$
	recalling the isomorphism $\varphi: \cF(\sunit)\rightarrow A$ of Theorem \ref{thm:inductionfucntor}.
\end{lemma}
\begin{proof}
Consider the diagram
\begin{equation*}
 \xymatrixcolsep{6pc}
 \xymatrix{
 \cF(\sunit) \ar[r]^{\varphi} \ar[d]^{\cF(i_W)} & A \ar[d]^{\widetilde{i}_{\cF(W)}} \\
 \cF(W\boxtimes W^*) \ar[r]^{f_{W,W^*}} \ar[d]^{\cF((\theta_W\circ f)\boxtimes 1_{W^*})} & \cF(W)\boxtimes_A\cF(W^*) \ar[d]^{(\theta_{\cF(W)}\circ\cF(f))\boxtimes_A 1_{\cF(W^*)}} \\
\cF(W\boxtimes W^*) \ar[r]^{f_{W,W^*}} \ar[d]^{\cF(\sR_{W,W^*})} & \cF(W)\boxtimes_A\cF(W^*) \ar[d]^{\cR^A_{\cF(W),\cF(W^*)}} \\
\cF(W^*\boxtimes W) \ar[r]^{f_{W^*,W}} \ar[d]^{\cF(e_W)} & \cF(W^*)\boxtimes_A \cF(W) \ar[d]^{\widetilde{e}_{W}} \\
\cF(\sunit) \ar[r]^{\varphi} & A\\
 }
\end{equation*}
The top and bottom squares commute by the definitions of $\widetilde{i}_{\cF(W)}$ and $\widetilde{e}_{\cF(W)}$ (recall Proposition \ref{lem:induceduals}). The second square commutes because $f$ is a natural isomorphism and because 
\begin{equation*}
 \cF(\theta_W\circ f)=\cF(\theta_W)\circ\cF(f)=\theta_{\cF(W)}\circ\cF(f)
\end{equation*}
(recall Corollary \ref{cor:inducedtheta}). Also the third square commutes by Theorem \ref{thm:Fisbraidedtensor}. Now the proposition follows because the left column of the diagram is $\cF(\tr(f))$ while the right column is $\tr(\cF(f))$.
\end{proof}

The results presented so far in this subsection build up to the final result which shows how induction relates categorical $S$-matrices in $\sC$ and $\repzA$,  provided $\sC$ and the subcategory of induced objects in $\repzA$ are both ribbon categories:
\begin{defi} 
For objects $W_1$ and $W_2$ of a ribbon category, the \textit{$S^\hopflink$-matrix} is $S^\hopflink_{W_1,W_2}=\tr(\cM_{W_1,W_2})$.
\end{defi}

\begin{theo}\label{prop:liftingS}
Suppose $\sC$ is a ribbon category with twist $\theta$ such that $\theta_A=1_A$. Then for objects $W_1$ and $W_2$ of $\sC$ such that $\cF(W_1)$ and $\cF(W_2)$ are objects in $\repzA$,
 $$S^\hopflink_{\cF(W_1),\cF(W_2)}=\varphi\circ\cF(S^\hopflink_{W_1,W_2})\circ\varphi^{-1}.$$
\end{theo}
\begin{proof}
We calculate
\begin{align*}
\varphi\circ\cF(S^\hopflink_{W_1,W_2})\circ\varphi^{-1}
&=\varphi\circ\cF(\tr(\cM_{W_1,W_2}))\circ\varphi^{-1}=\tr(\cF(\cM_{W_1,W_2}))\nonumber\\
&=\tr(\cM^A_{\cF(W_1),\cF(W_2)})=S^\hopflink_{\cF(W_1),\cF(W_2)},
\end{align*}
using Lemma \ref{lem:indtr}, Corollary \ref{Fmonodromy}, and Corollary \ref{conjtr}.
\end{proof}

\section{Vertex tensor categories}\label{sec:vtc}

The goal of this section is to interpret the abstract categorical constructions of the previous section in the vertex operator algebra context. As shown in \cite{HKL, CKL}, a vertex operator (super)algebra extension $V\subseteq A$ amounts to a (super)algebra in a braided tensor category of $V$-modules. In this context, the category $\repzA$ of local $A$-modules recalled in the previous section agrees (as categories) with the category of modules for the vertex operator algebra $A$ (see \cite{HKL}). Our main result in this section is that in fact these two categories agree as braided monoidal categories, where the category of vertex-algebraic $A$-modules has the braided monoidal category structure of \cite{HLZ8}. As a consequence, all the tensor-categorical constructions and results of Section \ref{sec:TensCats}, especially the induction functor, apply naturally to the vertex-algebraic setting. We start this section with a discussion of the basic definitions and constructions of vertex tensor category theory before moving on to the main results. In particular, we include a discussion of the natural superalgebra generalization of the vertex tensor category constructions in \cite{HLZ8} and show how these agree with the braided monoidal supercategory structure on $\repzA$ when $A$ is a vertex operator superalgebra extension of a vertex operator algebra $V$.

\subsection{Modules and intertwining operators}
\label{subsec:intw}

We are interested in extensions $V\subseteq A$ where $V$ is an ordinary $\ZZ$-graded vertex operator algebra and $A$ is one of the following four possibilities:
\begin{enumerate}
 \item A $\ZZ$-graded vertex operator algebra,
 
 \item A $\ZZ$-graded vertex operator superalgebra,
 
 \item A $\frac{1}{2}\ZZ$-graded vertex operator algebra, or
 
 \item A $\frac{1}{2}\ZZ$-graded vertex operator superalgebra.
\end{enumerate}
For the notion of ($\ZZ$-graded) vertex operator algebra $(V, Y, \unit, \omega)$ and corresponding notation, see for instance \cite{FLM, FHL, LL}. The definition of  $\frac{1}{2}\ZZ$-graded vertex operator algebra is exactly the same, except that half-integral conformal weights are allowed; quantum Hamiltonian reduction \cite{KRW, KW-qH} is a prominent source of $\frac{1}{2}\ZZ$-graded \VOAs, including the Bershadsky-Polyakov algebra discussed in Section \ref{sec:BP} below.

For the notion of vertex operator superalgebra, see for instance \cite{DL, Xu, WLi-thesis}. Since slightly variant definitions of vertex operator superalgebra exist in the literature, let us clarify what we mean by this term. For us, a vertex operator superalgebra $V$ has a parity decomposition $V^\even\oplus V^\odd$ as well as typically a second $\ztwo$-grading coming from the conformal weight $\frac{1}{2}\ZZ$-grading: $V=V_{+}\oplus V_{-}$ where
\begin{equation*}
 V_{+}=\bigoplus_{n\in\ZZ} V_{n}\hspace{2em}\mathrm{and}\hspace{2em}V_{-}=\bigoplus_{n\in\ZZ} V_{n+\frac{1}{2}}.
\end{equation*}
In this section, we will not assume that these two $\ztwo$-gradings are the same; in fact, one could be trivial and the other not, yielding the possibilities of $\ZZ$-graded vertex operator superalgebra and $\frac{1}{2}\ZZ$-graded vertex operator algebra. In any case, the vacuum and conformal vector of $V$ are assumed even in both senses ($\unit, \omega\in V^\even\cap V_+$) and the vertex operator $Y$ is even with respect to both $\ztwo$-gradings. This means that the involutions $P_V=1_{V^\even}\oplus (-1_{V^\odd})$ (the \textit{parity involution}) and $\theta_V=1_{V_{+}}\oplus (-1_{V_{-}})=e^{2\pi i L(0)}$ (the \textit{twist}) are vertex operator superalgebra automorphisms. For a parity-homogeneous element $v$ of a vertex operator superalgebra $V$, we use $\vert v\vert\in\ztwo$ to denote the parity of $v$.

Since all four possibilities for an extended algebra that we shall consider are special cases of $\frac{1}{2}\ZZ$-graded vertex operator superalgebra (where both, one, or none of $V^\odd$ and $V_{-}$ are zero), in this and the next two subsections, we will fix a $\frac{1}{2}\ZZ$-graded vertex operator superalgebra $V$. Later, in Section \ref{subsec:complexIntw} where we begin specifically studying extensions, we will take $V$ to be a $\ZZ$-graded vertex operator algebra and $A$ to be a $\frac{1}{2}\ZZ$-graded superalgebra. The modules for a vertex operator superalgebra that we will consider fall in the following class (see for instance \cite{HLZ1} for the vertex operator algebra case): 
\begin{defi}\label{defi:genlmodule} A \emph{grading-restricted generalized $V$-module} is a $\CC$-graded superspace 
	\begin{equation*}
	 W=\bigoplus_{i\in\ztwo} W^i = \bigoplus_{h\in\CC} W_{[h]}
	\end{equation*}
(graded by \textit{parity} and \textit{conformal weights}) equipped with a \textit{vertex operator}
	\begin{align*}
	 Y_W: V\otimes W & \rightarrow W[[x,x^{-1}]]\nonumber\\
	  v\otimes w & \mapsto Y_W(v, x)w=\sum_{n\in\ZZ} v_n w\,x^{-n-1}
	\end{align*}
such that the following axioms hold:
\begin{enumerate}
\item Grading compatibility: For $i=\even, \odd$, $W^i=\bigoplus_{h\in\CC} W^i_{[h]}$, where $W^i_{[h]}=W^i\cap W_{[h]}$.

 \item The \textit{grading restriction conditions}: For any $h\in\CC$, $W_{[h]}$ is finite dimensional and $W_{[h+n]}=0$ for $n\in\ZZ$ sufficiently negative.
 
 \item Evenness of the vertex operator: For any $v\in V^i$ and $w\in W^j$, $i,j\in\ztwo$, $Y_W(v,x)w\in W^{i+j}[[x,x^{-1}]]$.
 
 \item \textit{Lower truncation}: For any $v\in V$ and $w\in W$, $v_n w=0$ for $n$ sufficiently positive, that is, $Y_W(v,x)w\in W((x))$.
 
 \item The \textit{vacuum property}: $Y_W(\mathbf{1}, x)=1_W$.
 
 \item The \textit{Jacobi identity}: For any parity-homogeneous $u, v\in V$,
 \begin{align*}
  x_0^{-1}\delta\left(\dfrac{x_1-x_2}{x_0}\right) Y_W(u, x_1)Y_W(v, x_2) & -(-1)^{\vert u\vert \vert v\vert} x_0^{-1}\delta\left(\dfrac{-x_2+x_1}{x_0}\right) Y_W(v,x_2)Y_W(u,x_1)\nonumber\\
  & =x_2^{-1}\delta\left(\dfrac{x_1-x_0}{x_2}\right) Y_W(Y(u,x_0)v,x_2).
 \end{align*}

 \item The Virasoro algebra properties: If we write $Y_W(\omega, x)=\sum_{n\in\ZZ} L(n) x^{-n-2}$, then
 \begin{equation*}
  [L(m), L(n)]=(m-n)L(m+n)+\dfrac{m^3-m}{12}\delta_{m+n,0} c
 \end{equation*}
where $c$ is the central charge of $V$. Moreover, for any $h\in\CC$, $W_{[h]}$ is the generalized eigenspace of $L(0)$ with generalized eigenvalue $h$.

\item The \textit{$L(-1)$-derivative property}: For any $v\in V$,
\begin{equation*}
 Y_W(L(-1)v,x)=\dfrac{d}{dx}Y_W(v,x).
\end{equation*}
\end{enumerate}

If $(W_1, Y_{W_1})$ and $(W_2, Y_{W_2})$ are grading-restricted generalized $V$-modules, then a \textit{parity-homogeneous homomorphism} $f: W_1\rightarrow W_2$ is a parity-homogeneous linear map which satisfies
\begin{equation*}
f(Y_{W_1}(v,x)w_1)=(-1)^{\vert f\vert \vert v\vert } Y_{W_2}(v,x)f(w_1)
\end{equation*}
for $w_1\in W_1$ and parity-homogeneous $v\in V$. A \textit{homomorphism} $f: W_1\rightarrow W_2$ is any sum of an even and an odd homomorphism.
\end{defi}
\begin{rema}
 We do not require homomorphisms between grading-restricted generalized modules to be even. In fact, homomorphisms between $V$-modules $W_1$ and $W_2$ form a superspace. Note that some references (for example \cite{WLi-thesis}) do not include the sign factor in the definition of $V$-module homomorphism, even when odd homomorphisms are allowed. However, the sign factor is natural in light of the general superalgebra principle that a sign factor should occur whenever the order of two entities with parity is reversed. Moreover, this sign factor is necessary for the notion of superalgebra module homomorphism to agree with the notion of $\repA$-morphism from Section \ref{sec:TensCats} in the context of a superalgebra extension $V\subseteq A$. (Recall the parity factor in the definition of tensor product of morphisms in $\sC$ from Section \ref{subsec:superalgebra}.)
\end{rema}

\begin{rema}
 If $(W, Y_W)$ is a grading-restricted generalized $V$-module, then $(\Pi(W), Y_{\Pi(W)})$ is as well, where $\Pi(W)=W$ with reversed parity decomposition: $\Pi(W)^i=W^{i+\odd}$ for $i=\even, \odd$, and with $Y_{\Pi(W)}=Y_W$. Because $Y_W$ is even, the parity involution $P_W=1_{W^\even}\oplus (-1_{W^\odd})$ on $W$ is an odd $V$-module isomorphism between $W$ and $\Pi(W)$.
\end{rema}

\begin{nota}
	From now on, we will typically use \emph{module} to mean a {grading-restricted generalized module}. Also, for a parity-homogeneous element $w$ in a $V$-module, we will denote the parity of $w$ by $\vert w\vert$.
\end{nota}

If $(W, Y_W)$ is a $V$-module, then the \textit{contragredient module} is the \textit{graded dual}
\begin{equation*}
 W'=\bigoplus_{h\in\CC} W_{[h]}^*
\end{equation*}
with parity decomposition $(W')^i=\bigoplus_{h\in\CC} (W^i_{[h]})^*$ for $i=\even, \odd$. Since $V$ is $\frac{1}{2}\ZZ$-graded, there are two natural choices for the vertex operator $Y_{W'}$:
\begin{equation}\label{contramodvrtxop}
 \langle Y_{W'}^\pm(v,x)w', w\rangle_W=(-1)^{\vert v\vert \vert w'\vert }\langle w', Y_W(e^{xL(1)} (e^{\pm\pi i} x^{-2})^{L(0)} v, x^{-1})w\rangle_W,
\end{equation}
for $w\in W$ and parity-homogeneous $v\in V$, $w'\in W'$, where $\langle\cdot,\cdot\rangle_W$ is the natural bilinear pairing between $W$ and its graded dual. The proof that $Y_{W'}^\pm$ give $V$-module structures on $W'$ is an easy superalgebra generalization of the proof for the vertex operator algebra case  in \cite{FHL}. Since $Y_{W'}^+(v, x)=Y_{W'}^-(\theta_V(v),x)$ for $v\in V$, the twist $\theta_{W'}=e^{2\pi i L_{W'}(0)}$ is an isomorphism between $(W', Y_{W'}^+)$ and $(W', Y_{W'}^-)$ (due to the general fact that for any $V$-module $W$, $\theta_W(Y_W(v,x)w)=Y_W(\theta_V(v), x)\theta_W(w)$ for $v\in V$, $w\in W$; see for instance \cite[Equation 3.86]{HLZ2}). Note that $\theta_{W'}=e^{2\pi iL_{W'}(0)}$ is well defined: because $\omega$ is even with integer conformal weight, $L_{W'}(0)$ is independent of the vertex operator structure on $W'$.

\begin{rema}
 There are further $V$-module structures on $W'$ appearing in the literature. For example, the definition in \cite{Xu} replaces the $(-1)^{\vert v\vert \vert w'\vert }$ factor in \eqref{contramodvrtxop} with $(-1)^{\vert v\vert\vert w\vert}$. These structures are isomorphic to the ones in \eqref{contramodvrtxop} via the parity automorphism $P_{W'}=1_{(W')^\even}\oplus(-1_{(W')^\odd})$. More structures occur in \cite{WLi-thesis, Ya3} where the $(-1)^{\vert v\vert \vert w'\vert}$ factor in \eqref{contramodvrtxop} is replaced by $e^{\pm\pi i \vert v\vert/2}$. These structures are equivalent to \eqref{contramodvrtxop} via certain square roots of the parity involution on $W'$: $1_{(W')^\even}\oplus (\pm i 1_{(W')^\odd})$.
\end{rema}

\begin{rema}
 The Jacobi identity in the definition of grading-restricted generalized $V$-module can be replaced with the following commutativity and associativity axiom: for any $w\in W$, $w'\in W'$, and parity-homogeneous $u, v\in V$, the series
 \begin{equation*}
  \langle w', Y_W(u,z_1)Y_W(v, z_2) w\rangle_W
 \end{equation*}
for $\vert z_1\vert>\vert z_2\vert>0$,
\begin{equation*}
(-1)^{\vert u\vert \vert v\vert } \langle w', Y_W(v,z_2)Y_W(u,z_1) w\rangle_W
\end{equation*}
for $\vert z_2\vert>\vert z_1\vert>0$, and 
\begin{equation*}
 \langle w', Y_W(Y(u,z_1-z_2)v,z_2)w\rangle_W
\end{equation*}
for $\vert z_2\vert>\vert z_1-z_2\vert>0$ converge absolutely to a common rational function on their respective domains.
\end{rema}

We  now recall the definition of (logarithmic) intertwining operator (from, for instance, \cite{HLZ2} in the vertex operator algebra setting):
\begin{defi}\label{def:intwop}
	Suppose $W_1$,  $W_2$ and $W_3$ are modules for a vertex operator superalgebra $V$. A \textit{parity-homogeneous intertwining 
		operator} of type $\binom{W_3}{W_1\,W_2}$ is a parity-homogeneous linear map
	\begin{align*}
		\mathcal{Y}: W_1\otimes W_2&\rightarrow W_3[\mathrm{log}\,x]\{x\},
		\\w_{1}\otimes w_{2}&\mapsto \mathcal{Y}(w_{1},x)w_{(2)}=\sum_{n\in
			{\mathbb C}}\sum_{k\in\NN} (w_{1})^\cY_{n; k}
		w_{2}\,x^{-n-1} (\mathrm{log}\,x)^k\in W_3[\mathrm{log}\,x]\{x\}
	\end{align*}
	satisfying the following conditions:
	\begin{enumerate}
		\item  {\it Lower truncation}: For any $w_{1}\in W_1$, $w_{2}\in W_2$ and 
		$n\in\mathbb{C}$,
		\begin{equation}\label{log:ltc}
		(w_{1})^\cY_{n+m; k}w_{2}=0\;\;\mbox{ for }\;m\in {\mathbb
			N} \;\mbox{ sufficiently large, independently of}\,k.
		\end{equation}
		
		\item The {\it Jacobi identity}:
		\begin{align}\label{intwopjac}
		(-1)^{\vert\cY\vert \vert v\vert} x^{-1}_0\delta \left(\frac{x_1-x_2}{x_0}\right) 
		Y_{W_3}(v,x_1)\mathcal{Y}(w_{1},x_2) & - (-1)^{\vert v\vert \vert w_1\vert }
		x^{-1}_0\delta\left(\frac{x_2-x_1}{-x_0}\right)\mathcal{Y}(w_{1},x_2)Y_{W_2}(v
		,x_1)\nonumber\\
		& =  x^{-1}_2\delta \left(\frac{x_1-x_0}{x_2}\right) 
		\mathcal{Y}(Y_{W_1}(v,x_0)w_{1},x_2)
		\end{align}
		for parity-homogeneous $v\in V$ and $w_{1}\in W_1$.
		
		\item The {\it $L(-1)$-derivative property:} for any
		$w_{1}\in W_1$,
		\begin{equation}\label{intwopderiv}
		\mathcal{Y}(L(-1)w_{1},x)=\frac{d}{dx}\mathcal{Y}(w_{1},x).
		\end{equation}
	\end{enumerate}
An \textit{intertwining operator} of type $\binom{W_3}{W_1\,W_2}$ is any sum of an even and an odd intertwining operator.
\end{defi}
\begin{rema}
	We use $\mathcal{V}^{W_3}_{W_1, W_2}$ to denote the superspace of intertwining 
	operators of type $\binom{W_3}{W_1\,W_2}$; the dimension and superdimension of 
	$\mathcal{V}^{W_3}_{W_1 W_2}$ are the corresponding \textit{fusion rules}: 
	\begin{equation*}
	 {N^\pm}^{W_3}_{W_1, W_2} =\mathrm{dim}\left(\mathcal{V}^{W_3}_{W_1, W_2}\right)^\even\pm\mathrm{dim}\left(\mathcal{V}^{W_3}_{W_1, W_2}\right)^\odd.
	\end{equation*}
\end{rema}

By slight generalizations (for the possibility of non-even intertwining operators) of \cite[Proposition 7.1]{HL-tensor2},   \cite[Theorem 4.1]{Xu}, \cite[Proposition 3.3.2]{WLi-thesis}, \cite[Proposition 3.44]{HLZ2}, we have for any $V$-modules $W_1$, $W_2$, and $W_3$ and $r\in\ZZ$ an even isomorphism $\Omega_r: \mathcal{V}^{W_3}_{W_1 W_2}\rightarrow \mathcal{V}^{W_3}_{W_2 W_1}$ given by
\begin{equation*}
\Omega_r(\mathcal{Y})(w_2,x)w_1=(-1)^{\vert w_1\vert \vert w_2\vert} e^{xL(-1)}\mathcal{Y}(w_1, e^{(2r+1)\pi i} x)w_2
\end{equation*}
for $\mathcal{Y}\in\mathcal{V}^{W_3}_{W_1 W_2}$ and parity-homogeneous $w_1\in W_1$, $w_2\in W_2$. The inverse of $\Omega_r$ is $\Omega_{-r-1}$. When $\cY=Y_W$ is a module vertex operator, $\Omega_r(Y_W)$ is independent of $r$ since $Y_W(\cdot, x)\cdot$ involves only integral powers of $x$. In this case, we drop the subscript $r$ on $\Omega$ and use the notation $\Omega(Y_W)$. By generalizations of \cite[Proposition 7.3]{HL-tensor2}, \cite[Theorem 4.4]{Xu}, \cite[Proposition 3.3.4]{WLi-thesis},  \cite[Proposition 3.46]{HLZ2}, we also have for any $r\in\ZZ$ an even linear isomorphism $A_r: \mathcal{V}^{W_3}_{W_1\,W_2}\rightarrow\mathcal{V}^{W_2'}_{W_1\,W_3'}$ given by
\begin{equation*}
 \langle A_r(\mathcal{Y})(w_1, x)w_3', w_2\rangle_{W_2}=(-1)^{(\vert\cY\vert+\vert w_1\vert)\vert w_3'\vert}\langle w_3', \mathcal{Y}(e^{xL(1)} (e^{(2r+1)\pi i} x^{-2})^{L(0)} w_1, x^{-1})w_2\rangle_{W_3}
\end{equation*}
for $w_2\in W_2$ and parity-homogeneous $\mathcal{Y}\in\mathcal{V}^{W_3}_{W_1\,W_2}$, $w_1\in W_1$, and $w_3'\in W_3'$. The inverse of $A_r$ is $A_{-r-1}$. A consequence of these isomorphisms is that the fusion rules ${N^\pm}^{W_k'}_{W_i, W_j}$ for any permutation $(i, j, k)$ of $(1, 2, 3)$ are equal.

Let us now fix a full subcategory $\cC$ of $V$-modules. In order to apply the vertex tensor category theory of \cite{HLZ1}-\cite{HLZ8} to $\cC$, we will need to impose a number of conditions on $\cC$. The following conditions come from Assumption 10.1 in \cite{HLZ6}, and more assumptions will be clarified later:
\begin{assum}\label{assum:grading} 
All modules in $\mathcal{C}$ are $\RR$-graded by conformal weights, and for any module $W=\bigoplus_{n\in\RR} W_{[n]}$ in $\mathcal{C}$, there is some $K\in\ZZ_+$ such that $(L(0)-n)^K\cdot W_{[n]}=0$ for all $n\in\RR$.
\end{assum}
\begin{rema}
	The assumption that modules in $\cC$ are $\RR$-graded by conformal weights can be modified somewhat; see for instance \cite{H-correctZ}.
\end{rema}
\begin{rema}
 The assumption that the possible Jordan block sizes of $L(0)$ acting on conformal weight spaces are uniformly bounded for each $V$-module $W$ in $\cC$ implies that for any intertwining operator $\cY$ of type $\binom{W_3}{W_1\,W_2}$ in $\cC$, there is some $K\in\NN$ depending on $W_1$, $W_2$, and $W_3$ such that
	\begin{equation*}
	\cY(w_1, x)w_2=\sum_{n\in\RR}\sum_{k=0}^K (w_1)^\mathcal{Y}_{n; k} w_2\,x^{-n-1} (\mathrm{log}\,x)^k
	\end{equation*}
	for any $w_1\in W_1$ and $w_2\in W_2$. That is, finitely many powers of $\mathrm{log}\,x$ appear in the expansion of $\cY(w_1,x)w_2$ independently of $w_1$ and $w_2$. (See Proposition 3.20 in \cite{HLZ2}.)
\end{rema}

\subsection{Multivalued functions, intertwining maps, and compositions of intertwining operators}
\label{sec:mult}

In this subsection, we recall the notion of $P(z)$-intertwining map among $V$-modules for $z\in\CC^\times$ and discuss assumptions and properties of our category $\cC$ of $V$-modules related to convergence of multivalued
functions defined using compositions of intertwining operators.

We first fix some notation that we will use throughout this section for branches of logarithm. For any $z\in\CC^\times$, we use the notation
\begin{equation*}
\mathrm{log}\,z=\mathrm{log}\,\vert z\vert+i\,\mathrm{arg}\,z,
\end{equation*}
where $0\leq\mathrm{arg}\,z<2\pi$, to denote the principal branch of logarithm on $\CC^\times$ with a cut along the positive real axis. More generally, we use the notation
\begin{equation*}
l_p(z)=\mathrm{log}\,z+2\pi i p
\end{equation*}
for any $p\in\ZZ$, so that $\mathrm{log}\,z=l_0(z)$. For branches of logarithm with a cut along the negative real axis, we use the notation $l_{p-1/2}(z)$ for $p\in\ZZ$ to denote the branch with 
\begin{equation*}
(2p-1)\pi i\leq \mathrm{Im}\,l_{p-1/2}(z)<(2p+1)\pi i.
\end{equation*}
For a formal series $f(x)\in W[\mathrm{log}\,x]\lbrace x\rbrace$, where $W$ is some superspace, we use the notation
\begin{equation*}
f(e^{l_p(z)})=f(x)\vert_{x^n=e^{n\,l_p(z)},\,\mathrm{log}\,x=l_p(z)}
\end{equation*}
for any $p\in\frac{1}{2}\ZZ$ to denote the (not necessarily convergent in any sense) series obtained by substituting $x$ with $z$ using the branch $l_p(z)$ of logarithm. We use similar notation for substituting formal variables with complex numbers in multivariable formal series.

We will need to consider multivalued functions of several complex variables in what follows, so we discuss some terminology and properties of multivalued functions here. Recall that a multivalued function $f$ defined on a connected set $U\subseteq\CC^n$, $n\in\ZZ_+$, is \textit{analytic} if for any $(z_1,\ldots, z_n)\in U$, any value of $f(z_1,\ldots,z_n)$ is the value at $(z_1,\ldots, z_n)$ of a single-valued analytic branch of $f$ defined on a simply-connected open neighborhood of $(z_1,\ldots, z_n)$. If $f$ is a multivalued analytic function on a connected open set $U\subseteq\CC^n$ and $\widetilde{U}\subseteq U$ is a connected open subset, we say that $\widetilde{f}$ is a \textit{restriction} of $f$ to $\widetilde{U}$ if the values of $\widetilde{f}(z_1,\ldots,z_n)$ for $(z_1,\ldots,z_n)\in\widetilde{U}$ are a subset of the values of $f(z_1,\ldots, z_n)$, and if any two values of $\widetilde{f}$ at any two points in $\widetilde{U}$ can be obtained from each other by analytic continuation along a continuous path in $\widetilde{U}$.

We say that two multivalued functions $f_1$ and $f_2$ on an open set $U\subseteq\CC^n$ are equal if for any simply-connected open set $\widetilde{U}\subseteq U$ and any single-valued branch $\widetilde{f}_1$ of $f_1$ on $\widetilde{U}$, there is a single-valued branch $\widetilde{f}_2$ of $f_2$ on $\widetilde{U}$ which equals $\widetilde{f}_1$.
\begin{propo}\label{equalmulti}
Suppose $f_1$ and $f_2$ are multivalued analytic functions defined on connected open sets $U_1, U_2\subseteq\CC^n$, respectively. Then $f_1$ and $f_2$ have equal restrictions to an open connected subset $\widetilde{U}\subseteq U_1\cap U_2$ if for some particular non-empty simply-connected open set $U'\subseteq\widetilde{U}$ and some particular single-valued branch $f_1'$ of $f_1$ on $U'$, there is a single-valued branch $f_2'$ of $f_2$ on $U'$ which equals $f_1'$.
\end{propo}
\begin{proof}
Fix $(z_1',\ldots, z_n')\in U'$, and consider the restrictions $\widetilde{f}_1$ of $f_1$ and $\widetilde{f}_2$ of $f_2$ to $\widetilde{U}$ whose values at any point  $(z_1,\ldots,z_n)\in \widetilde{U}$ are obtained by analytic continuation of $f_1'$ and $f_2'$, respectively, along continuous paths in $\widetilde{U}$ from $(z_1',\ldots,z_n')$ to $(z_1,\ldots, z_n)$. To see that $\widetilde{f}_1$ and $\widetilde{f}_2$ are equal, suppose $U''$ is any non-empty simply-connected open subset of $\widetilde{U}$ and $f_1''$ is any single-valued branch of $\widetilde{f}_1$ on $U''$. Now, $f_1''$ is completely determined by its values in a simply-connected neighborhood of some fixed $(z_1'',\ldots,z_n'')\in U''$, and these values can be obtained by analytic continuation of the branch $f_1'$ along some continous path in $\widetilde{U}$ from $(z_1',\ldots,z_n')$ to $(z_1'',\ldots,z_n'')$. Then we can analytically continue $f_2'$ along the same path to obtain a single-valued branch $f_2''$ of $\widetilde{f}_2$ on $U''$ which is equal to $f_1''$. Since $U'$ and $f_1'$ are arbitrary, this proves that $\widetilde{f}_1=\widetilde{f}_2$.
\end{proof}

For a formal series $f(x)\in\CC[\mathrm{log}\,x]\lbrace x\rbrace$ which converges absolutely on some open set in $\CC$ when the formal variable $x$ is replaced by a complex number $z$ using any branch of logarithm, we use the notation $f(z)$ to denote the corresponding multivalued function. We use similar notation for multivariable formal series.

The multivalued analytic functions that we will consider arise from composing intertwining operators, or more precisely intertwining maps, among $V$-modules in our category $\cC$.
We now give the superalgebra generalization of the notion of $P(z)$-intertwining map among $V$-modules for $z\in\CC^\times$ from \cite{HL-tensor1, HLZ3}.
Here, the notation ``$P(z)$'' stands for a Riemann sphere with a negatively oriented puncture at $\infty$ with local coordinate $w\mapsto1/w$,
and positively oriented punctures at $z$ and $0$ with local coordinates $w\mapsto w-z$ and $w\mapsto w$, respectively.
First, if
\begin{equation*}
W=\bigoplus_{h\in\mathbb{C}} W_{[h]}
\end{equation*}
is a $\mathbb{C}$-graded superspace (such as a $V$-module), then the 
\textit{algebraic completion} of $W$ is the superspace
\begin{equation*}
\overline{W}=\prod_{h\in\mathbb{C}} W_{[h]},
\end{equation*}
where $\overline{W}^i=\overline{W^i}$ for $i=\even,\odd$.
For any $h\in\mathbb{C}$, we use $\pi_h$ to denote the canonical even projection 
$\overline{W}\rightarrow W_{[h]}$. It is easy to see that $\overline{W}=(W')^*$.
\begin{defi}
	Suppose $W_1$, $W_2$, and $W_3$ are modules in $\mathcal{C}$ and $z\in\mathbb{C}^\times$. A 
	\textit{parity-homogeneous $P(z)$-intertwining map} of type $\binom{W_3}{W_1\,W_2}$ is a parity-homogeneous linear map
	\begin{equation*}
	I: W_1\otimes W_2\rightarrow \overline{W_3}
	\end{equation*}
	satisfying the following conditions:
	\begin{enumerate}
		\item \textit{Lower truncation}: For any $w_{1}\in W_1$, $w_{2}\in W_2$ 
		and $n\in\mathbb{C}$,
		\begin{equation}\label{map:ltc}
		\pi_{n-m}(I(w_{1}\otimes w_{2}))=0\;\;\mbox{ for }\;m\in {\mathbb
			N} \;\mbox{ sufficiently large.}
		\end{equation}
		\item The {\it Jacobi identity}:
		\begin{align}\label{intwmapjac}
		(-1)^{\vert I\vert \vert v\vert} x^{-1}_0\delta \left(\frac{x_1-z}{x_0}\right) 
		Y_{W_3}(v,x_1) I(w_{1}\otimes w_{2}) & - 
		(-1)^{\vert v\vert \vert w_1\vert } x^{-1}_0\delta\left(\frac{z-x_1}{-x_0}\right)I(w_{(1)}\otimes 
		Y_{W_2}(v,x_1)w_{2})\nonumber\\
		& =  z^{-1}\delta \left(\frac{x_1-x_0}{z}\right) 
		I(Y_{W_1}(v,x_0)w_{1}\otimes w_{2})
		\end{align}
		for $w_2\in W_2$ and parity-homogeneous $v\in V$, $w_{1}\in W_1$.
	\end{enumerate}
A \textit{$P(z)$-intertwining map} of type $\binom{W_3}{W_1\,W_2}$ is any sum of an even and an odd $P(z)$-intertwining map.
\end{defi}
\begin{rema}
	We use $\mathcal{M}[P(z)]^{W_3}_{W_1 W_2}$, or simply $\mathcal{M}^{W_3}_{W_1 
		W_2}$ if $z$ is clear, to denote the superspace of $P(z)$-intertwining maps of type 
	$\binom{W_3}{W_1\,W_2}$.
\end{rema}

The definitions suggest that a $P(z)$-intertwining map is essentially an 
intertwining operator with the formal variable $x$ specialized to the complex 
number $z$. In fact, generalizing Proposition 4.8 in \cite{HLZ3} to the superalgebra setting, we have:
\begin{propo}\label{opmapiso}
	For any $p\in\ZZ$, there is an even linear isomorphism $\mathcal{V}^{W_3}_{W_1 
		W_2}\rightarrow\mathcal{M}^{W_3}_{W_1 W_2}$ given by
	\begin{equation*}
	\mathcal{Y}\mapsto I_{\mathcal{Y},p},
	\end{equation*}
	where
	\begin{equation*}
	I_{\mathcal{Y},p}(w_{1}\otimes w_{2})=\mathcal{Y}(w_{1}, e^{l_p(z)})w_2
	\end{equation*}
	for $w_{1}\in W_1$ and $w_{2}\in W_2$. The inverse is given by
	\begin{equation*}
	I\mapsto\mathcal{Y}_{I,p}
	\end{equation*}
	where 
	\begin{equation*}
	\mathcal{Y}_{I,p}(w_{1},x)w_{2}= 
	\left(\dfrac{e^{l_p(z)}}{x}\right)^{-L(0)} 
	I\left(\left(\dfrac{e^{l_p(z)}}{x}\right)^{L(0)} w_{1}\otimes 
	\left(\dfrac{e^{l_p(z)}}{x}\right)^{L(0)} w_{2}\right)
	\end{equation*}
	for $w_{1}\in W_1$ and $w_{2}\in W_2$.
\end{propo}
\begin{rema}
	We could construct an isomorphism as in Proposition \ref{opmapiso} for $p\in\ZZ+\frac{1}{2}$ as well, but we will not need this.
\end{rema}

We now impose the assumption that products and iterates of intertwining maps among modules in $\cC$ satisfy the convergence condition of Proposition 7.3 and Definition 7.4 in \cite{HLZ5}:
\begin{assum}\label{assum:conv}
Suppose $W_1$, $W_2$, $W_3$, $W_4$, and $M_1$ are modules in $\mathcal{C}$, $I_1$ is a $P(z_1)$-intertwining map of type $\binom{W_4}{W_1\,M_1}$, and $I_2$ is a $P(z_2)$-intertwining map of type $\binom{M_1}{W_2\,W_3}$. Then if $\vert z_1\vert>\vert z_2\vert>0$, the series
\begin{equation*}
\langle w_4', I_1(w_1\otimes I_2(w_2\otimes w_3))\rangle=\sum_{n\in\RR} \langle w_4', I_1(w_1\otimes\pi_n(I_2(w_2\otimes w_3)))\rangle
\end{equation*}
converges absolutely for any $w_1\in W_1$, $w_2\in W_2$, $w_3\in W_3$, and $w_4'\in W_4'$. Moreover, suppose $M_2$ is another module in $\mathcal{C}$, $I^1$ is a $P(z_2)$-intertwining map of type $\binom{W_4}{M_2\,W_3}$, and $I^2$ is a $P(z_0)$-intertwining map of type $\binom{M_2}{W_1\,W_2}$. Then if $\vert z_2\vert>\vert z_0\vert>0$, the series
\begin{equation*}
\langle w_4', I^1(I^2(w_1\otimes w_2)\otimes w_3)\rangle=\sum_{n\in\RR}\langle w_4', I^1(\pi_n(I^2(w_1\otimes w_2))\otimes w_3)\rangle
\end{equation*}
converges absolutely for any $w_1\in W_1$, $w_2\in W_2$, $w_3\in W_3$, and $w_4'\in W_4'$.
\end{assum}

The isomorphism in Proposition \ref{opmapiso} implies that products and iterates of intertwining operators, with formal variables specialized to complex numbers in the appropriate domain using any branch of logarithm, converge absolutely. In fact, a product of intertwining operators of the form
\begin{equation*}
\langle w_4', \cY_1(w_1, z_1)\cY_2(w_2,z_2)w_3\rangle
\end{equation*}
defines a multivalued analytic function on the domain $\vert z_1\vert>\vert z_2\vert>0$, and an iterate of intertwining operators of the form
\begin{equation*}
\langle w_4', \mathcal{Y}^1(\mathcal{Y}^2(w_1,z_0)w_2, z_2)w_3\rangle
\end{equation*}
defines a multivalued analytic function on the domain $\vert z_2\vert>\vert z_0\vert>0$ (see Proposition 7.14 in \cite{HLZ5}). If we set $z_0=z_1-z_2$, the product and iterate restrict to multivalued analytic functions on the intersection $\vert z_1\vert>\vert z_2\vert>\vert z_1-z_2\vert>0$.

We conclude this subsection with two propositions we shall use later to conclude that multivalued functions given by products and iterates of intertwining operators are equal. First we introduce two useful non-empty simply-connected open sets $S_1, S_2\subseteq\CC^2$ that were used to study branches of products and iterates of intertwining operators in \cite{C}:
\begin{equation*}
S_1=\lbrace (z_1, z_2)\in\CC^2\,\vert\,\mathrm{Re}\,z_1>\mathrm{Re}\,z_2>\mathrm{Re}(z_1-z_2)>0,\,\mathrm{Im}\,z_1>\mathrm{Im}\,z_2>\mathrm{Im}(z_1-z_2)>0\rbrace
\end{equation*}
and
\begin{equation*}
S_2=\lbrace (z_1, z_2)\in\CC^2\,\vert\,\mathrm{Re}\,z_2>\mathrm{Re}\,z_1>\mathrm{Re}(z_2-z_1)>0,\,\mathrm{Im}\,z_2>\mathrm{Im}\,z_1>\mathrm{Im}(z_2-z_1)>0\rbrace.
\end{equation*}
Then we have:
\begin{propo}\label{extassoc}
	Suppose $W_1$, $W_2$, $W_3$, $W_4$, $M_1$, and $M_2$ are modules in $\mathcal{C}$ and $\mathcal{Y}_1$, $\mathcal{Y}_2$, $\mathcal{Y}^1$, and $\mathcal{Y}^2$ are intertwining operators of types $\binom{W_4}{W_1\,M_1}$, $\binom{M_1}{W_2\,W_3}$, $\binom{W_4}{M_2\,W_3}$, and $\binom{M_2}{W_1\,W_2}$, respectively, with $\cY_2$ parity-homogeneous. If
	\begin{equation*}
	(-1)^{\vert\cY_2\vert\vert w_1\vert}\langle w_4', \cY_1(w_1, e^{\mathrm{log}\,r_1})\cY_2(w_2, e^{\mathrm{log}\,r_2})w_3\rangle =\langle w_4', \cY^1(\cY^2(w_1, e^{\mathrm{log}(r_1-r_2)})w_2, e^{\mathrm{log}\,r_2})w_3\rangle
	\end{equation*}
	for any $w_2\in W_2$, $w_3\in W_3$, $w_4'\in W_4'$, and parity-homogeneous $w_1\in W_1$, and for some $r_1, r_2\in\RR$ satisfying $r_1>r_2>r_1-r_2>0$, then
	\begin{equation*}
	(-1)^{\vert\cY_2\vert\vert w_1\vert}\langle w_4', \cY_1(w_1, e^{\mathrm{log}\,z_1})\cY_2(w_2, e^{\mathrm{log}\,z_2})w_3\rangle =\langle w_4', \cY^1(\cY^2(w_1, e^{\mathrm{log}(z_1-z_2)})w_2, e^{\mathrm{log}\,z_2})w_3\rangle
	\end{equation*}
	for any $w_2\in W_2$, $w_3\in W_3$, $w_4'\in W_4'$, and parity-homogeneous $w_1\in W_1$, and for any $(z_1, z_2)\in S_1$. Moreover, the multivalued analytic functions
	defined by the product of $\cY_1$ and $\cY_2$ on the domain $\vert z_1\vert>\vert z_2\vert>0$ and by the iterate of $\cY^1$ and $\cY^2$ on the domain $\vert z_2\vert>\vert z_1-z_2\vert>0$ have equal restrictions to the region $\vert z_1\vert>\vert z_2\vert>\vert z_1-z_2\vert>0$.
\end{propo}
\begin{proof}
	The second assertion of the proposition follows immediately from the first and Proposition \ref{equalmulti}. By Proposition 7.14 in \cite{HLZ5},
	\begin{equation*}
	P(z_1,z_2)=(-1)^{\vert\cY_2\vert\vert w_1\vert}\langle w_4', \cY_1(w_1, e^{l_{-1/2}(z_1)})\cY_2(w_2, e^{l_{-1/2}(z_2)})w_3\rangle
	\end{equation*}
	is analytic and single-valued on the simply-connected open set defined by $\vert z_1\vert>\vert z_2\vert>0$ and $\mathrm{arg}\,z_1$, $\mathrm{arg}\,z_2\neq\pi$, and
	\begin{equation*}
	I(z_1, z_2)=\langle w_4', \cY^1(\cY^2(w_1, e^{l_{-1/2}(z_1-z_2)})w_2, e^{l_{-1/2}(z_2)})w_3\rangle
	\end{equation*}
	is analytic and single-valued on the simply-connected open set defined by $\vert z_2\vert>\vert z_1-z_2\vert>0$ and $\mathrm{arg}\,z_2$, $\mathrm{arg}(z_1-z_2)\neq\pi$. Moreover, the partial derivatives of these functions are given by the $L(-1)$-derivative property for intertwining operators.
	
	Since $l_{-1/2}(r)=\mathrm{log}\,r$ for $r$ a positive real number, we have by assumption $P(r_1, r_2)=I(r_1, r_2)$ for some $r_1,r_2\in\RR$ satisfying $r_1>r_2>r_1-r_2>0$. Also, our assumption, the $L(-1)$-derivative property, and the evenness of $L(-1)$ imply that all partial derivatives of $P$ and $I$ agree at $(r_1, r_2)$. This means that the power series expansions of $P$ and $I$ around $(r_1, r_2)$ are the same, so there is a simply-connected open neighborhood of $(r_1, r_2)$ on which $P$ and $I$ are equal. Then it follows that $P(z_1, z_2)=I(z_1,z_2)$ whenever $(z_1, z_2)$ is in the connected component of $(r_1, r_2)$ in the open set defined by
	\begin{equation*}
	\vert z_1\vert>\vert z_2\vert>\vert z_1-z_2\vert>0\,\,\,\mathrm{and}\,\,\,\mathrm{arg}\,z_1, \mathrm{arg}\,z_2, \mathrm{arg}(z_1-z_2)\neq\pi.
	\end{equation*}
	Since $(r_1, r_2)$ is in the boundary of $S_1$ and $S_1$ is connected, $S_1$ is contained in the connected component of $(r_1, r_2)$. Thus $P(z_1, z_2)=I(z_1,z_2)$ for $(z_1,z_2)\in S_1$. Since also $l_{-1/2}(z_1)=\mathrm{log}\,z_1$, $l_{-1/2}(z_2)=\mathrm{log}\,z_2$, and $l_{-1/2}(z_1-z_2)=\mathrm{log}(z_1-z_2)$ for $(z_1,z_2)\in S_1$, the proposition follows.
\end{proof}

\begin{rema}\label{extassocreal}
	In the setting of Proposition \ref{extassoc} and its proof, note that any pair $(r_1', r_2')\in\RR^2$ such that $r_1'>r_2'>r_1'-r_2'>0$ also lies in the connected component of $(r_1, r_2)$ in the region defined by
	\begin{equation*}
	\vert z_1\vert>\vert z_2\vert>\vert z_1-z_2\vert>0\,\,\,\mathrm{and}\,\,\,\mathrm{arg}\,z_1, \mathrm{arg}\,z_2, \mathrm{arg}(z_1-z_2)\neq\pi.
	\end{equation*}
	Moreover, $l_{-1/2}(r)=\mathrm{log}\,r$ for $r$ a positive real number. Thus the conclusion of Proposition \ref{extassoc} holds with $(z_1, z_2)\in S_1$ replaced by $(r_1', r_2')\in\RR^2$ satisfying $r_1'>r_2'>r_1'-r_2'>0$.
\end{rema}

The second proposition has a similar proof to the first:
\begin{propo}\label{extskewassoc}
	Suppose $W_1$, $W_2$, $W_3$, $W_4$, $M_1$, and $M_2$ are modules in $\mathcal{C}$ and $\mathcal{Y}_1$, $\mathcal{Y}_2$, $\mathcal{Y}^1$, and $\mathcal{Y}^2$ are intertwining operators of types $\binom{M_1}{W_1\,W_3}$, $\binom{W_4}{W_2\,M_1}$, $\binom{W_4}{M_2\,W_3}$, and $\binom{M_2}{W_1\,W_2}$, respectively. If
	\begin{equation*}
	(-1)^{\vert w_1\vert \vert w_2\vert}\langle w_4', \cY_2(w_2, e^{\mathrm{log}\,r_2})\cY_1(w_1, e^{\mathrm{log}\,r_1})w_3\rangle =\langle w_4', \cY^1(\cY^2(w_1, e^{l_{-1/2}(r_1-r_2)})w_2, e^{\mathrm{log}\,r_2})w_3\rangle
	\end{equation*}
	for any $w_3\in W_3$, $w_4'\in W_4'$, for any parity-homogeneous $w_1\in W_1$, $w_2\in W_2$,  and for some $r_1, r_2\in\RR$ satisfying $r_2>r_1>r_2-r_1>0$, then
	\begin{equation*}
	(-1)^{\vert w_1\vert \vert w_2\vert}\langle w_4', \cY_2(w_1, e^{\mathrm{log}\,z_2})\cY_1(w_1, e^{\mathrm{log}\,z_1})w_3\rangle =\langle w_4', \cY^1(\cY^2(w_1, e^{l_{-1/2}(z_1-z_2)})w_2, e^{\mathrm{log}\,z_2})w_3\rangle
	\end{equation*}
	for any $w_3\in W_3$, $w_4'\in W_4'$, for any parity-homogeneous $w_1\in W_1$, $w_2\in W_2$, and for any $(z_1, z_2)\in S_2$. Moreover, the multivalued analytic functions
	defined by the product of $\cY_2$ and $\cY_1$ on the domain $\vert z_2\vert>\vert z_1\vert>0$ and by the iterate of $\cY^1$ and $\cY^2$ on the domain $\vert z_2\vert>\vert z_1-z_2\vert>0$ have equal restrictions to the region $\vert z_2\vert>\vert z_1\vert, \vert z_1-z_2\vert>0$.
\end{propo}

\subsection{An exposition of vertex tensor category constructions}
\label{subsec:VTC}

We now discuss tensor structure on the category $\mathcal{C}$ of $V$-modules, giving the natural generalization of \cite{HLZ8} to the superalgebra setting.
The constructions actually give more than braided monoidal supercategory structure: they yield \textit{vertex tensor category} structure, in the terminology of \cite{HL-VTC},
and we shall need this extra structure for our main results. The vertex tensor category structure reflects the complex analytic properties of intertwining operators and their compositions, while the braided monoidal (super)category structure corresponds to topological structure. For further exposition of relevant constructions and results, see \cite{HL-rev}. Most of the definitions and results here are contained (for the vertex operator algebra setting) in \cite{HLZ1}-\cite{HLZ8}, as well as in the earlier papers \cite{HL-tensor1}-\cite{HL-tensor3}, \cite{H-tensor4, H-rigidity}, but we include proofs for some propositions when \cite{HLZ8} does not contain a detailed proof, or when the superalgebra generality differs from the algebra case.

\subsubsection{$P(z)$-tensor products}
\label{subsubsec:pz}
We recall the definition of $P(z)$-tensor product of $V$-modules for $z\in\CC^\times$ from \cite{HLZ3}:
\begin{defi}\label{def:Pztensprod}
	Suppose $W_1$ and $W_2$ are modules in $\mathcal{C}$ and $z\in\CC^\times$. A \textit{$P(z)$-tensor product} of $W_1$ and $W_2$ in $\mathcal{C}$ is a module $W_1\boxtimes_{P(z)} W_2$ in $\mathcal{C}$ together with an even $P(z)$-intertwining map $\boxtimes_{P(z)}$ of type $\binom{W_1\boxtimes_{P(z)} W_2}{W_1\,W_2}$ which satisfies the following universal property: for any module $W_3$ in $\mathcal{C}$ and $P(z)$-intertwining map $I$ of type $\binom{W_3}{W_1\,W_2}$, there is a unique homomorphism $\eta_I: W_1\boxtimes_{P(z)} W_2\rightarrow W_3$ such that $\overline{\eta_I}\circ\boxtimes_{P(z)}=I$, where $\overline{\eta_I}$ denotes the natural extension of $\eta_I$ to $\overline{W_1\boxtimes_{P(z)} W_2}$.
\end{defi}

\begin{rema}\label{rem:VTCtoBTC}
To obtain braided monoidal supercategory structure on $\mathcal{C}$, one must choose a particular tensor product. 
It is conventional to choose $\boxtimes_{P(1)}$, and to simplify notation, we use $\boxtimes$ to denote $\boxtimes_{P(1)}$.
We remark that specialization of vertex tensor structure to braided tensor structure is not automatic due to convergence issues discussed further below.
\end{rema}

\begin{rema}
	A construction (subject to certain conditions) of $W_1\boxtimes_{P(z)} W_2$ as the contragredient of a certain $V$-module contained in $(W_1\otimes W_2)^*$ may be found in \cite{HLZ4}. Because $P(z)$-tensor products are defined by a universal property, any other construction is naturally isomorphic to this one.
\end{rema}

\begin{nota}
	We use the following notation for the tensor product intertwining map $\boxtimes_{P(z)}$ of type $\binom{W_1\boxtimes_{P(z)} W_2}{W_1\,W_2}$:
	\begin{equation*}
	\boxtimes_{P(z)}(w_1\otimes w_2)=w_1\boxtimes_{P(z)} w_2
	\end{equation*}
	for $w_1\in W_1$, $w_2\in W_2$. We also use the simplified notation $\boxtimes_{P(1)}(w_1\otimes w_2)=w_1\boxtimes w_2$.
\end{nota}

For any $z\in\CC^\times$, we want the $P(z)$-tensor product to be a functor from $\mathcal{C}\times\mathcal{C}$ to $\mathcal{C}$, so we also need tensor products of morphisms. If $f_1: W_1\rightarrow\widetilde{W}_1$ and $f_2: W_2\rightarrow\widetilde{W}_2$ are morphisms in $\mathcal{C}$, then $\boxtimes_{P(z)}\circ(f_1\otimes f_2)$, where $\boxtimes_{P(z)}$ is the tensor product intertwining map for $\widetilde{W}_1$ and $\widetilde{W}_2$, is a $P(z)$-intertwining map of type $\binom{\widetilde{W}_1\boxtimes_{P(z)}\widetilde{W}_2}{W_1\,W_2}$. Here $f_1\otimes f_2$ is the tensor product of linear maps in the tensor supercategory of superspaces, defined by 
\begin{equation*}
 (f_1\otimes f_2)(w_1\otimes w_2)=(-1)^{\vert f_2\vert \vert w_1\vert}\left(f_1(w_1)\otimes f_2(w_2)\right)
\end{equation*}
when $f_2$ is parity-homogeneous, for $w_2\in W_2$ and parity-homogeneous $w_1\in W_1$. Then we define
\begin{equation*}
f_1\boxtimes_{P(z)} f_2: W_1\boxtimes_{P(z)}W_2\rightarrow\widetilde{W}_1\boxtimes_{P(z)}\widetilde{W}_2
\end{equation*}
to be the unique morphism induced by the universal property of $W_1\boxtimes_{P(z)} W_2$ and $\boxtimes_{P(z)}\circ(f_1\otimes f_2)$. In particular,
\begin{equation}\label{homtensprod}
\overline{f_1\boxtimes_{P(z)} f_2}(w_1\boxtimes_{P(z)} w_2)=(-1)^{\vert f_2\vert \vert w_1\vert}\left(f_1(w_1)\boxtimes_{P(z)} f_2(w_2)\right)
\end{equation}
if $f_2$ is parity-homogeneous, for any $w_2\in W_2$ and parity-homogeneous $w_1\in W_1$. With this definition, the $P(z)$-tensor product of $V$-module homomorphisms satisfies the super interchange law
\begin{equation*}
 (f_1\boxtimes_{P(z)} f_2)\circ(g_1\boxtimes_{P(z)} g_2)=(-1)^{\vert f_2\vert \vert g_1\vert }(f_1\circ g_1)\boxtimes_{P(z)}(f_2\circ g_2)
\end{equation*}
when $f_2$ and $g_1$ are parity-homogeneous, because the tensor product of linear maps between superspaces satisfies this property.

\subsubsection{Parallel transport isomorphisms}
\label{subsubsec:paralleltransp}

To define braiding and associativity isomorphisms in $\cC$, we shall need the parallel transport
isomorphisms constructed in \cite{HL-VTC, H-rigidity, HLZ8}. Suppose $W_1$ and $W_2$ are modules in $\mathcal{C}$, $z_1, z_2\in\CC^\times$ and $\gamma$ is a continuous path in $\CC^\times$ from $z_1$ to $z_2$. The parallel transport isomorphism (below, we will show that it is indeed an isomorphism, and moreover, a natural isomorphism)
\begin{equation*}
T_\gamma= T_{\gamma;\,W_1, W_2}: W_1\boxtimes_{P(z_1)} W_2\rightarrow W_1\boxtimes_{P(z_2)} W_2
\end{equation*}
is the unique morphism induced by the universal property of $W_1\boxtimes_{P(z_1)} W_2$ and the $P(z_1)$-intertwining map $\cY_{\boxtimes_{P(z_2)}, 0}(\cdot, e^{l(z_1)})\cdot$ of type $\binom{W_1\boxtimes_{P(z_2)} W_2}{W_1\,W_2}$, where $l(z_1)$ is the branch of logarithm determined by $\mathrm{log}\,z_2$ and the path $\gamma$. Notice that $T_\gamma$ is uniquely determined by the condition
\begin{equation}\label{tgammachar}
\overline{T_\gamma}(w_1 \boxtimes_{P(z_1)} w_2)=\cY_{\boxtimes_{P(z_2)}, 0}(w_1, e^{l(z_1)})w_2 
\end{equation}
for all $w_1\in W_1$, $w_2\in W_2$.

Suppose $\gamma_1$ is a path in $\CC^\times$ from $z_1$ to $z_2$ and $\gamma_2$ is a path in $\CC^\times$ from $z_2$ to $z_3$. Let $\gamma_2\cdot\gamma_1$ denote the path from $z_1$ to $z_3$ obtained by following $\gamma_1$ and then $\gamma_2$. 
\begin{propo}\label{compofgamma}
	For modules $W_1$ and $W_2$ in $\mathcal{C}$, $T_{\gamma_2\cdot\gamma_1}=T_{\gamma_2}\circ T_{\gamma_1}$.
\end{propo}
\begin{proof}
	By \eqref{tgammachar} and the $L(0)$-conjugation property of intertwining operators \cite[Proposition 3.36(b)]{HLZ2},
	\begin{align*}
	\overline{T_{\gamma_2}\circ T_{\gamma_1}} & (w_1\boxtimes_{P(z_1)} w_2)  =\overline{T_{\gamma_2}}(\cY_{\boxtimes_{P(z_2)}, 0}(w_1, e^{l(z_1)})w_2)\nonumber\\
	& =e^{(l(z_1)-\mathrm{log}\,z_2) L(0)}\overline{T_{\gamma_2}}(\mathcal{Y}_{\boxtimes_{P(z_2)},0}(e^{-(l(z_1)-\mathrm{log}\,z_2) L(0)} w_1, e^{\mathrm{log}\,z_2})e^{-(l(z_1)-\mathrm{log}\,z_2) L(0)} w_2)\nonumber\\
	& =e^{(l(z_1)-\mathrm{log}\,z_2) L(0)}\overline{T_{\gamma_2}}(e^{-(l(z_1)-\mathrm{log}\,z_2) L(0)}w_1\boxtimes_{P(z_2)} e^{-(l(z_1)-\mathrm{log}\,z_2) L(0)} w_2)\nonumber\\
	& =e^{(l(z_1)-\mathrm{log}\,z_2) L(0)}\cY_{\boxtimes_{P(z_3)},0}(e^{-(l(z_1)-\mathrm{log}\,z_2) L(0)}w_1, e^{l'(z_2)})e^{-(l(z_1)-\mathrm{log}\,z_2) L(0)}w_2\nonumber\\
	& =\cY_{\boxtimes_{P(z_3)},0}(w_1, e^{l(z_1)+l'(z_2)-\mathrm{log}\,z_2})w_2
	\end{align*}
	for any $w_1\in W_1$ and $w_2\in W_2$, where $l(z_1)$ is the branch of logarithm determined by $\mathrm{log}\,z_2$ and $\gamma_1$, and $l'(z_2)$ is the branch of logarithm determined by $\mathrm{log}\,z_3$ and $\gamma_2$. To complete the proof, we just need to show that $l(z_1)+l'(z_2)-\mathrm{log}\,z_2$ is the branch of logarithm at $z_1$ determined by $\mathrm{log}\,z_3$ and $\gamma_2\cdot\gamma_1$. In fact, the branch of logarithm determined by $\mathrm{log}\,z_3$ and $\gamma_2\cdot\gamma_1$ has value 
	\begin{equation*}
	l'(z_2)=\mathrm{log}\,z_2+2\pi i p
	\end{equation*}
	at $z_2$ for some $p\in\ZZ$. Then this branch of logarithm has value $$l(z_1)+2\pi i p=l(z_1)+l'(z_2)-\mathrm{log}\,z_2$$ at $z_1$, as desired.
\end{proof}

\begin{corol}
	The parallel transport $T_\gamma$ is an isomorphism.
\end{corol}
\begin{proof}
The parallel transport $T_\gamma$ depends only on the homotopy class of $\gamma$ in $\CC^\times$ because the branch of logarithm $l(z_1)$ depends only on the homotopy class of $\gamma$. Therefore the preceding proposition shows that the inverse of $T_\gamma$ is the parallel transport $T_{\gamma'}$ where $\gamma'$ is the path from $z_2$ to $z_1$ obtained by reversing $\gamma$. Thus $T_\gamma$ is in fact an isomorphism. 
\end{proof}

\begin{propo}
	 The parallel transport isomorphisms determine an even natural isomorphism from $\boxtimes_{P(z_1)}$ to $\boxtimes_{P(z_2)}$. 
	 In particular, for any morphisms $f_1: W_1\rightarrow\widetilde{W}_1$ and $f_2: W_2\rightarrow\widetilde{W}_2$ in $\mathcal{C}$, 
	 \begin{equation*}
	 T_{\gamma;\,\widetilde{W}_1,\widetilde{W}_2}\circ (f_1\boxtimes_{P(z_1)} f_2)=(f_1\boxtimes_{P(z_2)} f_2)\circ T_{\gamma;\,W_1,W_2}.
	 \end{equation*}
\end{propo}	 
\begin{proof}
The evenness of $T_\gamma$ follows from \eqref{tgammachar} and the evenness of $\boxtimes_{P(z_1)}$ and $\cY_{\boxtimes_{P(z_2)},0}$.
Now observe that if $f_2$ is parity-homogeneous,
\begin{equation*}
\overline{T_{\gamma;\,\widetilde{W}_1, \widetilde{W}_2}\circ(f_1\boxtimes_{P(z_1)} f_2)}(w_1\boxtimes_{P(z_1)} w_2)=(-1)^{\vert f_2\vert \vert w_1\vert}\mathcal{Y}_{\boxtimes_{P(z_2)},0}(f_1(w_1), e^{l(z_1)})f_2(w_2),
\end{equation*}
for any $w_2\in W_2$ and parity-homogeneous $w_1\in W_1$, where $l(z_1)$ is the branch of logarithm determined by $\mathrm{log}\,z_2$ and $\gamma$. On the other hand, since $L(0)$ is even,
\begin{align*}
& \overline{(f_1\boxtimes_{P(z_2)} f_2)\circ T_{\gamma;\,W_1,W_2}}  (w_1\boxtimes_{P(z_1)} w_2) = \overline{f_1\boxtimes_{P(z_2)} f_2}\left(\cY_{\boxtimes_{P(z_2)},0}(w_1, e^{l(z_1)})w_2\right)\nonumber\\
& \hspace{3em}= e^{(l(z_1)-\mathrm{log}\,z_2) L(0)}\overline{f_1\boxtimes_{P(z_2)} f_2}\left(\mathcal{Y}_{\boxtimes_{P(z_2)},0}(e^{-(l(z_1)-\mathrm{log}\,z_2) L(0)}w_1, e^{\mathrm{log}\,z_2})e^{-(l(z_1)-\mathrm{log}\,z_2) L(0)}w_2\right)\nonumber\\
&\hspace{3em}=e^{(l(z_1)-\mathrm{log}\,z_2) L(0)}\overline{f_1\boxtimes_{P(z_2)} f_2}\left(e^{-(l(z_1)-\mathrm{log}\,z_2) L(0)}w_1\boxtimes_{P(z_2)} e^{-(l(z_1)-\mathrm{log}\,z_2) L(0)}w_2\right)\nonumber\\
&\hspace{3em}=(-1)^{\vert f_2\vert \vert w_1\vert } e^{(l(z_1)-\mathrm{log}\,z_2) L(0)}\left(e^{-(l(z_1)-\mathrm{log}\,z_2) L(0)}f_1(w_1)\boxtimes_{P(z_2)}e^{-(l(z_1)-\mathrm{log}\,z_2) L(0)} f_2(w_2)\right)\nonumber\\
&\hspace{3em} =(-1)^{\vert f_2\vert \vert w_1\vert }\mathcal{Y}_{\boxtimes_{P(z_2)},0}(f_1(w_1), e^{l(z_1)})f_2(w_2)
\end{align*}
for any $w_2\in W_2$ and parity-homogeneous $w_1\in W_1$. Thus the two compositions are the same.
\end{proof}

We will frequently use parallel transport isomorphisms corresponding to paths $\gamma$ from $z_1$ to $z_2$ in $\CC^\times$ with a cut along the positive real axis. That is, such paths will neither cross the positive real axis nor approach it from the lower half plane. This will ensure that $l(z_1)=\mathrm{log}\,z_1$. Since the corresponding parallel transport isomorphism depends only the endpoints $z_1$ and $z_2$, we will denote it by $T_{z_1\to z_2}$. That is, $T_{z_1\to z_2}: W_1\boxtimes_{P(z_1)} W_2\rightarrow W_1\boxtimes_{P(z_2)} W_2 $ is the isomorphism characterized by
\begin{equation}\label{zonetoztwochar}
\overline{T_{z_1\to z_2}}(w_1\boxtimes_{P(z_1)} w_2)=\mathcal{Y}_{\boxtimes_{P(z_2)},0}(w_1, e^{\mathrm{log}\,z_1})w_2
\end{equation}
for $w_1\in W_1$, $w_2\in W_2$.

\subsubsection{Unit isomorphisms}
\label{subsubsec:unit}

We now discuss unit isomorphisms in $\mathcal{C}$. The unit object in $\cC$ is $V$. For any module $W$ in $\mathcal{C}$ and $z\in\CC^\times$, we have a unique homomorphism $l_{P(z)}: V\boxtimes_{P(z)} W\rightarrow W$ induced by the universal property of $V\boxtimes_{P(z)} W$ and the $P(z)$-intertwining map $Y_W(\cdot, z)\cdot$ (we do not need to specify a branch of logarithm because $Y_W$ involves only integral powers of the formal variable $x$). In the braided monoidal supercategory structure on $\cC$, the left unit isomorphism is $l_W=l_{P(1)}$. For the right unit isomorphism, we have for any $z\in\CC^\times$ a unique homomorphism $r_{P(z)}: W\boxtimes_{P(z)} V\rightarrow W$ induced by the universal property of $W\boxtimes_{P(z)} V$ and the $P(z)$-intertwining map $\Omega(Y_W)(\cdot, z)\cdot$. Then the right unit isomorphism in the braided monoidal supercategory $\mathcal{C}$ is $r_W=r_{P(1)}$. The unit isomorphisms for $W$ in $\mathcal{C}$ are characterized by the conditions
\begin{equation*}
\overline{l_W}(v\boxtimes w)=Y_W(v, 1)w
\end{equation*}
and
\begin{equation*}
\overline{r_W}(w\boxtimes v)=(-1)^{\vert v\vert \vert w\vert }e^{L(-1)}Y_W(v,-1)w
\end{equation*}
for parity-homogeneous $v\in V$ and $w\in W$.

\subsubsection{Braiding isomorphisms}
\label{subsubsec:braiding}

Next we recall the braiding isomorphisms in $\mathcal{C}$.
Let $W_1$ and $W_2$ be modules in $\mathcal{C}$ and let $z\in\CC^\times$. First, we define
\begin{equation*}
\cR^+_{P(z);\,W_1,W_2}=\cR^+_{P(z)}: W_1\boxtimes_{P(z)} W_2\rightarrow W_2\boxtimes_{P(-z)} W_1
\end{equation*}
to be the unique morphism induced by the universal property of $W_1\boxtimes_{P(z)} W_2$ and the $P(z)$-intertwining map $\Omega_0(\mathcal{Y}_{\boxtimes_{P(-z)}, 0})(\cdot, e^{\mathrm{log}\,z})\cdot$ of type $\binom{W_2\boxtimes_{P(-z)} W_1}{W_1\,W_2}$. Thus $\cR^+_{P(z)}$ is determined by the property
\begin{align}\label{Rplusdef}
\overline{\cR^+_{P(z)}}(w_1\boxtimes_{P(z)} w_2) & =(-1)^{\vert w_1\vert \vert w_2\vert} e^{zL(-1)}\mathcal{Y}_{\boxtimes_{P(-z)},0}(w_2, e^{\mathrm{log}\,z+\pi i})w_1\nonumber\\
& =(-1)^{\vert w_1\vert \vert w_2\vert} e^{zL(-1)}\mathcal{Y}_{\boxtimes_{P(-z)},0}(w_2, e^{l_{1/2}(-z)})w_1
\end{align}
for parity-homogeneous $w_1\in W_1$ and $w_2\in W_2$. We also define
\begin{equation*}
\cR^-_{P(z);\,W_1,W_2}=\cR^-_{P(z)}: W_1\boxtimes_{P(z)} W_2\rightarrow W_2\boxtimes_{P(-z)} W_1
\end{equation*}
to be the unique morphism induced by the universal property of $W_1\boxtimes_{P(z)} W_2$ and the $P(z)$-intertwining map $\Omega_{-1}(\cY_{\boxtimes_{P(-z)},0})(\cdot, e^{\mathrm{log}\,z})\cdot$ of type $\binom{W_2\boxtimes_{P(-z)} W_1}{W_1\,W_2}$. Then $\cR^-_{P(z)}$ is determined by the condition
\begin{align}\label{Rminusdef}
\overline{\cR^-_{P(z)}}(w_1\boxtimes_{P(z)} w_2) & = (-1)^{\vert w_1\vert \vert w_2\vert } e^{zL(-1)}\mathcal{Y}_{\boxtimes_{P(-z)},0}(w_2, e^{\mathrm{log}\,z-\pi i})w_1\nonumber\\
& = (-1)^{\vert w_1\vert \vert w_2\vert } e^{zL(-1)}\mathcal{Y}_{\boxtimes_{P(-z)},0}(w_2, e^{l_{-1/2}(-z)})w_1
\end{align}
for parity-homogeneous $w_1\in W_1$ and $w_2\in W_2$.
\begin{propo}\label{R+R-}
	For any $z\in\CC^\times$, $\cR^+_{P(z);\,W_1,W_2}$ and $\cR^-_{P(-z);\,W_2,W_1}$ are inverses of each other.
\end{propo}
\begin{proof}
	For any $z\in\CC^\times$ and parity-homogeneous $w_1\in W_1$, $w_2\in W_2$, we have
	\begin{align}\label{R+R-calc}
	&\overline{\cR^-_{P(-z)}\circ \cR^+_{P(z)}}  (w_1\boxtimes_{P(z)} w_2) = (-1)^{\vert w_1\vert \vert w_2\vert}\overline{\cR^-_{P(-z)}}\left(e^{zL(-1)}\mathcal{Y}_{\boxtimes_{P(-z)},0}(w_2, e^{l_{1/2}(-z)})w_1\right)\nonumber\\
	& =(-1)^{\vert w_1\vert \vert w_2\vert}e^{zL(-1)} e^{(l_{1/2}(-z)-\mathrm{log}(-z))L(0)}\cdot\nonumber\\
	&\hspace{10em}\cdot\overline{\cR^-_{P(-z)}}\left(e^{-(l_{1/2}(-z)-\mathrm{log}(-z))L(0)}w_2\boxtimes_{P(-z)}e^{-(l_{1/2}(-z)-\mathrm{log}(-z))L(0)} w_1\right)\nonumber\\
	& =e^{zL(-1)} e^{(l_{1/2}(-z)-\mathrm{log}(-z))L(0)}e^{-zL(-1)}\cdot\nonumber\\
	&\hspace{10em}\cdot\mathcal{Y}_{\boxtimes_{P(z)},0}(e^{-(l_{1/2}(-z)-\mathrm{log}(-z))L(0)}w_1, e^{l_{-1/2}(z)})e^{-(l_{1/2}(-z)-\mathrm{log}(-z))L(0)}w_2,
	\end{align}
	using the evenness of $L(0)$. Now, $e^{(l_{1/2}(-z)-\mathrm{log}(-z))L(0)}=e^{2\pi i p L(0)}$ for $p=0$ or $1$. Then Remark 3.40 in \cite{HLZ2} implies that $e^{(l_{1/2}(-z)-\mathrm{log}(-z))L(0)}$ is a $V$-module homomorphism and thus commutes with $e^{-zL(-1)}$. Hence the right side of \eqref{R+R-calc} reduces to
	\begin{align*}
	& e^{(l_{1/2}(-z)-\mathrm{log}(-z))L(0)}\mathcal{Y}_{\boxtimes_{P(z)},0}(e^{-(l_{1/2}(-z)-\mathrm{log}(-z))L(0)}w_1, e^{l_{-1/2}(z)})e^{-(l_{1/2}(-z)-\mathrm{log}(-z))L(0)}w_2\nonumber\\
	& \hspace{5em}=\mathcal{Y}_{\boxtimes_{P(z)},0}(w_1, e^{l_{-1/2}(z)+l_{1/2}(-z)-\mathrm{log}(-z)})w_2\nonumber\\
	& \hspace{5em} =\mathcal{Y}_{\boxtimes_{P(z)},0}(w_1, e^{\mathrm{log}\,z})w_2=w_1\boxtimes_{P(z)} w_2
	\end{align*}
	because
	\begin{equation*}
	l_{-1/2}(z)+l_{1/2}(-z)-\mathrm{log}(-z)=(\mathrm{log}(-z)-\pi i)+(\mathrm{log}\,z+\pi i)-\mathrm{log}(-z)=\mathrm{log}\,z.
	\end{equation*}
	This shows that $\cR^-_{P(z)}\circ \cR^+_{P(z)}=1_{W_1\boxtimes_{P(z)} W_2}$. Similarly, $\cR^+_{P(z)}\circ \cR^-_{P(-z)}=1_{W_2\boxtimes_{P(-z)} W_1}$.
\end{proof}

Following \cite{HLZ8}, we can now define the braiding isomorphism in $\mathcal{C}$ to be
\begin{equation*}
\cR_{W_1,W_2}=T_{-1\to 1}\circ \cR^+_{P(1)}: W_1\boxtimes W_2\rightarrow W_2\boxtimes W_1.
\end{equation*}
The braiding isomorphism satisfies
\begin{align}\label{braidchar}
\overline{\cR_{W_1,W_2}}(w_1&\boxtimes w_2) =(-1)^{\vert w_1\vert \vert w_2\vert} e^{L(-1)}\overline{T_{-1\to 1}}\left(\mathcal{Y}_{\boxtimes_{P(-1)},0}(w_2, e^{l_{1/2}(-1)})w_1\right)\nonumber\\ 
&= (-1)^{\vert w_1\vert \vert w_2\vert}e^{L(-1)}\overline{T_{-1\to 1}}(w_2\boxtimes_{P(-1)} w_1) =(-1)^{\vert w_1\vert \vert w_2\vert}e^{L(-1)}\mathcal{Y}_{\boxtimes_{P(1)},0}(w_2, e^{\mathrm{log}(-1)})w_1\nonumber\\
&=(-1)^{\vert w_1\vert \vert w_2\vert}e^{L(-1)}\mathcal{Y}_{\boxtimes_{P(1)},0}(w_2, e^{\pi i})w_1
\end{align}
for parity-homogeneous $w_1\in W_1$ and $w_2\in W_2$. The braiding isomorphism $\cR_{W_1,W_2}$ is also naturally related to $\cR_{P(z)}^+$ for any $z\in\CC^\times$:
\begin{propo}\label{prop:relntoR+R-}
	For any $z\in\CC^\times$, $\cR_{W_1,W_2}=T_{-z\to 1}\circ \cR^+_{P(z)}\circ T_{1\to z}$. Moreover, $\cR_{W_1, W_2}^{-1}=T_{z\to 1}\circ \cR^-_{P(-z)}\circ T_{1\to -z}$. 
\end{propo}
\begin{proof}
	The second assertion follows immediately from the first by Proposition \ref{R+R-} and the fact that $T_{z_2\to z_1}=T_{z_1\to z_2}^{-1}$ for any $z_1, z_2\in\CC^\times$. To prove the first assertion, we observe that
	\begin{align*}
	\overline{T_{-z\to 1}\circ \cR^+_{P(z)}\circ T_{1\to z}} & (w_1\boxtimes w_2) = \overline{T_{-z\to 1}\circ \cR^+_{P(z)}}(\mathcal{Y}_{\boxtimes_{P(z)},0}(w_1, e^{\mathrm{log}\,1})w_2)\nonumber\\
	& \hspace{-4em}=e^{-\mathrm{log}\,z\,L(0)}\overline{T_{-z\to 1}\circ \cR^+_{P(z)}}(e^{\mathrm{log}\,z\,L(0)}w_1\boxtimes_{P(z)} e^{\mathrm{log}\,z\,L(0)}w_2)\nonumber\\
	& \hspace{-4em}=(-1)^{\vert w_1\vert \vert w_2\vert}e^{-\mathrm{log}\,z\,L(0)} e^{z L(-1)}\overline{T_{-z\to 1}}\left(\mathcal{Y}_{\boxtimes_{P(-z)},0}(e^{\mathrm{log}\,z\,L(0)} w_2, e^{l_{1/2}(-z)})e^{\mathrm{log}\,z\,L(0)} w_1\right)\nonumber\\
	& \hspace{-4em}=(-1)^{\vert w_1\vert \vert w_2\vert}e^{-\mathrm{log}\,z\,L(0)} e^{zL(-1)} e^{(l_{1/2}(-z)-\mathrm{log}(-z))L(0)}\cdot\nonumber\\
	&\hspace{-2em}\cdot\overline{T_{-z\to 1}}(e^{(\mathrm{log}\,z-l_{1/2}(-z)+\mathrm{log}(-z))L(0)}w_2\boxtimes_{P(-z)} e^{(\mathrm{log}\,z-l_{1/2}(-z)+\mathrm{log}(-z))L(0)} w_1)\nonumber\\
	&\hspace{-4em} =(-1)^{\vert w_1\vert \vert w_2\vert}e^{L(-1)} e^{-l_{-1/2}(z) L(0)}\mathcal{Y}_{\boxtimes_{P(1)},0}(e^{l_{-1/2}(z) L(0)}w_2, e^{\mathrm{log}(-z)})e^{l_{-1/2}(z) L(0)} w_1\nonumber\\
	&\hspace{-4em} =(-1)^{\vert w_1\vert \vert w_2\vert}e^{L(-1)}\mathcal{Y}_{\boxtimes_{P(1)},0}(w_2, e^{\mathrm{log}(-z)-l_{-1/2}(z)})w_1\nonumber\\
	&\hspace{-4em} =(-1)^{\vert w_1\vert \vert w_2\vert}e^{L(-1)}\mathcal{Y}_{\boxtimes_{P(1)},0}(w_2, e^{\pi i})w_1
	\end{align*}
	for any parity-homogeneous $w_1\in W_1$, $w_2\in W_2$. To obtain the fifth equality, we have used the commutation relation $z^{-L(0)} e^{zL(-1)}=e^{zz^{-1}L(-1)}z^{-L(0)}$ (see Remark 3.38 in \cite{HLZ2}) and the equality
	\begin{equation*}
	\mathrm{log}\,z-l_{1/2}(-z)+\mathrm{log}(-z)=\mathrm{log}\,z-(\mathrm{log}\,z+\pi i)+(l_{-1/2}(z)+\pi i)=l_{-1/2}(z).
	\end{equation*}
	Then by \eqref{braidchar}, $T_{-z\to 1}\circ \cR^+_{P(z)}\circ T_{1\to z}$ agrees with $\cR_{W_1, W_2}$, as desired.
\end{proof}

Since the naturality of braiding isomorphisms in a supercategory involves a sign factor, we include a proof here to illustrate the essential role of the sign factor in \eqref{homtensprod}:
\begin{propo}
 For parity-homogeneous homomorphisms $f_1: W_1\rightarrow\widetilde{W}_1$ and $f_2: W_2\rightarrow\widetilde{W}_2$ in $\cC$,
 \begin{equation*}
  \cR_{\widetilde{W}_1,\widetilde{W}_2}\circ(f_1\boxtimes f_2)=(-1)^{\vert f_1\vert \vert f_2\vert} (f_2\boxtimes f_1)\circ\cR_{W_1,W_2}.
 \end{equation*}
\end{propo}
\begin{proof}
 Using \eqref{homtensprod}, \eqref{braidchar}, and the evenness of $L(0)$, $L(-1)$, we have:
 \begin{align*}
  \overline{\cR_{\widetilde{W}_1,\widetilde{W}_2}\circ(f_1\boxtimes f_2)} & (w_1\boxtimes w_2) = (-1)^{\vert f_2\vert\vert w_1\vert}\overline{\cR_{\widetilde{W}_1,\widetilde{W}_2}}(f_1(w_1)\boxtimes f_2(w_2))\nonumber\\
  & = (-1)^{\vert f_2\vert\vert w_1\vert+(\vert f_1\vert+\vert w_1\vert)(\vert f_2\vert+\vert w_2\vert)} e^{L(-1)}\cY_{\boxtimes_{P(1)},0}(f_2(w_2), e^{\pi i})f_1(w_1)\nonumber\\
  & = (-1)^{\vert f_1\vert\vert f_2\vert+\vert w_1\vert\vert w_2\vert+\vert f_1\vert\vert w_2\vert} e^{L(-1)} e^{\pi i L(0)}\left(e^{-\pi i L(0)}f_2(w_2)\boxtimes e^{-\pi i L(0)} f_1(w_1)\right)\nonumber\\
  & =(-1)^{\vert f_1\vert\vert f_2\vert}\overline{f_2\boxtimes f_1}\left((-1)^{\vert w_1\vert\vert w_2\vert}e^{L(-1)}\cY_{\boxtimes_{P(1)},0}(w_2, e^{\pi i})w_1\right)\nonumber\\
  & =(-1)^{\vert f_1\vert\vert f_2\vert}\overline{(f_2\boxtimes f_1)\circ\cR_{W_1,W_2}}(w_1\boxtimes w_2)
 \end{align*}
for parity-homogeneous $w_1\in W_1$ and $w_2\in W_2$.
\end{proof}

We now show how to characterize the monodromy isomorphisms in $\cC$. Notice that \eqref{braidchar} implies that $\overline{\cR_{W_1,W_2}}\circ\boxtimes_{P(1)}=\Omega_0(\cY_{\boxtimes_{P(1)},0})(\cdot, e^{\mathrm{log}\,1})\cdot$; by Proposition \ref{opmapiso}, this means that
\begin{equation}\label{formalbraidchar}
 \cR_{W_1,W_2}(\cY_{\boxtimes_{P(1)},0}(w_1,x)w_2)=\Omega_0(\cY_{\boxtimes_{P(1)},0})(w_1,x)w_2=(-1)^{\vert w_1\vert \vert w_2\vert}e^{xL(-1)}\cY_{\boxtimes_{P(1)},0}(w_2,e^{\pi i} x)w_1
\end{equation}
for parity-homogeneous $w_1\in W_1$, $w_2\in W_2$. Now it is easy to see from \eqref{formalbraidchar} that the monodromy isomorphisms in $\cC$ are given by
 \begin{equation}\label{vrtxmonodromy}
  \cM_{W_1,W_2}(\cY_{\boxtimes_{P(1)},0}(w_1,x)w_2)=\cY_{\boxtimes_{P(1)},0}(w_1, e^{2\pi i} x)w_2
 \end{equation}
for $w_1\in W_1$ and $w_2\in W_2$.
Comparing this with \eqref{tgammachar}, we see immediately that the monodromy isomorphisms in $\cC$ are indeed given by the monodromy of certain paths in $\CC^\times$:
\begin{propo}\label{prop:monoandtrans}
For any $p\in\ZZ$, $\cM_{W_1,W_2}^p=T_{\gamma^p; W_1, W_2}$ where $\gamma^p$ is a continuous path from $1$ to $1$ in $\CC^\times$ that wraps around the origin $p$ times clockwise. 
\end{propo}

\subsubsection{Associativity isomorphisms}
\label{subsubsec:assoc}
Now we recall the associativity isomorphisms in $\mathcal{C}$ from \cite{HLZ8}. Suppose $W_1$, $W_2$, and $W_3$ are objects of $\mathcal{C}$. Then for $z_1, z_2\in\CC^\times$ which satisfy $\vert z_1\vert>\vert z_2\vert> 0$, the convergence condition for intertwining maps implies that there is an element
\begin{equation*}
w_1\boxtimes_{P(z_1)}(w_2\boxtimes_{P(z_2)} w_3)\in \overline{W_1\boxtimes_{P(z_1)} (W_2\boxtimes_{P(z_2)} W_3)}
\end{equation*}
for any $w_1\in W_1$, $w_2\in W_2$, and $w_3\in W_3$ defined by its action on $(W_1\boxtimes_{P(z_1)} (W_2\boxtimes_{P(z_2)} W_3))'$:
\begin{equation*}
\langle w', w_1\boxtimes_{P(z_1)}(w_2\boxtimes_{P(z_2)} w_3)\rangle=\sum_{n\in\RR} \langle w', w_1\boxtimes_{P(z_1)}\pi_n(w_2\boxtimes_{P(z_2)} w_3)\rangle
\end{equation*}
for any $w'\in (W_1\boxtimes_{P(z_1)} (W_2\boxtimes_{P(z_2)} W_3))'$. Similarly, for $z_0, z_2\in\CC^\times$ which satisfy $\vert z_2\vert>\vert z_0\vert>0$, there is an element 
\begin{equation*}
(w_1\boxtimes_{P(z_0)} w_2)\boxtimes_{P(z_2)} w_3\in\overline{(W_1\boxtimes_{P(z_0)} W_2)\boxtimes_{P(z_2)} W_3)}
\end{equation*}
which is defined by
\begin{equation*}
\langle w', (w_1\boxtimes_{P(z_0)} w_2)\boxtimes_{P(z_2)} w_3\rangle=\sum_{n\in\RR} \langle w', \pi_n(w_1\boxtimes_{P(z_0)} w_2)\boxtimes_{P(z_2)} w_3\rangle
\end{equation*}
for any $w'\in ((W_1\boxtimes_{P(z_0)} W_2)\boxtimes_{P(z_2)} W_3)'$. Then under suitable conditions, \cite[Theorem 10.3]{HLZ6} shows that for $z_1, z_2\in\CC^\times$ which satisfy
\begin{equation*}
\vert z_1\vert>\vert z_2\vert>\vert z_1-z_2\vert>0,
\end{equation*}
there is a natural isomorphism
\begin{equation*}
\cA_{z_1,z_2}: W_1\boxtimes_{P(z_1)} (W_2\boxtimes_{P(z_2)} W_3)\rightarrow (W_1\boxtimes_{P(z_1-z_2)} W_2)\boxtimes_{P(z_2)} W_3
\end{equation*}
characterized by the condition
\begin{equation}\label{zassocchar}
\overline{\cA_{z_1, z_2}}(w_1\boxtimes_{P(z_1)}(w_2\boxtimes_{P(z_2)} w_3))=(w_1\boxtimes_{P(z_1-z_2)} w_2)\boxtimes_{P(z_2)} w_3
\end{equation}
for any $w_1\in W_1$, $w_2\in W_2$, $w_3\in W_3$. This isomorphism is even because the tensor product intertwining maps $\boxtimes_{P(z)}$ for $z\in\CC^\times$ are even.

To obtain the natural associativity isomorphism $\cA_{W_1,W_2,W_3}$ in the braided monoidal supercategory $\mathcal{C}$, 
choose $r_1,r_2\in\RR$ which satisfy $r_1>r_2>r_1-r_2>0$ and set
\begin{equation*}
\cA_{W_1,W_2,W_3}=T_{r_2\to 1}\circ(T_{r_1-r_2\to 1}\boxtimes_{P(r_2)} 1_{W_3})\circ\cA_{r_1,r_2}\circ(1_{W_1}\boxtimes_{P(r_1)} T_{1\to r_2})\circ T_{1\to r_1}.
\end{equation*}
Note that $\cA_{W_1,W_2,W_3}$ is even because $\cA_{z_1,z_2}$ and the parallel transport isomorphisms are even. It is also independent of the choice of $r_1, r_2$:
\begin{propo}\label{associndofr}
	The associativity isomorphism $\cA_{W_1,W_2,W_3}$ is independent of the choice of $r_1,r_2\in\RR$ satisfying $r_1>r_2>r_1-r_2>0$.
\end{propo}
\begin{proof}
	We need to show that for any $r_1',r_2'\in\RR$ satisfying $r_1'>r_2'>r_1'-r_2'>0$, 
	\begin{equation*}
	\cA_{W_1,W_2,W_3}=T_{r_2'\to 1}\circ(T_{r_1'-r_2'\to 1}\boxtimes_{P(r_2')} 1_{W_3})\circ\cA_{r_1',r_2'}\circ(1_{W_1}\boxtimes_{P(r_1')} T_{1\to r_2'})\circ T_{1\to r_1'},
	\end{equation*}
	or equivalently,
	\begin{align*}
	\cA_{r_1,r_2}\circ & (1_{W_1}\boxtimes_{P(r_1)} T_{1\to r_2})\circ T_{1\to r_1}\circ T_{1\to r_1'}^{-1}\circ (1_{W_1}\boxtimes_{P(r_1')} T_{1\to r_2'}^{-1})\nonumber\\
	& =(T_{r_1-r_2\to 1}^{-1}\boxtimes_{P(r_2)} 1_{W_3})\circ T_{r_2\to 1}^{-1}\circ T_{r_2'\to 1}\circ(T_{r_1'-r_2'\to 1}\boxtimes_{P(r_2')} 1_{W_3})\circ\cA_{r_1',r_2'}.
	\end{align*}
	By Proposition \ref{compofgamma}, the fact that $T_{z_1\to z_2}^{-1}=T_{z_2\to z_1}$ for $z_1, z_2\in\CC^\times$, and the naturality of the parallel transport isomorphisms,
	\begin{equation*}
	(1_{W_1}\boxtimes_{P(r_1)} T_{1\to r_2})\circ T_{1\to r_1}\circ T_{1\to r_1'}^{-1}\circ (1_{W_1}\boxtimes_{P(r_1')} T_{1\to r_2'}^{-1})=(1_{W_1}\boxtimes_{P(r_1)} T_{r_2'\to r_2})\circ T_{r_1'\to r_1}
	\end{equation*}
	and
	\begin{equation*}
	(T_{r_1-r_2\to 1}^{-1}\boxtimes_{P(r_2)} 1_{W_3})\circ T_{r_2\to 1}^{-1}\circ T_{r_2'\to 1}\circ(T_{r_1'-r_2'\to 1}\boxtimes_{P(r_2')} 1_{W_3})=T_{r_2'\to r_2}\circ(T_{r_1'-r_2'\to r_1-r_2}\boxtimes_{P(r_2')} 1_{W_3}).
	\end{equation*}
	Thus we are reduced to proving
	\begin{equation*}
	\cA_{r_1,r_2}\circ(1_{W_1}\boxtimes_{P(r_1)} T_{r_2'\to r_2})\circ T_{r_1'\to r_1}=T_{r_2'\to r_2}\circ(T_{r_1'-r_2'\to r_1-r_2}\boxtimes_{P(r_2')} 1_{W_3})\circ\cA_{r_1',r_2'}.
	\end{equation*}
	
	By Corollary 7.17 in \cite{HLZ5}, it is enough to show that
	\begin{align*}
	&\overline{\cA_{r_1,r_2}\circ(1_{W_1}\boxtimes_{P(r_1)} T_{r_2'\to r_2})\circ T_{r_1'\to r_1}}(w_1\boxtimes_{P(r_1')}(w_2\boxtimes_{P(r_2')} w_3))\nonumber\\
	&\hspace{4em}=\overline{T_{r_2'\to r_2}\circ(T_{r_1'-r_2'\to r_1-r_2}\boxtimes_{P(r_2')} 1_{W_3})\circ\cA_{r_1',r_2'}}(w_1\boxtimes_{P(r_1')}(w_2\boxtimes_{P(r_2')} w_3))
	\end{align*}
	for all $w_1\in W_1$, $w_2\in W_2$, and $w_3\in W_3$. In fact, by \eqref{homtensprod}, \eqref{zonetoztwochar}, and \eqref{zassocchar},
	\begin{align}\label{assocwd1}
	&\overline{\cA_{r_1,r_2}\circ(1_{W_1}\boxtimes_{P(r_1)} T_{r_2'\to r_2})\circ T_{r_1'\to r_1}}(w_1\boxtimes_{P(r_1')}(w_2\boxtimes_{P(r_2')} w_3))\nonumber\\
	&\hspace{4em}=(\cA_{r_1,r_2}\circ\mathcal{Y}_{\boxtimes_{P(r_1),0}})(w_1, e^{\mathrm{log}\,r_1'})\mathcal{Y}_{\boxtimes_{P(r_2)},0}(w_2, e^{\mathrm{log}\,r_2'})w_3,
	\end{align}
	while
	\begin{align}\label{assocwd2}
	&\overline{T_{r_2'\to r_2}\circ(T_{r_1'-r_2'\to r_1-r_2}\boxtimes_{P(r_2')} 1_{W_3})\circ\cA_{r_1',r_2'}}(w_1\boxtimes_{P(r_1')}(w_2\boxtimes_{P(r_2')} w_3))\nonumber\\
	&\hspace{4em} =\mathcal{Y}_{\boxtimes_{P(r_2)},0}(\mathcal{Y}_{\boxtimes_{P(r_1-r_2)},0}(w_1, e^{\mathrm{log}(r_1'-r_2')})w_2, e^{\mathrm{log},r_2'})w_3.
	\end{align}
	Now, \eqref{zassocchar} shows that \eqref{assocwd1} equals \eqref{assocwd2} for all $w_1\in W_1$, $w_2\in W_2$, and $w_3\in W_3$ when $(r_1', r_2')=(r_1, r_2)$. Then Remark \ref{extassocreal} shows that \eqref{assocwd1} and \eqref{assocwd2} are equal for general $(r_1', r_2')$.
\end{proof}
\begin{rema}\label{rem:assocparallel}
 By the same argument as in the above proof, using Proposition \ref{extassoc},
	\begin{equation*}
	\cA_{W_1,W_2,W_3}=T_{z_2\to 1}\circ(T_{z_1-z_2\to 1}\boxtimes_{P(z_2)} 1_{W_3})\circ\cA_{z_1,z_2}\circ(1_{W_1}\boxtimes_{P(z_1)} T_{1\to z_2})\circ T_{1\to z_1}
	\end{equation*}
	for $(z_1, z_2)\in S_1$.
\end{rema}

The equality in Remark \ref{rem:assocparallel} need not hold for arbitrary $z_1,z_2\in\CC^\times$ such that $\vert z_1\vert>\vert z_2\vert>\vert z_1-z_2\vert>0$. To see what happens in general, note that $\cA_{z_1,z_2}$ is characterized by the equality
\begin{align*}
 \langle w', (\cA_{z_1,z_2}\circ\cY_{\boxtimes_{P(z_1)},0}) & (w_1, e^{\log z_1})\cY_{\boxtimes_{P(z_2)}, 0}(w_2, e^{\log z_2})w_3\rangle\nonumber\\
 & = \langle w', \cY_{\boxtimes_{P(z_2)}, 0}(\cY_{\boxtimes_{P(z_1-z_2)}, 0}(w_1, e^{\log(z_1-z_2)})w_2, e^{\log z_2})w_3\rangle
\end{align*}
for $w'\in((W_1\boxtimes_{P(z_1-z_2)} W_2)\boxtimes_{P(z_2)} W_3)'$, $w_1\in W_1$, $w_2\in W_2$, and $w_3\in W_3$. This implies an equality of single-valued branches of the multivalued functions given by the product and iterate of intertwining operators in a simply-connected neighborhood of $(z_1,z_2)$. We can analytically continue these single-valued branches along the following path $\Gamma$ which is contained in the region $\vert z_1'\vert>\vert z_2'\vert>\vert z_1'-z_2'\vert>0$: first hold $z_2$ fixed while rotating $z_1$ along the circle of radius $\vert z_1\vert$ until it is collinear with $0$ and $z_2$ (that is, rotate $z_1$ to $\frac{\vert z_1\vert}{\vert z_2\vert} z_2$); then rotate $\frac{\vert z_1\vert}{\vert z_2\vert} z_2$ and $z_2$ together clockwise to $(\vert z_1\vert, \vert z_2\vert)$. As we analytically continue along the first stage of $\Gamma$, the branch of logarithm for $z_1$ and $z_1-z_2$ may change from $\log$ to, say, $l$ and $\widetilde{l}$, respectively. All branches of logarithm remain unchanged during the second stage of $\Gamma$. Thus
\begin{align*}
 \langle w', (\cA_{z_1,z_2}\circ\cY_{\boxtimes_{P(z_1)},0}) & (w_1, e^{l(\vert z_1\vert)})\cY_{\boxtimes_{P(z_2)}, 0}(w_2, e^{\log \vert z_2\vert})w_3\rangle\nonumber\\
 & = \langle w', \cY_{\boxtimes_{P(z_2)}, 0}(\cY_{\boxtimes_{P(z_1-z_2)}, 0}(w_1, e^{\widetilde{l}(\vert z_1\vert-\vert z_2\vert)})w_2, e^{\log \vert z_2\vert})w_3\rangle
\end{align*}
for $w'\in((W_1\boxtimes_{P(z_1-z_2)} W_2)\boxtimes_{P(z_2)} W_3)'$, $w_1\in W_1$, $w_2\in W_2$, and $w_3\in W_3$. Similar to the proof of Proposition \ref{associndofr}, this equality is equivalent to the equality of isomorphisms:
\begin{align*}
 \cA_{z_1,z_2}\circ T_\gamma\circ & (1_{W_1}\boxtimes_{P(z_1)} T_{\vert z_2\vert\to z_2})\circ T_{\vert z_1\vert\to z_1}\nonumber\\
 &= (T_{\widetilde{\gamma}}\boxtimes_{P(z_2)} 1_{W_3})\circ(T_{\vert z_1\vert-\vert z_2\vert\to z_1-z_2}\boxtimes_{P(z_2)} 1_{W_3})\circ T_{\vert z_2\vert\to z_2}\circ\cA_{\vert z_1\vert, \vert z_2\vert},
\end{align*}
where $\gamma$ is a path from $z_1$ to itself such that $l(z_1)$ is the branch of logarithm determined by $\log z_1$ and $\gamma$, and $\widetilde{\gamma}$ is a path from $z_1-z_2$ to itself such that $\widetilde{l}(z_1-z_2)$ is the branch of logarithm determined by $\log(z_1-z_2)$ and $\widetilde{\gamma}$. Combining this with the relation between $\cA_{\vert z_1\vert, \vert z_2\vert}$ and $\cA_{W_1,W_2,W_3}$ in Proposition \ref{associndofr}, we get
\begin{align}\label{Az_1z_2general}
 (T_{\widetilde{\gamma}}^{-1} & \boxtimes_{P(z_2)} 1_{W_3})\circ  \cA_{z_1,z_2}\circ T_\gamma\nonumber\\
 &= (T_{1\to z_1-z_2}\boxtimes_{P(z_2)} 1_{W_3})\circ T_{1\to z_2}\circ\cA_{W_1,W_2,W_3}\circ T_{z_1\to 1}\circ(1_{W_1}\boxtimes_{P(z_1)} T_{z_2\to 1})
\end{align}
for any modules $W_1$, $W_2$, $W_3$ in $\cC$ and $(z_1,z_2)\in(\CC^\times)^2$ such that $\vert z_1\vert>\vert z_2\vert>\vert z_1-z_2\vert>0$. The paths $\gamma$ and $\widetilde{\gamma}$ depend only on $(z_1,z_2)$ and not on $W_1$, $W_2$, $W_3$.

\subsubsection{Existence of vertex tensor category structure}

In the tensor-categorical constructions on $\cC$ above, the vertex tensor category structure consists in the tensor products, unit, and braiding isomorphisms defined for each $z\in\CC^\times$; the associativity isomorphisms $\cA_{z_1,z_2}$ defined for $\vert z_1\vert>\vert z_2\vert>\vert z_1-z_2\vert> 0$; and the parallel transport isomorphisms. We have seen how braided monoidal supercategory structure, using $\boxtimes=\boxtimes_{P(1)}$ as the tensor functor, can be obtained from the vertex tensor category structure using parallel transports. However, the existence of vertex tensor category structure on $\cC$ is in general a hard question because several analytic and algebraic assumptions on $\cC$ are imposed in \cite{HLZ1}-\cite{HLZ8}. Here we review these assumptions and discuss some results showing when these assumptions hold.

So far, we have assumed that $\cC$ is a full subcategory of $V$-modules, all modules in $\cC$ are $\RR$-graded by conformal weights, and the Jordan block sizes of $L(0)$ acting on the conformal weight spaces $W_{[n]}$, $n\in\RR$, are uniformly bounded for each module in $\cC$ (recall Assumption \ref{assum:grading}). We have also assumed convergence for compositions of intertwining maps (recall Assumption \ref{assum:conv}). Based on Assumption 10.1 in \cite{HLZ6} and Assumption 12.1 in \cite{HLZ8}, we impose:
\begin{assum}
 The vertex operator superalgebra $V$ is an object of $\cC$, and $\cC$ is closed under images, contragredients, finite direct sums, $P(z)$-tensor products for at least one (equivalently, all) $z\in\CC^\times$.
\end{assum}

In discussing the associativity isomorphisms in $\cC$, we showed how the convergence condition for intertwining maps implies the existence of triple tensor product elements of the forms $w_1\boxtimes_{P(z_1)}(w_2\boxtimes_{P(z_2)} w_3)$ and $(w_1\boxtimes_{P(z_1-z_2)} w_2)\boxtimes_{P(z_2)} w_3$. However, this is not enough for the existence of associativity isomorphisms relating these triple tensor product elements. For this, sufficient conditions are given in \cite[Theorem 11.4]{HLZ7}. First, we say that a $V$-module $W$ is \textit{lower bounded} if $W_{[n]}=0$ for $n\in\RR$ sufficiently negative.
\begin{assum}\label{assum:convext}
 Every finitely-generated lower bounded $V$-module is an object of $\cC$, and the \textit{convergence and extension property} for either products or iterates of intertwining operators in $\cC$ holds. See \cite{HLZ7} for the technical statement of the convergence and extension property. Esentially, for a product of intertwining operators of the form
 \begin{equation}\label{convextprod}
  \langle w_4', \cY_1(w_1,z_1)\cY_2(w_2, z_2)w_3\rangle,
 \end{equation}
the convergence and extension property holds if \eqref{convextprod} converges absolutely to a multivalued analytic function on $\vert z_1\vert>\vert z_2\vert>0$ and can be analytically continued to a multivalued analytic function on $\vert z_2\vert>\vert z_1-z_2\vert>0$ with at worst logarithmic singularities at $z_1=z_2$ and $z_2=0$. The convergence and extension property for iterates is similar.
\end{assum}
\begin{rema}\label{rem:weaker_assum}
 The condition that every finitely-generated lower bounded $V$-module is an object of $\cC$ is not strictly necessary. In fact, only certain specific finitely-generated lower bounded $V$-modules appearing in the proof of \cite[Theorem 11.4]{HLZ7} must be in $\cC$ (see \cite{H-correctZ}). Since the first version of this paper was written, there has been progress in verifying this weaker condition in the case that $\cC$ is the category of $C_1$-cofinite $V$-modules; see especially Theorems 3.3.4 and 3.3.5 in \cite{CY}.
\end{rema}

Finally, to guarantee that the associativity isomorphisms in $\cC$ satisfy the pentagon axiom, we need to require that products of three intertwining operators in $\cC$ converge to suitable multivalued functions (see Assumption 12.2 in \cite{HLZ8}):
\begin{assum}\label{assum:threeconv}
 For objects $W_1$, $W_2$, $W_3$, $W_4$, $W_5$, $M_1$, and $M_2$ of $\cC$, intertwining operators $\cY_1$, $\cY_2$, and $\cY_3$ of types $\binom{W_5}{W_1\,M_1}$, $\binom{M_1}{W_2\,M_2}$, and $\binom{M_2}{W_3\,W_4}$, respectively, and $z_1,z_2,z_3\in\CC^\times$ satisfying $\vert z_1\vert>\vert z_2\vert>\vert z_3\vert>0$, the double sum
 \begin{equation*}
  \sum_{m,n\in\RR} \left\langle w_5', \cY_1(w_1,z_1)\pi_m\left(\cY_2(w_2,z_2)\pi_n\left(\cY_3(w_3,z_3)w_4\right)\right)\right\rangle
 \end{equation*}
converges absolutely for any $w_1\in W_1$, $w_2\in W_2$, $w_3\in W_3$, $w_4\in W_4$, and $w_5'\in W_5'$ and is a restriction to $\vert z_1\vert>\vert z_2\vert>\vert z_3\vert>0$ of a multivalued analytic function on the region given by $z_i\neq 0$, $z_i\neq z_j$ for $i=1,2,3$ and $j\neq i$. Moreover, near any possible singular point of the form $z_i=0,\infty$ or $z_i=z_j$ for $i=1,2,3$, $j\neq i$, this multivalued function can be expanded as a series having the same form as the expansion around a singular point of a solution to a differential equation with regular singular points.
\end{assum}

\begin{rema}
 In practice, one shows that Assumptions \ref{assum:convext} and \ref{assum:threeconv} hold by showing that products and iterates of intertwining operators are in fact solutions to differential equations with regular singular points (see for instance \cite{H-Virasoro, HL-tensorAffine, H-diff-eqns}).
\end{rema}
If all assumptions listed so far hold, then $\cC$ is a $\CC$-linear monoidal supercategory (see \cite[Theorem 12.15]{HLZ8}). In practice, one would like simple conditions on $V$ that guarantee all these assumptions hold. The strongest result of this type so far is due to Huang \cite{H-cofin} (and also see \cite{H-cofin} for relevant notation):
\begin{theo}[{\cite[Theorem 4.13]{H-cofin}}]\label{thm:VTCexists}
Let $V$ be a \VOA{} satisfying:
\begin{enumerate}
	\item $V$ is $C_1^a$-cofinite, 
	\item Every irreducible $V$-module is $\RR$-graded and $C_1$-cofinite,
	\item There exists a positive integer $N$ that bounds the differences in the lowest
	conformal weights of any two irreducible modules and such that the $N$th Zhu algebra $A_N(V)$ is finite dimensional.	
\end{enumerate}
Then the category $\cC$ of grading-restricted generalized $V$-modules is a vertex tensor category.
\end{theo}

An important consequence of this theorem is the following:
\begin{corol}[{\cite[Proposition 4.1, Theorem 4.13]{H-cofin}}]
 If $V$ is $C_2$-cofinite and has positive energy
($V_n=0$ for $n<0$ and $V_0=\CC\unit$), then the category $\cC$ of grading-restricted generalized $V$-modules is a vertex tensor category.
\end{corol}

\begin{rema}
As mentioned in Remark \ref{rem:weaker_assum}, there has been recent progress beyond Theorem \ref{thm:VTCexists} in verifying the assumptions of this subsection for the category $\cC_V^1$ of $C_1$-cofinite modules for a vertex operator algebra $V$. Especially, results in \cite{CY} show that $\cC_V^1$ is a vertex tensor category if it is closed under contragredients. Moreover, a noteworthy consequence of our results in the following subsections is that if $V$ is known to have a vertex tensor category $\cC$ of modules and $A\supseteq V$ is a (super)algebra extension of $V$ in $\cC$, then the category of $A$-modules in $\cC$ is also a vertex tensor category. That is, it is not necessary to directly verify the assumptions in this subsection for the category of $A$-modules in $\cC$. This idea was recently applied to singlet vertex operator algebras in \cite{CMY}.
\end{rema}

\subsection{Complex analytic formulation of categorical \texorpdfstring{$\repA$}{Rep A}-intertwining operators}
\label{subsec:complexIntw}

In this subsection, we begin studying (super)algebra extensions of vertex operator algebras. Here, we fix an ordinary $\ZZ$-graded vertex operator algebra $V$, and we assume that $V$ has a full abelian subcategory $\cC$ of modules that has vertex tensor category structure, and thus braided tensor category structure, as described in the previous subsection. Note that in the braided and vertex tensor category structure on $\cC$, all sign factors in the definitions of the previous subsection may be ignored.

Since our goal is to study superalgebra extensions of $V$ in $\cC$, we recall Theorem 3.2 and Remark 3.3 in \cite{HKL}, and their generalization to superalgebra extensions in Theorem 3.13 of \cite{CKL}:
\begin{theo}
 Vertex operator superalgebra extensions $A$ of $V$ in $\cC$ such that $V\subseteq A^\even$ are precisely superalgebras $(A,\mu,\iota_A)$ in the braided tensor category $\cC$ which satisfy:
 \begin{enumerate}
  \item $\iota_A$ is injective;
  \item $A$ is $\frac{1}{2}\ZZ$-graded by conformal weights: $\theta_A^2=1_A$ where $\theta_A=e^{2\pi i L_A(0)}$; and
  \item $\mu$ has no monodromy: $\mu\circ(\theta_A\boxtimes \theta_A)=\theta_A\circ \mu$.
 \end{enumerate}
\end{theo}
We now fix a vertex operator superalgebra extension $A$ in $\cC$ such that $V\subseteq A^\even$, and we recall the braided tensor supercategory $\sC$ and $\CC$-linear monoidal supercategory $\repA$ from Section \ref{sec:TensCats}. Since objects of $\cC$ are vector spaces, as a category we can identify $\sC$ simply with $\cC$, except that morphism spaces in $\sC$ are equipped with superspace structure (coming from parity decompositions of objects in $\cC$). With this identification, tensor products of objects, unit isomorphisms, and associativity isomorphisms in $\sC$ agree with those in $\cC$, but tensor products of morphisms and braiding isomorphisms in $\sC$ come with sign factors. Note also that $\sC$ has the vertex tensor category structure of the previous subsection, with all sign factors included. In fact, $\sC$ is the vertex tensor category of $V$-modules in $\cC$ where $V$ is considered as a vertex operator superalgebra with zero odd component.

Our goal is to identify categorical $\repA$-intertwining operators with certain vertex-algebraic intertwining operators among modules in $\cC$. For this purpose, let $(W_1, \mu_{W_1})$, $(W_2, \mu_{W_2})$, and $(W_3, \mu_{W_3})$ be objects of $\repA$. It is easy to see from Definition \ref{intwop} that a $\repA$-intertwining operator of type $\binom{W_3}{W_1\,W_2}$ is a $\cC$-morphism $\eta: W_1\boxtimes W_2\rightarrow W_3$ such that the following pairs of compositions are equal:
\begin{align}
&A\boxtimes (W_1\boxtimes W_2)\xrightarrow{\sA_{A,W_1,W_2}} (A\boxtimes W_1)\boxtimes W_2\xrightarrow{\mu_{W_1}\boxtimes 1_{W_2}} W_1\boxtimes W_2\xrightarrow{\eta} W_3\nonumber\\
&A\boxtimes(W_1\boxtimes W_2)\xrightarrow{1_A\boxtimes \eta} A\boxtimes W_3\xrightarrow{\mu_{W_3}} W_3.
\label{eqn:int2}
\end{align}
and
\begin{align}
&(W_1\boxtimes A)\boxtimes W_2\xrightarrow{\sR_{A,W_1}^{-1}\boxtimes 1_{W_2}} (A\boxtimes W_1)\boxtimes W_2\xrightarrow{\mu_{W_1}\boxtimes 1_{W_2}} W_1\boxtimes W_2\xrightarrow{\eta} W_3 \nonumber\\
&  (W_1\boxtimes A)\boxtimes W_2\xrightarrow{\sA_{W_1,A,W_2}^{-1}} W_1\boxtimes(A\boxtimes W_2)\xrightarrow{1_{W_1}\boxtimes\mu_{W_2}} W_1\boxtimes W_2\xrightarrow{\eta} W_3,
\label{eqn:int1}
\end{align}
The tensor products of morphisms and braiding isomorphism here are in the supercategory $\sC$, while the associativity isomorphisms in $\sC$ may be identified with associativity isomorphisms in $\cC$. We would like to show that \eqref{eqn:int2} and \eqref{eqn:int1} are equivalent to complex analytic associativity and skew-associativity conditions, respectively, for an intertwining operator associated to $\eta$.

Now if $\eta$ is a $\repA$-intertwining operator of type $\binom{W_3}{W_1\,W_2}$, then $I_\eta=\overline{\eta}\circ\boxtimes$ is a $P(1)$-intertwining map among a triple of $V$-modules of type $\binom{W_3}{W_1\,W_2}$, and we define the $V$-intertwining operator $\cY_\eta=\cY_{I_\eta, 0}$; note that
$$\overline{\eta}(w_1\fus{} w_2)=\cY_\eta(w_1,e^{\log 1})w_2.$$
For $z\in\mathbb{C}^\times$, let 
$$\eta_z=\eta\circ T_{z\rightarrow 1}: W_1\boxtimes_{P(z)} W_2\rightarrow W_3.$$
It is clear that $\eta=\eta_1$. Moreover,
\begin{align*}
\overline{\eta_z}(w_1\boxtimes_{P(z)} w_2) & =\overline{\eta\circ T_{z\to 1}}(w_1\fus{P(z)} w_2)\nonumber\\
& =\overline{\eta}(\mathcal{Y}_{\fus{P(1)},0}(w_1, e^{\mathrm{log}\,z})w_2)\nonumber\\
&=\overline{\eta}\left(e^{\log zL(0)}\cY_{\fus{P(1)},0}(e^{-\log zL(0)}w_1,e^{\log 1})e^{-\log zL(0)}w_2\right)\\
&=e^{\log zL(0)}\overline{\eta}\left(e^{-\log zL(0)}w_1 \boxtimes e^{-\log zL(0)}w_2\right)\\
&=e^{\log zL(0)}\cY_\eta(e^{-\log zL(0)}w_1,e^{\log 1})e^{-\log zL(0)}w_2\\
&=\cY_{\eta}(w_1,e^{\log z})w_2
\end{align*}
for any $z\in\CC^\times$, $w_1\in W_1$, and $w_2\in W_2$. This calculation shows that $\eta_z$ is the unique morphism induced by the universal property of $W_1\fus{P(z)} W_2$ and the $P(z)$-intertwining map $\mathcal{Y}_\eta(\cdot, e^{\mathrm{log}\,z})\cdot$.
\begin{rema}\label{rema:YW}
	For the special case where $\eta=\mu_W$ is a $\repA$-intertwining operator of type $\binom{W}{A\,W}$, we set $Y_W=\mathcal{Y}_{\mu_W}$. Then the preceding calculation shows that
	\begin{equation*}
	\overline{\mu_{W; z}}(a\boxtimes_{P(z)} w)=Y_W(a, e^{\mathrm{log}\,z})w
	\end{equation*}
	for $z\in\CC^\times$, $a\in A$, and $w\in W$.
\end{rema}

Our main theorem in this subsection is the following characterization of $\repA$ intertwining operators
in terms of complex variables. In the proof, we will need technical Proposition \ref{braidinganalogue}
presented in Section \ref{subsec:technical}.  The name \textit{skew-associativity} for the second property of $\cY_\eta$ in the theorem comes from \cite{Ro}.
\begin{theo}
	\label{thm:repAtocomplex}
	A parity-homogeneous $\mathcal{C}$-morphism $\eta: W_1\boxtimes W_2\rightarrow W_3$ is a $\repA$-intertwining operator if and only if the following two properties hold:
	\begin{enumerate}
		\item (Associativity) For $a\in A$ homogeneous, $w_1\in W_1$, $w_2\in W_2$,  $w_3'\in W_3'$, and $(z_1,z_2)\in S_1$, 
		\begin{equation}\label{assoc}
		(-1)^{\vert\eta\vert \vert a\vert}\langle w_3', Y_{W_3}(a, e^{\mathrm{log}\,z_1})\mathcal{Y}_\eta(w_1, e^{\mathrm{log}\,z_2})w_2\rangle=\langle w_3', \mathcal{Y}_\eta(Y_{W_1}(a, e^{\mathrm{log}(z_1-z_2)})w_1, e^{\mathrm{log}\,z_2})w_2\rangle.
		\end{equation}
		In particular,
		the multivalued analytic functions 
		\begin{equation*}
		(-1)^{\vert\eta\vert \vert a\vert}\langle w_3', Y_{W_3}(a, z_1)\mathcal{Y}_\eta(w_1, z_2)w_2\rangle
		\end{equation*}
		on the region $\vert z_1\vert>\vert z_2\vert>0$ and
		\begin{equation*}
		\langle w_3', \mathcal{Y}_\eta(Y_{W_1}(a, z_1-z_2)w_1, z_2)w_2\rangle
		\end{equation*}
		on the region $\vert z_2\vert>\vert z_1-z_2\vert>0$ have equal restrictions to their common domain.
		
		\item (Skew-associativity) For $a\in A$ and $w_1\in W_1$ homogeneous, $w_2\in W_2$,  $w_3'\in W_3'$, and $(z_1,z_2)\in S_2$,
		\begin{equation}\label{skewassoc}
		(-1)^{\vert a\vert \vert w_1\vert}\langle w_3', \mathcal{Y}_\eta(w_1, e^{\mathrm{log}\,z_2})Y_{W_2}(a, e^{\mathrm{log}\,z_1})w_2\rangle=\langle w_3', \mathcal{Y}_\eta(Y_{W_1}(a, e^{l_{-1/2}(z_1-z_2)})w_1, e^{\mathrm{log}\,z_2})w_2\rangle.
		\end{equation}
		In particular, the multivalued functions
		\begin{equation*}
		(-1)^{\vert a\vert \vert w_1\vert}\langle w_3', \mathcal{Y}_\eta(w_1, z_2)Y_{W_2}(a, z_1)w_2\rangle
		\end{equation*}
		on the region $\vert z_2\vert>\vert z_1\vert>0$ and
		\begin{equation*}
		\langle w_3', \mathcal{Y}_\eta(Y_{W_1}(a, z_1-z_2)w_1, z_2)w_2\rangle
		\end{equation*}
		on the region $\vert z_2\vert>\vert z_1-z_2\vert>0$ have equal restrictions to their common domain.
	\end{enumerate}
\end{theo}
\begin{proof}
By Proposition \ref{equalmulti} and the observation that the regions $\vert z_1\vert>\vert z_2\vert>\vert z_1-z_2\vert>0$ and $\vert z_2\vert>\vert z_1\vert,\,\vert z_1-z_2\vert>0$ are connected, it is enough to prove that $\eta$ is a $\repA$-intertwining operator if and only if \eqref{assoc} and \eqref{skewassoc} hold. In fact, we will prove that \eqref{eqn:int2} is equivalent to \eqref{assoc} and that \eqref{eqn:int1} is equivalent to \eqref{skewassoc}.

	First we consider the associativity. Fix $r_1,r_2\in \mathbb{R}$ such that $r_1>r_2>r_1-r_2>0$,
	and consider the following diagram:
	\begin{align}\label{assocdiag1}
	\xymatrixcolsep{12pc}
	\xymatrix{
		A\fus{P(r_1)}(W_1\fus{P(r_2)} W_2)  \ar[d]_{1_A\fus{P(r_1)} \eta_{r_2}}
		\ar[r]^(0.53){
			T_{r_1\to 1}\circ (1_A\boxtimes_{P(r_1)} T_{r_2\to 1}) }
		& A\fus{}(W_1\fus{}W_2) \ar[d]^{1_A\boxtimes \eta} \\
		A\fus{P(r_1)}W_3 \ar[r]^(0.52){T_{r_1\to 1}} \ar[rd]_{\mu_{W_3}; r_1}
		& A\fus{}W_3 \ar[d]^{\mu_{W_3}}
		\\
		& W_3
	}
	\end{align}
	The square commutes by definition of $\eta_{r_2}$ and the naturality of the parallel transport isomorphisms, and the triangle commutes by definition of $\mu_{W_3; r_1}$. We also consider the diagram:
	\begin{align}\label{assocdiag2}
	\xymatrixcolsep{12pc}
	\xymatrix{
		A\fus{P(r_1)}(W_1\fus{P(r_2)} W_2) 
		\ar[d]_{\cA_{r_1,r_2}}
		\ar[r]^(0.53){
			T_{r_1\to 1}\circ (1_A\boxtimes_{P(r_1)} T_{r_2\to 1}) }
		& A\fus{} (W_1\fus{}W_2) \ar[d]^{\sA_{A,W_1,W_2}} \\
		(A\fus{P(r_1-r_2)}W_1)\fus{P(r_2)}W_2 \ar[d]_{\mu_{W_1;\,r_1-r_2}\fus{P(r_2)}1_{W_2}}
		\ar[r]^(0.54){T_{r_2\to 1}\circ(T_{r_1-r_2\to 1}\fus{P(r_2)} 1_{W_2})}
		& (A\fus{}W_1)\fus{}W_2 \ar[d]^{\mu_{W_1}\fus{}1_{W_2}} \\
		W_1\fus{P(r_2)}W_2 \ar[r]^(0.52){T_{r_2\to 1}} \ar[rd]_{\eta_{r_2}}
		& W_1\fus{}W_2 \ar[d]^{\eta}
		\\
		& W_3
	}
	\end{align}
	In this diagram, the top square commutes by definition of $\cA_{A,W_1,W_2}$ (or $\sA_{A,W_1,W_2}$)  and the fact that $T_{z_1\to z_2}^{-1}=T_{z_2\to z_1}$ for $z_1,z_2\in\CC^\times$. The middle square commutes
	by definition of $\mu_{W_1; r_1-r_2}$ and naturality of parallel transports, and the  bottom triangle commutes by definition of $\eta_{r_2}$.
	
	Now, \eqref{eqn:int2} translates to the equality of morphisms in the right columns of the 
	two commuting diagrams above. Hence, \eqref{eqn:int2} is equivalent to the
	equality of the left columns.
	The equality of left columns is the equality of morphisms determined by the following mappings:
	\begin{equation*}
	a\fus{P(r_1)}(w_1\fus{P(r_2)}w_2)
	\mapsto (-1)^{\vert\eta\vert\vert a\vert}a\fus{P(r_1)}\cY_\eta(w_1,e^{\log r_2})w_2
	\mapsto (-1)^{\vert\eta\vert\vert a\vert}Y_{W_3}(a,e^{\log r_1})\cY_\eta(w_1,e^{\log r_2})w_2
	\end{equation*}
	and
	\begin{align*}
	a\fus{P(r_1)}(w_1 & \fus{P(r_2)}w_2)  \mapsto (a\fus{P(r_1-r_2)}w_1)\fus{P(r_2)}w_2\nonumber\\
	&\mapsto Y_{W_1}(a,e^{\log(r_1-r_2)}w_1)\fus{P(r_2)}w_2
	\mapsto \cY_\eta(Y_{W_1}(a,e^{\log(r_1-r_2)})w_1,e^{\log r_2})w_2
	\end{align*}
	for all $w_1\in W_1$, $w_2\in W_2$, and parity-homogeneous $a\in A$. (Everything is  absolutely convergent for the $r_1$ and $r_2$ we have chosen.)
	We therefore deduce that \eqref{eqn:int2} implies
	\begin{align}
	(-1)^{\vert\eta\vert\vert a\vert}Y_{W_3}(a,e^{\log r_1})\cY_\eta(w_1,e^{\log r_2})w_2=\cY_\eta(Y_{W_1}(a,e^{\log(r_1-r_2)})w_1,e^{\log r_2})w_2
	\label{eqn:complexint2}
	\end{align}
	as elements of $\overline{W_3}$
	for all $a\in A,w_1\in W_1,w_2\in W_1$. Then by Proposition \ref{extassoc}, \eqref{assoc} holds for all $(z_1,z_2)\in S_1$.
	
	Conversely, \eqref{assoc} for $(z_1,z_2)\in S_1$ implies \eqref{eqn:complexint2} because $(r_1,r_2)$ lies on the boundary of $S_1$. Then \eqref{eqn:complexint2} implies the equality
	of left columns of \eqref{assocdiag1} and \eqref{assocdiag2}, thereby implying the equality of the right columns,
	that is, \eqref{eqn:int2}. This completes the proof that \eqref{eqn:int2} is equivalent to \eqref{assoc}.

	Now we consider the skew-associativity. Fix $r_1, r_2\in\RR$ such that $r_2>r_1>2(r_2-r_1)>0$ (recall Proposition \ref{associndofr}) and consider the diagram:
	\begin{align}\label{skewassocdiag1}
	\xymatrixcolsep{12pc}
	\xymatrix{
		(W_1\fus{P(r_2-r_1)} A)\fus{P(r_1)}W_2 
		\ar[d]_{\cA_{r_2,r_1}^{-1}}
		\ar[r]^(0.54){T_{r_1\to 1}\circ
			(T_{r_2-r_1\to 1}\fus{P(r_1)}1_{W_2}) }
		& (W_1\fus{} A)\fus{}W_2 \ar[d]^{\cA_{W_1,A,W_2}^{-1}} \\
		W_1\fus{P(r_2)}(A\fus{P(r_1)}W_2) \ar[d]_{1_{W_1}\fus{P(r_2)}\mu_{W_2;\,r_1} }
		\ar[r]^(0.53){T_{r_2\rightarrow 1} \circ(1_{W_1}\fus{P(r_2)}T_{r_1\to 1})}
		& W_1\fus{}(A\fus{}W_2)\ar[d]^{1_{W_1}\fus{}\mu_{W_2}} \\
		W_1\fus{P(r_2)}W_2 \ar[r]^(0.55){T_{r_2\rightarrow 1}}\ar[rd]_{\eta_{r_2}}
		& W_1\fus{}W_2\ar[d]^{\eta}
		\\
		& W_3
	}
	\end{align}
	This diagram commutes for the same reasons as does \eqref{assocdiag2}. Moreover, the composition of morphisms in the left column is the morphism uniquely determined by
	\begin{align*}
	(w_1\fus{P(r_2-r_1)}a)\fus{P(r_1)}w_2\mapsto \cY_\eta(w_1,e^{\log r_2})Y_{W_2}(a,e^{\log r_1})w_2
	\end{align*}
	for $w_1\in W_1$, $a\in A$, and $w_2\in W_2$. We also consider the diagram:
	\begin{align}
	\xymatrixcolsep{12pc}
	\xymatrix{
		(W_1\fus{P(r_2-r_1)} A)\fus{P(r_1)}W_2 
		\ar[d]_{(\cR^-_{P(r_2-r_1)}\fus{P(r_2)}1_{W_2}) \circ T_{r_1\to r_2}}
		\ar[r]^(0.54){T_{r_1\to 1}\circ
			(T_{r_2-r_1\to 1}\fus{P(r_1)}1_{W_2}) }
		& (W_1\fus{} A)\fus{}W_2 \ar[d]^{\cR_{A,W_1}^{-1}\boxtimes 1_{W_2}} \\
		(A\fus{P(r_1-r_2)}W_1)\fus{P(r_2)}W_2 \ar[d]_{\mu_{W_1;\,r_1-r_2}\fus{P(r_2)}1_{W_2}}
		\ar[r]^(0.54){T_{r_2\to 1}\circ
			(T_{r_1-r_2\to 1}\fus{P(r_2)}1_{W_2})}
		& (A\fus{}W_1)\fus{}W_2 \ar[d]^{\mu_{W_1}\fus{}1_{W_2}} \\
		W_1\fus{P(r_2)}W_2 \ar[r]^(0.55){T_{r_2\rightarrow 1}} \ar[rd]_{\eta_{r_2}}
		& W_1\fus{}W_2 \ar[d]^{\eta}
		\\
		& W_3
	}
	\label{eqn:commM1}
	\end{align}
	The top square commutes by Proposition \ref{prop:relntoR+R-}, the naturality of the parallel transport isomorphisms, and Proposition \ref{compofgamma}. The middle square commutes by naturality of
	parallel transports and the definition of $\mu_{W_1; r_1-r_2}$, and the triangle commutes by definition of $\eta_{r_2}$. 
	Using Proposition \ref{braidinganalogue} below, the composition of the left arrows is the morphism determined by:
	\begin{align*}
	(w_1 & \fus{P(r_2-r_1)}a)\fus{P(r_1)}w_2
	\mapsto (-1)^{\vert a\vert\vert w_1\vert}\left(\cY_{\fus{P(r_1-r_2)}, 0}(a,e^{l_{-1/2}(r_1-r_2)})w_1\right)\fus{P(r_2)}w_2
	\\
	&=(-1)^{\vert a\vert\vert w_1\vert}\left( e^{(l_{-1/2}-\log)(r_1-r_2) L(0)}\cdot\right.\nonumber\\
	&\hspace{5em}\left.\cdot\left(e^{-(l_{-1/2}-\log)(r_1-r_2) L(0)}a\fus{P(r_1-r_2)}e^{-(l_{-1/2}-\log)(r_1-r_2) L(0)}w_1   \right)\right)\fus{P(r_2)}w_2\\
	&\mapsto (-1)^{\vert a\vert\vert w_1\vert}\left(e^{(l_{-1/2}-\log)(r_1-r_2) L(0)}\cdot\right.\nonumber\\
	&\hspace{5em}\left.\cdot Y_{W_1}(e^{-(l_{-1/2}-\log)(r_1-r_2) L(0)}a,e^{\log(r_1-r_2)} )e^{-(l_{-1/2}-\log)(r_1-r_2) L(0)}w_1\right)\fus{P(r_2)}w_2\\
	&=(-1)^{\vert a\vert\vert w_1\vert}Y_{W_1}(a,e^{l_{-1/2}(r_1-r_2)} )w_1\fus{P(r_2)}w_2\\
	&\mapsto (-1)^{\vert a\vert\vert w_1\vert}\cY_\eta(Y_{W_1}(a,e^{l_{-1/2}(r_1-r_2)} )w_1,e^{\log r_2})w_2
	\end{align*}
	for all $w_2\in W_2$ and parity-homogeneous $w_1\in W_1$, $a\in A$.
	
	Now, the equality of the morphisms in \eqref{eqn:int1} translates to the equality of the right arrows in the two diagrams above,
	which implies the equality of the left arrows.
	Thus \eqref{eqn:int1} implies
	\begin{align}\label{realskewassoc}
	(-1)^{\vert a\vert\vert w_1\vert}\cY_\eta(w_1,e^{\log r_2})Y_{W_2}(a,e^{\log r_1})w_2=
	\cY_\eta(Y_{W_1}(a,e^{l_{-1/2}(r_1-r_2)} )w_1,e^{\log r_2})w_2
	\end{align}
	as elements of $\overline{W_3}$. Then Proposition \ref{extskewassoc} implies that \eqref{skewassoc} holds for all $(z_1,z_2)\in S_2$.
	
	Conversely, if \eqref{skewassoc} holds for all $(z_1,z_2)\in S_2$, then \eqref{realskewassoc} follows because $(r_1, r_2)$ is in the boundary of $S_2$. This means that the compositions of morphisms in the left columns of \eqref{skewassocdiag1} and \eqref{eqn:commM1} are equal, so that the right columns are equal as well. Thus the equality of the morphisms in \eqref{eqn:int1} follows.
\end{proof}

\begin{rema}
	For a parity-homogeneous $\repA$-intertwining operator $\eta$ of type $\binom{W_3}{W_1\,W_2}$, we can combine the associativity and skew-associativity properties of the corresponding intertwining operator $\cY_\eta$ to obtain the following commutativity property:
	For any $w_2\in W_2$, $w_3'\in W_3'$, and parity-homogeneous $a\in A$, $w_1\in W_1$, the multivalued analytic functions
	\begin{equation*}
	(-1)^{\vert\eta\vert\vert a\vert}\langle w_3', Y_{W_3}(a, z_1)\mathcal{Y}_\eta(w_1, z_2)w_2\rangle
	\end{equation*}
	on the region $\vert z_1\vert>\vert z_2\vert>0$ and
	\begin{equation*}
	(-1)^{\vert a\vert\vert w_1\vert}\langle w_3', \mathcal{Y}_\eta(w_1, z_2)Y_{W_2}(a, z_1)w_2\rangle
	\end{equation*}
	on the region $\vert z_2\vert>\vert z_1\vert>0$ are analytic extensions of each other. Specifically, the branch
	\begin{equation*}
	(-1)^{\vert a\vert\vert w_1\vert}\langle w_3', \mathcal{Y}_\eta(w_1, e^{\mathrm{log}\,z_2})Y_{W_2}(a, e^{\mathrm{log}\,z_1})w_2\rangle
	\end{equation*}
	of $(-1)^{\vert a\vert\vert w_1\vert}\langle w_3', \mathcal{Y}_\eta(w_1, z_2)Y_{W_2}(a, z_1)w_2\rangle$ on $S_2$ is determined by the branch
	\begin{equation*}
	(-1)^{\vert\eta\vert\vert a\vert}\langle w_3', Y_{W_3}(a, e^{\mathrm{log}\,z_1})\mathcal{Y}_\eta(w_1, e^{\mathrm{log}\,z_2})w_2\rangle
	\end{equation*}
	of $(-1)^{\vert\eta\vert\vert a\vert}\langle w_3', Y_{W_3}(a, z_1)\mathcal{Y}_\eta(w_1, z_2)w_2\rangle$ on $S_1$ and a certain path $\gamma$ contained in the region $\vert z_2\vert>\vert z_1-z_2\vert>0$. We can choose the path $\gamma$ as follows. Choose $r_1,r_2\in\RR$ such that $r_1>r_2>r_1-r_2>0$ (so that $(r_1,r_2)$ is on the border of $S_1$), and let $\gamma$ be the path from $(r_1,r_2)$ to $(r_1, 2r_1-r_2)$ obtained by holding $r_1$ fixed and moving $r_2$ to $2r_1-r_2$ clockwise along the semicircle in the upper half plane whose center is $r_1$ and radius $r_1-r_2$. Notice that $(r_1, 2r_1-r_2)$ is on the border of $S_2$, and that along this path, $\mathrm{log}(r_1-r_2)$ changes continuously to $l_{-1/2}(r_1-(2r_1-r_2))$.
\end{rema}

\subsection{Vertex tensor category structure on \texorpdfstring{$\repzA$}{Rep0 A}}\label{subsec:VTConRepARep0A}

In this subsection we interpret the monoidal supercategory structure on $\repA$ and the braided monoidal supercategory structure on $\repzA$ using vertex-algebraic intertwining operators. This will allow us to show that the braided monoidal supercategory structure on $\mathrm{Rep}^0\,A$ constructed in \cite{KO} and reviewed in Section \ref{sec:TensCats} agrees with the vertex-algebraic braided monoidal supercategory structure on the category of grading-restricted generalized $A$-modules in $\cC$ constructed in \cite{HLZ8} and reviewed in Section \ref{subsec:VTC}.

\subsubsection{Objects in $\repA$ and $\repzA$}

First we describe objects and morphisms in $\repA$ in terms of intertwining operators:
\begin{propo}\label{propo:YWcomplex}
 An object of $\repA$ amounts to a $V$-module $W$ equipped with an even $V$-intertwining operator $Y_W$ of type $\binom{W}{A\,W}$ that satisfies the following properties:
 \begin{enumerate}
  \item (Associativity) For $a_1, a_2\in A$, $w\in W$, $w'\in W'$, and $(z_1,z_2)\in S_1$,
  \begin{equation*}
		\langle w', Y_{W}(a_1, e^{\mathrm{log}\,z_1})Y_W(a_2, e^{\mathrm{log}\,z_2})w\rangle=\langle w', Y_W(Y(a_1, e^{\mathrm{log}(z_1-z_2)})a_2, e^{\mathrm{log}\,z_2})w\rangle.
		\end{equation*}
In particular,  the multivalued analytic functions 
		\begin{equation*}
		\langle w', Y_{W}(a_1, z_1)Y_W(a_2, z_2)w\rangle
		\end{equation*}
		on the region $\vert z_1\vert>\vert z_2\vert>0$ and
		\begin{equation*}
		\langle w', Y_W(Y(a_1, z_1-z_2)a_2, z_2)w\rangle
		\end{equation*}
		on the region $\vert z_2\vert>\vert z_1-z_2\vert>0$ have equal restrictions to their common domain. 
\item (Unit) $Y_W(\mathbf{1}, x)=1_W$, or equivalently, $Y_W\vert_{V\otimes W}$ is the vertex operator of $V$ acting on $W$.
 \end{enumerate}
Parity-homogeneous morphisms in $\repA$ amount to parity-homogeneous $V$-module homomorphisms $f: W_1\rightarrow W_2$ such that
\begin{equation*}
 f(Y_{W_1}(a,x)w_1)=(-1)^{\vert f\vert\vert a\vert}Y_{W_2}(a,x)f(w_1)
\end{equation*}
for $w_1\in W_1$ and parity-homogeneous $a\in A$.
\end{propo}
\begin{proof}
 Given an object $(W,\mu_W)$ of $\repA$ as defined in Definition \ref{repSAdef}, we recall from Remark \ref{rema:YW} the $V$-intertwining operator $Y_W$ of type $\binom{W}{A\,W}$ characterized by
 \begin{equation}\label{YWchar}
  \overline{\mu_W}(a\boxtimes w)=Y_W(a, e^{\mathrm{log}\,1})w
 \end{equation}
for $a\in A$ and $w\in W$. Conversely, given a $V$-intertwining operator $Y_W$ of type $\binom{W}{A\,W}$, the universal property of $A\boxtimes W$ implies there is a unique $V$-homomorphism $\mu_W: A\boxtimes W\rightarrow W$ such that \eqref{YWchar} holds. 
By \eqref{YWchar}, the intertwining operator $Y_W$ is even if and only if $\mu_W$ is even. So we need to show that the associativity and unit properties for $Y_W$ in the proposition are equivalent to the associativity and unit properties of $\mu_W$ in Definition \ref{repSAdef}.

For the associativity, Theorem \ref{thm:repAtocomplex} and its proof imply that the associativity for $Y_W$ in the statement of the proposition is equivalent to the equality
\begin{equation*}
 \mu_W\circ(1_A\boxtimes \mu_W)=\mu_W\circ(\mu\boxtimes 1_W)\circ\sA_{A,A,W},
\end{equation*}
which is the associativity property for $\mu_W$ in Definition \ref{repSAdef}. 

For the unit property, consider the diagram
\begin{equation*}
 \xymatrixcolsep{4pc}
 \xymatrix{
 W \ar[r]^{l_{P(z)}^{-1}} \ar[rd]_{l_W^{-1}} & V\boxtimes_{P(z)} W \ar[r]^{\iota_A\boxtimes_{P(z)} 1_W} \ar[d]^{T_{z\rightarrow 1}} & A\boxtimes_{P(z)} W \ar[r]^{\mu_{W; z}} \ar[d]^{T_{z\rightarrow 1}} & W\\
  & V\boxtimes W \ar[r]^{\iota_A\boxtimes 1_W} & A\boxtimes W \ar[ru]_{\mu_W} & \\
 }
\end{equation*}
for $z\in\CC^\times$. The square commutes by naturality of the parallel transport isomorphisms, and the right triangle commutes by the definition of $\mu_{W; z}$ from the previous subsection. Recall also from Remark \ref{rema:YW} that
\begin{equation*}
 \overline{\mu_{W; z}}(a\boxtimes_{P(z)} w)=Y_W(a, e^{\mathrm{log}\,z})w
\end{equation*}
for $a\in A$ and $w\in W$. To see that the left triangle in the diagram commutes, we calculate, recalling the definition of $l_{P(z)}$ and temporarily using $\widetilde{Y}_W$ to denote the vertex operator of $V$ acting on $W$:
\begin{align*}
 \overline{l_{P(z)}\circ T_{1\rightarrow z}}(v\boxtimes w) & =\overline{l_{P(z)}}(\cY_{\boxtimes_{P(z)}, 0}(v, e^{\mathrm{log}\,1})w)\nonumber\\
 & =\overline{l_{P(z)}}\left( e^{-\mathrm{log}\,z\,L(0)}\cY_{\boxtimes_{P(z)},0}(e^{\mathrm{log}\,z\,L(0)} v, e^{\mathrm{log}\,z})e^{\mathrm{log}\,z\,L(0)} w\right)\nonumber\\
 & =e^{-\mathrm{log}\,z\,L(0)}\overline{l_{P(z)}}\left(e^{\mathrm{log}\,z\,L(0)} v\boxtimes_{P(z)} e^{\mathrm{log}\,z\,L(0)} w\right)\nonumber\\
 & =e^{-\mathrm{log}\,z\,L(0)}\widetilde{Y}_W(e^{\mathrm{log}\,z\,L(0)}, z)e^{\mathrm{log}\,z\,L(0)}w\nonumber\\
 & =\widetilde{Y}_W(v, 1)w=l_W(v\boxtimes w)
\end{align*}
for any $v\in V$ and $w\in W$, so indeed $l_W=l_{P(z)}\circ T_{1\rightarrow z}$. 

Now, the unit property for $\mu_W$ holds if and only if the bottom (or equivalently, the top) row of the diagram equals the identity on $W$. The top row is given by the mapping
\begin{align*}
 \widetilde{Y}_W(v, z)w\mapsto v\boxtimes_{P(z)} w\mapsto v\boxtimes_{P(z)} w\mapsto Y_W(v, e^{\mathrm{log}\,z})w,
\end{align*}
since $\iota_A$ is simply the inclusion of $V$ into $A$. This is the identity if and only if
\begin{equation*}
 \widetilde{Y}_W(v,z)=Y_W(v, e^{\mathrm{log}\,z})w
\end{equation*}
for any $v\in V$ and $w\in W$, or equivalently by Proposition \ref{opmapiso}, 
\begin{equation*}
 \widetilde{Y}_W(v, x)w=Y_W(v, x)w
\end{equation*}
for $v\in V$ and $w\in W$. Thus the unit property for $\mu_W$ holds if and only if $\widetilde{Y}_W=Y_W\vert_{V\otimes W}$. In particular, taking $v=\mathbf{1}$ yields $Y_W(\mathbf{1}, x)=1_W$ for an object of $\repA$. But in fact $Y_W(\mathbf{1}, x)=1_W$ is also equivalent to the unit property for $\mu_W$ because the composition in the top row in the diagram is also given by
\begin{equation*}
 w\mapsto\mathbf{1}\boxtimes_{P(z)} w\mapsto Y_W(\mathbf{1}, e^{\mathrm{log}\,z})w,
\end{equation*}
which equals the identity on $W$ if and only if $Y_W(\mathbf{1},x)=1_W$.

Now we show that a parity-homogeneous $V$-module homomorphism $f: W_1\rightarrow W_2$ between two modules in $\repA$ is a parity-homogeneous $\repA$-morphism in the sense of Definition \ref{repSAdef} if and only if
\begin{equation*}
 f(Y_{W_1}(a, x)w_1)=(-1)^{\vert f\vert\vert a\vert}Y_{W_2}(a, x)f(w_1)
\end{equation*}
for any $w_1\in W_1$ and parity-homogeneous $a\in A$. The diagrams
\begin{equation*}
 \xymatrixcolsep{4pc}
 \xymatrix{
 A\boxtimes_{P(z)} W_1 \ar[r]^{\mu_{W_1; z}} \ar[d]_{T_{z\rightarrow 1}} & W_1 \ar[r]^{f} & W_2\\
 A\boxtimes W_1 \ar[ur]_{\mu_{W_1}} & &\\
 }
\end{equation*}
and
\begin{equation*}
 \xymatrixcolsep{4pc}
 \xymatrix{
 A\boxtimes_{P(z)} W_1 \ar[r]^{1_A\boxtimes_{P(z)} f} \ar[d]_{T_{z\rightarrow 1}} & A\boxtimes_{P(z)} W_2 \ar[r]^{\mu_{W_2; z}} \ar[d]_{T_{z\rightarrow 1}} & W_2\\
 A\boxtimes W_1 \ar[r]^{1_A\boxtimes f} & A\boxtimes W_2 \ar[ur]_{\mu_{W_2}} & \\
 }
\end{equation*}
commute for any $z\in\CC^\times$ by the naturality of the parallel transport isomorphisms and the definitions of $\mu_{W_1; z}$ and $\mu_{W_2; z}$. Thus using \eqref{homtensprod},
\begin{equation*}
 f\circ\mu_{W_1}=\mu_{W_2}\circ(1_A\boxtimes f)
\end{equation*}
if and only if
\begin{equation*}
 \overline{f}(Y_{W_1}(a, e^{\mathrm{log}\,z})w_1)=(-1)^{\vert f\vert\vert a\vert}Y_{W_2}(a, e^{\mathrm{log}\,z})f(w_1)
\end{equation*}
for $w_1\in W_1$ and parity-homogeneous $a\in A$, or equivalently by Proposition \ref{opmapiso}, if and only if
\begin{equation*}
 f(Y_{W_1}(a, x)w_1)=(-1)^{\vert f\vert\vert a\vert}Y_{W_2}(a, x)f(w_1)
\end{equation*}
for $w_1\in W_1$ and parity-homogeneous $a\in A$.
\end{proof}

\begin{rema}
Suppose we are given a $V$-module $W$ with a $V$-intertwining operator of type
$\binom{W}{A\, W}$ that satisfies associativity and unit property as
in \cref{propo:YWcomplex}. Then, since $(W,\mu_W)$ is an object of $\repA$, $\mu_W$ is a $\repA$-intertwining operator
and  \cref{thm:repAtocomplex} shows that 
$Y_W$ also satisfies skew-associativity:
For any $w\in W$, $w'\in W'$, parity-homogeneous $a_1, a_2\in A$, and $(z_1,z_2)\in S_2$,
\begin{equation*}
(-1)^{\vert a_1\vert\vert a_2\vert}\langle w', Y_W(a_2, e^{\mathrm{log}\,z_2})Y_{W}(a_1, e^{\mathrm{log}\,z_1})w\rangle=\langle w_3', Y_W(Y_{W}(a_1, e^{l_{-1/2}(z_1-z_2)})a_2, e^{\mathrm{log}\,z_2})w\rangle.
\end{equation*}
In particular, the multivalued analytic functions 
\begin{equation*}
(-1)^{\vert a_1\vert\vert a_2\vert}\langle w', Y_W(a_2, z_2)Y_{W}(a_1, z_1)w\rangle
\end{equation*}
on the region $\vert z_2\vert>\vert z_1\vert>0$ and
\begin{equation*}
\langle w', Y_W(Y_{W}(a_1, z_1-z_2)a_2, z_2)w\rangle
\end{equation*}
on the region $\vert z_2\vert>\vert z_1-z_2\vert>0$ have equal restrictions to their common domain.
\end{rema}

The vertex-algebraic characterization of the category $\repzA$ was given in \cite[Theorem 3.4]{HKL} and in  \cite[Theorem 3.14]{CKL} for the superalgebra setting:
\begin{propo}
 An object $(W, Y_W)$ of $\repA$ is an object of $\repzA$ if and only if $(W, Y_W)$ is a grading-restricted generalized $A$-module in the category $\cC$.
\end{propo}

Since this proposition is in \cite{HKL, CKL} we only briefly discuss its proof. From \eqref{vrtxmonodromy}, a $\repA$-module $(W,Y_W)$ is in $\repzA$ if and only if
\begin{equation*}
 Y_W(a,x)w=Y_W(a, e^{2\pi i} x)w
\end{equation*}
for $a\in A$ and $w\in W$. This condition will hold if and only if $Y_W$ involves only integral powers of the formal variable $x$, so that $W$ is in $\repzA$ if and only if $Y_W(a,x)w\in W[[x,x^{-1}]]$ for $a\in A$ and $w\in W$, as is the case with grading-restricted generalized $A$-modules. The grading restriction conditions, vacuum property, and Virasoro algebra properties for a grading-restricted generalized $A$-module follow because $W$ is a grading-restricted generalized $V$-module and from the unit property of an object in $\repA$. Lower truncation and the $L(-1)$-derivative property follow because $Y_W$ is a $V$-intertwining operator. Finally, the Jacobi identity for $Y_W$ is equivalent to the associativity of $Y_W$ and the skew-symmetry of the vertex operator $Y$ on $A$, as in the proof of \cite[Theorem 3.2]{HKL}.

\subsubsection{Intertwining operators and tensor products in $\repA$ and $\repzA$}

Now we  relate the tensor product $\boxtimes_A$ on $\repA$ and $\repzA$ to intertwining operators. Suppose $W_1$, $W_2$, and $W_3$ are three modules in $\repA$. 
\begin{defi}\label{def:Aintwop}
 An \textit{$A$-intertwining operator} of type $\binom{W_3}{W_1\,W_2}$ is any sum of even and odd $V$-intertwining operators which satisfy the associativity and skew-associativity conditions of Theorem \ref{thm:repAtocomplex}
\end{defi}
Equivalently by Theorem \ref{thm:repAtocomplex}, a $V$-intertwining operator $\cY$ of type $\binom{W_3}{W_1\,W_2}$ is an $A$-intertwining operator if and only if the unique $V$-module homomorphism $\eta_\cY: W_1\boxtimes W_2\rightarrow W_3$ induced by $\cY(\cdot,e^{\log 1})\cdot$ and the universal property of $W_1\boxtimes W_2$ is a categorical $\repA$-intertwining operator.

For a pair of modules $W_1$ and $W_2$ in $\repA$, we denote by $\cY_{W_1,W_2}$ the $A$-intertwining operator $\cY_{\eta_{W_1,W_2}}$ induced by the $\repA$-intertwining operator $\eta_{W_1,W_2}: W_1\boxtimes W_2\rightarrow W_1\boxtimes_A W_2$. Then $W_1\boxtimes_A W_2$ satisfies the following universal property:
\begin{propo}\label{tensAunivprop}
 For any module $W_3$ in $\repA$ and $A$-intertwining operator $\cY$ of type $\binom{W_3}{W_1\,W_2}$, there is a unique $\repA$-morphism $f: W_1\boxtimes_A W_2\rightarrow W_3$ such that $f\circ\cY_{W_1,W_2}=\cY$.
\end{propo}
\begin{proof}
 If $\cY$ is an $A$-intertwining operator of type $\binom{W_3}{W_1\,W_2}$, let $\eta_\cY: W_1\boxtimes W_2\rightarrow W_3$ be the unique $\repA$-intertwining operator such that $\eta_\cY\circ\boxtimes=\cY(\cdot, e^{\mathrm{log}\,1})\cdot$. Then by Proposition \ref{prop:repAuniv}, there is a unique $\repA$-morphism $f: W_1\boxtimes_A W_2\rightarrow W_3$ such that $f\circ\eta_{W_1,W_2}=\eta_\cY$. Thus for $w_1\in W_1$ and $w_2\in W_2$,
 \begin{align*}
  \cY(w_1, e^{\mathrm{log}\,1})w_2=\overline{\eta_\cY}(w_1\boxtimes w_2) = \overline{f\circ\eta_{W_1,W_2}}(w_1\boxtimes w_2)=\overline{f}(\cY_{W_1,W_2}(w_1, e^{\mathrm{log}\,1})w_2).
 \end{align*}
By Proposition \ref{opmapiso}, it follows that $\cY(w_1, x)w_2=f(\cY_{W_1,W_2}(w_1, x)w_2)$, as desired.
\end{proof}

\begin{rema}
 For any $z\in\CC^\times$ and modules $W_1$, $W_2$ in $\repA$, we can define $W_1\boxtimes^A_{P(z)} W_2$ to be the object $W_1\boxtimes_A W_2$ of $\repA$ equipped with the $P(z)$-intertwining map $\boxtimes^A_{P(z)}=\cY_{W_1,W_2}(\cdot, e^{\mathrm{log}\,z})\cdot$. Then the same kind of argument as in the preceding proposition shows that this module and $P(z)$-intertwining map satisfy a universal property for an appropriate notion of $P(z)$-intertwining map in $\repA$. We will use this $P(z)$-tensor product in $\repA$ to obtain some elements of vertex tensor category structure on $\repA$ below, although $\repA$ will not have parallel transport or braiding isomorphisms.
\end{rema}

It is clear from the characterization \eqref{repAtensprodmorphdef} of tensor products of morphisms in $\repA$, and the characterization \eqref{homtensprod} of tensor products of morphisms in $\cC$, that for any morphisms $f_1: W_1\rightarrow \widetilde{W}_1$ and $f_2: W_2\rightarrow\widetilde{W}_2$ in $\repA$, with $f_2$ parity-homogeneous, we have
\begin{equation*}
 \overline{f_1\boxtimes_A f_2}(\cY_{W_1,W_2}(w_1, e^{\mathrm{log}\,1})w_2)=(-1)^{\vert f_2\vert\vert w_1\vert}\cY_{\widetilde{W}_1,\widetilde{W}_2}(f_1(w_1), e^{\mathrm{log}\,1})f_2(w_2)
\end{equation*}
for $w_2\in W_2$ and parity-homogeneous $w_1\in W_1$. Equivalently,
\begin{equation*}
 (f_1\boxtimes_A f_2)(\cY_{W_1,W_2}(w_1, x)w_2)=(-1)^{\vert f_2\vert\vert w_1\vert}\cY_{\widetilde{W}_1,\widetilde{W}_2}(f_1(w_1), x)f_2(w_2)
\end{equation*}
for $w_2\in W_2$ and parity-homogeneous $w_1\in W_1$. 
\begin{rema}
For any $z\in\CC^\times$, if we use the notation $f_1\boxtimes_A f_2=f_1\boxtimes^A_{P(z)} f_2$ when we view $f_1\boxtimes_A f_2$ as a morphism between $P(z)$-tensor products, we have
\begin{equation*}
 \overline{f_1\boxtimes^A_{P(z)} f_2}(w_1\boxtimes^A_{P(z)} w_2)=(-1)^{\vert f_2\vert\vert w_1\vert}f_1(w_1)\boxtimes^A_{P(z)} f_2(w_2)
\end{equation*}
for $w_2\in W_2$ and parity-homogeneous $w_1\in W_1$. 
\end{rema}

Now we consider the tensor product on $\repzA$. Since $\repzA$ is the category of grading-restricted generalized $A$-modules in $\cC$, we have two tensor products on $\repzA$: the tensor product $\boxtimes_A$ on $\repA$ and the vertex-algebraic tensor product of $A$-modules from \cite{HLZ3}. Thus to apply all results regarding $\boxtimes_A$ to the vertex-algebraic category of $A$-modules, it is critical to show these tensor products are the same. To do so, we must show that the notion of $A$-intertwining operator defined above among objects of $\repzA$ agrees with the notion of (logarithmic) intertwining operator among grading-restricted generalized modules for $A$ as a vertex operator superalgebra. To prove this, we need the following theorem:
\begin{theo}
	\label{thm:complextoformal}
	Given a vertex operator superalgebra $A$ and grading-restricted generalized $A$-modules $W_1$, $W_2$, and $W_3$,
	consider a parity-homogeneous map
	\begin{align*}
	\cY:W_1\otimes W_2&\longrightarrow W_3[\log x]\{x\}\\
	w_1\otimes w_2&\longmapsto \cY(w_1,x)w_2 = \sum_{n\in\mathbb{C}}\sum_{k\in\mathbb{Z}}(w_1)^\cY_{n;k}w_2\,x^{-n-1}(\log x)^k
	\end{align*}
	that satisfies the following properties:
	\begin{enumerate}
		\item \emph{Lower truncation:} For any $w_1\in W_1$, $w_2\in W_2$ and $n\in\CC$, $(w_1)^\cY_{n+m;k}w_2=0$ for sufficiently large $m\in\mathbb{N}$
		independently of $k$. \label{item:complexLower}
		\item The \emph{$L(-1)$-derivative formula}: $[L(-1),\cY(w_1,x)] = \dfrac{d}{dx}\cY(w_1,x) = \cY(L(-1)w_1,x)$ for $w_1\in W_1$.
		\label{item:complexL0}
		\item The \emph{$L(0)$-bracket formula:} $[L(0),\cY(w_1,x)] = x\dfrac{d}{dx}\cY(w_1,x) + \cY(L(0)w_1,x)$ for $w_1\in W_1$.
		\label{item:complexL-1}
		\item Given $w_3'\in W_3'$, $w_2\in W_2$, and parity-homogeneous $a\in A$, $w_1\in W_1$, the following series define multivalued analytic functions on the indicated
		domains that have equal restrictions to their respective intersections:
		\label{item:complexMain}
		\begin{align}
		(-1)^{\vert\cY\vert\vert a\vert}\langle w_3',Y_{W_3}(a,z_1)\cY(w_1,z_2)w_2 \rangle &\quad\quad |z_1|>|z_2|>0 \label{eqn:prodYI}\\
		\langle w_3',\cY(Y_{W_1}(a,z_1-z_2)w_1,z_2)w_2 \rangle &\quad\quad |z_2|>|z_1-z_2|>0 \label{eqn:assocIY} \\
		(-1)^{\vert a\vert\vert w_1\vert}\langle w_3',\cY(w_1,z_2)Y_{W_2}(a,z_1)w_2 \rangle &\quad\quad |z_2|>|z_1|>0. \label{eqn:prodIY}
		\end{align}
		Moreover, the principal branches (that is, $z_2=e^{\log z_2}$) of the first and second multivalued functions yield
		single-valued functions that are equal on the simply-connected domain $S_1$,
		and
		the principal branches  of the second and third multivalued functions yield
		single-valued functions that are equal on the simply connected domain $S_2$.
		\item There exists a multivalued analytic function $f$ defined on $ \{(z_1,z_2)\in\mathbb{C}^2\,|\,z_1,z_2,z_1-z_2\neq 0 \}$
		that restricts to the three multivalued functions above on their respective domains.
		\label{item:complexMainFunction}
	\end{enumerate}
	Then $\cY$ is a logarithmic intertwining operator of type $\binom{W_3}{W_1\,W_2}$ in the sense of Definition \ref{def:intwop}. In particular, $\cY$ satisfies the Jacobi identity.
\end{theo}
\begin{proof}
	First fix $z_2\in\mathbb{C}^\times$ such that $\mathrm{Re}\,z_2, \mathrm{Im}\,z_2>0$. Then $f_{z_2}(z_1)=f(z_1,e^{\log z_2})$ yields a possibly multivalued
	analytic function of $z_1$ defined on $\{z_1\in\mathbb{C}^\times\,|\, z_1\neq z_2 \}$.
	However, $f_{z_2}(z_1)$ has single-valued restrictions to each of $\{z_1\in\mathbb{C}^\times\,|\,|z_1|>|z_2|\}$,
	$\{z_1\in\mathbb{C}^\times\,|\,|z_2|>|z_1-z_2|>0\}$ and $\{z_1\in\mathbb{C}^\times\,|\,|z_2|>|z_1|\}$,
	and these restrictions are equal on certain simply-connected subsets of the intersections of these domains. Thus these restrictions define a single-valued analytic function on the union of these domains. This implies that $f_{z_2}$ has a single-valued restriction $\widetilde{f}_{z_2}$ to $\CC^\times\setminus\lbrace z_2\rbrace$ as follows:
	
	Fix $z_1'$ such that, say, $(z_1', z_2)\in S_2$, so that $z_1'$ is in the domain of both the iterate \eqref{eqn:assocIY} and the product \eqref{eqn:prodIY}. Then define $\widetilde{f}_{z_2}$ by restricting the values of $f_{z_2}$ at $z_1\in\CC^\times\setminus\lbrace z_2\rbrace$ to be only those obtainable by analytic continuation along paths in $\CC^\times\setminus\lbrace z_2\rbrace$ from $z_1'$ to $z_1$, starting with the value
	$\langle w_3',\cY(Y_{W_1}(a,z_1'-z_2)w_1,e^{\log z_2})w_2 \rangle$ (or equivalently, $(-1)^{\vert a\vert\vert w_1\vert}\langle w_3',\cY(w_1,e^{\log z_2})Y_{W_2}(a,z_1')w_2 \rangle$) of $f_{z_2}(z_1')$.
	
	To show that $\widetilde{f}_{z_2}$ is actually single-valued, we need to show that for any $z_1\in\CC^\times\setminus\lbrace z_2\rbrace$ and any two paths $\gamma_1$ and $\gamma_2$ in $\CC^\times\setminus\lbrace z_2\rbrace$ from $z_1'$ to $z_1$, the values of $f_{z_2}(z_1)$ determined by $\gamma_1$ and $\gamma_2$ are the same; equivalently, we need to show that the value of $f_{z_2}(z_1')$ is unchanged after analytic continuation along the path $\widetilde{\gamma}$ obtained by following $\gamma_1$ and then following the reversal of $\gamma_2$. In fact, $\widetilde{\gamma}$ is homotopic to some path consisting of a sequence of loops around $z_2$ and $0$ that stay in the domains of \eqref{eqn:assocIY} and \eqref{eqn:prodIY}, respectively. Since the union of these domains is contained in the domain of a single-valued branch of $f_{z_2}$, and since the value of $f_{z_2}(z_1')$ that we start with is the value of this single-valued branch, the value of $f_{z_2}(z_1')$ is unchanged after analytic continuation along $\widetilde{\gamma}$. Thus $\gamma_1$ and $\gamma_2$ determine the same value of $f_{z_2}$ at $z_1$, and $\widetilde{f}_{z_2}$ is a single-valued analytic restriction of $f_{z_2}$.

	Now, \eqref{eqn:prodYI} along with the lower truncation property for $Y_{W_3}$ shows that the restriction $\widetilde{f}_{z_2}$ has
	at worst a pole at $z_1=\infty$. Similarly, \eqref{eqn:assocIY} shows that $\widetilde{f}_{z_2}$ has
	at worst a pole at $z_1=z_2$ and \eqref{eqn:prodIY} shows that $\widetilde{f}_{z_2}$ has
	at worst a pole at $z_1=0$. A complex-analytic function with only singularities at $0,z_2,\infty$ that are at worst poles has to be a rational function. Now standard arguments using properties of formal delta-functions (see for instance Proposition 8.10.7 in \cite{FLM}, Propositions 3.4.1 and 4.5.1 in \cite{FHL}, or Propositions 3.4.1 and 4.4.2 in \cite{LL}) show that $I(w_1\otimes w_2):=\cY(w_1,e^{\log z_2})w_2$ satisfies the Jacobi identity and is a $P(z_2)$-intertwining map.
	
	Now we apply the $L(0)$-conjugation formula 
	\begin{equation*}
	 \cY(w_1, x)w_2=y^{-L(0)} \cY(y^{L(0)} w_1, xy)y^{L(0)} w_2,
	\end{equation*}
which follows from the $L(0)$-bracket formula, in the case $y=\frac{e^{\mathrm{log}\,z}}{x}$. This shows that $\cY(w_1,x)w_2=\cY_{I,0}(w_1,x)w_2$, for $w_1\in W_1$ and $w_2\in W_2$. Hence $\cY$ is an intertwining operator of type $\binom{W_3}{W_1\,W_2}$.
\end{proof}

This result allows us to obtain the following fundamental theorem:
\begin{theo}\label{repzAintwopcorrect}
 Suppose $W_1$, $W_2$, and $W_3$ are three modules in $\repzA$. Then $\cY$ is an $A$-intertwining of type $\binom{W_3}{W_1\,W_2}$ in the sense of Definition \ref{def:Aintwop} if and only if $\cY$ is an intertwining operator among $A$-modules in the sense of Definition \ref{def:intwop}.
\end{theo}
\begin{proof}
 An intertwining operator in the sense of Definition \ref{def:intwop} is an intertwining operator in the sense of Definition \ref{def:Aintwop} by standard commutativity and associativity results for intertwining operators (see for instance Propositions 4.2.4 and 4.3.4 in \cite{LL} for the case that $\cY$ is a module vertex operator; the proof for general $\cY$ is the same).
 
 Conversely, suppose $\cY$ satisfies the associativity and skew-associativity conditions of Theorem \ref{thm:complextoformal}.    Since $\cY$ is an intertwining operator among $V$-modules, it satisfies
	conditions \eqref{item:complexLower}, \eqref{item:complexL0} and \eqref{item:complexL-1}
	from Theorem \ref{thm:complextoformal}. It satisfies condition \eqref{item:complexMain} due
	to Theorem \ref{thm:repAtocomplex}. Condition \eqref{item:complexMainFunction} follows from Lemma 4.1 in \cite{H-genlratl} and our assumptions on the category $\mathcal{C}$ (in particular the associativity of intertwining operators among $V$-modules in $\cC$; see for instance Corollary 9.30 in \cite{HLZ6}). Therefore, by the conclusion of Theorem \ref{thm:complextoformal}, $\cY$ is an intertwining operator among $A$-modules in the sense of Definition \ref{def:intwop}.
\end{proof}

As a consequence of this result, we see that $\boxtimes_A$ in $\repzA$ agrees with the vertex-algebraic notion of tensor product of modules:
\begin{corol}\label{PztensprodsinrepzA}
 Suppose $W_1$ and $W_2$ are two modules in $\repzA$. Then for any $z\in\CC^\times$, $(W_1\boxtimes_A W_2, \cY_{W_1,W_2}(\cdot, e^{\mathrm{log}\,z})\cdot)$ is a $P(z)$-tensor product of $W_1$ and $W_2$.
\end{corol}
\begin{proof}
 Suppose $W_3$ is another module in $\repzA$ and $I$ is a $P(z)$-intertwining map of type $\binom{W_3}{W_1\,W_2}$. Then by Proposition \ref{tensAunivprop} and Theorem \ref{repzAintwopcorrect} there is a unique $A$-module homomorphism $f: W_1\boxtimes_A W_2\rightarrow W_3$ such that $f\circ\cY_{W_1,W_2}=\cY_{I,0}$. In particular, this means
 \begin{equation*}
  \overline{f}(\cY_{W_1,W_2}(w_1, e^{\mathrm{log}\,z})w_2) =\cY_{I,0}(w_1, e^{\mathrm{log}\,z})w_2= I(w_1\otimes w_2)
 \end{equation*}
for $w_1\in W_1$, $w_2\in W_2$, showing that $(W_1\boxtimes_A W_2, \cY_{W_1,W_2}(\cdot, e^{\mathrm{log}\,z})\cdot)$ satisfies the universal property of a $P(z)$-tensor product as given in Definition \ref{def:Pztensprod}.
\end{proof}

\begin{rema}
 In particular, the preceding corollary shows that $(W_1\boxtimes_A W_2, \boxtimes_A=\cY_{W_1,W_2}(\cdot, e^{\mathrm{log}\,1})\cdot)$ is a $P(1)$-tensor product of $W_1$ and $W_2$ in the category of $A$-modules in $\cC$. In the remainder of this subsection, we will show that the rest of the braided monoidal supercategory structure on $\repzA$ agrees with the braided monoidal supercategory structure on a category of $A$-modules constructed in \cite{HLZ8}.
\end{rema}

\subsubsection{Parallel transport isomorphisms in $\repzA$}

We have shown in Corollary \ref{PztensprodsinrepzA} that $P(z)$-tensor products in $\repzA$ exist for any $z\in\CC^\times$. Since the existence of natural parallel transport isomorphisms associated to continuous curves in $\CC^\times$ only depends on the existence of $P(z)$-tensor products and properties of intertwining operators, we have natural parallel transport isomorphisms in $\repzA$.

Suppose $W_1$ and $W_2$ are two modules in $\repzA$, and $\gamma$ is a continous path in $\CC^\times$ from $z_1$ to $z_2$. Then by \eqref{tgammachar}, the parallel transport isomorphism
\begin{equation*}
 T^A_\gamma: W_1\boxtimes_A W_2\rightarrow W_1\boxtimes_A W_2
\end{equation*}
is determined by the condition 
\begin{equation*}
 \overline{T^A_\gamma}(w_1\boxtimes^A_{P(z_1)} w_2)=\cY_{W_1,W_2}(w_1, e^{l(z_1)})w_2
\end{equation*}
for $w_1\in W_1$, $w_2\in W_2$, where $l(z_1)$ is the branch of logarithm determined by $\mathrm{log}\,z_2$ and the path $\gamma$. Now, if $l(z_1)=\mathrm{log}\,z_1+2\pi i p$ for some $p\in\ZZ$, Proposition \ref{opmapiso} implies that also
\begin{equation}\label{RepzAparatrans}
 T^A_\gamma(\cY_{W_1,W_2}(w_1,x)w_2)=\cY_{W_1,W_2}(w_1, e^{2\pi i p} x)w_2
\end{equation}
for $w_1\in W_1$, $w_2\in W_2$.

Similar to parallel transport isomorphisms among $V$-modules in $\cC$, we use the notation $T^A_{z_1\rightarrow z_2}$ when $\gamma$ is a path in $\CC^\times$ with a branch cut along the positive real axis. But note that since $l(z_1)=\mathrm{log}\,z_1$ for such a path, these isomorphisms are the identity given our realization of $P(z)$-tensor products in Corollary \ref{PztensprodsinrepzA}.

\begin{rema}
 If $W_1$ and $W_2$ are modules in $\repA$ but are not in $\repzA$, we cannot obtain parallel transport isomorphisms unless $\gamma$ has trivial monodromy. This is because vertex operators for modules in $\repA$ may have non-trivial monodromy (non-integral powers of the formal variable $x$), so that we cannot guarantee $\cY_{W_1,W_2}(\cdot, e^{2\pi i p} x)\cdot$ is an $A$-intertwining operator for $p\neq 0$ when $W_1$ and $W_2$ are not objects of $\repzA$.
\end{rema}

\subsubsection{Unit isomorphisms in $\repA$ and $\repzA$}

Here we describe the unit isomorphisms in $\repA$ in terms of $A$-intertwining operators. For the left unit isomorphisms, the definitions show that the diagram
\begin{equation*}
 \xymatrixcolsep{4pc}
 \xymatrix{
 A\otimes W \ar[r]^{\boxtimes} \ar[rd]_{Y_W(\cdot, e^{\mathrm{log}\,1})\cdot} & \overline{A\boxtimes W} \ar[r]^{\overline{\eta_{A,W}}} \ar[d]^{\overline{\mu_W}} & \overline{A\boxtimes_A W} \ar[dl]^{\overline{l^A_W}}\\
  & \overline{W} & \\
 }
\end{equation*}
commutes, from which we obtain
\begin{equation*}
 l^A_W(\cY_{A, W}(a, x)w)=Y_W(a,x)w
\end{equation*}
for $a\in A$ and $w\in W$ since $\overline{\eta_{A,W}}\circ\boxtimes =\cY_{A,W}(\cdot, e^{\mathrm{log}\,1})\cdot$. 
\begin{rema}
For $z\in\CC^\times$, if we use the notation $l^A_{P(z)}$ for $l^A_W$ when we view $l^A_W$ as a map on a $P(z)$-tensor product, we have
\begin{equation*}
 \overline{l^A_{P(z)}}(a\boxtimes^A_{P(z)} w)= Y_W(a, e^{\mathrm{log}\,z})w
\end{equation*}
for $a\in A$ and $w\in W$.
\end{rema}

For the right unit isomorphisms in $\repA$, we use Proposition \ref{prop:relntoR+R-}, the definition of $\mu_{W; -1}$, and the definition of $r^A_W$ to obtain the commutative diagram
\begin{equation*}
 \xymatrixcolsep{4pc}
 \xymatrix{
  & A\boxtimes_{P(-1)} W \ar[d]^{T_{-1\rightarrow 1}} \ar[rd]^{\mu_{W; -1}} & \\
  W\boxtimes A \ar[ru]^{\cR_{P(1)}^-} \ar[r]^{\sR_{A,W}^{-1}} \ar[rd]_{\eta_{W,A}} & A\boxtimes W \ar[r]^{\mu_W} & W\\
   & W\boxtimes_A A \ar[ru]_{r^A_W} & \\
 }
\end{equation*}
Then we use $\overline{\eta_{W,A}}\circ\boxtimes=\cY_{W,A}(\cdot,e^{\mathrm{log}\,1})\cdot$ and \eqref{Rminusdef} to calculate
\begin{align*}
 \overline{r^A_W}(\cY_{W,A}(w, e^{\mathrm{log}\,1})a) & =\overline{\mu_{W; -1}\circ \cR_{P(1)}^-}(w\boxtimes a)\nonumber\\
 & =\overline{\mu_{W; -1}}\left((-1)^{\vert w\vert\vert a\vert} e^{L(-1)}\cY_{\boxtimes_{P(-1)}, 0}(a, e^{-\pi i})w\right)\nonumber\\
 & =(-1)^{\vert w\vert\vert a\vert}\overline{\mu_{W; -1}}\left( e^{L(-1)} e^{-2\pi i L(0)}\cY_{\boxtimes_{P(-1)}, 0}(e^{2\pi i L(0)} a, e^{\pi i})e^{2\pi i L(0)} w\right)\nonumber\\
  & = (-1)^{\vert w\vert\vert a\vert}e^{L(-1)} e^{-2\pi i L(0)}\overline{\mu_{W; -1}}\left( e^{2\pi i L(0)} a\boxtimes_{P(-1)} e^{2\pi i L(0)} w\right)\nonumber\\
  & = (-1)^{\vert w\vert\vert a\vert}e^{L(-1)} e^{-2\pi i L(0)} Y_W(e^{2\pi i L(0)} a, e^{\pi i}) e^{2\pi i L(0)} w\nonumber\\
  & =(-1)^{\vert w\vert\vert a\vert}e^{L(-1)} Y_W(a, e^{-\pi i})w\nonumber\\
  & = \Omega_{-1}(Y_W)(w, e^{\mathrm{log}\,1})a
\end{align*}
for parity-homogeneous $w\in W$, $a\in A$. Thus by Proposition \ref{opmapiso}, we have
\begin{equation*}
 r^A_W(\cY_{W,A}(w,x)a)=\Omega_{-1}(Y_W)(w, x)a=(-1)^{\vert w\vert\vert a\vert}e^{xL(-1)} Y_W(a, e^{-\pi i} x)w
\end{equation*}
for parity-homogeneous $w\in W$, $a\in A$.
\begin{rema}
For $z\in\CC^\times$, if we use the notation $r^A_{P(z)}$ to denote $r^A_W$ when we view $r^A_W$ as a map from the $P(z)$-tensor product, we have
\begin{equation*}
 \overline{r^A_{P(z)}}(w\boxtimes^A_{P(z)} a)=(-1)^{\vert w\vert\vert a\vert} e^{zL(-1)} Y_W(a, e^{\mathrm{log}\,z-\pi i})w= (-1)^{\vert w\vert\vert a\vert}e^{zL(-1)} Y_W(a, e^{l_{-1/2}(-z)})w
\end{equation*}
for parity-homogeneous $w\in W$, $a\in A$.
\end{rema}

\begin{rema}
 If $W$ is a module in $\repzA$, the unit isomorphisms $l^A_W$ and $r^A_W$ have the same description. However, since in this case $Y_W$ has only integral powers of the formal variable $x$, we can simplify the notation as follows:
 \begin{align*}
  l^A_W(\cY_{A,W}(a,x)w) & =Y_W(a,x)w\nonumber\\
  \overline{l^A_{P(z)}}(a\boxtimes^A_{P(z)} w) & = Y_W(a, z)w
 \end{align*}
 for $a\in A$, $w\in W$, $z\in\CC^\times$, and
 \begin{align*}
  r^A_W(\cY_{W,A}(w, x)a) & =(-1)^{\vert w\vert\vert a\vert}e^{xL(-1)} Y_W(a, -x)w\nonumber\\
  \overline{r^A_{P(z)}}(w\boxtimes^A_{P(z)} a) & =(-1)^{\vert w\vert\vert a\vert}e^{zL(-1)} Y_W(a, -z)w
 \end{align*}
for parity-homogeneous $a\in A$, $w\in W$, and $z\in\CC^\times$. This is in agreement with the construction of unit isomorphisms in \cite{HLZ8}.
\end{rema}

\subsubsection{Braiding isomorphisms in $\repzA$}

Here we discuss the braiding isomorphisms in $\repzA$ in terms of intertwining operators. For modules $W_1$ and $W_2$ in $\repzA$, it is easy to see from \eqref{rep0Abraidingdef}, \eqref{formalbraidchar}, and the definition of $\cY_{W_1,W_2}$ that
\begin{equation*}
 \cR_{W_1,W_2}^A(\cY_{W_1,W_2}(w_1,x)w_2)=\Omega_0(\cY_{W_2,W_1})(w_1,x)w_2=(-1)^{\vert w_1\vert\vert w_2\vert}e^{xL(-1)}\cY_{W_2,W_1}(w_2, e^{\pi i} x)w_1
\end{equation*}
and
\begin{equation*}
 (\cR_{W_1,W_2}^A)^{-1}(\cY_{W_2,W_1}(w_2,x)w_1)=\Omega_{-1}(\cY_{W_1,W_2})(w_2,x)w_1=(-1)^{\vert w_1\vert\vert w_2\vert}e^{xL(-1)}\cY_{W_1,W_2}(w_1, e^{-\pi i} x)w_2
\end{equation*}
for parity-homogeneous $w_1\in W_1$, $w_2\in W_2$. Moreover, using Remark \ref{vrtxmonodromy}, the monodromy isomorphisms in $\repzA$ are characterized by
\begin{equation*}
 \cM_{W_1,W_2}^A(\cY_{W_1,W_2}(w_1,x)w_2)=\cY_{W_1,W_2}(w_1, e^{2\pi i} x)w_2
\end{equation*}
for $w_1\in W_1$ and $w_2\in W_2$.

\begin{rema}\label{vrtxAmonodromy}
By \eqref{RepzAparatrans}, the parallel transport isomorphism $T^A_\gamma$ on $W_1\boxtimes_A W_2$ is $(\cM^A_{W_1,W_2})^p$ where $p\in\ZZ$ corresponds to the monodromy of the closed curve $\gamma$ in $\CC^\times$.
\end{rema}

When we view $W_1\boxtimes_A W_2$ as the $P(z)$-tensor product in the category of $A$-modules, using $\boxtimes^A_{P(z)}=\cY_{W_1,W_2}(\cdot,e^{\mathrm{log}\,z})\cdot$, we can use the notation $\cR_{P(z); W_1, W_2}^{A; +}=\cR_{P(z)}^{A; +}$ for $\cR^A_{W_1,W_2}$ and $\cR_{P(z); W_1, W_2}^{A; -}=\cR_{P(z)}^{A; -}$ for $(\cR^A_{W_2, W_1})^{-1}$. With this notation we have
\begin{align*}
 \cR_{P(z)}^{A; +}( w_1\boxtimes^A_{P(z)} w_2) & =(-1)^{\vert w_1\vert\vert w_2\vert}e^{z L(-1)}\cY_{W_2,W_1}(w_2, e^{\mathrm{log}\,z+\pi i})w_1\nonumber\\
 &=(-1)^{\vert w_1\vert\vert w_2\vert}e^{z L(-1)}\cY_{W_2,W_1}(w_2, e^{l_{1/2}(-z)})w_1
\end{align*}
and
\begin{align*}
 \cR_{P(z)}^{A; -}( w_1\boxtimes^A_{P(z)} w_2) & =(-1)^{\vert w_1\vert\vert w_2\vert}e^{z L(-1)}\cY_{W_2,W_1}(w_2, e^{\mathrm{log}\,z-\pi i})w_1\nonumber\\
 &=(-1)^{\vert w_1\vert\vert w_2\vert}e^{z L(-1)}\cY_{W_2,W_1}(w_2, e^{l_{-1/2}(-z)})w_1
\end{align*}
for parity-homogeneous $w_1\in W_1$, $w_2\in W_2$. Also, the results of Propositions \ref{R+R-} and \ref{prop:relntoR+R-} hold with morphisms in $\cC$ replaced by the corresponding morphisms in $\repzA$ (recall that the parallel transport isomorphisms appearing in these propositions are trivial in $\repzA$ due to our particular realization of $P(z)$-tensor products in $\repzA$). Moreover, $\cR^A_{W_1,W_2}=T^A_{-1\rightarrow 1}\circ \cR^{A; +}_{P(1)}$ since $T^A_{-1\to 1}$ is trivial.

These results show that the braiding isomorphisms in $\repzA$ agree with the vertex-algebraic braiding isomorphisms between $A$-modules constructed in \cite{HLZ8} and reviewed in Subsection \ref{subsec:VTC}.

\subsubsection{Associativity isomorphisms in $\repA$ and $\repzA$}

Finally, we describe the associativity isomorphisms in $\repA$ (and also $\repzA$) in terms of intertwining operators. First, we have:
\begin{propo}
 Choose $r_1, r_2\in\RR_+$ such that $r_1>r_2>r_1-r_2>0$. Then for modules $W_1$, $W_2$, and $W_3$ in $\repA$, the associativity isomorphism $\cA^A_{W_1,W_2,W_3}$ is characterized by the equality
 \begin{align}\label{repAassocchar}
  \overline{\cA^A_{W_1,W_2,W_3}} & \left(\cY_{W_1, W_2\boxtimes_A W_3}(w_1, e^{\mathrm{log}\,r_1})\cY_{W_2,W_3}(w_2, e^{\mathrm{log}\,r_2})w_3\right)\nonumber\\
  &=\cY_{W_1\boxtimes_A W_2, W_3}(\cY_{W_1,W_2}(w_1, e^{\mathrm{log}(r_1-r_2)})w_2, e^{\mathrm{log}\,r_2})w_3
 \end{align}
for all $w_1\in W_1$, $w_2\in W_2$, and $w_3\in W_3$.
\end{propo}
\begin{proof}
 The definitions of $\sA_{W_1,W_2,W_3}$ and $\cA^A_{W_1,W_2,W_3}$ imply that $\cA^A_{W_1,W_2,W_3}$ is characterized by the commutativity of the diagram
 \begin{equation*}
  \xymatrixcolsep{6pc}
  \xymatrix{
  W_1\boxtimes_{P(r_1)}(W_2\boxtimes_{P(r_2)} W_3) \ar[r]^{\cA_{r_1,r_2}} \ar[d]_{T_{r_1\rightarrow 1}\circ(1_{W_1}\boxtimes_{P(r_1)} T_{r_2\rightarrow 1})} & (W_1\boxtimes_{P(r_1-r_2)} W_2)\boxtimes_{P(r_2)} W_3 \ar[d]^{T_{r_2\rightarrow 1}\circ(T_{r_1-r_2\rightarrow 1}\boxtimes_{P(r_2)} 1_{W_3})} \\
  W_1\boxtimes(W_2\boxtimes W_3) \ar[r]^{\sA_{W_1,W_2,W_3}} \ar[d]_{\eta_{W_1, W_2\boxtimes_A W_3}\circ(1_{W_1}\boxtimes\eta_{W_2,W_3})} & (W_1\boxtimes W_2)\boxtimes W_3 \ar[d]^{\eta_{W_1\boxtimes_A W_2, W_3}\circ(\eta_{W_1,W_2}\boxtimes 1_{W_3})} \\
  W_1\boxtimes_A(W_2\boxtimes_A W_3) \ar[r]^{\cA^A_{W_1,W_2,W_3}} & (W_1\boxtimes_A W_2)\boxtimes_A W_3 \\
  }
 \end{equation*}
Now, the composition of the top and right arrows in this diagram is given by
\begin{align*}
 w_1 & \boxtimes_{P(r_1)} (w_2\boxtimes_{P(r_2)} w_3) \mapsto (w_1\boxtimes_{P(r_1-r_2)} w_2)\boxtimes_{P(r_2)} w_3\mapsto (\cY_{\boxtimes_{P(1)},0}(w_1, e^{\mathrm{log}\,(r_1-r_2)})w_2)\boxtimes_{P(r_2)} w_3\nonumber\\
 & \mapsto \cY_{\boxtimes_{P(1)},0}(\cY_{\boxtimes_{P(1)},0}(w_1, e^{\mathrm{log}\,(r_1-r_2)})w_2, e^{\mathrm{log}\,r_2})w_3\mapsto \cY_{\boxtimes_{P(1)},0}(\cY_{W_1,W_2}(w_1, e^{\mathrm{log}\,(r_1-r_2)})w_2, e^{\mathrm{log}\,r_2})w_3\nonumber\\
  & \mapsto\cY_{W_1\boxtimes_A W_2, W_3}(\cY_{W_1,W_2}(w_1, e^{\mathrm{log}\,(r_1-r_2)})w_2, e^{\mathrm{log}\,r_2})w_3
\end{align*}
for $w_1\in W_1$, $w_2\in W_2$, and $w_3\in W_3$. Similarly, the composition of the left and bottom arrows in the diagram is given by
\begin{equation*}
 w_1\boxtimes_{P(r_1)}(w_2\boxtimes_{P(r_2)} w_3)\mapsto \overline{\cA^A_{W_1,W_2,W_3}}\left(\cY_{W_1, W_2\boxtimes_A W_3}(w_1, e^{\mathrm{log}\,r_1})\cY_{W_2,W_3}(w_2, e^{\mathrm{log}\,r_2})w_3\right)
\end{equation*}
for $w_1\in W_1$, $w_2\in W_2$, and $w_3\in W_3$. Thus $\cA^A_{W_1,W_2,W_3}$ is indeed characterized by \eqref{repAassocchar}. 
\end{proof}

As an immediate consequence of Proposition \ref{extassoc}, we then have:
\begin{corol}
 For $(z_1,z_2)\in S_1$, the associativity isomorphism $\cA^A_{W_1,W_2,W_3}$ satisfies
 \begin{align*}
  \overline{\cA^A_{W_1,W_2,W_3}} & \left(\cY_{W_1,W_2\boxtimes_A W_3}(w_1, e^{\mathrm{log}\,z_1})\cY_{W_2,W_3}(w_2, e^{\mathrm{log}\,z_2})w_3\right)\nonumber\\
  & =\cY_{W_1\boxtimes_A W_2, W_3}(\cY_{W_1,W_2}(w_1, e^{\mathrm{log}(z_1-z_2)})w_2, e^{\mathrm{log}\,z_2})w_3
 \end{align*}
for $w_1\in W_1$, $w_2\in W_2$, and $w_3\in W_3$. In particular, for $w_1\in W_1$, $w_2\in W_2$, $w_3\in W_3$, and $w'\in((W_1\boxtimes_A W_2)\boxtimes_A W_3)'$, the multivalued functions
\begin{align*}
 \langle w', \overline{\cA^A_{W_1,W_2,W_3}}  \left(\cY_{W_1,W_2\boxtimes_A W_3}(w_1, z_1)\cY_{W_2,W_3}(w_2, z_2)w_3\right)\rangle
\end{align*}
on $\vert z_1\vert>\vert z_2\vert>0$ and 
\begin{align*}
 \langle w',  \cY_{W_1\boxtimes_A W_2, W_3}(\cY_{W_1,W_2}(w_1, z_1-z_2)w_2, z_2)w_3\rangle
\end{align*}
on $\vert z_2\vert>\vert z_1-z_2\vert>0$ have equal restrictions to the region $\vert z_1\vert>\vert z_2\vert>\vert z_1-z_2\vert>0$.
\end{corol}

To obtain full vertex tensor category structure, we need, as in $\cC$, isomorphisms
\begin{equation*}
 \cA^A_{z_1,z_2}: W_1\boxtimes^A_{P(z_1)}(W_2\boxtimes^A_{P(z_2)} W_3)\rightarrow (W_1\boxtimes^A_{P(z_1-z_2)} W_2)\boxtimes^A_{P(z_2)} W_3
\end{equation*}
for any $\vert z_1\vert>\vert z_2\vert>\vert z_1-z_2\vert>0$ which satisfy
\begin{equation}\label{repzAassocz_1z_2}
 \overline{\cA^A_{z_1,z_2}}\left(w_1\boxtimes^A_{P(z_1)} (w_2\boxtimes^A_{P(z_2)} w_3)\right)=(w_1\boxtimes^A_{P(z_1-z_2)} w_2)\boxtimes^A_{P(z_2)} w_3
\end{equation}
for $w_1\in W_1$, $w_2\in W_2$, and $w_3\in W_3$. The preceding corollary shows that when $(z_1,z_2)\in S_1$, we can take $\cA^A_{z_1,z_2}=\cA^A_{W_1,W_2,W_3}$, but this will not work in general. Instead, assuming $W_1$, $W_2$, and $W_3$ are modules in $\repzA$, we can take
\begin{equation}\label{repzaAz_1z_2}
 \cA^A_{z_1,z_2}=(T^A_{\widetilde{\gamma}}\boxtimes^A_{P(z_2)} 1_{W_3})\circ\cA^A_{W_1,W_2,W_3}\circ (T^A_{\gamma})^{-1}
\end{equation}
where $\gamma$ and $\widetilde{\gamma}$ are paths from $z_1$ to itself and $z_1-z_2$ to itself, respectively, as in \eqref{Az_1z_2general}. Then reversing the steps used to derive \eqref{Az_1z_2general} shows that \eqref{repzAassocz_1z_2} holds. Thus we can conclude the following theorem:
\begin{theo}\label{thm:associnrepzA}
 For modules $W_1$, $W_2$, and $W_3$ in $\repzA$ and $z_1,z_2\in\CC^\times$ satisfying $\vert z_1\vert>\vert z_2\vert>\vert z_1-z_2\vert>0$, there are natural isomorphisms
 \begin{equation*}
 \cA^A_{z_1,z_2}: W_1\boxtimes^A_{P(z_1)}(W_2\boxtimes^A_{P(z_2)} W_3)\rightarrow (W_1\boxtimes^A_{P(z_1-z_2)} W_2)\boxtimes^A_{P(z_2)} W_3
\end{equation*}
such that
\begin{equation*}
 \overline{\cA^A_{z_1,z_2}}\left(w_1\boxtimes^A_{P(z_1)} (w_2\boxtimes^A_{P(z_2)} w_3)\right)=(w_1\boxtimes^A_{P(z_1-z_2)} w_2)\boxtimes^A_{P(z_2)} w_3
\end{equation*}
for $w_1\in W_1$, $w_2\in W_2$, and $w_3\in W_3$. Moreover, for $(z_1,z_2)\in S_1$,
\begin{equation*}
 \cA^A_{W_1,W_2,W_3}=T^A_{z_2\rightarrow 1}\circ(T_{z_1-z_2\rightarrow 1}\boxtimes^A_{P(z_2)} 1_{W_3})\circ\cA^A_{z_1,z_2}\circ(1_{W_1}\boxtimes^A_{P(z_1)} T^A_{1\rightarrow z_2})\circ T^A_{1\rightarrow z_1};
\end{equation*}
the same relation holds with $(z_1,z_2)\in S_1$ replaced by $(r_1,r_2)\in\RR^2$ such that $r_1>r_2>r_1-r_2>0$.
\end{theo}

By Theorem \ref{thm:associnrepzA}, the associativity isomorphisms $\cA^A$ in $\repzA$ agree with the associativity isomorphisms obtained in \cite{HLZ8} from vertex tensor category structure on the category of $A$-modules. Combining this fact with all the other results and discussion in this subsection, we have the following fundamental theorem on $\repzA$ and its braided monoidal category structure:
\begin{theo}\label{thm:repzABTCvrtxSalg}
 Suppose $A$ is a vertex operator superalgebra extension of a vertex operator algebra $V$ such that $V\subseteq A^\even$, and suppose $\cC$ is a category of grading-restricted generalized $V$-modules that contains $A$ and satisfies conditions necessary so that $\cC$ has the vertex tensor category structure (and thus braided tensor category structure) given in \cite{HLZ1}-\cite{HLZ8}. Then:
  \begin{enumerate}
  \item $\repzA$ is the category of grading-restricted generalized $A$-modules in $\cC$  (\cite{HKL}, \cite{CKL}).
  \item $\repzA$ has vertex tensor category structure (and thus braided monoidal supercategory structure) as given in \cite{HLZ1}-\cite{HLZ8}. 
  \item The Huang-Lepowsky-Zhang braided monoidal supercategory structure on $\repzA$ given in \cite{HLZ8} and Section \ref{subsec:VTC} is isomorphic to the Kirillov-Ostrik braided monoidal supercategory structure on $\repzA$ given in \cite{KO} and Section \ref{sec:TensCats}.
 \end{enumerate}
\end{theo}

\begin{rema}\label{rema:strongHKL}
 As mentioned above, the isomorphism between $\repzA$ and the category of grading-restricted generalized $A$-modules in $\cC$ was obtained in Theorem 3.4 of \cite{HKL}, and in Theorem 3.14 of \cite{CKL} for the superalgebra generality. In Theorem \ref{thm:repzABTCvrtxSalg}, we have considerably strengthened these results by showing that the isomorphism is actually an isomorphism of braided monoidal supercategories.
\end{rema}

\begin{rema}
 An entirely analogous theorem to Theorem \ref{thm:repzABTCvrtxSalg} shows that the braided tensor category structure on the subcategory $\underline{\repzA}$ of $\repzA$ consisting only of even morphisms is equivalent to the braided tensor category structure induced from vertex tensor category structure on the category of $A$-modules in $\cC$ with only even morphisms. The proof of such a theorem is exactly the same as the proof of Theorem \ref{thm:repzABTCvrtxSalg} (except that all sign factors involving parity of morphisms may be deleted) since all structure morphisms discussed in the previous subsections are even. 
%
\end{rema}

\subsection{Induction as a vertex tensor functor}\label{subsec:vrtxtensfunctor}

In Theorems \ref{thm:inductionfucntor} and \ref{thm:Fisbraidedtensor}, we showed that induction is a tensor functor on $\sC$ and a braided tensor functor on the subcategory $\sC^0$ of objects that induce to $\repzA$. It is easy to see that $\sC^0$ is a vertex tensor category, since it inherits parallel transport, unit, associativity, and braiding isomorphisms from $\sC$; moreover, $\sC^0$ is closed under the $P(1)$-tensor product (and thus all $P(z)$-tensor products) because induction is a tensor functor. In this subsection, we show that induction $\cF: \sC^0\rightarrow\repzA$ is compatible with all the vertex tensor category structures on $\sC^0$ and $\repzA$. That is, induction is a \textit{vertex tensor functor} in the following sense:
\begin{enumerate}
	\item There is an even isomorphism $\varphi: \cF(V)\rightarrow A$.
	
	\item For every $z\in\CC^\times$, there is an even natural isomorphism $f_{P(z)}: \cF\circ\boxtimes_{P(z)}\rightarrow\boxtimes^A_{P(z)}\circ(\cF\times\cF)$.
	
	\item The natural isomorphisms $f_{P(z)}$ are compatible with parallel transport isomorphisms in $\sC^0$ and $\repzA$ in the sense that the diagram
	\begin{equation*}
	\xymatrixcolsep{5pc}
	\xymatrix{
	 \cF(W_1\boxtimes_{P(z_1)} W_2) \ar[r]^{\cF(T_{\gamma; W_1, W_2})} \ar[d]_{f_{P(z_1); W_1, W_2}} & \cF(W_1\boxtimes_{P(z_2)} W_2) \ar[d]^{f_{P(z_2); W_1,W_2}} \\
	 \cF(W_1)\boxtimes^A_{P(z_1)} \cF(W_2) \ar[r]^{T^A_{\gamma; \cF(W_1), \cF(W_2)}} & \cF(W_1)\boxtimes^A_{P(z_2)} \cF(W_2) \\
	 }
	\end{equation*}
commutes for any modules $W_1$, $W_2$ in $\sC^0$ and any continuous path $\gamma$ from $z_1$ to $z_2$ in $\CC^\times$.
	
	\item The natural isomorphisms $f_{P(z)}$ and the isomorphism $\varphi$ are compatible with the unit isomorphisms in $\sC^0$ and $\repzA$ in the sense that the diagrams
	\begin{align*}
	&\xymatrixcolsep{4pc}
	\xymatrix{
		\cF(V\boxtimes_{P(z)} W) \ar[r]^{f_{P(z); V, W}} \ar[d]_{\cF(l_{P(z); W})} & 
		\cF(V)\boxtimes^A_{P(z)} \cF(W) \ar[d]^{\varphi\boxtimes^A_{P(z)} 1_{\cF(W)}} \\
		\cF(W) \ar[r]^{(l^A_{P(z); \cF(W)})^{-1}} & A\boxtimes^A_{P(z)} \cF(W)\\
	}
	\end{align*}
	and
	\begin{align*}
	&\xymatrixcolsep{4pc}
	\xymatrix{
		\cF(W\boxtimes_{P(z)} V) \ar[r]^{f_{P(z); W, V}} \ar[d]_{\cF(r_{P(z); W})} & 
		\cF(W)\boxtimes^A_{P(z)} \cF(V) \ar[d]^{1_{\cF(W)}\boxtimes^A_{P(z)} \varphi} \\
		\cF(W) \ar[r]^{(r^A_{P(z); \cF(W)})^{-1}} & \cF(W)\boxtimes^A_{P(z)} A\\
	}\nonumber
	\end{align*}
	commute for any object $W$ in $\sC^0$, $z\in\CC^\times$.
	
	\item The natural isomorphisms $f_{P(z)}$ are compatible with the associativity isomorphisms in $\sC^0$ and $\repzA$ in the sense that the diagram
	\begin{equation}\label{zFassoc}
	\xymatrixcolsep{6pc}
	\xymatrix{
		\cF(W_1\boxtimes_{P(z_1)}(W_2\boxtimes_{P(z_2)} W_3)) 
		\ar[r]^{\cF(\cA_{z_1,z_2})} \ar[d]^{ f_{P(z_1); W_1,W_2\boxtimes_{P(z_2)} W_3}} & 
		\cF((W_1\boxtimes_{P(z_1-z_2)} W_2)\boxtimes_{P(z_2)} W_3) \ar[d]_{f_{P(z_2); W_1\boxtimes_{P(z_1-z_2)} W_2,W_3}}\\
		\cF(W_1)\boxtimes^A_{P(z_1)} \cF(W_2\boxtimes_{P(z_2)} W_3) 
		\ar[d]^{1_{\cF(W_1)}\boxtimes^A_{P(z_1)} f_{P(z_2); W_2,W_3}} & 
		\cF(W_1\boxtimes_{P(z_1-z_2)} W_2)\boxtimes^A_{P(z_2)} \cF(W_3) \ar[d]_{f_{P(z_1-z_2); W_1,W_2}\boxtimes^A_{P(z_2)} 1_{\cF(W_3)}}\\
		\cF(W_1)\boxtimes^A_{P(z_1)} (\cF(W_2)\boxtimes^A_{P(z_2)} \cF(W_3)) \ar[r]^-{\mathcal{A}^A_{z_1,z_2}} & 
		(\cF(W_1)\boxtimes^A_{P(z_1-z_2)} \cF(W_2))\boxtimes^A_{P(z_2)} \cF(W_3)\\
	}
	\end{equation}
	commutes for all modules $W_1$, $W_2$, $W_3$ in $\sC^0$ and $z_1, z_2\in\CC^\times$ such that $\vert z_1\vert>\vert z_2\vert>\vert z_1-z_2\vert>0$.
	\item 
	The natural isomorphisms $f_{P(z)}$ are compatible with the braidings in $\sC^0$ and $\repzA$ in the sense that the diagram
	\begin{equation}\label{zBr}
	\xymatrixcolsep{4pc}
	\xymatrix{
	\cF(W_1\boxtimes_{P(\pm z)}W_2)\ar[r]^-{\cF(\cR^{\pm}_{P(\pm z)})}
	\ar[d]^-{f_{P(\pm z); W_1,W_2}} &
	\cF(W_2\boxtimes_{P(\mp z)}W_1)
	\ar[d]^-{f_{P(\mp z); W_2,W_1}}\\
	\cF(W_1)\boxtimes^A_{P(\pm z)}\cF(W_2)\ar[r]^-{\cR^{A;\pm}_{P(\pm z)}}&
	\cF(W_2)\boxtimes^A_{P(\mp z)}\cF(W_1)
	}
	\end{equation}
	commutes for all modules $W_1,W_2$ in $\sC^0$ and $z\in\CC^\times$.
\end{enumerate}

\begin{theo}\label{thm:Fisvertextensor}
The induction functor $\cF:\sC^0\rightarrow\repzA$  is a vertex tensor functor with respect to the vertex tensor category structure on $\sC^0$ inherited from $\cC$ and the vertex tensor category structure on $\repzA$ given in Theorem \ref{thm:repzABTCvrtxSalg}.
\end{theo}
\begin{proof}
Throughout this proof, $W$ and $W_i$ for $i=1,2,3$ will denote objects of $\sC^0$. We will derive the required properties using $\varphi=r_A: \cF(V)\rightarrow A$ and
\begin{align*}
f_{P(z); W_1, W_2} = T^A_{1\to z}\circ f_{W_1,W_2}\circ\cF(T_{z\to 1}): \cF(W_1\boxtimes_{P(z)} W_2)\rightarrow\cF(W_1)\boxtimes^A_{P(z)} \cF(W_2)
\end{align*}
for $z\in\CC^\times$ and modules $W_1$, $W_2$ in $\sC^0$,
using properties of parallel transports and the fact that $\cF$ is a braided monoidal functor. First note that for each $z\in\CC^\times$, $f_{P(z)}$ is an even natural isomorphism since it is a composition of even natural isomorphisms.

To show that the $f_{P(z)}$ are compatible with parallel transports, suppose $\gamma$ is a continuous path from $z_1$ to $z_2$ in $\CC^\times$ and let $\widetilde{\gamma}$ denote the path from $1$ to itself obtained by first following a path from $1$ to $z_1$ in $\CC^\times$ with a cut along the positive real axis, then following $\gamma$ from $z_1$ to $z_2$, and then following a path from $z_2$ to $1$ in $\CC^\times$ with a cut along the positive real axis. Thus $T^{(A)}_{\widetilde{\gamma}}=T^{(A)}_{z_2\to 1}\circ T^{(A)}_{\gamma}\circ T^{(A)}_{1\to z_1}$. By Proposition \ref{prop:monoandtrans} and Remark \ref{vrtxAmonodromy}, $T_{\widetilde{\gamma}; W_1, W_2}=\cM_{W_1,W_2}^p$ and $T^A_{\widetilde{\gamma}; \cF(W_1), \cF(W_2)}=(\cM^A_{\cF(W_1),\cF(W_2)})^p$, where $p\in\ZZ$ gives the monodromy of $\widetilde{\gamma}$. Using these facts together with
Corollary \ref{Fmonodromy}, we calculate
\begin{align*}
 f_{P(z_2); W_1, W_2}\circ  \cF(T_{\gamma; W_1, W_2}) & = T^A_{1\to z_2}\circ f_{W_1,W_2}\circ\cF(T_{z_2\to 1})\circ\cF(T_{\gamma})\nonumber\\
 &= T^A_{1\to z_2}\circ f_{W_1,W_2}\circ\cF(T_{\widetilde{\gamma}})\circ\cF(T_{z_1\to 1})\nonumber\\
 & = T^A_{1\to z_2}\circ f_{W_1,W_2}\circ\cF(\cM_{W_1,W_2}^p)\circ\cF(T_{z_2\to 1})\nonumber\\
 &= T^A_{1\to z_2}\circ(\cM^A_{\cF(W_1),\cF(W_2)})^p\circ f_{W_1,W_2}\circ\cF(T_{z_1\to 1})\nonumber\\
 & =T^A_{1\to z_2} \circ T^A_{\widetilde{\gamma}}\circ f_{W_1,W_2}\circ\cF(T_{z_1\to 1})\nonumber\\
 & = T^A_\gamma\circ T^A_{1\to z_1}\circ f_{W_1,W_2}\circ\cF(T_{z_1\to 1} = T^A_{\gamma; \cF(W_1),\cF(W_2)}\circ f_{P(z_1); W_1, W_2},
\end{align*}
as desired.

We prove that $f_{P(z)}$ and $\varphi$ are compatible with the left unit isomorphisms.
Every morphism in the following diagram is invertible.
The outer rectangle commutes since $\cF$ is a monoidal functor,
the top by definition of $f_{P(z)}$, the left by properties of unit morphisms in $\sC$, 
the bottom by properties of unit morphisms in $\repA$, and the right by
naturality of parallel transports. Hence, the square in the middle commutes.
\begin{align*}
\xymatrixcolsep{2.5pc}
\xymatrix@R-0.5pc{
\cF(V\boxtimes W) \ar[rrr]^{f_{V, W}} \ar[ddd]_-{\cF(\sleft_{W})} 
\ar[dr]^-{\cF(T_{1\rightarrow z})}
& &  &  
\cF(V)\boxtimes_A \cF(W) \ar[ddd]^{\varphi\boxtimes_A 1_{\cF(W)}}
\ar[dl]_-{T^A_{1\rightarrow z}} \\
& \cF(V\boxtimes_{P(z)} W) \ar[r]^-{f_{P(z); V, W}} \ar[d]_{\cF(l_{P(z); W})} & 
\cF(V)\boxtimes^A_{P(z)} \cF(W) \ar[d]^{\varphi\boxtimes^A_{P(z)} 1_{\cF(W)}} & \\
&	\cF(W) \ar[r]^{(l^A_{P(z); \cF(W)})^{-1}} & A\boxtimes^A_{P(z)} \cF(W) & \\
\cF(W) \ar[ru]^{1_{\cF(W)}=\cF(1_W)} \ar[rrr]^-{(l^A_{\cF(W)})^{-1}} & & & A\boxtimes_A \cF(W)
\ar[lu]^-{T^A_{1\rightarrow z}}.}
\end{align*}
The proof of the right unit property is exactly similar.

To prove that \eqref{zFassoc} commutes, first recall from \eqref{Az_1z_2general} and \eqref{repzaAz_1z_2} that in both $\sC^0$ and $\repzA$ (suppressing module subscripts from isomorphisms),
\begin{align*}
 ((T^{(A)}_{\widetilde{\gamma}})^{-1} & \boxtimes^{(A)}_{P(z_2)} 1)\circ  \cA^{(A)}_{z_1,z_2}\circ T^{(A)}_\gamma\nonumber\\
 &= (T^{(A)}_{1\to z_1-z_2}\boxtimes^{(A)}_{P(z_2)} 1)\circ T^{(A)}_{1\to z_2}\circ\cA^{(A)}\circ T^{(A)}_{z_1\to 1}\circ(1\boxtimes^{(A)}_{P(z_1)} T^{(A)}_{z_2\to 1}),
\end{align*}
where $\gamma$ and $\widetilde{\gamma}$ are certain paths from $z_1$ to itself and $z_1-z_2$ to itelf, respectively, that depend only on $(z_1, z_2)$. (Recall that with our realization of $P(z)$-tensor products in $\repzA$, parallel transports for paths where the branch of logarithm does not change are the identity.) Thus
\begin{align*}
 T^{(A)}_{z_2\to 1}\circ(T^{(A)}_{\widetilde{\gamma}'}\boxtimes^{(A)}_{P(z_2)} 1)\circ\cA^{(A)}_{z_1,z_2} = \cA^{(A)}\circ T^{(A)}_{\gamma'}\circ(1\boxtimes^{(A)}_{P(z_1)} T^{(A)}_{z_2\to 1}),
\end{align*}
where $\gamma'$ is a path from $z_1$ to $1$ obtained by first following the reversal of $\gamma$ and then following a path from $z_1$ to $1$ in $\CC^\times$ with a cut along the positive real axis, and where $\widetilde{\gamma}'$ is a path from $z_1-z_2$ to $1$ obtained by first following the reversal of $\widetilde{\gamma}$ and then following a path from $z_1-z_2$ to $1$ in $\CC^\times$ with a cut along the positive real axis. Using this relation, together with the definition of the $f_{P(z)}$, the naturality of the $f_{P(z)}$, the compatibility of the $f_{P(z)}$ with parallel transports, and the naturality of parallel transports, we obtain the following two commutative diagrams. Here, to save space, we use the notation $\widetilde{\cdot}$ for $\cF(\cdot)$, we use $z_0=z_1-z_2$, and we suppress module subscripts on most isomorphisms:
\begin{equation*}
  \xymatrixcolsep{4pc}
 \xymatrix{
 \widetilde{W_1\fus{P(z_1)}(W_2\fus{P(z_2)}W_3)} \ar[r]^{\widetilde{1\fus{P(z_1)} T_{z_2\to 1}}} \ar[d]^{f_{P(z_1)}} & \widetilde{W_1\fus{P(z_1)}(W_2\fus{}W_3)} \ar[r]^{\widetilde{T}_{\gamma'}} \ar[d]^{f_{P(z_1)}} & \widetilde{W_1\fus{}(W_2\fus{}W_3)} \ar[d]^{f_{W_1,W_2\fus{}W_3}} \\
 \widetilde{W}_1\boxtimes^A_{P(z_1)}(\widetilde{W_2\fus{P(z_2)}W_3}) \ar[r]^{1\boxtimes^A_{P(z_1)}\widetilde{T}_{z_2\to 1}} \ar[d]^{1\boxtimes^A_{P(z_1)} f_{P(z_2)}} & \widetilde{W}_1\boxtimes^A_{P(z_1)}(\widetilde{W_2\fus{}W_3}) \ar[r]^{T^A_{\gamma'}} \ar[d]^{1\boxtimes^A_{P(z_1)} f_{W_1,W_2}} & \widetilde{W}_1\fus{A}(\widetilde{W_2\fus{}W_3}) \ar[d]^{1\fus{A} f_{W_1,W_2}} \\
 \widetilde{W}_1\boxtimes^A_{P(z_1)}(\widetilde{W}_2\boxtimes^A_{P(z_2)}\widetilde{W}_3) \ar[r]^{1\boxtimes^A_{P(z_1)} T^A_{z_2\to 1}} \ar[d]^{\cA^A_{z_1,z_2}} & \widetilde{W}_1\boxtimes^A_{P(z_1)}(\widetilde{W}_2\fus{A}\widetilde{W}_3) \ar[r]^{T^A_{\gamma'}} & \widetilde{W}_1\fus{A}(\widetilde{W}_2\fus{A}\widetilde{W}_3) \ar[d]^{\cA^A_{\widetilde{W}_1,\widetilde{W}_2,\widetilde{W}_3}} \\
 (\widetilde{W}_1\boxtimes^A_{P(z_0)}\widetilde{W}_2)\boxtimes^A_{P(z_2)}\widetilde{W}_3 \ar[r]^{T^A_{\widetilde{\gamma}'}\boxtimes^A_{P(z_2)} 1} & (\widetilde{W}_1\fus{A}\widetilde{W}_2)\boxtimes^A_{P(z_2)}\widetilde{W}_3 \ar[r]^{T^A_{z_2\to 1}} & (\widetilde{W}_1\fus{A}\widetilde{W}_2)\fus{A}\widetilde{W}_3 \\
 }
\end{equation*}
and
\begin{equation*}
 \xymatrixcolsep{4pc}
 \xymatrix{
 \widetilde{W_1\fus{P(z_1)}(W_2\fus{P(z_2)}W_3)} \ar[r]^{\widetilde{1\fus{P(z_1)} T_{z_2\to 1}}} \ar[d]^{\widetilde{\cA}_{z_1,z_2}} & \widetilde{W_1\fus{P(z_1)}(W_2\fus{}W_3)} \ar[r]^{\widetilde{T}_{\gamma'}} & \widetilde{W_1\fus{}(W_2\fus{}W_3)} \ar[d]^{\widetilde{\cA}_{W_1,W_2,W_3}} \\
 \widetilde{(W_1\fus{P(z_0)} W_2)\fus{P(z_2)} W_3} \ar[r]^{\widetilde{T_{\widetilde{\gamma}'}\fus{P(z_2)} 1}} \ar[d]^{f_{P(z_2)}} & \widetilde{(W_1\fus{}W_2)\fus{P(z_2)} W_3} \ar[r]^{\widetilde{T}_{z_2\to 1}} \ar[d]^{f_{P(z_2)}} & \widetilde{(W_1\fus{}W_2)\fus{}W_3} \ar[d]^{f_{W_1\fus{}W_2,W_3}} \\
 (\widetilde{W_1\fus{P(z_0)} W_2})\boxtimes^A_{P(z_2)} \widetilde{W}_3 \ar[r]^{\widetilde{T}_{\widetilde{\gamma}'}\boxtimes^A_{P(z_2)} 1} \ar[d]^{f_{P(z_0)}\boxtimes^A_{P(z_2)} 1} & (\widetilde{W_1\fus{}W_2})\boxtimes^A_{P(z_2)} \widetilde{W}_3 \ar[r]^{T^A_{z_2\to 1}} \ar[d]^{f_{W_1,W_2}\boxtimes^A_{P(z_2)} 1} & (\widetilde{W_1\boxtimes W_2})\fus{A}\widetilde{W}_3 \ar[d]^{f_{W_1,W_2}\fus{A} 1} \\
 (\widetilde{W}_1\boxtimes_{P(z_0)}^A \widetilde{W}_2)\boxtimes_{P(z_2)}^A \widetilde{W}_3 \ar[r]^{T^A_{\widetilde{\gamma}'}\boxtimes_{P(z_2)}^A 1} & (\widetilde{W}_1\fus{A}\widetilde{W}_2)\boxtimes_{P(z_2)}^A\widetilde{W}_3 \ar[r]^{T^A_{z_2\to 1}} & (\widetilde{W}_1\fus{A}\widetilde{W}_2)\fus{A}\widetilde{W}_3 \\
 }
\end{equation*}
The right columns of these two diagrams are equal because $\cF$ is a tensor functor, so the left columns are also equal, proving that the $f_{P(z)}$ are compatible with the associativity isomorphisms in $\sC^0$ and $\repzA$.

For proving that \eqref{zBr} commutes, we use the following diagram:
\begin{equation}\label{zBrproof}
\xymatrix{
\cF(W_1\boxtimes W_2) \ar[rrr]^-{\cF(\sR_{W_1,W_2})\,\,or\,\,\cF(\sR^{-1}_{W_2,W_1})}
\ar[ddd]^-{f_{W_1,W_2}}
\ar[dr]^-{\cF(T_{1\rightarrow \pm z})}
& & & 
\cF(W_2\boxtimes W_1)
\ar[ddd]^-{f_{W_2,W_1}}
\ar[dl]_-{\cF(T_{1\rightarrow \mp z})}
\\
& 	\cF(W_1\boxtimes_{P(\pm z)}W_2)\ar[r]^-{\cF(\cR^{\pm}_{P(\pm z)})}
	\ar[d]^-{f_{P(\pm z); W_1, W_2}} &
	\cF(W_2\boxtimes_{P(\mp z)}W_1)
	\ar[d]^-{f_{P(\mp z); W_2, W_1}} & \\
&	\cF(W_1)\boxtimes^A_{P(\pm z)}\cF(W_2)\ar[r]^-{\cR^{A;\pm}_{P(\pm z)}}&
	\cF(W_2)\boxtimes^A_{P(\mp z)}\cF(W_1) &\\
\cF(W_1)\boxtimes_A \cF(W_2) 
\ar[rrr]^-{\cR^{A}_{W_1,W_2}\,\,or\,\,(\cR^A_{W_2,W_1})^{-1}}
\ar[ur]^-{T^A_{1\rightarrow \pm z}}
& & & \cF(W_2)\boxtimes_A \cF(W_1)
\ar[ul]^-{T^A_{1\rightarrow \mp z}}	
}
\end{equation}
where every arrow is invertible, the outer square commutes by Theorem \ref{thm:Fisbraidedtensor},
the top and bottom squares commute by \cref{prop:relntoR+R-},
and the left and right squares commute by definition of $f_{P(z)}$.
\end{proof}

\subsection{Some analytic lemmas}
\label{subsec:technical}

Here we state and prove a proposition that was used in the proof of Theorem \ref{thm:repAtocomplex}. Also, for use in the next section, we generalize to $\repA$ some results on ideals and submodules which are well known for ordinary $A$-modules, that is, objects of $\repzA$ (see for instance \cite{LL}).

For now, the setting is the same as in Section \ref{subsec:VTC}; in particular, $\cC$ is a vertex tensor category of modules for a vertex operator superalgebra. The following result is similar to \cite[Proposition 12.13]{HLZ8}:
\begin{propo}\label{braidinganalogue}
	Let $W_1$, $W_2$, and $W_3$ be modules in $\mathcal{C}$ and suppose $r_1, r_2\in\RR$ satisfy $r_2>r_1>2(r_2-r_1)>0$. Then
	\begin{align*}
	&\overline{(\cR_{P(r_2-r_1)}^-\boxtimes_{P(r_2)} 1_{W_3})\circ T_{r_1\rightarrow r_2}}\left((w_1\boxtimes_{P(r_2-r_1)} w_2)\boxtimes_{P(r_1)} w_3\right)\nonumber\\
	&\hspace{10em}=(-1)^{\vert w_1\vert\vert w_2\vert}\left(\mathcal{Y}_{\boxtimes_{P(r_1-r_2)},0}(w_2, e^{l_{-1/2}(r_1-r_2)})w_1\right)\boxtimes_{P(r_2)}w_3
	\end{align*}
	for all $w_3\in W_3$ and parity-homogeneous $w_1\in W_1$, $w_2\in W_2$.
\end{propo}
\begin{rema}
	In the statement of Proposition \ref{braidinganalogue}, we could replace $(r_1, r_2)$ with $(z_1, z_2)$ coming from a larger set (such as an appropriate connected set related to $S_2$). But here we choose to restrict $r_1, r_2\in\RR$  for simplicity and with an eye towards applying Proposition \ref{extskewassoc}.
\end{rema}

\begin{proof}
	By definition (see \eqref{zonetoztwochar} and \eqref{Rminusdef}) we have
	\begin{align*}
	& \langle w',\overline{(\cR_{P(r_2-r_1)}^-\boxtimes_{P(r_2)} 1_{W_3})\circ T_{r_1\rightarrow r_2}}((w_1\boxtimes_{P(r_2-r_1)} w_2)\boxtimes_{P(r_1)} w_3)\rangle\nonumber\\
	& =\sum_{n\in\RR}\left\langle w', \overline{(\cR_{P(r_2-r_1)}^-\boxtimes_{P(r_2)} 1_{W_3})\circ T_{r_1\rightarrow r_2}}\left(\pi_n(w_1\boxtimes_{P(r_2-r_1)} w_2)\boxtimes_{P(r_1)} w_3\right)\right\rangle\nonumber\\
	& =\sum_{n\in\RR} \left\langle w', \mathcal{Y}_{\boxtimes_{P(r_2)},0}\left(\pi_n((-1)^{\vert w_1\vert\vert w_2\vert}e^{(r_2-r_1)L(-1)}\mathcal{Y}_{\boxtimes_{P(r_1-r_2)},0}(w_2,e^{l_{-1/2}(r_1-r_2)})w_1), e^{\mathrm{log}\,r_1}\right)w_3\right\rangle\nonumber\\
	& =(-1)^{\vert w_1\vert\vert w_2\vert}\sum_{n\in\RR}\sum_{j\geq 0}\dfrac{(r_2-r_1)^j}{j!}\cdot\nonumber\\
	&\hspace{5em}\cdot\left\langle w', \mathcal{Y}_{\boxtimes_{P(r_2)},0}\left( L(-1)^j\pi_{n-j}(\mathcal{Y}_{\boxtimes_{P(r_1-r_2)},0}(w_2, e^{l_{-1/2}(r_1-r_2)})w_1), e^{\mathrm{log}\,r_1}\right)w_3\right\rangle
	\end{align*}
	for $w'\in((W_1\boxtimes_{P(r_1-r_2)} W_2)\boxtimes_{P(r_2)} W_3)'$.
	Now we would like to use the $L(-1)$-derivative property and the Taylor theorem, but we cannot do this directly because we would need to rearrange the iterated sum over $n$ and $j$ by replacing $\pi_{n-j}$ with $\pi_n$, but we cannot assume the double sum over $n$ and $j$ converges absolutely. To deal with this problem, we consider the open set $U\subseteq\CC$ consisting of $\varepsilon\in\CC$ such that 
	\begin{equation*}
	r_1-2(r_2-r_1)>\vert\varepsilon\vert,\,\,\vert r_2-r_1+\varepsilon\vert>r_2-r_1,\,\,\mathrm{and}\,\,\vert r_2+\varepsilon\vert>r_1.
	\end{equation*}
	From the assumptions on $r_1$ and $r_2$, $U$ is non-empty and $\varepsilon=0$ is on the boundary of $U$. For notational convenience, we also set $\mathcal{Y}_{\boxtimes_{P(r_2)},0}=\mathcal{Y}^1$, $\mathcal{Y}_{\boxtimes_{P(r_1-r_2)},0}=\mathcal{Y}^2$, and $r_1-r_2=r_0$. 
	
	Now from Proposition 12.7, part 4 in \cite{HLZ8}, the double sum
	\begin{align*}
	\sum_{m,n\in\RR} & \left\langle w', \mathcal{Y}^1(\pi_m(\Omega(Y_{W_2\boxtimes_{P(r_0)} W_1})(\pi_n(\mathcal{Y}^2(w_2, e^{l_{-1/2}(r_0)})w_1), -r_0+\varepsilon)\mathbf{1}), e^{\mathrm{log}\,r_1})w_3\right\rangle\nonumber\\
	& =\sum_{m,n\in\RR}\left\langle w',\mathcal{Y}^1(\pi_m(e^{(-r_0+\varepsilon)L(-1)}\pi_n(\mathcal{Y}^2(w_2, e^{l_{-1/2}(r_0)})w_1)), e^{\mathrm{log}\,r_1})w_3\right\rangle
	\end{align*}
	for $w'\in((W_1\boxtimes_{P(r_1-r_2)} W_2)\boxtimes_{P(r_2)} W_3)'$ is absolutely convergent when 
	\begin{equation*}
	\vert -r_0+\varepsilon\vert>\vert r_0\vert >0\,\,\mathrm{and}\,r_1>\vert r_0\vert+\vert -r_0+\varepsilon\vert>0.
	\end{equation*}
	In particular, the double sum is absolutely convergent when $\varepsilon\in U$, so from now on we assume $\varepsilon\in U$. Since the double sum converges absolutely, we can write it as an iterated sum as follows:
	\begin{align*}
	\sum_{m,n\in\RR} & \left\langle w',\mathcal{Y}^1(\pi_m(e^{(-r_0+\varepsilon)L(-1)}\pi_n(\mathcal{Y}^2(w_2, e^{l_{-1/2}(r_0)})w_1)), e^{\mathrm{log}\,r_1})w_3\right\rangle\nonumber\\
	& =\sum_{m\in\RR}\sum_{n\in\RR} \left\langle w',\mathcal{Y}^1(\pi_m(e^{(-r_0+\varepsilon)L(-1)}\pi_n(\mathcal{Y}^2(w_2, e^{l_{-1/2}(r_0)})w_1)), e^{\mathrm{log}\,r_1})w_3\right\rangle\nonumber\\
	& =\sum_{m\in\RR}\sum_{j\geq 0}\dfrac{(-r_0+\varepsilon)^j}{j!}\left\langle w',\mathcal{Y}^1(L(-1)^j\pi_{m-j}(\mathcal{Y}^2(w_2, e^{l_{-1/2}(r_0)})w_1), e^{\mathrm{log}\,r_1})w_3\right\rangle\nonumber\\
	& =\sum_{m\in\RR}\sum_{j\geq 0}\dfrac{(-r_0+\varepsilon)^j}{j!}\left\langle w',\mathcal{Y}^1(\pi_{m}(L(-1)^j\mathcal{Y}^2(w_2, e^{l_{-1/2}(r_0)})w_1), e^{\mathrm{log}\,r_1})w_3\right\rangle\nonumber\\
	& =\sum_{m\in\RR} \left\langle w',\mathcal{Y}^1(\pi_{m}(e^{(-r_0+\varepsilon)L(-1)}\mathcal{Y}^2(w_2, e^{l_{-1/2}(r_0)})w_1), e^{\mathrm{log}\,r_1})w_3\right\rangle.
	\end{align*}
	
	On the other hand, because of the absolute convergence of the double sum, we can reindex the third sum above and use the $L(-1)$-derivative property for intertwining operators to obtain
	\begin{align*}
	\sum_{m\in\RR}\sum_{j\geq 0} & \dfrac{(-r_0+\varepsilon)^j}{j!}\left\langle w',\mathcal{Y}^1(L(-1)^j\pi_{m-j}(\mathcal{Y}^2(w_2, e^{l_{-1/2}(r_0)})w_1), e^{\mathrm{log}\,r_1})w_3\right\rangle\nonumber\\
	& =\sum_{m\in\RR}\sum_{j\geq 0}\left.\dfrac{(-r_0+\varepsilon)^j}{j!}\left(\dfrac{d}{d x}\right)^j\left\langle w',\mathcal{Y}^1(\pi_{m}(\mathcal{Y}^2(w_2, e^{l_{-1/2}(r_0)})w_1), x)w_3\right\rangle\right\vert_{x=e^{\mathrm{log}\,r_1}}\nonumber\\
	& =\sum_{m\in\RR}\left.\left\langle w',\mathcal{Y}^1(\pi_{m}(\mathcal{Y}^2(w_2, e^{l_{-1/2}(r_0)})w_1), x-r_0+\varepsilon)w_3\right\rangle\right\vert_{x=e^{\mathrm{log}\,r_1}}.
	\end{align*}
	Since $r_1-r_0+\varepsilon=r_2+\varepsilon$, we need to determine the correct branch of $\mathrm{log}(r_2+\varepsilon)$ to use when we substitute $x=e^{\mathrm{log}\,r_1}$. Now,
	\begin{equation*}
	\mathrm{log}(x-r_0+\varepsilon)\vert_{x=e^{\mathrm{log}\,r_1}}=\left.\left(\mathrm{log}\,x+\mathrm{log}\left(1+\dfrac{-r_0+\varepsilon}{x}\right)\right)\right\vert_{x=e^{\mathrm{log}\,r_1}}=\mathrm{log}\,r_1+\mathrm{log}\left(1+\dfrac{-r_0+\varepsilon}{r_1}\right),
	\end{equation*}
	where $\mathrm{log}\,\left(1+\frac{-r_0+\varepsilon}{r_1}\right)$ represents the standard power series for $\mathrm{log}(1+x)$ evaluated at $x=\frac{-r_0+\varepsilon}{r_1}$. Since $\vert\varepsilon\vert<r_1-2 \vert r_0\vert$ for $\varepsilon\in U$, we have $\vert -r_0+\varepsilon\vert<\vert r_1\vert$, as is necessary for the power series to converge. Moreover, $1+(-r_0+\varepsilon)/r_1=(r_2+\varepsilon)/r_1$ is contained in the disk of radius $1$ centered at $1$. Since the power series $\mathrm{log}(1+x)$ converges to the branch of logarithm which equals $0$ at $1$ and is holomorphic on the disk of radius $1$ centered at $1$, we see that we must have
	\begin{equation*}
	\mathrm{log}\left(1+\dfrac{-r_0+\varepsilon}{r_1}\right)=l_{-1/2}\left(\dfrac{r_2+\varepsilon}{r_1}\right)=l_{-1/2}(r_2+\varepsilon)-\mathrm{log}\,r_1,
	\end{equation*}
	where the second equality holds because $r_1$ is a positive real number. Thus we conclude that $\mathrm{log}(x-r_0+\varepsilon)\vert_{x=e^{\mathrm{log}\,r_1}}=l_{-1/2}(r_2+\varepsilon)$ for $\varepsilon\in U$.

	To summarize, we have now shown that
	\begin{align}
	\sum_{m\in\RR} & \left\langle w',\mathcal{Y}^1(\pi_{m}(e^{(-r_0+\varepsilon)L(-1)}\mathcal{Y}^2(w_2, e^{l_{-1/2}(r_0)})w_1), e^{\mathrm{log}\,r_1})w_3\right\rangle\label{firstseries}\\
	& = \sum_{m\in\RR}\left\langle w',\mathcal{Y}^1(\pi_{m}(\mathcal{Y}^2(w_2, e^{l_{-1/2}(r_0)})w_1), e^{l_{-1/2}(r_2+\varepsilon)})w_3\right\rangle\label{secondseries}
	\end{align}
	for all $\varepsilon\in U$. Our goal is to show that they are also equal for $\varepsilon=0$, in which case \eqref{firstseries} is
	\begin{equation*}
	(-1)^{\vert w_1\vert\vert w_2\vert}\left\langle w', \overline{(\cR_{P(r_2-r_1)}^-\boxtimes_{P(r_2)} 1_{W_3})\circ T_{r_1\rightarrow r_2}}((w_1\boxtimes_{P(r_2-r_1)} w_2)\boxtimes_{P(r_1)} w_3)\right\rangle
	\end{equation*}
	and \eqref{secondseries} is
	\begin{align*}
	& \left\langle w', \mathcal{Y}_{\boxtimes_{P(r_2)},0}\left(\mathcal{Y}_{\boxtimes_{P(r_1-r_2)},0}(w_2, e^{l_{-1/2}(r_1-r_2)})w_1, e^{l_{-1/2}(r_2)}\right) w_3\right\rangle \nonumber\\
	& \hspace{6em}=\left\langle w', \mathcal{Y}_{\boxtimes_{P(r_2)},0}\left(\mathcal{Y}_{\boxtimes_{P(r_1-r_2)},0}(w_2, e^{l_{-1/2}(r_1-r_2)})w_1, e^{\mathrm{log}\,r_2}\right) w_3\right\rangle \nonumber\\
	&\hspace{6em} =\left\langle w', \left(\mathcal{Y}_{\boxtimes_{P(r_1-r_2)},0}(w_2, e^{l_{-1/2}(r_1-r_2)})w_1\right)\boxtimes_{P(r_2)} w_3\right\rangle.
	\end{align*}
	It is enough to show that both series converge to analytic functions of $\varepsilon$ in a neighborhood of $0$, since their equality on $U$ will then show they are equal everywhere on their common domain, including $\varepsilon=0$.
	
	First, \eqref{secondseries}
	converges absolutely to an analytic function in a neighborhood of $\varepsilon=0$ since the iterate of intertwining operators converges to an analytic function of $\varepsilon$ in the region given by the conditions
	\begin{equation*}
	\vert r_2+\varepsilon\vert>\vert r_0\vert\,\,\,\mathrm{and}\,\,\,\mathrm{arg}(r_2+\varepsilon)\neq\pi
	\end{equation*}
	(\cite[Proposition 7.14]{HLZ5}). This region includes $\varepsilon=0$ because $r_2>\vert r_0\vert$. To analyze \eqref{firstseries},
	first note that $l_{-1/2}(r_0)=\mathrm{log}(-r_0)-\pi i$, so that
	\begin{equation*}
	e^{(-r_0+\varepsilon)L(-1)}\mathcal{Y}^2(w_2, e^{l_{-1/2}(r_0)})w_1=(-1)^{\vert w_1\vert\vert w_2\vert} e^{\varepsilon L(-1)}\Omega_{-1}(\mathcal{Y}^2)(w_1, e^{\mathrm{log}(-r_0)})w_2.
	\end{equation*}
	Now we need the following lemma:
	\begin{lemma}
		Suppose $\mathcal{Y}$ is an intertwining operator of type $\binom{W}{W_1\,W_2}$ where $W$ is a $V$-module. For all $m\in\RR$  and for $r$ a positive real number, there is a neighborhood $\widetilde{U}$ of $0$ in $\CC$ such that for $\varepsilon\in\widetilde{U}$,
		\begin{equation*}
		\pi_m(e^{\varepsilon L(-1)}\mathcal{Y}(w_1, e^{\mathrm{log}\,r})w_2)=\sum_{j\geq 0} \dfrac{\varepsilon^j}{j!} \pi_m(\mathcal{Y}(w_1, e^{l_{-1/2}(r+\varepsilon)}) L(-1)^j w_2)
		\end{equation*}
		for all $w_1\in W_1$ and $w_2\in W_2$.
	\end{lemma}
	\begin{proof}
		It is enough to show that
		\begin{equation*}
		\langle w', e^{\varepsilon L(-1)}\mathcal{Y}(w_1, e^{\mathrm{log}\,r})w_2\rangle=\sum_{j\geq 0} \dfrac{\varepsilon^j}{j!}\langle w', \mathcal{Y}(w_1, e^{l_{-1/2}(r+\varepsilon)}) L(-1)^j w_2\rangle
		\end{equation*}
		for any $w'\in W'$ and $\varepsilon$ in some neighborhood of $0$. We first observe that
		\begin{equation*}
		I(\varepsilon)= \langle w', e^{\varepsilon L(-1)}\mathcal{Y}(w_1, e^{\mathrm{log}\,r})w_2\rangle=\langle e^{\varepsilon L(1)} w', \mathcal{Y}(w_1, e^{\mathrm{log}\,r})w_2\rangle
		\end{equation*}
		is a polynomial in $\varepsilon$, and so converges for all $\varepsilon\in\CC$. On the other hand,
		\begin{equation*}
		p(z, \varepsilon)=\sum_{j\geq 0} \dfrac{\varepsilon^j}{j!}\langle w', \mathcal{Y}(w_1, e^{l_{-1/2}(z)}) L(-1)^j w_2\rangle=\langle w', \mathcal{Y}(w_1, e^{l_{-1/2}(z)})\Omega(Y_{W_2})(w_2,\varepsilon)\mathbf{1}\rangle
		\end{equation*}
		converges absolutely to a single-valued analytic function of $z$ and $\varepsilon$ in the region given by the conditions
		\begin{equation*}
		\vert z\vert>\vert\varepsilon\vert >0\,\,\,\mathrm{and}\,\,\,\mathrm{arg}\,z\neq\pi
		\end{equation*}
		(recall \cite[Proposition 7.14]{HLZ5}; we do not need to specify a branch of $\mathrm{log}\,\varepsilon$ or impose a condition on $\mathrm{arg}\,\varepsilon$ since $\Omega(Y_{W_2})(\cdot,x)\cdot$ has only integral powers of $x$). But in fact, $p(z,\varepsilon)$ remains analytic in $\varepsilon$ at $\varepsilon=0$ because
		\begin{equation*}
		\sum_{j\geq 0} \dfrac{\varepsilon^j}{j!}\langle w', \mathcal{Y}(w_1, e^{l_{-1/2}(z)}) L(-1)^j w_2\rangle
		\end{equation*}
		for $z$ fixed is a power series in $\varepsilon$ with positive radius of convergence at least $\vert z\vert$. Also, $p(z,0)$ is analytic when $\vert z\vert>0$ and $\mathrm{arg}\,z\neq\pi$. Now, since $r$ is a positive real number, there is a neighborhood $\widetilde{U}$ of $0$ such that $\vert r+\varepsilon\vert>\vert\varepsilon\vert$ and $\mathrm{arg}(r+\varepsilon)\neq\pi$ for $\varepsilon\in\widetilde{U}$. Then $P(\varepsilon)=p(r+\varepsilon,\varepsilon)$ is analytic in $\varepsilon$ on $\widetilde{U}$.
		
		Now, $I(\varepsilon)$ and $P(\varepsilon)$ are both analytic functions on $\widetilde{U}$. To show that they are equal on $\widetilde{U}$, it is sufficient to show that all derivatives of $I$ and $P$ agree at $\varepsilon=0$, because this will show that they have the same power series expansion at $0$ and thus are equal in a neighborhood of $0$. In fact, for $K\geq 0$,
		\begin{align}\label{Ideriv}
		\left.\left(\dfrac{d}{d\varepsilon}\right)^K I(\varepsilon)\right\vert_{\varepsilon=0} & =\left.\langle w', e^{\varepsilon L(-1)} L(-1)^K\mathcal{Y}(w_1, e^{\mathrm{log}\,r})w_2\rangle\right\vert_{\varepsilon=0}\nonumber\\
		&=\sum_{k=0}^K \binom{K}{k}\langle w', \mathcal{Y}(L(-1)^k w_1, e^{\mathrm{log}\,r})L(-1)^{K-k} w_2\rangle,
		\end{align}
		where for the second equality we use the $L(-1)$-commutator formula
		\begin{equation*}
		L(-1)\mathcal{Y}(w_1,x)w_2=\mathcal{Y}(L(-1)w_1,x)w_2+\mathcal{Y}(w_1,x)L(-1)w_2
		\end{equation*}
		(see for instance  \cite[Equation 3.28]{HLZ2}). On the other hand, using Equation 7.37 in the statement of \cite[Proposition 7.14]{HLZ5},
		\begin{align*}
		\left(\dfrac{d}{d\varepsilon}\right)^K  P(\varepsilon)& =\sum_{k=0}^K \binom{K}{k} \left.\dfrac{\partial^K}{\partial z^k\,\partial \varepsilon^{K-k}} p(z,\varepsilon)\right\vert_{z=r+\varepsilon}\nonumber\\
		&=\sum_{k=0}^{K}\binom{K}{k}\left.\left(\dfrac{d}{dx}\right)^k\langle w',\mathcal{Y}(w_1,x)e^{\varepsilon L(-1)}L(-1)^{K-k}w_2\rangle\right\vert_{x=e^{l_{-1/2}(r+\varepsilon)}}\nonumber\\
		& =\sum_{k=0}^K\binom{K}{k}\langle w',\mathcal{Y}(L(-1)^k w_1, e^{l_{-1/2}(r+\varepsilon)})e^{\varepsilon L(-1)}L(-1)^{K-k}w_2\rangle,
		\end{align*}
		which agrees with \eqref{Ideriv} at $\varepsilon=0$ because $\mathrm{log}\,r=l_{-1/2}(r)$. This completes the proof of the lemma.
	\end{proof}
	
	Continuing with the proof of the proposition, we can use the lemma to see that for a neighborhood of $\varepsilon=0$, we have the equality as series over $m\in\RR$:
	\begin{align*}
	\sum_{m\in\RR} & \left\langle w',\mathcal{Y}^1(\pi_{m}(e^{(-r_0+\varepsilon)L(-1)}\mathcal{Y}^2(w_2, e^{l_{-1/2}(r_0)})w_1), e^{\mathrm{log}\,r_1})w_3\right\rangle\nonumber\\
	& =(-1)^{\vert w_1\vert\vert w_2\vert}\sum_{m\in\RR} \left\langle w',\mathcal{Y}^1(\pi_{m}(e^{\varepsilon L(-1)}\Omega_{-1}(\mathcal{Y}^2)(w_1, e^{\mathrm{log}(-r_0)})w_2), e^{\mathrm{log}\,r_1})w_3\right\rangle\nonumber\\
	& =(-1)^{\vert w_1\vert\vert w_2\vert}\sum_{m\in\RR}\sum_{j\geq 0} \dfrac{\varepsilon^j}{j!} \left\langle w',\mathcal{Y}^1(\pi_{m}(\Omega_{-1}(\mathcal{Y}^2)(w_1, e^{l_{-1/2}(-r_0+\varepsilon)})L(-1)^j w_2), e^{\mathrm{log}\,r_1})w_3\right\rangle\nonumber\\
	& =(-1)^{\vert w_1\vert\vert w_2\vert}\sum_{m\in\RR}\sum_{n\in\RR} \left\langle w',\mathcal{Y}^1(\pi_{m}(\Omega_{-1}(\mathcal{Y}^2)(w_1, e^{l_{-1/2}(-r_0+\varepsilon)})\pi_n(\Omega(Y_{W_2})(w_2,\varepsilon)\mathbf{1})), e^{\mathrm{log}\,r_1})w_3\right\rangle.
	\end{align*}
	By Proposition 12.7, part 3, of \cite{HLZ8} and its proof, the last of the above series, considered as a double sum over $m,n\in\RR$, converges absolutely, when
	\begin{equation*}
	r_1>\vert -r_0+\varepsilon\vert>\vert\varepsilon\vert>0,
	\end{equation*}
	to an analytic function of $\varepsilon$ defined in some open neighborhood surrounding $0$ which might not include $0$ itself. However, we show that this double series remains analytic at $0$. Note that
	\begin{align}\label{epsilonpowerseries}
	\sum_{m,n\in\RR} & \left\langle w',\mathcal{Y}^1(\pi_{m}(\Omega_{-1}(\mathcal{Y}^2)(w_1, e^{l_{-1/2}(z_0)})\pi_n(\Omega(Y_{W_2})(w_2,\varepsilon)\mathbf{1})), e^{\mathrm{log}\,r_1})w_3\right\rangle\nonumber\\
	& =\sum_{n\in\RR}\sum_{m\in\RR} \left\langle w',\mathcal{Y}^1(\pi_{m}(\Omega_{-1}(\mathcal{Y}^2)(w_1, e^{l_{-1/2}(z_0)})\pi_n(e^{\varepsilon L(-1)} w_2)), e^{\mathrm{log}\,r_1})w_3\right\rangle\nonumber\\
	& =\sum_{j\geq 0} \dfrac{\varepsilon^j}{j!}\left\langle w',\mathcal{Y}^1(\Omega_{-1}(\mathcal{Y}^2)(w_1, e^{l_{-1/2}(z_0)})L(-1)^j w_2, e^{\mathrm{log}\,r_1})w_2\right\rangle
	\end{align}
	defines an analytic function in $z_0$ and $\varepsilon$ when $z_0$ is contained in a sufficiently small open neighborhood of $-r_0$ and when $\varepsilon$ is contained in a sufficiently small open neighborhood (possibly not including $0$ itself) surrounding $0$. But for each such $z_0$, \eqref{epsilonpowerseries} is a power series in $\varepsilon$ with positive radius of convergence, so the analytic function it defines is still analytic at $\varepsilon=0$. Then replacing $z_0$ with $-r_0+\varepsilon$ for $\varepsilon$ sufficiently small, we see that 
	\begin{equation*}
	\sum_{m,n\in\RR} \left\langle w',\mathcal{Y}^1(\pi_{m}(\Omega_{-1}(\mathcal{Y}^2)(w_1, e^{l_{-1/2}(-r_0+\varepsilon)})\pi_n(\Omega(Y_{W_2})(w_2,\varepsilon)\mathbf{1})), e^{\mathrm{log}\,r_1})w_3\right\rangle
	\end{equation*}
	is an analytic function of $\varepsilon$ at $0$, and this means that
	\eqref{firstseries} converges absolutely to an analytic function of $\varepsilon$ in an open set containing $0$.
	
	Finally, since \eqref{firstseries} and \eqref{secondseries} converge absolutely to analytic functions on some common open set containing $0$, since this common open set has a non-empty intersection with $U$ (since $0$ is contained in the boundary of $U$), and since the two analytic functions agree on $U$, we can conclude that \eqref{firstseries} and \eqref{secondseries} are equal when $\varepsilon=0$. This completes the proof of the proposition.
\end{proof}

Now we take as our setting an extension $V\subseteq A$ where $A$ is a vertex operator superalgebra containing $V$ in its even part. It is not necessary to assume $A$ is an object in a vertex tensor category of $V$-modules since modules in $\repA$ can be defined as in Proposition \ref{propo:YWcomplex}; we just need the relevant compositions of vertex operators to converge in their respective domains. 

We recall the notion of left (respectively, right and two-sided) ideal in $A$ (see for instance \cite[Definition 3.9.7 and Remark 3.9.8]{LL} in the vertex operator algebra setting):
a (not necessarily parity-graded) subspace $I\subseteq A$ such that for all $a\in A$, $b\in I$,  $Y(a,x)b\in I((x))$
(respectively, $Y(b,x)a\in I((x))$ and both $Y(b,x)a, Y(a,x)b\in I((x))$). Equivalently,
for all $z\in\CC^\times$, $Y(a,z)b\in \overline{I}$
(respectively, $Y(b,z)a\in \overline{I}$ and both $Y(b,z)a, Y(a,z)b\in \overline{I}$).
For an object $(W, Y_W)$ in $\repA$ and a subset $T\subseteq W$, define (see\ \cite[Equation 4.5.33]{LL})
\[\mathrm{Ann}_A(T)=\{a\in A\,|\, Y_W(a,x)t=0\,\,\mathrm{for\,\,all}\,\, t\in T\}.\]
\begin{lemma}\label{lem:parityhomonesided}
	A parity-graded one-sided ideal $I$ of a \VOSA{} $A$ is two-sided.
\end{lemma}
\begin{proof}
The result follows by observing that given parity-homogeneous elements $a\in A$, $b\in I$, we have 
\begin{equation*}
Y(a, x)b=(-1)^{|a||b|}e^{x L(-1)} Y(b, -x)a
\end{equation*}
using skew-symmetry, and moreover $L(-1)$ preserves one-sided ideals (see \cite[Remark 3.9.8]{LL}). 
\end{proof}	
Now we have the following analogue of \cite[Theorem 4.5.11]{LL}:
\begin{lemma}\label{lem:repAannisideal}
The annihilator $\mathrm{Ann}_A(T)$ is a left ideal of $A$. 
If $T$ contains the (non-zero) parity-homogeneous components of its elements, then $\mathrm{Ann}_A(T)$ is parity-graded and hence a two-sided ideal.
\end{lemma}
\begin{proof}
 The annihilator $\mathrm{Ann}_A(T)$  is clearly a subspace of $A$. Now we need to show that for any $a\in A$, $b\in\mathrm{Ann}_A(T)$, $n\in\ZZ$, and $t\in T$,
 \begin{equation*}
  Y_W(a_n b, x) t= 0.
 \end{equation*}
This is equivalent to
\begin{equation*}
 \langle w', Y_W(a_n b, x)t\rangle=0
\end{equation*}
for any $w'\in W'$, which, using the theory of unique expansion sets (see \cite[Definition 7.5 and Proposition 7.8]{HLZ5}) and our assumptions on the conformal weights of $W$, is equivalent to
\begin{equation*}
 \langle w', Y_W(a_n b, e^{\log z_2})t\rangle=0
\end{equation*}
for all $z_2$ in some non-empty open subset of $\CC^\times$. To show this, fix $z_2\in\CC^\times$ such that $\mathrm{Re}\,z_2$, $\mathrm{Im}\,z_2>0$. Then applying the associativity property of Proposition \ref{propo:YWcomplex} to $Y_W$ and using $b\in\mathrm{Ann}_A(T)$, we get
\begin{align*}
 \sum_{n\in\ZZ} \langle w', Y_W(a_n b, e^{\log z_2})t\rangle z_0^{-n-1} & =\langle w', Y_W(Y(a, z_0)b, e^{\log z_2})t\rangle\nonumber\\
 &=\langle w', Y_W(a, e^{\log(z_0+z_2)})Y_W(b, e^{\log z_2})t\rangle =0,
\end{align*}
for any $z_0\in\CC^\times$ such that $(z_0+z_2, z_2)\in S_1$. Such $z_0$ form a non-empty open set, so the (single-valued) meromorphic function $\sum_{n\in\ZZ} \langle w', Y_W(a_n b, e^{\log z_2})t\rangle z_0^{-n-1}$ is identically zero on its entire domain (all $z_0$ such that $0<\vert z_0\vert<\vert z_2\vert$). In particular, all coefficients of its Laurent series expansion are zero. Thus
\begin{equation*}
 \langle w', Y_W(a_n b, e^{\log z_2})t\rangle=0
\end{equation*}
for all $n\in\ZZ$ and all $z_2\in\CC^\times$ such that $\mathrm{Re}\,z_2$, $\mathrm{Im}\,z_2>0$. Since such $z_2$ form a non-empty open set, it follows that $\langle w', Y_W(a_n b, x)t\rangle =0$, as desired.
Thus $\mathrm{Ann}_A(T)$ is a left ideal. 

For the second statement, if $b\in\mathrm{Ann}_A(T)$ has parity decomposition $b=b^\even+b^\odd$, then for each non-zero parity-homogeneous $t\in T$, 
the coefficients of $Y_W(b^\even,x)t$ and $Y_W(b^\odd,x)t$ 
are homogeneous of different parities. 
Thus because their sum is zero, they are individually zero and $b^\even, b^\odd\in\mathrm{Ann}_A(T)$.
Now, $\mathrm{Ann}_A(T)$ is two-sided by \cref{lem:parityhomonesided}.
\end{proof}

Now we prove for objects in $\repA$ the analogue of \cite[Proposition 4.5.6]{LL}. For a module $(W, Y_W)$ in $\repA$ and a subset $T\subseteq W$, recall that the $A$-submodule $\langle T\rangle$ generated by $T$ is the smallest
subspace of $W$ containing $T$ and closed under the vertex operator $Y_W$. The submodule $\langle T\rangle$ is graded by conformal weights since it is closed under $L(0)$, but it may not be parity-graded.
\begin{lemma}\label{lem:genmodspan}
 The submodule $\langle T\rangle$ is spanned by the vectors $a_{n; k} t$ for $a\in A$, $n\in\RR$, $k\in\NN$, and $t\in T$. The parity-graded submodule generated by $T$ is spanned by the parity-homogeneous components of $a_{n;k} t$ for $a\in A$, $n\in\RR$, $k\in\NN$, and $t\in T$.
\end{lemma}
\begin{proof}
 Set $\widetilde{W}$ to be the span of the vectors $a_{n; k} t$ for $a\in A$, $n\in\RR$, $k\in\NN$, and $t\in T$. To show that $\widetilde{W}=\langle T\rangle$, it is sufficient to show that $\widetilde{W}$ is closed under $Y_W$. We claim that $\widetilde{W}$ is graded by conformal weights. Indeed, it follows from the straightforward generalization of \cite[Proposition 4.5.7]{LL} to logarithmic intertwining operators among $V$-modules, as in the proof of \cite[Proposition 4.5.6]{LL}, that $\widetilde{W}$ is a $V$-submodule of $W$. In particular, $\widetilde{W}$ is $L(0)$-stable and thus graded by conformal weights (see \cite[Remark 2.13]{HLZ1}).

 Now set 
 $$\widetilde{W}^\perp =\lbrace w'\in W'\,\vert\,\langle w', \widetilde{w}\rangle=0\,\,\mathrm{for\,\,all}\,\,\widetilde{w}\in\widetilde{W}\rbrace.$$
 Since $\widetilde{W}$ is graded by conformal weights, so is $\widetilde{W}^\perp$. Then it follows that $\widetilde{W}=(\widetilde{W}^\perp)^\perp$ since this relation holds when restricted to each finite-dimensional weight space. In other words, $\widetilde{w}\in\widetilde{W}$ if and only if $\langle w',\widetilde{w}\rangle=0$ for all $w'\in \widetilde{W}^\perp$. Thus $\widetilde{W}=\langle T\rangle$ if and only if for all $a_1\in A$,
 \begin{equation*}
  \langle w', Y_W(a_1,x_1)(a_2)_{n; k} t\rangle=0
 \end{equation*}
for all $w'\in\widetilde{W}^\perp$, $a_2\in A$, $n\in\RR$, $k\in\NN$, $t\in T$. Using \cite[Proposition 7.8]{HLZ5}, this is equivalent to
\begin{equation*}
 \langle w', Y_W(a_1, e^{\log z_1})(a_2)_{n; k} t\rangle =0
\end{equation*}
for all $z_1$ in some non-empty open subset of $\CC^\times$. We take $z_1$ such that $\mathrm{Re}\,z_1, \mathrm{Im}\,z_1>0$. Then by the associativity property of Proposition \ref{propo:YWcomplex},
\begin{align*}
 \sum_{n\in\RR} \sum_{k\in\NN} & \langle w', Y_W(a_1, e^{\log z_1})(a_2)_{n; k} t\rangle e^{(-n-1)\log z_2} (\log z_2)^k = \langle w', Y_W(a_1, e^{\log z_1})Y_W(a_2, e^{\log z_2})t\rangle\nonumber\\
& = \langle w', Y_W(Y(a_1, e^{\log(z_1-z_2)})a_2, e^{\log z_2})t\rangle =\sum_{n\in\RR} \langle w', Y_W(\pi_n(Y(a_1, e^{\log(z_1-z_2)})a_2), e^{\log z_2})t\rangle = 0
\end{align*}
for all $z_2$ such that $(z_1,z_2)\in S_1$. Since such $z_2$ form a non-empty open set, \cite[Proposition 7.8]{HLZ5} implies that each coefficient of $e^{(-n-1)\log z_2} (\log z_2)^k$ in the first series above equals $0$, as desired.

Now the subspace spanned by the parity-homogeneous components of $a_{n; k} t$ for $a\in A$, $n\in\RR$, $k\in\NN$, and $t\in T$ is certainly parity-graded and contained in any parity-graded submodule of $W$ containing $T$. Then the above argument with $T$ replaced by the parity-homogeneous components of elements of $T$ shows that this subspace is a submodule. Thus the second conclusion of the lemma follows.
\end{proof}

\section{Applications}\label{sec:applications}

In this section, we apply our results on tensor categories
for \VOSA{} extensions, especially the induction functor, to several explicit examples.

\subsection{Examples of vertex operator superalgebra extensions}

In general, it is difficult to prove that one
indeed has a \VOSA{} extension. Therefore, we start with some strategies.
\subsubsection{Simple current extensions}\label{sec:simplecur}

Certain types of simple current extensions were studied in \cite{CKL}.
Here we summarize a few important results.
Let $\cC$ be a ribbon vertex tensor category of generalized modules for 
 a simple \VOA{} $V$.
Let $J$ be a simple current, that is, $J$ is simple and 
there exists a simple object $J^{-1}$ such that $J^{-1}\fus{} J\cong V$.
Let $c_{J,J}$ be the scalar such that
$\cR_{J,J}=c_{J,J} 1_{J\fus{}J}$ and let $\theta_J$ be the scalar by which the twist $\theta$ acts on $J$.
\begin{theo}\label{thm:spinstats}\textup{\cite[Corollary\ 2.8]{CKL}}
	The following relation (spin statistics) holds:
	\[ c_{J,J} = \theta_J \qdim(J).\]
\end{theo}
\begin{theo}\label{thm:fourcases} \textup{\cite[Theorem\ 3.9]{CKL}}
	Let $J$ be a self-dual simple current ($J\fus{}J\cong V$)
	with $\theta_J=\pm 1$.
	This leads to the following four cases, with $A=V\oplus J$.
	\begin{enumerate}
		\item When $\theta_J = 1$ and $\qdim(J)=1$, $A$
		has a natural structure of a $\ZZ$-graded \VOA{}.
		\item When $\theta_J = 1$ and $\qdim(J)=-1$, $A$
		has a natural structure of a $\ZZ$-graded \VOSA{}
		with even part $V$ and odd part $J$.
		\item When $\theta_J = -1$ and $\qdim(J)=1$, $A$
		has a natural structure of a $\frac{1}{2}\ZZ$-graded \VOSA{},
		with even part $V$ and odd part $J$.
		\item When $\theta_J = -1$ and $\qdim(J)=-1$, $A$
		has a natural structure of a $\frac{1}{2}\ZZ$-graded \VOA{}.
	\end{enumerate}
\end{theo}
In many cases, $\qdim(J)$ is easily deduced from modular $S$-matrices. 
Naturally, we refer to the cases with $\qdim(J)=1$ as producing
``correct statistics'' and $\qdim(J)=-1$ as ``wrong statistics.''  In the next subsection, we will derive Verlinde formulae corresponding to these cases. A natural generalization of the above theorem
when $J$ is not necessarily self-dual was also given in
\cite{CKL}, using \cite{Car}.

There is also a criterion based on conformal weights that determines
when objects in $\cC$ induce to objects of $\rep^0 A$.
Let $G$ be the group generated by a simple current $J$ and
let 
\[
A = \bigoplus_{g\in G}J^g.
\]
 Assume that $A$ is a \VOSA{}, which is the case if the twist satisfies
\[
\theta_{J^g} \in \pm 1\qquad \text{and} \qquad \theta_{J^g} = \theta_{J^{g+2}}
\]
for all $g$  \cite[Theorem\ 3.12]{CKL}.
Let $h_\bullet$ denote conformal dimensions, that is, lowest conformal weights.
\begin{theo}\label{thm:rep0conformal} \textup{\cite[Theorems\ 3.18 and 3.20]{CKL}}
	Let $P$ be an indecomposable
	generalized $V$-module in $\cC$.
	\begin{enumerate}
		\item If $P$ is a subquotient of a tensor product of several simple
		modules, then $\cF(P)$ is in $\rep^0 A$ if and only if
		$h_{J\fus{}P}- h_{J}- h_{P}\in\ZZ$.\label{it:inf}
		\item Suppose $J$ has finite order and $P$ satisfies both $\dim(\hom(P,P)) < \infty$ and
		$\dim(\hom(J\fus{}P,J\fus{}P)) < \infty$.  Assume also that $L(0)$ has
		Jordan blocks of bounded size on both $P$ and $J\fus{}P$.  Then
		$\cF(P)$ is an object of $\rep^0A$ if and only if
		$h_{J\fus{}P}- h_{J}- h_{P}\in\ZZ$.\label{it:fin}
	\end{enumerate}
\end{theo}
If $\cC$ is locally finite, the finiteness assumptions in part (\ref{it:fin}) are
satisfied automatically. Among other things, this criterion led
to a new proof of rationality of parafermions for affine Lie algebras
at positive integral levels in \cite{CKLR}. We shall use this and
similar criteria below.

We now present two propositions that we shall frequently use below. The first is stated for extensions that are not necessarily of 
simple current type:
\begin{propo}\label{prop:Fsimple}
Let $A$ be a vertex operator (super)algebra extension of $V$ with $V\subseteq A^\even$ and both $A$, $V$ simple, and suppose $A=\bigoplus_{i\in I}A_i$ as a $V$-module with each $A_i$ a simple $V$-module. If $W$ is a simple $V$-module such that each non-zero $A_i\fus{V}W$ is a simple
$V$-module and each non-zero $A_i\fus{V}W \not\cong A_j\fus{V}W$ unless $i=j$, then $\cF(W)$ is  a simple object of $\repA$. 
\end{propo}
\begin{proof}
Since $A_i\boxtimes_V W\not\cong A_j\fus{V}W$ unless $i=j$, the $A_i$ are non-isomorphic $V$-modules. Since $V\subseteq A^\even$, the parity-homogeneous components of $A$ are $V$-modules, and so because the $A_i$ are distinct and simple, the Density Theorem implies that $A^\even$ and $A^\odd$ are direct sums of certain $A_i$. That is, each $A_i$ has definite parity.

We have $\cF(W)=\bigoplus_{i\in I}A_i\fus{V} W$. We will identify the component isomorphic to $W$ in 
$\cF(W)$ with $W$ and let $W_j=A_j\fus{V} W$. Since $W$ is a simple $V$-module, it must have definite parity as an object of $\sC$; thus the parity-homogeneous components of $\cF(W)$ are $A^\even\boxtimes_V W$ and $A^\odd\boxtimes_V W$.
If $X$ is a non-zero $A$-submodule of $\cF(W)$, then because the $A_i\fus{V} W$ are simple and mutually inequivalent, the Density Theorem implies there is some $j\in I$ such that $W_j=A_j\fus{V}W\subseteq X$.
Let $\langle W_j \rangle$ be the $A$-submodule of $\cF(W)$ generated
by $W_j$.

Now, we prove that $\langle W_j\rangle=\cF(W)$, so that $\cF(W)=\langle W_j\rangle \subseteq X$, that is, $X=\cF(W)$. 
We use the observation that for any non-zero and parity-homogeneous $a\in A$, $w\in W$,
\[ \left\langle a_{n;k} w\,\vert \,n\in \RR,\,k\in\NN   \right\rangle
= \langle w\rangle, \]
 where angle braces denote the generated submodule
 and $a_{n;k}w$ is the coefficient of $x^{-n-1}(\log x)^k$ in $Y_W(a,x)w$.
Indeed, it is clear that $\left\langle a_{n;k} w\,\vert\,n\in \RR,\,k\in\NN  \right\rangle
\subseteq \langle w\rangle$. Then $a$ is in the $\ZZ/2\ZZ$-graded annihilator ideal (see Lemma \ref{lem:repAannisideal}) of the image of $w$ in the $\ZZ/2\ZZ$-graded quotient module $\langle w \rangle / \left\langle a_{n;k} w\,\vert\,n\in \RR,\,k\in\NN    \right\rangle$. Since $a\neq 0$ and $A$ is simple, this means $A$ annihilates the image of $w$ in the quotient. Since the image of $w$ also generates the quotient, the quotient is zero and $\langle a_{n;k} w\,\vert\, n\in\RR,\,k\in\NN  \rangle=\langle w\rangle$, justifying the observation. Now  take a non-zero $a\in A_j$ and non-zero $w\in W$; both of these vectors are parity-homogeneous. Then the observation implies
\begin{equation*}
 \langle W_j\rangle\supseteq\langle a_{n;k}w\,\vert\,n\in\RR,\,k\in\NN\rangle=\langle w\rangle=\cF(W)
\end{equation*}
as desired; here we are identifying $w\in W$ with $\unit\boxtimes_{V} w\in V\boxtimes_{V} W$. 
\end{proof}

\begin{propo}\label{prop:SimpleCurrentSmatrix}
	Let $\mathcal C$ be a semisimple vertex tensor category of modules for a simple \VOA{} $A_0$, and let $A$ be a simple (super)algebra in $\cC$ such that 
	$A=\bigoplus_{g\in G} A_g$ for $G$ a finite abelian group and $A_{g}\boxtimes A_h\cong A_{gh}$ for $g,h\in G$. 
	Assume fixed-point freeness: for every simple module $W$ in $\cC$, $A_g\boxtimes W\cong W$ implies $g=1$. 
	Then:
	\begin{enumerate}
	\item\label{it:SimpleCurrentSmatrix1} If $W$ is simple in $\sC$, then $\cF(W)$ is simple in $\repA$.
	\item\label{it:SimpleCurrentSmatrix2} Every simple object in $\repA$ is isomorphic to an induced object.
	\item\label{it:SimpleCurrentSmatrix3} If $W$ is simple in $\sC$, then $\cF(W)$ is isomorphic to $\cF(A_g\boxtimes W)$ in $\repA$ for any $g\in G$.
	\item Suppose $\cC$ is a ribbon category and $\theta_A=1_A$. If $W_1, W_2$ are two simple objects of $\sC$ such that $\cF(W_1),\cF(W_2)$ are objects of $\rep^0A$, then for every $g, h$ in $G$,
	\[
	S^\hopflink_{\cF(A_g\boxtimes W_1), \cF(A_h\boxtimes W_2)} = S^\hopflink_{\cF(W_1), \cF(W_2)}.
	\]
	\end{enumerate}
\end{propo}
\begin{proof}
	Since $A_g\boxtimes A_{-g}\cong A_0$, each $A_g$ is a simple current, and hence $A_g\boxtimes W$ is simple for any simple object $W$  of $\cC$ (see \cite{CKLR}). Now the first assertion is immediate from \cref{prop:Fsimple}.
	
	Now let $X$ be a simple object in $\repA$. By the  semisimplicity of $\cC$ (and thus also of $\sC$) we have $X \cong_{\sC}\bigoplus_{i\in I} X_i$, where the $X_i$ are simple objects in $\sC$. 
	For any such $X_i$,
	we have the (non-zero, even) $\cC$-inclusion $X_i \hookrightarrow \cG(X)$ where $\cG$ is the restriction functor, so 
	by \cref{lem:FGadjoint}, there is a non-zero, even $\repA$-morphism $f:\cF(X_i) \rightarrow X$. Since $A_g\boxtimes\cdot$ is fixed-point free, \cref{prop:Fsimple} shows that $\cF(X_i)$ is simple. Then by Schur's Lemma,  $f: \cF(X_i)\rightarrow X$ is an isomorphism.

	This argument also implies that $\cF(X_i)\cong \cF(X_j)\cong X$ whenever $X_i$, $X_j$ are two summands appearing in the direct sum decomposition of $\cG(X)$.
	Now, if $W$ is simple in $\cC$, then $A_g\boxtimes W$ and $W\cong A_0\boxtimes W$ are summands in $\cG(\cF(W))$. Hence, the previous argument shows that $\cF(W)\cong \cF(A_g\boxtimes W)$ for any $g\in G$.

	From Proposition \ref{prop:inducedrepzAribbon}, induced objects in $\repzA$ form a ribbon category since $\theta_A=1_A$.    Then the conclusion about the $S^\hopflink$-matrix is easy by properties of traces and monodromies:
	if $f_1:X_1\rightarrow \widetilde{X}_1$ and  $f_2:X_2\rightarrow \widetilde{X}_2$ are any even isomorphisms in a ribbon supercategory, then 
	\begin{align*}	
	S^{\hopflink}_{\widetilde{X}_1,\widetilde{X}_2}=\tr(\cM_{\widetilde{X}_1,\widetilde{X}_2})=\tr((f_1\boxtimes f_2)\circ \cM_{X_1,X_2} \circ (f_1^{-1}\boxtimes f_2^{-1}))
	=\tr(\cM_{X_1,X_2}) = S^{\hopflink}_{X_1,X_2},
	\end{align*}
	using Corollary \ref{conjtr}.
\end{proof}

\begin{rema}
 In Propositions \ref{prop:Fsimple} and \ref{prop:SimpleCurrentSmatrix}, there is some ambiguity about what is a simple object in $\repA$: should it have no proper non-zero $A$-submodules, or should it have no proper non-zero $\ZZ/2\ZZ$-graded $A$-submodules? We will discuss this issue in more detail in Section \ref{sec:Verlinde}, but for now we note that it does not affect the results here. In the proof of Proposition \ref{prop:Fsimple}, no $\ZZ/2\ZZ$-grading assumption was needed for the submodule $X$, because the conclusion that $\langle W_j\rangle\subseteq X$ for some $j$ followed from properties of modules for the vertex operator algebra $V$. Thus under the assumptions of Proposition \ref{prop:Fsimple}, simple $V$-modules induce to modules in $\repA$ which are simple in both possible senses. The distinction between the two possible meanings of simplicity in $\repA$ also does not affect Proposition \ref{prop:SimpleCurrentSmatrix} because the evenness of the morphism $f: \cF(X_i)\rightarrow X$ in the proof of part 2 implies that its kernel and image are $\ZZ/2\ZZ$-graded; hence Schur's Lemma applies regardless of what is meant by the simplicity of $X$.
\end{rema}

\begin{exam}
Some of the best-known rational and $C_2$-cofinite vertex algebras are those associated to even lattices $L$, denoted $V_L$ \cite{FLM, D}. 
The inequivalent simple modules $M_\lambda$ are parameterized by $\lambda \in L'/L$ with $L'$ the lattice dual to $L$. Moreover, the fusion rules are given by addition in $L'/L$, that is,
\[
M_\lambda  \boxtimes M_\mu \cong M_{\lambda + \mu}.
\]
Especially, every simple object is a simple current. Their dimensions $\qdim(M_\lambda) =1$  follow immediately from the modular transformations of characters. 
 Let $L \subseteq N \subseteq L'$. If $N$ is an even lattice, then the $M_\lambda$ for $\lambda \in N/L$ are $\ZZ$-graded, so
\[
A= \bigoplus_{\lambda \in N/L} M_\lambda
\]
is a vertex algebra extension of $V$. If $N$ is an integral lattice, then $A$ is a vertex superalgebra.
\end{exam}
More interesting vertex operator (super)algebras can be constructed from known ones via standard constructions such as kernel of screenings, cohomologies, and cosets. 
We illustrate this next in a few examples involving well-known \VOAs{} such as lattice and Virasoro \VOAs{}, and affine \VOAs{} associated to $\mathfrak{sl}_2$. We start with extensions of tensor products of triplet algebras as subalgebras of lattice vertex superalgebras.

\subsubsection{Extensions associated to $\cW(p)$-algebras via kernels of screening charges}

The general idea of this section is as follows: Let $V$ be a \VOA, $W$ a $V$-module, and $Q$ a screening operator, that is, the zero-mode of an intertwiner from $V$ to $W$, such that its kernel is a vertex operator subalgebra $U\subseteq V$. Then one uses extensions of $V$ to construct extensions of the subalgebra $U$. In our example here, $V$ will be a lattice \VOA{}, 
and 
we will construct some new \VOAs{} as extensions of tensor
products of $\cW(p)$-algebras. The $\cW(p)$-algebras are non-rational \VOAs{}, but they are $C_2$-cofinite, simple, and self-contragredient. They are by far the best understood class of \VOAs{} of this type: see for example \cite{Kau, FGST, AM4, AM1, AM2, TW1}. We use \cite{CRW} as reference.

Define
$\alpha_+=\sqrt{2p}$ and $\alpha_-=-\sqrt{2/p}$.  Consider the \VOA{}
$V_{\sqrt{2p}\ZZ}$ associated to the lattice $\sqrt{2p}\ZZ$. The vertex subalgebra of $V_{\sqrt{2p}\ZZ}$ given
by
\[
X_1^+(p)=\cW(p) = \mathrm{Ker}\left(\int Q_- : V_{\sqrt{2p}\ZZ} \rightarrow
  V_{\sqrt{2p}\ZZ + \alpha_-} \right)
 \]
   is the triplet algera $\cW(p)$, where $Q_-$ is the
vertex operator corresponding to the dual lattice element
$e^{\alpha_-}$, and $\int$ stands for its zero mode. The subalgebra $\cW(p)$ has a
different conformal vector than $V_{\alpha_+\ZZ}$, given by
$$\frac{1}{2}a(-1)^2+\frac{p-1}{\sqrt{2p}}a(-2)$$
with $\langle a,a\rangle=1$. The kernel
\[
X_1^-(p)= \mathrm{Ker}\left(\int Q_- : V_{\sqrt{2p}\ZZ+p} \rightarrow
  V_{\sqrt{2p}\ZZ + p+\alpha_-} \right)
 \]
is a simple $\cW(p)$-module denoted by $X_1^-(p)$.
From \cite[Equation 4.5]{CKL}, the dimension of this module is
\[
\text{qdim}(X_1^-(p))= - (-1)^p.
\]
Let
\[
\cW:= \cW(p_1) \otimes \dots \otimes \cW(p_n)
\]
for a choice of $n$ triplet algebras $\cW(p_1), \dots, \cW(p_n)$. 
Define two lattices
\[
L := \sqrt{2p_1} \ZZ \oplus  \sqrt{2p_2} \ZZ \oplus \dots \oplus \sqrt{2p_{n-1}} \ZZ \oplus  \sqrt{2p_n}\ZZ \qquad \text{and}\qquad \] \[
N := \sqrt{\frac{p_1}{2}} \ZZ \oplus  \sqrt{\frac{p_2}{2}} \ZZ \oplus \dots \oplus \sqrt{\frac{p_{n-1}}{2}} \ZZ \oplus \sqrt{\frac{p_n}{2}}\ZZ
\]
so that $N/L\cong (\ZZ/2\ZZ)^n$. Let $L\subseteq R\subseteq N$ be another lattice, so that $R/L$ can be identified with a binary code $C_R\subseteq (\ZZ/2\ZZ)^n$.
We require $R$ to be an integral lattice so that one has the lattice \VOA{} $V_R$ if $R$ is even and the lattice \VOSA{} otherwise. 
Let 
\[
\cW_R :=  \bigcap_{i=1}^n  \mathrm{Ker}\left(\int Q^{(i)}_- : V_R \rightarrow V_{R+\alpha_-^{(i)}}\right) 
\]
where $Q^{(i)}_-$ is the screening charge of $\cW(p_i)$ and $\alpha_-^{(i)}$ is the dual lattice vector with all entries equal to zero except the $i$-th one being $-\sqrt{2/p_i}$.
Then $\cW_R$ is a vertex algebra as it is the joint kernel of screenings of a \VOA. 
Defining $\pi: \ZZ/2\ZZ \rightarrow \{\pm \}$ by $\pi(0)=+$ and $\pi(1)=-$, 
we have
\[
\cW_R \cong \bigoplus_{x\in C_R} M(x), \qquad\qquad M(x) = X_1^{\pi(x_1)}(p_1) \otimes \dots \otimes  X_1^{\pi(x_n)}(p_n)
\]
as a module for $\cW$. 
Note the conformal dimension
\[
\Delta(X^-_1(p_i)) = \frac{3p_i-2}{4}.
\]
We have several interesting cases:
\begin{enumerate}
\item If $p_i\equiv0\mod 4$ for all $i=1, \dots, n$ , then $N$ is an even lattice and hence so is $R$. It follows that $V_R$ is a \VOA{} and hence so is $\cW_R$. Since $\Delta(X^-_1(p_i)) = \frac{3p_i-2}{4}\equiv\frac{1}{2} \mod\ZZ$, even code words correspond to $\ZZ$-graded modules and odd code words to $(\ZZ+\frac{1}{2})$-graded modules. This means that if $C_R$ is an even code, then $\cW_R$ is a $\ZZ$-graded \VOA{}, and otherwise  $\cW_R$ is a $\frac{1}{2}\ZZ$-graded \VOA. 
\item If $p_i\equiv2\mod 4$ for all $i=1, \dots, n$, then $N$ is an odd lattice. It follows that $R$ is an even lattice if $C_R$ is an even code and an odd lattice otherwise. Hence $V_R$ is a \VOA{} if $C_R$ is even and a \VOSA{} otherwise. 
 Since $\Delta(X^-_1(p_i)) = \frac{3p_i-2}{4}\equiv 0  \mod\ZZ$, all modules are $\ZZ$-graded. This means that if $C_R$ is an even code, then $\cW_R$ is a $\ZZ$-graded \VOA{}, and otherwise $\cW_R$ is a $\ZZ$-graded \VOSA{}. 
\item If $p_i\equiv1 \mod 4$ for all $i=1, \dots, n$ and $C_R$ is a doubly even code, then $R$ is an even lattice. Hence $V_R$ is a \VOA{}. 
 Since $\Delta(X^-_1(p_i)) = \frac{3p_i-2}{4}\equiv \frac{1}{4}  \mod\ZZ$, $\cW_R$ is $\ZZ$-graded, that is, $\cW_R$ is a $\ZZ$-graded \VOA. 
\end{enumerate}

\subsubsection{Extensions associated to Virasoro algebras via screenings and cosets}

Now we will construct extensions of tensor products of Virasoro algebras. We present two methods, using screening charges and cosets. 
While the first method is more general (and similar to the previous section), the second has the advantage that it allows one to prove that the extension algebra is a simple \VOA{}. The general idea is as follows: 

Let $V$ be a \VOA{} and $U\subseteq V$ a vertex operator subalgebra with a different conformal vector. Then the subalgebra 
$$C = \text{Com}(U, V) =\lbrace c\in V\,\,\vert\,\, [Y(c, x_1), Y(u,x_2)]=0\,\,\text{for all}\,\,u\in U\rbrace$$
 is called the commutant or coset subalgebra of $V$. Often, $U$ is an affine vertex subalgebra of $V$, and it turns out that many important \VOAs{} appear as cosets of affine vertex subalgebras inside larger structures. Probably the best-known family is $V = V^{k-1}(\mathfrak g) \otimes L_1(\mathfrak g)$ and $U = V^{k}(\mathfrak g)$ for $\mathfrak g$ a simply-laced Lie algebra and $V^{k}(\mathfrak g)$ the universal affine \VOA{} of $\mathfrak g$ at level $k$ and $L_{k}(\mathfrak g)$ its simple quotient. In this case the coset is the principal $W$-algebra of $\mathfrak g$ at a certain level $\ell(k)$ depending on $k$ \cite{ACL2}. This means that the coset of $U = V^{k}(\mathfrak g)$ in $V = V^{k-2}(\mathfrak g) \otimes L_1(\mathfrak g)\otimes L_1(\mathfrak g)$ is automatically an extension of the tensor product of the principal $W$-algebras of $\mathfrak g$ at levels  $\ell(k-1)$ and $\ell(k)$. The principal $W$-algebra of $\mathfrak{sl}_2$ at level $k$ is the Virasoro algebra at central charge $13-6(k+2)-6(k+2)^{-1}$.
It turns out that more general coset theorems can be very effectively used to construct many families of extensions of tensor products of $W$-algebras; see \cite[Theorem 9.1]{ACF}.

We turn to the example of the Virasoro algebra. 
As reference on screening charges and the subtleties on choosing contours we refer to \cite{KT1, KT2, TW2}; we also use \cite{IK, FMS}.  Denote by $\vir(u, v)$ the simple rational Virasoro \VOA{}
of central charge
\[
c=1-6\frac{(u-v)^2}{uv}.
\]
Then let $\alpha_+=\sqrt{2v/u}$ and $\alpha_-=-\sqrt{2u/v}$ so that
$\alpha_+\alpha_-=-2$. Further set $\alpha_0=\alpha_++\alpha_-$.  We
set
\[
\alpha_{r,s} = \frac{1-r}{2}\alpha_++\frac{1-s}{2}\alpha_-
\]
Modules of $\vir(u, v)$ are then constructible as cohomology of an $s$-fold contour, a
screening charge $Q_{[s]}$, acting on Fock modules of a rank one Heisenberg
\VOA. The result is
\begin{equation} \label{eq:cohomfock}
  H^{Q_{[s]}}(\mathcal F_\lambda)=
  \begin{cases}
    W_{r, s} & \text{if} \
    \lambda= \alpha_{r,s } \ \text{for} \
    1\leq r \leq u-1 \ \text{and} \ 1\leq s \leq v-1 \\
    0 & \text{otherwise} 
\end{cases} .
\end{equation}
The $W_{r, s}$ are simple, every simple $\vir(u, v)$-module is
isomorphic to at least one of them, and further
$W_{r, s} \cong W_{u-r, v-s}$.  The conformal dimension of $W_{r, s}$
is
\[
  h_{r, s} = \frac{1}{2} \alpha_{r, s} (\alpha_{r, s} -\alpha_0)
  = \frac{(rv-us)^2-(u-v)^2}{4uv}
\]
with $\alpha_0=\alpha_++\alpha_-$. The Virasoro \VOA{} itself is $W_{1, 1}$. We next want to extend
\[
V=\vir(u_1, v_1) \otimes \dots \otimes \vir(u_n,v_n) 
\]
for some choice of $n$ Virasoro \VOAs.
Let 
\[
\alpha_\pm^{(i)} :=\pm \sqrt{2 \left(\frac{v_{i}}{u_i}\right)^{\pm 1}} \qquad\text{and}\qquad \alpha_{r, s}^{(i)} = \frac{1-r}{2}\alpha_+^{(i)}+\frac{1-s}{2}\alpha_-^{(i)}.
\]
Let 
\[
L= \alpha^{(1)}_{2,1}\ZZ \oplus  \alpha^{(2)}_{2,1}\ZZ \oplus  \dots \oplus \alpha^{(n-1)}_{2,1}\ZZ \oplus \alpha^{(n)}_{2,1}\ZZ  
\]
and $N$ an even sublattice. Let $Q^{(i)}=Q^{(i)}_{[1]}$. The cohomology $H^Q(V_N)$ of
\[
Q = Q^{(1)} +  \dots + Q^{(n)}
\]
associated to the $n$ screening charges $Q^{(1)}, \dots, Q^{(n)}$
corresponding to the choice of $n$ Virasoro \VOAs{} is then a module for
$V$. Since $V_N$ a \VOA, the same is true for $H^Q(V_N)$
provided the latter is $\ZZ$-graded by $L(0)$. If we replace $N$ by an integral lattice we might end up with a
\VOSA{}. This grading question is 
described by the inner product with a vector
\[
\rho = \alpha_0^{(1)} + \dots +  \alpha_0^{(n)}, \qquad  \alpha_0^{(i)}=\alpha_+^{(i)}+\alpha_-^{(i)},
\]
namely:
\begin{theo}
  If $N$ is an even lattice extending the rank $n$ Heisenberg \VOA,
  then $H^Q(V_N)$ is a \VOA{} if and only if $(\rho, \lambda) \in
  2\mathbb Z$ for all $\lambda$ in $N$.  Let $L\supseteq N$ be another
  lattice such that $L/N=\mathbb Z/2\mathbb Z$ and with glue vector
  $\gamma$, that is $L=N\cup (N+\gamma)$.  Then we have the following
  cases:
\begin{enumerate}
\item $H^Q(V_L)$ is a $\ZZ$-graded \VOA{} if $\gamma^2$ and $(\rho, \gamma)$ are even.
\item $H^Q(V_L)$ is a $\frac{1}{2}\ZZ$-graded \VOA{} if $\gamma^2$ is even
  and $(\rho, \gamma)$ is odd.
\item $H^Q(V_L)$ is a $\frac{1}{2}\ZZ$-graded \VOSA{} if $\gamma^2$ is odd and
  $(\rho, \gamma)$ is even.
\item $H^Q(V_L)$ is a $\ZZ$-graded \VOSA{} if $\gamma^2$ and $(\rho, \gamma)$ are odd.
\end{enumerate}
\end{theo}
The following illustrates the theorem:
\begin{corol}
  Let $u, v, w$ be positive integers such that both $v$ and $w$ are
  coprime to $u$. Then
\[
\mathcal A_{u; v, w} := \bigoplus_{i=1}^{u-1} W^{(u, v)}_{i, 1} \otimes W^{(u, w)}_{i, 1}
\]
for $v+w=2du$ has a \VOSA{} structure extending
$\vir(u, v)\otimes \vir(u, w)$ if $d$ is odd and a \VOA{} structure
extending $\vir(u, v)\otimes \vir(u, w)$ if $d$ is even.
\end{corol}
\begin{proof}
  We have that $\alpha_+^{(1)}=\sqrt{2v/u}$ and
  $\alpha_+^{(2)}=\sqrt{2w/u}$. Define
  $\lambda = \frac{1}{2} \left( \alpha_+^{(1)}, \alpha_+^{(2)}\right)
  \in \mathbb R^2$ and the rank one lattice $L=\lambda \mathbb
  Z$. Then from \eqref{eq:cohomfock} it is clear that
\[
  H^Q(\mathcal F_{(1-i)\lambda})=
  \begin{cases}
    W^{(u, v)}_{i, 1} \otimes W^{(u, w)}_{i, 1}
    & \text{if} \  1\leq i \leq u-1
    \\ 0 & \text{otherwise} 
\end{cases} 
\]
and hence 
\[
H^Q(V_L) = \bigoplus_{i=1}^{u-1} W^{(u, v)}_{i, 1} \otimes W^{(u, w)}_{i, 1}.
\]
It remains to verify conformal dimensions. We have 
\[
\Delta(\cF_\lambda)=\frac{\lambda^2}{2}=\frac{1}{4} \frac{v+w}{u}=\frac{d}{2}
\]
so that $V_L$ is a \VOA{} if $d$ is even and a \VOSA{} if $d$ is odd.
Moreover the conformal dimension of $H^Q(\mathcal F_{(1-i)\lambda})$
is
\begin{align*}
  \Delta\left(H^Q(\mathcal F_{(1-i)\lambda})\right) & =
  \frac{(vi-u)^2-(v-u)^2}{4uv}+\frac{(wi-u)^2-(w-u)^2}{4uw}\nonumber\\
  & = 1-i
  +\frac{(i^2-1)(v+w)}{4u}= 1-i+\frac{d(i^2-1)}{2}
\end{align*}
so that we have
\[
  \Delta(F_\lambda)\equiv\Delta\left(H^Q(\mathcal F_{(1-i)\lambda})\right)
  \mod\ \mathbb Z,
\]
that is, $H^Q(V_L)$ is also a \VOSA{} if $d$ is odd and a \VOA{} if $d$
is even.
\end{proof}

It is not clear from the construction whether the vertex operator (super)algebra structure on $\mathcal A_{u; v, w}$ is simple. When $d=1$, another construction shows that $\mathcal A_{u; v, w}$ can be given a simple vertex operator (super)algebra structure:
\begin{theo}Let
\[
k= \frac{v}{u-v}-2=\frac{u}{w-u}-3,\qquad \qquad v+w=2u\qquad\text{and} \qquad v<u.
\]  Then
\[
\mathcal A_{u; v, w} \cong C_k := \com\left(L_{k+2}(\mathfrak{sl}_2), \mathcal F(4)\otimes L_{k}(\mathfrak{sl}_2) \right)
\]
as $\vir(u, w)\otimes \vir(u, v)$-modules. In particular, $C_k$ is a simple \VOA{} extension of a rational \VOA.
\end{theo}
\begin{proof}
Let $\omega_0, \omega_1$ be the fundamental weights of $\widehat{\mathfrak{sl}}_2$, so that $L_k(\mathfrak{sl}_2)=L(k\omega_0)$.
Recall that the vertex operator superalgebra $\mathcal F(4)$ of four free fermions is isomorphic to 
\[
\mathcal F(4) \cong \left(L_1(\mathfrak{sl}_2) \otimes L_1(\mathfrak{sl}_2)\right) \oplus \left(L(\omega_1)\otimes L(\omega_1)\right).
\]
Using Remark $10.3$ of \cite{IK} it is a computation to verify that
\begin{equation}\nonumber
\begin{split}
\mathcal F(4) \otimes L(k\omega_0) &\cong  \left(L(\omega_0)\otimes L(\omega_0) \otimes L(k\omega_0)\right)\oplus \left(L(\omega_1)\otimes L(\omega_1) \otimes L(k\omega_0)\right) \\
&\cong \left( \bigoplus_{\substack {m=1\\ m \ \text{odd} }}^{u-1}  L(\omega_0) \otimes L\left((k+2-m) \omega_0+(m-1)\omega_1\right)\otimes W_{m, 1}^{(u, v)}\right) \oplus \\
&\qquad \oplus  \left( \bigoplus_{\substack {m=2\\ m \ \text{even} }}^{u-1}  L(\omega_1) \otimes L\left((k+2-m) \omega_0+(m-1)\omega_1\right)\otimes W_{m, 1}^{(u, v)}\right)\\
&\cong \bigoplus_{ {m=1 }}^{u-1}  \bigoplus_{\substack {m'=1 \\ m'\ \text{odd}}}^{w-1}  L\left((k+3-m') \omega_0+(m'-1)\omega_1\right)\otimes W_{m, m'}^{(u, w)}\otimes W_{m, 1}^{(u, v)}
\end{split}
\end{equation}
as $L_{k+2}(\mathfrak{sl}_2) \otimes \vir(u, w)\otimes \vir(u, v)$-module. 
In other words
\[
\mathcal A_{u; v, w} \cong \com\left(L_{k+2}(\mathfrak{sl}_2), \mathcal F(4)\otimes L_{k}(\mathfrak{sl}_2) \right)=: C_k
\]
as $\vir(u, w)\otimes \vir(u, v)$-modules. Cosets of this type are simple \cite[Corollary 2.2]{ACKL} since $L_{k+2}(\mathfrak{sl}_2)$ is rational. Simplicity also follows without using rationality from  \cite{Li, Car} since all three \VOSAs{} involved have positive energy and unique vacuum, since the weight-$1$ space of each is in the kernel of the Virasoro mode $L(1)$, and since 
  the unique non-degenerate supersymmetric invariant bilinear form on $\mathcal F(4) \otimes L(k\omega_0)$ restricts to  a non-degenerate supersymmetric invariant bilinear form on $L_{k+2}(\mathfrak{sl}_2) \otimes C_k$. See the following remark for a discussion of this issue.
\end{proof}

\begin{rema}
In many cases, simplicity of a \VOA{} is equivalent to existence of a non-degenerate symmetric invariant bilinear form. A precise statement is:

Let $V=\bigoplus_{n\in\ZZ}V_n$ be a \VOA{} such that $V_n=0$ for $n<0$ (positive energy) and $V_0=\CC{\bf 1}$ (unique vacuum) and $L(1)V_1=0$. Then 
there is a unique (up to scale) symmetric invariant bilinear form $B_V$ on $V$ \cite[Theorem 3.1]{Li}. 
Let $I\subsetneq V$ be an ideal; then $I\cap V_0=\{0\}$ and thus by \cite[Equation 3.1]{Li},
\[
B_V(u, v) = 0 \quad \textrm{for all}\ u\in I, \ v\in V.
\]
In other words, this bilinear form is non-degenerate if and only if $V$ is simple.

Let $V$ now be a \VOSA{} which is an order-two simple current extension of its even vertex operator subalgebra $V^\even$.
 That is, the odd part $V^\odd$ of $V$ is a self-dual simple current.  
Assume that $V^\even$ satisfies the same conditions as $V$ above, so that $V^\even$ is simple if and only if the unique invariant symmetric bilinear form on $V^\even$ is non-degenerate. 
This induces a supersymmetric invariant bilinear form on $V=V^\even\oplus V^\odd$ by \cite[Lemma 3.2.5]{Car} and the same argument as above carries over so that this induced bilinear form is non-degenerate if and only if $V$ is simple.
\end{rema}

The well-known fusion rules of the Virasoro minimal models
allow us to induce $W^{(u, v)}_{1, s} \otimes W^{(u, w)}_{1, t}$. Call $A$ the superalgebra object corresponding to either $C_k$ or $\mathcal A_{u; v, w}$.
\begin{corol} 
  $\cF(W^{(u, v)}_{1, s} \otimes W^{(u, w)}_{1, t})$ is in $\rep^0A$ if and only
  if $s+t\in 2\ZZ$.
\end{corol}
\begin{proof}
  We know from Lemma \ref{propo:rep0inductioncriteria} that a module $W$
  induces to a $\rep^0A$-object if and only if the monodromy
  $\cM_{A,W}=1_{A\fus{} W}$.  Now, using Virasoro fusion rules,
  \begin{align*}
    (W^{(u, v)}_{i, 1} \otimes W^{(u, w)}_{i, 1})
    \fus{}(W^{(u, v)}_{1, s} \otimes W^{(u, w)}_{1, t})
    &=(W^{(u, v)}_{i, 1}\fus{}W^{(u, v)}_{1,s}) \otimes (W^{(u, w)}_{i, 1}\fus{}
      W^{(u, w)}_{1, t})=W^{(u,v)}_{i,s}\otimes W^{(u,v)}_{i,t}.
  \end{align*}
 Using balancing and conformal dimensions,
  $\cM_{A,W}=1_{A\fus{}W}$ if and only if
    $h^{(u,v)}_{i,1}+h^{(u,w)}_{i,1}+h^{(u,v)}_{1,s}+h^{(u,w)}_{1,t}-
    h^{(u,v)}_{i,s}-h^{(u,w)}_{i,t}\in\ZZ$, for all $i$.
    This in turn happens if and only if $(i-1)(s+t)/2\in\ZZ$ for all $i$.
\end{proof} 
There are also extensions of $E_6$ and $E_8$ type:
\begin{corol}\label{cor:Etype}
Let $u=11$, $v=10$, and $w=12$, that is, $k+2=10$; 
then 
\[
C_k \oplus \cF(W^{(u, v)}_{1, 1} \otimes W^{(u, w)}_{1, 7})
\]
is a \VOSA{} extension of $C_k$. 

Let $u=29$, $v=28$, and $w=30$, that is, $k+2=28$; 
then 
\[
C_k \oplus \cF(W^{(u, v)}_{1, 1} \otimes W^{(u, w)}_{1, 11}) \oplus \cF(W^{(u, v)}_{1, 1} \otimes W^{(u, w)}_{1, 19}) \oplus \cF(W^{(u, v)}_{1, 1} \otimes W^{(u, w)}_{1, 29})
\]
is a \VOSA{} extension of $C_k$. 
  \end{corol}
\begin{proof}
The well-known type $E_6$, respectively $E_8$, extensions of $L_{10}\left(\mathfrak{sl}_2\right)$ and $L_{28}\left(\mathfrak{sl}_2\right)$ are given by \cite[Example 5.3]{KO}:
\[
L_{10}\left(\mathfrak{sl}_2\right) \oplus L\left( 4 \omega_0+6\omega_1\right), \qquad
\text{respectively} \qquad
L_{28}\left(\mathfrak{sl}_2\right) \oplus L\left( 18 \omega_0+10\omega_1\right) \oplus L\left( 10 \omega_0+18\omega_1\right) \oplus L\left( 28 \omega_1\right).
\]
The corresponding statement for the even vertex  operator subalgebra of $C_k$ then follows by mirror extension for rational \VOAs{} \cite{Lin, CKM}. The odd part of $C_k$ is a simple current of order two for the even vertex operator subalgebra of $C_k$. It induces to a simple current of the extension of the even vertex operator subalgebra of $C_k$ and the corresponding simple current exensions are the claimed \VOSA{} extensions of $C_k$.\end{proof}

Another extension of $L_k(\mathfrak{sl}_2)$ is the following. 
Let $k=-2+\frac{1}{p}$ for a positive integer $p$. This is not an admissible level for $\mathfrak{sl}_2$ and in that instance $L_k(\mathfrak{sl}_2)$ has a simple \VOA{} extension \cite{Cr}
\[
\mathcal Y_p  = \bigoplus_{m =0}^\infty  (2m+1)L\left( (k-2m) \omega_0+2m\omega_1\right). 
\]
In fact the group $SU(2)$ acts via automorphisms and 
\[
\mathcal Y_p  \cong \bigoplus_{m =0}^\infty  \rho_{2m\omega_1} \otimes L\left( (k-2m) \omega_0+2m\omega_1\right). 
\]
as $SU(2) \otimes L_k(\mathfrak{sl}_2)$-module \cite{ACGY}. 
It is an interesting question to identify and understand the Virasoro \VOA{} extension Com$\left( V_{k+1}(\mathfrak{sl}_2), L_1(\mathfrak{sl}_2) \otimes \mathcal Y_p \right)$. A natural expectation is that this coset is closely related to the triplet algebra $\cW(p+1)$.

\begin{rema}
Another interesting observation is that $L_{k+2}(\mathfrak{sl}_2) \subseteq L_{k}(\mathfrak{sl}_2)\otimes \mathcal F(3)$ so that $\mathcal F(1)\subseteq C_k$. Let $D_k=\com\left(\mathcal F(1), C_k\right)$. 
Then \cite[Example 7.8]{CL1} suggests that the even vertex operator subalgebra of $D_k$ is strongly generated by three fields of conformal dimension $2$, $4$, and $6$. There are two known families of vertex algebras of this type, namely 
the principal $W$-algebra of type $B_3$ and the $\mathbb Z_2$-orbifold of the $N=1$ super Virasoro algebra and indeed this coset coincides with the simple 
 $\mathbb Z_2$-orbifold of the $N=1$ super Virasoro algebra at corresponding central charge \cite[Remark 5.3]{CFL}.
\end{rema}

\subsection{Verlinde formulae}\label{sec:Verlinde}

One of the high points of the theory  of regular \VOA s is the Verlinde formula, conjectured by Verlinde \cite{V} and proven by Huang \cite{H-verlinde}, for the fusion rules of the \VOA{} in terms of the modular $S$-matrix. We start by recalling this formula:

Let $V$ be a regular \VOA{}, $X$ a $V$-module, and $[\rep V]$ the set of equivalence classes of simple $V$-modules. The modular $S$-matrix is defined by the modular $S$-transformation of torus one-point functions. For this
let $L(0)$ be the Virasoro
zero-mode of $V$ and $c$ the central charge; then the character of $X$ is defined to be the graded trace
\[
  \ch[X](\tau, v) := \text{tr}_X\left(q^{L(0)-\frac{c}{24}} o(v)\right), \qquad q=e^{2\pi i\tau},
\]
for $v\in V$ and $o(v)$ the zero-mode of the
corresponding vertex operator. Including zero-modes $o(v)$ ensures that characters of inequivalent simple $V$-modules are linearly independent. Thus the $\CC$-linear span of characters  is isomorphic to the vector space $\CC[\rep V]$
whose basis elements are labeled by the elements of $[\rep V]$.  We
fix one representative of each equivalence class.  Abusing notation slightly, we denote this set of
representatives by $[\rep V]$ as well.  Then M\"{o}bius
transformations on the linear span of characters define an action of SL$(2,\ZZ)$. If $v\in V$ has
conformal weight $\mathrm{wt}[v]$ with respect to Zhu's modified
vertex operator algebra structure on $V$ \cite{Zhu}, the modular
$S$-transformation has the form
\begin{equation}
\begin{split}
  \ch[X]\left(-\frac{1}{\tau}, v\right)
  &= \tau^{\text{wt}[v]} \sum_{Y\in [\rep V]} S_{X, Y}  \ch[Y](\tau, v). \end{split}
\end{equation}
Recall that the fusion rules are defined as the 
dimensions  of spaces of intertwining
operators. Verlinde's formula for $\ZZ$-graded \VOAs{} is then
\begin{equation}
  \begin{split}
    {N}_{X,Y}^{W}&=  \sum\limits_{Z \in  [\rep V]}
    \dfrac{{S}_{X,Z}\cdot {S}_{Y,Z}\cdot ({S}^{-1})_{Z,W}}{{S}_{V, Z}},
\end{split}
\end{equation}
and the fusion product satisfies 
\[
  X \boxtimes_V Y \cong \bigoplus_{W\in[\rep V]} {N}_{X,Y}^{W}\cdot W.
\]

In this subsection, we will get Verlinde formulae for regular \VOSA s of correct and wrong boson-fermion statistics, as well as for regular \VOA s with wrong statistics. The main idea is to use the induction functor to show
and then exploit the correspondence between simple modules for $V$ and for
$A=V\oplus\sic$, where $A$ is a simple current extension of order 2.

\subsubsection{Preliminaries}
We begin by discussing the notions of simplicity of objects in $\repA$ and $\underline{\repA}$.
Recall from \cref{SCabelian} that $\sC$ is abelian and 
from \cref{rem:RepAnotab} that $\repA$ is not necessarily abelian (although its underlying
category $\underline{\repA}$ is abelian). For this reason, some care is needed in dealing with the notion of simple object in $\repA$. It is useful to define simplicity in $\repA$ and $\underline{\repA}$ in terms of submodules:
\begin{defi}
	An \textit{$A$-submodule} of an object $(X,\mu_X)$ in $\repA$ is an object $W$ of $\sC$ together with an injection $h:W\rightarrow X$ such that $\im(\mu_X\circ (\id_A\boxtimes h))\subseteq\im(h)$. We say that two $A$-submodules $(W,h)$ and $(\widetilde{W},\widetilde{h})$ are \textit{equivalent} if there is an $\sC$-isomorphism $k: W\rightarrow\widetilde{W}$ such that $\widetilde{h}\circ k =h$. 
\end{defi}

\begin{rema}
	For a vertex operator superalgebra extension $V\subseteq A$ with $V$ contained in the even part of $A$, an equivalence class of $A$-submodules of an object $(X,\mu_X)$ in $\repA$ amounts to a not-necessarily-$\ZZ/2\ZZ$-graded subspace of $X$ that is closed under the vertex operator $Y_X$ for $A$ acting on $X$.
\end{rema}

\begin{defi}\label{def:simple}
	Supppose $A$ is a vertex operator superalgebra extension of $V$ such that $V\subseteq A^\even$. An object $(X,\mu_X)$ is \textit{simple} as an object of $\repA$, respectively as an object of $\underline{\repA}$, if $0$ and $X$ are the only $A$-submodules, respectively the only $\ZZ/2\ZZ$-graded $A$-submodules, of $X$ up to equivalence.
\end{defi}

Recall that sometimes an object in a category is defined to be simple if all its endomorphisms are scalar multiples of the identity. This notion of simplicity is equivalent to Definition \ref{def:simple} if $\underline{\repA}$ is semisimple, that is, if every object of $\underline{\repA}$ is a direct sum of finitely many simple objects in the sense of Definition \ref{def:simple}:
\begin{propo}\label{propo:twosimpledefequiv}
	Suppose $A$ is a vertex operator superalgebra extension of $V$ such that $V\subseteq A^\even$. If $(X,\mu_X)$ is simple as an object of $\repA$, respectively as an object of $\underline{\repA}$, then $\mathrm{End}_{\repA}(X)=\CC 1_X$, respectively $\mathrm{End}_{\underline{\repA}}(X)=\CC 1_X$. The converse holds if $\underline{\repA}$ is semisimple.
\end{propo}
\begin{proof}
	Suppose $(X,\mu_X)$ is simple as an object of $\underline{\repA}$ and $f\in\mathrm{End}_{\underline{\repA}}(X)$. Since $f$ preserves the finite-dimensional weight spaces of $X$, $f$ has an eigenvalue $\lambda$. Since $f$ is even, the corresponding eigenspace $X_\lambda$ is a non-zero $\ZZ/2\ZZ$-graded $A$-submodule of $X$. Since $X$ is a simple object of $\underline{\repA}$, $X_\lambda=X$ and $f=\lambda 1_X$.
	
	Now suppose $(X,\mu_X)$ is simple as an object of $\repA$ and $f\in\mathrm{End}_{\repA}(X)$ with parity decomposition $f=f^\even+f^\odd$.  Both $f^\even$ and $f^\odd$ are $\repA$-endomorphisms of $X$; we need to show that $f^\even$ is a scalar multiple of $1_X$ and $f^\odd=0$. Now, $f^\even$ is a scalar multiple of $1_X$ by the argument of the previous paragraph. As for $f^\odd$, suppose $\lambda$ is an eigenvalue of $f^\odd$, which exists because $f^\odd$ preserves the finite-dimensional weight spaces of $X$ due to the evenness of $L(0)$, and let $X_\lambda$ be the corresponding eigenspace. We claim that 
	\begin{equation*}
		W=\lbrace w+iP_X(w)\,\vert\,w\in X_\lambda\rbrace
	\end{equation*}
	is an $A$-submodule of $X$, where $P_X$ is the parity involution of $X$. Indeed, for $a\in A$ with parity decomposition $a=a^\even+a^\odd$, we have
	\begin{align*}
		Y_X(a,x)(w+i P_X(w)) & =\left(Y_X(a^\even,x)w+iY_X(a^\odd,x)P_X(w)\right) + \left(iY_X(a^\even,x)P_X(w)+Y_X(a^\odd,x)w\right)\nonumber\\
		& =\left(Y_X(a^\even,x)w+iY_X(a^\odd,x)P_X(w)\right) +iP_X\left(Y_X(a^\even,x)w-iP_X(Y_X(a^\odd,x)w)\right)\nonumber\\
		& =\left(Y_X(a^\even,x)w+iY_X(a^\odd,x)P_X(w)\right) + iP_X\left(Y_X(a^\even,x)w+iY_X(a^\odd,x)P_X(w)\right)
	\end{align*}
	for $w\in X_\lambda$, and the coefficients of powers of $x$ and $\log x$ in $Y_X(a^\even,x)w+iY_X(a^\odd,x)P_X(w)$ lie in $X_\lambda$ because
	\begin{align*}
		f^\odd\left(Y_X(a^\even,x)w+iY_X(a^\odd,x)P_X(w)\right) & = Y_X(a^\even,x)f^\odd(w)-iY_X(a^\odd,x)(f^\odd\circ P_X)(w)\nonumber\\
		& = Y_X(a^\even,x)f^\odd(w)+i Y_X(a^\odd,x)(P_X\circ f^\odd)(w)\nonumber\\
		& =\lambda\left(Y_X(a^\even,x)w+iY_X(a^\odd,x)P_X(w)\right).
	\end{align*}
	Because $X$ is simple as an object of $\repA$, either $W=0$ or $W=X$. If $W=0$, then $w=-iP_X(w)$ for all $w\in X_\lambda$, so that $X_\lambda=P_X(X_\lambda)$. If on the other hand $W=X$, then for any non-zero $w\in X_\lambda$, there is some non-zero $w'\in X_\lambda$ such that $w=w'+iP_X(w')$; thus $P_X(w')=-i(w-w')$ and $X_\lambda\cap P_X(X_\lambda)\neq 0$. Either way, $X_\lambda\cap P_X(X_\lambda)\neq 0$; but it is easy to see that $P_X(X_\lambda)$ is the eigenspace of $f^\odd$ with eigenvalue $-\lambda$. Thus we must have $\lambda=0$. In particular, $\ker f^\odd$ is non-zero; $\ker f^\odd$ is also an $A$-submodule because $f^\odd$ is parity-homogeneous. Thus since $X$ is simple, $\ker f^\odd=X$ and $f^\odd=0$, as desired. 
	
	For the converse, assume $\underline{\repA}$ is semisimple. If $(X,\mu_X)$ is an object of $\repA$ such that $\mathrm{End}_{\underline{\repA}}=\CC 1_X$, suppose that $W$ is a $\ZZ/2\ZZ$-graded $A$-submodule of $X$. Then because $\underline{\repA}$ is semisimple (with simple objects defined as in Definition \ref{def:simple}), there is a $\ZZ/2\ZZ$-graded $A$-submodule $\widetilde{W}$ of $X$ such that $X=W\oplus\widetilde{W}$. Then the projection $q_W\circ p_W$ with respect to this direct sum decomposition is a $\underline{\repA}$-endomorphism of $X$, which must be a scalar multiple of $1_X$. This can only happen if $W=0$ or $W=X$, so $X$ is simple as an object of $\underline{\repA}$.
	
	Now suppose $\mathrm{End}_{\repA}=\CC 1_X$, and suppose $W\subseteq X$ is an $A$-submodule. It is easy to see that $P_X(W)$ is also an $A$-submodule. Now, the $A$-submodules $W+P_X(W)$ and $W\cap P_X(W)$ are $P_X$-stable, which means that they are $\ZZ/2\ZZ$-graded. Since by assumption $\mathrm{End}_{\repA}=\CC 1_X=\mathrm{End}_{\underline{\repA}}(X)$, the previous paragraph shows that $X$ is simple in $\underline{\repA}$. Thus $X+P_X(W)$ and $W\cap P_X(W)$ are both either $0$ or $X$. If $W+P_X(W)=0$, then $W=0$. If on the other hand $W\cap P_X(W)=X$, then $W=X$. The other possibility is $W+P_X(W)=X$ and $W\cap P_X(W)=0$. In this case, $X=W\oplus P_X(W)$ as an object of $\repA$ with neither $W$ nor $P_X(W)$ equal to $0$. But this is a contradiction since then projection onto either $W$ or $P_X(W)$ is a non-scalar $\repA$-endomorphism of $X$. Thus either $W=0$ or $W=X$, and $X$ is simple as an object of $\repA$.
\end{proof}

It is clear that simplicity in $\repA$ is a stronger condition than simplicity in $\underline{\repA}$. The next two propositions, which are {independent of the statistics of $A$}, characterize the discrepancy between these two conditions; they can be proved
similarly to \cite{Wak, Ya}, but we
reproduce the proofs for completeness.  

\begin{propo}\label{propo:fRepAsimplicity}
	Suppose $A$ is a vertex operator superalgebra extension of $V$ such that $V\subseteq A^\even$, and suppose $X=X^\even\oplus X^\odd$ is a simple object of $\underline{\repA}$.
	Then either
	$X$ is simple as an object of $\repA$ or there exist
	an $A^\zero$-isomorphism $\varphi:X^\zero\xrightarrow{\sim}X^\one$ and $A$-submodules $W_\pm\subseteq X$ such that
	\begin{align}
		X &= W_+ \oplus W_- \quad as\,\,A^\even\text{-modules},\,\,and\,\,hence\,\,as\,\,V\text{-modules},\nonumber\\
		W_\pm &=\{ x^\zero \pm \varphi(x^\zero)\,|\, x^\zero\in X^\zero\}\label{eqn:Ypm},\\
		X^\zero&\cong X^\one\cong W_+\cong W_- \quad as\,\,
		A^\zero\text{-modules},\,\,and\,\,hence\,\,as\,\,V\text{-modules}
		\nonumber.
	\end{align}
\end{propo}
\begin{proof}
	We may assume $X$ is not simple as an object of $\repA$ so that there is a non-zero, proper $A$-submodule $W_+\subseteq X$. As in the proof of the preceding proposition, $W_-=P_X(W_+)$ is an $A$-submodule, and $W_++W_-$ and $W_+\cap W_-$ are $\ZZ/2\ZZ$-graded $A$-submodules. Since $X$ is simple as an object of $\underline{\repA}$, we have $W_++W_-=X$ and $W_+\cap W_-=0$, so that $X=W_+\oplus W_-$ as $V$-modules.
	
	Now any $x^\zero\in X^\zero$ can be written uniquely as
	$x^\zero = w_++w_-$ for some $w_\pm\in W_\pm$. Since $P_X(x^\even)=x^\even$, the uniqueness implies that $w_-=P_X (w_+)$ and so
	$x^\zero = w_++P_X (w_+) = 2w_+^\zero$. Similarly, any $x^\odd\in X^\odd$ can be written uniquely as $x^\odd=w_+-P_X(w_+)=2w_+^\odd$.  Therefore, there is a unique
	linear isomorphism $\varphi: X^\zero\xrightarrow{\sim} X^\one$ satisfying	\begin{align}\label{eqn:basicvarphi}x^\zero+\varphi(x^\zero)\in W_+\end{align}
	given by $\varphi(w_+^\zero) = w_+^\one$; that is, $\varphi$ maps $x^\even=w_++P_X(w_+)$ to $x^\odd=w_+-P_X(w_+)$.
	It is now clear that \eqref{eqn:Ypm} is satisfied.
	We just need to
	prove that $\varphi$ is a morphism of $A^\zero$-modules.  Since
	$x^\zero+\varphi(x^\zero)\in W_+$,
	\begin{align*}
		Y_X(a^\zero,x)(x^\zero+\varphi(x^\zero))&= Y_X(a^\zero,x)x^\zero + Y_X(a^\zero,x)\varphi(x^\zero)
		\in W_+[\log x]\{x\},\\
		Y_X(a^\zero,x)x^\zero &\in X^\zero[\log x]\{x\},\quad  Y_X(a^\zero,x)\varphi(x^\zero) \in X^\one[\log x]\{x\}.
	\end{align*}
	for any $a^\zero\in A^\zero$. These parity considerations together with \eqref{eqn:basicvarphi}
	now show that 
	\begin{align}
		\varphi(Y_X(a^\zero,x)x^\zero) = Y_X(a^\zero,x)\varphi(x^\zero).\label{eqn:phiY}
	\end{align}
	Now
	the last assertion of \eqref{eqn:Ypm} is clear.
\end{proof}

\begin{propo}\label{propo:supersimpleA}
	Suppose $A$ is a \VOSA{} extension of $V$ such that
	$V\subseteq A^\zero$. Then $A$ is simple as an object of $\repA$ if and
	only if it is simple as an object of $\underline{\repA}$.	
\end{propo}
\begin{proof}
	We only need to prove the ``if'' direction as the ``only if'' direction is clear. 
	Assume towards a contradiction that $A$ is simple in $\underline{\repA}$
	but not in $\repA$. Let $\varphi: A^\zero\rightarrow A^\one$ be the $A^\zero$-isomorphism of  Proposition \ref{propo:fRepAsimplicity}.
	Then for any $a^\zero, b^\zero\in A^\zero$, $a^\one\in A^\one$,
	the first equation below is \eqref{eqn:phiY} and the second follows
	similarly:
	\begin{align}
		\varphi(Y(a^\zero,x)b^\zero) = Y(a^\zero,x)\varphi(b^\zero),
		\label{eqn:phi00}\\
		\varphi(Y(a^\one,x)\varphi(b^\zero)) = Y(a^\one,x)b^\zero.
	\end{align}
	Now we get
	\begin{align*}
		\varphi(Y(\varphi(a^\zero),x)\varphi(b^\zero))
		&=Y(\varphi(a^\zero),x)b^\zero
		=e^{xL(-1)}Y(b^\zero,-x)\varphi(a^\zero)
		=e^{xL(-1)}\varphi(Y(b^\zero,-x)a^\zero)\\
		&=\varphi(e^{xL(-1)}Y(b^\zero,-x)a^\zero)
		=\varphi(Y(a^\zero,x)b^\zero),
	\end{align*}
	which implies
	\begin{align*}
		Y(\varphi(a^\zero),x)\varphi(b^\zero)=
		Y(a^\zero,x)b^\zero
	\end{align*}
	since $\varphi$ is an isomorphism.
	However, parity considerations show 
	\begin{align*}
		Y(\varphi(a^\zero),x)\varphi(b^\zero)
		=-e^{xL(-1)}Y(\varphi(b^\zero),-x)\varphi(a^\zero)
		=-e^{xL(-1)}Y(b^\zero,-x)a^\zero
		=-Y(a^\zero,x)b^\zero,
	\end{align*}
	implying that $Y(a^\zero,x)b^\zero$ is identically $0$,
	contradicting the properties of the vacuum.
\end{proof}

In  light of the previous proposition, we omit specifying the category
when we say that a \VOSA{} extension is simple.

\begin{rema}
	For Propositions \ref{propo:twosimpledefequiv}, \ref{propo:fRepAsimplicity}, and \ref{propo:supersimpleA}, we do not need to assume that $A$ is an object of a vertex tensor category of $V$-modules, since the categories $\repA$ and $\underline{\repA}$ can be defined using Proposition \ref{propo:YWcomplex}. We just need to assume the convergence of the relevant compositions of $V$-intertwining operators.
\end{rema}

\begin{propo}\label{propo:X0simple}
	Suppose $A=V\oplus J$ is a simple \VOSA{} and $A$ is an object in a vertex tensor category $\cC$ of $V$-modules.
	Then $V$ is simple and $\sic$ is a simple current
	for $V=A^\zero$.  Moreover, if $X=X^\zero\oplus X^\one$ is simple in $\underline{\repA}$ then
	$X^\zero$ is simple as a $V$-module and $X^\one\cong \sic\boxtimes X^\zero$.
\end{propo}
\begin{proof}
	The conclusions that $V$ is simple and $J$ is a simple current follow from variants of \cite{DM} and \cite{Miy, CarM}. The simplicity of $V$ and $J$ is also a special case of the last conclusion of the proposition. For the simple current property of $J$, see also \cite[Theorem 3.1 and Appendix A]{CKLR}.

	Now, if $W^\zero$ is a proper $V$-submodule of $X^\zero$, then 
	\[W^\zero\oplus \mathrm{Span}\{j_{n;k} w\,|\, j\in A^\one=\sic, n\in\RR, k\in\NN, w\in W^\zero   \}  \]	
	is a proper $\ZZ/2\ZZ$-graded $A$-submodule of $X$ by Lemma \ref{lem:genmodspan}. Since $X$ is simple in $\underline{\repA}$, this means $W^\even=0$ and $X^\even$ is a simple $V$-module. Similarly, $X^\odd$ is a simple $V$-module. 
	
	Next, to show that $X^\odd\cong J\boxtimes X^\even$, we may assume that $X$ is non-zero. We claim that
	\begin{equation*}
		\textrm{Span}\lbrace j_{n; k} x^\even\,\vert\,j\in J, n\in\RR,  k\in\NN, x^\even\in X^\even\rbrace
	\end{equation*}
	is non-zero. Indeed, if not, then $J\subseteq\mathrm{Ann}_{A}(X^\even)$, and this ideal (recall Lemma \ref{lem:repAannisideal}) is then $A$ since $A$ is simple. In particular, $\unit\in\mathrm{Ann}_A(X^\even)$ implies $X^\even=0$. Then since $j_{n;k} x^\odd\in X^\even$ for $j\in J$, $n\in\RR$, $k\in\NN$, and $x^\odd\in X^\odd$, the same argument shows $\mathrm{Ann}_A(X^\odd)=A$ and $X^\odd=0$. This contradicts $X\neq 0$, proving the claim. Now the $V$-intertwining operator $Y_X\vert_{J\otimes X^\even}$ induces a non-zero homomorphism $J\boxtimes X^\even\rightarrow X^\odd$. Since $\sic$ is a simple current and $X^\even$, $X^\odd$ are simple,
	we have $\sic\fus{}X^\zero\cong X^\one$ by Schur's Lemma.
\end{proof}  

\begin{corol}\label{corol:Xsimpleisoinduced}
	Suppose $A=V\oplus J$ is a simple vertex operator superalgebra and $A$ is an object in a vertex tensor category $\cC$ of $V$-modules. If $X=X^\even\oplus X^\odd$ is simple in $\underline{\repA}$, then $X\cong\cF(X^\even)$ in $\repA$.
\end{corol}
\begin{proof}
	The preceding proposition and its proof show that
	\begin{equation*}
		\mu_X\vert_{A\boxtimes X^\even}=\mu_X\circ(1_A\boxtimes q_{X^\even}): \cF(X^\even)=(V\boxtimes X^\even)\oplus(J\boxtimes X^\even)\rightarrow X=X^\even\oplus X^\odd
	\end{equation*}
	is a $V$-module isomorphism. But it is easy to see from the associativity of $\mu_X$ that $\mu_X\vert_{A\boxtimes X^\even}$ is a morphism in $\repA$. Thus it is an isomorphism in $\repA$.
\end{proof}

In the other direction, we can exploit the $\ZZ/2\ZZ$ grading to conclude
that simple objects induce to simple objects, provided $\underline{\repA}$ is semisimple:
\begin{corol}\label{corol:inducedsimple}
	Suppose $A=V\oplus J$ is a simple \VOSA{} such that $A$ is an object in a vertex tensor category $\cC$ of $V$-modules and $\underline{\repA}$ is semisimple. Then simple modules $X^\zero$ in $\sC$
	induce to simple modules in $\underline{\repA}$.
\end{corol}  
\begin{proof}
	Since $\underline{\repA}$ is semisimple, it is enough to show that if $W=W^\even\oplus W^\odd$ is a non-zero $\underline{\repA}$-simple $A$-submodule, then $W=\cF(X^\even)=X^\even\oplus(J\boxtimes X^\even)$. Since $X^\even$ is simple and $J$ is a simple current by Proposition \ref{propo:X0simple}, $J\boxtimes X^\even$ is simple as well. Thus it is enough to show that $W^\even$ and $W^\odd$ are both non-zero. Since $W$ is simple in $\underline{\repA}$, Proposition \ref{propo:X0simple} implies $W^\even=0$ if and only if $W^\odd\cong J\boxtimes W^\even=0$ as well. Therefore since $W$ is non-zero, the conclusion follows.
\end{proof}  

We shall show below that if $A$ has wrong statistics, that is, if $A$ is $\ZZ$-graded, then every simple object of $\underline{\repA}$ is also simple in $\repA$.
Thus to derive Verlinde formulae in these cases, we can enumerate simple objects using either definition of simplicity. For \VOSAs{} of correct statistics, however,
we shall need to work with simple objects in $\underline{\repA}$.

\begin{nota}
  Henceforth, in this subsection on Verlinde formulae, $\cC$ will be a
  modular tensor category of modules for a regular
  vertex operator algebra $V$.  In particular, we have the 
  modular $S$-matrix $S^\chi$ arising from characters of irreducible $V$-modules and the categorical $S$-matrix $S^\hopflink$ arising from traces of monodromies for irreducible $V$-modules. The two $S$-matrices are equal up
  to a non-zero constant depending solely on $\cC$.  Therefore, the
  lemmas in this section will be deduced for $S^\hopflink$ but will
  hold, with suitable modifications, for $S^\chi$ as well. Thus for this section, we drop the
  superscripts and use just $S$.

  Unlike in previous sections where we used $W_i$ to denote modules for various \VOAs{}, 
  here for simplicity of notation we shall use capital letters $U,W,X,\dots$ to denote modules.

  We shall fix a simple (super)algebra $A$ in $\cC$ that is an
  extension of  $V$ by a simple current $\sic$ of order $2$. We
  shall assume that $\cM_{J,J}=1_V$. We denote
  $\qdim(\sic)=d$; note that $d=\pm 1$.  We denote the set of isomorphism  classes of simple objects in $\cC$ by $\cS$.  For a simple object
  $X\in\cS$, we define $\epsilon_X$ via $\cM_{X,\sic}=\epsilon_X 1_{X\boxtimes J}$.  Since
  $J\boxtimes J\cong V$, $\epsilon_X=\pm 1$ for any simple object
  $X\in\cS$ \cite{CKL}.  Since $\cM_{J,J}=1_V$, it is clear that
  $\epsilon_{X\boxtimes\sic}=\epsilon_X$  \cite{CKL}.
\end{nota}

We note a general lemma which we shall use below.
\begin{lemma}\label{lem:Sinverse}
  Let $\cC$ be a modular tensor category with unit object $\unit$ and let $\sic$ be a simple
  current with $\qdim(\sic)=d$. For a simple object $X$, define
  $\epsilon_X$ by $\cM_{X,J}=\epsilon_X 1_{X\boxtimes J}$.
  Then for any simple objects $W$, $X$ and $i,j,t\in\NN$,
  \begin{align}
    (S^{\pm 1})_{W\boxtimes \sic^t,X} = (\epsilon_X d)^{\pm t}(S^{\pm 1})_{W,X}.
  \end{align}
  In particular, if $\sic$ has order $2$ and
  $\cM_{\sic,\sic}=1_{\unit}$ then $d,\epsilon_W, \epsilon_X =\pm 1$ and
  \begin{align}
    (S^{\pm 1})_{W\boxtimes \sic^i,X \boxtimes \sic^j}
    = d^{i+j}(\epsilon_W^j\epsilon_X^i)(S^{\pm 1})_{W,X}.
    \label{eqn:orbitfoldingSinv}
  \end{align}
\end{lemma}
\begin{proof}
First we have the identity
\[
S_{\sic, X} = \tr\left( \cM_{\sic, X}\right) = \epsilon_X S_{\unit, \sic\boxtimes X}=  
\epsilon_X \dfrac{S_{\unit, \sic} S_{\unit, X}}{S_{\unit, \unit}} =\epsilon_X d\, S_{\unit, X}.
\]
  Note from \cite{BK} that in ribbon categories,
  $(W\boxtimes \sic)^*\cong \sic^*\boxtimes X^*$, where $^*$ denotes the
  dual. Also, $(S^{-1})_{X,W} = \alpha S_{X,W^*}$, where
  $\alpha$ is a global non-zero constant depending only on $\cC$; this
  follows from formulas (3.2.d) and (3.2.h) in \cite{T}.  Therefore,
  for simple objects $W, X\in\cS$,
  \begin{equation} \nonumber
  \begin{split}
    (S^{-1})_{W\boxtimes \sic^t,X}
    &= (S^{-1})_{X,W\boxtimes \sic^t}=
       \alpha S_{X,(W\boxtimes \sic^t)^*} = \alpha S_{X,\sic^{-t}\boxtimes W^*}
       =\alpha \left(\frac{S_{X,\sic^{-1}}}{S_{X,\unit}}\right)^t S_{X, W^*} \\
     &= \alpha \left(\frac{S_{X,\sic}}{S_{X,\unit}}\right)^{-t} S_{X, W^*}
       =\alpha (\epsilon_X d)^{-t} S_{X,W^*}
       =(\epsilon_X d)^{-t}(S^{-1})_{X,W}
     =(\epsilon_X d)^{-t}(S^{-1})_{W,X}.
  \end{split}
  \end{equation}
  The other case is
  \[
  S_{W\boxtimes \sic^t,X} = S_{\unit,X} \frac{S_{W\boxtimes \sic^t,X}}{ S_{\unit,X}} = S_{\unit,X} \frac{S_{W,X}}{ S_{\unit,X}} \left(\frac{S_{\sic,X}}{ S_{\unit,X}}\right)^t = S_{W,X} \left(\epsilon_X d\right)^t.
  \]
\end{proof}

\subsubsection{Superalgebras with wrong statistics}
Here we assume that $A=V\oplus\sic$ is a simple, regular vertex operator
superalgebra such that $J\boxtimes J\cong V$,
$\theta_A=1_A$, $\qdim(\sic)=d=-1$ and
$\cM_{\sic,\sic}=1_{V}$. Examples include affine Lie superalgebra vertex operator superalgebras, viewed as superalgebras in the
representation categories of their even subalgebras, provided the even subalgebra is regular.  Since $A$ is a
superalgebra, every object in $\repA$ by definition has a $\ZZ/2\ZZ$-grading. In the following lemma, we collect basic facts regarding simple objects in $\repA$, most of which are specializations of previous results: 
\begin{lemma}\label{lem:superwrongstatsbasic}
We have the following:
\begin{enumerate}
\item For every simple object $X\in\cS$, $J\boxtimes X\not\cong X$. 
\label{it:superwrongstatsbasic1}
\item For every simple object $X=X^\zero\oplus X^\one$  in $\underline{\rep A}$, $X^\zero$ and
  $X^\one$ are simple in $\cC$, and moreover, $X\cong \cF(X^\zero)$, that is,
  $X^\one\cong \sic\boxtimes X^\zero$. 
\label{it:superwrongstatsbasic2}
\item If $X$ is simple in $\underline{\repA}$ then $X$ is simple in $\repA$.
\label{it:superwrongstatsbasic2.5}
\item Given a simple object $X^\zero\in \cS$, $\cF(X^\zero)$ is simple in $\repA$.
\label{it:superwrongstatsbasic3}
\item Given a simple object $X=X^\zero\oplus X^\one$ in $\rep A$, 
  $X$ is an object in $\repzA$ if and only if $\epsilon_{X^\zero}=1$.
  \label{it:superwrongstatsbasic4}
\item Given a simple object $X=X^\zero\oplus X^\one$ in $\rep A$, we have
  $S^\hopflink_{\sic,X^\zero}=-\epsilon_{X^\zero}\qdim(X^\zero)$.
  \label{it:superwrongstatsbasic5}
\item Given simple objects $X=X^\zero\oplus X^\one$, $W=W^\zero\oplus W^\one$ in $\repA$ and $i,j\in\ztwo$, we
  have
  $(S^{\pm
    1})_{X^i,W^j}=(-1)^{i+j}\epsilon_{X^\zero}^j\epsilon_{W^\zero}^i(S^{\pm
    1})_{X^\zero,W^\zero}$.
\label{it:superwrongstatsbasic6}
\end{enumerate}
\end{lemma}
\begin{proof}
%
For \eqref{it:superwrongstatsbasic1}, the categorical dimension of the simple $V$-module $X$ is non-zero since $\cC$ is modular. Now, 
$\qdim(J\boxtimes X) = \qdim(J)\qdim(X) = -\qdim(X)$.
\eqref{it:superwrongstatsbasic2} is Proposition \ref{propo:X0simple} and Corollary \ref{corol:Xsimpleisoinduced} specialized to this situation. 
\cref{propo:fRepAsimplicity}, (\ref{it:superwrongstatsbasic1}),
and (\ref{it:superwrongstatsbasic2}) imply (\ref{it:superwrongstatsbasic2.5}).
Proposition \ref{prop:Fsimple} and \eqref{it:superwrongstatsbasic1} (or alternatively Corollary \ref{corol:inducedsimple} and (\ref{it:superwrongstatsbasic2.5})) imply \eqref{it:superwrongstatsbasic3}.
 Proposition \ref{propo:rep0inductioncriteria} and (\ref{it:superwrongstatsbasic2})  imply \eqref{it:superwrongstatsbasic4}.
\eqref{it:superwrongstatsbasic5} is clear from the definition of $S^\hopflink$
and $\qdim(\sic)=-1$. Finally, \eqref{it:superwrongstatsbasic6} is Lemma \ref{lem:Sinverse} applied to this case. 
\end{proof}  

\begin{nota}
  Given $X=X^\zero\oplus X^\one$ in $\rep A$, we let
  $\ch^\pm[X]=\ch[X^\zero]\pm \ch[X^\one]$. We refer to $\ch^-$ as the supercharacter.  If $X$ is simple, we say that $X$ is untwisted if
  $\epsilon_{X^\even}=1$ and twisted if $\epsilon_{X^\even}=-1$. 
  By  Lemma \ref{lem:superwrongstatsbasic}\eqref{it:superwrongstatsbasic4}, untwisted objects are precisely  the simple $A$-modules in $\cC$; twisted modules, on the other hand, are twisted with respect to the parity automorphism of $A$.

  We also let
  $\sim$ be the equivalence relation on $\cS$ given by
  $X\sim X\boxtimes J$.  The equivalence classes $\cS/\sim$ correspond
  to simple objects in $\repA$, with a fixed choice of
  representative in each class determining the parity. We fix such a
  choice now.  The corresponding set of representatives can be identified with $[\repA]$ 
  in the Introduction.
  These representatives will be denoted by a superscript
  $\zero$, such as $X^\zero$, and correspondingly, we will let
  $\cS = \cS^\zero \cup \cS^\one$, where the union is disjoint thanks to
  \eqref{it:superwrongstatsbasic1} above.
	
	As in the Introduction, we let $[\repzA], [\reptw]\subseteq [\repA]$ be the representatives of simple
	untwisted and twisted objects, respectively.
\end{nota}	

\begin{defi}\label{defi:adim}
	Let $A=V\oplus J$ be a \VOSA{} such that the representation category of its even part $V$ is a modular tensor category, and 
	let $X$ be a simple $A$-module. 
	We define
	\[
	\adim^\pm[X]:= \lim_{iy\rightarrow 0^+} \frac{\text{ch}^\pm[X](iy)}{\text{ch}^\pm[A](iy)}
	= \lim_{iy\rightarrow i\infty} \frac{\text{ch}^\pm[X](-1/iy)}{\text{ch}^\pm[A](-1/iy)}, 
	\]
	whenever the limit exists. 
	These asymptotic dimensions were analyzed and used extensively in \cite{DW3, DJX}.
	In those papers, asymptotic dimensions were denoted $\qdim$, but we avoid this notation
	since for us $\qdim$ is the categorical dimension, that is, the trace of the identity morphism.
\end{defi}

The calculation for the modular $S$-transformation of (super)characters is as
follows.  Let $X=X^\zero\oplus X^\one$ be a simple object in $\repA$
and let $v\in V$ be homogeneous of weight $\mathrm{wt}[v]$ with respect to Zhu's modified vertex operator algebra structure on $V$ \cite{Zhu}:
\begin{align*}
  (\ch^\pm[X])(-1/\tau,v)
  &=\tau^{\mathrm{wt}[v]}\sum_{W\in\cS}(S_{X^\zero,W} \pm S_{X^\one,W})\ch[W](\tau,v)\\
  &=\tau^{\mathrm{wt}[v]}\sum_{W\in\cS}(1 \pm d\epsilon_{W})S_{X^\zero,W}\ch[W](\tau,v)\\
  &=\tau^{\mathrm{wt}[v]}\sum_{W^\zero\in\cS^\zero}(1 \pm d\epsilon_{W^\zero})(S_{X^\zero,W^\zero}\ch[W^\zero](\tau,v)
    +S_{X^\zero,W^\one}\ch[W^\one](\tau,v))\\
  &=\tau^{\mathrm{wt}[v]}\sum_{W^\zero\in\cS^\zero}S_{X^\zero,W^\zero}(1 \pm d\epsilon_{W^\zero})
    (\ch[W^\zero](\tau,v)+d\epsilon_{X^\zero}\ch[W^\one](\tau,v))\\
  &=\tau^{\mathrm{wt}[v]}\sum_{W^\zero\in\cS^\zero}S_{X^\zero,W^\zero}\cdot
    (1 \pm d\epsilon_{W^\zero}) \cdot (\ch^{d\epsilon_{X^\zero}}[\cF(W^\zero)](\tau,v))\\
  &=\tau^{\mathrm{wt}[v]}\sum_{W^\zero\in\cS^\zero}S_{X^\zero,W^\zero}\cdot
    (1 \mp \epsilon_{W^\zero}) \cdot (\ch^{-\epsilon_{X^\zero}}[\cF(W^\zero)](\tau,v))\\
  &=\tau^{\mathrm{wt}[v]}\sum_{\substack{W^\zero\in\cS^\zero\\\epsilon_{W^\zero}=\mp 1 }}
  2\cdot S_{X^\zero,W^\zero} \cdot (\ch^{-\epsilon_{X^\zero}}[\cF(W^\zero)](\tau,v)).
\end{align*}
Note that $\cF(W^\zero)$ exhaust the representatives of simple objects in $\repA$, up to
parity.  Taking into account various cases, we can summarize how the
(super)characters close under modular transformations for a simple
object $X$ in $\repA$.
\begin{center}
\renewcommand{\arraystretch}{1.5}
\begin{tabular}{r|c|c}
  $X=X^\zero\oplus X^\one$ & $\ch^+[X]$ & $\ch^-[X]$\\
  \hline
  $X$ is untwisted & $\sum\ch^-[\mathrm{twisted}]$ & $\sum\ch^-[\mathrm{untwisted}]$\\
  $X$ is twisted & $\sum\ch^+[\mathrm{twisted}]$ & $\sum\ch^+[\mathrm{untwisted}]$
\end{tabular}
\end{center}

We can summarize the information in the following way.  We let
$\cS^\even =\{ X^\even_1,X^\even_2,\dots,W^\even_1,W^\even_2,\dots \}$ where the $X$'s induce
to untwisted modules and the $W$'s induce to twisted modules.  Then, we
have the following same pattern in the $S^{\pm 1}$-matrices:
\begin{align}
\renewcommand{\arraystretch}{1.2}
S = \begin{tabu}{r|rrrrrr}
       & X^\zero_1  & X^\one_1  &\dots  & W^\zero_1  & W^\one_1  & \dots \\
\hline
X^\zero_1  & a      & -a     &\dots  & b      & -b     & \dots  \\
X^\one_1  & -a     &  a     &\dots  & b      & -b     & \dots  \\
\vdots & \vdots & \vdots &\vdots & \vdots & \vdots & \vdots \\
W^\zero_1  & b      &  b     &\dots  & c      &  c     & \dots  \\
W^\one_1  & -b     & -b     &\dots  & c      &  c     & \dots  \\
\vdots & \vdots & \vdots &\vdots & \vdots & \vdots & \vdots 
\end{tabu},\quad
S^{-1} = \begin{tabu}{r|rrrrrr}
& X^\zero_1  & X^\one_1  &\dots  & W^\zero_1  & W^\one_1  & \dots \\
\hline
X^\zero_1  &  e     & -e     &\dots  & f      & -f     & \dots  \\
X^\one_1  & -e     &  e     &\dots  & f      & -f     & \dots  \\
\vdots & \vdots & \vdots &\vdots & \vdots & \vdots & \vdots \\
W^\zero_1  & f      &  f     &\dots  & g      &  g     & \dots  \\
W^\one_1  & -f     & -f     &\dots  & g      &  g     & \dots  \\
\vdots & \vdots & \vdots &\vdots & \vdots & \vdots & \vdots 
\end{tabu}
\renewcommand{\arraystretch}{1},
\end{align}
and hence, if we change the basis, denoting $X^\pm_i = X^\zero_i\pm X^\one_i$
and similarly for the $W$'s, we get the following two matrices
respectively, which have a pattern similar to each other:
\begin{align}
  \renewcommand{\arraystretch}{1.2}
  \widetilde{S} = \begin{tabu}{r|rrrrrr}
    & X^+_1  & X^-_1  &\dots  & W^+_1  & W^-_1  & \dots \\
    \hline
    X^+_1  & 0      &  0     &\dots  & 0      & 2b     & \dots  \\
    X^-_1  &  0     & 2a     &\dots  & 0      &  0     & \dots  \\
    \vdots & \vdots & \vdots &\vdots & \vdots & \vdots & \vdots \\
    W^+_1  & 0      &  0     &\dots  &2c      &  0     & \dots  \\
    W^-_1  & 2b     &  0     &\dots  & 0      &  0     & \dots  \\
    \vdots & \vdots & \vdots &\vdots & \vdots & \vdots & \vdots 
  \end{tabu},\quad \widetilde{S}^{-1} = \begin{tabu}{r|rrrrrr}
    & X^+_1  & X^-_1  &\dots  & W^+_1  & W^-_1  & \dots \\
    \hline
    X^+_1  & 0      &  0     &\dots  & 0      & 2f     & \dots  \\
    X^-_1  &  0     & 2e     &\dots  & 0      &  0     & \dots  \\
    \vdots & \vdots & \vdots &\vdots & \vdots & \vdots & \vdots \\
    W^+_1  & 0      &  0     &\dots  &2g      &  0     & \dots  \\
    W^-_1  & 2f     &  0     &\dots  & 0      &  0     & \dots  \\
    \vdots & \vdots & \vdots &\vdots & \vdots & \vdots & \vdots
\end{tabu}
\renewcommand{\arraystretch}{1}.
\end{align}
In particular, we have the following observation:
\begin{quote}
  Given simple modules $U=U^\zero\oplus U^\one$ and $Z=Z^\zero\oplus Z^\one$ in $\repA$,
  there exists a unique coefficient in the $\widetilde{S}^{\pm 1}$
  matrices corresponding to $U^\pm$ and $Z^\pm$ that can be non-zero.
  We denote this coefficient
  $(\widetilde{S}^{\pm 1})_{U,Z}$. For example, in the matrices
  above,
  $\widetilde{S}_{X_1,W_1}:=\widetilde{S}_{X_1^+,W_1^-}=2 S_{X^\zero_1,
    W^\zero_1}$.
\end{quote}

The $T$-transformation is easy. Note that by balancing, for any simple
$X^\zero\in\cS$, (after identifying $\theta$ with the scalar by which it
acts), we have: $\theta_{X^\zero\boxtimes \sic}=\epsilon_{X^\zero}\theta_{X^\zero}$.
Hence for a simple $\repA$-object $X=X^\zero\oplus X^\one$,
\begin{align*}
  \ch^\pm[X](\tau+1)
  &= \ch[X^\zero](\tau+1) \pm \ch[X^\one](\tau+1)\\
  &=\theta_{X^\zero}\ch[X^\zero](\tau)\pm\epsilon_{X^\zero}\theta_{X^\zero}\ch[X^\one](\tau)\\
  &=\theta_{X^\zero}\ch^{\pm\epsilon_{X^\zero}}[X](\tau).
\end{align*}

Now we turn to fusion rules.  Let
$U=U^\zero\oplus U^\one, W=W^\zero\oplus W^\one, X=X^\zero\oplus X^\one$ be simple objects in
$\repA$.  We let
${N^\pm}_{U,W}^X=N_{U^\zero,W^\zero}^{X^\zero}\pm N_{U^\zero,W^\zero}^{X^\one}$. First,
\begin{align*}
  N_{U^\zero,W^\zero}^{X^\zero}
  &=\sum_{Z \in\cS}\dfrac{S_{U^\zero,Z}\cdot
    S_{W^\zero,Z}\cdot (S^{-1})_{Z,X^\zero}}{S_{V,Z}}\\
  &=\sum_{Z^\zero\in\cS^\zero}\left(1+\dfrac{\epsilon_{U^\zero}\epsilon_{W^\zero}}
    {\epsilon_{X^\zero}}\right)
    \dfrac{S_{U^\zero,Z^\zero}\cdot S_{W^\zero,Z^\zero}\cdot (S^{-1})_{Z^\zero,X^\zero}}{S_{V,Z^\zero}}\\
 &=\sum_{Z^\zero\in\cS^\zero}
  2\cdot\dfrac{S_{U^\zero,Z^\zero}\cdot S_{W^\zero,Z^\zero}\cdot (S^{-1})_{Z^\zero,X^\zero}}{S_{V,Z^\zero}}. 	
\end{align*}
if $\epsilon_{X^\zero}=\epsilon_{U^\zero}\epsilon_{W^\zero}$.
Noting that $\epsilon_{X^\zero}=\epsilon_{X^\one}$ (since
$\cM_{\sic,\sic}=1_{V}$), we see that if $\epsilon_{X^\zero}=\epsilon_{U^\zero}\epsilon_{W^\zero}$:
\begin{align*}
  {N^\pm}_{U,W}^X &=\sum_{Z^\zero\in\cS^\zero}2\cdot\dfrac{S_{U^\zero,Z^\zero}\cdot S_{W^\zero,Z^\zero}\cdot \left((S^{-1})_{Z^\zero,X^\zero}\pm (S^{-1})_{Z^\zero,X^\one} \right)
  }{S_{V,Z^\zero}}\\
                  &=\sum_{Z^\zero\in\cS^\zero} 2\cdot(1\pm d\epsilon_{Z^\even})\cdot\dfrac{S_{U^\zero,Z^\zero}\cdot S_{W^\zero,Z^\zero}\cdot (S^{-1})_{Z^\zero,X^\zero}}{S_{V,Z^\zero}}\\
     	&=\sum_{\substack{Z^\zero\in\cS^\zero\\\epsilon_{Z^\zero}=\mp 1}}4\cdot\dfrac{S_{U^\zero,Z^\zero}\cdot S_{W^\zero,Z^\zero}\cdot (S^{-1})_{Z^\zero,X^\zero}}{S_{V,Z^\zero}}\\
	&=\sum_{\substack{Z \,\text{simple}\\\epsilon_{Z^\zero}=\mp 1}}
	\dfrac{\widetilde{S}_{U,Z}\cdot \widetilde{S}_{W,Z}\cdot (\widetilde{S}^{-1})_{Z,X}}{\widetilde{S}_{A,Z}}
\end{align*}
From this, we get the following fusion (super)coefficients,
recalling that for a simple object $Z=Z^\zero\oplus Z^\one$ in $\repA$, we
let $Z^\pm$ be the elements $Z^\zero\pm Z^\one$ in the Grothendieck ring. With $U,W,X\in [\repA]$ and
\begin{align*}
t&: [\repA] \rightarrow \ZZ/2\ZZ, \quad
X\mapsto \begin{cases}
1 &\text{if}\quad X\in[\rep^{\text{tw}}A]\\
0 &\text{if}\quad X\in[\repzA],
\end{cases}
\end{align*}
we have:
\begin{equation}
\begin{split}
{N^+}_{U,W}^{X}&=  \delta_{t(U)+t(W), t(X)}\sum\limits_{Z \in [\rep^{\text{tw}}A]}
\dfrac{\widetilde{S}_{U,Z}\cdot \widetilde{S}_{W,Z}\cdot (\widetilde{S}^{-1})_{Z,X}}{\widetilde{S}_{V, Z}}\\
{N^-}_{U,W}^{X}&=  \delta_{t(U)+t(W), t(X)} \sum\limits_{Z \in [\repzA]}
\dfrac{\widetilde{S}_{U,Z}\cdot \widetilde{S}_{W,Z}\cdot (\widetilde{S}^{-1})_{Z,X}}{\widetilde{S}_{V, Z}}.
\end{split}
\end{equation}

We understand the numbers $N^\pm$ as follows. Let $U=U^\zero\oplus U^\one$,
$W=W^\zero\oplus W^\one$, and $X=X^\zero\oplus X^\one$ be simple modules in $\repA$, and let
$\Pi$ denote the parity reversal functor, that is,
$\Pi(X) = X^\one\oplus X^\zero$.  Note that $\cF(X^\one) = \Pi(X)$.  Since
$\cF(U^\zero)\boxtimes_A \cF(W^\zero) \cong \cF(U^\zero\boxtimes W^\zero)$, we have that
\begin{align*}
\cF(U^\zero)\boxtimes_A \cF(W^\zero) &= \bigoplus_{X\in\cS} N_{U^\zero,W^\zero}^{X}\cF(X) \\
&= \bigoplus_{X^\zero\in\cS^\zero} N_{U^\zero,W^\zero}^{X^\zero}\cF(X^\zero) \oplus  \bigoplus_{X^\one\in\cS^\one} N_{U^\zero,W^\zero}^{X^\one}\cF(X^\one).
\end{align*}
Correspondingly, intertwining operators in $\cC$ of type
$\binom{X^\zero}{U^\zero\, W^\zero}$ lift to even intertwining operators, and
those of type $\binom{X^\one}{U^\zero\, W^\zero}$ lift to odd intertwining
operators.  Therefore,  ${N^+}^X_{U\,W}$ is the dimension of the
space of $A$-intertwining operators of type $\binom{X}{U\,W}$, and ${N^-}^{X}_{U\,W}$ is the corresponding superdimension.

Now we analyze the properties of asymptotic dimensions.
Suppose there exists a unique  untwisted simple module $Z_{(+)}$ with lowest conformal dimension 
among untwisted simple modules,  and there also exists a unique simple twisted module $Z_{(-)}$
of lowest conformal dimension among twisted simple modules.
We assume that $\ch^-[Z_{(+)}]$ does not vanish
and that the lowest conformal weight spaces of $Z_{(\pm)}$ are purely even.
For a simple $A$-module $X$, we have:
\begin{align}
\adim^\pm[X] = \lim_{iy\rightarrow i\infty}\dfrac{\ch^\pm[X](-1/iy)}{\ch^\pm[A](-1/iy)} = 
\dfrac{\widetilde{S}_{X,Z_{(\mp)}}}{\widetilde{S}_{A,Z_{(\mp)}}}
=\dfrac{S_{X^0,Z^0_{(\mp)}}}{S_{V,Z^0_{(\mp)}}}.
\end{align}
Now, for simple modules $X$ and $Y$ in $\repA$, we have that:
\begin{align}
\adim^\pm[X]\adim^{\pm}[Y]&=\dfrac{S_{X^\zero,Z^\zero_{(\mp)}}}{S_{V,Z^\zero_{(\mp)}}} \dfrac{S_{Y^\zero,Z^\zero_{(\mp)}}}{S_{V,Z^\zero_{(\mp)}}}
\nonumber\\
&=\sum_{\substack{W^\zero\in \cS^\zero\\ \epsilon_{W^\zero}=
	\epsilon_{X^\zero}\epsilon_{Y^\zero}}}N_{X^\zero, Y^\zero}^{W^\zero} \dfrac{S_{W^\zero,Z^\zero_{(\mp)}}}{S_{V,Z^\zero_{(\mp)}}}
+N_{X^\zero, Y^\zero}^{W^\one} \dfrac{S_{W^\one,Z^\zero_{(\mp)}}}{S_{V,Z^\zero_{(\mp)}}}
\nonumber\\
&=\sum_{\substack{W^\zero\in \cS^\zero\\ \epsilon_{W^\zero}=
		\epsilon_{X^\zero}\epsilon_{Y^\zero}}}\left(N_{X^\zero, Y^\zero}^{W^\zero} + d\epsilon_{Z^\zero_{(\mp)}} N_{X^\zero, Y^\zero}^{W^\one}\right)
\dfrac{S_{W^\zero,Z^\zero_{(\mp)}}}{S_{V,Z^\zero_{(\mp)}}}\nonumber\\
&=\sum_{\substack{W^\zero\in \cS^\zero\\\epsilon_{W^\zero}=
		\epsilon_{X^\zero}\epsilon_{Y^\zero}}}{N^\pm}_{X, Y}^{W}
\dfrac{\widetilde{S}_{W,Z_{(\mp)}}}{\widetilde{S}_{A,Z_{(\mp)}}}
=\sum_{\substack{W^\zero\in \cS^\zero\\ \epsilon_{W^\zero}=
		\epsilon_{X^\zero}\epsilon_{Y^\zero}}}{N^\pm}_{X, Y}^{W}\adim^\pm[W].
\label{eqn:adimSupWrong}
\end{align}
This shows that asymptotic dimensions respect the tensor product on $\repA$.

\subsubsection{Algebras with wrong statistics}
Now let $A=V\oplus\sic$ be a vertex operator algebra such that
$\theta_\sic=-1_\sic$, $\qdim(\sic)=-1$, and
$\cM_{\sic,\sic}=1_{V}$.  In this case, objects of $\repA$ do not have an inherent $\ZZ/2\ZZ$-grading.
However,  (\ref{it:SimpleCurrentSmatrix1})--(\ref{it:SimpleCurrentSmatrix3}) of Proposition \ref{prop:SimpleCurrentSmatrix} apply in this setting (since $\qdim(J) =-1$ implies fixed-point freeness) to yield the following lemma:
\begin{lemma}
Every simple object $X$ in $\repA$ is an induced object; in particular, $X=X^\zero\oplus X^\one$
  where $X^\zero$ is simple in $\cC$ and $X^\one=\sic\boxtimes X^\zero$.
   Conversely, simple objects in $\cC$ induce to simple objects in $\repA$.
\end{lemma}

From Lemma \ref{lem:Sinverse}, the properties of Hopf links are the same as in the wrong statistics superalgebra case. Thus every statement in the wrong statistics superalgebra case goes through here as well after fixing an arbitrary choice of
$\ZZ/2\ZZ$-decomposition for every simple object of $\repA$, except for
the $T$-transformation. We list relevant results in the Introduction, Section \ref{sec:BPlikevoas}, and we now only indicate the difference in the modular $T$-transformation.

In this case, for a simple module
$X^\zero$ in $\cC$, after identifying $\theta$ with the scalar by which it
acts, we have $\theta_{X^\zero\boxtimes \sic}=-\epsilon_{X^\zero}\theta_{X^\zero}$.
Hence, for a simple module $X=X^\zero\oplus X^\one$ in $\repA$,
\begin{align*}
  \ch^\pm[X](\tau+1)
  &= \ch[X^\zero](\tau+1) \pm \ch[X^\one](\tau+1)\\
  &=\theta_{X^\zero}\ch[X^\zero](\tau)\mp\epsilon_{X^\zero}\theta_{X^\zero}\ch[X^\one](\tau)\\
  &=\theta_{X^\zero}\ch^{\mp\epsilon_{X^\zero}}[X](\tau).
\end{align*}

The analysis of asymptotic dimensions carried out in \eqref{eqn:adimSupWrong}
also goes through if 
we assume that $\ch^-[Z_{(-)}]$ does not vanish
and that the lowest conformal weight spaces of $Z_{(\pm)}$ are purely even.

\subsubsection{Superalgebras with correct statistics}\label{CSsuperalgVerlinde}

Recall  that  ``correct  statistics''   means  that  for
$A=V\oplus\sic$,  $\qdim(\sic)=1$ and $\theta_\sic = \pm 1_{\sic}$. This  means that  for a  simple module
$X^\zero$ in $\cC$,  we   do not  necessarily  have   $X^\zero\not\cong  \sic\boxtimes
X^\zero$. 
This complicates  things somewhat,  especially when $A$
is a vertex operator algebra. 

From Corollary \ref{corol:inducedsimple}, simple modules of $\cC$ induce to simple modules of $\underline{\repA}$, and from Corollary \ref{corol:Xsimpleisoinduced}, simple modules of $\underline{\repA}$ are induced from simple modules in $\cC$.
In the discussion below, it will be important to consider those modules of $\repA$
that are simple in $\underline{\repA}$.

\begin{nota}
  Consider the set of representatives of isomorphism classes of simple objects in $\underline{\repA}$ up to parity.
  Each such simple module can be written uniquely as
  $X\cong X^\zero\oplus X^\one$ where the $X^i$ are in $\cS$ with $X^\one\cong \sic\boxtimes X^\zero$. 
  Let $\cS^\zero$ be those
  $X^\zero$ such that $X^\zero\not\cong X^\one$. Let $\cS^\one$ be the set of
  simple $V$-modules obtained by tensoring $\sic$ to the objects in $\cS^\zero$ and
  let $\cS^f= \cS \backslash (\cS^\zero \cup \cS^\one)$, that is, those simple modules
  that are fixed under the action of $\sic$.
\end{nota}  

In this case, we have the following $S$-matrix relations:
\begin{lemma}\label{lem:SACSbasics}
  Let $W$ and $X$ be simple $V$-modules in $\cS$ and let $X^{(i)}=\sic^i\boxtimes X$ for $i\in\NN$.
  \begin{enumerate}
  \item For $i,j\in\NN$, $(S^{\pm})_{W^{(i)},X^{(j)}}=\epsilon_X^i\epsilon_W^j(S^{\pm})_{W,X}$.
  \label{it:SACSbasics1}
  \item If $X\in \cS^f$, then $\epsilon_X=-1$.
    \label{it:SACSbasics2}
  \item If $X\in\cS^f$ and if $\epsilon_W=-1$, then $(S^{\pm})_{W,X}=0$.
  \label{it:SACSbasics3}
  \end{enumerate}
\end{lemma}  
\begin{proof}
 Item \eqref{it:SACSbasics1} is Lemma \ref{lem:Sinverse} applied to this setting. Item \eqref{it:SACSbasics2} follows from the balancing equation: $\theta_{J\boxtimes X}=\cM_{J,X}\circ(\theta_J\boxtimes\theta_X)$. For item \eqref{it:SACSbasics3}, it follows from item \eqref{it:SACSbasics1} that $$(S^\pm)_{W, X}=(S^\pm)_{W, X\boxtimes J}=\epsilon_W(S^\pm)_{W,X}=-(S^\pm)_{W,X},$$
 so that $(S^\pm)_{W,X}=0$.
\end{proof}

Now let $X=X^\even\oplus X^\odd$, with $X^i\in\cS$, be a simple module in $\underline{\repA}$. 
We proceed as for wrong statistics superalgebras, but now using Lemma \ref{lem:SACSbasics},  to obtain:
\begin{align*}
  (\ch^\pm[X])&(-1/\tau)
  =\sum_{W^\zero\in \cS^\zero,\epsilon_{W^\zero}=\pm 1}
    2\cdot S_{X^\zero,W^\zero} \ch^{\epsilon_{X^\zero}}[\cF(W^\zero)](\tau)
    +\sum_{W^\zero\in\cS^f}(1\mp 1)\cdot S_{X^\zero,W^\zero}  \ch[W^\zero](\tau)\\
  &=\sum_{W^\zero\in \cS^\zero,\epsilon_{W^\zero}=\pm 1}
    2\cdot S_{X^\zero,W^\zero}  \ch^{\epsilon_{X^\zero}}[\cF(W^\zero)](\tau)
    +\sum_{W^\zero\in\cS^f}(1\mp 1)\cdot S_{X^\zero,W^\zero}
    \cdot \dfrac{1}{2}\cdot \ch[\cF(W^\zero)](\tau).
\end{align*}  
We can summarize the information as before. Let $X^\zero_1, X^\zero_2,\dots$ be simple $V$-modules in $\cS^\zero$ that induce to untwisted
modules, let $W^\zero_1, W^\zero_2,\dots$ be simple $V$-modules in $\cS^\zero$ that induce to
twisted modules, and let $W^f_1, W^f_2,\dots$ be $V$-modules in $\cS^f$.
By Lemma \ref{lem:SACSbasics}, the $S$-matrix for $\cC$ has the following pattern:
\begin{align}
\renewcommand{\arraystretch}{1.2}
S = \begin{tabu}{r|rrrrrrrr}
       & X^\zero_1  & X^\one_1  &\dots  & W^\zero_1  & W^\one_1  & \dots &W^f_1 &\dots\\
\hline
X^\zero_1  & a      &  a     &\dots  & b      &  b     & \dots & e & \dots \\
X^\one_1  & a      &  a     &\dots  & -b     & -b     & \dots  & -e & \dots\\
\vdots & \vdots & \vdots &\vdots & \vdots & \vdots & \vdots & \vdots & \vdots\\
W^\zero_1  & b      & -b     &\dots  & c      &  -c    & \dots  & 0 & \dots \\
W^\one_1  &  b     & -b     &\dots  & -c     &  c     & \dots  & 0 & \dots \\
\vdots & \vdots & \vdots &\vdots & \vdots & \vdots & \vdots & \vdots & \dots \\
W^f_1  & e       & -e     &\dots  & 0      &  0     & \dots  & 0 & \dots \\
\vdots & \vdots & \vdots &\vdots & \vdots & \vdots & \vdots & \vdots & \vdots 
\end{tabu}.
\end{align}
The $S^{-1}$-matrix follows the exact same pattern.  If we change the
basis, denoting $X^\pm_i=X_i^\zero \pm X_i^\one$ and similarly 
$W^\pm_i=W_i^\zero\pm W_i^\one$, $W^{f+}_i=W^f_i+W^f_i$, we obtain the following
\emph{non-symmetric} matrix, with a similar pattern for its inverse.
This matrix, as before, describes the transformation laws for characters
and supercharacters of simple modules in $\repA$.
\begin{align}
\renewcommand{\arraystretch}{1.2}
\widetilde{S} = \begin{tabu}{r|rrrrrrrr}
       & X^+_1  & X^-_1  &\dots  & W^+_1  & W^-_1  & \dots &W^{f+}_1 &\dots\\
\hline
X^+_1  & 2a     &  0     &\dots  & 0      & 0     & \dots & 0 & \dots \\
X^-_1  & 0      &  0     &\dots  & 2b      & 0      & \dots  & e & \dots\\
\vdots & \vdots & \vdots &\vdots & \vdots & \vdots & \vdots & \vdots & \vdots\\
W^+_1  &  0      & 2b     &\dots  & 0      &  0     & \dots  & 0 & \dots \\
W^-_1  &  0     &  0     &\dots  & 0      &  2c     & \dots  & 0 & \dots \\
\vdots & \vdots & \vdots &\vdots & \vdots & \vdots & \vdots & \vdots & \dots \\
W^{f+}_1  & 0       & 2e     &\dots  & 0      &  0     & \dots  & 0 & \dots \\
\vdots & \vdots & \vdots &\vdots & \vdots & \vdots & \vdots & \vdots & \vdots 
\end{tabu}.
\end{align}
Again, as before, we have:
\begin{quote}
	Given simple modules $U=U^\zero\oplus U^\one$ and $Z=Z^\zero\oplus Z^\one$,
	there exists a unique coefficient in the $\widetilde{S}^{\pm 1}$
	matrices corresponding to $U^\pm$ and $Z^\pm$ that can be non-zero.
	We denote this coefficient
	$(\widetilde{S}^{\pm 1})_{U,Z}$.  For example, in the matrices
	above,
	$\widetilde{S}_{W_1,X_1}:=\widetilde{S}_{W_1^+,X_1^-}=2 S_{W^\zero_1,
		X^\zero_1}$.
	However, $\widetilde{S}$ is \emph{non-symmetric} when one
	subscript corresponds to a fixed-point object.
\end{quote}
From the $\widetilde{S}$-matrix, we can summarize the (super)character $S$-transformation properties as:
\begin{center}
	\renewcommand{\arraystretch}{1.5}
	\begin{tabular}{r|c|c}
		$X=X^\zero\oplus X^\one$ & $\ch^+[X]$ & $\ch^-[X]$\\
		\hline
		$X$ is untwisted & $\sum\ch^+[\mathrm{untwisted}]$ & $\sum\ch^+[\mathrm{twisted}]$\\
		$X$ is twisted, $X^\zero\not\cong X^\one$& $\sum\ch^-[\mathrm{untwisted}]$ & $\sum\ch^-[\mathrm{twisted}]$\\
		$X$ is twisted, $X^\zero\cong X^\one$ &    $\sum\ch^-[\mathrm{untwisted}]$          & $0$
	\end{tabular}
\end{center}

As in the wrong statistics case, we calculate the even and odd fusion coefficients.
For simple modules $U=U^\zero\oplus U^\one$, $W=W^\zero\oplus W^\one$, $X=X^\zero\oplus X^\one$,
\begin{align*}
N^{X^\zero}_{U^\zero, W^\zero} 
&=\sum_{\substack{Z^\zero\in\cS^\zero\\ \epsilon_{X^\zero}=\epsilon_{U^\zero}\epsilon_{W^\zero}  }}
2\cdot \dfrac{S_{U^\zero,Z^\zero} S_{W^\zero,Z^\zero} (S^{-1})_{Z^\zero,X^\zero}}{S_{V, Z^\zero}}
+\sum_{\substack{Z^\zero\in\cS^f\\ \epsilon_{X^\zero}=\epsilon_{U^\zero}=\epsilon_{W^\zero}=1  }}
\dfrac{S_{U^\zero,Z^\zero} S_{W^\zero,Z^\zero} (S^{-1})_{Z^\zero,X^\zero}}{S_{V, Z^\zero}}.
\end{align*}
Therefore,
\begin{align}
{N^{\pm}}^X_{U, W} 
&=\sum_{\substack{Z^\zero\in\cS^\zero\\ \epsilon_{X^\zero}=\epsilon_{U^\zero}\epsilon_{W^\zero}  }}
2\cdot (1\pm \epsilon_{Z^\zero})\cdot\dfrac{S_{U^\zero,Z^\zero} S_{W^\zero,Z^\zero} (S^{-1})_{Z^\zero,X^\zero}}{S_{V, Z^\zero}}\nonumber\\
&+\sum_{\substack{Z^\zero\in\cS^f\\ \epsilon_{X^\zero}=\epsilon_{U^\zero}=\epsilon_{W^\zero}=1  }}
 (1\mp 1)\cdot
\dfrac{S_{U^\zero,Z^\zero} S_{W^\zero,Z^\zero} (S^{-1})_{Z^\zero,X^\zero}}{S_{V, Z^\zero}}.
\label{eqn:supAlgCorStatFus}
\end{align}
Setting 
\[
t : [\repA] \rightarrow \ZZ/2\ZZ, \hspace{1em} X \mapsto \begin{cases} \even &
\qquad \text{if} \ X \in [\repzA] \\ \odd & \qquad \text{if} \ X \in \
[\rep^{\text{tw}}A]\end{cases}\qquad\text{and} \qquad n_W = \begin{cases} 2
&\qquad \text{if} \ W \ \in[\rep^{\text{fix}}A] \\ 1
&\qquad\text{else} \end{cases},
\]
we thus have
\begin{equation}
\begin{split}
{N^+}_{U,W}^{X}&= n_X\cdot \delta_{t(U)+t(W), t(X)}\sum\limits_{Z \in [\repzA]}
\dfrac{\widetilde{S}_{U,Z}\cdot \widetilde{S}_{W,Z}\cdot (\widetilde{S}^{-1})_{Z,X}}{\widetilde{S}_{A, Z}}\\
{N^-}_{U,W}^{X}&=  \delta_{t(U)+t(W), t(X)} \sum\limits_{Z \in [\rep^{\text{tw}}A]}
\dfrac{\widetilde{S}_{U,Z}\cdot \widetilde{S}_{W,Z}\cdot (\widetilde{S}^{-1})_{Z,X}}{\widetilde{S}_{A, Z}}
\end{split}\label{eqn:supAlgCorStatVer}
\end{equation}
Again, we remind the reader that $\widetilde{S}$ is not symmetric for all values of the subscripts.

Now we analyze the asymptotic dimensions. 
Suppose there exist $Z_{(\pm)}$ as in the wrong statistics case. We assume that
$\ch^-[Z_{(-)}]$ does not vanish; in particular, $Z_{(-)}$ is not a fixed-point-type twisted module
and the lowest conformal weight spaces of $Z_{(\pm)}$ are purely even.
Then we have:
\begin{align*}
\adim^\pm[X] = \dfrac{\widetilde{S}_{X,Z_{(\pm)}}}{\widetilde{S}_{A,Z_{(\pm)}}}=\dfrac{S_{X^\zero,Z^\zero_{(\pm)}}}{S_{V,Z^\zero_{(\pm)}}}.
\end{align*}
Therefore,
\begin{align*}
\adim^\pm[X]&\adim^{\pm}[Y]=\dfrac{S_{X^\zero,Z^\zero_{(\pm)}}}{S_{V,Z^\zero_{(\pm)}}}\dfrac{S_{Y^\zero,Z^\zero_{(\pm)}}}{S_{V,Z^\zero_{(\pm)}}}=\sum_{\substack{W\in \cS \\ \epsilon_{W}=\epsilon_{X^\zero}\epsilon_{Y^\zero} }} N_{X^\zero,Y^\zero}^{W}\dfrac{S_{W,Z^\zero_{(\pm)}}}{S_{V,Z^\zero_{(\pm)}}}\\
&=\sum_{
	\substack{W^\zero\in \cS^\zero \\ \epsilon_{W^\zero}=\epsilon_{X^\zero}\epsilon_{Y^\zero}}}  \left(N_{X^\zero,Y^\zero}^{W^\zero}\dfrac{S_{W^\zero,Z^\zero_{(\pm)}}}{S_{V,Z^\zero_{(\pm)}}} + N_{X^\zero,Y^\zero}^{W^\one}\dfrac{S_{W^\one,Z^\zero_{(\pm)}}}{S_{V,Z^\zero_{(\pm)}}}\right)
+ \sum_{ \substack{W^f\in \cS^f\\ \epsilon_{X^\zero}\epsilon_{Y^\zero}=-1}
	} N_{X^\zero,Y^\zero}^{W^f}\dfrac{S_{W^f,Z^\zero_{(\pm)}}}{S_{V,Z^\zero_{(\pm)}}}\\
&=\sum_{\substack{W^\zero\in \cS^\zero\\\epsilon_{W^\zero}=\epsilon_{X^\zero}\epsilon_{Y^\zero}}} \left(N_{X^\zero,Y^\zero}^{W^\zero}\pm N_{X^\zero,Y^\zero}^{W^\one}\right)\dfrac{S_{W^\zero,Z^\zero_{(\pm)}}}{S_{V,Z^\zero_{(\pm)}}} + 
\sum_{\substack{W^f\in \cS^f\\ \epsilon_{X^\zero}\epsilon_{Y^\zero}=-1}} N_{X^\zero,Y^\zero}^{W^f}\dfrac{S_{W^f,Z^\zero_{(\pm)}}}{S_{V,Z^\zero_{(\pm)}}}\\
&=\sum_{\substack{W \text{\,non-fixed-point\,simple}\\\epsilon_{W^\zero}=\epsilon_{X^\zero}\epsilon_{Y^\zero}}} {N^\pm}_{X,Y}^{W}\dfrac{\widetilde{S}_{W,Z_{(\pm)}}}{\widetilde{S}_{A,Z_{(\pm)}}} + 
\sum_{\substack{W^f\in \cS^f\\ \epsilon_{X^\zero}\epsilon_{Y^\zero}=-1}} N_{X^\zero,Y^\zero}^{W^f}\dfrac{S_{W^f,Z^\zero_{(\pm)}}}{S_{V,Z^\zero_{(\pm)}}}\\
&=\sum_{\substack{W \text{\,non-fixed-point\,simple}\\\epsilon_{W^\zero}=\epsilon_{X^\zero}\epsilon_{Y^\zero}}} {N^\pm}_{X,Y}^{W}\adim^{\pm}[W] + 
\sum_{\substack{W^f\in \cS^f\\ \epsilon_{X^\zero}\epsilon_{Y^\zero}=-1}} N_{X^\zero,Y^\zero}^{W^f}\dfrac{S_{W^f,Z^\zero_{(\pm)}}}{S_{V,Z^\zero_{(\pm)}}}.
\end{align*}
Hence, for $\adim^-$ we have:
\begin{align}
\adim^-[X]\adim^-[Y]=\sum_{\substack{W \text{simple}\\\epsilon_{W^\zero}=\epsilon_{X^\zero}\epsilon_{Y^\zero}}}{N^-}_{X,Y}^W\adim^-[W],
\label{eqn:supAlgCorStatadim}
\end{align}
so asymptotic superdimension respects the tensor product on $\repA$. The relation $\adim^+[X]\adim^+[Y]=\adim^+[X\boxtimes_A Y]$ also holds if $X$ and $Y$ are both untwisted or both twisted.

\begin{rema} Several observations about fixed-point simple modules are in order;
let $T$ be such a module.
First, by definition, (and also from \eqref{eqn:supAlgCorStatFus}), 
${N^-}^{T}_{X,Y}=0$, ${N^-}^{W}_{T,Y}=0$, ${N^-}^{W}_{X,T}=0$.
For such a $T$, one can alternately define ${N^+}^{T}_{X,Y} = N_{X^\zero, Y^\zero}^{T^\zero}$, which also
makes \eqref{eqn:supAlgCorStatVer} uniform.
Since $\ch^-[T]=0$, the negative asymptotic dimension $\adim^-[T]=0$.
It is now clear that \eqref{eqn:supAlgCorStatadim} completely disregards such modules $T$.
\end{rema}

Since this scenario is a little tricky with an asymmetric $\widetilde{S}$-matrix, we present a quick example:
\begin{exam}
Let $A$ be the vertex operator superalgebra of one free fermion.
Up to isomorphisms, it has one simple untwisted module, say $U=A$,
and one twisted module, say $T$. 
The twisted module is a fixed point.
We have the following characters:
\begin{align*}
\ch^+[U] &= q^{-1/24}\prod_{n\geq 0}(1+q^{n+1/2}) = \dfrac{\eta(\tau)^2}{\eta(\tau/2)\eta(2\tau)}, \quad 
&&\ch^-[U] =q^{-1/24} \prod_{n\geq 0}(1-q^{n+1/2}) = \dfrac{\eta(\tau/2)}{\eta(\tau)},\\
\ch^+[T] &= 2q^{1/12}\prod_{n\geq 1}(1+q^n)=2\dfrac{\eta(2\tau)}{\eta(\tau)},\quad 
&& \ch^-[T]  =0.
\end{align*}
We also have the following fusion rules as modules for the even part of $A$:
\begin{align*}
N_{U^i, U^j}^{U^k} = \delta_{i+j, k},\quad
N_{U^i,T^j}^{T^j} = 1,\quad
 N_{T^i, T^i}^{U^j}=1.
\end{align*}
The $\widetilde{S}$-matrix can be calculated using the transformation properties of $\eta$:
\begin{align*}
\renewcommand{\arraystretch}{1.2}
\widetilde{S} = \begin{tabu}{r|ccc}
     & U^+ & U^- & T^+ \\ \hline
U^+  & 1   &  0 & 0 \\
U^-  & 0   &  0 & 1/\sqrt{2} \\
T^+  & 0 &  \sqrt{2}   & 0,
\end{tabu}=\widetilde{S}^{-1}.
\end{align*}
Now we are ready to calculate the fusion rules.
\begin{align}\nonumber
{N^+}_{U, U}^U &= \sum_{\epsilon_Z=1}\dfrac{\widetilde{S}_{U,Z}
	\widetilde{S}_{U,Z}(\widetilde{S}^{-1})_{Z,U}}{\widetilde{S}_{A,Z}}
=\dfrac{\widetilde{S}_{U,U}
	\widetilde{S}_{U,U}\widetilde{S}_{U,U}}{\widetilde{S}_{A,U}}
=1,\\ \nonumber
{N^-}_{U, U}^U &= \sum_{\epsilon_Z=-1}\dfrac{\widetilde{S}_{U,Z}
	\widetilde{S}_{U,Z}(\widetilde{S}^{-1})_{Z,U}}{\widetilde{S}_{A,Z}}
=\dfrac{\widetilde{S}_{U,T}
	\widetilde{S}_{U,T}\widetilde{S}_{T,U}}{\widetilde{S}_{U,T}}
=\dfrac{(1/\sqrt{2})(1/\sqrt{2})(\sqrt{2})}{(1/\sqrt{2})} =1,\\ \nonumber
{N^+}_{U, T}^T &= \sum_{\epsilon_Z=1}\,2\cdot\dfrac{\widetilde{S}_{U,Z}
	\widetilde{S}_{T,Z}(\widetilde{S}^{-1})_{Z,T}}{\widetilde{S}_{A,Z}}
=\,2\cdot\dfrac{\widetilde{S}_{U,U}
	\widetilde{S}_{T,U}\widetilde{S}_{U,T}}{\widetilde{S}_{U,U}}
=2{(1/\sqrt{2})(\sqrt{2})} =2,\\ \nonumber
{N^-}_{U, T}^T &= \sum_{\epsilon_Z=-1}\dfrac{\widetilde{S}_{U,Z}
	\widetilde{S}_{T,Z}(\widetilde{S}^{-1})_{Z,T}}{\widetilde{S}_{A,Z}}
=\dfrac{\widetilde{S}_{U,T}
	\widetilde{S}_{T,T}\widetilde{S}_{T,T}}{\widetilde{S}_{U,T}}
=0,\\ \nonumber
{N^+}_{T, T}^U &= \sum_{\epsilon_Z=1}\dfrac{\widetilde{S}_{T,Z}
	\widetilde{S}_{T,Z}(\widetilde{S}^{-1})_{Z,U}}{\widetilde{S}_{A,Z}}
=\dfrac{\widetilde{S}_{T,U}
	\widetilde{S}_{T,U}\widetilde{S}_{U,U}}{\widetilde{S}_{U,U}}
=2,\\ \nonumber
{N^-}_{T, T}^U &= \sum_{\epsilon_Z=-1}\dfrac{\widetilde{S}_{T,Z}
	\widetilde{S}_{T,Z}(\widetilde{S}^{-1})_{Z,U}}{\widetilde{S}_{A,Z}}
	=\dfrac{\widetilde{S}_{T,T}
		\widetilde{S}_{T,T}\widetilde{S}_{T,U}}{\widetilde{S}_{U,T}}
		=0.
\end{align}
\end{exam}

\subsection{Cosets by lattice vertex operator algebras: General results}\label{sec:coset}

If $V$ is a \VOA{} and $U$ a vertex operator subalgebra of $V$ (with a different conformal vector), then the commutant $C=\com\left(U, V\right)$ is called the coset vertex operator subalgebra of $U$ in $V$. 
If $U=\com\left(C, V\right)$, then $U$ and $C$ are said to be a mutually commuting or Howe pair in $V$. The aim of this section is to show that much of the
representation theories of the two \VOAs{} $V$ and $C$ are related, and thus one can reconstruct structure in one of them from
that in the other.

\subsubsection{Coset modules}

Surely the simplest cosets are those by lattice or
Heisenberg vertex operator subalgebras. 
Recently, some results regarding such cosets in the general non-rational
setting have been given \cite{CKLR},  but for this section, we concentrate on the rational picture and rederive several
results well-established in the literature in a new and more general way
using the machinery developed thus far. Most notably, we exhibit a natural use
of the induction functor and the characterization of $\rep^0 A$. We begin by recalling
several results from \cite{CKLR} and elsewhere. Recall that a vertex operator algebra is of CFT type if it has no negatively-graded weight spaces and if $V_0=\CC\unit$.
\begin{theo}  \textup{\cite[Corollary\ 4.13]{CKLR}}
  Let $V$ be a simple, self-contragredient, rational, and
  $C_2$-cofinite \VOA{} of CFT type. If $L$ is an even positive
  definite lattice such that the lattice vertex operator algebra $V_L$ is a vertex subalgebra of $V$ 
  and $V_L$ is maximal with this
  property,  then the coset $C=\com(V_L, V)$ is simple, rational, and
  $C_2$-cofinite.
\end{theo}
We note an easy lemma which we shall frequently use:
\begin{lemma}
	Let $A$, $B$ and $A\otimes B$ be regular \VOA s. 
	Then
	$$(A_1\otimes B_1)\fus{A\otimes B}(A_2\otimes B_2)\cong (A_1\fus{A}A_2)\otimes (B_1\fus{B}B_2)$$
	for $A$-modules $A_1$, $A_2$ and $B$-modules $B_1$, $B_2$.
\end{lemma}

\begin{rema}
	If the vertex operator algebras involved are not regular, then the conclusion of the lemma still holds
	when the left side is already known to be a module that cleanly
	factors as a tensor product of two modules.
	This is the case, for instance, when $A$ is a lattice or Heisenberg 	
	\VOA{}
	and if we are working with a semisimple category of modules.
\end{rema}
Let $V=M_0, M_1, \dots, M_n$ denote the inequivalent simple
$V$-modules, and let $L'$ be the dual lattice of $L$.  Then there is a lattice $N$,
$L \subseteq N \subseteq L'$
such that
\[
V= \bigoplus_{\nu+L\in N/L} V_{\nu+L} \otimes M_{0,\nu}
\]
as a $V_L\otimes C$-module, with $C=M_{0,0}$. 
The finite group $N/L$ acts by automorphisms on $V$;
hence each $V_{\nu+L} \otimes M_{0,\nu}$ is a simple $V_L\otimes C$-module \cite{DM1, DLM}
and also a simple current for $V_L\otimes C$ \cite{CarM}. 
\begin{rema}\label{rem:isomod}
The simple currents $V_{\nu+L} \otimes M_{0,\nu}$ are fixed-point free since this holds
for the simple modules for a lattice \VOA{}.
Hence, Proposition \ref{prop:SimpleCurrentSmatrix} implies that
two simple $V$-modules are isomorphic if and only if they are isomorphic as  $V_L\otimes C$-modules.
\end{rema}
By \cite{CKLR}, all $M_{0,\nu}$ are simple currents for $C=M_{0,0}$ with fusion rules
\[
M_{0,{\nu_1}} \boxtimes_C M_{0,{\nu_2}} = M_{0,{\nu_1+\nu_2}} \, \text{for }\, 
\nu_1,\nu_2\in N/L.
\]
Moreover, they are pairwise inequivalent, which is a consequence of the maximality of $L$ or $V_L$.
A similar $V_L\otimes C$-decomposition holds \cite[Theorem\ 3.8]{CKLR} for 
each simple $V$-module $M_i$:
\[
M_i= \bigoplus_{\lambda+L\in L'/L} V_{\lambda+L} \otimes M_{i,\lambda}, 
\]
with
\[
  M_{0,\nu} \boxtimes_C M_{i,\lambda}
  = M_{i,\nu+\lambda} \, \text{for }\, \nu\in N/L, \lambda\in L'/L.
\]
Simplicity of $M_i$ implies that if
$M_{i,\lambda_1}, M_{i,\lambda_2}\neq 0$ then
$\lambda_1-\lambda_2\in N/L$.  Therefore, for each $M_i$, we fix a choice of $\lambda_i\in L'/L$ such that $\lambda_0=0$ and
\begin{align}
  M_i = \bigoplus_{\nu+L\in N/L} V_{\lambda_i + \nu+L}\otimes M_{i,\lambda_i+\nu}.
  \label{eqn:Midecomp}
\end{align}
The $M_{i,\lambda}$ are simple
$C$-modules by \cite[Theorem\ 3.8]{CKLR}, but for any fixed $i$, they need not be pairwise
non-isomorphic.  
In any case, the decomposition \eqref{eqn:Midecomp}
implies that each $M_i$ can be obtained as the induction of a simple
$V_L\otimes C$-module, for example
$V_{\lambda_i+L}\otimes M_{i,\lambda_i}$. We shall investigate additional properties of the $M_{i,\mu}$ below.

If $X$ is a simple $C$-module then by \cite[Theorem 4.3, Lemma 4.9]{CKLR} and the proof of \cite[Theorem 4.12]{CKLR}, there exists a $\lambda+L\in L'/L$ such that
\begin{align}
\cF(V_{\lambda+\nu'+L} \otimes X) \cong_{V_{L}\otimes C} \bigoplus_{\nu+L\in N/L} V_{\lambda+\nu'+\nu+L} \otimes \left(M_{0,\nu} \boxtimes_C X\right)
\label{eqn:tenslift}
\end{align}
is an untwisted $V$-module if and only if $\nu'+L \in N'/L$, where 
$L\subseteq N'\subseteq L'$ is the lattice dual to $N$.  This result follows directly from the characterization of $\rep^0 V$ in \cite{HKL, CKL}.  As a corollary, setting $\nu'=0$, we can always tensor a given simple $C$-module $X$ with a $V_L$-module such that the result induces to an untwisted $V$-module. This $V$-module is
simple by Proposition \ref{prop:Fsimple}, so 
\[
\cF(V_{\lambda+L} \otimes X) \cong M_i
\]
for some $i \in \{0, \dots, n\}$.  It would make sense to
identify $X$ with $M_{i, \lambda}$, but this identification is not
unique and needs to be clarified.  Consider the special case
$X=C=M_{0,0}$; then for $\nu+L\in N'/L$,
\begin{align}
  M^{\nu}:=\cF(V_{\nu+L} \otimes C)
  \label{eqn:Asimplecurrents}
\end{align}
satisfy the following for $\nu_1,\nu_2\in N'/L$:
\begin{align}
M^{\nu_1} \boxtimes_V M^{{\nu_2}} &= \cF(V_{{\nu_1}+L} \otimes C) \boxtimes_V \cF(V_{{\nu_2}+L} \otimes C) \nonumber\\
&\cong \cF\left((V_{{\nu_1}+L} \otimes C)\boxtimes_{V_L\otimes C} (V_{{\nu_2}+L} \otimes C) \right) =
\cF(V_{{\nu_1}+{\nu_2}+L} \otimes C)
\cong M^{{\nu_1}+{\nu_2}}. 
\label{eqn:MupperGroup}
\end{align}

\begin{propo}\label{Mupperuntwisted}
For $\nu+L\in N'/L$, the	$M^\nu$ are pairwise non-isomorphic untwisted $V$-modules for $V$,
and they are simple currents for $V$.
\end{propo}
\begin{proof}
Since the $M_{0,\nu}$ are pairwise inequivalent for $\nu\in N/L$, the $M^\nu$ are pairwise non-isomorphic since their
decompositions as $V_L\otimes C$-modules are non-isomorphic.
From \eqref{eqn:MupperGroup}, it is clear that $M^{-\nu}\fus{V}M^\nu\cong M^0\cong V$, so they are invertible
objects, that is, simple currents \cite{CKLR}.
Since all vertex operator algebras under consideration are rational, tensor products and hence monodromy factor over the tensor product of vertex operator algebras.
This immediately implies that the monodromy of $V_{\nu+L}\otimes C$ with $V_{\mu+L}\otimes M_{0,\mu}$ is trivial for any 
$\mu\in N$ and  $\nu\in N'$, which in turn implies that $\cF(V_{\nu+L}\otimes C)$ is a module in $\rep^0 V$.
\end{proof}

We have the following isomorphisms among the $M_{i,\lambda_i+\nu}$:
\begin{theo}\label{thm:cosetModuleClassification}
For $\lambda+L,\mu+L\in L'/L$ and $0\leq i,j\leq n$,
$M_{i, \lambda} \cong M_{j, \mu}$ if and only if $\mu-\lambda \in N'$ and $M^{\mu-\lambda} \boxtimes_V M_i \cong M_j$.
In particular, $M_{i,\lambda}\cong M_{j,\lambda}$ if and only if $i=j$.
\end{theo}
\begin{proof}
The ``if'' direction follows from the induction functor and the ``only if'' direction
follows from the observation \eqref{eqn:tenslift} and Proposition \ref{Mupperuntwisted}.
\end{proof}

Recall that for each $M_i$, we have fixed a choice of $\lambda_i+L\in L'/L$.
We define an equivalence relation on the set $\{ (i,\lambda_i+\nu)\,|\, 0\leq i\leq n, \nu+L\in N/L \}$:
\[
\  (i,\psi) \sim (j,\phi) \ \leftrightarrow M_{i,\psi} \cong M_{j,\phi},
\]
and we let $S$ be the set of equivalence classes under $\sim$. \begin{corol}\label{cor:numberofcosetmods}
The number of inequivalent simple $C$-modules is
\[\vert S\vert = \frac{(n+1) |N/L|}{|N'/L|}. 
\]
\end{corol}
\begin{proof}
There are $n+1$ inequivalent simple $V$-modules, all obtained by inducing $V_L\otimes C$-modules. For each simple $C$-module there exist $|N'/L|$ inequivalent simple $V_L$-modules such that their tensor product induces to a (simple) $V$-module. Moreover, by by Proposition \ref{prop:SimpleCurrentSmatrix}, each simple $V$-module can be obtained by inducing  any one of $|N/L|$ inequivalent simple   $V_L\otimes C$-modules .
\end{proof}

\subsubsection{Fusion}\label{subsubsec:cosetfusion}

The coset decomposition gives a relation between fusion rules.
We define fusion coefficients for both $V$- and $C$-modules as usual:
\[
M_i \boxtimes_V M_j \cong \bigoplus_{k=0}^n N(V)_{ij}^k M_k,
\]
and for $(i,\nu_i)$, $(j,\nu_j)\in S$, 
\[
M_{i,\nu_i} \boxtimes_C M_{j,\nu_j} \cong \bigoplus_{(k, \nu_k)\in S} N(C)_{(i,\nu_i)(j,\nu_j)}^{(k,\nu_k)} M_{k,\nu_k}
\]
where the latter sum is over $S$, the set of inequivalent $C$-modules, that is, we use the identification of the last theorem. 
Since intertwining operators preserve Heisenberg weights, we have
\[
N(V)_{ij}^k \neq 0 \Longrightarrow \nu_k-\nu_i-\nu_j+L \in  N/L. 
\]
\begin{theo}\label{thm:cosetfusion}
Fusion rules are related as follows:
\begin{align*}
&M_{i,\mu_i} \boxtimes_C M_{j,\mu_j} 
\cong\bigoplus_{\substack{k=0 
		}
  }^n N(V)_{ij}^k\, M_{k, \mu_i+\mu_j},
\quad   \text{(here,\, the\, summands\, are\, pairwise\, inequivalent)}\\
&  M_i \boxtimes_V M_j
\cong \bigoplus_{\substack{ \mu+L \in N'/L \\ (k, \lambda_i+\lambda_j+\mu)\in S}}
  N(C)_{(i,\lambda_i)(j,\lambda_j)}^{(k,\lambda_i+\lambda_j+\mu)} 
M^{-\mu} \boxtimes_{V}M_k.
\end{align*}
\end{theo}
\begin{proof}
  First we prove $N(C)^{(k,\mu_i+\mu_j)}_{(i,\mu_i) (j,\mu_j)} =N(V)^k_{i j}$ using the following natural maps between spaces of
  intertwining operators:
\begin{align}
\cY_C\binom{M_{k,\mu_i+\mu_j}}{M_{i,\mu_i}\,M_{j,\mu_j}}
&\xrightarrow{\cong}
\cY_{V_L\otimes C}
\binom{V_{\mu_i+\mu_j+L}\otimes M_{k,\mu_i+\mu_j}}
{V_{\mu_i+L}\otimes M_{i,\mu_i}\,V_{\mu_j+L}\otimes M_{j,\mu_j}}
\nonumber\\
&\xrightarrow{\cong} \hom_{V_L\otimes C}
\left( (V_{\mu_i+L}\otimes M_{i,\mu_i})\fus{V_L\otimes C} 
(V_{\mu_j+L}\otimes M_{j,\mu_j}), V_{\mu_i+\mu_j+L}\otimes M_{k,\mu_i+\mu_j}  \right)
\nonumber\\
&\xhookrightarrow{\cF} \hom_{V}
\left( \cF((V_{\mu_i+L}\otimes M_{i,\mu_i})\fus{V_L\otimes C} 
  (V_{\mu_j+L}\otimes M_{j,\mu_j})),\cF(V_{\mu_i+\mu_j+L}\otimes M_{k,\mu_i+\mu_j})  \right)
  \nonumber\\
&\xrightarrow{\cong} \hom_{V}
\left( \cF(V_{\mu_i+L}\otimes M_{i,\mu_i})\fus{V} 
  \cF(V_{\mu_j+L}\otimes M_{j,\mu_j}),
  \cF(V_{\mu_i+\mu_j+L}\otimes M_{k,\mu_i+\mu_j})  \right)
  \nonumber\\
&\xrightarrow{\cong}
\hom_V(M_i\fus{V}M_j,M_k)  \xrightarrow{\cong} \cY_{V}\binom{M_k}{M_i\,M_j}
\label{eqn:isomliftfusion}
\end{align}
The first map is a bijection by \cite[Theorem 2.10]{ADL}.  
Since all
spaces of intertwining operators here are finite dimensional (we
are working in a $C_2$-cofinite setting), it suffices to give an
injection in the reverse direction.  By analyzing Heisenberg weights, any non-zero intertwining operator in
$ \cY_{V}\binom{M_k}{M_i\,M_j}$ when restricted to
$(V_{\mu_i+L}\otimes M_{k,\mu_i})\otimes(V_{\mu_j+L}\otimes
M_{k,\mu_j})$ must have coefficients in
$V_{\mu_i+\mu_j+L}\otimes M_{k,\mu_i+\mu_j}$.  Since $M_i$ and $M_j$ are
simple $V$-modules,  \cite[Proposition 11.9]{DL} implies that the restriction $ \cY_{V}\binom{M_k}{M_i\,M_j}\rightarrow \cY_{V_L\otimes C}
\binom{V_{\mu_i+\mu_j+L}\otimes M_{k,\mu_i+\mu_j}}
{V_{\mu_i+L}\otimes M_{i,\mu_i}\,V_{\mu_j+L}\otimes M_{j,\mu_j}}$ is an injection. Now, $\cY_C\binom{M_{k,\mu_k}}{M_{i,\mu_i}\,M_{j,\mu_j}}=0$ unless $M_{k,\mu_k}\cong M_{k',\mu_i+\mu_j}$ for some $k'$ using the composition of the first three injections in \eqref{eqn:isomliftfusion} (with $\mu_i+\mu_j$ replaced by $\mu_k$ in $C$-modules but not in $V_L$-modules), the observation \eqref{eqn:tenslift}, and Theorem \ref{thm:cosetModuleClassification}. Moreover, the modules $M_{k,\mu_i+\mu_j}$ for $0\leq k\leq n$ are distinct by Theorem \ref{thm:cosetModuleClassification}. This gives the first conclusion of the theorem.

For the second statement,
\begin{align*}
M_i\fus{V}M_j &\cong \cF(V_{\lambda_i+L}\otimes M_{i,\lambda_i})\fus{V}\cF(V_{\lambda_j+L}\otimes M_{j,\lambda_j})\\
&\cong \cF( ( V_{\lambda_i+L}\otimes M_{i,\lambda_i} ) \fus{V_L\boxtimes C} (V_{\lambda_j+L}\otimes M_{j,\lambda_j}))\\
&\cong \bigoplus_{ \lambda+L\in L'/L,\, (k,\nu) \in S}
 N(V_L\otimes C)_{V_{\lambda_i+L}\otimes M_{i,\lambda_i},\, V_{\lambda_j+L}\otimes M_{j,\lambda_j} }^{V_{\lambda+L}\otimes M_{k,\nu}}
\cF(V_{\lambda+L}\otimes M_{k,\nu})\\
&\cong \bigoplus_{\lambda+L\in L'/L,\, (k,\nu) \in S}
N(V_L)_{V_{\lambda_i+L} V_{\lambda_j+L}}^{V_{\lambda+L}} N(C)_{(i,\lambda_i)\,(j,\lambda_j)}^{(k,\nu)}
\cF(V_{\lambda+L}\otimes M_{k,\nu})\\
&\cong \bigoplus_{(k,\nu)\in S}
N(C)_{(i,\lambda_i)\,(j,\lambda_j)}^{(k,\nu)}
\cF(V_{\lambda_i+\lambda_j+L}\otimes M_{k,\nu}).
\end{align*}
The module on the left side is in $\rep^0 V$, but 
from \eqref{eqn:tenslift}, $\cF(V_{\lambda_i+\lambda_j+L}\otimes M_{k,\nu})$
is not in $\rep^0 V$ unless $\nu-\lambda_i-\lambda_j\in N'$.
Thus
\begin{align*}
M_i\fus{V}M_j &\cong \bigoplus_{(k,\nu)\in S}
N(C)_{(i,\lambda_i)\,(j,\lambda_j)}^{(k,\nu)}
\cF(V_{\lambda_i+\lambda_j+L}\otimes M_{k,\nu})\\
&\cong \bigoplus_{\substack{\mu+L\in N'/L, \\ (k,\lambda_i+\lambda_j+\mu)\in S}}
N(C)_{(i,\lambda_i)\,(j,\lambda_j)}^{(k,\lambda_i+\lambda_j+\mu)}
\cF(V_{\lambda_i+\lambda_j+L}\otimes M_{k,\lambda_i+\lambda_j+\mu})\\
&\cong \bigoplus_{\substack{\mu+L\in N'/L, \\ (k,\lambda_i+\lambda_j+\mu)\in S}}
N(C)_{(i,\lambda_i)\,(j,\lambda_j)}^{(k,\lambda_i+\lambda_j+\mu)}
M^{-\mu}\fus{V}M_k
\end{align*}
as desired
\end{proof}

For Heisenberg cosets
of affine Lie algebra \VOAs{} at positive integral levels,
these results have appeared previously in \cite{DLWY,DR,DW3,ADJR}.

\subsubsection{Modular transformations and characters}\label{subsec:cosetmodtrans}

A $V$-module $M_i$ is graded by conformal weights as well as the lattice $L'$:
\[
M_i = \bigoplus_{\substack{n \in h_i+\ZZ \\ \lambda\in L'}} M_{n, \lambda}.
\]
The character of $M_i$ can be refined to a component of a vector-valued Jacobi form by
\[
\text{ch}[M_i]\left(u, \tau\right):= q^{-\frac{c_V}{24}}\sum_{\substack{n \in h_i+\ZZ \\ \lambda\in L'}} \text{dim}\left(M_{n, \lambda}\right) q^n z^\lambda
\]
with $q=e^{2\pi i \tau}$ as usual and $z^\lambda:= e^{2\pi i \langle u, \lambda\rangle}$ for $u$ in the complexification of $L'$ and $\langle\cdot,\cdot\rangle$ the bilinear form of $L'$. 
Modularity in this setting is due to \cite{KMa} while the modularity of ordinary traces as for example those of $C$-modules is the famous result of Zhu \cite{Zhu}.

 We thus define the modular $S$- and $T$-matrices as follows:
\[
\text{ch}[M_i](u, \tau+1) = T^{\chi,V}_i  \text{ch}[M_i](u, \tau)\qquad\text{and}\qquad
\text{ch}[M_i]\left(\frac{u}{\tau},-\frac{1}{\tau}\right) = e^{\pi \frac{u^2}{\tau}} \sum_{j=0}^n S^{\chi,V}_{i,j}  \text{ch}[M_j](u, \tau),
\]
as well as
\[
\text{ch}[M_{i,\nu_i}](\tau+1) = T^{\chi,C}_{i,\nu_i}  \text{ch}[M_{i,\nu_i}](\tau),\qquad\text{and}\qquad
\text{ch}[M_{i,\nu_i}]\left(-\frac{1}{\tau}\right) =  \sum_{(j, \nu_j)\in S} S^{\chi,C}_{(i,\nu_i),(j,\nu_j)}  \text{ch}[M_{(j,\nu_j)}](\tau).
\]
\begin{theo}\label{thm:ST}
The modular $S$- and $T$-matrices are related as follows:
\[
T^{\chi,V}_{i} =  e^{\pi i \left(\langle \nu_i,\nu_i\rangle - \frac{\rank(L)}{12}\right)} T^{\chi,C}_{i,\nu_i} \qquad \text{and}\qquad
S^{\chi,V}_{i,j} = e^{2\pi i \langle\nu_i, \nu_j\rangle}
 \frac{|N/L|}{\sqrt{|L'/L|}} S^{\chi,C}_{(i,\nu_i),(j,\nu_j)} .
\]
\end{theo}
\begin{proof}
  The $T$-matrix relation follows from the definitions, the fact that induction preserves twists and hence conformal dimensions mod $\ZZ$, and the fact that central charges satisfy $c_V=c_{V_L\otimes C}=\mathrm{rank}(L)+c_C$.

For the $S$-matrix statement,  Theorem \ref{prop:liftingS} shows
that for modules $\cF(B), \cF(C)$ in  $\rep^0 V$,
\[
S^{\hopflink}_{\cF(B),\cF(C)} = \cF(S^\hopflink_{B,C}) 
\]
where the left side is the Hopf link in Rep$^0\,V$ and the
right side is the Hopf link in $\mathcal C$, the category of $V_L\otimes C$-modules.  
The modular and categorical
$S$-matrices are proportional \cite[Theorem 4.5]{H-rigidity}, 
with proportionality constant the dimension $D(\mathcal{C})$ of the category, that is,
\[
S^{\chi}=\dfrac{1}{D(\mathcal{C})}S^{\hopflink}.
\]
\cite[Theorem 4.5]{KO} shows that $\rep^0V$ is a modular tensor category and that
\[D(\rep^0 V)=D({\mathcal C})/\qdim_{\mathcal C}(V).\]	
Note that \cite[Theorem 3.4]{HKL} and our main result, \cref{thm:repzABTCvrtxSalg}, show that $\rep^0\,V$ is precisely
the modular vertex tensor category of untwisted modules for the regular \VOA{} $V$. 

Next, the dimensions of all simple currents appearing in a \VOA{} extension of correct statistics 
are $1$. Indeed, fix such a simple current, say $J$. Then, $\theta_J=\id_J$
and moreover, skew-symmetry considerations coupled with \eqref{formalbraidchar} 
show that $\cR_{J,J}=\id_{J\boxtimes J}$. Now, the spin statistics theorem (\cite{CKL}, 
\cite[Exercise 8.10.15]{EGNO}) shows that $\qdim(J)=1$.
We therefore have $\text{qdim}_{\mathcal C}(V) =|N/L|$,
since $\qdim$ is additive over direct sums.

Finally,  we have:
\begin{align*}
S^{\chi,V}_{i,j}&=\dfrac{1}{D(\rep^0V)}S^{\hopflink}_{i,j}
=\dfrac{\qdim_{\mathcal{C}}(V)}{D(\mathcal{C})}S^{\hopflink}_{i,j}
=\dfrac{|N/L|}{D(\mathcal{C})}S^{\hopflink}_{i,j}
=\dfrac{|N/L|}{D(\mathcal{C})}S^{\hopflink}_{V_{\nu_i+L}\otimes M_{(i,\nu_i)},V_{\nu_j+L}\otimes M_{(j,\nu_j)}}\\
&= {|N/L|}\cdot S^{\chi,V}_{V_{\nu_i+L}\otimes M_{(i,\nu_i)},V_{\nu_j+L}\otimes M_{(j,\nu_j)}}
={|N/L|}\cdot S^{\chi,V_L}_{V_{\nu_i+L},V_{\nu_j+L}}
\cdot	 S^{\chi,C}_{{(i,\nu_i)},{(j,\nu_j)}}\\
&={|N/L|}\cdot\frac{
	 	e^{2\pi i \langle \nu_i,\nu_j\rangle}
}{\sqrt{|L'/L|}}
\cdot	 S^{\chi,C}_{{(i,\nu_i)},{(j,\nu_j)}}.
\end{align*}
Note that $S^{\chi, V_L}_{V_{\nu_i+L}, V_{\nu_j+L}}$ can be calculated as follows. First,  $S^{\hopflink, V_L}_{V_{\nu_i+L}, V_{\nu_j+L}}=e^{2\pi i\langle\nu_i,\nu_j\rangle}$ by calculating monodromies. Then by \cite[Theorem 4.5]{H-rigidity}, this differs from $S^{\chi, V_L}_{V_{\nu_i+L}, V_{\nu_j+L}}$ only by a proportionality constant which we can calculate using well-known modularity properties of lattice theta functions.
\end{proof}

Finally let us compute the relation of the associated characters.
\begin{theo}\label{thm:char}
Characters are related as follows:
\[
  \ch[M_i](u, \tau)
  = \sum_{\nu+L\in  N/L}\frac{\theta_{\lambda_i+\nu+L}(u,\tau)}
  {\eta(\tau)^{\rank(L)}}  \ch[M_{i,\lambda_i+\nu}](\tau)
\]
and
\[
  \ch[M_{i,\nu_i}](\tau)
  = \frac{\eta(\tau)^{\rank(L)}}{|L'/L|\,\theta_{\nu_i+L}(0, \tau) }
  \sum_{\gamma+L\in L'/L} e^{-2\pi i \langle\nu_i,\gamma\rangle}\ch[M_i](\gamma, \tau)
\]
\end{theo}
\begin{proof}
  The first statement is just the character version of the module
  decomposition. For the second, we use
\[
  \theta_{\lambda+L}(u+\gamma, \tau) = e^{2\pi i \langle\gamma, \lambda\rangle}
  \theta_{\lambda+L}(u, \tau)
\]
together with the orthogonality relations for irreducible characters of $L'/L$,
\[
  \sum_{\gamma+L \in  L'/L} e^{2\pi i \langle\gamma,\mu -\lambda\rangle}
  = |L'/L| \delta_{\mu+L , \lambda+L} 
\]
for $\mu+L,\lambda+L\in L'/L$,
so that we get the projection
\[
  \delta_{\mu+L, \lambda+L}\theta_{\lambda+L}(u, \tau) = \frac{1}{|L'/L|}
  \sum_{\gamma+L \in  L'/L}e^{-2\pi i \langle\gamma, \lambda\rangle}  \theta_{\mu+L}(u+\gamma, \tau). 
\]
Then we calculate using the first statement of the theorem:
\begin{align*}
 \dfrac{1}{\vert L'/L\vert} & \sum_{\gamma+L\in L'/L} e^{-2\pi i\langle\nu_i,\gamma\rangle}  \ch[M_i](u+\gamma,\tau)\nonumber\\
 &= \dfrac{1}{\vert L'/L\vert} \sum_{\gamma+L\in L'/L}\sum_{\nu+L\in N/L} e^{-2\pi i\langle\nu_i,\gamma\rangle} \dfrac{\theta_{\nu_i+\nu+L}(u+\gamma,\tau)}{\eta(\tau)^{\mathrm{rank}(L)}} \ch[M_{i,\nu_i+\nu}](\tau)\nonumber\\
 & =\sum_{\nu+L\in N/L} \delta_{\nu_i+\nu+L,\nu_i+L} \dfrac{\theta_{\nu_i+L}(u,\tau)}{\eta(\tau)^{\mathrm{rank}(L)}} \ch[M_{i,\nu_i+\nu}](\tau) = \dfrac{\theta_{\nu_i+L}(u,\tau)}{\eta(\tau)^{\mathrm{rank}(L)}}\ch[M_{i,\nu_i}](\tau).
\end{align*}
Setting $u=0$ yields the second conclusion of the theorem.
\end{proof}

\subsubsection{Solving the coset problem}

The above results give a precise way to understand $C$ if $V$ is
sufficiently well understood:
\begin{enumerate}
\item Start with a simple, rational, $C_2$-cofinite \VOA{} $V$ that
  contains a lattice \VOA{} $V_L$ where $L$ is positive definite and
  even.  Prove that $V_L$ is maximal or find the maximal extension of
  $V_L$ contained in $V$.
\item Decompose $V$ as a $V_L\otimes C$-module:
\[
V = \bigoplus_{\mu+L\in  N/L} V_{\mu+L} \otimes M_{0,\mu}.
\]
This can be done on the level of characters since Jacobi theta functions
of lattices are linearly independent. This determines $N$
and hence $N'$.
\item We have shown that there are certain simple current $V$-modules $M^{\mu}$
  for $\mu\in N'/L$ which form an abelian
  intertwining algebra of type $N'/L$. These determine the set $S$ of
  inequivalent simple $C$-modules: all we need to know are the
  fusion rules of the $M^\mu$ with any of the $M_i$. 
\item Get characters, fusion rules, and $S$- and $T$-matrices using the above theorems. 
\end{enumerate}
The only subtlety arising here is that fixed points of the
$M^\mu$ make enumeration of the inequivalent simple $C$-modules more
technical. When $V$ is a simple affine \VOA, $N'/L$ is given by
spectral flow (automorphisms coming from affine Weyl translations)
which gives explicit character relations, so this is actually
verifiable on the level of characters.

\subsubsection{Solving the inverse coset problem}

Now assume that $C$ is sufficiently well understood. Can we
reconstruct $V$? This question actually fits very well in a program
one of us has developed together with A. Linshaw \cite{CL2, CL1}. In \cite{CL1}, we
found minimal strong generating sets of families of coset \VOA{}s $C_k$ inside larger structures $A_k$. If the strong generator type
coincides with the one of a known $W$-algebra, then one might ask if they
are isomorphic. In practice, there will be special values of
the deformation parameter $k$ for which $A_k$ is believed to be interesting, for example admissible levels of $W$-algebras \cite{KW, Ar2}. In this case a strategy might be to find the coset, and
use its properties to deduce properties of $V$:
\begin{enumerate}
\item Suppose we have a \VOA{} $V$ and we believe (or might even know) that $V_L$ and $C$ form a commuting pair inside $V$. 
Then use \cite{CKLR} to find a lattice $N\supseteq L$ such that 
\[
\bigoplus_{\mu+L\in N/L} V_{\mu+L} \otimes C_\mu
\]
is a \VOA. Here the $C_\mu$ must be simple current $C$-modules that give rise to an abelian intertwining algebra of type $N/L$. There are usually very few choices for that which means there are few possibilities for $N$. 
\item Prove that 
\[
V\cong \bigoplus_{\mu+L\in N/L} V_{\mu+L} \otimes C_\mu.
\]
The strategy is to prove that the operator product algebra of a given type is unique and hence both sides are homomorphic images of the same \VOA{} and hence are isomorphic if both of them are simple. This strategy is feasible in many cases and has been succesfully applied in the setting of minimal $W$-algebras, \cite{ACL, ACKL}. It would be nice to extend this further, as for example settling \cite[Conjecture 4.3]{CKL}.
\item Use our results, that is, induction, to determine all $V$-modules and compute their fusion rules and modular properties as above. Recall from Remark \ref{rem:isomod} that simple $V$-modules are isomorphic if and only if they are isomorphic as $V_L\otimes C$-modules.
\end{enumerate}
This inverse coset construction has now been efficiently applied to understand the representation theory of the simple affine \VOSA{} $L_k\left(\mathfrak{osp}(1|2)\right)$ at positive integer level $k$ as an extension of $L_k(\mathfrak{sl}_2) \otimes \vir(k+2, 2k+3)$ \cite{CFK}. Further examples of understanding more subtle affine \VOSAs{} via the inverse coset construction are under investigation.

\subsection{Coset theories: Examples}

In this subsection we use the theory and results of previous sections to study the representations of some regular vertex operator (super)algebras arising as lattice cosets. The examples presented here illustrate how induction can be used to solve both coset and inverse coset problems.

\subsubsection{Type $A$ parafermion algebras}

Parafermion algebras are cosets of Heisenberg (equivalently, lattice) vertex subalgebras of affine vertex operator algebras. They first appeared in \cite{LW, DL} and have received attention in number theory due to their connection with integer partitions \cite{LW}. As they have been treated in many fairly recent works \cite{ALY, DLY, DLWY, DR, DW1, DW2, DW3}, we will not derive new results here. However, these algebras nicely illustrate the efficiency of the theory of vertex operator algebra extensions. We focus on parafermion algebras of type $A$. 

Fix $n\geq 1$ and
consider the simple affine \VOA{} $L_{k}(\mathfrak{sl}_{n+1})$ of type $A_n$ at positive integer level
$k$. The Cartan subalgebra of $\mathfrak{sl}_{n+1}$, embedded into the conformal weight-$1$ subspace of $L_{k}(\mathfrak{sl}_{n+1})$, generates a Heisenberg vertex subalgebra with different conformal vector. The maximal lattice vertex subalgebra of $L_{k}(\mathfrak{sl}_{n+1})$ extending this Heisenberg algebra is $V_L$ for $L=\sqrt{k}A_n$, and the parafermion algebra is then defined to be the coset $C=\mathrm{Com}(V_L, L_{k}(\mathfrak{sl}_{n+1}))$. The lattice determining the decomposition of $L_{k}(\mathfrak{sl}_{n+1})$ as a $V_L\otimes C$-module is $N=\frac{1}{\sqrt{k}} A_n$.

The inequivalent simple $L_{k}(\mathfrak{sl}_{n+1})$-modules are the simple highest-weight
$\widehat{\mathfrak{sl}}_{n+1}$-modules $L(\Lambda)$ of highest weight
$\Lambda=\lambda_0\omega_0 + \dots + \lambda_n\omega_n$, where the
$\omega_i$ are the fundamental weights of $\widehat{\mathfrak{sl}}_{n+1}$ and $\lambda_0+\cdots+\lambda_n=k$; we denote the set
of such weights by $P_k^+$. Recall from \eqref{eqn:Asimplecurrents} that $L_{k}(\mathfrak{sl}_{n+1})$ has simple currents parametrized by cosets in
\begin{equation*}
 N'/L=\sqrt{k} A_n'/\sqrt{k} A_n \cong A_n'/A_n\cong\ZZ/(n+1)\ZZ.
\end{equation*}
Thus these must be all $n+1$ of the simple current $L_{k}(\mathfrak{sl}_{n+1})$-modules classified in \cite{F}: they are the modules $L(k\omega_i)$, $0\leq i\leq n$, with fusion rules
\[
  L(k\omega_i) \boxtimes L(k\omega_j)
  = L(k\omega_{\overline{i+j}}), \qquad \overline{i+j} =
  \begin{cases}
    i+j & \text{if}\ i+j\leq k
    \\ i+j-k & \text{else}
  \end{cases}
\]
and action on simple modules given by
\[
L(k\omega_j) \boxtimes L(\Lambda) = L(s_j(\Lambda))
\]
where the $s_j$ cyclically permute the fundamental weights,
that is, $s_j(\omega_i)=\omega_{\overline{i+j}}$.  This is the same as
addition of weight modulo $kA_n$.  


Now given $\Lambda=\lambda_0\omega_0+\cdots+\lambda_n\omega_n\in P_k^+$,
let $\overline{\Lambda}
=\frac{1}{\sqrt{k}}(\lambda_1\omega_1+\cdots+\lambda_n\omega_n)\in L'$. Recalling \eqref{eqn:Midecomp}, simple
modules of the parafermion coset are labeled $M_{\overline{\Lambda}, \Gamma}$
 for $\Lambda\in P_k^+$, $\Gamma+L\in L'/L$,
with the condition that $\overline{\Lambda}-\Gamma \in 
N/L=\frac{1}{\sqrt{k}} A_n/ \sqrt{k}A_n$.
Further, by Theorem \ref{thm:cosetModuleClassification},
\begin{equation}
  M_{\overline{\Lambda}, \Gamma} \cong M_{\overline{\Lambda'}, \Gamma'} \ \
  \text{if and only if there exists}\ 0\leq j\leq n\ \text{such that}
  \ \ \Lambda'=s_j(\Lambda) \ \ \text{and} \ \ \Gamma'=
  \Gamma+\sqrt{k}\omega_j.\label{eqn:affinecoseteqv}
  \end{equation}

Fusion rules for $L_k(\mathfrak{sl}_{n+1})$-modules can in principal be computed using
Verlinde's formula, but the combinatorics is involved. For $A_1$
and $A_2$ explicit formulae are known (see for instance \cite{FMS, BSVW}). For $A_1$, $P_{k}^+=\{((k-i)\omega_0+i\omega_1)\,|\, 0\leq i \leq k   \}$, and
abbreviating $(k-i)\omega_0 + i\omega_1$ by $\pi(k,i)$, we have
\begin{align*}
L_{\mathfrak{sl}_2}(\pi(k,i))\fus{}L_{\mathfrak{sl}_2}(\pi(k,j))
  =\sum_{\substack{|i-j|\leq\,t\,\leq\text{min}(i+j,2k-i-j)  \\ i+j+t\equiv 0\,\text{mod}\,2}}
L_{\mathfrak{sl}_2}(\pi(k,t)).
\end{align*}
Therefore, from Theorem \ref{thm:cosetfusion}, we
precisely recover the fusion rules of \cite{DW3}:
for $\nu_i+L,\nu_j+L\in N/L$,
\begin{align*}
  M_{\overline{\pi(k,i)},\, \overline{\pi(k,i)} + \nu_i}
  \fus{C}M_{\overline{\pi(k,j)},\, \overline{\pi(k,j)} + \nu_j}
  =\bigoplus_{\substack{|i-j|\leq\,t\,\leq\text{min}(i+j,2k-i-j)
  \\ i+j+t\equiv 0\,\text{mod}\,2}}
  M_{\overline{\pi(k,t)},\, \overline{\pi(k,i)} + \overline{\pi(k,j)} + \nu_i+\nu_j} .
\end{align*}
Note that when $i+j+t$ is even, $
\overline{\pi(k,t)}-\overline{\pi(k,i)}-\overline{\pi(k,j)}\in N$,
so every summand is well-defined.
Also, from \eqref{eqn:affinecoseteqv},
it is clear that the summands are mutually inequivalent.

Type $A_1$ parafermion characters can be read off from Theorem \ref{thm:char}, and modular character transformations from Theorem \ref{thm:ST}. The $T$-matrix is easy:
%

\[
T^{\chi, C}_{M_{\overline{\pi(k,a)},\, \overline{\pi(k,a)} + \nu}} 
= \exp\left(-\pi i \left\langle \overline{\pi(k,a)} + \nu , \overline{\pi(k,a)} + \nu \right \rangle
+\frac{\pi i}{12}
+2\pi i\left( \dfrac{a(a+2)}{4(k+2)}-\dfrac{k}{8(k+2)}\right)\right)
\]
For the $S$-matrix, for brevity, set
\[
  W_a = M_{\overline{\pi(k,a)},\, \overline{\pi(k,a)} + \nu_a},\quad
  W_b = M_{\overline{\pi(k,b)},\, \overline{\pi(k,b)} + \nu_b},
\]
where $\nu_a+L,\nu_b+L\in N/L$. Then,
\[
  S^{\chi,C}_{W_a,\,W_b} = \sqrt{\frac{2}{k}}
  \cdot \sqrt{\frac{2}{k+2}}
  \cdot \sin\left(
  \frac{\pi(a+1)(b+1)}{k+2}\right)
  \cdot \exp\left(-2\pi i 
  \left\langle
    \overline{\pi(k,a)} + \nu_a
    ,
    \overline{\pi(k,b)} + \nu_b\right\rangle 
     \right).
\]

\begin{rema}
It is reasonable to conjecture that the category of admissible-level $L_k(\mathfrak{sl}_2)$-modules as stated in \cite{CR2, AM3}
has the structure of a vertex tensor category. If this is true, then one can also study the parafermion cosets of admissible-level affine \VOAs{} in a manner very similar to the rational case. This is interesting since these parafermion cosets allow for large extensions that are conjecturally $C_2$-cofinite but not rational \cite{ACR}.
\end{rema}

\subsubsection{Rational $N=2$ superconformal algebras $\mathcal L_k$}\label{sec:N=2}

In this section, we discuss a family of \VOSA{}s of correct statistics: the rational $N=2$ superconformal algebras. They are important in superstring theory, where the chiral algebra of the underlying conformal field theory must be the $N=2$ superconformal algebra or an extension. Thus they appeared in the physics literature already in the 1980s; see for example \cite{DPZ}. In the mathematics literature, modules were classified in \cite{Ad1} and fusion rules in \cite{Ad2, Wak2}; we rederive these results here. Interesting results are also possible beyond rationality \cite{Sa}. 
 
 Fix a positive integer level $k$; the $N=2$
superconformal algebra $\mathcal L_k$ is rational at such levels with central charge
$c=3k/(k+2)$. Although $\mathcal{L}_k$ is normally defined as a vertex operator superalgebra structure on the simple level-$k$ vacuum module for the $N=2$ superconformal Lie superalgebra (see for example \cite[Section 1]{Ad1}), we shall here use a lattice coset realization of $\mathcal{L}_k$ due to \cite{CL1}. To begin, we need to introduce some lattices, using notation as in \cite{CKLR, CL1}:
\[
L_\alpha = \ZZ\alpha, \qquad L_\beta=\ZZ\beta,
\qquad L_\gamma=\ZZ\gamma, \qquad L_\mu=\ZZ\mu
\]
with
\[
\alpha^2=2k, \qquad \beta^2=4, \qquad \alpha\beta=0, \qquad 
\gamma=\alpha+\frac{k}{2}\cdot\beta, 
\qquad \mu=\alpha-\beta
\]
so that 
\begin{align*}
\mu^2=2(k+2), \qquad \gamma^2=k(k+2), \qquad\gamma\mu=0,\\
\alpha = \frac{2}{k+2} \gamma + \frac{k}{k+2}\mu, \qquad
\beta = \frac{2}{k+2} \gamma - \frac{2}{k+2}\mu.
\end{align*}
Let $a\alpha+b\beta$ in $L'_\alpha\oplus L'_\beta$, which implies that
$a\in \dfrac{1}{2k}\ZZ$, $b\in\dfrac{1}{4}\ZZ$.
We have
\begin{equation}\label{eq:latticevectors}
  a\alpha+b\beta= \frac{ka-2b}{k+2}\mu + \frac{2(a+b)}{k+2}\gamma
  \in  L'_\mu\oplus \frac{1}{2}L'_\gamma. 
\end{equation}
Now we can present the coset realization of $\mathcal{L}_k$ from \cite{CL1}:
\[
  \mathcal L_k = \com\left( V_{L_\mu}, L_k(\mathfrak{sl}_2)
    \otimes V_{L_{\frac{1}{2}\beta}}\right),
\]
and the even vertex operator subalgebra of $\mathcal L_k$ is
 \[
   \mathcal L_k^{\even} = \com\left( V_{L_\mu},
     L_k(\mathfrak{sl}_2) \otimes V_{L_\beta}\right).  
\]
Our strategy is to first understand $\mathcal L_k^{\even}$-modules
 and then induce to $\mathcal L_k$.

Recall from the preceding parafermion example that $V_{L_\alpha}$ is a vertex subalgebra of
$L_k(\mathfrak{sl}_2)$, and the $V_{L_\alpha}$-modules appearing in the
decomposition of $L_k(\mathfrak{sl}_2)$ are exactly
$V_{\frac{n}{k}\alpha+L_\alpha}$ for $n=0, 1, \dots, k-1$. It follows
from \eqref{eq:latticevectors} that $V_{\nu+L_\mu}$ is a submodule of
$L_k(\mathfrak{sl}_2) \otimes V_{L_\beta}$ if and only if
$\nu+L\in N/L_\mu$ with $N=2L'_\mu$.  Further, we have
$N'=\frac{1}{2}L_\mu$ so $L_\mu$ has index two in
$N'$. 
Denote by $\omega_0, \omega_1$ the fundamental weights of
$\widehat{\mathfrak{sl}}_2$ so that the simple
$L_k(\mathfrak{sl}_2)$-modules are precisely the integrable highest-weight
modules $L(\lambda)$ with highest weight $(k-\lambda)\omega_0+ \lambda\omega_1$
for $\lambda\in\{ 0, 1, \dots, k\}$.  Using
\eqref{eq:latticevectors} it follows that for $b\in\ZZ$,
$V_{\nu+L_\mu}$ is a submodule of
$L(\lambda)\otimes V_{\frac{b}{4}\beta+L_\beta}$ if and only if
$\nu-\frac{\lambda-b}{2(k+2)}\mu\in N$ so that we get 
\begin{equation}
L(\lambda)\otimes V_{\frac{b}{4}\beta+L_\beta} \cong 
\bigoplus_{\nu-\frac{\lambda-b}{2(k+2)}\mu+L_{\mu}\in N/L_{\mu}} V_{\nu+L_\mu}
\otimes M(\lambda, b, \nu) 
\label{eqn:VLmuLeven}
\end{equation}
as $V_{L_\mu}\otimes \mathcal L_k^{\even}$-modules and further,
all $M(\lambda, b, \nu)$ are simple
$\mathcal{L}_k^{\even}$-modules and every simple
$\mathcal{L}_k^{\even}$-module is isomorphic to exactly two of
these since $|N'/L_{\mu}|$=2 (see  \cref{thm:cosetModuleClassification} and \cref{cor:numberofcosetmods}).
It remains to identify those that are isomorphic. The order-$2$ simple current of
$L_k(\mathfrak{sl}_2) \otimes V_{L_\beta}$ that identifies isomorphic
$\mathcal L_k^{\even}$-modules is
\begin{align*}
 M^{\frac{1}{2}\mu} = \cF(V_{\frac{1}{2}\mu+L_\mu}\otimes \mathcal{L}_k^{\even}) = \bigoplus_{\nu+L_\mu\in N/L_\mu} V_{\frac{1}{2}\mu+\nu+L_\mu}\otimes M(0,0,\nu)\cong L(\lambda)\otimes V_{\frac{b}{4}\beta+L_\beta}
\end{align*}
for some $\lambda\in\lbrace 0, 1,\ldots,k\rbrace$ and $b\in\lbrace 0,1,2,3\rbrace$. Since $L(\lambda)\otimes V_{\frac{b}{4}\beta+L_\beta}$ is a direct sum of the $V_{L_\alpha}\otimes V_{L_\beta}$-modules $V_{\left(\frac{\lambda}{2k}+\frac{n}{k}\right)\alpha+L_\alpha}\otimes V_{\frac{b}{4}\beta+L_\beta}$ for $0\leq n\leq k-1$ and since $\frac{1}{2}\mu=\frac{1}{2}\alpha-\frac{1}{2}\beta$, we have
\begin{equation*}
 \left(\dfrac{1}{2}-\dfrac{\lambda}{2k}-\dfrac{n}{k}\right)\alpha\in L_\alpha \quad \textrm{and} \quad \left(-\dfrac{1}{2}-\dfrac{b}{4}\right)\beta\in L_\beta.
\end{equation*}
From this it follows that $\lambda=k$ and $b=2$, that is,
\begin{equation*}
 M^{\frac{1}{2}\mu}\cong L(k)\otimes V_{\frac{1}{2}\beta+L_\beta}.
\end{equation*}
Then from Theorem \ref{thm:cosetModuleClassification} we get
\begin{equation}
M(\lambda, b, \nu) \cong M(\lambda', b', \nu') \quad \text{if and only if} \quad \lambda'=k-\lambda, \ \ \nu'\equiv\nu+\frac{1}{2}\mu\ (\text{mod}\ L_\mu), \ \ b'\equiv b+2\ (\text{mod}\ 4).
\label{eqn:N2modisom}
\end{equation}
Fusion rules and modular character transformations of these modules now follow from Theorems \ref{thm:cosetfusion}, \ref{thm:ST}, and \ref{thm:char}. Namely,
\begin{align*}
M(\lambda, b, \nu) \fus{\mathcal{L}_k^{\even}} M(\lambda', b', \nu') 
&\cong \bigoplus_{0\leq\lambda''\leq k} {(N^{\mathfrak{sl}_2})}_{\lambda, \lambda'}^{\lambda''} M(\lambda'', b+b', \nu+\nu')\\
&\cong \bigoplus_{\substack{|\lambda-\lambda'| \ \leq\   \lambda'' \ \leq\  \mathrm{min}(\lambda+\lambda' , 2k  - \lambda-\lambda') \\
		\lambda+\lambda'+\lambda'' \equiv 0\,\, \mathrm{mod}\, 2 }} 
 M(\lambda'', b+b', \nu+\nu'),
\end{align*}
In this decomposition, the summands are pairwise inequivalent.
Now we calculate the characters and modular data. Let $0\leq \lambda,\lambda'\leq k$, $b,b'\in\ZZ/4\ZZ$ and  $\nu,\nu'\in L'_{\mu}/L_{\mu}$.
Using Theorem \ref{thm:ST} we have:
\begin{align}
T&{}^{\chi}_{M(\lambda,b,\nu)}= \exp\left(2\pi i\left( \frac{\lambda(\lambda+2)}{4(k+2)} +\frac{b^2}{8} - \dfrac{1}{2}\langle \nu,\nu\rangle -\dfrac{k}{8(k+2)}\right)\right),\label{LevenT}\\
S&{}^{\chi}_{M(\lambda,b,\nu),\,M(\lambda',b',\nu')}
=\dfrac{\sqrt{|L'_\mu/L_\mu|}}{|N/L_\mu|}\cdot 
\exp\left(-2\pi i \langle\nu,\nu'\rangle \right) 
\cdot S^{\chi, L_k(\mathfrak{sl}_2)}_{\lambda,\ \lambda'}\cdot S^{\chi, V_{L_\beta}}_{V_{b\beta/4+L_\beta},\,V_{b'\beta/4+L_\beta}}
\nonumber\\
&=\dfrac{2}{k+2}\cdot 
\exp\left(-2\pi i \langle\nu,\nu'\rangle\right)
\cdot\sin\left(\frac{\pi(\lambda+1)(\lambda'+1)}{k+2}\right)
\cdot \frac{1}{2}
\exp\left(\pi i \frac{bb'}{2}\right)\nonumber\\
& =\dfrac{1}{k+2}\cdot\sin\left(\frac{\pi(\lambda+1)(\lambda'+1)}{k+2}\right)
\cdot\exp\left(2\pi i\left(\dfrac{bb'}{4}-\langle\nu,\nu'\rangle\right)\right)
\label{eqn:N2S}
\end{align}
From Theorem \ref{thm:char} and \eqref{eqn:VLmuLeven}, we have the following characters
(with the Jacobi variable $u$ denoting the Heisenberg weights corresponding to $V_{L_{\mu}}$):
\begin{align*}
\ch&[L(\lambda)\otimes V_{\frac{b}{4}\beta + L_{\beta}}](u,\tau)=
\sum_{\nu+L_\mu\in N/L_\mu}\dfrac{\theta_{\frac{\lambda-b}{2(k+2)}\mu + \nu + L_{\mu}}(u,\tau)}{\eta(\tau)}
\ch\left[M\left(\lambda,b,\frac{\lambda-b}{2(k+2)}\mu+\nu\right)\right](\tau)\\
\ch&\left[M\left(\lambda,b,\nu\right)\right](\tau)=\dfrac{\eta(\tau)}{2(k+2)\cdot\theta_{\nu + L_{\mu}}(0,\tau)}
\sum_{\gamma+L_\mu\in L_{\mu}'/L_{\mu}}e^{-2\pi i \left\langle\nu,\gamma \right\rangle}
\ch[L(\lambda)\otimes V_{\frac{b}{4}\beta + L_{\beta}}](\gamma,\tau).
\end{align*}

We can now study modules for the vertex operator superalgebra $\mathcal{L}_k$ via induction by viewing $\mathcal{L}_k$ as a superalgebra in the representation category of
$\mathcal L_k^{\even}$. In particular, induction yields a classification of untwisted and twisted simple $\mathcal{L}_k$-modules, as well as fusion rules. We also get $S$-transformations for (super)characters of simple $\mathcal{L}_k$-modules as in Section \ref{CSsuperalgVerlinde}.

From the fusion coefficients, $S$- and $T$-matrices, or otherwise,
it is clear that $J = M(0,2,0)$ is a simple current for 
$\mathcal L_k^{\even}=M(0,0,0)$
such that $\theta_J = -1_J$ and $\qdim(J)=1$.
Therefore, $\mathcal{L}_k=M(0,0,0)\oplus M(0,2,0)$ is a $\frac{1}{2}\ZZ$-graded vertex operator superalgebra with even part graded by $\ZZ$ and odd part graded by $\frac{1}{2}+\ZZ$, that is, $\mathcal{L}_k$ is a vertex operator superalgebra of correct statistics.
Further, the fusion of the simple current $M(0, 2, 0)$ is 
\[
M(0, 2, 0) \boxtimes_{\mathcal{L}^\even_k} M(\lambda, b, \nu) \cong M(\lambda, b+2, \nu)
\]
so that the induced modules are
\[
  \cF\left(M(\lambda, b, \nu)\right)
  = M(\lambda, b, \nu) \oplus M(\lambda, b+2, \nu).
\]
From \eqref{eqn:N2modisom} we conclude that $M(\lambda, b, \nu) \cong M(\lambda, b+2, \nu)$ is impossible since $0\not\equiv \frac{1}{2}\mu\,(\text{mod}\,L_\mu)$.
This implies that there are no twisted modules of fixed-point type, so by Proposition \ref{prop:SimpleCurrentSmatrix}, each induced module is simple, and moreover every simple $\mathcal{L}_k$-module is induced from a simple $\mathcal{L}_k^\even$-module. It is easy to decide whether induced modules are twisted or untwisted using the twist, that is, conformal dimension. From the $T$-matrix \eqref{LevenT}, the conformal dimension
$h_{\lambda, b, \nu}$ of $M(\lambda, b, \nu)$ is
\[
  h_{\lambda, b, \nu} = \frac{\lambda(\lambda+2)}{4(k+2)}
  +\frac{b^2}{8}-\frac{\langle\nu,\nu\rangle}{2}.
\]
From this and the balancing equation, $\cF\left(M(\lambda, b, \nu)\right)$ is untwisted if $b$ is even and twisted if $b$ is odd.

Analyzing the fusion for $\mathcal{L}^\even_k=M(0,0,0)$-modules, we have the following
fusion for $\mathcal{L}_k$-modules:
\begin{align*}
{N^+\,}_{\cF(M(\lambda,b,\nu)),\,\cF(M(\lambda',b',\nu'))}^{\cF(M(\lambda'',b'',\nu''))}&=
\begin{cases}
(N^{\mathfrak{sl}_2})^{\lambda''}_{\lambda,\lambda'} 
& \quad \mathrm{if}\,\,\, b+b'-b''\equiv 0,2\,\, \mathrm{mod}\,\, 4,\,\,\,
\nu+\nu'-\nu''\in L_{\mu},\\
(N^{\mathfrak{sl}_2})^{k-\lambda''}_{\lambda,\lambda'} 
& \quad \mathrm{if}\,\,\, b+b'-b''\equiv 0,2\,\, \mathrm{mod}\,\, 4,\,\,\,
\nu+\nu'-\nu''-\mu/2\in L_{\mu}.
\end{cases},\\
{N^-\,}_{\cF(M(\lambda,b,\nu)),\,\cF(M(\lambda',b',\nu'))}^{\cF(M(\lambda'',b'',\nu''))}&=
\begin{cases}
(N^{\mathfrak{sl}_2})^{\lambda''}_{\lambda,\lambda'} 
& \quad \mathrm{if}\,\,\, b+b'-b''\equiv 0\,\, \mathrm{mod}\,\, 4,\,\,\,
\nu+\nu'-\nu''\in L_{\mu},\\
(N^{\mathfrak{sl}_2})^{k-\lambda''}_{\lambda,\lambda'} 
& \quad \mathrm{if}\,\,\, b+b'-b''\equiv 2\,\, \mathrm{mod}\,\, 4,\,\,\,
\nu+\nu'-\nu''-\mu/2\in L_{\mu},\\
-(N^{\mathfrak{sl}_2})^{\lambda''}_{\lambda,\lambda'} 
& \quad \mathrm{if}\,\,\, b+b'-b''\equiv 2\,\, \mathrm{mod}\,\, 4,\,\,\,
\nu+\nu'-\nu''\in L_{\mu},\\
-(N^{\mathfrak{sl}_2})^{k-\lambda''}_{\lambda,\lambda'} 
& \quad \mathrm{if}\,\,\, b+b'-b''\equiv 0\,\, \mathrm{mod}\,\, 4,\,\,\,
\nu+\nu'-\nu''-\mu/2\in L_{\mu}
\end{cases} .
\end{align*}
Note how the conditions on $b$, $b'$, and $b''$ for obtaining non-zero fusion rules implies that the tensor product of two twisted or two untwisted modules is untwisted and the tensor product of one twisted and one untwisted module is twisted.

Now we summarize the $S$-transformations for (super)characters of
simple $\mathcal L_k$-modules.
The simple modules for $\mathcal L_k^{\even}=M(0,0,0)$ are indexed by a set $\cS$ of triples $(\lambda,b,\nu)$, where $0\leq \lambda\leq k$, $0\leq b\leq 3$ and $\nu+L_\mu\in L'_\mu/L_\mu$.
Let us fix a choice of representatives of $\cS$, say $\cS^\even$, under the equivalence $X\sim J\boxtimes X$.
Further choose $\cS^{\even,\text{loc}}$ amongst $\cS^\even$ for which $b$ is even, that is, these induce to untwisted modules,
and let $\cS^{\even,\text{tw}}=\cS^\even\setminus\cS^{\even,\text{loc}}$.
Then we have the following:
for $(\lambda,b,\nu)\in \cS^{\even,\text{loc}}$,
\begin{align*}
\ch^+[\cF(M(\lambda,b,\nu))]\left(-\dfrac{1}{\tau}\right)=
\sum_{(\lambda',b',\nu')\in \cS^{\even,\text{loc}}}
2\cdot S^{\chi}_{M(\lambda,b,\nu),M(\lambda',b',\nu')}
\cdot\ch^+[\cF(M(\lambda',b',\nu'))](\tau)\\
\ch^-[\cF(M(\lambda,b,\nu))]\left(-\dfrac{1}{\tau}\right)=
\sum_{(\lambda',b',\nu')\in \cS^{\even,\text{tw}}}
2\cdot S^{\chi}_{M(\lambda,b,\nu),M(\lambda',b',\nu')}
\cdot\ch^+[\cF(M(\lambda',b',\nu'))](\tau),
\end{align*}
and for $(\lambda,b,\nu)\in \cS^{\even,\text{tw}}$,
\begin{align*}
\ch^+[\cF(M(\lambda,b,\nu))]\left(-\dfrac{1}{\tau}\right)=
\sum_{(\lambda',b',\nu')\in \cS^{\even,\text{loc}}}
2\cdot S^{\chi}_{M(\lambda,b,\nu),M(\lambda',b',\nu')}
\cdot\ch^-[\cF(M(\lambda',b',\nu'))](\tau)\\
\ch^-[\cF(M(\lambda,b,\nu))]\left(-\dfrac{1}{\tau}\right)=
\sum_{(\lambda',b',\nu')\in \cS^{\even,\text{tw}}}
2\cdot S^{\chi}_{M(\lambda,b,\nu),M(\lambda',b',\nu')}
\cdot\ch^-[\cF(M(\lambda',b',\nu'))](\tau),
\end{align*}
where the $S$-matrix for $\mathcal{L}_k^\even$ is given by \eqref{eqn:N2S}.   Fusion rules for $\mathcal{L}_k$-modules may also be computed using the $S$-matrix for the above (super)character transformations and the Verlinde formula of Section \ref{CSsuperalgVerlinde}.

\subsubsection{Bershadsky-Polyakov algebra}\label{sec:BP}

Finally, we discuss two examples of vertex operator (super)algebras with wrong statistics. One source of such algebras is quantum Hamiltonian reduction associated to certain nilpotent elements of a simple Lie (super)algebra. Such $W$-algebras and their cosets have recently appeared prominently as invariants of interfaces of $4$-dimensional $\mathcal N =4$ supersymmetric gauge theories \cite{GR}, and these algebras enjoy a triality \cite{CL3} vastly generalizing Feigin-Frenkel duality \cite{FF} and the coset realization of principal $W$-algebras \cite{ACL2} for type $A$. The first new members are the subregular $W$-algebras of $\mathfrak{sl}_n$, which have fields of conformal weight $\frac{n}{2}$, and the principal $W$-superalgebras of $\mathfrak{sl}_{n|1}$, which have odd fields of conformal weight $\frac{n+1}{2}$. Hence these have wrong statistics if $n$ is odd. Moreover, the natural generalization of Feigin-Frenkel duality, conjectured by Feigin and Semikhatov \cite{FS2}, relates the subregular $W$-algebra of $\mathfrak{sl}_n$ at level $k$ to the principal $W$-superalgebra of $\mathfrak{sl}_{n|1}$ at level $\ell$, where the levels are assumed to be non-critical and related by $(k+n)(\ell+n-1)=1$ \cite{CGN}. We discuss now a simple example of such a dual pair in the instance $n=3$. This subregular $W$-algebra of $\mathfrak{sl}_3$ was first discussed by Bershadsky and Polyakov \cite{Ber, Pol}  and is thus named after them.

The Bershadsky-Polyakov algebra at level $k=-1/2$, which corresponds to the parameter  $\ell=1$ in \cite{ACL}, is an example of a vertex operator algebra with wrong statistics, and it is rational \cite{Ar1}. We obtain its representation theory using an inverse coset construction. Specifically, the Bershadsky-Polyakov algebra is constructed here as a simple current extension of a tensor product of commuting lattice and Virasoro vertex operator algebras.

Consider the Virasoro minimal model vertex operator algebra $\vir(3,5)$ 
of central charge $c=-3/5$. This algebra is simple, rational, $C_2$-cofinite, and of CFT type.
It has four irreducible modules up to isomorphism, $W_{1,s}$ for $1\leq s\leq 4$,
with $\vir(3,5)=W_{1,1}$. The conformal dimension of $W_{1,s}$ is $h_{1,s}=( (5-3s)^2 - 4)/60$,
and tensor products are given by:
\[
W_{1,s}\fus{}W_{1,n} \cong \bigoplus_{\substack{l = |n-s| +1 \\ 
		l+n+s \text{\,odd}
		}}^{\text{min}(n+s-1,9-n-s)} W_{1,l}.
\]
In other words, the fusion rules are:
\begin{align*}
N(\vir(3,5))_{(1,t),(1,t')}^{(1,t'')}= 
\begin{cases}
1 & \quad \mathrm{if}\,\, |t -t'|+1 \leq t'' \leq  \mathrm{min}\{t +t'-1,9-t-t'\},\,\,  t +t'+t''\,\,\mathrm{is\,\,odd},\\
0 & \quad \mathrm{otherwise} .
\end{cases}
\end{align*}
In particular, $W_{1,4}$ is a simple current of order $2$ and conformal dimension $3/4$; it has quantum dimension
$-1$ as can be seen from the $S$-matrix coefficients \eqref{eqn:virS} below.

Let $L=\beta\ZZ$ be an even positive definite lattice such that $\langle\beta,\beta\rangle=6$.
We consider the vertex operator algebra $V=  V_{L}\otimes \vir(3,5)$. 
It has a simple current extension $A= V \oplus J$ 
with $J=V_{\frac{1}{2}\beta+L}\otimes W_{1,4}$. 
Using quantum dimensions, it can be quickly checked that
this extension is a vertex operator algebra with wrong statistics \cite{CKL}. The algebra
$A$ is precisely the Bershadsky-Polyakov algebra at $\ell=1$ \cite{Kawa, ACL}.

Now let $V_{\frac{i}{6}\beta+L}\otimes W_{1,s}$ be a simple module for $V$,
and let us denote its conformal weight by $h_{i,1,s}$.
Since fusion factors over tensor products of vertex operator algebras in the rational setting,
we have:
\begin{align*}
\cF(V_{\frac{i}{6}\beta+L}\otimes W_{1,s}) &\cong \left(V_{\frac{i}{6}\beta+L}\otimes W_{1,s}\right)
\oplus  \left( V_{\frac{i+3}{6}\beta+L}\otimes W_{1,5-s}\right),\\
h_{i+3,1,5-s}&=\frac{i^2+6i + 9}{12} + \frac{(10 - 3s)^2-4}{60},\quad 
h_{i,1,s}=\frac{i^2}{12} + \frac{(5-3s)^2-4}{60}\\
\Rightarrow h_{i+3,1,5-s} - h_{i,1,s} + \ZZ &= \frac{i-s}{2} + \ZZ.
\end{align*}
Therefore,
\begin{align}
\cF(V_{\frac{i}{6}\beta +L}\otimes W_{1,s}) = 
\begin{cases}
\text{Untwisted} \quad &\text{if}\,\,\, i-s \equiv 1 \pmod{2},\\
\text{Twisted} \quad &\text{if}\,\,\, i-s \equiv 0 \pmod{2}.
\end{cases}
\label{eqn:BPinduction}
\end{align}

Up to isomorphism, there are 24 simple $V$-modules $X_{i,s}=V_{\frac{i}{6}\beta+L}\otimes W_{1,s}$ for $0\leq i\leq 5$ and $1\leq s\leq 4$. These modules fall into 12 classes under the equivalence relation $X\sim J\fus{} X$ for $X$ simple, and of these equivalence classes, 6 induce to untwisted and 6 induce to twisted modules. Moreover, every simple (un)twisted $A$-module can be realized as one of these. Let $\cS^\even$ denote the pairs $(i,s)$ corresponding to a choice of equivalence class representatives of $V$-modules, and let $\cS^{\even,\text{loc}}$, respectively $\cS^{\even,\text{tw}}$, denote the set of pairs $(i,s)\in\cS^\even$ such that $\cF(X_{i,s})$ is untwisted, respectively twisted. Thus every untwisted, respectively twisted, simple $A$-module is isomorphic to exactly one of the induced modules $\cF(X_{i,s})$ for $(i,s)\in\cS^{\even,\text{loc}}$, respectively $(i,s)\in\cS^{\even,\text{tw}}$. The fusion of simple $A$-modules is given by
\begin{align*}
\cF(X_{i,r}) \fus{A} \cF(X_{j,s}) &\cong \cF(X_{i,r} \fus{V} X_{j,s})
\cong\bigoplus_{\substack{l = |r-s| +1 \\ l+r+s \text{\,odd}}}^{\text{min}(r+s-1,9-r-s)} 
\cF(X_{i+j,l}).
\end{align*}
Note how fusion respects the twisting \eqref{eqn:BPinduction}. 
Noting also that $\cF(X_{i+j,s})\cong \cF(X_{i+j+3,5-s})$, the fusion rules are:
\begin{align}
\label{BP:fusioncoeff}
N_{\cF(X_{i,r}), \cF(X_{j,s}) }^{\cF(X_{k,t})}  &= 
\begin{cases}
N(\vir(3,5))_{(1,r), (1,s)}^{(1,t)}  &\quad \text{if\,}\, k-i-j\equiv 0 \pmod{6}\\
N(\vir(3,5))_{(1,r), (1,s)}^{(1,5-t)} &\quad \text{if\,}\, k-i-j\equiv 3 \pmod{6}\\
0 &\quad \text{otherwise}.
\end{cases}
\end{align}

Now we discuss $S$-transformations of characters of (twisted) $A$-modules. First we recall the modular $S$-matrix for $\vir(3,5)$ from \cite[Equation 10.134]{FMS}:
\begin{align*}
S^\chi_{W_{(r,s)}\, W_{(\rho, \sigma)}} &= 2\sqrt{\dfrac{2}{15}} (-1)^{1+s\rho +r\sigma}
\sin\left( \dfrac{5\pi}{3}r\rho \right)\sin\left( \dfrac{3\pi}{5} s\sigma\right),
\end{align*}
where $W_{1,1}=W_{(1,1)}$ but for other $s$, $W_{1,s} = W_{(2,5-s)}$.
Therefore we obtain, for $2\leq s,\sigma\leq 4$,
\begin{align*}
S^\chi_{W_{1,1}\, W_{1, 1}} &= - 2\sqrt{\dfrac{2}{15}} 
\sin\left( \dfrac{5\pi}{3} \right)\sin\left( \dfrac{3\pi}{5} \right)=\sqrt{\dfrac{2}{5}}\sin\left(\dfrac{3\pi}{5}\right),\\
S^\chi_{W_{1,1}\, W_{1, \sigma}} &= 2\sqrt{\dfrac{2}{15}} (-1)^{\sigma}
\sin\left( \dfrac{10\pi}{3} \right)\sin\left( \dfrac{3\pi}{5} (5-\sigma)\right)=(-1)^{\sigma+1}\sqrt{\dfrac{2}{5}}\sin\left(\dfrac{3\pi\sigma}{5}\right),\\
S^\chi_{W_{1,s}\, W_{1, \sigma}} &= -2\sqrt{\dfrac{2}{15}} 
\sin\left( \dfrac{20\pi}{3} \right)\sin\left( \dfrac{3\pi}{5} (5-s)(5-\sigma)\right)=(-1)^{s+\sigma}\sqrt{\dfrac{2}{5}}\sin\left(\dfrac{3\pi s\sigma}{5}\right).
\end{align*}
Note that in fact
\begin{equation}\label{eqn:virS}
S^\chi_{W_{1,s}\, W_{1, \sigma}} = (-1)^{s+\sigma}\sqrt{\dfrac{2}{5}}\sin\left(\dfrac{3\pi s\sigma}{5}\right)
\end{equation}
for all $1\leq s,\sigma\leq 4$. As in Section \ref{subsec:cosetmodtrans}, the modular $S$-matrix for $V_L$ is given by:
\[
S^\chi_{V_{\frac{a}{6}\beta+L}, V_{\frac{b}{6}\beta+L}} = 
\frac{1}{6}\exp\left(2\pi i \frac{ab}{6} \right).
\]
Using these $S$-matrices, (super)characters of $A$-modules have the following modular $S$-transformations: for $(i,s)\in \cS^{\even,\text{loc}}$,
\begin{align*}
\ch^+[\cF(X_{i,s})]\left(-\dfrac{1}{\tau}\right)=
\sum_{(j,r)\in \cS^{\even,\text{tw}}}
2\cdot S^{\chi, V_L}_{V_{\frac{i}{6}\beta+ L}, V_{\frac{j}{6}\beta+L}} 
\cdot
S^{\chi,\vir(3,5)}_{W_{1,s}, W_{1,r}} 
\cdot\ch^-[\cF(X_{j,r})](\tau)\\
\ch^-[\cF(X_{i,s})]\left(-\dfrac{1}{\tau}\right)=
\sum_{(j,r)\in \cS^{\even,\text{loc}}}
2\cdot S^{\chi, V_L}_{V_{\frac{i}{6}\beta+ L}, V_{\frac{j}{6}\beta+L}} \cdot
S^{\chi,\vir(3,5)}_{W_{1,s}, W_{1,r}} 
\cdot\ch^-[\cF(X_{j,r})](\tau),
\end{align*}
and for $(i,s)\in \cS^{\even,\text{tw}}$,
\begin{align*}
\ch^+[\cF(X_{i,s})]\left(-\dfrac{1}{\tau}\right)=
\sum_{(j,r)\in \cS^{\even,\text{tw}}}
2\cdot S^{\chi, V_L}_{V_{\frac{i}{6}\beta+ L}, V_{\frac{j}{6}\beta+L}} \cdot
S^{\chi,\vir(3,5)}_{W_{1,s}, W_{1,r}} 
\cdot\ch^+[\cF(X_{j,r})](\tau)\\
\ch^-[\cF(X_{i,s})]\left(-\dfrac{1}{\tau}\right)=
\sum_{(j,r)\in \cS^{\even,\text{loc}}}
2\cdot S^{\chi, V_L}_{\frac{i}{6}\beta+ L, \frac{j}{6}\beta+L} \cdot
S^{\chi,\vir(3,5)}_{W_{1,s}, W_{1,r}} 
\cdot\ch^+[\cF(X_{j,r})](\tau).
\end{align*}
The fusion rules for $A$-modules can also be derived using these $S$-matrices for $A$ together with the Verlinde formula of Section \ref{sec:Verlinde}.

\subsubsection{Superalgebra with wrong statistics}\label{sec:wrongsuper}

We can easily modify the inverse coset construction of the previous example to study the representation theory of a vertex operator superalgebra with wrong statistics. This algebra is in fact the simple principal $W$-superalgebra of $\mathfrak{sl}_{3|1}$ at level $\ell = -1/3$, by a special case of the main theorems of both \cite{CGN, CL3}. 

We again work with $\vir(3,5)$, but now we tensor it with $V_{L}$ where
$L=\gamma\ZZ$ with
$\langle\gamma,\gamma\rangle=10$. Let $V=V_{L}\otimes\vir(3,5)$.
As before, $V$ is a simple, rational, $C_2$-cofinite, vertex operator algebra of CFT type;
$V$ has an order-$2$ simple current $J=V_{\frac{1}{2}\gamma + L}\otimes W_{1,4}$ with $\qdim(J)=-1$ and $\theta_J=\id_J$. We set $A=V\oplus J$, and by \cite{CKL}, $A$ is a superalgebra of wrong statistics.
By Proposition \ref{prop:SimpleCurrentSmatrix}, every simple module in $\repA$ is an induced module, and conversely, every simple $V$-module
induces to a simple $\repA$-module. Let $X_{i,s}= V_{\frac{i}{10}\gamma+ L}\otimes W_{1,s}$,
and let us denote its conformal weight by $h_{i,1,s}$.
Then
\begin{align*}
\cF(X_{i,s}) &= X_{i,s}\oplus X_{i+5,5-s}\\
h_{i+5,1,5-s}&=\frac{i^2+10i + 25}{20} + \frac{(10 - 3s)^2-4}{60},\quad 
h_{i,1,s}=\frac{i^2}{20} + \frac{(5-3s)^2-4}{60}\\
\Rightarrow h_{i+5,1,5-s} - h_{i,1,s} + \ZZ &= \frac{i-s+1}{2} + \ZZ.
\end{align*}
Therefore,
\begin{align*}
\cF(X_{i,s}) = 
\begin{cases}
\text{Untwisted} \quad &\text{if}\,\,\, i-s \equiv 1 \pmod{2}\\
\text{Twisted} \quad &\text{if}\,\,\, i-s \equiv 0 \pmod{2}
\end{cases} .
\label{eqn:wrongstatBPinduction}
\end{align*}
Similar to the previous example, there are 40 simple $V$-modules up to isomorphism
indexed by the pairs $(i,j)$ with $0\leq i\leq 9$,
$1\leq j\leq 4$, which fall 
into 20 classes under the identification $X\sim J\boxtimes X$.
 Of these equivalence classes, 10  induce to untwisted modules  
and 10 to twisted modules. As before, we denote sets of representatives of these classes by $\cS^\even$, $\cS^{\even,\text{loc}}$, and $\cS^{\even,\text{tw}}$. Up to (possibly odd) isomorphism,
each untwisted, respectively twisted, simple $A$-module  is induced from exactly one module $X_{i,s}$ for $(i,s)\in\cS^{\even,\text{loc}}$, respectively $(i,s)\in\cS^{\even,\text{tw}}$.

Similar to \eqref{BP:fusioncoeff},  we have the following (super)fusion rules:
\begin{align*}
(N^\pm)_{\cF(X_{i,s}), \cF(X_{j,r})}^{\cF(X_{k,t})}
&={N}_{X_{i,s}, X_{j,r}}^{X_{k,t}}\pm{N}_{X_{i,s}, X_{j,r}}^{X_{5+k,5-t}}\nonumber\\
&=\begin{cases}
\hphantom{\pm } N(\vir(3,5))_{(1,s),\,(1,r)}^{(1,t)}\,\,\text{ if } k\equiv i+j\pmod{10}\\
\pm N(\vir(3,5))_{(1,s),\,(1,r)}^{(1,5-t)}\,\,\text{ if } k\equiv i+j+5\pmod{10}
\end{cases} .
\end{align*}
The $S$-transformations of (super)characters of $\cF(X_{i,s})$ for $(i,s)\in\cS^\even$ have exactly the same form as for the Bershadsky-Polyakov algebra; the only difference is that we use the $S$-matrix for $V_{\gamma\ZZ}$ rather than $V_{\beta\ZZ}$. Fusion rules for $A$-modules may also be computed using the resulting $S$-matrix for $A$.

For further examples of regular vertex operator superalgebras with wrong statistics, see the analysis of the affine \VOSA{}  
 $L_k\left(\mathfrak{osp}(1|2)\right)$ and its parafermion coset,  at positive integer level $k$,  in \cite{CFK}.

\bibliographystyle{alpha}

\bigskip

\noindent
(T. C.) Department of Mathematical and Statistical Sciences,
University of Alberta, 
Edmonton, Alberta T6G 2G1, Canada\\
\emph{E-mail:} \texttt{creutzig@ualberta.ca}

\medskip

\noindent
(S. K.) Department of Mathematical and Statistical Sciences,
University of Alberta, 
Edmonton, Alberta T6G 2G1, Canada\\
\textit{Current address:} Department of Mathematics, University of Denver, Denver, Colorado 80210, USA
\emph{E-mail:} \texttt{shashank.kanade@du.edu}

\medskip

\noindent
(R. M.) Department of Mathematics,
Vanderbilt University, 
Nashville, Tennessee 37240, USA\\
\textit{Current address:} Yau Mathematical Sciences Center, Tsinghua University, Beijing 100084, China\\
\emph{E-mail:} \texttt{rhmcrae@tsinghua.edu.cn}
\end{document}